\documentclass[a4paper,10pt,twoside]{article}
\usepackage{geometry}
\usepackage{amssymb}
\usepackage{amsbsy}
\usepackage{amsmath}
\usepackage{amsthm}
\usepackage{amsfonts}
\usepackage{paralist}
\usepackage{enumitem}
\usepackage{titlesec}
\usepackage{fancyhdr}
\usepackage{xr}
\usepackage{multind}
\usepackage[utf8x]{inputenc}
\usepackage{multicol}
\usepackage{graphicx}
\usepackage{color}
\usepackage{comment}
\usepackage{ifthen}

\newcommand*{\mybackslash}{\mathbin{\reflectbox{$/$}}}

\newboolean{version_pdf_hyperref}
\newboolean{version_pdf_print}
\newboolean{version_dvi}

\setboolean{version_pdf_hyperref}{true}
\setboolean{version_pdf_print}{false}
\setboolean{version_dvi}{false}

\ifthenelse{\boolean{version_dvi} \OR \boolean{version_pdf_print}}
{
\usepackage{url}
}

\numberwithin{equation}{section}

\ifthenelse{\boolean{version_pdf_hyperref}}
{
\usepackage[pdftex]{hyperref}
\hypersetup{
  a4paper=true,
  pdfauthor={Ulrich Thiel},
  pdftitle={Mackey functors and abelian class field theories},
  colorlinks=true,
  linkcolor=black,
  urlcolor=black,
  citecolor=black
 }
}{}

\begingroup\makeatletter\ifx\SetFigFont\undefined%
\gdef\SetFigFont#1#2#3#4#5{%
  \reset@font\fontsize{#1}{#2pt}%
  \fontfamily{#3}\fontseries{#4}\fontshape{#5}%
  \selectfont}%
\fi\endgroup%

\titleformat{\section}{\center \normalfont \Large \bfseries}{\thesection.}{.3em}{}
\titlespacing{\section}{0pt}{*0}{*0}

\titleformat{\subsection}{\center \normalfont\large\bfseries}{\thesubsection.}{.5em}{}
\titlespacing{\subsection}{0pt}{20pt}{0pt}

\swapnumbers
\newtheoremstyle{theo_style_std}{8pt}{8pt}{}{}{\bfseries}{.}{ }{mendoa}
\theoremstyle{theo_style_std}
\newtheorem{thm}{Theorem}[subsection]
\newtheorem{prop}[thm]{Proposition}
\newtheorem{lemma}[thm]{Lemma}
\newtheorem{cor}[thm]{Corollary}
\newtheorem{defn}[thm]{Definition}
\newtheorem{ass}[thm]{Assumption}

\newtheorem{conv}[thm]{Convention}
\newtheorem{para}[thm]{}
\newtheorem{open}[thm]{Open Problem}

\reversemarginpar

\input xy
\xyoption{all}
\xyoption{arrow}

\newdir{ >}{{ }*!/-9pt/@{>}}

\newcommand{\word}[1]{\textit{#1}\index{general}{#1}}
\newcommand{\words}[2]{\textit{#1}\index{general}{#2}}
\newcommand{\invword}[1]{\index{general}{#1}}
\newcommand{\noword}[1]{\textit{#1}}

\newcommand{\wordsym}[1]{\index{symbol}{#1}}

\newcommand{\mc}[1]{}

\ifthenelse{\boolean{version_pdf_hyperref}}
{
\newcommand{\fixme}[1]{}
}{}
\ifthenelse{\boolean{version_pdf_print} \or \boolean{version_dvi}}
{
\newcommand{\fixme}[1]{}
}{}

\newcommand{\NN}{\mathbb{N}}
\newcommand{\ZZ}{\mathbb{Z}}
\newcommand{\QQ}{\mathbb{Q}}
\newcommand{\RR}{\mathbb{R}}

\newcommand{\FF}{\mathbb{F}}
\newcommand{\PP}{\mathbb{P}}

\renewcommand{\AA}{\mathbb{A}}

\newcommand{\eps}{\varepsilon}

\renewcommand{\lim}[1]{\underset{#1}{\mathrm{lim}}}
\newcommand{\modd}{ \ \mathrm{mod} \ }

\renewcommand{\ker}{\mathrm{Ker}}
\newcommand{\coker}{\mathrm{Coker}}
\newcommand{\im}{\mathrm{Im}}
\renewcommand{\hom}{\mathrm{Hom}}

\newcommand{\id}{\mathrm{id}}

\newcommand{\subs}{\subseteq}
\newcommand{\sups}{\supseteq}
\newcommand{\dopgleich}{\mathrel{\mathop:}=}
\newcommand{\gleichdop}{=\mathrel{\mathop:}}

\newcommand{\ra}{\rightarrow}

\newcommand{\lra}{\longrightarrow}

\newcommand{\RA}{\Rightarrow}

\newcommand{\lRA}{\Longrightarrow}
\newcommand{\LRA}{\Leftrightarrow}
\newcommand{\lLRA}{\Longleftrightarrow}

\newcommand{\ol}[1]{\overline{#1}}

\newcommand{\oset}[1]{\overset{#1}}

\newcommand{\tn}[1]{\textnormal{#1}}

\newcommand{\fr}[1]{\mathfrak{#1}}
\newcommand{\ca}[1]{\mathcal{#1}}

\newcommand{\ro}[1]{\mathrm{#1}}
\newcommand{\se}[1]{\mathsf{#1}}
\newcommand{\name}[1]{{#1}}

\newcommand{\ssys}[1]{\mathfrak{S}_{\mathrm{#1}}}
\newcommand{\msys}[1]{\mathfrak{M}_{\mathrm{#1}}}
\newcommand{\esys}[1]{\mathfrak{E}_{\mathrm{#1}}}

\newcommand{\dsys}[1]{\mathfrak{D}_{\mathrm{#1}}}
\newcommand{\con}{\mathrm{con}}
\newcommand{\res}{\mathrm{res}}
\newcommand{\ind}{\mathrm{ind}}
\newcommand{\comm}[1]{\mathrm{Com}^{\ro{#1}}}
\newcommand{\ver}{\mathrm{Ver}}

\newcommand{\gal}{\ro{Gal}}
\newcommand{\lMod}[1]{{_{#1}\mathsf{Mod}}}
\newcommand{\lmod}[1]{{_{#1}\mathsf{mod}}}

\newcommand{\discgmod}{ {_G}\mathsf{Ab}^{\mathrm{d}}}
\newcommand{\tateco}{\widehat{\mathrm{H}}}
\newcommand{\calo}{{\scriptstyle{\mathcal{O}}}}
\newcommand{\mspec}{\mathrm{mSpec}}

\usepackage[ngerman,french,russian,american]{babel}

\usepackage{fouriernc}

\makeindex{general}
\makeindex{symbol}

\author{Ulrich Thiel\footnote{University of Kaiserslautern. Email: \texttt{thiel at mathematik.uni-kl.de}}}
\title{Mackey functors and abelian class field theories}

\begin{document}

\maketitle
\thispagestyle{empty}

\begin{abstract}
Motivated by the work of \name{Jürgen Neukirch} and \name{Ivan Fesenko} we propose a general definition of an abelian class field theory from a purely group-theoretical and functorial point of view. This definition allows a modeling of abelian extensions of a field inside more general objects than the invariants of a discrete module over the absolute Galois group of the field. The main objects serving as such models are cohomological Mackey functors as they have enough structure to make several reduction theorems of classical approaches work in this generalized setting and, as observed by Fesenko, they even have enough structure to make Neukirch's approach to class field theories via Frobenius lifts work. This approach is discussed in full detail and in its most general setting, including the pro-$P$ setting proposed by Neukirch. As an application and justification of this generalization we describe Fesenko's approach to class field theory of higher local fields of positive characteristic, where the modeling of abelian extensions takes place inside the cohomological Mackey functor formed by the Milnor--Par\v{s}in $\ro{K}$-groups.

The motivation for this work (which is the author's Diplom thesis) was the attempt to understand what a class field theory is and to give a single-line definition which captures certain common aspects of several instances of class field theories. We do not claim to prove any new theorem here, but we think that our general and uniform approach offers a point of view not discussed in this form in the existing literature.
\end{abstract}

\newpage
\tableofcontents

\newpage
\section{Introduction}

The author's first contact with class field theories was \name{Jürgen Neukirch}'s presentation of this concept in his famous book \cite[chapter IV]{Neu99_Algebraic-Number_0} on algebraic number theory. Unfortunately, and this may be explained by the author's lack of knowledge and intuition in algebraic number theory, the author did not understand what the goal of all the investigations was until the very end of the presentation where the main results were summarized in a theorem. Surprised by the fact that this theorem was not presented at the beginning as a motivation for all further considerations, the author started to think about the abstract core of this theorem and tried to define what exactly a class field theory is. Such a single definition did not exist in the literature and it seemed that this was an obvious concept for people working in algebraic number theory. However, the author's abstract considerations were further motivated by the exercises in this chapter where Neukirch proposes a generalization of his approach, mentioning that this generalization has been applied by \name{Ivan Fesenko} in \cite{Fes92_Multidimensional-local_0}. This thesis is the result of trying to understand what a class field theory is and trying to give a single-line definition (see \ref{cft_definition}) which captures certain aspects of several explicit class field theories. Although no new theorems are proven and the reader familiar with class field theories will not find anything surprising in this thesis, the author still hopes to communicate at least a certain point of view which is not presented in this form in the existing literature. \\

To give an overview of the contents, we first have to present one result of the above considerations, namely what the idea of a class field theory is (we will give a more detailed motivation at the beginning of chapter \ref{chap:acfts}). The essence of an abelian class field theory (ACFT for short) is to provide for a field $k$ a description of the finite abelian extensions of finite separable extensions $K|k$. There may exist of course several formalizations of this fuzzy concept and in this thesis we will develop one particular formalization. We understand an ACFT as consisting of three parts: a group-theoretical part, a functorial or compatibility part and an arithmetic part. 

In the first place, an ACFT should model for each finite separable extension $K|k$ the finite Galois extensions of $K$ as certain subgroups of some abelian group $C(K)$ in such a way that this model is faithful on the lattice of abelian extensions of $K$ and such that abelianized Galois groups can be calculated in this model. More explicitly, there should exist a map $\Phi(K,-): \ca{E}^{\tn{f}}(K) \ra \ca{E}^{\ro{a}}(C(K))$ from the set $\ca{E}^{\tn{f}}(K)$ of all finite Galois extensions of $K$ to the set $\ca{E}^{\ro{a}}(C(K))$ of all subgroups of $C(K)$ such that the restriction of $\Phi(K,-)$ to the lattice $\ca{L}(K) \subs \ca{E}^{\tn{f}}(K)$ of finite abelian extensions is an injective morphism of lattices and there should exist an isomorphism, called \textit{reciprocity morphism},
\[
\rho_{L|K}: \gal(L|K)^{\ro{ab}} \ra C(K)/\Phi(K,L)
\]
for each $L \in \ca{E}^{\tn{f}}(K)$. We can depict the passage from the field internal theory to its model as the following scheme
\[
\begin{array}{rcl}
K & \rightsquigarrow  & C(K) \\
L \in \ca{E}^{\tn{f}}(K) & \rightsquigarrow & \Phi(K,L) \leq C(K) \\
\gal(L|K) & \rightsquigarrow & C(K)/\Phi(K,L).
\end{array}
\]

These data are what we will refer to as the \textit{group-theoretical part} of an ACFT. The \textit{functorial part} of an ACFT is the requirement that this passage should satisfy several compatibility relations. It is an essential part of this thesis to precisely define these compatibility relations and to introduce objects that capture these relations. These objects will be the RIC-functors introduced in chapter \ref{chap:ric_functors}. The abbreviation RIC stands for \textit{Restriction-Induction-Conjugation}, three natural operations which occur in several places of mathematics. In particular, such a triplet of operations exists on the abelianizations above and therefore the abelian groups $C(K)$ should also be connected with such operations, that is, they should form a RIC-functor $C$ and the reciprocity morphisms should be compatible with these operations, that is, the reciprocity morphism should be a morphism of RIC-functors. The \textit{arithmetic part} of an ACFT will remain to be a fuzzy concept and can be described as the condition that an ACFT should be tied to the ``arithmetic'' of $k$. The meaning of this is two-fold. On the one hand, it means that the groups $C(K)$ should be obtained directly from $k$ itself. By what we have said about ACFTs, this would imply that all the abelian extensions and their Galois groups can be computed by data directly connected to $k$. \name{Claude Chevalley} explains this philosophy in \cite{Che40_La-theorie-du-corps_0} as follows:
\begin{quote}
\foreignlanguage{french}{L'objet de la th\'eorie du corps de classes est de montrer comment les extensions abeliennes d'un corps de nombres alg\'ebriques $k$ peuvent \^{e}tres d\'etermin\'ees par des \'el\'ements tir\`es de la connaissance de $k$ lui-m\^{e}me; ou, si l'on veut pr\'esenter les choses en termes dialectiques, comment un corps poss\^ede en soi les \'el\'ements de son propre d\'epassement (et ce, sans aucune contradiction interne!).}\footnote{``The object of class field theory is to show how the abelian extensions of an algebraic number field $k$ can be determined by objects taken from our knowledge of $k$ itself; or, if one wishes to present things in dialectical terms, how a field contains within itself the elements of its own surpassing (and this without any internal contradiction!).''}
\end{quote}

To get an idea of this, we note that for local (respectively global) fields there exists an ACFT in the sense discussed so far, called \word{local class field theory} (respectively \word{global class field theory}). For a local field $k$ the RIC-functor $C$ of the local class field theory is given by the multiplicative groups $C(K) = \ro{GL}_1(K)$ with the obvious conjugation, restriction and induction morphisms. For a global field $k$ the RIC-functor $C$ of the global class field theory is given by the id\`ele class groups $C(K) = \ro{GL}_1(\AA_K)/\ro{GL}_1(K)$ with the canonical conjugation, restriction and induction morphisms, where $\AA_K$ is the ad\`ele ring of $K$. In both local and global class field theory it is evident that $C$ is more or less directly obtained from $k$. 

On the other hand, the meaning of being tied to the ``arithmetic'' of $k$ is that an ACFT should, at least in the case of local or global fields, provide information about the ring of integers in the extensions such as ramification. Both the local and global class field theories provide such information. As it is not clear if all explicit ACFTs share certain arithmetic properties and as our general group-theoretical and functorial point of view of ACFTs does not allow a formalization of such properties, we will ignore the arithmetic part of ACFTs in this thesis. These have to be uncovered in explicit situations. \\

As indicated by the examples of local and global class field theories, the main source of the RIC-functors $C$ in which the modeling of abelian extensions takes place are the invariants of discrete modules over the absolute Galois group $\gal(k)$ of the ground field. But from the general point of view using RIC-functors this would be an unnecessary restriction. In fact, it turns out that important reduction theorems, which reduce the amount of work necessary to verify that a morphism is a reciprocity morphism and thus gives a class field theory, already work when $C$ is just a cohomological Mackey functor. Moreover, \name{Neukirch}'s approach to class field theories presented in \cite[chapter IV]{Neu99_Algebraic-Number_0} can be generalized to this setting. This was observed by \name{Ivan Fesenko} in \cite{Fes92_Multidimensional-local_0} who applied this to get an ACFT describing the abelian extensions of a higher local field and thus giving \name{Aleksej Par\v{s}hin}'s original approach to this problem presented in \cite{0579.12012} the interpretation as being just another instantiation of Neukirch's theory, although it is of course not trivial to see this. The cohomological Mackey functor $C$ in this case is formed by the Milnor--Par\v{s}in $\ro{K}$-groups and the point is that the values of $C$ cannot be identified with the invariants of a discrete module so that the classical theory, which just deals with discrete modules, cannot be used. This (and further problems discussed in \ref{chap:acfts}) should be a motivation and justification for the general point of view of this thesis. \\

In chapter \ref{chap:ric_functors} the main objects needed to define an ACFT from the proposed point of view, namely the RIC-functors, are introduced. Mackey functors are identified as special cases of RIC-functors and the important examples of discrete modules and abelianization as cohomological Mackey functors are discussed. In chapter \ref{chap:acfts} a special type of RIC-functors, the representations, are identified as being the right objects to model abelian extensions. A definition of ACFTs in this language is then given and several abstract theorems about them are provided, most importantly the reduction theorems. The heart of this thesis is chapter \ref{chap:fs_cfts} where Neukirch's approach to class field theories is generalized as far as possible, namely to cohomological Mackey functors instead of discrete modules and to a pro-$P$ setting, that is, to fields whose absolute Galois group is a pro-$P$ group which admits a quotient of the form $\ZZ_P = \prod_{p \in P} \ZZ_p$ with $P$ being a set of prime numbers. The only reference for the passage from discrete modules to cohomological Mackey functors in this theory is \cite{Fes92_Multidimensional-local_0} and the proof given there is kept rather short. Moreover, the passage from the $\widehat{\ZZ}$ to the $\ZZ_P$-setting, which involves a shift in the interpretation of Neukirch's theory, is not mentioned elsewhere in the literature except for an exercise given by Neukirch which was the motivation for this generalization. Therefore, this theory is discussed full detail. Chapter \ref{chap:disc_vals} is an overview about discrete valuations and in particular discrete valuations of higher rank. This chapter is a preparation for chapter \ref{sect:cfts_for_disc_val} where the way to obtain the classical local class field theory and Fesenko's higher local class field theory is described. The appendix contains a summary of some results about topological groups and projective limits of topological groups which the author has written to convince himself that all arguments work and to reduce the amount of external references. Moreover, a general concept of certain universal abelian quotients of compact groups is introduced which generalizes the theory in the case of profinite groups presented in \cite[section 3.4]{RibZal00_Profinite-Groups_0} but which is not needed elsewhere in this thesis. The reader who is either familiar with topological groups or is not interested in class field theories describing more special abelian quotients than the maximal abelian quotient can ignore nearly all references to the appendix. \\

The author's original plan was to also present in full detail Fesenko's approach to higher local class field theory, and in particular to verify the properties of the Milnor--Par\v{s}in $\ro{K}$-groups needed to get a class field theory. Unfortunately, it turned out that too many details are involved in this theory (not only about Milnor $\ro{K}$-theory but also about sequential topologies). The time available for writing this Diplom thesis did not suffice for the author to work through all these details and therefore some details of Fesenko's approach to higher local class field theory are only sketched. However, the author tried to make this everything as precise as possible and so this thesis may also be viewed as a guide to this approach. \\

I would like to thank my supervisor, Prof.\ Dr.\ \name{Gunter Malle} (University of Kaiserslautern), for the motivation to work on this topic, for the possibility to ask questions at any time and for reading all this. Moreover, I would like to thank Prof.\ \name{Ivan Fesenko} (University of Nottingham) for sending me copies of some of his papers which were nowhere else available. \\

The author wants to close this introduction with the following quote by \name{Alexander Grothendieck} found in \cite[lettre de 19.9.1956]{ColSer01_Correspondance-Grothendieck-Serre_0} which describes very accurately the author's familiarity with class field theory:
\begin{quote}
\foreignlanguage{french}{$\lbrack ... \rbrack$ j'ai revu un peu la th\'eorie du corps de classes, dont j'ai enfin l'impression d'avoir compris les \'enonc\'es essentiels (bien entendu, pas les d\'emonstrations!).}\footnote{``$\lbrack ... \rbrack$ I have been reviewing class field theory, of which I finally have the impression that I understand the main results (but not the proofs of course!).''}
\end{quote}

\vspace{40pt}

\vspace{60pt}

\textbf{Conventions.}

As set theoretic foundation we use NBG set theory as presented in \cite[chapter 4]{Men97_Introduction-to-Mathematical_0} extended by the axiom of choice and extended by the axiom of regularity. We assume consistency of this theory.

All constructions concerning limits in categories are formulated in the language of \cite{HofMor09_Contributions-to-the-structure_0}, although everything should be standard terminology.

A ring is always a ring with unit and a morphism of rings is always unit preserving. For a ring $k$ the category of left unital $k$-modules is denoted by $\lMod{k}$ and its full subcategory of finitely generated $k$-modules is denoted by $\lmod{k}$. 

For a commutative ring $k$ the category of $k$-algebras is denoted by ${_k\se{Alg}}$ and the category of commutative $k$-algebras is denoted by ${_k\se{CAlg}}$. We moreover define a \word{general $k$-algebra} to be a $k$-module $A$ equipped with a $k$-bilinear map $A \times A \ra A$ and we denote the category of general $k$-algebras by ${_k\se{GenAlg}}$.

Topological spaces are in general not assumed to be separated. The closure of a subset $A$ of a topological space is denoted by $\ro{cl}(A)$. \wordsym{$\ro{cl}(A)$}. The notions \textit{quasi compact} and \textit{compact} are used as in \cite{Bou71_Topologie-Generale_0}. The category of topological groups is denoted by $\se{TGrp}$, the category of topological abelian groups is denoted by $\se{TAb}$ and the category of separated (compact, locally compact) abelian groups is denoted by $\se{TAb}^{\ro{s}}$ ($\se{TAb}^{\ro{com}}$, $\se{TAb}^{\ro{lc}}$).

If $G$ is a topological group, then we write $H \leq_{\ro{c}} G$ ($H \leq_{\ro{o}} G$, $H \lhd_{\ro{c}} G$, $H \lhd_{\ro{o}} G$) to denote that $H$ is a closed subgroup (open subgroup, closed normal subgroup, open normal subgroup) of $G$. We denote by $\ca{E}^{\ro{a}}(G)$ ($\ca{E}^{\ro{t}}(G)$, $\ca{E}^{\ro{f}}(G)$) the set of all abstract (closed, closed of finite index) normal subgroups of $G$. 

Fields are always commutative. A field extension $L \sups K$ is also denoted by $L|K$. The Galois group of a field extension $L|K$ is denoted by $\gal(L|K)$ and the absolute Galois group of a field $k$ is denoted by $\gal(k) \dopgleich \gal(k^{\ro{s}}|k)$, where $k^{\ro{s}}$ is the separable closure of $k$. By $\ca{E}^{\ro{f}}(K)$ we denote the set of all finite Galois extensions of $K$.

\wordsym{$\lMod{k}$} \wordsym{$\se{TGrp}$} \wordsym{$\se{TAb}$} \wordsym{$\se{TAb}^{\ro{s}}$} \wordsym{$\se{TAb}^{\ro{com}}$} \wordsym{$\se{TAb}^{\ro{lc}}$} \wordsym{$k^{\ro{s}}$}

\wordsym{$\ca{E}^{\ro{a}}(G)$} \wordsym{$\ca{E}^{\ro{t}}(G)$} \wordsym{$\ca{E}^{\ro{f}}(G)$}
\wordsym{$\ca{E}^{\ro{f}}(K)$}

\newpage
\section{RIC-functors} \label{chap:ric_functors}

In this chapter we will introduce a mathematical object that provides a formal framework for the following situation which is encountered in several places of mathematics: suppose we are given a group $G$, a set $\ssys{b}$ of ``interesting'' or ``relevant'' subgroups of $G$, a category $\ca{C}$ and an object $\Phi(H)$ of $\ca{C}$ attached to each $H \in \ssys{b}$ describing the group $H$ itself or describing objects connected to $H$. An example would be to take as $\ssys{b}$ the set of all subgroups of $G$ and as $\Phi(H)$ the cohomology group $\ro{H}^n(H,A)$ of a fixed $G$-module $A$ for fixed $n \in \NN$ or the additive group of the representation ring $\ro{R}_k(H)$ over a commutative ring $k$ (assuming that $G$ is finite to be careful with notations).
Now, the mathematical context of the map $\Phi:\ssys{b} \ra \ro{Ob}(\ca{C})$ might provide relations between the objects $\Phi(H)$ and $\Phi(I)$ for certain $H,I \in \ssys{b}$. More precisely, we might have have a $\ca{C}$-morphism $\res_{I,H}^\Phi:\Phi(H) \ra \Phi(I)$, called \textit{restriction} morphism, whenever $H \in \ssys{b}$ and $I \in \ssys{r}(H)$, where $\ssys{r}(H)$ is a set containing the subgroups of $H$ to which a restriction is allowed. Similarly, we might have a $\ca{C}$-morphism $\ind_{H,I}^\Phi:\Phi(I) \ra \Phi(H)$, called \textit{induction}, whenever $H \in \ssys{b}$ and $I \in \ssys{i}(H)$. Moreover, we might have a $\ca{C}$-morphism $\con_{g,H}^\Phi:\Phi(H) \ra \Phi(^g \! H)$,  called \textit{conjugation}, for each $H \in \ssys{b}$ and $g \in G$.\footnote{Here, we have to assume of course that $\ssys{b}$ is closed under conjugation with elements from $G$.}
In the case of the cohomology groups such relations are indeed present: we can take as restriction  the usual restriction, as induction the corestriction and as conjugation the multiplication with the conjugating element. Since the corestriction is only defined for subgroups of finite index, we see that $\ssys{i}(H)$ is in this case only allowed to contain subgroups of finite index of $H$. For the representation ring we also have such relations: we can take as restriction, induction and conjugation the morphisms assigning to a representation its restriction, induction and conjugation.

\begin{center}
\scalebox{1.0}{
\begin{picture}(0,0)%
\includegraphics{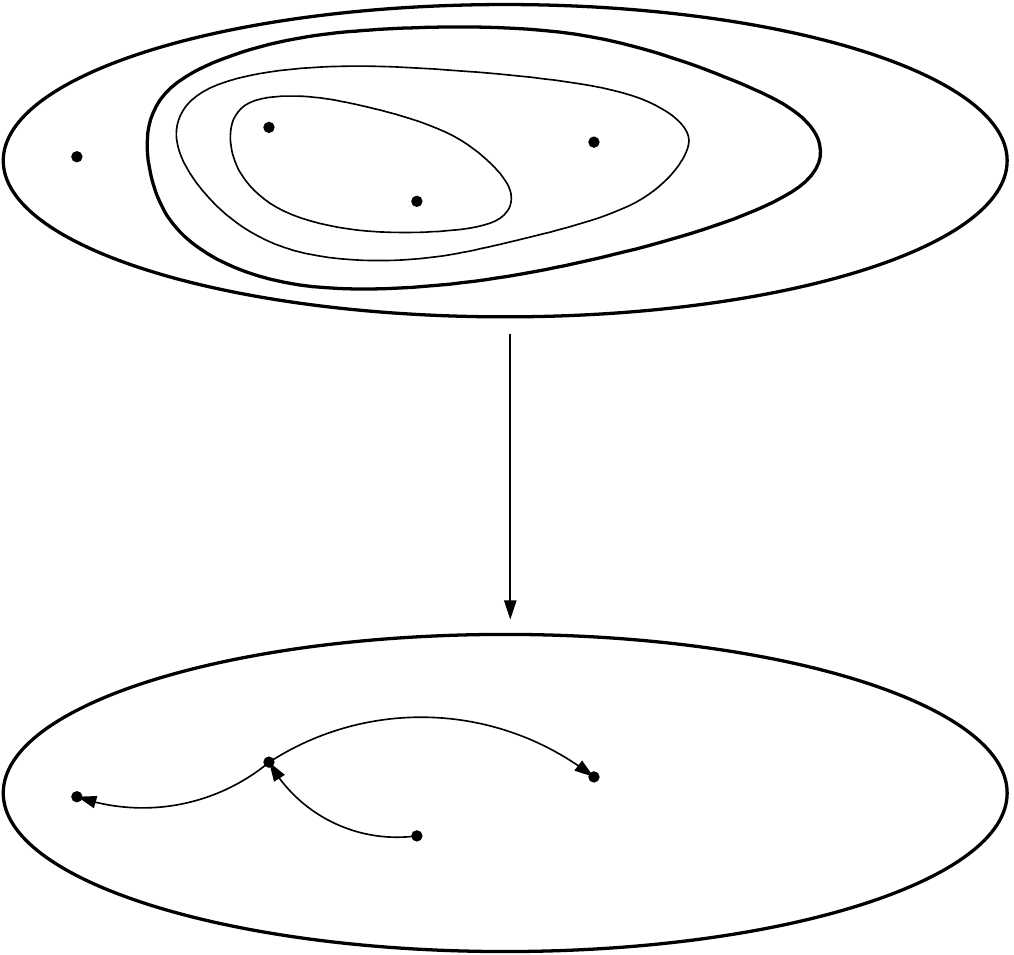}%
\end{picture}%
\setlength{\unitlength}{4144sp}%
\begin{picture}(4620,4359)(2573,-7646)
\put(2745,-6832){\makebox(0,0)[lb]{\smash{{\SetFigFont{8}{9.6}{\rmdefault}{\mddefault}{\updefault}{\color[rgb]{0,0,0}$\Phi(^g \! H)$}%
}}}}
\put(2858,-3953){\makebox(0,0)[lb]{\smash{{\SetFigFont{8}{9.6}{\rmdefault}{\mddefault}{\updefault}{\color[rgb]{0,0,0}$^g \! H$}%
}}}}
\put(6368,-7576){\makebox(0,0)[lb]{\smash{{\SetFigFont{8}{9.6}{\rmdefault}{\mddefault}{\updefault}{\color[rgb]{0,0,0}$\mathcal{C}$}%
}}}}
\put(6368,-4673){\makebox(0,0)[lb]{\smash{{\SetFigFont{8}{9.6}{\rmdefault}{\mddefault}{\updefault}{\color[rgb]{0,0,0}$\mathfrak{S}_{\mathrm{b}}$}%
}}}}
\put(3916,-3908){\makebox(0,0)[lb]{\smash{{\SetFigFont{8}{9.6}{\rmdefault}{\mddefault}{\updefault}{\color[rgb]{0,0,0}$H$}%
}}}}
\put(4568,-4245){\makebox(0,0)[lb]{\smash{{\SetFigFont{8}{9.6}{\rmdefault}{\mddefault}{\updefault}{\color[rgb]{0,0,0}$J$}%
}}}}
\put(5401,-3976){\makebox(0,0)[lb]{\smash{{\SetFigFont{8}{9.6}{\rmdefault}{\mddefault}{\updefault}{\color[rgb]{0,0,0}$I$}%
}}}}
\put(4455,-3795){\makebox(0,0)[lb]{\smash{{\SetFigFont{8}{9.6}{\rmdefault}{\mddefault}{\updefault}{\color[rgb]{0,0,0}$\mathfrak{S}_{\mathrm{i}}(H)$}%
}}}}
\put(5783,-4088){\makebox(0,0)[lb]{\smash{{\SetFigFont{8}{9.6}{\rmdefault}{\mddefault}{\updefault}{\color[rgb]{0,0,0}$\mathfrak{S}_{\mathrm{r}}(H)$}%
}}}}
\put(6053,-3660){\makebox(0,0)[lb]{\smash{{\SetFigFont{8}{9.6}{\rmdefault}{\mddefault}{\updefault}{\color[rgb]{0,0,0}$\mathfrak{S}_{\mathrm{b}}(H)$}%
}}}}
\put(4995,-5415){\makebox(0,0)[lb]{\smash{{\SetFigFont{8}{9.6}{\rmdefault}{\mddefault}{\updefault}{\color[rgb]{0,0,0}$\Phi$}%
}}}}
\put(4523,-6473){\makebox(0,0)[lb]{\smash{{\SetFigFont{8}{9.6}{\rmdefault}{\mddefault}{\updefault}{\color[rgb]{0,0,0}$\mathrm{res}_{I,H}^\Phi$}%
}}}}
\put(5401,-6878){\makebox(0,0)[lb]{\smash{{\SetFigFont{8}{9.6}{\rmdefault}{\mddefault}{\updefault}{\color[rgb]{0,0,0}$\Phi(I)$}%
}}}}
\put(3510,-6697){\makebox(0,0)[lb]{\smash{{\SetFigFont{8}{9.6}{\rmdefault}{\mddefault}{\updefault}{\color[rgb]{0,0,0}$\Phi(H)$}%
}}}}
\put(4568,-7148){\makebox(0,0)[lb]{\smash{{\SetFigFont{8}{9.6}{\rmdefault}{\mddefault}{\updefault}{\color[rgb]{0,0,0}$\Phi(J)$}%
}}}}
\put(3848,-7193){\makebox(0,0)[lb]{\smash{{\SetFigFont{8}{9.6}{\rmdefault}{\mddefault}{\updefault}{\color[rgb]{0,0,0}$\mathrm{ind}_{H,J}^\Phi$}%
}}}}
\put(3173,-7103){\makebox(0,0)[lb]{\smash{{\SetFigFont{8}{9.6}{\rmdefault}{\mddefault}{\updefault}{\color[rgb]{0,0,0}$\mathrm{con}_{g,H}^\Phi$}%
}}}}
\end{picture}%
}

\footnotesize{The idea of a RIC-functor (note that to make this picture readable, we assumed that $\ssys{i}(H) \subs \ssys{r}(H)$ but we will not assume this in general).\footnote{The two ellipses in the picture are supposed to be seen as discs in 3-space.}}
\end{center}

If one wants to abstractly characterize what the above mathematical structure is about, then one immediately realizes that the pivot in this characterization is the upper disc in the picture which can be seen as the template or skeleton of this structure. From a more general point of view, and for our discussion of class field theories we need a slightly more general point of view, this upper disc can be described as consisting of a set $\dsys{b}$ enriched with an action $\mu$ of some group $G$ and a family $\dsys{\star} = \lbrace \dsys{\star}(x) \mid x \in \dsys{b} \rbrace$ of subsets $\dsys{\star}(x) \subs \dsys{b}$ for $\star \in \lbrace \ro{r},\ro{i} \rbrace$. Now, the mathematical structure we are interested in can be characterized as consisting of such a ``domain'' $\fr{D} = (\dsys{b},\dsys{r},\dsys{i},G,\mu)$ together with a map $\Phi:\dsys{b} \ra \ro{Ob}(\ca{C})$ into the object class of a category $\ca{C}$ and together with restriction, induction and conjugation morphisms between objects determined by $\fr{D}$ and $\Phi$.

In general the restriction, induction and conjugation morphisms should satisfy some natural compatibility relations, for example the restriction morphisms should satisfy the triviality $\res_{x,x}^\Phi = \id_{\Phi(x)}$ for all $x \in \dsys{b}$, the transitivity $\res_{z,y}^\Phi \circ \res_{y,x}^\Phi = \res_{z,x}^\Phi$ for all $x \in \dsys{b}$, $y \in \dsys{r}(x)$ and $z \in \dsys{r}(y)$, and the equivariance $\con_{g,y}^\Phi \circ \res_{y,x}^\Phi = \res_{^g \! y, ^g \! x}^\Phi \circ \con_{g,x}^\Phi$ for all $x \in \dsys{b}$, $y \in \dsys{r}(x)$ and $g \in G$, where $^g \! x = \mu(g,x)$. Similar relations should hold for the induction and conjugation morphisms. But to make these relations always well-defined we have to put several conditions on the sets $\dsys{\star}(x)$ for $\star \in \lbrace \ro{r},\ro{i} \rbrace$. For example, to make the transitivity of the restriction morphisms always well-defined we have to assume that $\dsys{r}(y) \subs \dsys{r}(x)$ for $x \in \dsys{b}$ and $y \in \dsys{r}(x)$. These conditions lead to the definition of a \textit{RIC-domain} and the collection of the data $\Phi(x)$ together with restriction, induction and conjugation morphisms satisfying the compatibility relations is what we will call a $\ca{C}$-valued \textit{RIC-functor} on $\fr{D}$ or simply \textit{RIC-functor}, where \textit{RIC} is of course an abbreviation for \textit{Restriction-Induction-Conjugation}. We may summarize the aim of this chapter as being the exploration of ``RIC-phenomena'', that is, of mathematical structures possessing a notion of restriction, induction and conjugation, and we may now describe this precisely as the study of RIC-functors.\footnote{There may exist of course other abstract approaches to these phenomena, but our approach is exactly what we will need.} \\

In \ref{sect:ric_functors} a precise definition of the category of RIC-functors is given and some of its properties are investigated. In \ref{sect:gss_fuctors} a special type of RIC-domains, the $G$-subgroup systems, are introduced which describe the structures mentioned at the beginning. Moreover, in this section the notion of a cohomological Mackey functor is introduced which will be of fundamental importance in class field theory. In \ref{sect:disc_gmod} the main example of cohomological Mackey functors, the discrete $G$-modules, are discussed. In \ref{sect:abelianization} another important type of cohomological Mackey functors, the abelianizations, are discussed.

\subsection{Definition and basic properties of RIC-functors} \label{sect:ric_functors}

\begin{para} \wordsym{$\mu' \leq \mu$}
If $\mu:G \times X \ra X$ and $\mu':G' \times X' \ra X'$ are two left group actions on sets, then we write $\mu' \leq \mu$ if and only if $\mu$ can be restricted to an action of $G'$ on $X'$ and this restriction coincides with $\mu'$, that is, $G' \leq G$, $X' \subs X$ and $\mu'(g',x') = \mu(g',x')$ for all $g' \in G'$ and $x' \in X'$.
\end{para}

\begin{defn} \label{para:ric_domain}
A \word{RIC-domain} is a tuple $(\dsys{b},\dsys{r},\dsys{i},G,\mu)$ consisting of:
\begin{compactitem}
\item a non-empty set $\dsys{b}$,
\item two families $\dsys{r} = \lbrace \dsys{r}(x) \mid x \in \dsys{b} \rbrace$ and $\dsys{i} = \lbrace \dsys{i}(x) \mid x \in \dsys{b} \rbrace$ of subsets $\dsys{r}(x), \dsys{i}(x) \subs \dsys{b}$,
\item a group $G$,
\item a left $G$-action $\mu:G \times \dsys{b} \ra \dsys{b}$, $(g,x) \mapsto {^g \! x}$, on $\dsys{b}$,
\end{compactitem}
such that the following conditions are satisfied for all $x \in \dsys{b}$ and $\star \in \lbrace \ro{r},\ro{i} \rbrace$:
\begin{enumerate}[label=(\roman*),noitemsep,nolistsep]
\item \label{item:ric_domain_id} $x \in \dsys{\star}(x)$.
\item \label{item:ric_domain_comp} If $y \in \dsys{\star}(x)$, then $\dsys{\star}(y) \subs \dsys{\star}(x)$.
\item \label{item:ric_domain_conj} $^g \! \dsys{\star}(x) = \dsys{\star}(^g \! x)$ for all $g \in G$, where we define $^g \! \dsys{\star}(x) \dopgleich \lbrace ^g \! y \mid y \in \dsys{\star}(x) \rbrace$.
\end{enumerate}
\end{defn}

\begin{para} \wordsym{$\fr{D}' \leq \fr{D}$}
If $\fr{D}' = (\dsys{b}',\dsys{r}',\dsys{i}',G',\mu')$ and $\fr{D} = (\dsys{b},\dsys{r},\dsys{i},G,\mu)$ are RIC-domains, then we write $\fr{D}' \leq \fr{D}$ if and only if $\mu' \leq \mu$ and $\dsys{\star}'(x) \subs \dsys{\star}(x)$ for all $x \in \dsys{b}'$ and $\star \in \lbrace \tn{r},\tn{i} \rbrace$.
A RIC-domain of the form $(\dsys{b},\dsys{r},\dsys{i},1,\mathbf{1})$, where $1$ is the trivial group and $\mathbf{1}$ denotes the trivial action, is called an \word{RI-domain}.
\end{para}

\begin{para} \label{para:equi_ric_motivation}
In most situations the RIC-domain is obtained and inherits structure from an enveloping set $X$ equipped with a partial order and a monotone action of a group. In the following paragraph we will make this precise. The reader may then already take a look at \ref{para:gss_first_props} for the most important example of RIC-domains.
\end{para}

\begin{para} \label{para:def_of_equi_ric} \label{para:max_equi_ric} \wordsym{$\dsys{b}(x)$}
A \words{RIC-domain in a partially ordered set}{RIC-domain!in poset} $(X,\leq)$ is a RIC-domain $(\dsys{b},\dsys{r},\dsys{i},G,\mu)$ with $\dsys{b} \subs X$ and 
\[
\dsys{r}(x),\dsys{i}(x) \subs \dsys{b}(x) \dopgleich \lbrace y \mid y \in \dsys{b} \tn{ and } y \leq x \rbrace
\]
for all $x \in \dsys{b}$. 
An \word{equiordered set} is a tuple $(X,\leq,G,\mu)$ consisting of a poset $(X,\leq)$ and a left $G$-action $\mu$ on $X$ such that $y \leq x$ implies $^g \! y \leq {^g \! x}$ for all $x,y \in X$ and $g \in G$, where $^g \! x \dopgleich \mu(g,x)$. Let $\ca{X} = (X,\leq,G,\mu)$ be an equiordered set and let $(\dsys{b},\dsys{i},\dsys{r},1,\mathbf{1})$ be an RI-domain in $(X,\leq)$ such that $^g \! \dsys{b} \subs \dsys{b}$ and $^g \! \dsys{\star}(x) = \dsys{\star}(^g \! x)$ for all $g \in G$ and $x \in \dsys{b}$. Then $(\dsys{b},\dsys{i},\dsys{r},G,\mu|_{G \times \dsys{b}})$ is a RIC-domain.\footnote{$\mu|_{G \times \dsys{b}}$ denotes here the simultaneous restriction to $G \times \dsys{b}$ and corestriction to $\dsys{b}$.} Any RIC-domain of this form is called a \words{RIC-domain in the equiordered set}{RIC-domain!in equiordered set} $\ca{X}$ and we usually just write $(\dsys{b},\dsys{r},\dsys{i})$ for such RIC-domains.
Let $\dsys{b} = X$ and for $x \in \dsys{b}$ let $\dsys{r}(x) = \dsys{i}(x) = \dsys{b}(x) = \lbrace y \mid y \in \dsys{b} \tn{ and } y \leq x \rbrace$. Then the reflexivity and transitivity of the order together with the monotonicity of the $G$-action imply that $(\dsys{b},\dsys{r},\dsys{i})$ is the maximal RIC-domain in $\ca{X}$. 
\end{para}

\begin{para} \wordsym{$\fr{D}' \cap \fr{D}$}
Let $\fr{D}' = (\dsys{b}',\dsys{r}',\dsys{i}')$ and $\fr{D} = (\dsys{b},\dsys{r},\dsys{i})$ be two RIC-domains in the equiordered set $\ca{X} = (X,\leq,G,\mu)$ such that $\dsys{b}' \cap \dsys{b} \neq \emptyset$. Then 
\[
\fr{D}' \cap \fr{D} \dopgleich (\dsys{b}'',\dsys{r}'',\dsys{i}'')
\]
with $\dsys{b}'' \dopgleich \dsys{b}' \cap \dsys{b}$ and $\dsys{\star}''(x) \dopgleich \dsys{\star}'(x) \cap \dsys{\star}(x)$ for each $x \in \dsys{b}''$ is also a RIC-domain in $\ca{X}$.
\end{para}

\begin{defn} \label{para:def_of_functor_on_ric} \wordsym{$\se{Fct}(\fr{D},\ca{C})$}
A \word{RIC-functor} is a tuple consisting of
\begin{compactitem}
\item a RIC-domain $\fr{D} = (\dsys{b},\dsys{r},\dsys{i},G,\mu)$, where $\mu(g,x)$ is denoted by $^g \! x$ for all $g \in G$ and $x \in \dsys{b}$,
\item a category $\ca{C}$,
\item a map $\Phi: \dsys{b} \ra \ro{Ob}(\ca{C})$,
\item a $\ca{C}$-morphism $\res_{y,x}^\Phi:\Phi(x) \ra \Phi(y)$ for each $x \in \dsys{b}$ and $y \in \dsys{r}(x)$, called \word{restriction},
\item a $\ca{C}$-morphism $\ind_{x,y}^\Phi:\Phi(y) \ra \Phi(x)$ for each $x \in \dsys{b}$ and $y \in \dsys{i}(x)$, called \word{induction},
\item a $\ca{C}$-morphism $\con_{g,x}^\Phi:\Phi(x) \ra \Phi(^g \! x)$ for each $x \in \dsys{b}$ and $g \in G$, called \word{conjugation},
\end{compactitem}
such that the following properties are satisfied:
\begin{enumerate}[label=(\roman*),noitemsep,nolistsep]
\item (Triviality)
\[
\res_{x,x}^\Phi = \ind_{x,x}^\Phi = \con_{1,x}^\Phi = \id_{\Phi(x)}
\]
for all $x \in \dsys{b}$.\footnote{These morphisms are defined since $x \in \dsys{\star}(x)$ by \ref{para:ric_domain}\ref{item:ric_domain_id}.}
\item (Transitivity)
\[
\res_{z,y}^\Phi \circ \res_{y,x}^\Phi = \res_{z,x}^\Phi 
\]
for all $x \in \dsys{b}$, $y \in \dsys{r}(x)$ and $z \in \dsys{r}(y)$;
\[
\ind_{x,y}^\Phi \circ \ind_{y,z}^\Phi = \ind_{x,z}^\Phi 
\]
for all $x \in \dsys{b}$, $y \in \dsys{i}(x)$ and $z \in \dsys{i}(y)$; and
\[
\con_{g',^g \! x}^\Phi \circ \con_{g,x}^\Phi = \con_{g'g,x}^\Phi
\]
for all $x \in \dsys{b}$ and all $g,g' \in G$.\footnote{As $y \in \dsys{\star}(x)$, it follows from \ref{para:ric_domain}\ref{item:ric_domain_comp} that $\dsys{\star}(y) \subs \dsys{\star}(x)$ and consequently $z \in \dsys{\star}(x)$. Hence, both $\res_{z,x}^\Phi$ and $\ind_{x,z}^\Phi$ are defined.}

\item (Equivariance)
\[
\con_{g,y}^\Phi \circ \res_{y,x}^\Phi = \res_{^g \! y, ^g \! x}^\Phi \circ \con_{g,x}^\Phi
\]
for all $x \in \dsys{b}$, $y \in \dsys{r}(x)$ and $g \in G$; and
\[
\con_{g,x}^\Phi \circ \ind_{x,y}^\Phi = \ind_{^g \! x, ^g \! y}^\Phi \circ \con_{g,y}^\Phi
\]
for all $x \in \dsys{b}$, $y \in \dsys{i}(x)$ and $g \in G$.\footnote{This is well-defined by \ref{para:ric_domain}\ref{item:ric_domain_conj}.} \\

\end{enumerate}

To simplify notations, we denote the tuple consisting of the map $\Phi$ and all the restriction, induction and conjugation morphisms again by $\Phi$ so that the RIC-functor becomes a triple $(\fr{D},\Phi,\ca{C})$ which we also denote by $\Phi:\fr{D} \ra \ca{C}$. We call $\fr{D}$ the \words{domain}{RIC-functor!domain} and $\ca{C}$ the \words{codomain}{RIC-functor!codomain} of the RIC-functor. A RIC-functor with domain a RI-domain will be called a \word{RI-functor}. \\

A \words{morphism}{RIC-functor!morphism} $\varphi$ between two RIC-functors $\Phi,\Psi: \fr{D} \ra \ca{C}$, denoted by $\varphi:\Phi \ra \Psi$, is a family $\varphi = \lbrace \varphi_x \mid x \in \dsys{b} \rbrace$ of $\ca{C}$-morphisms $\varphi_x:\Phi(x) \ra \Psi(x)$ such that the following diagrams commute for all $x \in \dsys{b}$, $y \in \dsys{r}(x)$, $z \in \dsys{i}(x)$ and $g \in G$:
\[ \hspace{-5pt}
\xymatrix{
\Phi(x) \ar[rr]^{\varphi_x} \ar[d]_{\res_{y,x}^\Phi} & & \Psi(x) \ar[d]^{\res_{y,x}^\Psi} & & \Phi(x) \ar[rr]^{\varphi_x} & & \Psi(x) & & \Phi(x) \ar[rr]^{\varphi_x} \ar[d]_{\con_{g,x}^\Phi} & & \Psi(x) \ar[d]^{\con_{g,x}^\Psi} \\
\Phi(y) \ar[rr]_{\varphi_y} & & \Psi(y) & & \Phi(z) \ar[rr]_{\varphi_z} \ar[u]^{\ind_{x,z}^\Phi} & & \Psi(z) \ar[u]_{\ind_{x,z}^\Psi} & & \Phi(^g \! x) \ar[rr]_{\varphi_{^g \! x}} & & \Psi(^g \! x)
}
\]

The composition of two morphisms $\varphi:\Phi \ra \Psi, \psi:\Psi \ra \Theta$ is defined by $\psi \circ \varphi = \lbrace \psi_x \circ \varphi_x \mid x \in \dsys{b} \rbrace$ which is a morphism $\psi \circ \varphi:\Phi \ra \Theta$. With respect to this composition the morphism $\id_{\Phi}:\Phi \ra \Phi$ defined by the family $\lbrace \id_{\Phi(x)} \mid x \in \dsys{b} \rbrace$ is an identity. The set of all RIC-functors with domain $\fr{D}$ and codomain $\ca{C}$ together with morphisms, composition and identity as defined above forms a category which is denoted by $\se{Fct}(\fr{D},\ca{C})$.
\end{defn}

\begin{ass}
For the rest of this section we fix a RIC-domain $\fr{D} = (\dsys{b},\dsys{r},\dsys{i},G,\mu)$ and a category $\ca{C}$. We write $^g \! x \dopgleich \mu(g,x)$ for all $g \in G$ and $x \in \dsys{b}$.
\end{ass}

\begin{para}
If $\Phi:\fr{D} \ra \ca{C}$ is a RIC-functor, then it follows immediately from the triviality and transitivity of the conjugation morphisms that $\con_{g,x}^\Phi$ is an isomorphism with inverse $\con_{g^{-1},^g \! x}^\Phi$ for each $x \in \dsys{b}$ and $g \in G$. 
\end{para}

\begin{para} \label{para:ric_functors_dom_codom_change} \wordsym{$(-)^1$}
If $\fr{D}'$ is a RIC-domain with $\fr{D}' \leq \fr{D}$, then it is easy to verify that we get by restriction a functor $(-)|_{\fr{D}'}:\se{Fct}(\fr{D},\ca{C}) \ra \se{Fct}(\fr{D}',\ca{C})$. 
Let $\ca{C}'$ be another category and let $L:\ca{C}' \ra \ca{C}$ be a functor. If $\Phi \in \se{Fct}(\fr{D},\ca{C}')$, then it is easy to verify that the following data define a functor $L \circ \Phi \in \se{Fct}(\fr{D},\ca{C})$:
\begin{compactitem}
\item $(L \circ \Phi)(x) \dopgleich L( \Phi(x))$ for all $x \in \dsys{b}$.
\item $\res_{y,x}^{L \circ \Phi} \dopgleich L( \res_{y,x}^\Phi)$ for all $x \in \dsys{b}$ and $y \in \dsys{r}(x)$.
\item $\ind_{x,y}^{L \circ \Phi} \dopgleich L( \ind_{x,y}^\Phi)$ for all $x \in \dsys{b}$ and $y \in \dsys{i}(x)$.
\item $\con_{g,x}^{L \circ \Phi} \dopgleich L( \con_{g,x}^\Phi)$ for all $x \in \dsys{b}$ and $g \in G$.
\end{compactitem}
\end{para}

\begin{para}
It is evident that if $A \in \ca{C}$, then the following data define a functor $\Phi_A \in \se{Fct}(\fr{D},\ca{C})$:
\begin{compactitem}
\item $\Phi_A:\dsys{b} \ra \ca{C}$ is the constant map $x \mapsto A$.
\item $\res_{y,x}^{\Phi_A} \dopgleich \id_{A}$ for each $x \in \dsys{b}$ and $y \in \dsys{r}(x)$.
\item $\ind_{x,y}^{\Phi_A} \dopgleich \id_{A}$ for each $x \in \dsys{b}$ and $y \in \dsys{i}(x)$.
\item $\con_{g,x}^{\Phi_A} \dopgleich \id_{A}$ for each $x \in \dsys{b}$ and $g \in G$.
\end{compactitem}

The functor $\Phi_A$ is called the \word{constant functor} with value $A$ and is also simply denoted by $A$.\footnote{The notation $\Phi_A$ or $A$ for this functor is slightly imprecise as we are suppressing the domain and codomain. But this should in general be clear from the context and therefore we stick to this simple notation.}
\end{para}

\begin{para} \label{para:functor_iso_on_values}
It is easy to see that a morphism $\varphi$ in $\se{Fct}(\fr{D},\ca{C})$ is an isomorphism if and only if $\varphi_x$ is an isomorphism in $\ca{C}$ for each $x \in \dsys{b}$.
\end{para}

\begin{prop} \label{para:functor_limit_inheritance}
The following holds:
\begin{enumerate}[label=(\roman*),noitemsep,nolistsep]
\item Let $\ca{I}$ be a category and let $D: \ca{I} \ra \se{Fct}(\fr{D},\ca{C})$ be a functor. For $x \in \dsys{b}$ let $D^x$ be the following functor:
\[
\begin{array}{rcl}
D^x: \ca{I} & \lra & \ca{C} \\
i & \longmapsto & D(i)(x) \\
i \overset{f}{\lra} j & \longmapsto & D(i)(x) \overset{D(f)_x}{\lra} D(j)(x).
\end{array}
\]
If $D^x$ has a limit in $\ca{C}$ for all $x \in \dsys{b}$, then $D$ has a limit in $\se{Fct}(\fr{D},\ca{C})$ and
\[
(\ro{lim} \ D)(x) = D^x
\]
for all $x \in \dsys{b}$. In particular, if $\ca{C}$ has $\ca{I}$-limits, then $\se{Fct}(\fr{D},\ca{C})$ also has $\ca{I}$-limits.

\item The categorical dual of the above statement also holds.
\end{enumerate}
\end{prop}

\begin{proof} 
This is a standard proof in category theory.
\end{proof}

\begin{prop} \label{para:functor_cat_abelian}
Let $\ca{A}$ be a preadditive category. The following holds:
\begin{enumerate}[label=(\roman*),noitemsep,nolistsep]
\item $\se{Fct}(\fr{D},\ca{A})$ is canonically preadditive.
\item \label{item:mono_local_global} If $\ca{A}$ has kernels, then a morphism $\varphi: \Phi \ra \Psi$ in $\se{Fct}(\fr{D},\ca{A})$ is a monomorphism if and only if $\varphi_x:\Phi(x) \ra \Psi(x)$ is a monomorphism in $\ca{A}$ for each $x \in \dsys{b}$.
\item The categorical dual of \ref{item:mono_local_global} also holds.
\item If $\ca{A}$ is abelian, then $\se{Fct}(\fr{D},\ca{A})$ is also abelian.
\end{enumerate}
\end{prop}

\begin{proof} \hfill

\begin{asparaenum}[(i)]
\item This is straightforward.

\item Let $\varphi:\Phi \ra \Psi$ be a morphism in $\se{Fct}(\fr{D},\ca{A})$. Since $\ca{A}$ has kernels, the category $\se{Fct}(\fr{D},\ca{A})$ also has kernels by \ref{para:functor_limit_inheritance}. Let $k:K \ra \Phi$ be the kernel of $\varphi$. As $\se{Fct}(\fr{D},\ca{A})$ is preadditive, an application of \ref{para:mono_kernel_relation} show that $\varphi$ is a monomorphism if and only if $k = 0$. It follows from \ref{para:functor_limit_inheritance} that $k_x:K(x) \ra \Phi(x)$ is the kernel of $\varphi_x$ for each $x \in \dsys{b}$ and therefore $k = 0$ if and only if $k_x = 0$ for all $x \in \dsys{b}$. Again by \ref{para:mono_kernel_relation}, we have $k_x = 0$ if and only if $\varphi_x$ is a monomorphism in $\ca{A}$. Hence, $\varphi$ is a monomorphism if and only if $\varphi_x$ is a monomorphism for each $x \in \dsys{b}$.

\item This is evident. 

\item By \ref{para:functor_limit_inheritance} the category $\se{Fct}(\fr{D},\ca{A})$ has a zero object and has pullbacks and pushouts so it remains to show that $\se{Fct}(\fr{D},\ca{A})$ is normal, that is, every monomorphism is a kernel and every epimorphism is a cokernel. So, let $\varphi: \Phi \ra \Psi$ be a monomorphism in $\se{Fct}(\fr{D},\ca{A})$. Let $c:\Psi \ra \coker(\varphi)$ be the cokernel of $\varphi$. An application of \ref{para:functor_limit_inheritance} shows that $c_x$ is the cokernel of $\varphi_x$ for each $x \in \dsys{b}$. Since in an abelian category every monomorphism is the kernel of its cokernel, this implies that $\varphi_x = \ker(c_x)$. But then, again by \ref{para:functor_limit_inheritance}, we already have $\varphi = \ker(c)$. This shows that $\varphi$ is a kernel. In the same way one can prove that every epimorphism in $\se{Fct}(\fr{D},\ca{A})$ is a cokernel.
\end{asparaenum}
\vspace{-\baselineskip}
\end{proof}

\begin{para}
In the following paragraphs we will introduce subfunctors, quotient functors and certain Hom-functors. To make the discussion straightforward and as we do not need these constructions in a more general setting, we will restrict to $\se{TAb}$-valued functors.
\end{para}

\begin{defn} \label{para:def_of_subfunctor}
Let $\Phi \in \se{Fct}(\fr{D},\se{TAb})$. A functor $\Phi' \in \se{Fct}(\fr{D},\se{TAb})$ is said to be a (\words{closed}{subfunctor!closed}) \word{subfunctor} of $\Phi$, denoted by $\Phi' \leq \Phi$ ($\Phi' \leq_{\ro{c}} \Phi$), if $\Phi'(x)$ is a (closed) subgroup of $\Phi(x)$ for each $x \in \dsys{b}$, and for each $x \in \dsys{b}$, $y \in \dsys{r}(x)$, $z \in \dsys{i}(x)$ and $g \in G$ the diagrams
\[
\xymatrix{
\Phi'(x) \ \ar@{>->}[r] \ar[d]_{\res_{y,x}^{\Phi'}} & \Phi(x) \ar[d]^{\res_{y,x}^{\Phi}} & & \Phi'(x) \ \ar@{>->}[r]  & \Phi(x) && \Phi'(x) \ \ar@{>->}[r] \ar[d]_{\con_{g,x}^{\Phi'}} & \Phi(x) \ar[d]^{\con_{g,x}^{\Phi}} \\
\Phi'(y) \ \ar@{>->}[r] & \Phi(y) & & \Phi'(z) \ \ar@{>->}[r] \ar[u]^{\ind_{z,x}^{\Phi'}} & \Phi(z) \ar[u]_{\ind_{z,x}^{\Phi}} & & \Phi'(y) \ \ar@{>->}[r] & \Phi(y)
}
\]
commute, where the horizontal morphisms are the embeddings.
\end{defn}

\begin{para} \label{para:subfunctor_relations}
If $\Phi'$ is a subfunctor of $\Phi \in \se{Fct}(\fr{D},\se{TAb})$, then the following assertions hold:
\[
\begin{array}{ll}
\Phi'(x) \leq \Phi(x) & \tn{for all } x \in \dsys{b} \\
\res_{y,x}^\Phi( \Phi'(x)) \subs \Phi'(y) & \tn{for all } x \in \dsys{b} \tn{ and } y \in \dsys{r}(x) \\
\ind_{x,y}^\Phi( \Phi'(y)) \subs \Phi'(x) & \tn{for all } x \in \dsys{b} \tn{ and } y \in \dsys{i}(x) \\
\con_{g,x}^\Phi( \Phi'(x)) \subs \Phi'(^g \! x) & \tn{for all } x \in \dsys{b} \tn{ and } g \in G.
\end{array}
\]
Conversely, if $\Phi':\dsys{b} \ra \se{TAb}$ is a map such that the above relations are satisfied, then there is a canonical way to make $\Phi'$ into a subfunctor of $\Phi$. Similar statements hold for closed subfunctors.
\end{para}

\begin{prop}
Let $\Phi \in \se{Fct}(\fr{D},\se{TAb})$ and let $\Phi' \leq \Phi$. Then there exists a unique functor $\Phi/\Phi' \in \se{Fct}(\fr{D},\se{TAb})$, called the \words{quotient of $\Phi$ by $\Phi'$}{functor!quotient}, such that $(\Phi/\Phi')(x) = \Phi(x)/\Phi'(x)$ for each $x \in \dsys{b}$,  and for each $x \in \dsys{b}$, $y \in \dsys{r}(x)$, $z \in \dsys{i}(x)$ and $g \in G$ the diagrams
\[
\xymatrix{
\Phi(x) \ar@{->>}[r] \ar[d]_{\res_{y,x}^\Phi} & \Phi(x)/\Phi'(x) \ar[d]^{\res_{y,x}^{\Phi/\Phi'}} & & \Phi(x) \ar@{->>}[r] & \Phi(x)/\Phi'(x) & & \Phi(x) \ar@{->>}[r] \ar[d]_{\con_{g,x}^\Phi} & \Phi(x)/\Phi'(x) \ar[d]^{\con_{g,x}^{\Phi/\Phi'}} \\ 
\Phi(y) \ar@{->>}[r] & \Phi(y)/\Phi'(y) & & \Phi(z) \ar@{->>}[r] \ar[u]^{\ind_{x,z}^\Phi} & \Phi(z)/\Phi'(z) \ar[u]_{\ind_{x,z}^{\Phi/\Phi'}} & & \Phi(^g \! x) \ar@{->>}[r] & \Phi(^g \! x)/\Phi'(^g \! x)
}
\]
commute, where the horizontal morphisms are the quotient morphisms.
\end{prop}

\begin{proof}
Due to the relations in \ref{para:subfunctor_relations}, the morphism $\res_{y,x}^\Phi$ induces a morphism 
\[
\res_{y,x}^{\Phi/\Phi'}: \Phi(x)/\Phi'(x) \ra \Phi(y)/\Phi'(y)
\]
making the above diagram commutative (note that $\res_{y,x}^{\Phi/\Phi'}$ is indeed continuous by \ref{para:morphism_quotient_induced}). In the same way the induction and conjugation morphisms of $\Phi/\Phi'$ are defined and it is evident that these satisfy all the relations making $\Phi/\Phi'$ into a functor. The uniqueness is obvious.
\end{proof}

\begin{para}
We will now discuss a construction which will be important for our most general point of view of class field theories. Suppose we are given two functors $\Delta,\Phi \in \se{Fct}(\fr{D},\se{TAb})$. An application of \ref{para:hom_co_top_grp} shows that $\hom_{\se{TAb}}(\Delta(x),\Phi(x))$ is a topological abelian group with respect to the compact-open topology for each $x \in \dsys{b}$ and now it is a natural question whether there is a canonical way to make the map $\hom(\Delta,\Phi):\dsys{b} \ra \se{TAb}$, $x \mapsto \hom_{\se{TAb}}(\Delta(x),\Phi(x))$, into a functor $\hom(\Delta,\Phi) \in \se{Fct}(\fr{D},\se{TAb})$.\footnote{Confer \ref{sect:compact_open} for a discussion of the compact-open topology.} In the following paragraph we will see that this is indeed possible under certain conditions on the functors. To discuss this in general, we introduce the following notion: a \word{dualizing functor} on $\fr{D}$ is a functor $D \in \se{Fct}(\fr{D},\se{TAb}^{\ro{lc}})$ whose restriction and induction morphisms are isomorphisms. The main examples of dualizing functors are constant functors with value a locally compact abelian group.
\end{para}

\begin{prop} \label{para:hom_of_ric}
Let $D$ be a dualizing functor on $\fr{D}$ and let $\Phi \in \se{Fct}(\fr{D},\se{TAb}^{\ro{lc}})$. The following data define a functor $\hom(D,\Phi) \in \se{Fct}(\fr{D},\se{TAb})$:
\begin{compactitem}
\item $\hom(D,\Phi): \dsys{b} \ra \se{TAb}$ is given by $x \mapsto \hom_{\se{TAb}}(D(x),\Phi(x))$, where $\hom_{\se{TAb}}(D(x),\Phi(x))$ is considered as a topological group with respect to the compact-open topology.
\item $\res_{y,x}^{\hom(D,\Phi)}: \hom_{\se{TAb}}(D(x),\Phi(x)) \ra \hom_{\se{TAb}}(D(y),\Phi(y))$ is given by
\[
\chi \longmapsto \res_{y,x}^\Phi \circ \chi \circ (\res_{y,x}^D)^{-1}
\]
for each $x \in \dsys{b}$ and $y \in \dsys{r}(x)$.
\item $\ind_{x,y}^{\hom(D,\Phi)}:\hom_{\se{TAb}}(D(y),\Phi(y)) \ra \hom_{\se{TAb}}(D(x),\Phi(x))$ is given by
\[
\chi \longmapsto \ind_{x,y}^\Phi \circ \chi \circ (\ind_{x,y}^D)^{-1} 
\]
for each $x \in \dsys{b}$ and $y \in \dsys{i}(x)$.
\item $\con_{g,x}^{\hom(D,\Phi)}: \hom_{\se{TAb}}(D(x),\Phi(x)) \ra \hom_{\se{TAb}}(D(^g \! x),\Phi(^g \! x))$ is given by
\[
\chi \longmapsto \con_{g,x}^\Phi \circ \chi \circ (\con_{g,x}^D)^{-1} 
\]
for each $x \in \dsys{b}$.\footnote{Note that $(\con_{g,x}^D)^{-1} = \con_{g^{-1},x}^D$.}
\end{compactitem}
\end{prop}

\begin{proof}
First note that the restriction, induction and conjugation morphisms are well-defined morphisms of abstract groups since the restriction, induction and conjugation morphisms of $D$ and $\Phi$ are continuous by assumption. Moreover, according to \ref{para:hom_co_comp_cont} these morphisms are indeed continuous with respect to the compact-open topology. The proof is now straightforward.
\end{proof}

\begin{prop} \label{para:hom_functor_props}
The following holds:
\begin{compactenum}[(i)]
\item If $D$ is a dualizing functor on $\fr{D}$, then the maps
\[
\begin{array}{rcl}
\hom(D,-): \se{Fct}(\fr{D},\se{TAb}^{\ro{lc}}) & \lra & \se{Fct}(\fr{D},\se{TAb}) \\
\Phi & \longmapsto & \hom(D,\Phi) \\
\varphi:\Phi \ra \Psi & \longmapsto & \lbrace \hom_{\se{TAb}}(D(x), \varphi_{x}) \mid x \in \dsys{b} \rbrace
\end{array}
\]
define a functor.
\item $\hom(\ZZ,\Phi)$ is canonically isomorphic to $\Phi$ for any $\Phi \in \se{Fct}(\fr{D},\se{TAb}^{\ro{lc}})$.\footnote{Here we consider the constant RIC-functor $\ZZ$ on $\fr{D}$ and equip $\ZZ$ with the discrete topology.}

\end{compactenum}
\end{prop}

\begin{proof} \hfill

\begin{asparaenum}[(i)]
\item First, we have to verify that $\hom(D,\varphi) = \lbrace \hom_{\se{TAb}}(D(x), \varphi_{x}) \mid x \in \dsys{b} \rbrace$ defines a morphism between $\hom(D,\Phi)$ and $\hom(D,\Psi)$. Since $\Phi(x)$ is locally compact, it follows from \ref{para:hom_co_comp_cont} that 
\[
\hom_{\se{TAb}}(D(x),\varphi_{x}):\hom(D,\Phi)(x) \ra \hom(D,\Psi)(x)
\]
is continuous and therefore it is a morphism in $\se{TAb}$. It is easy to see that 
\[
\hom(D,\varphi): \hom(D,\Phi) \ra \hom(D,\Psi)
\]
is a morphism in $\se{Fct}(\fr{D},\se{TAb})$ and now it is obvious that $\hom(D,-)$ is a functor.

\item For $x \in \dsys{b}$ the map 
\[
\begin{array}{rcl}
\varphi_x: \hom(\ZZ,\Phi)(x) = \hom_{\se{TAb}}(\ZZ,\Phi(x)) & \lra & \Phi(x) \\
\chi & \longmapsto & \chi(1).
\end{array}
\]
is an isomorphism in $\se{TAb}$ by \ref{para:hom_co_zz}. It is easy to verify that $\varphi = \lbrace \varphi_{x} \mid x \in \dsys{b} \rbrace$ is compatible with the restriction, induction and conjugation morphisms and therefore $\varphi:\hom(\ZZ,\Phi) \ra \Phi$ is a morphism in $\se{Fct}(\fr{D},\se{TAb})$. As $\varphi_{x}$ is an isomorphism for all $x \in \dsys{b}$, it follows from \ref{para:functor_iso_on_values} that $\varphi$ is also an isomorphism.
\end{asparaenum} 
\end{proof}

\subsection{$G$-subgroup systems} \label{sect:gss_fuctors}

\begin{para} 
The motivating setting for the introduction of RIC-functors given at the beginning of this chapter is now revealed by a particular kind of RIC-domains, the $G$-subgroup systems.
\end{para}

\begin{defn} \label{para:gss_first_def}
A \word{$G$-subgroup system} for a group $G$ is a RIC-domain in the equiordered set 
\[
(\ro{Grp}(G), \subs, G, \mu),
\]
where $\ro{Grp}(G) \dopgleich \lbrace H \mid H \leq G \rbrace$ and $\mu$ is the conjugation action on $\ro{Grp}(G)$. We denote the maximal RIC-domain in this equiordered set again by $\ro{Grp}(G)$.
\end{defn}

\begin{para}
As the notion of a $G$-subgroup system is of great importance, we will make its definition explicit. A $G$-subgroup system is a triple $(\ssys{b},\ssys{r},\ssys{i})$ consisting of:
\begin{compactitem}
\item a non-empty set $\ssys{b}$ of subgroups of $G$ which is closed under conjugation,
\item families $\lbrace \ssys{\star}(H) \mid H \in \ssys{b} \rbrace$ of subsets $\ssys{\star}(H) \subs \ssys{b}(H) = \lbrace I \in \ssys{b} \mid I \leq H \rbrace$ for all $H \in \ssys{b}$,
\end{compactitem}
such that the following conditions are satisfied for all $H \in \ssys{b}$:
\begin{enumerate}[label=(\roman*),noitemsep,nolistsep]
\item $H \in \ssys{\star}(H)$.
\item If $I \in \ssys{\star}(H)$, then $\ssys{\star}(I) \subs \ssys{\star}(H)$.
\item $^g \! \ssys{\star}(H) = \ssys{\star}(^g \! H)$ for all $g \in G$.
\end{enumerate}

The reason we did not immediately define a $G$-subgroup system in this (admittedly simpler) way is that we want to view it as a special type of RIC-domains because later another type of RIC-domains is defined which will have the same definition as in \ref{para:gss_first_def} but just in another equiordered set and so we can view both as being essentially the same.
\end{para}

\begin{para} \label{para:gss_first_props}
For later applications, we introduce the following list of properties a $G$-subgroup system $\fr{S}$ might satisfy:
\begin{enumerate}[label=(\roman*),noitemsep,nolistsep]
\item \label{item:gss_first_props_r_finite} If $H \in \ssys{b}$ and $I \in \ssys{r}(H)$, then $\lbrack H:I \rbrack < \infty$.
\item \label{item:gss_first_props_i_finite} If $H \in \ssys{b}$ and $I \in \ssys{i}(H)$, then $\lbrack H:I \rbrack < \infty$.
\item \label{item:gss_first_props_dcoset_finite} If $H \in \ssys{b}$, $I \in \ssys{r}(H)$ and $J \in \ssys{i}(H)$, then $|I \mybackslash H / J| < \infty$.
\item \label{item:gss_first_props_mformula_compat} If $H \in \ssys{b}$, $I \in \ssys{r}(H)$ and $J \in \ssys{i}(H)$, then $I \cap J \in \ssys{r}(J)$ and $I \cap J \in \ssys{i}(I)$. \\
\end{enumerate}

If \ref{item:gss_first_props_r_finite} is satisfied, then $\fr{S}$ is called \words{R-finite}{subgroup system!R-finite}, if \ref{item:gss_first_props_i_finite} is satisfied, then $\fr{S}$ is called \words{I-finite}{subgroup system!I-finite} and if $\fr{S}$ is both R-finite and I-finite, then it is called \words{RI-finite}{subgroup system!RI-finite}. It follows from \ref{thm:basic_prop_of_double_cos} that if $\fr{S}$ is R-finite or I-finite, then \ref{item:gss_first_props_dcoset_finite} is already satisfied.

\end{para}

\begin{prop} \wordsym{$\ro{Grp}(G)^{\tn{t}}$} \wordsym{$\ro{Grp}(G)^{\tn{r-f}}$} \wordsym{$\ro{Grp}(G)^{\tn{i-f}}$} \wordsym{$\ro{Grp}(G)^{\tn{ri-f}}$} \wordsym{$\ro{Grp}(G)^{\tn{f}}$}
Let $G$ be a topological group. The following holds:
\begin{enumerate}[label=(\roman*),noitemsep,nolistsep]
\item Let $\ssys{b}$ be the set of closed subgroups of $G$. For $H \in \ssys{b}$ and $\star \in \lbrace \ro{r},\ro{i} \rbrace$ let $\ssys{\star}(H) = \ssys{b}(H)$. Then $(\ssys{b},\ssys{r},\ssys{i})$ is a $G$-subgroup system denoted by $\ro{Grp}(G)^{\tn{t}}$. 
\item Let $\ssys{b}$ be the set of closed subgroups of $G$. For $H \in \ssys{b}$ let 
\[
\ssys{r}(H) = \lbrace I \mid I \in \ssys{b}(H) \tn{ and } \lbrack H:I \rbrack < \infty \rbrace
\]
and $\ssys{i}(H) = \ssys{b}(H)$. Then $(\ssys{b},\ssys{r},\ssys{i})$ is an R-finite $G$-subgroup system denoted by $\ro{Grp}(G)^{\tn{r-f}}$. Similarly, the I-finite $G$-subgroup system $\ro{Grp}(G)^{\tn{i-f}}$ and the RI-finite $G$-subgroup system $\ro{Grp}(G)^{\tn{ri-f}}$ are defined.

\item Let $\ssys{b}$ be the set of closed subgroups of finite index of $G$. For $H \in \ssys{b}$ let $\ssys{r}(H) = \ssys{i}(H) = \ssys{b}(H)$. Then $(\ssys{b},\ssys{r},\ssys{i})$ is an RI-finite $G$-subgroup system denoted by $\ro{Grp}(G)^{\tn{f}}$.
\end{enumerate}
\end{prop}

\begin{proof}
This is easy to verify.
\end{proof}

\begin{para}
The above $G$-subgroup systems are also defined for an abstract group $G$ by considering it with the discrete topology. Note that $\ro{Grp}(G)^{\tn{t}}$ is the maximal $G$-subgroup system consisting of closed subgroups. Similar statements hold for the other $G$-subgroup systems defined above. If $G$ is a quasi-compact group, then $\ro{Grp}(G)^{\tn{f}}_{\ro{b}}$ consists exactly of the open subgroups of $G$ and for $H \in \ro{Grp}(G)^{\tn{f}}_{\ro{b}}$ both $\ro{Grp}(G)^{\tn{f}}_{\ro{r}}$ and $\ro{Grp}(G)^{\tn{f}}_{\ro{i}}$ consist exactly of the open subgroups of $H$ (confer also \ref{prop:prop_of_top_grps}.) This is in particular the case if $G$ is profinite. 

\end{para}

\begin{para} \label{para:ss_functors_several_relations}
There are several important RIC-functors with a $G$-subgroup system as domain and whose restriction, induction and conjugation morphisms satisfy additional relations. In the following paragraphs we will discuss two such relations, namely the Mackey formula and the cohomologicality\footnote{It seems that there is no better word. This word is at least also used in \cite{Yos83_On-G-functors-II:-Hecke_0} and \cite{BleBol04_Cohomological-Mackey_0}.} relation. \end{para}

\begin{para}
A Mackey functor is a RIC-functor with domain being a certain nice enough $G$-subgroup system
and codomain being a preadditive category such that both a stability condition and the \word{Mackey formula} are satisfied. The stability condition ensures that the conjugations $\con_{g,H}^\Phi$ are trivial not only for $g = 1$ but also for all $g \in H$. The Mackey formula is a relation involving the restriction, induction and conjugation morphisms. Although this relation might look strange and complicated at first, it appears surprisingly often in mathematics and it will be of fundamental importance in generalizing certain concepts in class field theory. The origin of this relation is group representation theory and it was formulated in this form by \name{George Mackey} in \cite{Mac51_On-Induced-Representations_0}.

In order to make the Mackey formula always well-defined we will have to put some additional conditions on the $G$-subgroup system leading to the definition of a $G$-Mackey system.
\end{para}

\begin{defn} \wordsym{$\se{Stab}(\fr{S},\ca{C})$}
Let $\fr{S}$ be a $G$-subgroup system. A RIC-functor $\Phi:\fr{S} \ra \ca{C}$ is called \words{stable}{RIC-functor!stable} if $\con_{h,H}^\Phi = \id_{\Phi(H)}$ for all $H \in \ssys{b}$ and $h \in H$. The full subcategory of $\se{Fct}(\fr{S},\ca{C})$ consisting of stable RIC-functors is denoted by $\se{Stab}(\fr{S},\ca{C})$.
\end{defn}

\begin{defn}
Let $G$ be a group. A \words{$G$-Mackey system}{Mackey system} is a $G$-subgroup system satisfying \ref{para:gss_first_props}\ref{item:gss_first_props_dcoset_finite} and \ref{para:gss_first_props}\ref{item:gss_first_props_mformula_compat}. 
\end{defn}

\begin{prop} \label{para:ssys_main_examples} 
If $G$ is a topological group, then the $G$-subgroup systems $\ro{Grp}(G)^{\star}$ are $G$-Mackey systems for $\star \in \lbrace \tn{r-f}, \tn{i-f}, \tn{ri-f}, \tn{f} \rbrace$.
\end{prop}

\begin{proof}
It is obvious that all these subgroup systems satisfy \ref{para:gss_first_props}\ref{item:gss_first_props_dcoset_finite} and so it remains to show that \ref{para:gss_first_props}\ref{item:gss_first_props_mformula_compat} is satisfied. We proof this for $\fr{S} = \ro{Grp}(G)^{\tn{r-f}}$, the other cases are similar. Let $H \in \ssys{b}$, $I \in \ssys{r}(H)$ and $J \in \ssys{i}(H)$. It is obvious that $I \cap J$ is a closed subgroup of $I$ and consequently $I \cap J \in \ssys{i}(I)$. It is also obvious that $I \cap J$ is a closed subgroup of $J$ and since the canonical bijection between the left coset spaces $J/I \cap J$ and $JI/I$ yields the inequality $| J/I \cap J | = | JI/I | \leq | H/I | = \lbrack H:I\rbrack < \infty$, we conclude that $I \cap J \in \ssys{r}(J)$.
\end{proof}

\begin{defn} \wordsym{$\se{Mack}(\fr{M},\ca{A})$}
A RIC-functor $\Phi:\fr{M} \ra \ca{A}$ with $\fr{M}$ a $G$-Mackey system for a group $G$ and $\ca{A}$ a preadditive category is called a \word{Mackey functor} if $\Phi$ is stable and for each $H \in \msys{b}$, $I \in \msys{r}(H)$, $J \in \msys{i}(H)$ and every complete set $R$ of representatives of $I \mybackslash H / J$ the relation
\[
\res_{I,H}^\Phi \circ \ind_{H,J}^\Phi = \sum_{h \in R} \ind_{I,I \cap {^h \! J}}^\Phi \circ \con_{h,I^h \cap J}^\Phi \circ \res_{I^h \cap J,J}^\Phi
\]
holds in $\hom_{\ca{A}}(\Phi(J),\Phi(I))$. This relation is called the \word{Mackey formula}. We denote the full subcategory of $\se{Fct}(\fr{M},\ca{A})$ consisting of Mackey functors by $\se{Mack}(\fr{M},\ca{A})$.
\end{defn}

\begin{para}
To see that the Mackey formula is well-defined, first note that $|R| < \infty$ by \ref{para:gss_first_props}\ref{item:gss_first_props_dcoset_finite} and consequently the sum is finite. Since $I \in \msys{r}(H)$, we have $I^h \in h^{-1} \msys{r}(H)h = \msys{r}(H^h) = \msys{r}(H)$ by \ref{para:ric_domain}\ref{item:ric_domain_conj}. As $J \in \msys{i}(H)$, it follows from \ref{para:gss_first_props}\ref{item:gss_first_props_mformula_compat} that $I^h \cap J \in \msys{r}(J)$ and therefore $\res_{I^h \cap J,J}^\Phi$ is defined. Similarly, $^h \! J \in \msys{i}(H)$ which implies $I \cap {^h \! J} \in \msys{i}(I)$ and consequently $\ind_{I,I \cap {^h \! J}}^\Phi$ is defined.
\end{para}

\begin{para}
If $\Phi \in \se{Stab}(\fr{M},\ca{A})$, then the sum on the right hand side of the Mackey formula does not depend on the choice of the complete set $R$ of representatives of $I \mybackslash H / J$ and therefore it is enough to verify the Mackey formula for \textit{one} $R$. To see this, let $R = \lbrace h_1,\ldots,h_n \rbrace$ and $R' = \lbrace h_1',\ldots,h_n' \rbrace$ be two complete sets of representatives of $I \mybackslash H / J$ and assume without loss of generality that $I h_k J = I h_k' J$ for all $k \in \lbrace 1, \ldots, n \rbrace$. Then for each $k$ we have $h_k' = u_k h_k v_k$ for some $u_k \in I$ and $v_k \in J$. Hence, using the triviality and equivariance we get
\begin{align*}
& \ \ind_{I, I \cap {^{h_k'} \! J}}^\Phi \circ \con_{h_k', I^{h_k'} \cap J}^\Phi \circ \res_{I^{h_k'} \cap J, J}^\Phi \\
=& \ \ind_{I, I \cap {^{u_kh_kv_k} \! J}}^\Phi \circ \con_{u_kh_kv_k, I^{u_kh_kv_k} \cap J}^\Phi \circ \res_{I^{u_kh_kv_k} \cap J, J}^\Phi \\
=& \ \ind_{I, I \cap {^{u_kh_k} \! J}}^\Phi \circ \con_{u_k (h_kv_k), I^{h_kv_k} \cap J}^\Phi \circ \res_{I^{h_kv_k} \cap J,J}^\Phi \\
=& \ \ind_{I, I \cap {^{u_kh_k} \! J}}^\Phi \circ \left( \con_{u_k, {^{h_kv_k}(I^{h_kv_k} \cap J)}}^\Phi \circ \con_{h_kv_k,I^{h_kv_k} \cap J}^\Phi \right) \circ \res_{I^{h_kv_k} \cap J,J}^\Phi\\
=& \ \left( \ind_{I, I \cap {^{u_kh_k}\!J}}^\Phi \circ \con_{u_k,I \cap {^{h_k}\!J}}^\Phi \right) \circ \con_{h_kv_k,I^{h_kv_k} \cap J}^\Phi  \circ \res_{I^{h_kv_k} \cap J,J}^\Phi \\
=& \ \left( \con_{u_k,I}^\Phi \circ \ind_{I,I \cap {^{h_k}\!J}}^\Phi \right) \circ \left( \con_{h_k, ^{v_k}(I^{h_kv_k} \cap J)} \circ \con_{v_k,I^{h_kv_k} \cap J} \right) \circ \res_{I^{h_kv_k} \cap J,J}^\Phi \\
=& \ \ind_{I,I \cap {^{h_k}\!J}}^\Phi \circ \con_{h_k, I^{h_k} \cap J}^\Phi \circ \left( \con_{v_k,I^{h_kv_k} \cap J} \circ \res_{I^{h_kv_k} \cap J,J}^\Phi \right) \\
=& \ \ind_{I,I \cap {^{h_k}\!J}}^\Phi \circ \con_{h_k, I^{h_k} \cap J}^\Phi \circ \left(  \res_{^{v_k}(I^{h_kv_k} \cap J), ^{v_k}\!J}^\Phi \circ \con_{v_k,J}^\Phi  \right) \\
=& \ \ind_{I,I \cap {^{h_k}\!J}}^\Phi \circ \con_{h_k, I^{h_k} \cap J}^\Phi \circ \res_{I^{h_k} \cap J, J}^\Phi.
\end{align*}

This shows that the Mackey formula is independent of the choice of $R$. Moreover, using the equivariance we see that the Mackey formula can also be written as
\[
\res_{I,H}^\Phi \circ \ind_{H,J}^\Phi = \sum_{h \in R} \ind_{I, I \cap {^h \! J}}^\Phi \circ \res_{I \cap {^h \! J}, ^h \! J}^\Phi \circ \con_{h,J}^\Phi.
\]
\end{para}

\begin{para} \label{para:mackey_functors_cat_props_similar}
As the stability relation and the Mackey formula are both functorial, it is easy to see that the statements of \ref{para:functor_iso_on_values}, \ref{para:functor_limit_inheritance} and \ref{para:functor_cat_abelian} similarly hold for Mackey functors. 
If $L:\ca{A}' \ra \ca{A}$ is an additive functor of preadditive categories, then the functor $L \circ (-): \se{Fct}(\fr{M},\ca{A}') \ra \se{Fct}(\fr{M},\ca{A})$ restricts to a functor $\se{Mack}(\fr{M},\ca{A}') \ra \se{Mack}(\fr{M},\ca{A})$ for any $G$-Mackey system $\fr{M}$. 
\end{para}

\begin{defn}
A RIC-functor $\Phi$ with domain a $G$-subgroup system $\fr{S}$ and codomain a preadditive category $\ca{A}$ is called \words{cohomological}{RIC-functor!cohomological} if 
\[
\ind_{H,I}^\Phi \circ \res_{I,H}^\Phi = \lbrack H:I \rbrack \cdot \id_{\Phi(H)}
\]
holds in $\hom_{\ca{A}}(\Phi(H),\Phi(H))$ for all $H \in \ssys{b}$ and $I \in \ssys{r}(H) \cap \ssys{i}(H)$.
We denote the full subcategory of $\se{Mack}(\fr{M},\ca{A})$ consisting of cohomological Mackey functors by $\se{Mack}^{\ro{c}}(\fr{M},\ca{A})$. It is easy to see that the statements of \ref{para:mackey_functors_cat_props_similar} similarly hold for cohomological Mackey functors.
\end{defn}

\begin{para}
As none of the above is new but is formulated in a slightly more general setting, we should compare our definitions with the existing ones to see if all notions are compatible. 
\begin{compactenum}[$\bullet$]
\item If $G$ is a finite group and $k$ is a ring, then the definition of the category $\se{Mack}(\ro{Grp}(G), \lMod{k})$ coincides with the definition of the category of $\lMod{k}$-valued Mackey functors for $G$ given by \name{Serge Bouc} in \cite[section 1.1.1]{Bou97_Green-functors_0} (there a Mackey functor is only defined in precisely this situation).

\item The notion of a restriction functor (respectively a conjugation functor) defined by \name{Robert Boltje} in \cite[chapter 1]{Bol98_A-general-theory_0} can be recovered from our definition of a RIC-functor by choosing a subgroup system in which no non-trivial inductions are allowed (respectively no non-trivial inductions and restrictions are allowed).

\item If $(\ca{C},\ca{O})$ is a Mackey system as defined by \name{Werner Bley} and \name{Robert Boltje} in \cite[definition 2.1]{BleBol04_Cohomological-Mackey_0}, then $\fr{M} \dopgleich (\msys{b},\msys{r},\msys{i}) \dopgleich (\ca{C},\ca{C},\ca{O})$ is a Mackey system in our sense and if $k$ is a commutative ring, then the category $\se{Mack}(\fr{M},\lMod{k})$ (respectively $\se{Mack}^{\ro{c}}(\fr{M},\lMod{k})$) coincides with the category of $k$-Mackey functors on $(\ca{C},\ca{O})$ (respectively with the category of cohomological $k$-Mackey functors on $(\ca{C},\ca{O})$) as defined in \cite[definition 2.3]{BleBol04_Cohomological-Mackey_0}.\footnote{Note that in \cite[definition 2.1]{BleBol04_Cohomological-Mackey_0} the condition $H \in \ca{O}(H)$ for each $H \in \ca{C}$ is missing but is necessary.}
\end{compactenum}
\end{para} 

\begin{para}
The definition of a Mackey functor given by Werner Bley and Robert Boltje in \cite{BleBol04_Cohomological-Mackey_0} is already very flexible but still too rigid for our purposes. One particular problem with this definition is that the Mackey system $\fr{M} = (\ca{C},\ca{C},\ca{O})$ is always I-finite and the restrictions are required to exists for all subgroups, that is, we always have $\msys{r}(H) = \msys{b}(H) = \lbrace I \mid I \in \msys{b} \tn{ and } I \leq H \rbrace$. The discussion of the abelianization functors in \ref{sect:abelianization} will show that there exist common situations where these Mackey systems would be unnecessarily restrictive because in this case the induction morphisms are just inclusions, which of course do not need a finite index condition, while the restriction morphisms are the transfer morphisms which are only defined for a finite index.
\end{para}

\begin{para}
There exist at least two other definitions of the notion of a Mackey functor. One is due to \name{Andreas Dress} who defines in \cite[\S4]{Dre73_Contributions-to-the-theory_0} a Mackey functor from a category $\ca{A}$ with finite coproducts to an arbitrary category $\ca{B}$ to be a bifunctor $(M_*,M^*):\ca{A} \ra \ca{B}$ satisfying certain conditions. It is then proven in \cite[proposition 5.1]{Dre73_Contributions-to-the-theory_0} that if for a finite group $G$ the category $\ca{A}$ is the category of finite $G$-sets and if $k$ is a commutative ring, then the category of Dress-type Mackey functors $(M_*,M^*):\ca{A} \ra \lMod{k}$ is equivalent to $\se{Mack}(\ro{Grp}(G), \lMod{k})$.

One further definition (or more precisely identification) is due to \name{Harald Lindner} who proves in \cite[theorem 4]{Lin76_A-remark-on-Mackey_0} that the category of Dress-type Mackey functors $(M_*,M^*):\ca{A} \ra \ca{B}$ is under certain conditions on $\ca{A}$ canonically isomorphic to the category of finite product preserving functors from the span $\ro{Sp}(\ca{A})$ of $\ca{A}$ to $\ca{B}$.
\end{para}

\begin{open} \label{para:mackey_functor_definitions_problem}
Is there a definition of a Mackey functor that encompasses the definition given by Dress and the definition given by Bley and Boltje and perhaps even encompasses our definition? In \cite[theorem 2.7]{BleBol04_Cohomological-Mackey_0} it is proven that for a Mackey system $(\ca{C},\ca{O})$ in the sense of \cite[definition 2.1]{BleBol04_Cohomological-Mackey_0} and a commutative ring $k$ the category $\se{Mack}((\ca{C},\ca{C},\ca{O}),\lMod{k})$ is $k$-linearly equivalent to a certain category of pairs of functors on generalizations of orbit categories which satisfy certain conditions. This extends the results by Dress mentioned above and the category defined by Bley and Boltje looks very similar to a category of Dress-type Mackey functors but in general the definitions differ in two details. First, whereas Dress is considering a pair of functors $(M_*,M^*)$ which are defined on the \textit{same} category $\ca{A}$, Bley and Boltje consider a pair of functors $(M_*,M^*)$ which are defined on categories that are \textit{different} in general (one of the categories might be larger). Second, the pair of functors considered by Dress just has to satisfy a condition dealing with finite coproducts whereas the functors by Bley and Boltje have to satisfy a condition dealing with infinite coproducts.

On the other hand, there is the question of how Lindner's interpretation of Mackey functors as proper functors from a span category fits into this framework. What happens if the conditions Lindner is imposing on the domain category is dropped? Is this a generalization into a different direction?

\newpage
\[
\xymatrix{
\tn{Our definition} \ar@{<-->}[rr]^-? \ar@{<-->}[dr]^-? & & \tn{Lindner} \ar@{<-->}[dl]^-? \\
\tn{Boltje--Bley} \ar@{<-->}[r]^-? \ar@{ >->}[u] & \tn{Dress} \\
\tn{Bouc} \ar@{ >->}[u] \ar@{ >->}[ur]
}
\]
\end{open}

\begin{para} \textbf{Generalization.}
When considering for any (finite) group $G$ the Grothendieck group $\ro{G}_0^k(G) \dopgleich \ro{G}_0(kG)$ for a commutative ring $k$, we see that this naturally gives an RI-functor $\ro{G}_0^k: \ro{Grp}^{\ro{f}} \ra {_k\se{Mod}}$, where $\ro{Grp}^{\ro{f}}$ is the RI-domain $(\ro{Grp}^{\ro{f}}, \subs, 1, \mathbf{1})$. As the representation ring is defined for any group, this functor shows in particular, that there are interesting RI-functors that are defined for any group (\word{global point of view}) and that a $G$-subgroup system as RIC-domain for a fixed group $G$ would be unnatural and too restrictive in this situation (\word{local point of view}).

The RI-functor $\ro{G}_0^k$ moreover suggests the following generalization of our definition: as $\ro{G}_0^k$ is not only a $k$-module but also a commutative $k$-algebra, we would like to include this information in the functor, i.e.\ we would like to have a ${_k\se{CAlg}}$-valued RI-functor assigning to each finite group $G$ the commutative ring $\ro{G}_0^k(G)$ and not just an ${_k\se{Mod}}$-valued functor. It is obvious that the restrictions are ring morphisms but this is in general not true for the induction morphisms and so there is no way to get a ${_k\se{Alg}}$-valued RI-functor. This problem can be solved by generalizing the codomain of an RI-functor as follows: first, we define a \word{categorical wedge} to be a diagram $\xymatrix{ \ca{C}_{\ro{r}} & \ca{C}_{\ro{b}} \ar[l]_{F_{\ro{r}}} \ar[r]^{F_{\ro{i}}} & \ca{C}_{\ro{i}} }$ of categories. Now, for a RIC-domain $\fr{D}$ and a categorical wedge $\ca{C}$ as above, we define a (generalized) \word{RIC-functor} with domain $\fr{D}$ and codomain $\ca{C}$ to be a tuple consisting of
\begin{compactitem}
\item a map $\Phi: \dsys{b} \ra \ro{Ob}(\ca{C}_{\ro{b}})$,
\item a $\ca{C}_{\ro{r}}$-morphism $\res_{y,x}^\Phi:F_{\ro{r}}( \Phi(x) ) \ra F_{\ro{r}}(\Phi(y))$ for each $x \in \dsys{b}$ and $y \in \dsys{r}(x)$, called \word{restriction},
\item a $\ca{C}_{\ro{i}}$-morphism $\ind_{x,y}^\Phi:F_{\ro{i}}(\Phi(y)) \ra F_{\ro{i}}(\Phi(x))$ for each $x \in \dsys{b}$ and $y \in \dsys{i}(x)$, called \word{induction},
\item a $\ca{C}_{\ro{b}}$-morphism $\con_{g,x}^\Phi:\Phi(x) \ra \Phi(^g \! x)$ for each $x \in \dsys{b}$ and $g \in G$, called \word{conjugation},
\end{compactitem}
such that the following properties are satisfied:
\begin{enumerate}[label=(\roman*),noitemsep,nolistsep]
\item (Triviality)
\[
\res_{x,x}^\Phi = \id_{F_{\ro{r}}(\Phi(x))}, \quad \ind_{x,x}^\Phi = \id_{F_{\ro{i}}(\Phi(x))}, \quad \con_{1,x}^\Phi = \id_{\Phi(x)}
\]
for all $x \in \dsys{b}$.\footnote{These morphisms are defined since $x \in \dsys{\star}(x)$.}
\item (Transitivity)
\[
\res_{z,y}^\Phi \circ \res_{y,x}^\Phi = \res_{z,x}^\Phi 
\]
for all $x \in \dsys{b}$, $y \in \dsys{r}(x)$ and $z \in \dsys{r}(y)$;
\[
\ind_{x,y}^\Phi \circ \ind_{y,z}^\Phi = \ind_{x,z}^\Phi 
\]
for all $x \in \dsys{b}$, $y \in \dsys{i}(x)$ and $z \in \dsys{i}(y)$; and
\[
\con_{g',^g \! x}^\Phi \circ \con_{g,x}^\Phi = \con_{g'g,x}^\Phi
\]
for all $x \in \dsys{b}$ and all $g,g' \in G$.\footnote{As $y \in \dsys{\star}(x)$, it follows that $\dsys{\star}(y) \subs \dsys{\star}(x)$ and consequently $z \in \dsys{\star}(x)$. Hence, both $\res_{z,x}^\Phi$ and $\ind_{x,z}^\Phi$ are defined.}

\item (Equivariance)
\[
F_{\ro{r}}(\con_{g,y}^\Phi) \circ \res_{y,x}^\Phi = \res_{^g \! y, ^g \! x}^\Phi \circ F_{\ro{r}}(\con_{g,x}^\Phi)
\]
for all $x \in \dsys{b}$, $y \in \dsys{r}(x)$ and $g \in G$; and
\[
F_{\ro{i}}(\con_{g,x}^\Phi) \circ \ind_{x,y}^\Phi = \ind_{^g \! x, ^g \! y}^\Phi \circ F_{\ro{i}}(\con_{g,y}^\Phi)
\]
for all $x \in \dsys{b}$, $y \in \dsys{i}(x)$ and $g \in G$.\footnote{This is well-defined by the properties of $\fr{D}$.} \\

\end{enumerate}

A \words{morphism}{functor!morphism} $\varphi$ between two generalized RI-functors $\Phi,\Psi: \fr{D} \ra \ca{C}$, denoted by $\varphi:\Phi \ra \Psi$, is a family $\varphi = \lbrace \varphi_x \mid x \in \dsys{b} \rbrace$ of $\ca{C}_{\ro{b}}$-morphisms $\varphi_x:\Phi(x) \ra \Psi(x)$ satisfying the following conditions:
\begin{enumerate}[label=(\roman*),noitemsep,nolistsep]
\item For each $x \in \dsys{b}$ and $y \in \dsys{r}(x)$ the diagram
\[
\xymatrix{
F_{\ro{r}}(\Phi(x)) \ar[rr]^{F_{\ro{r}}(\varphi_x)} \ar[d]_{\res_{y,x}^\Phi} & & F_{\ro{r}}(\Psi(x)) \ar[d]^{\res_{y,x}^\Psi} \\
F_{\ro{r}}(\Phi(y)) \ar[rr]_{F_{\ro{r}}(\varphi_y)} & & F_{\ro{r}}(\Psi(y))
}
\]
commutes.

\item For each $x \in \dsys{b}$ and $y \in \dsys{i}(x)$ the diagram
\[
\xymatrix{
F_{\ro{i}}(\Phi(x)) \ar[rr]^{F_{\ro{i}}(\varphi_x)} & & F_{\ro{i}}(\Psi(x)) \\
F_{\ro{i}}(\Phi(y)) \ar[rr]_{F_{\ro{i}}(\varphi_y)} \ar[u]^{\ind_{x,y}^\Phi} & & F_{\ro{i}}(\Psi(y)) \ar[u]_{\ind_{x,y}^\Psi}
}
\]
commutes.

\item For each $x \in \dsys{b}$ and $g \in G$ the diagram
\[
\xymatrix{
\Phi(x) \ar[rr]^{\varphi_x} \ar[d]_{\con_{g,x}^\Phi} & & \Psi(x) \ar[d]^{\con_{g,x}^\Psi} \\
\Phi(^g \! x) \ar[rr]_{\varphi_{^g \! x}} & & \Psi(^g \! x)
}
\]

commutes. \\
\end{enumerate}

The composition of two morphisms $\varphi:\Phi \ra \Psi, \psi:\Psi \ra \Theta$ is defined by $\psi \circ \varphi = \lbrace \psi_x \circ \varphi_x \mid x \in \dsys{b} \rbrace$ which is a morphism $\psi \circ \varphi:\Phi \ra \Theta$. With respect to this composition the morphism $\id_{\Phi}:\Phi \ra \Phi$ defined by the family $\lbrace \id_{\Phi(x)} \mid x \in \dsys{b} \rbrace$ is an identity. The set of all generalized RIC-functors with domain $\fr{D}$ and codomain $\ca{C}$ together with morphisms, composition and identity as defined above forms a category which is denoted by $\se{Fct}(\fr{D},\ca{C})$. If $\dsys{b}$ is a set, then $\se{Fct}(\fr{D},\ca{C})$ is a category. \\

When interpreting an RI-domain in the obvious categorical sense, we see that a generalized RI-functor is just a diagram of categories
\[
\xymatrix{
\dsys{r} \ar@/_1pc/[dd]_{\res^\Phi} & & \dsys{i}^\circ \ar@/^1pc/[dd]^{\ind^\Phi} \\
& \dsys{b} \ar[dd]^\Phi \ar@{>->}[ul] \ar@{>->}[ur] \\
\ca{C}_{\ro{r}} & & \ca{C}_{\ro{i}} \\
& \ca{C}_{\ro{b}} \ar[ul]^{F_{\ro{r}}} \ar[ur]_{F_{\ro{i}}}
}
\]

Now, in the example of the RI-functor $\ro{G}_0^k:\ro{Grp}^{\ro{f}} \ra {_k\se{Mod}}$ discussed above, we can make it into an RI-functor with $\xymatrix{ {_k\se{CAlg}} & {_k\se{CAlg}} \ar@{=}[l] \ar[r]^F & {_k\se{Mod}} }$ as codomain, where $F$ is the forgetful functor assigning to a commutative $k$-algebra its underlying $k$-module. In this way all the information is preserved and thus solving the problem mentioned above. The categorical wedge $\xymatrix{ {_k\se{CAlg}} & {_k\se{CAlg}} \ar@{=}[l] \ar[r]^F & {_k\se{Mod}} }$ will be called the \word{Frobenius wedge} over $k$ and will be denoted $\ca{F}_k$ from now on. The notion of a Frobenius functor defined in \cite{CurRei87_Methods-of-representation_0}, which is motivated by an additional important property $\ro{G}_0^k$ satisfies, can now be nicely defined in our terms: a \word{Frobenius functor} over $k$ is an RI-functor $\Phi: \ro{Grp}^{\ro{f}} \ra \ca{F}_k$ satisfying the \word{Frobenius identity}:
\[
x \cdot \ind_{G,H}(y) = \ind_{G,H}( \res_{H,G}(x) \cdot y)
\]
for all inclusions $H \subs G$, $x \in \Phi(G)$ and $y \in \Phi(H)$.

Moreover, the notion of a Green-functor defined in \cite{Gre71_Axiomatic-representation_0}, which again is motivated by $\ro{G}_0^k$, becomes an easy instantiation of our notions: first, define the \word{Green wedge} over $k$ to be the categorical wedge 
\[
\xymatrix{
& _k\se{GenAlg} \ar@{=}[dl] \ar[dr]^F \\
_k\se{GenAlg} & & \lMod{k} 
}
\]
and denote it by $\ca{G}_k$. Now, for a (finite) group $G$, a \word{$k$-linear $G$-Green functor} over $k$ is an RIC-functor $\Phi: \ro{Grp}(G) \ra \ca{G}_k$ such that the underlying RIC-functor into ${_k\se{Mod}}$ is a Mackey-Functor and the underlying RI-functor into $\ca{F}_k$ is a Frobenius functor.
\end{para}

\subsection{Discrete $G$-modules} \label{sect:disc_gmod}

\begin{para}
In this section we will discuss a fascinating interpretation of Mackey functors, namely that for a topological group $G$ and certain ``nice'' $G$-Mackey systems $\fr{M}$ the Mackey functors $\fr{M} \ra \se{Ab}$ can be interpreted as generalizations of discrete $G$-modules. More specifically, we will show that, under certain conditions on $\fr{M}$ (that are satisfied for example by $\ro{Grp}(G)^{\tn{f}}$ for a profinite group $G$) the category $\discgmod$ of discrete $G$-modules is equivalent to a full subcategory of $\se{Mack}(\fr{M},\se{Ab})$ formed by the Mackey functors having Galois descent. It is probably this equivalence that motivated \name{Jürgen Neukirch} to call Mackey functors also \words{modulations}{modulation}. The equivalence between $\discgmod$ for a profinite group $G$ and the Mackey functors $\se{Grp}(G)^{\tn{f}} \ra \se{Ab}$ having Galois descent was mentioned without proof in \cite[section 3]{Neu94_Micro-primes_0}, \cite[chapter IV, \S6]{Neu99_Algebraic-Number_0} and \cite[chapter 1, \S5]{NeuSchWin08_Cohomology-of-Number_0} (there may of course exist an easy argument the author was not aware of so that no proof is needed). We will prove a generalized version of this fact and recover in \ref{para:cor_profinite_descent_equiv} the situation mentioned in the references.
\end{para}

\begin{ass} \label{para:disc_g_mod_ass} \wordsym{$\discgmod$} 
Throughout this section we fix a topological group $G$ and an I-finite $G$-Mackey system $\fr{M}$. The category of discrete $G$-modules is denoted by $\discgmod$. 
\end{ass}

\begin{para}
Our first task is to construct a functor $\discgmod \ra \se{Mack}(\fr{M},\se{Ab})$ and the idea is that we assign to a discrete $G$-module $A$ the functor $A_*:\fr{M} \ra \se{Ab}$ assigning to each $H \in \msys{b}$ the $H$-invariants of $A$, that is, $A_*(H) = A^H$. The restriction will be the canonical inclusion, the induction will be the norm map and the conjugation will be the $G$-action on $A$. This construction works already on a general $G$-subgroup system $\fr{S}$ but we have to assume that $\fr{S}$ is I-finite because the norm map $A^I \ra A^H$ is only defined if $I$ is of finite index in $H$.
\end{para}

\begin{prop} \label{para:discgmod_to_functor} \wordsym{$A_*$} \wordsym{$\ro{H}_{\fr{S}}^0(A)$}
Let $\fr{S}$ be an I-finite $G$-subgroup system. Let $A \in \discgmod$. The following data define a cohomological and stable RIC-functor $A_* \dopgleich \ro{H}_{\fr{S}}^0(A): \fr{S} \ra \se{Ab}$:
\begin{compactitem}
\item $A_*(H) \dopgleich A^H$ for each $H \in \ssys{b}$.
\item $\res_{I,H}^{A_*}: A_*(H) \ra A_*(I)$ is given by the canonical inclusion $A^H \rightarrowtail A^I$ for each $H \in \ssys{b}$ and $I \in \ssys{r}(H)$.
\item $\ind_{H,I}^{A_*}: A_*(I) \ra A_*(H), a \mapsto \sum_{h \in T} ha$ for each $H \in \ssys{b}$, $I \in \ssys{i}(H)$ and an arbitrary left transversal $T$ of $I$ in $H$. 
\item $\con_{g,H}^{A_*}: A_*(H) \ra A_*(^g \! H), a \mapsto ga$, for each $H \in \ssys{b}$ and $g \in G$. \\
\end{compactitem}

Moreover, $\ro{H}_{\fr{M}}^0(A):\fr{M} \ra \se{Ab}$ is a cohomological Mackey functor. If the $G$-subgroup system is fixed, then we prefer the notation $A_*$ to $\ro{H}_{\fr{S}}^0(A)$. 
\end{prop}

\begin{proof}
We first check that the maps defined above are well-defined. Let $H \in \ssys{b}$, $g \in G$ and $a \in A_*(H) = A^H$. If $g' \in {^g \! H}$, then $g' = ghg^{-1}$ for some $h \in H$ and
\[
g'( \con_{g,H}^{A_*}(a)) = g'(ga) = ghg^{-1}(ga) = gha = ga = \con_{g,H}^{A_*}(a).
\]
Hence, $\con_{g,H}^{A_*}(a) \in A^{^g \! H}  = A_*(^g \! H)$. The map $\res_{I,H}^{A_*}$ is obviously well-defined. It remains to check that $\ind_{H,I}^{A_*}$ is well-defined and independent of the choice of the left transversal $T$. Let $H \in \ssys{b}$, $I \in \ssys{i}(H)$ and let $T = \lbrace h_1,\ldots,h_n \rbrace$, $T' = \lbrace h_1',\ldots,h_n' \rbrace$ be two left transversals of $I$ in $H$.\footnote{Note that $\lbrack H:I \rbrack < \infty$ since $\fr{S}$ is I-finite by assumption.} It is easy to see that there exists a permutation $\sigma \in \ro{S}_n$ and elements $u_i \in I$ such that $h_i' = h_{\sigma(i)} u_{\sigma(i)}$ for each $i=1,\ldots,n$. If $a \in A_*(I) = A^I$, we have
\[
\sum_{i=1}^n h_i' a = \sum_{i=1}^n h_{\sigma(i)} u_{\sigma(i)} a = \sum_{i=1}^n h_{\sigma(i)} a = \sum_{i=1}^n h_i a,
\]

so $\ind_{H,I}^{A_*}$ is independent of the choice of $T$. Moreover, if $h \in H$, then $hT = \lbrace h h_1,\ldots,h h_n \rbrace$ is also a left transversal of $I$ in $H$ and because of the independence of the induction morphisms on the choice of the right transversal we get
\[
\sum_{i=1}^n h_i a = \sum_{i=1}^n h h_i a = h \sum_{i=1}^n h_i a.
\]
Hence, $\ind_{H,I}^{A_*}(a) \in A^H = A_*(H)$. It is easily verified that the above morphisms satisfy all the relations making $A_*$ into a cohomological and stable RIC-functor on $\fr{S}$.

It remains to verify the Mackey formula on $\fr{M}$. Let $H \in \fr{M}_{\ro{b}}$, $I \in \fr{M}_{\ro{r}}(H)$, $J \in \fr{M}_{\ro{i}}(H)$ and let $R = \lbrace h_1,\ldots,h_n \rbrace$ be a complete set of representatives of $I \mybackslash H / J$. For each $i$ let $_i T = \lbrace t_{i,1}, \ldots, t_{i,{n_i}} \rbrace$ be a left transversal of $I \cap {^{h_i} \! J}$ in $I$. Then $T = \lbrace t_{i,j} h_i \mid i \in \lbrace 1,\ldots,n \rbrace \tn{ and } j \in \lbrace 1,\ldots,n_i \rbrace \rbrace$ is a left transversal of $J$ in $H$ by \ref{para:double_coset_rep_lift}. If $a \in A_*(J)$, this implies
\[
\sum_{h \in R} \ind_{I,I \cap ^h \! J}^{A_*} \circ \con_{h,I^h \cap J}^{A_*} \circ \res_{I^h \cap J,J}^{A_*}(a) = \sum_{i=1}^n \ind_{I,I \cap ^{h_i} \! J}^{A_*}( h_i a) = \sum_{i=1}^n \sum_{j=1}^{n_i} t_{i,j} h_i a
\]
\[
 = \sum_{h \in T} ha = \ind_{H,J}^{A_*}(a) = \res_{I,H}^{A_*} \circ \ind_{H,J}^{A_*}(a).
\]
\end{proof}

\begin{para}
By definition, $\ro{H}_{\fr{S}}^0(A)(H) = A^H$ is equal to the (abstract) cohomology group $\ro{H}^0(H,A)$ of the $G$-module $A$. The conjugation, restriction and induction morphisms of $\ro{H}_{\fr{S}}^0(A)$ are precisely the morphisms defined in group cohomology. Using dimension shifting one can show that the higher (abstract) cohomology groups define RIC-functors $\ro{H}_{\fr{S}}^n(A)$ having the same properties as $\ro{H}_{\fr{S}}^0(A)$.
\end{para}

\begin{cor} \wordsym{$\varphi_*$}
Under the conditions of \ref{para:discgmod_to_functor} the following maps define a functor:
\[
\hspace{90pt}
\begin{array}{rcl}
-_*: \discgmod & \lra & \se{Stab}^{\ro{c}}({\fr{S},\se{Ab}}) \\
A & \longmapsto & A_* \\
\varphi: A \ra B & \longmapsto & \varphi_*: A_* \ra B_* \\
 & & \varphi_{*,H} \dopgleich \varphi|_{A^H}: A_*(H) \ra B_*(H)
\end{array}
\]

\end{cor}

\begin{proof}
First note that if $a \in A^H$ and $h \in H$ then $h \varphi(a) = \varphi(ha) = \varphi(a)$, so $\varphi_{*,H}$ is well-defined. The compatibility of $\varphi_*$ with the restriction morphisms is obvious and the compatibility of $\varphi_*$ with the conjugation and induction morphisms is due to the $G$-equivariance of $\varphi$. Hence, $\varphi_*:A_* \ra B_*$ is a morphism of RIC-functors. Obviously, $(\id_A)_* = \id_{A_*}$ and $(\psi \circ \varphi)_* = \psi_* \circ \varphi_*$.
\end{proof}

\begin{para}
Now that we have a functor $-_*: \discgmod \ra \se{Mack}(\fr{M},\se{Ab})$, we want to construct a functor $-^*$ in the opposite direction. The idea is that for a Mackey functor $\Phi:\fr{M} \ra \se{Ab}$ the conjugation morphisms of $\Phi$ define for each $U \in \msys{b}$ with $U \lhd G$ a $G/U$-module structure on $\Phi(U)$ and then we use \ref{para:gmodule_struct_on_ind_limit} to get a canonical $G$-module structure on the colimit of the $\Phi(U)$, $U \in \ca{U}$, with respect to the restriction morphisms, where $\ca{U} \subs \msys{b}$ is a subset of normal subgroups of $G$ which contains enough groups to make this construction work. We will now make precise what the set $\ca{U}$ has to satisfy.
\end{para}

\begin{defn} \label{para:discgmod_funtor_system_ass}
A \word{descent basis} for a $G$-subgroup system $\fr{S}$ is a subset $\ca{U} \subs \ssys{b}$ satisfying the following conditions:
\begin{enumerate}[label=(\roman*),noitemsep,nolistsep]
\item \label{para:discgmod_funtor_system_ass_filter_basis} $\ca{U}$ is a filter basis on $G$ consisting of open normal subgroups of $G$.\footnote{Confer \ref{prop:filter_basis} for the definition of a filter basis.}
\item \label{para:discgmod_funtor_system_ass_res} $\ca{U}(H) \dopgleich \lbrace U \mid U \in \ca{U} \tn{ and } U \leq H \rbrace \subs \ssys{r}(H)$ for each $H \in \ssys{b}$.
\item \label{para:discgmod_funtor_system_ass_cofinal} $\ca{U}$ is cofinal in $(\ssys{b},\sups)$.\footnote{This is equivalent to $\ca{U}(H) \neq \emptyset$ for each $H \in \ssys{b}$.}
\end{enumerate}
\end{defn}

\begin{para}
Note that condition \ref{para:discgmod_funtor_system_ass}\ref{para:discgmod_funtor_system_ass_cofinal} already implies that all groups in $\ssys{b}$ are open in $G$.
\end{para}

\begin{prop} \label{para:r_basis_for_qc}
If $G$ is quasi compact and $\ca{U}$ is the set of all open normal subgroups of $G$, then $\ca{U}$ is a descent basis for $\ro{Grp}(G)^{\tn{f}}$.
\end{prop}

\begin{proof}
Let $\fr{M} = \ro{Grp}(G)^{\tn{f}}$. As $G$ is quasi compact, the set $\msys{b}$ consists of all open subgroups of $G$ and $\msys{r}(H) = \msys{b}(H)$. Hence, $\ca{U} \subs \msys{b}$ and both \ref{para:discgmod_funtor_system_ass}\ref{para:discgmod_funtor_system_ass_filter_basis} and \ref{para:discgmod_funtor_system_ass}\ref{para:discgmod_funtor_system_ass_res} are satisfied. To see that \ref{para:discgmod_funtor_system_ass}\ref{para:discgmod_funtor_system_ass_cofinal} is satisfied, let $H \in \msys{b}$. By \ref{para:normal_core} the normal core $\ro{NC}_G(H)$ of $H$ in $G$ is an open normal subgroup of $G$ contained in $H$. Hence, $\ro{NC}_G(H) \in \ca{U}(H)$.
\end{proof}

\begin{prop} \wordsym{$\Phi^*$} \wordsym{$\varphi^*$}
\label{para:mod_structure_on_functor}
Let $\fr{S}$ be a $G$-subgroup system, let $\Phi \in \se{Stab}(\fr{S},\se{Ab})$, let $U \in \ssys{b}$ and let $\ro{N}_G(U)$ be the normalizer of $U$ in $G$. Then the map
\[
\begin{array}{rcl}
(\ro{N}_G(U)/U) \times \Phi(U) & \lra & \Phi(U) \\
(\ol{g}, a) & \longmapsto & \con_{g,U}^{\Phi}(a), 
\end{array}
\]
where $g \in \ro{N}_G(U)$ is a representative of $\ol{g}$, defines a $\ro{N}_G(U)/U$-module structure on $\Phi(U)$.
\end{prop}

\begin{proof}
First note that as $U \lhd \ro{N}_G(U)$, we have $^g \! U = U$ and therefore $\con_{g,U}^\Phi(a) \in \Phi(^g \! U) = \Phi(U)$ for each $g \in \ro{N}_G(U)$. Moreover, this map is independent of the choice of the representative: if $g,g' \in \ro{N}_G(U)$ are two representatives, then $g' = gh$ for some $u \in U$ and using the stability of $\Phi$ and the transitivity of the conjugation morphisms we get
\[
\con_{g',U}^\Phi = \con_{gu,U}^\Phi = \con_{g,^u \! U}^\Phi \circ \con_{u,U}^\Phi = \con_{g, U}^\Phi \circ \id_{\Phi(U)} = \con_{g,U}^\Phi. 
\]

The triviality, transitivity and additivity of the conjugation morphisms now imply that the given map defines an $\ro{N}_G(U)/U$-module structure on $\Phi(U)$.
\end{proof}

\begin{conv}
If $U \in \ssys{b}$ and $H \leq \ro{N}_G(U)$, then $\Phi(U)$ is as an $H/U$-module always considered with respect to the $\ro{N}_G(U)/U$-module structure above.
\end{conv}

\begin{prop} \label{para:discgmod_functor_upperstar}
Let $\fr{S}$ be a $G$-subgroup system with a descent basis $\ca{U}$ and let $\Phi \in \se{Stab}(\fr{S},\se{Ab})$. The following holds:
\begin{enumerate}[label=(\roman*),noitemsep,nolistsep]
\item Let $\Phi^*_{\ca{U}} := \ro{colim} \ \ca{I}$, where $\ca{I}$ is the inductive system in $\se{Ab}$ with index set $(\ca{U}, \sups)$, objects $\Phi(U), U \in \ca{U}$, and morphisms $\res_{V,U}^\Phi: \Phi(U) \ra \Phi(V)$ for each $U \sups V$ in $\ca{U}$.\footnote{The assumptions on $\ca{U}$ imply that $\ca{I}$ is a well-defined inductive system.} Then $\Phi^*$ admits a unique structure of a discrete $G$-module such that for each $U \in \ca{U}$ and $g \in G$ the diagram
\[
 \xymatrix{
 \Phi^*_{\ca{U}} \ar[r]^{\mu_g} & \Phi^*_{\ca{U}} \\
 \Phi(U) \ar[u]^{\iota_U} \ar[r]_{\con_{g,U}^\Phi} & \Phi(U) \ar[u]_{\iota_U}
 }
\] 

commutes, where $\mu_g:\Phi^*_{\ca{U}} \ra \Phi^*_{\ca{U}}$ denotes the action of $g$ on $\Phi^*_{\ca{U}}$ and $\iota_U: \Phi(U) \ra \Phi^*_{\ca{U}}$ denotes the canonical morphism.

\item If $\varphi:\Phi \ra \Psi$ is a morphism in $\se{Stab}(\fr{S},\se{Ab})$, then there exists a unique $G$-module morphism $\varphi^*_{\ca{U}}: \Phi^*_{\ca{U}} \ra \Psi^*_{\ca{U}}$ such that for each $U \in \ca{U}$ the diagram
\[
\xymatrix{
\Phi^*_{\ca{U}} \ar[r]^{\varphi^*} & \Psi^*_{\ca{U}} \\
\Phi(U) \ar[u]^{\iota_U} \ar[r]_{\varphi_U} & \Psi(U) \ar[u]_{\iota_U}
}
\]
commutes.
\end{enumerate}
\end{prop}

\begin{proof} \hfill

\begin{asparaenum}[(i)]
\item According to \ref{para:mod_structure_on_functor} each $\Phi(U), U \in \ca{U}$, is a $G/U$-module, where $\ol{g} \in G/U$ acts on $\Phi(U)$ by $\con_{g,U}^\Phi$. With respect to this module structure, the morphism $\res_{V,U}^\Phi$ is $G/V$-equivariant for each pair $U \sups V$ in $\ca{U}$ since $\con_{g,V}^\Phi \circ \res_{V,U}^\Phi = \res_{V,U}^\Phi \circ \con_{g,U}^\Phi$. The existence and uniqueness of the $G$-module structure on $\Phi^*_{\ca{U}}$ is therefore just an application of \ref{para:gmodule_struct_on_ind_limit}. 

\item Since $\varphi$ is a morphism of functors, we have a commutative diagram
\[
\xymatrix{
\Phi(U) \ar[r]^{\varphi_U} \ar[d]_{\res_{V,U}^\Phi} & \Phi(U) \ar[d]^{\res_{V,U}^\Phi} \\
\Phi(V) \ar[r]_{\varphi_V} & \Psi(V)
}
\]

for all $U \sups V$ in $\ca{U}$. Hence, the existence and uniqueness of the $G$-module morphism $\varphi^*_{\ca{U}}: \Phi^*_{\ca{U}} \ra \Psi^*_{\ca{U}}$ is again just an application of proposition \ref{para:gmodule_struct_on_ind_limit}. 
\end{asparaenum}
\vspace{-\baselineskip}
\end{proof}

\begin{cor}
Under the conditions of \ref{para:discgmod_functor_upperstar} the following maps define a functor:
\[
\hspace{5pt}
\begin{array}{rcl}
-^*_{\ca{U}}:\se{Stab}(\fr{S},\se{Ab}) & \lra & \discgmod \\
\Phi & \longmapsto & \Phi^*_{\ca{U}} \\
\varphi: \Phi \ra \Psi & \longmapsto & \varphi^*_{\ca{U}}: \Phi^*_{\ca{U}} \ra \Psi^*_{\ca{U}}.
\end{array}
\]
\end{cor}

\begin{proof}
This follows immediately from the uniqueness property of the assigned morphisms.
\end{proof}

\begin{para}
We will always take the standard presentation of the colimit as a disjoint union of sets modulo an equivalence relation and we write $\lbrack U,a \rbrack \dopgleich \iota_U(a) \in \Phi_{\ca{U}}^*$ for the image of an element $a \in \Phi(U)$ in the colimit $\Phi_{\ca{U}}^*$.
\end{para}

\begin{para}
Now, we can investigate the relation between the functors $-_*$ and $-^*_{\ca{U}}$. We will show that if $-^*_{\ca{U}}$ is restricted to $\se{Mack}(\fr{M},\se{Ab})$, then these functors define an adjunction between $\discgmod$ and $\se{Mack}(\fr{M},\se{Ab})$. In the proof it will also become apparent that the Mackey formula is needed to get this adjunction, so that we (unsurprisingly) really have to restrict $-^*_{\ca{U}}$ to $\se{Mack}(\fr{M},\se{Ab})$. Although we could show this adjunction by proving that $\hom( -^*_{\ca{U}}, -) \cong \hom(-, -_*)$ as functors, we will prove it instead by setting up a counit-unit pair between $-^*_{\ca{U}}$ and $-_*$. The reason for this approach is that we want to find the full subcategories on which these functors define an equivalence and this can be done easily with the counit-unit pair.
\end{para}

\begin{thm} 
Let $\ca{U}$ be a descent basis for $\fr{M}$. Then $-^*_{\ca{U}}:\se{Mack}(\fr{M},\se{Ab}) \ra \discgmod$ is left-adjoint to $-_*:\discgmod \ra \se{Mack}(\fr{M},\se{Ab})$.
\end{thm}

\begin{proof}
To simplify notations, we will write $(-)^*$ instead of $(-)^*_{\ca{U}}$. We will prove the existence of an adjunction by defining a counit-unit pair $(\eps,\eta)$, that is, two morphisms of functors
\[
\begin{array}{rcl}
\eps: (-_*)^* & \lra & \id_{\discgmod} \\
\eta: \id_{\se{Mack}(\fr{S},\se{Ab})} & \lra & (-^*)_*
\end{array}
\]
satisfying the relations
\begin{align*}
\Phi^* &= \eps(\Phi^*) \circ \eta(\Phi)^*: \Phi^* \overset{\eta(\Phi)^*}{\lra} ((\Phi^*)_*)^* \overset{\eps(\Phi^*)}{\lra} \Phi^* \\
A_* &= \eps(A)_* \circ \eta(A_*): A_* \overset{\eta(A_*)}{\lra} ((A_*)^*)_* \overset{\eps(A)_*}{\lra} A_*
\end{align*}
for all $\Phi \in \se{Mack}(\fr{M},\se{Ab})$ and $A \in \discgmod$.

We define $\eps$ by
\[
\begin{array}{rcl}
\eps(A): (A_*)^* & \lra & A \\
\lbrack U,a \rbrack & \longmapsto & a
\end{array}
\]
for each $A \in \discgmod$. To verify that $\eps$ indeed defines a morphism of functors, we first check that $\eps(A)$ is well-defined. If $[U,a] \in (A_*)^*$, then $U \in \ca{U}$ and $a \in A_*(U) = A^U \subs A$, so $\eps(A)$ really maps into $A$. If $[U,a] = [V,b] \in (A_*)^*$, then $\res_{W,U}^{A_*}(a) = \res_{W,V}^{A_*}(b)$ for some $W \in \ca{U}$ with $W \leq U \cap V$. Since $a \in A_*(U) = A^U \subs A$, $b \in A_*(V) = A^V \subs A$ and the maps $\res_{W,U}^{A_*}:A^U \ra A^W$, $\res_{W,V}^{A_*}:A^V \ra A^W$ are just the canonical inclusions, this implies $a = b$ and therefore $\eps(A)([U,a]) = \eps(A)([V,b])$. Hence, $\eps(A)$ is well-defined. Moreover, $\eps(A)$ is additive because
\[
 \eps(A)([U,a] + [U,b]) = \eps(A)([W, \res_{W,U}^{A_*}(a) + \res_{W,V}^{A_*}(b)]) 
\]
\[
= \eps(A)([W,a+b]) = a+b = \eps(A)([U,a]) + \eps(A)([V,b])
\]
 
and $\eps(A)$ is $G$-equivariant because
\[
\eps(A)(g [U,a]) = \eps(A)([U, (g \ \ro{mod} \ U) a]) = \eps(A)([U, \con_{g,U}^{A_*}(a)]) = \eps(A)([U, ga]) = ga = g \eps(A)([U,a]).
\]

This shows that $\eps(A) \in \hom_{\discgmod}( (A_*)^*, A)$. Finally, if $\varphi \in \hom_{\discgmod}(A,B)$, then
\[
(\varphi_*)^*([U,a]) = [U,\varphi_{*,U}(a)] = [U,\varphi(a)]
\]

for each $[U,a] \in (A_*)^*$ and therefore
\[
\xymatrix{
(A_*)^* \ar[d]_{\eps(A)} \ar[r]^{(\varphi_*)^*} & (B_*)^* \ar[d]^{\eps(B)} \\
A \ar[r]_{\varphi} & B
}
\]

commutes. This shows that $\eps$ is a morphism of functors. 

Now, we define $\eta$. For $\Phi \in \se{Mack}(\fr{M},\se{Ab})$ the morphism $\eta(\Phi): \Phi \ra (\Phi^*)_*$ is given by the family $\eta(\Phi) = \lbrace \eta(\Phi)_H \mid H \in \msys{b} \rbrace$, where for each $H \in \msys{b}$ we choose using \ref{para:discgmod_funtor_system_ass}\ref{para:discgmod_funtor_system_ass_cofinal} some $U \in \ca{U}(H) \subs \ssys{r}(H)$ and define
\[
\begin{array}{rcl}
\eta(\Phi)_H: \Phi(H) & \lra & (\Phi^*)_*(H) = (\Phi^*)^H \\
a & \longmapsto & [U, \res_{U,H}^\Phi(a)].
\end{array}
\]

Again, we have to verify that $\eta$ is a morphism of functors and we start by proving that $\eta$ is well-defined. It is obvious hat $\eta(\Phi)_H$ really maps into $\Phi^*$. If $h \in H$ and $a \in \Phi(H)$, then
\[
h \eta(\Phi)_H(a) = h[U,\res_{U,H}^\Phi(a)] = [U, \con_{h,U}^\Phi \circ \res_{U,H}^\Phi(a)]
\]
\[
 = [U,\res_{U,H}^\Phi \circ \con_{h,H}^\Phi(a)] = [U,\res_{U,H}^\Phi(a)] = \eta(\Phi)_H(a)
\]

and therefore $\eta(\Phi)_H(a) \in (\Phi^*)^H = (\Phi^*)_*(H)$. To see that $\eta(\Phi)_H$ is independent of the choice of $U$, let $U,V \in \ca{U}(H)$. Since $\ca{U}$ is a filter basis, there exists $W \in \ca{U}$ with $W \leq U \cap V$ and consequently
\[
\res_{W,U}^\Phi \circ \res_{U,H}^\Phi(a) = \res_{W,H}^\Phi(a) = \res_{W,V}^\Phi \circ \res_{V,H}^\Phi(a).
\]
Hence, $[U, \res_{U,H}^\Phi(a)] = [V,\res_{V,H}^\Phi(a)]$. Now, we will verify that $\eta(\Phi)$ is compatible with the conjugation, restriction and induction morphisms. If $H \in \msys{b}$ and $U \in \ca{U}(H)$, then also $U \in \ca{U}(^g \! H)$ for each $g \in G$ since $U \lhd G$. Thus, for any $a \in \Phi(H)$ we have 
\[
 \con_{g,H}^{(\Phi^*)_*} \circ \eta(\Phi)_H(a) = g [U, \res_{U,H}^\Phi(a)] = [U, \con_{g,U}^\Phi \circ \res_{U,H}^\Phi(a)]
\]
\[
= [U, \res_{U, {^g \! H}}^\Phi \circ \con_{g,H}^\Phi(a) ] = \eta(\Phi)_{^g \! H} \circ \con_{g,H}^\Phi(a)
\]
and therefore
\[
\xymatrix{
\Phi(H) \ar[r]^{\eta(\Phi)_H} \ar[d]_{\con_{g,H}^\Phi} & (\Phi^*)_*(H) \ar[d]^{\con_{g,H}^{(\Phi^*)_*}} \\
\Phi(^g \! H) \ar[r]_{\eta(\Phi)_{^g \! H}} & (\Phi^*)_*(^g \! H)
}
\]

commutes. If $H \in \msys{b}$, $I \in \msys{r}(H)$ and $U \in \ca{U}(I)$, then also $U \in \ca{U}(H)$ and for any $a \in \Phi(H)$ this implies
\[
 \res_{I,H}^{(\Phi^*)_*} \circ \eta(\Phi)_H(a) = [U, \res_{U,H}^\Phi(a)] = [U, \res_{U,I}^\Phi \circ \res_{I,H}^\Phi(a)] = \eta(\Phi)_I \circ \res_{I,H}^\Phi(a).
\]

Hence,
\[
\xymatrix{
\Phi(H) \ar[r]^{\eta(\Phi)_H} \ar[d]_{\res_{I,H}^\Phi} & (\Phi^*)_*(H) \ar[d]^{\res_{I,H}^{(\Phi^*)_*}} \\
\Phi(I) \ar[r]_{\eta(\Phi)_{I}} & (\Phi^*)_*(I)
}
\]

commutes. Finally, let $H \in \msys{b}$, $I \in \msys{i}(H)$ and $U \in \ca{U}(I)$. Let $T$ be a left transversal of $I$ in $H$. As $U^g = U \leq I$ for all $g \in G$, an application of \ref{para:transversal_to_double} implies that $T$ is also a complete set of representatives of $U \mybackslash H / I$. Using the fact that $U \in \ca{U}(H)$, we thus have for any $a \in \Phi(I)$ the relation
\[
 \eta(\Phi)_H \circ \ind_{H,I}^\Phi(a) = [U, \res_{U,H}^\Phi \circ \ind_{H,I}^\Phi(a)] = [U, \sum_{h \in T} \ind_{U, U \cap {^h \! I}}^\Phi \circ \con_{h, U^h \cap I}^\Phi \circ \res_{U^h \cap I, I}^\Phi(a)] 
\]
\[
= [U, \sum_{h \in T} \ind_{U,U}^\Phi \circ \con_{h,U}^\Phi \circ \res_{U,I}^\Phi(a)] = [U, \sum_{h \in T} \con_{h,U}^\Phi \circ \res_{U,I}^\Phi(a)] = \sum_{h \in T} h [U, \res_{U,I}^\Phi(a)] = 
\]
\[
= \sum_{h \in T} h \eta(\Phi)_I(a)  = \ind_{H,I}^{(\Phi^*)_*} \circ \eta(\Phi)_I(a)
\]

and therefore
\[
\xymatrix{
\Phi(H) \ar[r]^{\eta(\Phi)_H} & (\Phi^*)_*(H) \\
\Phi(I) \ar[u]^{\ind_{H,I}^\Phi} \ar[r]_{\eta(\Phi)_I} & (\Phi^*)_*(I) \ar[u]_{\ind_{H,I}^{(\Phi^*)_*}}
}
\]
commutes. This proves that $\eta(\Phi) \in \hom_{\se{Mack}(\fr{M},\se{Ab})}( \Phi, (\Phi^*)_*)$. Now, let $\varphi \in \hom_{\se{Mack}(\fr{M},\se{Ab})}(\Phi,\Psi)$. Let $H \in \msys{b}$ and $U \in \ca{U}(H)$. Then, for any $a \in \Phi(H)$ we have
\[
\eta(\Psi)_H \circ \varphi_H(a) = [U, \res_{U,H}^\Psi \circ \varphi_H(a)] = [U, \varphi_U \circ \res_{U,H}^\Phi(a)] = \varphi^*([U,\res_{U,H}^\Phi(a)]) 
\]
\[
= ((\varphi^*)_*)_H([U,\res_{U,H}^\Phi(a)]) = (\varphi^*)_* \circ \eta(\Phi)_H(a)
\]

and consequently
\[
\xymatrix{
\Phi \ar[r]^{\varphi} \ar[d]_{\eta(\Phi)} & \Psi \ar[d]^{\eta(\Psi)} \\
(\Phi^*)_* \ar[r]_{(\varphi^*)_*} & (\Psi^*)_*
}
\]
commutes. Hence, $\eta$ is a morphism of functors.

To prove that $(\eps,\eta)$ is a counit-unit pair, it remains to verify the relations given above. If $\Phi \in \se{Mack}(\fr{M},\se{Ab})$, then just by definition of $\eps$ and $\eta$ we have
\[
\xymatrix{
\Phi^* \ni [U,a] \ar@{|->}[rr]^{\eta(\Phi)^*} & & [U, [U,a]] \ar@{|->}[rr]^{\eps(\Phi^*)} & & [U,a],
}
\]

so $\eps(\Phi^*) \circ \eta(\Phi)^* = \id_{\Phi^*}$. In the same way, for $A \in \discgmod$ and $H \in \msys{b}$ we have
\[
\xymatrix{
A_*(H) = A^H \ni a \ar@{|->}[rr]^-{\eta(A_*)_H} & & [U,a] \ar@{|->}[rr]^{\eps(A)_*} & & a,
}
\]
and therefore $\eps(A)_* \circ \eta(A_*) = \id_{A_*}$. This finally proves that $(-)^*:\se{Mack}(\fr{M},\se{Ab}) \ra \discgmod$ is left-adjoint to $(-)_*:\discgmod \ra \se{Mack}(\fr{M},\se{Ab})$ via the co\-unit\--unit pair $(\eps,\eta)$.
\end{proof}

\begin{defn}
Let $\fr{S}$ be a $G$-subgroup system, let $\Phi \in \se{Stab}(\fr{S},\se{Ab})$, let $H \in \ssys{b}$ and let $U \in \ssys{r}(H)$ with $U \lhd H$. Then $\Phi$ is said to have \words{$(H,U)$-Galois descent}{Galois descent} if $\res_{U,H}^\Phi: \Phi(H) \ra \Phi(U)^{H/U}$ is an isomorphism. Here $\Phi(U)^{H/U}$ denotes the invariants with respect to the $H/U$-module structure on $\Phi(U)$ described in \ref{para:mod_structure_on_functor}. Note that $\res_{U,H}^\Phi$ really maps into $\Phi(U)^{H/U}$ since $\con_{h,U}^\Phi \circ \res_{U,H}^\Phi = \res_{U,H}^\Phi \circ \con_{h,H}^\Phi = \res_{U,H}^\Phi$.
\end{defn}

\begin{para}
If $\fr{S}$ is an I-finite $G$-subgroup system and $A \in \discgmod$, then it is easy to see that $\ro{H}_{\fr{S}}^0(A)$ has $(H,U)$-Galois descent for all $H \in \ssys{b}$ and $U \in \ssys{r}(H)$ with $U \lhd H$.
\end{para}

\begin{defn} \wordsym{$\se{Mack}^{\ca{U}\tn{-}\ro{Gal}}(\fr{M},\se{Ab})$}
Let $\ca{U}$ be a descent basis for $\fr{M}$. A Mackey functor $\Phi \in \se{Mack}(\fr{M},\se{Ab})$ is said to have \words{$\ca{U}$-Galois descent}{Mackey functor!with Galois descent} if $\Phi$ has $(H,U)$-Galois descent for all $H \in \msys{b}$ and $U \in \ca{U}(H)$. The full subcategory of $\se{Mack}(\fr{M},\se{Ab})$ consisting of Mackey functors having $\ca{U}$-Galois descent is denoted by $\se{Mack}^{\ca{U}\tn{-}\ro{Gal}}(\fr{M},\se{Ab})$.
\end{defn}

\begin{prop} \label{para:discgmod_galois_descent}
Let $\ca{U}$ be a descent basis for $\fr{M}$ The following holds:
\begin{enumerate}[label=(\roman*),noitemsep,nolistsep]
\item The adjunction $(\eps,\eta): (-)^*_{\ca{U}} \dashv (-)_*$ induces an equivalence between the full subcategory of $\discgmod$ consisting of all $A \in \discgmod$ such that $A = \bigcup_{U \in \ca{U}} A^U$ and the category $\se{Mack}^{\ca{U}\tn{-}\ro{Gal}}(\fr{M},\se{Ab})$.
\item If $\ca{U}$ is a neighborhood basis of $1 \in G$, then $(\eps,\eta):(-)^*_{\ca{U}} \dashv (-)_*$ induces an equivalence between $\discgmod$ and $\se{Mack}^{\ca{U}\tn{-}\ro{Gal}}(\fr{M},\se{Ab})$.
\end{enumerate}
\end{prop}

\begin{proof} Again we will write $(-)^*$ for $(-)^*_{\ca{U}}$.

\begin{asparaenum}[(i)]
\item Let $\ca{C}$ be the full subcategory of $\discgmod$ consisting of all objects $A \in \discgmod$ such that $\eps(A)$ is an isomorphism and let $\ca{D}$ be the full subcategory of $\se{Mack}(\fr{M},\se{Ab})$ consisting of all objects $\Phi \in \se{Mack}(\fr{M},\se{Ab})$ such that $\eta(\Phi)$ is an isomorphism. The adjunction $(-)^* \dashv (-)_*$ induces an equivalence between the categories $\ca{C}$ and $\ca{D}$ because the conditions on $\ca{C}$ and $\ca{D}$ imply that $A_* \in \ca{D}$ for all $A \in \ca{C}$ and $\Phi^* \in \ca{C}$ for all $\Phi \in \ca{D}$ so that that $(-)^*$ and $(-)_*$ restrict to functors between $\ca{C}$ and $\ca{D}$ and are equivalences with a natural isomorphism given by the counit-unit pair. Thus, it remains to show that $\ca{D} = \se{Mack}^{\ca{U}\tn{-}\ro{Gal}}(\fr{M},\se{Ab})$ and that $\ca{C}$ consists of all $A \in \discgmod$ such that $A = \bigcup_{U \in \ca{U}} A^U$.

First, suppose that $\Phi \in \se{Mack}^{\ca{U}\tn{-}\ro{Gal}}(\fr{M},\se{Ab})$. We have to show that $\eta(\Phi): \Phi \ra (\Phi^*)_*$ is an isomorphism. According to \ref{para:functor_iso_on_values} it is enough to prove that $\eta(\Phi)_H: \Phi(H) \ra (\Phi^*)_*(H) = (\Phi^*)^H$ is an isomorphism for each $H \in \msys{b}$. Let $H \in \msys{b}$ and let $U \in \ca{U}(H)$. If $a,b \in \Phi(H)$ such that $\eta(\Phi)_H(a) = \eta(\Phi)_H(b)$, then $[U,\res_{U,H}^\Phi(a)] = [U,\res_{U,H}^\Phi(b)]$ and so there exists $W \in \ca{U}$ with $W \leq U \cap V$ and $\res_{W,H}^\Phi(a) = \res_{W,H}^\Phi(b)$. As $\Phi$ has Galois descent and as $W \in \ca{U}(H)$, the morphism $\res_{W,H}^\Phi:\Phi(H) \ra \Phi(W)^{H/W}$ is an isomorphism and therefore $a=b$. Hence, $\eta(\Phi)_H$ is a monomorphism. Now, let $[U,a] \in (\Phi^*)_*(H) = (\Phi^*)^H$. Then $U \in \ca{U}$ and $a \in \Phi(U)$. Moreover, for any $h \in H$ we have
\[
h[U,a] = [U,a] \lRA [U, \con_{h,U}^\Phi(a)] = [U,a]
\]
and so there exists $V \in \ca{U}(U)$ such that $\res_{V,U}^\Phi \circ \con_{h,U}^\Phi(a) = \res_{V,U}^\Phi(a)$. This implies $\con_{h,V}^\Phi \circ \res_{V,U}^\Phi(a) = \res_{V,U}^\Phi(a)$ and therefore $\res_{V,U}^\Phi(a) \in \Phi(V)^{H/V}$. As $\Phi$ has Galois descent and $V \in \ca{U}(H)$, the morphism $\res_{V,H}^\Phi: \Phi(H) \ra \Phi(V)^{H/V}$ is an isomorphism. Hence, there exists $b \in \Phi(H)$ with $\res_{V,H}^\Phi(b) = \res_{V,U}^\Phi(a)$ and therefore
\[
[U,a] = [V,\res_{V,H}^\Phi(b)] = \eta(\Phi)_H(b).
\] 

This shows that $\eta(\Phi)_H$ is an epimorphism. Thus, $\eta(\Phi)$ is an isomorphism and consequently $\Phi \in \ca{D}$. On the other hand, let $\Phi \in \ca{D}$. Then $\eta(\Phi)$ is an isomorphism and so $\eta(\Phi)_H: \Phi(H) \ra (\Phi^*)_*(H) = (\Phi^*)^H$ is an isomorphism for each $H \in \msys{b}$. We have to show that $\Phi$ has Galois descent. Suppose, there exists $H \in \msys{b}$ and $U \in \ca{U}(H)$ such that $\res_{U,H}^\Phi(a) = \res_{U,H}^\Phi(b)$ for some $a,b \in \Phi(H)$. Then 
\[
\eta(\Phi)_H(a) = [U, \res_{U,H}^\Phi(a)] = [U, \res_{U,H}^\Phi(b) ] = \eta(\Phi)_H(b)
\]

and this is a contradiction to the fact that $\eta(\Phi)_H$ is a monomorphism. Hence, $\res_{U,H}^\Phi$ is a monomorphism for each $H \in \msys{b}$ and $U \in \ca{U}(H)$. Now, let $H \in \msys{b}$, $U \in \ca{U}(H)$ and let $a \in \Phi(U)^{H/U}$. Then $[U,a] \in (\Phi^*)^H$ and as $\eta(\Phi)_H$ is an epimorphism, there exists $b \in \Phi(H)$ such that $[U,a] = \eta(\Phi)_H(b) = [U, \res_{U,H}^\Phi(b)]$. Thus, there exists $W \in \ca{U}(U)$ such that $\res_{W,U}^\Phi(a) = \res_{W,U}^\Phi \circ \res_{U,H}^\Phi(b)$. Since $\res_{W,U}^\Phi$ is a monomorphism by the above, this implies $a = \res_{U,H}^\Phi(b)$. Hence, $\res_{U,H}^\Phi:\Phi(H) \ra \Phi(U)^{H/U}$ is an epimorphism and thus an isomorphism. Consequently, $\Phi$ has Galois descent. \\

Now, let $A \in \discgmod$ such that $A = \bigcup_{U \in \ca{U}} A^U$. The morphism $\eps(A):(A_*)^* \ra A$ is given by $[U,a] \mapsto a$. If $\eps(A)([U,a]) = \eps(A)([V,b])$, then $a=b$ and therefore $[U,a] = [V,b]$. Hence, $\eps(A)$ is a monomorphism. Let $a \in A$. By assumption, there exists $U \in \ca{U}$ such that $a \in A^U$. Hence, $a \in A^U = A_*(U)$ and therefore $[U,a] \in (A_*)^*$. As $\eps(A)([U,a]) = a$, this shows that $\eps(A)$ is an epimorphism. Consequently, $\eps(A)$ is an isomorphism and we have $A \in \ca{C}$.

On the other hand, let $A \in \ca{C}$. The morphism $\eps(A): (A_*)^* \ra A$ is given by $[U,a] \mapsto a$ and is an isomorphism by assumption. Hence, if $a \in A$, there exists $[U,b] \in (A_*)^*$ such that $\eps(A)([U,b]) = a$. By definition, $U \in \ca{U}$ and $b \in A_*(U) = A^U$. As $\eps(A)([U,b]) = b$, this shows that $a = b \in A^U$. Consequently, $A = \bigcup_{U \in \ca{U}} A^U$.

\item Let $A \in \discgmod$. Since $A$ is discrete, we have $A = \bigcup_{U \in \ca{U}'} A^U$ where $\ca{U}'$ is the set of open subgroups of $G$ (confer \ref{para:char_disc_modules}). As $\ca{U}'$ consists of open subgroups  and as $\ca{U}$ is a filter basis of $1 \in G$, there exists for each $U \in \ca{U}'$ some $V \in \ca{U}$ such that $V \leq U$. Hence, $A = \bigcup_{V \in \ca{U}} A^V$ and consequently the category $\ca{C}$ from above is in this case equal to $\discgmod$.
\end{asparaenum} \vspace{-\baselineskip}
\end{proof}

\begin{cor}
Let $\ca{U}$ be a descent basis for $\fr{M}$. Then every $\Phi \in \se{Mack}^{\ca{U}\tn{-}\ro{Gal}}(\fr{M},\se{Ab})$ is cohomological.
\end{cor}

\begin{proof}
By the above $\Phi$ is isomorphic to a functor of the form $A_*$ for some $A \in \discgmod$ and as this functor is cohomological by \ref{para:discgmod_to_functor}, it follows that $\Phi$ is also cohomological.
\end{proof}

\begin{cor} \label{para:cor_profinite_descent_equiv}
If $G$ is profinite and if $\ca{U}$ is the set of all open normal subgroups of $G$, then the categories $\discgmod$ and $\se{Mack}^{\ca{U}\tn{-}\ro{Gal}}(\ro{Grp}(G)^{\tn{f}},\se{Ab})$ are equivalent.
\end{cor}

\begin{proof}
As $G$ is quasi compact, the set $\ca{U}$ is indeed a descent basis for $\ro{Grp}(G)^{\tn{f}}$ by \ref{para:r_basis_for_qc}. Moreover, since $G$ is profinite, it follows from \ref{para:char_of_pro_finite} that $\ca{U}$ is a neighborhood basis of $1 \in G$. Now, the statement follows from \ref{para:discgmod_galois_descent}.
\end{proof}

\begin{conv}
When considering a profinite group $G$ and the $G$-Mackey system $\ro{Grp}(G)^{\tn{f}}$, then we will always use as descent basis $\ca{U}$ the set of all open normal subgroups of $G$ and we remove the reference to $\ca{U}$ from all notations if nothing else is mentioned.  
\end{conv}

\subsection{Abelianizations} \label{sect:abelianization}

\begin{para}

We will discuss in this section a general framework for abelianizations of topological groups. More precisely, we will consider a $G$-Mackey system $\fr{M} \leq \ro{Grp}^{\tn{r-f}}(G)$ for a topological group $G$ and then turn the assignment $\msys{b} \ra \se{TAb}$, $H \mapsto H/\ro{R}(H)$, for general coabelian subgroups $\ro{R}(H) \lhd H$ into a Mackey functor $\fr{M} \ra \se{TAb}$ by using as induction and conjugation morphisms the obvious morphisms and as restriction morphism the transfer map which is only defined for subgroups of finite index so that we had to assume that $\fr{M}$ is R-finite. But of course the family $\lbrace \ro{R}(H) \mid H \in \msys{b} \rbrace$ has to satisfy several conditions to make this work. For example, if $H \in \msys{b}$ and $I \in \msys{i}(H)$, then we must have $\ro{R}(I) \leq \ro{R}(H)$ to make the canonical inclusion $I \rightarrowtail H$ induce a morphism $I/\ro{R}(I) \ra H/\ro{R}(H)$. Another condition is needed for the conjugation and restriction morphisms and this leads to the introduction of an abelianization system on $\fr{M}$ which is a family $\lbrace \ro{R}(H) \mid H \in \msys{b} \rbrace$ as above satisfying all the conditions needed to make this construction work. 

The primary example of an abelianization system is of course the family $\lbrace \comm{t}(H) \mid H \in \msys{b} \rbrace$, where $\comm{t}(H)$ is the topological commutator subgroup of $H$.\footnote{Confer \ref{sec:top_commutator}} But in the same framework we can now also discuss for a compact group $G$ and a variety $\ca{A}$ of compact abelian groups the functor assigning to each $H \in \msys{b}$ the maximal complete proto-$\ca{A}$ quotient of $H$.\footnote{Confer the appendix \ref{sec:proj_limits} and in particular \ref{sec:proto_for_compact} and \ref{para:maximal_proto_c_for_variety} for these notions. However, the reader may simply ignore this additional application or assume $G$ to be profinite so that this becomes the theory of maximal pro-$\ca{A}$ quotients as explained in \cite[section 3.4]{RibZal00_Profinite-Groups_0}} This includes for example the case of a profinite group $G$ and the variety $\ca{A}$ of finite abelian $p$-groups for a prime number $p$ so that we get a functor assigning the maximal abelian pro-$p$ quotients. This generalization is of course straightforward as it is just about reducing the conjugation, inclusion and transfer maps modulo a compatible system of coabelian subgroups but it is still nice to have one general framework and notation for abelianizations.

 This section is inspired by \cite[section 4]{Neu94_Micro-primes_0}, \cite[chapter IV, \S6]{Neu99_Algebraic-Number_0}, \cite[chapter 1, \S5]{NeuSchWin08_Cohomology-of-Number_0} and \cite[section 3]{Wei07_Frattini-extensions_0}.
\end{para}

\begin{ass}
Throughout this section we fix a topological group $G$ and a $G$-subgroup system $\fr{S}$ with $\fr{S} \leq \ro{Grp}^{\tn{r-f}}(G)$. Recall that by definition all $H \in \ssys{b}$ are closed in $G$.
\end{ass}

\begin{prop} \label{para:pretransfer} \wordsym{$\ro{V}_{H,G}^T$} \wordsym{$\ro{V}_{H,G}^{\ro{R}(H)}$}
Let $H$ be a closed subgroup of finite index of $G$. Let $T = ( t_1,\ldots,t_n )$ be a right transversal of $H$ in $G$ with a fixed ordering of its elements. The following holds:
\begin{enumerate}[label=(\roman*),noitemsep,nolistsep]

\item The map
\[
\begin{array}{rcl}
\ro{V}_{H,G}^T: G & \lra & H \\
g & \longmapsto & \prod_{i=1}^n \kappa_T(t_ig)
\end{array}
\]
is continuous, where $\kappa_T$ is the map extracting the $H$-part of an element of $G$.\footnote{Confer \ref{para:t_remover} for the definition of $\kappa_T$.} It is called the \word{pretransfer} from $G$ to $H$ with respect to $T$.

\item \label{item:pretranfer_possible_reduction} If $U \lhd G$ and $U \leq H$, then $\ro{V}_{H,G}^T(U) \subs U$.

\item \label{item:pretransfer_mor} Let $\ro{R}(H)$ be a coabelian\footnote{That is, $\ro{R}(H) \lhd H$ and $H/\ro{R}(H)$ is abelian.} subgroup of $H$ and let $q:H \ra H/\ro{R}(H)$ be the quotient morphism. Then $\ro{V}_{H,G}^{\ro{R}(H)} \dopgleich q \circ \ro{V}_{H,G}^T:G \ra H/\ro{R}(H)$ is a morphism of topological groups which is independent of the choice of $T$. 
\end{enumerate}
\end{prop}

\begin{proof} \hfill

\begin{asparaenum}[(i)]

\item By \ref{para:t_remover_continuous_for_finite} the $T$-remover $\kappa_T$ is continuous and as multiplication on $G$ is continuous, the map $G \ra H$, $g \mapsto \kappa_T(t_ig)$ is continuous for all $i \in \lbrace 1, \ldots, n \rbrace$. Hence, $\ro{V}_{H,G}^T$ is continuous.

\item Let $g \in U$. Since $U$ is a normal subgroup of $G$, we have $t_iU = Ut_i$ and so there exists $u_i \in U$ such that $t_ig = u_it_i$ for each $i \in \lbrace 1,\ldots,n \rbrace$. As $U \leq H$, the equation $t_ig = u_it_i$ implies that $\kappa_T(t_ig) = u_i \in U$ and therefore $\ro{V}_{H,G}^T(g) \in U$.

\item Since $\ro{V}_{H,G}^T$ and $q$ are continuous, $\ro{V}_{H,G}^{\ro{R}(H)}$ is also continuous. To prove that this map is multiplicative, let $g_1,g_2 \in G$. Let $f_j = \sigma_{T,g_j} \in \ro{S}_n$ be the $T$-permutation of $g_j$ for $j \in \lbrace 1,2 \rbrace$ (confer \ref{para:t_remover}). Then $t_ig_j = \kappa_T(t_i g_j) t_{f_j(i)}$ for $j \in \lbrace 1,2 \rbrace$ and $i \in \lbrace 1,\ldots,n \rbrace$. For each $i \in \lbrace 1,\ldots,n \rbrace$ we have the following relations
\[
\begin{array}{c}
t_{f_1^{-1}(i)} g_1 = \kappa_T(t_{f_1^{-1}(i)} g_1) t_i \\
g_2 = t_i^{-1} \kappa_T(t_ig_2) t_{ f_2(i) }
\end{array}
\]

and multiplication of these relations yields
\[
t_{f_1^{-1}(i)} g_1 g_2 = \kappa_T(t_{f_1^{-1}(i)} g_1) \kappa_T(t_ig_2) t_{ f_2(i) },
\]
so
\[
\kappa_T(t_{f_1^{-1}(i)} g_1 g_2) = \kappa_T(t_{f_1^{-1}(i)} g_1) \kappa_T(t_ig_2).
\]

Noting that $H/\ro{R}(H)$ is abelian, we get
\begin{align*}
\ro{V}_{H,G}^T(g_1g_2) &\equiv \prod_{i=1}^n \kappa_T(t_i g_1g_2) \equiv \prod_{i=1}^n \kappa_T(t_{f_1^{-1}(i)} g_1 g_2) \equiv \prod_{i=1}^n \kappa_T(t_{f_1^{-1}(i)} g_1) \kappa_T(t_ig_2) \displaybreak[0] \\
&\equiv \prod_{i=1}^n \kappa_T(t_{f_1^{-1}(i)} g_1) \cdot \prod_{i=1}^n \kappa_T(t_ig_2) \equiv \prod_{i=1}^n \kappa_T(t_{i} g_1) \cdot \prod_{i=1}^n \kappa_T(t_ig_2) \displaybreak[0] \\
&\equiv \ro{V}_{H,G}^T(g_1) \cdot \ro{V}_{H,G}^T(g_2) \ \ro{mod} \ \ro{R}(H)
\end{align*}

and therefore
\[
\ro{V}_{H,G}^{\ro{R}(H)}(g_1g_2) = \ro{V}_{H,G}^{\ro{R}(H)}(g_1) \cdot \ro{V}_{H,G}^{\ro{R}(H)}(g_2).
\]
Hence, for fixed $T$ the map $\ro{V}_{H,G}^{\ro{R}(H)}$ is a morphism of topological groups. To prove independence of the choice of $T$, let $T'=( t_1',\ldots,t_n' )$ be another right transversal of $H$ in $G$. Let $g \in G$. Then there exist permutations $f_1,f_2 \in \ro{S}_n$ such that
\begin{equation*} \label{equ:verlagerung_1}
t_i' = \kappa_T(t_i') t_{f_1(i)} 
\end{equation*}
and
\begin{equation} \label{equ:verlagerung_2}
t_ig = \kappa_T(t_ig) t_{f_2(i)}
\end{equation}
 for each $i \in \lbrace 1,\ldots,n \rbrace$. By setting $j = f_1^{-1} f_2 f_1(i)$ in the relation $t_{f_1(j)} = \kappa_T(t_j')^{-1} t_j'$ we get
\begin{equation} \label{equ:verlagerung_3}
t_{f_2f_1(i)} = \kappa_T(t_{f_1^{-1}f_2f_1(i)}')^{-1} t_{f_1^{-1} f_2 f_1(i)}'.
\end{equation}

It follows that
\[
t_i' g = \kappa_T(t_i') t_{f_1(i)} g \underbrace{=}_{(\ref{equ:verlagerung_2})} \kappa_T(t_i') \kappa_T( t_{f_1(i)} g) t_{f_2f_1(i)} \underbrace{=}_{(\ref{equ:verlagerung_3})}  \kappa_T(t_i') \kappa_T( t_{f_1(i)} g) \kappa_T(t_{f_1^{-1}f_2f_1(i)}')^{-1} t_{f_1^{-1} f_2 f_1(i)}'.
\]
and consequently
\[
\kappa_{T'}(t_i'g) = \kappa_T(t_i') \kappa_T( t_{f_1(i)} g) \kappa_T(t_{f_1^{-1}f_2f_1(i)}')^{-1}.
\]

Hence, we get
\begin{align*}
\ro{V}_{H,G}^{T'}(g) &\equiv \prod_{i=1}^n \kappa_{T'}(t_i'g) \equiv \prod_{i=1}^n \kappa_T(t_i') \kappa_T( t_{f_1(i)} g) \kappa_T(t_{f_1^{-1}f_2f_1(i)}')^{-1} \\
&\equiv \prod_{i=1}^n \kappa_T(t_i') \kappa_T( t_{i} g) \kappa_T(t_{i}')^{-1} \equiv \prod_{i=1}^n \kappa_T( t_{i} g) \equiv \ro{V}_{H,G}^T(g) \ \ro{mod} \ \ro{R}(H).
\end{align*}

This proves independence of the choice of $T$.
\end{asparaenum}
\vspace{-\baselineskip}
\end{proof}

\begin{defn}
Let $H$ be a closed subgroup of finite index of $G$. A \word{transfer inducing pair} for $H$ in $G$ consists of a coabelian subgroup $\ro{R}(H)$ of $H$ and a coabelian subgroup $\ro{R}(G)$ of $G$ such that $\ro{R}(G) \subs \ker(\ro{V}_{H,G}^{\ro{R}(H)})$. By \ref{para:pretransfer} this is equivalent to $\ro{V}_{H,G}^T(\ro{R}(G)) \subs \ro{R}(H)$ for one (and then any) right transversal $T$ of $H$ in $G$.
\end{defn}

\begin{prop} \wordsym{$\ver_{H,G}^{\ro{R}(H),\ro{R}(G)}$} \label{para:transfer}
 Let $H$ be a closed subgroup of finite index of $G$ and let $\lbrace \ro{R}(H),\ro{R}(G) \rbrace$ be a transfer inducing pair for $H$ in $G$. The following holds:
\begin{enumerate}[label=(\roman*),noitemsep,nolistsep]
\item The morphism $\ro{V}_{H,G}^{\ro{R}(H)}:G \ra H/\ro{R}(H)$ induces a morphism of topological groups 
\[
\ver_{H,G}^{\ro{R}(H),\ro{R}(G)}: G/\ro{R}(G) \lra H/\ro{R}(H)
\]
which is explicitly given by
\[
g \ \ro{mod} \ \ro{R}(G) \longmapsto \ro{V}_{H,G}^T(g) \ \ro{mod} \ \ro{R}(H)
\]
for a right transversal $T$ of $H$ in $G$. This morphism is called the \word{transfer} from $G$ to $H$ with respect to the transfer inducing pair.\footnote{The German word for transfer is \word{Verlagerung} and so, for historical reasons, we denote the transfer by $\ver$.}

\item \label{item:transfer_id} If $G = H$, then $\ver_{G,G}^{\ro{R}(G),\ro{R}(G)} = \id_{G/\ro{R}(G)}$.

\item \label{item:transfer_transitive} Let $H'$ be another closed subgroup of $G$ such that $H \leq H' \leq G$ and let $\ro{R}(H') \lhd H'$ such that $\lbrace \ro{R}(H'),\ro{R}(G) \rbrace$ is a transfer inducing pair for $H'$ in $G$ and $\lbrace \ro{R}(H),\ro{R}(H') \rbrace$ is a transfer inducing pair for $H$ in $H'$. Then $\ver_{H,H'}^{\ro{R}(H),\ro{R}(H')} \circ \ver_{H',G}^{\ro{R}(H'),\ro{R}(G)} = \ver_{H,G}^{\ro{R}(H),\ro{R}(G)}$.

\end{enumerate}
\end{prop}

\begin{proof} \hfill

\begin{asparaenum}[(i)]
\item This follows immediately from \ref{para:morphism_quotient_induced}.

\item We can choose the right transversal $T = \lbrace 1 \rbrace$ of $G$ in $G$ and then it is obvious that $\ro{V}_{G,G}^T = \id_{G}$. It follows that $\ver_{G,G}^{\ro{R}(G),\ro{R}(G)} = \id_{G/\ro{R}(G)}$.

\item First note that the assumptions imply that both $\ver_{H,H'}^{\ro{R}(H),\ro{R}(H')}$ and $\ver_{H',G}^{\ro{R}(H'),\ro{R}(G)}$ are defined. Let $T = \lbrace t_1,\ldots,t_m \rbrace$ be a right transversal of $H'$ in $G$ and let $U = \lbrace u_1,\ldots, u_n \rbrace$ be a right transversal of $H$ in $H'$. Then $UT = \lbrace u_j t_i \mid i \in \lbrace 1,\ldots,m \rbrace, j \in \lbrace 1,\ldots,n \rbrace \rbrace$ is a right transversal of $H$ in $G$. Let $f_1 = \sigma_{T,g} \in \ro{S}_m$ and $f_{2,i} = \sigma_{UT,\kappa_T(t_ig)} \in \ro{S}_n$. Then
\[
t_i g = \kappa_T(t_ig) t_{f_1(i)}
\]
and
\[
u_j \kappa_T(t_ig) = \kappa_U(u_j \kappa_T(t_ig)) u_{f_{2,i}(j)}
\]
for all $i=1,\ldots,m$ and $j=1,\ldots,n$. Combining both relations yields
\[
u_j t_i g = u_j  \kappa_T(t_ig) t_{f_1(i)} = \kappa_U(u_j \kappa_T(t_ig)) u_{f_{2,i}(j)} t_{f_1(i)} 
\]
and therefore
\[
\kappa_{UT}( u_j t_i g ) = \kappa_U(u_j \kappa_T(t_ig)).
\] 

Hence,
\begin{align*}
\ver_{H,G}^{\ro{R}(H),\ro{R}(G)}(g \ \ro{mod} \ \ro{R}(G)) &= \ro{V}_{H,G}^{UT}(g) \ \ro{mod} \ \ro{R}(H) = \prod_{i=1}^m \prod_{j=1}^n \kappa_{UT}(u_j t_i g) \ \ro{mod} \ \ro{R}(H)  \\
&= \prod_{i=1}^m \prod_{j=1}^n \kappa_U(u_j \kappa_T(t_ig))  \ \ro{mod} \ \ro{R}(H) = \prod_{i=1}^m \ro{V}_{H,H'}^{U}( \kappa_T(t_i)g)  \ \ro{mod} \ \ro{R}(H) \\
&= \prod_{i=1}^m \ver_{H,H'}^{\ro{R}(H),\ro{R}(H')}( \kappa_T(t_ig) \ \ro{mod} \ \ro{R}(H') ) \\
&= \ver_{H,H'}^{\ro{R}(H),\ro{R}(H')}( \prod_{i=1}^m \kappa_T(t_ig) \ \ro{mod} \ \ro{R}(H') ) \\
&= \ver_{H,H'}^{\ro{R}(H),\ro{R}(H')}( \ro{V}_{H',G}^T(g) \ \ro{mod} \ \ro{R}(H') ) \\
&= \ver_{H,H'}^{\ro{R}(H),\ro{R}(H')}( \ver_{H',G}^{\ro{R}(H'),\ro{R}(G)}(g \ \ro{mod} \ \ro{R}(G))).
\end{align*}

\end{asparaenum}
\end{proof}

\begin{defn} \label{para:ab_system} \wordsym{$\se{Ab}(\fr{S})$}
An \word{abelianization system} on $\fr{S}$ is a family $\ro{R} = \lbrace \ro{R}(H) \mid H \in \ssys{b} \rbrace$, where $\ro{R}(H)$ is a closed coabelian subgroup of $H$ for each $H \in \ssys{b}$, satisfying the following properties:
\begin{enumerate}[label=(\roman*),noitemsep,nolistsep]
\item \label{item:ab_system_conj} $\ro{R}(^g \! H) = g \ro{R}(H) g^{-1}$ for all $g \in G$.
\item \label{item:ab_system_res} $\ro{R}(H) \subs \ker(\ro{V}_{I,H}^{\ro{R}(I)})$ for each $H \in \ssys{b}$ and $I \in \ssys{r}(H)$.\footnote{Note that as $\fr{S} \leq \ro{Grp}(G)^{\tn{r-f}}$, we know that $I$ is a closed subgroup of finite index of $H$ and therefore $\ver_{I,H}^{\ro{R}(H),\ro{R}(I)}$ is defined and continuous.}\footnote{This condition is equivalent to $\ro{V}_{I,H}^T(\ro{R}(H)) \subs \ro{R}(I)$ for one (any) right transversal $T$ of $I$ in $H$.}
\item \label{item:ab_system_ind} $\ro{R}(I) \subs \ro{R}(H)$ for each $H \in \ssys{b}$ and $I \in \ssys{i}(H)$. 
\end{enumerate}
The set of all abelianization systems on $\fr{S}$ is denoted by $\se{Ab}(\fr{S})$.
\end{defn}

\begin{prop} \wordsym{$\ro{R}_{\ca{A}}$} \wordsym{$\comm{t}$}
The following holds:
\begin{compactenum}[(i)]
\item For $H \in \ssys{b}$ let $\comm{t}(H)$ be the topological commutator subgroup of $H$ (confer \ref{para:commutator_subgroup}). Then $\comm{t,\fr{S}} \dopgleich \lbrace \comm{t}(H) \mid H \in \ssys{b} \rbrace$ is an abelianization system on $\fr{S}$. For each $H \in \ssys{b}$ the quotient $H/\comm{t}(H)$ is the maximal separated abelian quotient of $H$.
\item Assume that $G$ is compact. Let $\ca{A}$ be a variety of compact abelian groups (confer \ref{para:def_of_formation_variety}). For $H \in \ssys{b}$ let $\ro{R}_{\ca{A}}(H)$ be the $\ca{A}$-radical of $H$ (confer \ref{prop:maximal_cplproto_quotient}). Then $\ro{R}_{\ca{A}}^{\fr{S}} \dopgleich \lbrace \ro{R}_{\ca{A}}(H) \mid H \in \ssys{b} \rbrace$ is an abelianization system on $\fr{S}$. For each $H \in \ssys{b}$ the quotient $H/\ro{R}_{\ca{A}}(H)$ is the maximal complete proto-$\ca{A}$ quotient of $H$.
\end{compactenum}
\end{prop}

\begin{proof} 
This follows from \ref{para:props_of_commuator_subgroup}\ref{item:top_commutator_is_normal} and \ref{para:props_of_commuator_subgroup}\ref{item:top_commutator_mor} respectively from \ref{para:maximal_proto_c_for_variety}\ref{item:maximal_proto_c_mor}.
\end{proof}

\begin{prop} \label{para:non_fibered_r_ab} \wordsym{$\pi_{\ro{R}}$}
If $\ro{R} \in \se{Ab}(\fr{S})$, then the following data define a cohomological and stable functor $\pi_{\ro{R}}:\fr{S} \ra \se{TAb}$:
\begin{compactitem}
\item $\pi_{\ro{R}}:\ssys{b} \ra \se{TAb}$ is given by $H \mapsto H/\ro{R}(H)$.
\item $\con_{g,H}^{\pi_{\ro{R}}}: \pi_{\ro{R}}(H) \ra \pi_{\ro{R}}(^g \! H)$ is induced by the conjugation $H \ra {^g \! H}$, $h \mapsto ghg^{-1}$, for each $H \in \ssys{b}$ and $g \in G$.
\item $\res_{I,H}^{\pi_{\ro{R}}} \dopgleich \ver_{I,H}^{\ro{R}(I),\ro{R}(H)}: \pi_{\ro{R}}(H) \ra \pi_{\ro{R}}(I)$ for each $H \in \ssys{b}$ and $I \in \ssys{r}(H)$.
\item $\ind_{H,I}^{\pi_{\ro{R}}}:\pi_{\ro{R}}(I) \ra \pi_{\ro{R}}(H)$ is induced by the inclusion $I \rightarrowtail H$ for each $H \in \ssys{b}$ and $I \in \ssys{i}(H)$. \\
\end{compactitem}

\end{prop}

\begin{proof}
The given data is well-defined by definition of an abelianization system.
The triviality condition is obviously satisfied by the conjugation and induction morphisms and \ref{para:transfer}\ref{item:transfer_id} implies that it is also satisfied by the restriction morphisms. The transitivity condition is again obviously satisfied by the conjugation and induction morphisms and \ref{para:transfer}\ref{item:transfer_transitive} implies that it is also satisfied by the restriction morphisms. The equivariance condition obviously holds for the induction morphisms and to prove it for the restriction morphisms, let $g \in G$, $H \in \ssys{b}$, $I \in \ssys{r}(H)$ and let $T$ be a right transversal of $I$ in $H$. Then $^gT$ is a right transversal of $^g \! I$ in $^g \! H$. If $h \in H$ and $t \in T$, then $th = \kappa_T(th) t'$ for some $t' \in T$ and therefore
\[
gthg^{-1} = g \kappa_T(th) t' g^{-1} = \left( g \kappa_T(th) g^{-1} \right) \left( g t' g^{-1} \right) \in {^g \! I}{ ^g T}
\]
and this implies $\kappa_{^g T}(gthg^{-1}) = g \kappa_T(th)g^{-1}$. It follows that
\begin{align*}
& \res_{^g \! I,^g \! H}^{\pi_{\ro{R}}} \circ \con_{g,H}^{\pi_{\ro{R}}}( h \ \ro{mod} \ \ro{R}(H) ) = \res_{^g \! I,^g \! H}^{\pi_{\ro{R}}}( ghg^{-1} \ \ro{mod} \ \ro{R}(^g \!H) ) \\
= \ & \ver_{^g \! I, {^g \! H}}^{\ro{R}(^g \! I),\ro{R}(^g \! H)}( ghg^{-1} \ \ro{mod} \ \ro{R}(^g \!H) ) = \prod_{t \in T} \kappa_{^g T}( (gtg^{-1})  (ghg^{-1}) ) \ \ro{mod} \ \ro{R}(^g \! I) \\
= \ & \prod_{t \in T} \kappa_{^g T}( gthg^{-1} ) \ \ro{mod} \ \ro{R}(^g \! I) = \prod_{t \in T} g \kappa_{T}( th) g^{-1} \ \ro{mod} \ \ro{R}(^g \! I)  \\
= \ & g ( \prod_{t \in T}  \kappa_{T}( th) ) g^{-1} \ \ro{mod} \ \ro{R}(^g \! I) = \con_{g,I}^{\pi_{\ro{R}}}( \prod_{t \in T}  \kappa_{T}( th) \ \ro{mod} \ \ro{R}(I) ) \\
= \ &  \con_{g,I}^{\pi_{\ro{R}}} \circ \ver_{I,H}^{\ro{R}(I),\ro{R}(H)}(h \ \ro{mod} \ \ro{R}(H)) = \con_{g,I}^{\pi_{\ro{R}}} \circ \res_{I,H}^{\pi_{\ro{R}}}( h   \ \ro{mod} \ \ro{R}(H) ).
\end{align*}

To show that $\pi_{\ro{R}}$ is cohomological, first note that $\se{TAb}$ is canonically preadditive and therefore the cohomologicality property is defined. Let $H \in \ssys{b}$, $I \in \ssys{r}(H) \cap \ssys{i}(H)$ and let $T=\lbrace t_1,\ldots,t_n \rbrace$ be a right transversal of $I$ in $H$. For $h \in H$ let $f = \sigma_{h,T} \in \ro{S}_n$. Then $t_ih = \kappa_T(t_ih) t_{f(i)}$ for all $i \in \lbrace 1,\ldots,n \rbrace$. Hence, $\kappa_T(t_ih) = t_ih t_{f(i)}^{-1}$ and thus we get
\begin{align*}
&  \ind_{H,I}^{\pi_{\ro{R}}} \circ \res_{I,H}^{\pi_{\ro{R}}}( h \ \ro{mod} \ \ro{R}(H) ) =  \ind_{H,I}^{\pi_{\ro{R}}} \circ \ver_{I,H}^{\ro{R}(I),\ro{R}(H)}(h \ \ro{mod} \ \ro{R}(H))  \\
= \ & \ind_{H,I}^{\pi_{\ro{R}}}( \prod_{i=1}^n \kappa_T(t_i h) \ \ro{mod} \ \ro{R}(I) ) = \prod_{i=1}^n \kappa_T(t_i h) \ \ro{mod} \ \ro{R}(H) \\
= \ & \prod_{i=1}^n t_ih t_{f(i)}^{-1}  \ \ro{mod} \ \ro{R}(H) = \prod_{i=1}^n h t_i t_{f(i)}^{-1}  \ \ro{mod} \ \ro{R}(H) \\
= \ & \prod_{i=1}^n h t_i t_{i}^{-1}  \ \ro{mod} \ \ro{R}(H) = \prod_{i=1}^n h  \ \ro{mod} \ \ro{R}(H) = (h  \ \ro{mod} \ \ro{R}(H))^{[H:I]}.
\end{align*}

Finally, $\pi_R$ is stable since $^h \! H/\ro{R}(^h \! H) = H/\ro{R}(H)$ is abelian for any $h \in H$ and consequently 
\[
\con_{h,H}^{\pi_{\ro{R}}}(x \modd \ro{R}(H)) = hxh^{-1} \modd \ro{R}(H) = x \modd \ro{R}(H)
\]
for all $x \in H$.
\end{proof}

\begin{prop}
If $\fr{M}$ is a $G$-Mackey system with $\fr{M} \leq \ro{Grp}(G)^{\tn{r-f}}$ and if $\ro{R} \in \se{Ab}(\fr{M})$, then $\pi_{\ro{R}}$ is a cohomological Mackey functor.
\end{prop}

\begin{proof}
It remains to verify the Mackey formula. Let $H \in \msys{b}$, $I \in \msys{r}(H)$, $J \in \msys{i}(H)$ and let $R = \lbrace \rho_1,\ldots,\rho_n \rbrace$ be a complete set of representatives of $I \mybackslash H / J$. For each $i \in \lbrace 1,\ldots,n \rbrace$ let $T_i = \lbrace t_{i,1},\ldots,t_{i,m_i} \rbrace$ be a right transversal of $I^{\rho_i} \cap J$ in $J$. Then $T = \lbrace \rho_i t_{i,j} \mid i \in \lbrace 1,\ldots,n \rbrace \tn{ and } j \in \lbrace 1,\ldots,m_i \rbrace \rbrace$ is a right transversal of $I$ in $H$ by \ref{para:double_coset_rep_lift}. Let $x \in J$. For each $i \in \lbrace 1,\ldots,n \rbrace$ let $f_i = \sigma_{T_i,x} \in \ro{S}_{m_i}$. Then 
\[
 t_{i,j} x = \kappa_{T_i}( t_{i,j} x) t_{i,f_i(j)}
\]
for all $j \in \lbrace 1,\ldots,m_i \rbrace$. Multiplication with $\rho_i$ yields
\[
 \rho_i t_{i,j} x = \rho_i \kappa_{T_i}( t_{i,j} x) t_{i,f_i(j)} = ( \rho_i \kappa_{T_i}( t_{i,j} x) \rho_i^{-1}) (\rho_i t_{i,f_i(j)}).
\]

Since $\rho_i \kappa_{T_i}( t_{i,j} x) \rho_i^{-1} \in {^{\rho_i} \! (I^{\rho_i} \cap J) } = I \cap {^{\rho_i} \! J} \leq I$ and $\rho_i t_{i,f_i(j)} \in T$, it follows that 
\[
\kappa_T(\rho_i t_{i,j} x) = \rho_i \kappa_{T_i}( t_{i,j} x) \rho_i^{-1}.
\]

Hence,
\begin{align*}
& \prod_{i=1}^n \ind_{I, I \cap {^{\rho_i} \! J}}^{\pi_{\ro{R}}} \circ \con_{\rho_i, I^{\rho_i} \cap J}^{\pi_{\ro{R}}} \circ \res_{I^{\rho_i} \cap J, J}^{\pi_{\ro{R}}}(x \ \ro{mod} \ \ro{R}(J)) \\
= \ & \prod_{i=1}^n \ind_{I, I \cap {^{\rho_i} \! J}}^{\pi_{\ro{R}}} \circ \con_{\rho_i, I^{\rho_i} \cap J}^{\pi_{\ro{R}}} \circ \ver_{I^{\rho_i} \cap J, J}^{\ro{R}(I^{\rho_i} \cap J), \ro{R}(J)}(x \ \ro{mod} \ \ro{R}(J)) \\
= \ & \prod_{i=1}^n \ind_{I, I \cap {^{\rho_i} \! J}}^{\pi_{\ro{R}}} \circ \con_{\rho_i, I^{\rho_i} \cap J}^{\pi_{\ro{R}}} \left( \prod_{j=1}^{m_i} \kappa_{T_i}(t_{i,j} x) \ \ro{mod} \ \ro{R}( I^{\rho_i} \cap J ) \right) \\
= \ & \prod_{i=1}^n \ind_{I, I \cap {^{\rho_i} \! J}}^{\pi_{\ro{R}}} \left( \rho_i \left( \prod_{j=1}^{m_i} \kappa_{T_i}(t_{i,j} x)  \right) \rho_i^{-1} \ \ro{mod} \ \ro{R}( I \cap {^{\rho_i} \! J} ) \right) \\
= \ & \prod_{i=1}^n \ind_{I, I \cap {^{\rho_i} \! J}}^{\pi_{\ro{R}}} \left( \prod_{j=1}^{m_i} \rho_i \kappa_{T_i}(t_{i,j} x) \rho_i^{-1} \ \ro{mod} \ \ro{R}( I \cap {^{\rho_i} \! J} ) \right) \\
= \ & \prod_{i=1}^n \prod_{j=1}^{m_i} \rho_i \kappa_{T_i}(t_{i,j} x) \rho_i^{-1} \ \ro{mod} \ \ro{R}(I) = \prod_{i=1}^n \prod_{j=1}^{m_i} \kappa_T(\rho_i t_{i,j} x) \ \ro{mod} \ \ro{R}(I) \\
= \ & \ver_{I,H}^{\ro{R}(I),\ro{R}(H)}(x \ \ro{mod} \ \ro{R}(H)) = \res_{I,H}^{\pi_{\ro{R}}}( x  \ \ro{mod} \ \ro{R}(H) ) = \res_{I,H}^{\pi_{\ro{R}}} \circ \ind_{H,J}^{\pi_{\ro{R}}}(  x  \ \ro{mod} \ \ro{R}(J) ).
\end{align*}

\end{proof}

\begin{para}
By the above we have in particular 
\[
\pi_{\ro{ab}}^{\fr{S}} \dopgleich \pi_{\comm{t,\fr{S}}} \in \se{Stab}^{\ro{c}}(\fr{S},\se{TAb}^{\ro{s}}).
\]
Moreover, we have
\[
\pi_{\ro{ab}}^{\star,G} \dopgleich \pi_{\ro{ab}}^{\ro{Grp}(G)^{\star}} \in \se{Mack}^{\ro{c}}(\ro{Grp}(G)^{\star},\se{TAb}^{\ro{s}})
\]
for $\star \in \lbrace \tn{r-f},\tn{ri-f},\tn{f} \rbrace$. Although by definition $\pi_{\ro{ab}}^{\fr{S}}$ is just the restriction of $\pi_{\ro{ab}}^{\tn{r-f},G}$ to $\fr{S}$, the reason for keeping track of $\fr{S}$ in the notation is to make precise, where this functor lives. This is important when considering morphisms between functors. If $G$ is compact and if $\ca{A}$ is a variety of compact abelian groups, then we define similarly $\pi_{\ca{A}}^{\fr{S}}$ and $\pi_{\ca{A}}^{\star,G}$.
\end{para}

\begin{para}
If $\ro{R} \in \se{Ab}(\fr{S})$ and $\fr{S}' \leq \fr{S}$, then obviously $\ro{R}|_{\fr{S}'} \dopgleich \lbrace \ro{R}(H) \mid H \in \ssys{b}' \rbrace$ is an abelianization system on $\fr{S}'$ and $\pi_{\ro{R}}|_{\fr{S}'} = \pi_{\ro{R}|_{\fr{S}'}}$.
\end{para}

\begin{prop}
If $\ro{R} \in \se{Ab}(\fr{S})$, then $\pi_{\ro{R}}$ is canonically isomorphic to a quotient of $\pi_{\ro{ab}}^{\fr{S}}$ by a closed subfunctor. Conversely, a quotient of $\pi_{\ro{ab}}^{\fr{S}}$ by a closed subfunctor is canonically isomorphic to $\pi_{\ro{R}}$ for some $\ro{R} \in \se{Ab}(\fr{S})$.
\end{prop}

\begin{proof} 
To simplify notations, we set $\pi_{\ro{ab}} \dopgleich \pi_{\ro{ab}}^{\fr{S}}$.
Let $\ro{R} \in \se{Ab}(\fr{S})$ and define $\Phi: \ssys{b} \ra \se{TAb}$, $H \mapsto R(H)/\comm{t}(H)$. Since $H/\ro{R}(H)$ is a separated abelian group, we have $\ro{R}(H) \geq \comm{t}(H)$ and therefore $\Phi$ is well-defined. An application of \ref{prop:prop_of_quots}\ref{item:top_grp_quot_image_closed_open} shows that $\Phi(H)$ is a closed subgroup of $H/\comm{t}(H)$. Now, it is easy to verify that $\Phi$ is a closed subfunctor of $\pi_{\ro{ab}}$ and that the canonical isomorphisms
\[
\pi_{\ro{R}}(H) = H/\ro{R}(H) \ra (H/\comm{t}(H))/(\ro{R}(H)/\comm{t}(H)) = \pi_{\ro{ab}}(H)/\Phi(H)
\]
define an isomorphism $\pi_{\ro{R}} \cong \pi_{\ro{ab}}/\Phi$.

Conversely, let $\Phi$ be a closed subfunctor of $\pi_{\ro{ab}}$. Then $\Phi(H) = \ro{R}(H)/\comm{t}(H)$ for some closed subgroup $\ro{R}(H)$ of $H$ with $\ro{R}(H) \geq \comm{t}(H)$. It is easy to verify that $\ro{R} = \lbrace \ro{R}(H) \mid H \in \ssys{b} \rbrace$ is an abelianization system on $\fr{S}$ and that the canonical isomorphisms
\[
\pi_{\ro{ab}}(H)/\Phi(H) = (H/\comm{t}(H))/(\ro{R}(H)/\comm{t}(H)) \ra H/\ro{R}(H) = \pi_{\ro{R}}(H)
\]
define an isomorphism $\pi_{\ro{ab}}/\Phi \cong \pi_{\ro{R}}$.
\end{proof}

\begin{para}
The above shows that the R-abelianizations $\pi_{\ro{R}}$ give a particular presentation of the quotients of $\pi_{\ro{ab}}^{\fr{S}}$ by closed subfunctors. Hence, we could replace the discussion above by a definition of the functor $\pi_{\ro{ab}}^{\fr{S}}$ and a definition of R-abelianizations as quotients of $\pi_{\ro{ab}}^{\fr{S}}$ by closed subfunctors. But on the one hand, the definition of $\pi_{\ro{ab}}^{\fr{S}}$ and $\pi_{\ro{R}}$ is very similar (except that we had to define the notion of abelianization systems) and on the other hand, the presentation of the quotients in the form $\pi_{\ro{R}}$ is easier and more natural to work with. 
\end{para}

\begin{para}
In the following proposition we give an alternative presentation of $\ro{V}_{H,G}^{\ro{R}(H)}$ that will later be used in the discussion of Fesenko--Neukirch class field theories.
\end{para}

\begin{prop} \label{para:transfer_alternative_rep} \wordsym{$\lambda_g(\rho)$}
Let $H$ be a closed subgroup of finite index of $G$. Let $g \in G$ and let $R$ be a complete set of representatives of $H \mybackslash G / \langle g \rangle$. For $\rho \in R$ define
\[
\lambda_g(\rho) \dopgleich \ro{min} \ \lbrace j \mid j \in \NN_{>0} \tn{ and } \rho g^j \rho^{-1} \in H \rbrace.
\]
The following holds:
\begin{enumerate}[label=(\roman*),noitemsep,nolistsep]
\item Let $j = k \lambda_g(\rho) + q$ with $k \in \ZZ$ and $q \in \lbrace 0,\ldots, \lambda_g(\rho) -1 \rbrace$. Then 
\[
H \rho g^j = \left \lbrace \begin{array}{ll} H \rho g^q & \tn{if } q \neq 0 \\ H \rho g^{\lambda_g(\rho)} & \tn{if } q = 0. \end{array} \right.
\]
Moreover, $\rho g^j \rho^{-1} \in H$ if and only if $q=0$, that is, $j \in \lambda_g(\rho)\ZZ$.
\item The set $T = \lbrace \rho g^j \mid \rho \in R \tn{ and } j \in \lbrace 1,\ldots,\lambda_g(\rho) \rbrace \rbrace$ is a right transversal of $H$ in $G$.
\item The relation
\[
\ro{V}_{H,G}^{\ro{R}(H)}(g) = \prod_{\rho \in R}^n \rho g^{\lambda_g(\rho)} \rho^{-1} \modd \ro{R}(H)
\]
holds.

\end{enumerate}
\end{prop}

\begin{proof}
First note that as $\lbrack G:H \rbrack < \infty$, there exists $j \in \NN_{>0}$ such that $\rho g^j \rho^{-1} = (\rho g \rho^{-1})^j \in H$ so that $\lambda_g(\rho)$ is well-defined. To see this, let $\ro{NC}_G(H)$ be the normal core of $H$ in $G$. Since $\lbrack G:H \rbrack < \infty$, it follows from \ref{para:normal_core} that $\lbrack G:\ro{NC}_G(H) \rbrack < \infty$. Hence, $(\rho_i g \rho_i^{-1})^{\lbrack G:\ro{NC}_G(H) \rbrack} \in \ro{NC}_G(H) \leq H$.

\begin{asparaenum}[(i)]
\item By definition we have $\rho g^{\lambda_g(\rho)} \rho^{-1} \gleichdop h \in H$. This yields on the one hand $\rho = h^{-1} \rho g^{\lambda_g(\rho)}$ and on the other hand $\rho g^{k \lambda_g(\rho)} \rho^{-1} = (\rho g^{\lambda_g(\rho)} \rho^{-1})^k = h^k$ which implies $\rho g^{k \lambda_g(\rho)} = h^k \rho$. Hence,
\[
\rho g^{k\lambda_g(\rho) + q} = \rho g^{k\lambda_g(\rho)} g^q = h^k \rho g^q = h^k h^{-1} \rho g^{\lambda_g(\rho)+q}
\]
and so we can write
\[
\rho g^{k\lambda_g(\rho) + q} = \left\lbrace \begin{array}{ll} h^k \rho g^q & \tn{if } q \neq 0 \\  h^{k-1} \rho g^{\lambda_g(\rho)} & \tn{if } q = 0. \end{array} \right. 
\]
In particular, we have $H \rho g^j = H \rho g^q$ if $q \neq 0$ and $H \rho g^j = H \rho g^{\lambda_g(\rho)}$ if $q = 0$. 

Now, suppose that $\rho g^j \rho^{-1} \gleichdop h' \in H$. Using the above, we get $h' \rho = \rho g^j = h^k \rho g^q$ and consequently $h' = h^k \rho g^q \rho^{-1}$. This implies that $\rho g^q \rho^{-1} = h^{-k} h' \in H$ but as $q \in \lbrace 0,\ldots,\lambda_{g}(\rho)-1 \rbrace$, we must have $q = 0$ by definition of $\lambda_g(\rho)$. This shows that necessarily $q = 0$ and it is obvious that $\rho g^j \rho^{-1} \in H$ if $q = 0$.

\item Let $\sigma \in G$. Since $G = \coprod_{\rho \in R} H \rho \langle g \rangle$, we can write $\sigma = h \rho g^i$ for some $h \in H$, $\rho \in R$ and $i \in \ZZ$. It follows from the above that $H \rho g^i = H \rho g^j$ for some $j \in \lbrace 1,\ldots,\lambda_g(\rho) \rbrace$ and consequently $\sigma \in HT$ what shows that $G = \bigcup_{t \in T} Ht$. To see that this union is disjoint, suppose that $H \rho_1 g^{j_1} = H \rho_2 g^{j_2}$ with $j_i \in \lbrace 1, \ldots,\lambda_g(\rho_i) \rbrace$. Since $G = \coprod_{\rho \in R}^n H \rho \langle g \rangle$, this implies that $\rho_1 = \rho_2 \gleichdop \rho$ and if $j_1 = j_2$, then we are done. So, suppose without loss of generality that $j_1 > j_2$. We can write $\rho g^{j_1} = h \rho g^{j_2}$ for some $h \in H$ and consequently $\rho g^{j_1-j_2} \rho^{-1} = h \in H$. Since $j_1-j_2 \in \NN_{>0}$, it follows from the definition of $\lambda_g(\rho)$ that $j_1 - j_2 \geq \lambda_g(\rho)$. But as $j_1,j_2 \in \lbrace 1,\ldots,\lambda_g(\rho) \rbrace$, this is not possible. Hence, $T$ is a right transversal of $H$ in $G$.

\item Let $T$ be the right transversal from above. Let $\rho \in R$ and let $j \in \lbrace 1, \ldots, \lambda_g(\rho) \rbrace$. If $j < \lambda_g(\rho)$, then $(\rho g^j)g = \rho g^{j+1} \in T$ and therefore $\kappa_T( (\rho g^j)g ) = 1$. If $j = \lambda_g(\rho)$, then $(\rho g^j)g = (\rho g^{\lambda_g(\rho)} \rho^{-1}) (\rho g) \in HT$ and consequently $\kappa_T((\rho g^{\lambda_g(\rho)})g) = \rho g^{\lambda_g(\rho)} \rho^{-1}$. Hence,
\[
\ro{V}_{H,G}^{\ro{R}(H)}(g) = \prod_{t \in T} \kappa_T(t_i g) \modd \ro{R}(H) = \prod_{\substack{\rho \in R \\ j \in \lbrace 1,\ldots,\lambda_g(\rho) \rbrace}} \kappa_T((\rho g^j) g) \modd \ro{R}(H) = \prod_{\rho \in R} \rho g^{\lambda_g(\rho)} \rho^{-1} \modd \ro{R}(H) .
\]
\end{asparaenum} \vspace{-\baselineskip}
\end{proof}

\newpage
\section{Abelian class field theories} \label{chap:acfts}

We recall from the introduction that an ACFT should model for each finite separable extension $K | k$ of a field $k$ the finite Galois extensions of $K$ as certain subgroups of some abelian group $C(K)$ in such a way that this model is faithful on the lattice of abelian extensions of $K$ and such that abelianized Galois groups can be calculated in this model. More explicitly, there should exist a map $\Phi(K,-): \ca{E}^{\tn{f}}(K) \ra \ca{E}^{\ro{a}}(C(K))$ from the set $\ca{E}^{\tn{f}}(K)$ of all finite Galois extensions of $K$ to the set $\ca{E}^{\ro{a}}(C(K))$ of all subgroups of $C(K)$ such that the restriction of $\Phi(K,-)$ to the lattice $\ca{L}(K) \subs \ca{E}^{\tn{f}}(K)$ of finite abelian extensions is an injective morphism of lattices and there should exist an isomorphism, called \textit{reciprocity morphism},
\[
\rho_{L|K}: \gal(L|K)^{\ro{ab}} \ra C(K)/\Phi(K,L)
\]
for each $L \in \ca{E}^{\tn{f}}(K)$. We depicted the passage from the field internal theory to its model as the scheme
\[
\begin{array}{rcl}
K & \rightsquigarrow  & C(K) \\
L \in \ca{E}^{\tn{f}}(K) & \rightsquigarrow & \Phi(K,L) \leq C(K) \\
\gal(L|K) & \rightsquigarrow & C(K)/\Phi(K,L)
\end{array}
\]
and referred to these data as the \textit{group-theoretical part} of an ACFT because by Galois theory the data above corresponds to a map $\Phi(H,-):\ca{E}^{\tn{f}}(H) \ra \ca{E}^{\ro{a}}(C(H))$ for each open subgroup $H$ of $G \dopgleich \gal(k)$ which is injective on the lattice of coabelian open subgroups of $H$, and to an isomorphism
\[
\rho_{(H,U)}: (H/U)^{\ro{ab}} \ra C(H)/\Phi(H,U)
\]
for each $U \in \ca{E}^{\tn{f}}(H)$. We will now entirely shift to this group-theoretical perspective and come to the functorial part of an ACFT.
Using the fact that $G$ is profinite and thus compact, we can use \ref{para:abelianization_of_quotient} to canonically rewrite the abelianization of $H/U$ as
\begin{align*}
(H/U)^{\ro{ab}} &= (H/U)/\comm{t}(H/U) = (H/U)/(U \comm{t}(H)/U) \cong H/U \comm{t}(H) \\
 & \cong (H/\comm{t}(H))/ ( (U \comm{t}(H))/\comm{t}(H)) \cong H^{\ro{ab}}/( (U \comm{t}(H))/\comm{t}(H))
\end{align*}
and now it is easy to see that the scheme
\[
\begin{array}{rcl}
H & \rightsquigarrow  & C_{\tn{taut}}(H) \dopgleich H^{\ro{ab}} = H/\comm{t}(H) \\
U \in \ca{E}^{\tn{f}}(H) & \rightsquigarrow &  \Phi_{\tn{taut}}(H,U) \dopgleich U \comm{t}(H)/\comm{t}(H) \leq H/\comm{t}(H) \\
H/U & \rightsquigarrow & C_{\tn{taut}}(H)/\Phi_{\tn{taut}}(H,U) = (H/\comm{t}(H))/(U \comm{t}(H)/\comm{t}(H))
\end{array}
\]
with reciprocity morphisms the isomorphisms above satisfies all of the properties discussed above. This is of course a tautological example because in general the abelianized Galois groups $\gal(L|K)^{\ro{ab}}$ are hard to understand and it is one task of an abelian class field theory to provide a more accessible presentation of them, so from this point of view this tautological example is worthless. But as it captures precisely the (finite) ``abelian theory'' of $G = \gal(k)$, we should use it to force a general ACFT to share its ``RIC-character'' in the following two-fold sense:
\begin{enumerate}[label=(\roman*),noitemsep,nolistsep]
\item Since the groups $C_{\tn{taut}}(H)$ are equal to the values of the RIC-functor $\pi_{\ro{ab}} = \pi_{\ro{ab}}^{\fr{S}} \in \se{Fct}(\fr{S},\se{TAb})$ with $\fr{S} = \ro{Grp}(G)^{\tn{f}}$, we should in general force the single groups $C(H)$ of an ACFT to be the values of a RIC-functor $C \in \se{Fct}(\fr{S},\se{TAb})$. 
\item It is not hard to see that:
\begin{enumerate}[label=(\alph*),noitemsep,nolistsep]
\item The conjugation of $\pi_{\ro{ab}}$ induces a morphism
\[
C_{\tn{taut}}(H)/ \Phi_{\tn{taut}}(H,U) \ra C_{\tn{taut}}(^g \! H)/ \Phi_{\tn{taut}}(^g \! H, {^g \! U})
\]
for all $H \in \ssys{b}$, $U \in \ca{E}^{\tn{f}}(H)$ and $g \in G$.
\item \label{item:acft_mot_res_induced} The restriction of $\pi_{\ro{ab}}$ induces a morphism
\[
C_{\tn{taut}}(H)/ \Phi_{\tn{taut}}(H,U) \ra C_{\tn{taut}}(I)/ \Phi_{\tn{taut}}(I,U)
\]

for all $H \in \ssys{b}$, $I \in \ssys{r}(H)$ and $U \in \ca{E}^{\tn{f}}(H) \cap \ca{E}^{\tn{f}}(I)$.

\item The induction of $\pi_{\ro{ab}}$ induces a morphism
\[
C_{\tn{taut}}(I)/ \Phi_{\tn{taut}}(I,V) \ra C_{\tn{taut}}(H)/ \Phi_{\tn{taut}}(H,U)
\]

for all $H \in \ssys{b}$, $I \in \ssys{i}(H)$ and $V \in \ca{E}^{\tn{f}}(I)$, $U \in \ca{E}^{\tn{f}}(H)$ with $V \leq U$.
\end{enumerate}

Consequently, the conjugation, restriction and induction morphisms of the RIC-functor $C$ of a general ACFT should also have this property. Moreover, these induced morphisms should correspond to those induced by $\pi_{\ro{ab}}$ under the isomorphisms $\theta_{(H,U)}:C_{\tn{taut}}(H)/ \Phi_{\tn{taut}}(H,U) \ra C(H)/ \Phi(H,U)$ obtained by composing the canonical isomorphism $C_{\tn{taut}}(H)/ \Phi_{\tn{taut}}(H,U) \ra (H/U)^{\ro{ab}}$ with $\rho_{(H,U)}$. \\
\end{enumerate}

This ``RIC-character'' of an ACFT is what we will refer to as the \textit{functorial part} of an ACFT. We already note that under the canonical isomorphisms $C_{\tn{taut}}(H)/ \Phi_{\tn{taut}}(H,U) \ra (H/U)^{\ro{ab}}$ the morphism in \ref{item:acft_mot_res_induced} corresponds to the transfer morphism
\[
\ver_{I/U,H/U}^{\comm{t}(I/U),\comm{t}(H/U)}: (H/U)^{\ro{ab}} \ra (I/U)^{\ro{ab}}
\]
so that the functorial part really gives an ACFT the natural internal structures of abelianizations.

By considering an appropriate RIC-domain $\fr{E}$ consisting of pairs $(H,U)$ with $H$ an open subgroup of $G$ and $U$ an open normal subgroup of $H$, the functorial part can be compressed into the conditions that the maps $\Phi(H,-)$ define a subfunctor $\Phi$ of a lifting $C^{\fr{E}}$ of $C$ to $\fr{E}$ and the isomorphisms $\theta_{(H,U)}$ define an isomorphism $\theta: C_{\tn{taut}}^{\fr{E}}/ \Phi_{\tn{taut}} \ra C^{\fr{E}}/\Phi$ in $\se{Fct}(\fr{E},\se{TAb})$. This fact was one of the motivations for the introduction of the general notion of RIC-domains. \\

Some non-classical ACFTs make it necessary that we generalize the group-theoretical and functorial parts so that also these theories become part of our framework. The following is a list of the intended generalizations and their origins:
\begin{compactitem}[$\bullet$]
\item It is easy to see that the RIC-functors $C$ in the local and global class field theories mentioned in the introduction are cohomological Mackey functors having Galois-descent and are thus obtained by taking invariants of a discrete $G$-module (in the local class field theory we obviously have $C = (\ro{GL}_1(k^{\ro{s}}))_*$ and in the global class field theory it is not hard to see (although it is not entirely obvious, confer \cite[chapter VI, proposition 2.5]{Neu99_Algebraic-Number_0}) that $C$ has Galois-descent). This would serve as a motivation for always assuming that the RIC-functor $C$ of an ACFT is a cohomological Mackey functor with Galois descent. This is, in the language of discrete $G$-modules instead of cohomological Mackey functors with Galois-descent, precisely what is assumed in ``classical approaches'' to ACFTs as \name{Jürgen Neukirch}'s class field theory discussed in \cite[chapter IV]{Neu99_Algebraic-Number_0} and the (classical) Nakayama--Tate duality discussed in \cite[chapter 3]{NeuSchWin08_Cohomology-of-Number_0}. Ignoring that from our point of view the usage of a general RIC-functor is absolutely natural, the Galois-descent assumption would unnecessarily restrict the realm of ACFTs. The reason for this is that \name{Ivan Fesenko} has demonstrated in \cite{Fes92_Multidimensional-local_0} that it is possible to extend Neukirch's class field theory to an ACFT with $C$ being a general cohomological Mackey functor (not necessarily having Galois descent) and then instantiated this theory for higher local fields with $C$ indeed being a cohomological Mackey functor not necessarily having Galois descent. This shows that there exist ACFTs with $C$ not being isomorphic to a discrete $G$-module and thus not being part of the classical approaches to ACFTs. Moreover, it shows that it is possible to extend an entire classical theory, the \word{Neukirch theory}, to cohomological Mackey functors. We take these two observations as the final motivation for discussing ACFTs with $C$ being a general RIC-functor. The assumption of $C$ being a cohomological Mackey functor is in most parts of our abstract discussion of ACFTs not needed but it will later be a central point in abstract theorems about class field theories and in the Neukirch theory.

\item The abelian local $p$-class field theory for totally ramified extensions discussed by \name{Fesenko} in \cite{Fes95_Abelian-local_0} provides in its easiest situation (that is, when the Galois group $\Gamma_K$ of the maximal abelian unramified $p$-extension $\widetilde{K}|K$ is isomorphic to $\ZZ_p$) reciprocity morphisms 
\[
\rho_{L|K}: \gal(L|K)^{\ro{ab}} \ra C(K)/\Phi(K,L)
\]
only for totally ramified $p$-extensions $L|K$. Moreover, although this theory has a functorial part as discussed above, it only exists when passing between totally ramified extensions. We will solve this problem by the introduction of two-dimensional $G$-subgroup systems and in particular $G$-spectra which can be considered as being composed of a $G$-subgroup system $\fr{S} \leq \ro{Grp}(G)^{\tn{r-f}}$ with a set of ``extensions'' $\ca{E}(H)$ attached to each of the ``base groups'' $H \in \ssys{b}$. The functor $C$ of an ACFT is then just defined on $\ssys{b}$ and reciprocity morphisms just exist for the selected extensions. All this can be easily defined in the language of RIC-functors. For the above situation one can then choose $\ssys{b}$ as the set of all open subgroups of $G = \gal(k)$ but as $\ssys{\star}(H)$ only open subgroups which correspond to totally ramified extensions of $(k^\ro{s})^H$ and as $\ca{E}(H)$ only open normal subgroups which correspond to totally ramified $p$-extensions of $(k^\ro{s})^H$.

\item Another problem that occurs in Fesenko's abelian local $p$-class field theory is the fact that the composition of totally ramified extensions is not necessarily totally ramified so that the abelian totally ramified extensions do not necessarily form a lattice. Therefore a definition of the faithfulness of an ACFT as discussed above does not make sense in this situation. We solve this problem by requiring an ACFT to be faithful on any lattice consisting of coabelian subgroups.

\item The general reciprocity morphisms of Fesenko's abelian local $p$-class field theory for totally ramified extensions (that is, when $\Gamma_K$ is not necessarily isomorphic to $\ZZ_p$) are isomorphisms
\[
\rho_{L|K}: \hom_{\se{TAb}}(\Gamma_K, \gal(L|K)^{\ro{ab}}) \ra C(K)/\Phi(K,L)
\]
for totally ramified $p$-extensions $L \sups K$. To include also this ACFT in our discussion, we will allow an ACFT to carry just the internal structure of a ``dualized'' tautological ACFT which can easily be defined in the language of RIC-functors.

\item The $p$-class field theories discussed by \name{Thomas Weigel} in \cite{Wei07_Frattini-extensions_0} are ACFTs which describe just the maximal abelian $p$-quotient of the Galois groups instead of the maximal abelian quotient. This motivates to consider ACFTs which carry the internal structure of general abelianizations which is again easy to define in the language of RIC-functors.

\item In our abstract discussion of ACFT it makes no difference to replace the profinite group $G = \gal(k)$ by a general compact group $G$. \\
\end{compactitem}

Nearly all objects discussed so far (including their straightforward generalizations) have been introduced to make it possible to define a notion of ACFTs that takes care of all the above generalizations and restrictions, while still possessing enough structure to make certain abstract theorems work. This highly general approach has unfortunately the disadvantage of bringing a major complexity to our theory because we have to take care of a lot of extra data. At the end of each abstract consideration we will therefore provide a corollary for the most relevant situations. \\

In \ref{sect:reps_of_2dim_ss} the notions of two-dimensional $G$-subgroup systems and representations of them are defined. In \ref{sec:taut_cft} the tautological class field theory is formalized and in \ref{sec:ind_reps} an important type of representations, the induction representations, is defined. Finally, in \ref{sect:acfts} the definition of an ACFT is given and several abstract theorems are proven and in \ref{para:ind_cfts} a very important further reduction theorem is proven. The last section \ref{para:cfts_for_profinite} is just intended to fix some notations in the case of field extensions.

\subsection{Two-dimensional $G$-subgroup systems} \label{sect:reps_of_2dim_ss}

\begin{para}
In order to formulate the functorial part of an ACFT in the language of RIC-functors we have to find the proper RIC-domains. These will be a special case of the two-dimensional $G$-subgroup systems that we will introduce in this section. We also define the notion of representations of regular two-dimensional $G$-subgroup systems which capture the ``modeling process'' described at the beginning of the introduction.
\end{para}

\begin{ass}
Throughout this section $G$ is a topological group.
\end{ass}

\begin{defn}
We define
\[
\ro{Grp}^2(G) \dopgleich \lbrace (H,U) \mid H \leq_{\ro{c}} G \tn{ and } U \lhd_{\ro{c}} H \rbrace
\]
and equip this set with the partial order
\[
(I,V) \leq (H,U) :\lLRA I \leq H \tn{ and } V \leq U.
\]
We let $p_1:\ro{Grp}^2(G) \ra \ro{Grp}(G)$, $(H,U) \mapsto H$, be the projection onto the first factor and $p_2: \se{Grp}^2(G) \ra \ro{Grp}(G)$, $(H,U) \mapsto U$, be the projection onto the second factor. The group $G$ acts on $\ro{Grp}^2(G)$ via component-wise conjugation $^g \! (H,U) \dopgleich (^g \! H, {^g \! U})$. We define a \words{two-dimensional $G$-subgroup system}{$G$-subgroup system!two-dimensional} as a RIC-domain in the equiordered set $(\ro{Grp}^2(G),\leq,G,\mu)$, where $\mu$ is the conjugation action.
\end{defn}

\begin{defn}
We mostly just write $(\esys{b},\esys{r},\esys{i})$ for a two-dimensional $G$-subgroup system $(\esys{b},\esys{r},\esys{i},G,\mu)$. Moreover, we define $\esys{b}^\flat \dopgleich p_1(\esys{b})$ and $\ro{Ext}(\fr{E},H) \dopgleich p_2(p_1^{-1}(H))$ for $H \in \esys{b}^\flat$. We say that $\fr{E}$ is \words{E-finite}{two-dimensional $G$-subgroup system} if $\lbrack H:U \rbrack < \infty$ for all $(H,U) \in \esys{b}$ and we say that $\fr{E}$ is \words{finite}{two-dimensional $G$-subgroup system!finite} if it is E-finite and $\lbrack G:H \rbrack < \infty$ for all $H \in \esys{b}^\flat$.
\end{defn}

\begin{defn}
Let $\fr{E}$ be a two-dimensional $G$-subgroup system. For $\star \in \lbrace \ro{r},\ro{i} \rbrace$ and $H \in \esys{b}^\flat$ let $\esys{\star}^\flat(H) \dopgleich p_1( \bigcup_{(H,U) \in \esys{b}}  \esys{\star}(H,U))$ and $\esys{\star}^\flat \dopgleich \lbrace \esys{\star}^\flat(H) \mid H \in \esys{b} \rbrace$. Then  $\fr{E}$ is called \words{regular}{subgroup system!regular} if $\fr{E}^\flat \dopgleich (\esys{b}^\flat,\esys{r}^\flat,\esys{i}^\flat)$ is a $G$-subgroup system.
\end{defn}

\begin{para}
A regular two-dimensional $G$-subgroup system can thus be thought of as consisting of a $G$-subgroup system $\fr{E}^\flat$ with a set of ``extensions'' $\ro{Ext}(\fr{E},H)$ attached to each $H \in \esys{b}^\flat$.

\newpage
\begin{center}
\scalebox{1.0}{
\begin{picture}(0,0)%
\includegraphics{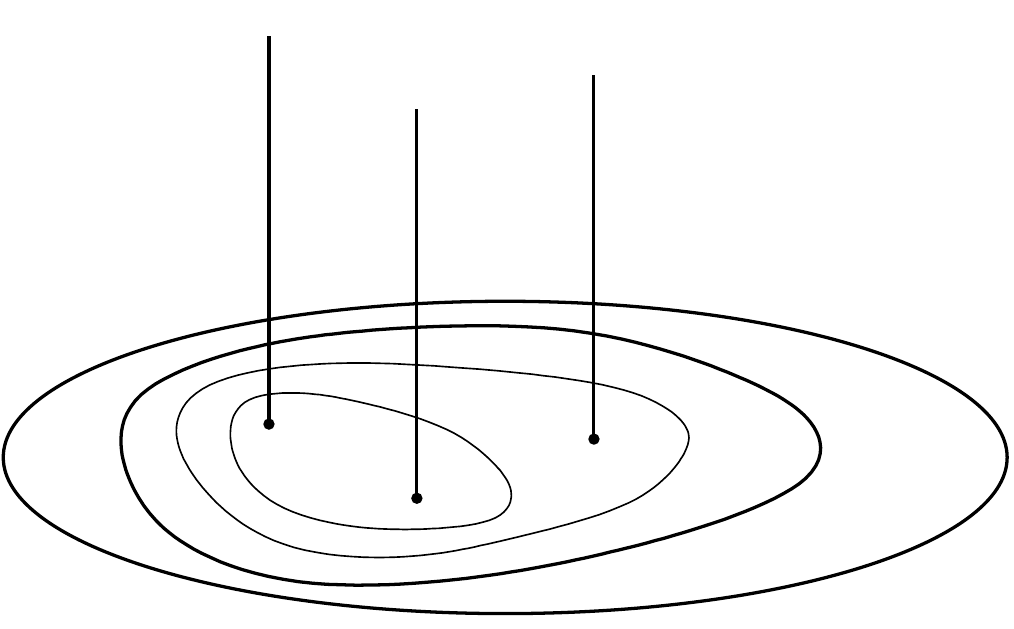}%
\end{picture}%
\setlength{\unitlength}{4144sp}%
\begin{picture}(4620,2864)(2573,-4797)
\put(5243,-2220){\makebox(0,0)[lb]{\smash{{\SetFigFont{8}{9.6}{\rmdefault}{\mddefault}{\updefault}{\color[rgb]{0,0,0}$\mathrm{Ext}(\mathfrak{E},I)$}%
}}}}
\put(5401,-3976){\makebox(0,0)[lb]{\smash{{\SetFigFont{8}{9.6}{\rmdefault}{\mddefault}{\updefault}{\color[rgb]{0,0,0}$I$}%
}}}}
\put(4568,-4245){\makebox(0,0)[lb]{\smash{{\SetFigFont{8}{9.6}{\rmdefault}{\mddefault}{\updefault}{\color[rgb]{0,0,0}$J$}%
}}}}
\put(3871,-3908){\makebox(0,0)[lb]{\smash{{\SetFigFont{8}{9.6}{\rmdefault}{\mddefault}{\updefault}{\color[rgb]{0,0,0}$H$}%
}}}}
\put(4635,-3862){\makebox(0,0)[lb]{\smash{{\SetFigFont{8}{9.6}{\rmdefault}{\mddefault}{\updefault}{\color[rgb]{0,0,0}$\mathfrak{E}_{\mathrm{i}}^\flat(H)$}%
}}}}
\put(5738,-4133){\makebox(0,0)[lb]{\smash{{\SetFigFont{8}{9.6}{\rmdefault}{\mddefault}{\updefault}{\color[rgb]{0,0,0}$\mathfrak{E}_{\mathrm{r}}^\flat(H)$}%
}}}}
\put(6098,-3615){\makebox(0,0)[lb]{\smash{{\SetFigFont{8}{9.6}{\rmdefault}{\mddefault}{\updefault}{\color[rgb]{0,0,0}$\mathfrak{E}_{\mathrm{b}}^\flat(H)$}%
}}}}
\put(6413,-4718){\makebox(0,0)[lb]{\smash{{\SetFigFont{8}{9.6}{\rmdefault}{\mddefault}{\updefault}{\color[rgb]{0,0,0}$\mathfrak{E}_{\mathrm{b}}^\flat$}%
}}}}
\put(3758,-2040){\makebox(0,0)[lb]{\smash{{\SetFigFont{8}{9.6}{\rmdefault}{\mddefault}{\updefault}{\color[rgb]{0,0,0}$\mathrm{Ext}(\mathfrak{E},H)$}%
}}}}
\put(4433,-2378){\makebox(0,0)[lb]{\smash{{\SetFigFont{8}{9.6}{\rmdefault}{\mddefault}{\updefault}{\color[rgb]{0,0,0}$\mathrm{Ext}(\mathfrak{E},J)$}%
}}}}
\end{picture}%
}

\footnotesize{A regular two-dimensional $G$-subgroup system.}
\end{center}

\end{para}

\begin{prop}
Let $\fr{E}$ be a regular two-dimensional $G$-subgroup system, let $\ca{C}$ be a category and let $C \in \se{Fct}(\fr{E}^\flat,\ca{C})$. Then the following data define a RIC-functor $C^{\fr{E}} \in \se{Fct}(\fr{E},\ca{C})$:
\begin{compactitem}
\item $C^{\fr{E}}(H,U) \dopgleich C(H)$ for each $(H,U) \in \esys{b}$.
\item $\res_{(I,V),(H,U)}^{C^{\fr{E}}} \dopgleich \res_{I,H}^C$ for each $(H,U) \in \esys{b}$ and $(I,V) \in \esys{r}(H,U)$.
\item $\ind_{(H,U),(I,V)}^{C^{\fr{E}}} \dopgleich \ind_{H,I}^C$ for each $(H,U) \in \esys{b}$ and $(I,V) \in \esys{i}(H,U)$.
\item $\con_{g,(H,U)}^{C^{\fr{E}}} \dopgleich \con_{g,H}^C$ for each $(H,U) \in \esys{b}$ and $g \in G$.
\end{compactitem}
\end{prop}

\begin{proof}
This is evident, noting that the definition of $\esys{\star}^\flat(H)$ ensures that the restriction and induction morphisms are always well-defined.  
\end{proof}

\begin{defn} \wordsym{$\se{Rep}(\fr{E})$}
A \word{representation} of a regular two-dimensional $G$-subgroup system $\fr{E}$ is a quotient
$C^{\fr{E}}/\Phi \in \se{Fct}(\fr{E},\se{TAb})$ with $C \in \se{Fct}(\fr{E}^\flat,\se{TAb})$ and a subfunctor $\Phi \in \se{Fct}(\fr{E},\se{TAb})$ of $C^{\fr{E}}$. The RIC-functor $C$ is called the \word{class functor} of the representation and $\Phi$ is called the \word{extension functor} of the representation. The full subcategory of $\se{Fct}(\fr{E},\se{TAb})$ consisting of representations of $\fr{E}$ is denoted by $\se{Rep}(\fr{E})$.
\end{defn}

\begin{conv}
If nothing else is mentioned, then we always assume for a representation $C^{\fr{E}}/\Phi$ that the abelian groups $C(H)$ are written multiplicatively. This convention is due to our intended application of representations in class field theory.
\end{conv}

\begin{para}
In the following we will introduce some notations that will be used in the subsequent sections.
\end{para}

\begin{defn}
Let $\fr{E}$ be a regular two-dimensional $G$-subgroup system. 
\begin{compactitem}
\item We define
\[
\esys{b}^{\tn{cyc}} \dopgleich \lbrace (H,U) \mid (H,U) \in \esys{b} \tn{ and } H/U \tn{ is cyclic} \rbrace.
\]
For $n \in \NN_{>0}$ we define
\[
\esys{b}^n \dopgleich \lbrace (H,U) \mid (H,U) \in \esys{b} \tn{ and } H/U \tn{ is cyclic of order } n \rbrace
\]
and
\[
\ro{Ext}^n(\fr{E},H) \dopgleich \lbrace U \mid U \in \ro{Ext}(\fr{E},H) \tn{ and } H/U \tn{ is cyclic of order } n \rbrace.
\]
\item A \words{lattice}{two-dimensional $G$-subgroup system!lattice} in $\fr{E}$ is a family $\ca{L} = \lbrace \ca{L}(H) \mid H \in \esys{b}^\flat \rbrace$, where $\ca{L}(H) \subs \ro{Ext}(\fr{E},H)$ is a lattice with respect to products and intersections.  

\item If $\ro{R} \in \se{Ab}(\fr{E}^\flat)$, then we define
\[
\esys{b}^{\ro{R}} \dopgleich \lbrace (H,U) \mid (H,U) \in \esys{b} \tn{ and } U \geq \ro{R}(H) \rbrace
\]
and
\[
\ro{Ext}_{\ro{R}}(\fr{E},H) \dopgleich \lbrace U \mid U \in \ro{Ext}(\fr{E},H) \tn{ and } U \geq \ro{R}(H) \rbrace.
\]

\item We define $\esys{b}^{\ro{R},\tn{cyc}} = \esys{b}^{\tn{cyc}} \cap \esys{b}^{\ro{R}} $. Similarly we define $\esys{b}^{\ro{R},n}$ and $\ro{Ext}_{\ro{R}}^n(\fr{E},H)$.

\item An \word{$\ro{R}$-lattice} in $\fr{E}$ is a lattice $\ca{L} = \lbrace \ca{L}(H) \mid H \in \esys{b}^\flat \rbrace$ in $\fr{E}$ such that $\ca{L}(H) \subs \ro{Ext}_{\ro{R}}(\fr{E},H)$.

\item A representation $C^{\fr{E}}/\Phi$ of $\fr{E}$ is called \words{$\ca{L}$-faithful}{representation!$\ca{L}$-faithful} if $\Phi(H,-)|_{\ca{L}(H)}:\ca{L}(H) \ra \ca{E}^{\ro{a}}(C(H))$ is an injective morphism of lattices.
\end{compactitem}

\end{defn}

\subsection{Tautological class field theories} \label{sec:taut_cft}

\begin{para}
In this section we define the tautological class field theories as motivated in the introduction. These can be considered as representations which are naturally defined on a special type of regular two-{di\-men\-sion\-al} $G$-subgroup systems, the $G$-spectra. After defining the tautological class field theories we give two alternative presentations which are easier to work with.
\end{para}

\begin{ass}
Throughout this section $G$ is a topological group.
\end{ass}

\begin{defn}
A \word{$G$-spectrum} is a regular two-dimensional $G$-subgroup system $\fr{E}$ satisfying the following additional properties:
\begin{enumerate}[label=(\roman*),noitemsep,nolistsep]
\item If $(H,U) \in \esys{b}$ and $(I,V) \in \esys{r}(H,U)$, then $V = U$.
\item $\fr{E}^\flat$ is R-finite. 
\end{enumerate}

\end{defn}

\begin{para}
In the following we will discuss a general concept for how to construct a canonical $G$-spectrum from a $G$-subgroup system with extensions. All $G$-spectra used in applications will be obtained in this way, in particular the one necessary for Fesenko's $p$-class field field theories.
\end{para}

\begin{defn} \label{para:ss_extension}
An \words{extension}{subgroup system!extension} of a $G$-subgroup system $\fr{S}$ is a family $\ca{E} = \lbrace \ca{E}(H) \mid H \in \ssys{b} \rbrace$ with $\ca{E}(H)$ a non-empty set of closed normal subgroups of $H$ such that the following conditions are satisfied for each $H \in \ssys{b}$ and $\star \in \lbrace \ro{r},\ro{i} \rbrace$:
\begin{enumerate}[label=(\roman*),noitemsep,nolistsep]
\item $H \in \ca{E}(H)$.
\item \label{item:ss_extension_conj} $^g \! \ca{E}(H) = \ca{E}(^g \! H)$ for all $g \in G$.
\end{enumerate} 
\end{defn}

\begin{para}
If $\fr{S}$ is a $G$-subgroup system, then it is evident that $\ca{E}^\star \dopgleich \ca{E}_{\fr{S}}^\star \dopgleich \lbrace \ca{E}^\star(H) \mid H \in \ssys{b} \rbrace$ is an extension of $\fr{S}$ for $\star \in \lbrace \tn{t}, \tn{f} \rbrace$.
\end{para}

\begin{prop}
Let $\ca{E}$ be an extension of a $G$-subgroup system $\fr{S} \leq \ro{Grp}(G)^{\tn{r-f}}$. Let
\[
\esys{b} \dopgleich \lbrace (H,U) \mid H \in \ssys{b} \tn{ and } U \in \ca{E}(H) \rbrace
\]
and for $(H,U) \in \esys{b}$ let
\begin{align*}
& \esys{r}(H,U) \dopgleich \lbrace (I,U) \mid I \in \ssys{r}(H) \tn{ and } U \in \ca{E}(I) \rbrace \\
& \esys{i}(H,U) \dopgleich \lbrace (I,V) \mid I \in \ssys{i}(H) \tn{ and } V \in \ca{E}(I) \tn{ and } V \leq U \rbrace.
\end{align*}

Then $\fr{E} = (\esys{b},\esys{r},\esys{i})$ is a regular two-dimensional $G$-subgroup system with $\fr{E}^\flat = \fr{S}$ and $\ro{Ext}(\fr{E},H) = \ca{E}(H)$ for all $H \in \ssys{b}$. We write $\fr{E} = \ro{Sp}(\fr{S},\ca{E})$.
\end{prop}

\begin{proof}
This is easy to verify.
\end{proof}

\begin{para}
We define $\ro{Sp}^{\tn{t}}(G) \dopgleich \ro{Sp}(\ro{Grp}(G)^{\tn{r-f}},\ca{E}^{\tn{t}})$ and $\ro{Sp}^{\star}(G) \dopgleich \ro{Sp}(\ro{Grp}(G)^{\star}, \ca{E}^{\tn{f}})$ for $\star \in \lbrace \tn{r-f},\tn{ri-f},\tn{f} \rbrace$. 
\end{para}

\begin{defn}
A \word{coabelian classification datum} for $G$ is a triple $\fr{K} = (\fr{E},\ro{R},D)$ consisting of a $G$-spectrum $\fr{E}$, an abelianization system $\ro{R} \in \se{Ab}(\fr{E}^\flat)$ and a dualizing functor $D \in \se{Fct}(\fr{E}^\flat,\se{TAb}^{\ro{lc}})$. The set of all coabelian classification data for $G$ is denoted by $\ro{CCD}(G)$. If $(\fr{E}',\ro{R}',D')$ and $(\fr{E},\ro{R},D)$ are coabelian classification data for $G$, then we define
\[
(\fr{E}',\ro{R}',D') \leq (\fr{E},\ro{R},D) : \lLRA \fr{E}' \leq \fr{E} \tn{ and } \ro{R}' = \ro{R}|_{ (\fr{E}')^\flat} \tn{ and } D' = D|_{(\fr{E}')^\flat}.
\]
\end{defn}

\begin{para}
A coabelian classification datum is the initial datum for an ACFT. The $G$-subgroup system $\fr{E}^\flat$ contains the ``base groups'' $H$ whose extensions $\ro{Ext}(\fr{E},H)$ are to be described. The abelianization system $\ro{R}$ describes how the ``Galois groups'' $H/U$ are abelianized and the dualizing functor $D$ describes how these Galois groups are ``dualized''.
\end{para}

\begin{para}
The main example of a coabelian classification datum for $G$ is $\fr{K}(\fr{E})_{\tn{ab}} \dopgleich (\fr{E},\comm{t,\fr{E}^\flat},\ZZ)$ for a $G$-spectrum $\fr{E}$ and in particular $
\fr{K}(G)_{\tn{ab}}^{\star} \dopgleich \fr{K}(\ro{Sp}(G)^\star)_{\tn{ab}}$ for $\star \in \lbrace \tn{t},\tn{r-f},\tn{ri-f},\tn{f} \rbrace$. Assuming that $G$ is compact, another important example is $\fr{K}(G)_{\ca{A}}^{\star} \dopgleich (\ro{Sp}^{\star}(G), \ro{R}_{\ca{A},\ro{Grp}(G)^{\star}}, \ZZ)$ for a variety $\ca{A}$ of compact abelian groups and $\star \in \lbrace \tn{t},\tn{r-f},\tn{ri-f},\tn{f} \rbrace$.
\end{para}

\begin{ass}
For the rest of this section we fix a coabelian classification datum $\fr{K} = (\fr{E},\ro{R},D)$ for $G$.
\end{ass}

\begin{para}
As motivated in the introduction, we will now define the tautological $\fr{K}$-class field theory which extracts from $G$ the information as prescribed by $\fr{K}$ in a tautological way, namely as follows
\[
\begin{array}{rcl}
H \in \fr{E}^\flat & \rightsquigarrow  & C_{\fr{K},\tn{taut}}(H) = H/\ro{R}(H) \\
U \in \ro{Ext}(\fr{E},H) & \rightsquigarrow &  \Phi_{\fr{K},\tn{taut}}(H,U) = U \ro{R}(H)/\ro{R}(H) \leq H/\ro{R}(H) \\
H/U & \rightsquigarrow & C_{\fr{K},\tn{taut}}(H)/\Phi_{\fr{K},\tn{taut}}(H,U).
\end{array}
\]
\end{para}

\begin{prop} \label{para:tautological_cft}
The map 
\[
\begin{array}{rcl}
\Phi_{\fr{K},\tn{taut}}:\esys{b} & \lra & \se{TAb} \\
(H,U) & \longmapsto & U\ro{R}(H)/\ro{R}(H)
\end{array}
\]
is a subfunctor of $\pi_{\ro{R}}^{\fr{E}} = ( \pi_{\ro{R}} )^{\fr{E}} \in \se{Fct}(\fr{E},\se{TAb})$. The representation $\pi_{\ro{R}}^{\fr{E}} / \Phi_{\fr{K},\tn{taut}} \in \se{Fct}(\fr{E},\se{TAb})$ of $\fr{E}$ is denoted by $\widehat{\Pi}_{\fr{K}}$ and is called the \words{tautological $\fr{K}$-class field theory}{class field theory!tautological}.
\end{prop}

\begin{proof}
Let $\Phi = {\Phi_{\fr{K},\tn{taut}}}$ and let $(H,U) \in \esys{b}$. It is obvious that $\Phi(H,U) = U\ro{R}(H)/\ro{R}(H)$ is a subgroup of $\pi_{\ro{R}}^{\fr{E}}(H,U) = \pi_{\ro{R}}(H) = H/\ro{R}(H)$. By \ref{para:ab_system}\ref{item:ab_system_conj} we have
\[
g(U\ro{R}(H))g^{-1} = (gUg^{-1})(g\ro{R}(H)g^{-1}) = {^g \! U}\ro{R}(^g \! H)
\] 
and therefore
\begin{align*}
\con_{g,(H,U)}^{\pi_{\ro{R}}^\fr{E}}( \Phi(H,U)) &= \con_{g,H}^{\pi_{\ro{R}}}( \Phi(H,U)) = \con_{g,H}^{\pi_{\ro{R}}}( U\ro{R}(H)/\ro{R}(H)) = {^g \! U}\ro{R}(^g \! H)/\ro{R}(^g \! H) \\
& = \Phi(^g \! H, {^g \! U}) = \Phi(^g \! (H, U)).
\end{align*}

Now, let $(I,U) \in \esys{r}(H,U)$. Then $I \in \esys{r}^\flat(H)$. Moreover, $U \lhd H$ and $U \lhd I$. Let $T$ be a right transversal of $I$ in $H$. Then, using \ref{para:pretransfer}\ref{item:pretranfer_possible_reduction}, \ref{para:pretransfer}\ref{item:pretransfer_mor} and \ref{para:ab_system}\ref{item:ab_system_res}, we get
\begin{align*}
\res_{(I,U),(H,U)}^{\pi_{\ro{R}}^\fr{E}}( \Phi(H,U)) &= \res_{I,H}^{\pi_{\ro{R}}}( \Phi(H,U)) = \res_{I,H}^{\pi_{\ro{R}}}( U\ro{R}(H)/\ro{R}(H)) = \ver_{I,H}^{\ro{R}(I),\ro{R}(H)}( U\ro{R}(H)/\ro{R}(H)) \\
&= \ro{V}_{I,H}^{\ro{R}(I)}(U \ro{R}(H)) = \ro{V}_{I,H}^{\ro{R}(I)}(U) \cdot \ro{V}_{I,H}^{\ro{R}(I)}(\ro{R}(H)) = \ro{V}_{I,H}^{\ro{R}(I)}(U) = \ro{V}_{I,H}^T( U) \modd \ro{R}(I)  \\
&\subs U \modd \ro{R}(I) = U\ro{R}(I)/\ro{R}(I) = \Phi(I,U).
\end{align*}

Finally, let $(I,V) \in \esys{i}(H,U)$. Then $I \in \esys{i}^\flat(H)$ and so it follows from \ref{para:ab_system}\ref{item:ab_system_ind} that $\ro{R}(I) \leq \ro{R}(H)$. Moreover, $V \leq U$ and therefore $V \ro{R}(I) \leq V \ro{R}(H) \leq U \ro{R}(H)$. Hence, we get
\begin{align*}
\ind_{(H,U),(I,V)}^{\pi_{\ro{R}}^\fr{E}}(\Phi(I,V)) &= \ind_{H,I}^{\pi_{\ro{R}}}(\Phi(I,V)) = \ind_{H,I}^{\pi_{\ro{R}}}( V\ro{R}(I)/\ro{R}(I)) = (V\ro{R}(I))\ro{R}(H)/\ro{R}(H) \\
&= V\ro{R}(H)/\ro{R}(H) \subs U\ro{R}(H)/\ro{R}(H) = \Phi(H,U).
\end{align*}
\end{proof}

\begin{prop}
$\widehat{\Pi}_{\fr{K}} \in \se{Fct}(\fr{E},\se{TAb})$ is canonically isomorphic to the functor $\Pi_{\fr{K}} \in \se{Fct}(\fr{E},\se{TAb})$ given by the following data:
\begin{compactitem}
\item $\Pi_{\fr{K}}(H,U) \dopgleich H/U\ro{R}(H)$ for each $(H,U) \in \esys{b}$.
\item $\con_{g,(H, U)}^{\Pi_{\fr{K}}}: H/U\ro{R}(H) \ra {^g \! H}/{^g \! U} \ro{R}(^g \! H)$ is induced by the conjugation $H \ra {^g \! H}$, $h \mapsto ghg^{-1}$, for each $(H,U) \in \esys{b}$ and $g \in G$.
\item $\res_{(I,U),(H, U)}^{\Pi_{\fr{K}}} \dopgleich \ver_{I,H}^{U \ro{R}(I),U \ro{R}(H)}:  H/U\ro{R}(H) \ra I/U\ro{R}(I)$ for each $(H,U) \in \esys{b}$ and $(I,U) \in \esys{r}(H,U)$.
\item $\ind_{(H,U),(I,V)}^{\Pi_{\fr{K}}}: I/V\ro{R}(I) \ra H/U\ro{R}(H)$ is induced by the inclusion $I \ra H$ for each $(H,U) \in \esys{b}$ and $(I,V) \in \esys{i}(H,U)$.
\end{compactitem}

\end{prop}

\begin{proof}
This follows using the isomorphism $\widehat{\Pi}_{\fr{K}}(H,U) = (H/\ro{R}(H))/(U \ro{R}(H)/\ro{R}(H)) \cong H/U\ro{R}(H)$ and \ref{para:pretransfer}\ref{item:pretranfer_possible_reduction}, \ref{para:pretransfer}\ref{item:pretransfer_mor} and \ref{para:ab_system}\ref{item:ab_system_res}.
\end{proof}

\begin{lemma} \label{para:transfer_reduction_2}
Let $H$ be a closed subgroup of finite index of $G$ and let $U$ be a normal subgroup of $G$ such that $U \leq H$. Assume that $H/U$ is closed in $G/U$.\footnote{This is for example satisfied if $G$ is quasi-compact and $U$ is closed in $G$ because in this case the quotient morphism $G \ra G/U$ is closed by the closed map lemma.}  Let $q:G \ra G/U$ be the quotient morphism and let $T = (t_1,\ldots,t_n)$ be a right transversal of $H$ in $G$. Then $\ol{T} \dopgleich q(T) = (q(t_1),\ldots,q(t_n))$ is a right transversal of $H/U$ in $G/U$ and the diagram
\[
\xymatrix{
G \ar[rr]^{\ro{V}_{H,G}^T} \ar[d]_q & & H \ar[d]^q \\
G/U \ar[rr]_{\ro{V}_{H/U,G/U}^{\ol{T}}} & & H/U
}
\]
commutes.

\end{lemma}

\begin{proof}
It is evident that $q(T)$ is a right transversal of $H/U$ in $G/U$ and that $q \circ \kappa_T = \kappa_{\ol{T}} \circ q$ (confer also \ref{para:transversal_reduction}). Hence, for $g \in G$ we have
\[
q \circ \ro{V}_{H,G}^T(g) = q(\prod_{i=1}^n \kappa_T(t_ig)) = \prod_{i=1}^n q(\kappa_T(t_ig)) = \prod_{i=1}^n \kappa_{\ol{T}}(q(t_ig)) = \prod_{i=1}^n \kappa_{\ol{T}}(q(t_i)q(g)) = \ro{V}_{H/U,G/U}^{\ol{T}}(q(g)).
\]

\end{proof}

\begin{prop} \label{para:ab_rep_variety}
Assume that $G$ is compact and that $\fr{K} \leq \fr{K}_{\ca{A}}^{\tn{t}}(G)$ for a variety $\ca{A}$ of compact abelian groups. Then $\Pi_{\fr{K}} \in \se{Fct}(\fr{E},\se{TAb})$ is canonically isomorphic to $\pi_{\fr{K}} \in \se{Fct}(\fr{E},\se{TAb})$ given by the following data:
\begin{compactitem}
\item $\pi_{\fr{K}}(H,U) \dopgleich \pi_{\ca{A}}(H/U)$ for each $(H,U) \in \esys{b}$.
\item $\con_{g,(H,U)}^{\pi_{\fr{K}}} \dopgleich \pi_{\ca{A}}(\con_{g,(H, U )})$, where $\con_{g,(H, U )}:H/U \ra {^g \! H}/{^g \! U}$ is the morphism induced by the conjugation $H \ra {^g \! H}$, $h \mapsto ghg^{-1}$.
\item $\res_{(I,U),(H, U)}^{\pi_{\fr{K}}} \dopgleich \ver_{I/U,H/U}^{\ro{R}_{\ca{A}}(I/U),\ro{R}_{\ca{A}}(H/U)}$ for each $(H,U) \in \esys{b}$ and $(I,U) \in \esys{r}(H,U)$.
\item $\ind_{(H,U),(I,V)}^{\pi_{\fr{K}}} \dopgleich \pi_{\ca{A}}(\ind_{(H,U),(I,V)})$, where $\ind_{(H,U),(I,V)}: I/V \ra H/U$ is the morphism induced by the inclusion $I \ra H$ for each $(H,U) \in \esys{b}$ and $(I,V) \in \esys{i}(H,U)$.

\end{compactitem}

\end{prop}

\begin{proof}
First, we have to verify that the definition of $\pi_{\fr{K}}$ is well-defined. Let $(H,U) \in \esys{b}$. Since $G$ is compact and $H$ is closed in $G$, it follows that $H$ is also compact. Moreover, as $U$ is closed in $H$, the quotient $H/U$ is compact. Hence, $\pi_{\ca{A}}(H/U)$ is defined and $\pi_{\ca{A}}(H/U) \in \se{cplproto}\tn{-}\ca{A}$. The conjugation and induction morphisms are obviously well-defined. To see that the restriction morphisms are well-defined, let $(H,U) \in \esys{b}$ and let $(I,U) \in \esys{r}(H,U)$. Due to the closed map lemma the quotient morphism $q:H \ra H/U$ is closed and therefore $I/U$ is a closed subgroup of $H/U$. Moreover, it follows from \ref{para:transfer_reduction_2} that $I/U$ is of finite index in $H/U$ and that $\lbrace U \ro{R}_{\ca{A}}(I)/U,U\ro{R}_{\ca{A}}(H)/U \rbrace$ is a transfer inducing pair for $I/U$ in $H/U$. Hence, $\ver_{I/U,H/U}^{U\ro{R}_{\ca{A}}(I)/U,U\ro{R}_{\ca{A}}(H)/U}$ is defined. According to \ref{para:pi_of_quotient} we have $U\ro{R}_{\ca{A}}(I)/U = \ro{R}_{\ca{A}}(I/U)$ and $U\ro{R}_{\ca{A}}(H)/U = \ro{R}_{\ca{A}}(H/U)$. This shows that $\res_{(I,U),(H,U)}^{\pi_{\fr{K}}}$ is defined.

Now, it is easy to see that the canonical isomorphisms
\[
\varphi_{(H, U)}: \Pi_{\fr{K}}(H,U) = H/U\ro{R}_{\ca{A}}(H) \ra (H/U)/(U\ro{R}_{\ca{A}}(H)/U) = (H/U)/\ro{R}_{\ca{A}}(H/U) = \pi_{\ca{A}}(H/U) = \pi_{\fr{K}}(H,U)
\]
are compatible with the conjugation, restriction and induction morphisms. As $\Pi_{\fr{K}} \in \se{Fct}(\fr{E},\se{TAb})$, this implies that $\pi_{\fr{K}} \in \se{Fct}(\fr{E},\se{TAb})$ and that $\varphi$ is an isomorphism between these RIC-functors.
\end{proof}

\subsection{Induction representations} \label{sec:ind_reps}

\begin{para}
Motivated by the structure the tautological $\fr{K}$-class field theory has in certain situations we will discuss a special type of representations in this section. These representations are of particular importance in class field theory.
\end{para}

\begin{ass}
Throughout this section $G$ is a topological group.
\end{ass}

\begin{para} \label{para:taut_cft_inductive}
Let $\fr{K} = (\fr{E},\ro{R},D) \in \ro{CCD}(G)$. If $\ro{Ext}(\fr{E},H) \subs \esys{i}^\flat$, then we have $U \ro{R}(H)/\ro{R}(H) = \ind_{H,U}^{\pi_{\ro{R}}} \pi_{\ro{R}}(U)$ and consequently
\[
\widehat{\Pi}_{\fr{K}} = \pi_{\ro{R}}^{\fr{E}} / \ind^{\pi_{\ro{R}}},
\]
where $\ind^{\pi_{\ro{R}}}(H,U) \dopgleich \ind_{H,U}^{\pi_{\ro{R}}} \pi_{\ro{R}}(U)$. This observation leads us to consider in general representations of the form $C^\fr{E}/\ind^C$ with $C \in \se{Fct}(\fr{E}^\flat,\se{TAb})$ and $\ind^C(H,U) = \ind_{H,U}^C C(U)$. But even if $\ind^C$ is well-defined (that is, $\ro{Ext}(\fr{E},H) \subs \esys{i}^\flat$), it is not clear whether this is a subfunctor of $C^\fr{E}$. To provide a general situation in which this works, we will introduce the notion of Mackey covers of a $G$-spectrum. 
\end{para}

\begin{defn}
Let $\fr{E}$ be a $G$-spectrum. A \words{cover}{$G$-spectrum!cover} of $\fr{E}$ is a $G$-subgroup system $\fr{S}$ such that $\fr{E}^\flat \leq \fr{S}$ and $\ro{Ext}(\fr{E},H) \subs \ssys{i}(H)$ for each $H \in \esys{b}^\flat$. A \words{Mackey cover}{$G$-spectrum!Mackey cover} of $\fr{E}$ is a cover $\fr{M}$ of $\fr{E}$ which is a $G$-Mackey system such that additionally the following condition holds: if $(H,U) \in \esys{b}$ and $(I,V) \in \esys{i}(H,U)$, then $V \in \msys{i}(U)$.
\end{defn}

\begin{para}
It is obvious that $\ro{Grp}(G)^{\tn{r-f}}$ is a Mackey cover of any $G$-spectrum and that $\ro{Grp}(G)^{\tn{f}}$ is a Mackey cover of any finite $G$-spectrum.
\end{para}

\begin{defn}
Let $\fr{E}$ be a $G$-spectrum. For a cover $\fr{S}$ of $\fr{E}$ and $C \in \se{Fct}(\fr{S},\se{TAb})$ we define the map
\[
\begin{array}{rcl}
\ind^C_{\fr{E}}: \esys{b} & \lra & \se{TAb} \\
(H,U) & \longmapsto & \ind_{H,U}^C C(U).
\end{array}
\]
If $\ind_{\fr{E}}^C$ is a subfunctor of $C^{\fr{E}} \in \se{Fct}(\fr{E},\se{TAb})$, then we define
\[
\widehat{\ro{H}}^0_{\fr{E}}(C) \dopgleich C^{\fr{E}}/\ind^C_{\fr{E}} \in \se{Rep}(\fr{E}).
\]
Any representation of $\fr{E}$ of the form $\widehat{\ro{H}}^0_{\fr{E}}(C)$ is called an  \word{induction representation}.
\end{defn}

\begin{para} \label{para:ab_ind_rep}
Let $\fr{K} = (\fr{E},\ro{R},D) \in \ro{CCD}(G)$. If $\ro{Ext}(\fr{E},H) \subs \esys{i}^\flat$, then $\widehat{\Pi}_{\fr{K}} = \tateco^0_{\fr{E}}(\pi_{\ro{R}})$ by \ref{para:taut_cft_inductive}. In particular, if $\fr{K} \leq \fr{K}(G)_{\tn{ab}}^{\tn{t}}$, then $\widehat{\Pi}_{\fr{K}} = \tateco_{\fr{E}}^0(\pi_{\ro{ab}}^{\tn{t},G})$. Similarly, if $G$ is compact and $\fr{K} \leq \fr{K}_{\ca{A}}^{\tn{t}}(G)$ for a variety $\ca{A}$ of compact abelian groups, then $\widehat{\Pi}_{\fr{K}} = \tateco^0_{\fr{E}}(\pi_{\ca{A}}^{\tn{t,G}})$. The last two examples are important because isolated on $\fr{E}$ they might not look like an induction representation. 
\end{para}

\begin{prop} \label{para:induction_subfunctor}
Let $\fr{E}$ be a $G$-spectrum, let $\fr{M}$ be a Mackey cover of $\fr{E}$ and let $C \in \se{Mack}(\fr{M},\se{TAb})$. Then $\ind^C_{\fr{E}}$ is a subfunctor of $C^{\fr{E}} \in \se{Fct}(\fr{E},\se{TAb})$ and consequently $C^{\fr{E}}/\ind^C_{\fr{E}}$ is a representation of $\fr{E}$.
\end{prop}

\begin{proof}
Let $(H,U) \in \esys{b}$. First note that $\ind_{H,U}^C C(U)$ is defined since $U \in \ro{Ext}(\fr{E},H) \subs \msys{i}(H)$. Obviously, $\ind^C_{\fr{E}}(H,U)$ is a subgroup of $C^{\fr{E}}(H,U) = C(H)$. Moreover, using the equivariance of $C$, we get
\begin{align*}
\con_{g,(H,U)}^{C^{\fr{E}}}( \ind^C_{\fr{E}}(H,U) ) &= \con_{g,H}^C( \ind^C_{\fr{E}}(H,U) ) = \con_{g,H}^C \circ \ind_{H,U}^C C(U) = \ind_{^g \! H, ^g \! U}^C \circ \con_{g,U}^C( C(U)) \\
&\subs \ind_{^g \! H, ^g \! U}^C C(^g \! U) = \ind^C({^g \! H}, ^g \! U) = \ind^C_{\fr{E}}( ^g \! (H,U)).
\end{align*}

Let $(I,U) \in \esys{r}(H,U)$ and let $R$ be any complete set of representatives of $I \mybackslash H / U$. By definition, we have $I \in \esys{r}^\flat(H) \subs \msys{r}(H)$ and $U \in \ro{Ext}(\fr{E},H) \subs \msys{i}(H)$. Using the fact that $U$ is normal in $H$, an application of the Mackey formula yields
\begin{align*}
\res_{(I,U),(H,U)}^{C^{\fr{E}}}( \ind^C_{\fr{E}}(H,U)) &= \res_{I,H}^C( \ind^C_{\fr{E}}(H,U)) = \res_{I,H}^C \circ \ind_{H,U}^C C(U) \\
&= \prod_{h \in R} \ind_{I,I \cap {^h \! U}}^C \circ \con_{h,I^h \cap U}^C \circ \res_{I^h \cap U,U}^C C(U) \\
&= \prod_{h \in R} \ind_{I,U}^C \circ \con_{h,U}^C \circ \res_{U,U}^C C(U)  \\
&\subs \ind_{I,U}^C C(U) = \ind^C_{\fr{E}}(I,U).
\end{align*}

Finally, let $(I,V) \in \esys{i}(H,U)$. Then $I \in \esys{i}^\flat(H) \subs \msys{i}(H)$ and $V \in \ro{Ext}(\fr{E},I) \subs \msys{i}(I)$. Moreover, $U \in \ro{Ext}(\fr{E},H) \subs \msys{i}(H)$ and $V \in \msys{i}(U)$. Consequently, we can use the transitivity of the induction morphisms to get
\begin{align*}
\ind_{(H,U),(I,V)}^{C^{\fr{E}}}( \ind^C_{\fr{E}}(I,V) ) &= \ind_{H,I}^C( \ind^C_{\fr{E}}(I,V) ) = \ind_{H,I}^C \circ \ind_{I,V}^C C(V) = \ind_{H,V}^C C(V) \\
&= \ind_{H,U}^C \circ \ind_{U,V}^C C(V) \subs \ind_{H,U}^C C(U) = \ind^C_{\fr{E}}(H,U).
\end{align*}
\end{proof}

\begin{para}
Let $\fr{E} \leq \ro{Sp}(G)^{\tn{ri-f}}$ and let $\fr{M} \leq \ro{Grp}(G)^{\tn{i-f}}$ be a Mackey cover of $\fr{E}$. Let $C \in \discgmod$ and let $C_* = \ro{H}_{\fr{M}}^0(C) \in \se{Mack}^{\ro{c}}(\fr{M},\se{Ab})$. Then $\widehat{\ro{H}}_{\fr{E}}^0(C_*)$ is defined and for any $(H,U) \in \esys{b}$ we have
\[
\widehat{\ro{H}}_{\fr{E}}^0(C_*)(H,U) = C_*(H)/\ind_{H,U}^{C_*} C_*(U) = C^H/ \ro{N}_{H/U} C^U = \widehat{\ro{H}}^0( H/U, C^U),
\] 
where $\ro{N}_{H/U}:C^U \ra C^H$ is the norm map in Tate cohomology and $\widehat{\ro{H}}^0( H/U, C^U)$ is the zeroth Tate cohomology group of the $H/U$-module $C^U$.
\end{para}

\subsection{Definition and basic properties of abelian class field theories} \label{sect:acfts}

\begin{para}
In this section we finally discuss abelian class field theories. After giving the definition we will prove several abstract statements about class field theories which are modeled upon the corresponding results in \cite{Fes92_Multidimensional-local_0}, \cite{Fes95_Abelian-local_0} and \cite{Neu99_Algebraic-Number_0}. We restrict here to compact groups so that the dualization will not produce problems and the tautological class field theories have the correct interpretation.
\end{para}

\begin{ass}
Throughout this section $G$ is a compact group and $\fr{K} = (\fr{E},\ro{R},D)$ is a coabelian classification datum for $G$.
\end{ass}

\begin{para}
We call $\hom(D^{\fr{E}},\widehat{\Pi}_{\fr{K}}) \in \se{Fct}(\fr{E},\se{TAb})$ the \words{dualization of the tautological $\fr{K}$-class field theory}{class field theory!dualization of tautological}. Note that as $G$ is compact and as $(H,U) \in \esys{b}$ implies that $H$ is a closed subgroup of $G$ and that $U$ is a closed subgroup of $H$, the values $\widehat{\Pi}_{\fr{K}}(H,U) \cong H/ U \ro{R}(H)$ are compact by \ref{prop:compact_product_closed} so that $\widehat{\Pi}_{\fr{K}} \in \se{Fct}(\fr{E},\se{TAb}^{\ro{com}})$. Hence, $\hom(D^{\fr{E}},\widehat{\Pi}_{\fr{K}})$ is defined by \ref{para:hom_of_ric}. 
\end{para}

\begin{defn} \label{cft_definition}
A \words{$\fr{K}$-reciprocity morphism}{reciprocity morphism} is an isomorphism $\theta: \hom(D^{\fr{E}},\widehat{\Pi}_{\fr{K}}) \ra C^{\fr{E}}/\Phi$ in $\se{Fct}(\fr{E},\se{TAb})$, where $C^{\fr{E}}/\Phi$ is a representation of $\fr{E}$. A \words{$\fr{K}$-class field theory}{class field theory} is a $\fr{K}$-reciprocity morphism $\theta: \hom(D^{\fr{E}},\widehat{\Pi}_{\fr{K}}) \ra C^{\fr{E}}/\Phi$ such that the representation $C^{\fr{E}}/\Phi$ is $\ca{L}$-faithful for any $\ro{R}$-lattice $\ca{L}$ in $\fr{E}$.
\end{defn}

\begin{para}
All what we have discussed in the introduction, including all generalizations, is now captured by this single definition. 
\end{para}

\begin{para}
The identity on the tautological $\fr{K}$-class field theory is obviously a $\fr{K}$-class field theory. A similar statement does not hold for the dualization of the tautological $\fr{K}$-class field theory because it is not a representation.
\end{para}

\begin{para}
Since $\widehat{\Pi}_{\fr{K}} \cong \Pi_{\fr{K}}$, an application of \ref{para:hom_functor_props} shows that $\hom(D^{\fr{E}},\widehat{\Pi}_{\fr{K}}) \cong \hom(D^{\fr{E}},\Pi_{\fr{K}})$. If $D = \ZZ$, then an application of \ref{para:hom_functor_props} shows that $\hom(\ZZ^{\fr{E}},\widehat{\Pi}_{\fr{K}}) \cong \widehat{\Pi}_{\fr{K}}$. 
\end{para}

\begin{para}
Suppose that $\fr{K} \leq \fr{K}(G)_{\tn{ab}}^{\tn{t}}$. 
%
In this case we have $\widehat{\Pi}_{\fr{K}} \cong \Pi_{\fr{K}} \cong \pi_{\fr{K}}$ and so we can identify a morphism $\theta: \hom(D^{\fr{E}},\widehat{\Pi}_{\fr{K}}) \ra C^{\fr{E}}/\Phi$ with a morphism $\hom(D^{\fr{E}},\pi_{\fr{K}}) \ra C^{\fr{E}}/\Phi$ and if $D = \ZZ$, we can identify this with a morphism $\pi_{\fr{K}} \ra C^{\fr{E}}/\Phi$. %
If $\widehat{\ro{H}}^0_{\fr{E}}(C)$ is an induction representation of $\fr{E}$ which is isomorphic to the tautological $\fr{K}$-class field theory, then we have the symmetric picture
\[
\xymatrix{
\widehat{\ro{H}}^0_{\fr{E}}(\pi_{\tn{ab}}^{\tn{t},G}) \ar@{<->}[rr]^-\cong & & \widehat{\ro{H}}^0_{\fr{E}}(C) \\
& \fr{E} \ar@{-->}[ul] \ar@{-->}[ur]
}
\]
The same holds for $G$ being compact and a variety $\ca{A}$ of compact abelian groups. This observation is a hint that class field theories are in general induction representations.
\end{para}

\begin{para}
Proving that a given morphism $\theta:\hom(D^{\fr{E}},\Pi_{\fr{K}}) \ra C^{\fr{E}}/\Phi$ is a $\fr{K}$-class field theory might require a lot of work. Therefore it is important to find abstract theorems about class field theories which, at least in good situations, help to reduce the amount of work. The rest of this section and also the next section is concerned with such reduction theorems. The first important result is that, under very mild conditions, the fact that $\theta$ is an isomorphism already implies that it is a class field theory so that we do not have to care about the $\ca{L}$-faithfulness any more. Then the question is if there are conditions that imply that $\theta$ is already an isomorphism. One answer we can provide is that it is enough to check that $\theta_{(H,U)}$ is an isomorphism only for all $(H,U) \in \esys{b}^{\ro{R},\ell^m}$ for any prime number $\ell$ and $m \in \NN$. But as this is a strong reduction, several assumptions on $\fr{K}$ and $C$ have to hold. In the following definition we will already give some of these conditions. Although this looks like a long and complicated list, the conditions are obvious in most situations coming from applications.  
\end{para}

\begin{defn} \label{para:ccd_coherence} \hfill
\begin{compactitem}[$\bullet$]
\item $\fr{K}$ is called \words{L-coherent}{coabelian classification datum!L-coherent} if the following condition holds for all $H \in \esys{b}^\flat$:
\begin{enumerate}[label=(\roman*),noitemsep,nolistsep]
\item \label{item:ccd_coherence_2dim_ind} If $U_1,U_2 \in \ro{Ext}_{\ro{R}}(\fr{E},H)$ with $U_2 \leq U_1$, then $(H,U_2) \in \esys{i}(H,U_1)$.
\item \label{item:ccd_coherence_lattice_hom_faith} If $U_1,U_2 \in \ro{Ext}_{\ro{R}}(\fr{E},H)$ with $U_2 \leq U_1$ and $\hom_{\se{TAb}}(D(H),U_1/U_2) = 1$, then  $U_1 = U_2$.\footnote{Note that as $U_2 \geq \ro{R}(H)$, the quotient $H/U_2$ is abelian and consequently $U_1/U_2 \leq H/U_2$ is abelian.} \\
\end{enumerate}

\item $\fr{K}$ is called \words{I-coherent}{coabelian classification datum!I-coherent} if $\fr{E} \leq \ro{Sp}(G)^{\tn{r-f}}$ and the following conditions hold for each $H \in \esys{b}^\flat$:
\begin{enumerate}[resume,label=(\roman*),noitemsep,nolistsep]
\item \label{item:ccd_coherence_incl} $\ro{Ext}(\fr{E},H) \subs \esys{i}^\flat(H)$.
\item \label{item:ccd_coherence_attract} If $U \in \ro{Ext}(\fr{E},H)$ and $U_1 \lhd_{\ro{c}} H$ with $U \leq U_1$, then:
\begin{enumerate}[label=(\alph*),noitemsep,nolistsep]
\item $U_1 \in \ro{Ext}(\fr{E},H)$ and $U \in \ro{Ext}(\fr{E},U_1)$.\footnote{Note that this is well-defined by \ref{item:ccd_coherence_incl}.}
\item $(U_1,U) \in \esys{i}(H,U)$ and $(H,U) \in \esys{i}(H,U_1)$.
\end{enumerate}

\item \label{item:ccd_coherence_attract2} If $U \in \ro{Ext}(\fr{E},H)$ and $I$ is an open subgroup of $H$ with $U \leq I \leq H$, then $(I,U) \in \esys{i}(H,U)$. 

\item \label{item:ccd_coherence_hom_surj} If $U_1,U_2 \in \ro{Ext}(\fr{E},H)$ with $U_2 \leq U_1$, then $\ind_{(H,U_1),(H,U_2)}^{\hom(D^{\fr{E}},\Pi_{\fr{K}})}$ is surjective. \\
\end{enumerate}

\item $\fr{K}$ is called \words{coherent}{coabelian classification datum!coherent} if it is L-coherent and I-coherent.
\end{compactitem}
\end{defn}

\begin{para}
Conditions \ref{item:ccd_coherence_lattice_hom_faith} and \ref{item:ccd_coherence_hom_surj} are obviously satisfied if $D = \ZZ$. If $\fr{K}$ is I-coherent, then \ref{item:ccd_coherence_2dim_ind} is already satisfied. It is easy to see that any $\fr{K} = (\ro{Sp}(G)^\star,\ro{R},\ZZ)$ and so in particular $\fr{K}(G)_{\tn{ab}}^\star$ is coherent for $\star \in \lbrace \tn{r-f},\tn{ri-f},\tn{f} \rbrace$. The coabelian classification datum that can be used to describe Fesenko's $p$-class field theories is also coherent.
\end{para}

\begin{para}
Before we start with the reduction theorems we will discuss one important point that was already highlighted in \ref{para:ab_ind_rep}. In most situations coming from applications the class functor $C$ of a class field theory is naturally defined on $\ro{Grp}(G)^{\tn{f}}$ or even on $\ro{Grp}(G)^{\tn{i-f}}$ and thus on $\ro{Grp}(G)^{\tn{ri-f}}$. This is in particular the case if $C$ comes from a discrete $G$-module. For some reduction theorems and in general for setting up class field theories like the Fesenko--Neukirch class field theory a lot of auxiliary constructions are necessary and these auxiliary constructions will require $C$ to be defined on arbitrary open subgroups of a group $H \in \esys{b}^\flat$. Hence, we have to make sure that $C$ comes from a big enough $G$-subgroup system. This leads to the notion of arithmetic $G$-subgroup systems of which $\ro{Grp}(G)^{\tn{f}}$ and $\ro{Grp}(G)^{\tn{ri-f}}$ are examples. The general setup for these advanced theorems and class field theories is then to require that $C$ is defined on an arithmetic Mackey cover of $\fr{E}$. As stated above, this holds in most situations coming from applications. Moreover, this assumption does not contradict our general point of view since the class field theory itself still lives just on $\fr{E}$ and the smaller $\fr{E}$ is the smaller is the amount of work necessary to prove that a given morphism is a class field theory.

\begin{center} 
\scalebox{1.0}{
\begin{picture}(0,0)%
\includegraphics{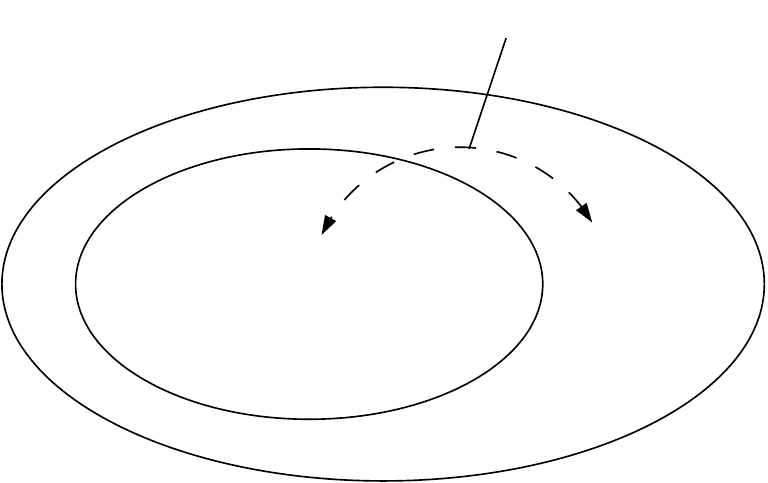}%
\end{picture}%
\setlength{\unitlength}{4144sp}%
\begin{picture}(3502,2196)(5224,-4568)
\put(8269,-4449){\makebox(0,0)[lb]{\smash{{\SetFigFont{8}{9.6}{\rmdefault}{\mddefault}{\updefault}{\color[rgb]{0,0,0}$\mathfrak{M}$}%
}}}}
\put(6188,-3718){\makebox(0,0)[lb]{\smash{{\SetFigFont{8}{9.6}{\rmdefault}{\mddefault}{\updefault}{\color[rgb]{0,0,0}$C^{\mathfrak{E}}/\Phi$}%
}}}}
\put(6807,-2479){\makebox(0,0)[lb]{\smash{{\SetFigFont{8}{9.6}{\rmdefault}{\mddefault}{\updefault}{\color[rgb]{0,0,0}$\textnormal{auxiliary constructions}$}%
}}}}
\put(8044,-3661){\makebox(0,0)[lb]{\smash{{\SetFigFont{8}{9.6}{\rmdefault}{\mddefault}{\updefault}{\color[rgb]{0,0,0}$C$}%
}}}}
\put(7088,-4336){\makebox(0,0)[lb]{\smash{{\SetFigFont{8}{9.6}{\rmdefault}{\mddefault}{\updefault}{\color[rgb]{0,0,0}$\mathfrak{E}$}%
}}}}
\end{picture}%
}

\footnotesize{The class functor $C$ is defined on an arithmetic Mackey cover $\fr{M}$ of $\fr{E}$. To prove that the representation $C^{\fr{E}}/\Phi$ of $\fr{E}$ is a class field theory, auxiliary constructions outside $\fr{E}$ are necessary.}
\end{center}

\end{para}

\begin{defn}
A $G$-subgroup system $\fr{S}$ is called \words{arithmetic}{$G$-subgroup system!arithmetic} if the following condition holds: if $H \in \ssys{b}$, then $\ssys{\star}(H)$ contains all open subgroups of $H$ for $\star \in \lbrace \ro{r},\ro{i} \rbrace$.
\end{defn}

\begin{thm}  \label{para:cft_lattice_thm}
Suppose that $\fr{K}$ is L-coherent. Let $C^{\fr{E}}/\Phi \in \se{Rep}(\fr{E})$ and let $\theta: \hom(D^{\fr{E}},\Pi_{\fr{K}}) \ra C^{\fr{E}}/\Phi$ be an isomorphism. The following holds for $(H,U) \in \esys{b}$:
\begin{enumerate}[label=(\roman*),noitemsep,nolistsep]
\item $\Phi(H,U)$ is closed in $C(H)$.
\item If $U$ is open in $H$ and either $D = \ZZ$ or $D \in \se{Fct}(\fr{E}^\flat,\se{TAb}^{\ro{com}})$, then $\Phi(H,U)$ is open in $C(H)$.
\item If $U \ro{R}(H) \in \ro{Ext}(\fr{E},H)$, then $\Phi(H,U \ro{R}(H)) = \Phi(H,U)$.
\item If $g \in G$, then $\con_{g,H}^C( \Phi(H,U)) = \Phi(^g \! (H,U))$.
\item The map $\Phi(H,-):\ro{Ext}(\fr{E},H) \ra \ca{E}^{\ro{t}}(C(H))$ is monotone.
\item \label{item:cft_lattice_thm_identif} Let $U_1,U_2 \in \ro{Ext}(\fr{E},H)$ such that $U_2 \leq U_1$. Identify 
\[
\hom_{\se{TAb}}(D(H),U_1\ro{R}(H)/U_2\ro{R}(H)) \subs \hom_{\se{TAb}}(D(H),H/U_2\ro{R}(H)) = \hom(D^{\fr{E}},\Pi_{\fr{K}})(H,U_2)
\]
using $\hom_{\se{TAb}}(D(H),\iota)$, where $\iota:U_1\ro{R}(H)/U_2 \ro{R}(H) \ra H/U_2\ro{R}(H)$ is the inclusion. Then
\[
\theta_{(H,U_2)}( \hom_{\se{TAb}}(D(H),U_1\ro{R}(H)/U_2\ro{R}(H))  ) = \Phi(H,U_1)/\Phi(H,U_2) \leq C(H)/\Phi(H,U_2).
\]
\item If $U_1,U_2 \in \ro{Ext}_{\ro{R}}(\fr{E},H)$ such that $U_1 U_2 \in \ro{Ext}(\fr{E},H)$ and $U_1 \cap U_2 \in \ro{Ext}(\fr{E},H)$, then 
\[
\Phi(H,U_1 U_2) = \Phi(H,U_1) \cdot \Phi(H,U_2).
\]

\item \label{item:cft_lattice_thm_intersect} If $U_1,U_2 \in \ro{Ext}_{\ro{R}}(\fr{E},H)$ such that $U_1 \cap U_2 \in \ro{Ext}(\fr{E},H)$, then 
\[
\Phi(H,U_1 \cap U_2) = \Phi(H,U_1) \cap \Phi(H,U_2).
\]
\item \label{item:cft_lattic_injective} Let $\ca{L}$ be an R-lattice in $\fr{E}$. Then 
\[
\Phi(H,-)|_{\ca{L}(H)}: \ca{L}(H) \ra \ca{E}^{\ro{t}}(C(H))
\]
is injective.
\end{enumerate}

\end{thm}

\newpage
\begin{proof}\hspace{-5pt}\footnote{The proof given here is partially based on the proof of \cite[corollary to the theorem in section 1.8]{Fes95_Abelian-local_0}.} \hfill

\begin{asparaenum}[(i)]
\item Since $H$ is compact and both $U$ and $\ro{R}(H)$ are closed in $H$, the product $U \ro{R}(H)$ is also closed in $H$ by \ref{prop:compact_product_closed}. Hence, $H/U\ro{R}(H)$ is separated and therefore $\hom(D^{\fr{E}},\Pi_{\fr{K}})(H,U) = \hom_{\se{TAb}}(D(H),H/U\ro{R}(H))$ is also separated by \ref{para:hom_co_separated}. As 
\[
\theta_{(H, U)}: \hom(D^{\fr{E}},\Pi_{\fr{K}})(H,U) \ra C(H)/\Phi(H,U)
\]
is an isomorphism, it follows that $C(H)/\Phi(H,U)$ is also separated and therefore $\Phi(H,U)$ is a closed subgroup of $C(H)$ (confer also \ref{prop:prop_of_quots}).

\item Since $U$ is open, the product $U \ro{R}(H)$ is open in $H$. If $D = \ZZ$, then the canonical isomorphism $\Pi_{\fr{K}} \ra \hom(\ZZ,\Pi_{\fr{K}})$ composed with $\theta$ is an isomorphism $\vartheta: \Pi_{\fr{K}} \ra C^{\fr{E}}/\Phi$ so that in particular $\vartheta_{(H,U)}: H/U \ro{R}(H) \ra C(H)/\Phi(H,U)$ is an isomorphism. Since $U \ro{R}(H)$ is open in $H$, the quotient $H/U \ro{R}(H)$ is discrete and consequently, $C(H)/\Phi(H,U)$ is discrete (confer \ref{prop:prop_of_quots}). But this implies that $\Phi(H,U)$ is open in $C(H)$.

If $D \in \se{Fct}(\fr{E}^\flat,\se{TAb}^{\ro{com}})$, then an application of \ref{para:hom_co_discrete} shows that $\hom(D^{\fr{E}},\Pi_{\fr{K}})(H,U)$ is discrete and consequently $C(H)/\Phi(H,U)$ is discrete. This again implies that $\Phi(H,U)$ is open in $C(H)$.

\item Since $\theta$ is a morphism, we have the following commutative diagram
\[
\xymatrix{
\hom_{\se{TAb}}(D(H),H/U\ro{R}(H)) \ar@{=}[d] \ar@{=}[rr] & & \hom_{\se{TAb}}(D(H),H/((U\ro{R}(H))\ro{R}(H))) \ar@{=}[d]  \\
\hom(D^{\fr{E}},\Pi_{\fr{K}})(H,U) \ar[d]_{\theta_{(H,U)}}^{\cong} \ar[rr]^{\ind_{(H, U\ro{R}(H)),(H,U)}^{\hom(D,\Pi_{\fr{K}})}} & & \hom(D^{\fr{E}},\Pi_{\fr{K}})(H,U\ro{R}(H)) \ar[d]^{\theta_{(H,U\ro{R}(H))}}_{\cong} \\
C(H)/\Phi(H,U) \ar[rr]^{\ind_{(H,U \ro{R}(H)),(H,U)}^{C^{\fr{E}}/\Phi}} & & C(H)/\Phi(H,U \ro{R}(H)) \\
C(H) \ar@{->>}[u] \ar@{=}[rr] & & C(H) \ar@{->>}[u]
}
\]

As $H/((U\ro{R}(H))\ro{R}(H)) = H/U\ro{R}(H)$, we have $\ind_{(H, U\ro{R}(H)),(H,U)}^{\hom(D,\Pi_{\fr{K}})} = \id$ and therefore $\ind_{(H,U \ro{R}(H)),(H,U)}^{C^{\fr{E}}/\Phi}$ is an isomorphism. But as this morphism is induced by the quotient morphism $C(H) \ra C(H)/\Phi(H,U\ro{R}(H))$, this implies $\Phi(H,U) = \Phi(H,U \ro{R}(H))$.

\item Since $\Phi$ is a subfunctor of $C^{\fr{E}}$, we have
\[
\con_{g,H}^C(\Phi(H,U)) \subs \Phi( ^g \! (H,U)).
\]
Similarly, we have $\con_{g^{-1},H}^C(\Phi(^g \! (H,U))) \subs \Phi(H,U)$ which implies that
\[
\Phi(^g \! (H,U)) \subs \con_{g,H}^C(\Phi(H,U)).
\]
Hence, $\con_{g,H}^C(\Phi(H,U)) = \Phi(^g \! (H,U))$.

\item Let $U_1,U_2 \in \ro{Ext}(\fr{E},H)$ with $U_2 \leq U_1$. Then $(H,U_2) \in \esys{i}(H,U_1)$ by \ref{para:ccd_coherence}\ref{item:ccd_coherence_2dim_ind}, and as $\Phi$ is a subfunctor of $C^{\fr{E}}$ we have
\[
\Phi(H,U_2) =  \ind_{H,H}^C( \Phi(H,U_2)) = \ind_{(H,U_1),(H,U_2)}^{C^{\fr{E}}}( \Phi(H,U_2) ) \subs \Phi(H,U_1).
\]
Hence, $\Phi(H,-)$ is monotone.

\item Since $(H,U_2) \in \esys{i}(H,U_1)$, we have the following commutative diagram
\[
\xymatrix{
& & \hom_{\se{TAb}}(D(H),H/U_1\ro{R}(H)) \ar[r]^-{\theta_{(H, U_1)}}_-\cong & C(H)/\Phi(H,U_1) \\
\hom_{\se{TAb}}(D(H),U_1\ro{R}(H)/U_2\ro{R}(H)) \ar[rr]_-{\hom_{\se{TAb}}(D(H),\iota)} & & \hom_{\se{TAb}}(D(H),H/U_2\ro{R}(H)) \ar[r]_-{\theta_{(H, U_2)}}^-\cong \ar[u]^{\ind_{(H,U_1),(H,U_2)}^{\hom(D^{\fr{E}},\Pi_{\fr{K}})}} & C(H)/\Phi(H,U_2) \ar[u]_{\ind_{(H,U_1),(H,U_2)}^{C^{\fr{E}}/\Phi}}
}
\]
Noting that $\ind_{(H,U_1),(H,U_2)}^{\hom(D^{\fr{E}},\Pi_{\fr{K}})} = \hom_{\se{TAb}}(D(H),q)$, where $q:H/U_2\ro{R}(H) \ra H/U_1\ro{R}(H)$ is the morphism induced by the quotient morphism $H \ra H/U_1\ro{R}(H)$, we see that the sequence
\[
\hom_{\se{TAb}}(D(H),U_1\ro{R}(H)/U_2\ro{R}(H)) \ra \hom_{\se{TAb}}(D(H),H/U_2\ro{R}(H)) \ra \hom_{\se{TAb}}(D(H),H/U_1\ro{R}(H))
\]
in the diagram above is exact in $\se{Ab}$. This yields
\begin{align*}
& \ind_{(H,U_1),(H,U_2)}^{C^{\fr{E}}/\Phi} \circ \theta_{(H, U_2)}( \hom_{\se{TAb}}(D(H),U_1\ro{R}(H)/U_2\ro{R}(H))) \\
&=  \ \theta_{(H, U_1)} \circ \ind_{(H,U_1),(H,U_2)}^{\hom(D^{\fr{E}},\Pi_{\fr{K}})}(\hom_{\se{TAb}}(D(H),U_1\ro{R}(H)/U_2\ro{R}(H))) \\
&= \ \theta_{(H, U_1)}(1) = 1
\end{align*}
and therefore
\begin{align*}
\theta_{(H, U_2)}( \hom_{\se{TAb}}(D(H),U_1\ro{R}(H)/U_2\ro{R}(H))) &\subs \ker(\ind_{(H,U_1),(H,U_2)}^{C^{\fr{E}}/\Phi}) \\
& = \Phi(H,U_1)/\Phi(H,U_2) \\
& \subs C(H)/\Phi(H,U_2),
\end{align*}
where we used that $\Phi(H,U_2) \subs \Phi(H,U_1)$. 

Conversely, we have
\begin{align*}
1 = \theta_{(H, U_1)}^{-1}(1) &= \theta_{(H,U_1)}^{-1} \circ \ind_{(H,U_1),(H,U_2)}^{C^{\fr{E}}/\Phi}( \Phi(H,U_1)/\Phi(H,U_2)) \\
 &= \ind_{(H,U_1),(H,U_2)}^{\hom(D^{\fr{E}},\Pi_{\fr{K}})} \circ \theta_{(H, U_2)}^{-1}( \Phi(H,U_1)/\Phi(H,U_2))
\end{align*}
and therefore
\[
\theta_{(H, U_2)}^{-1}( \Phi(H,U_1)/\Phi(H,U_2)) \subs \ker(\ind_{(H,U_1),(H,U_2)}^{\hom(D^{\fr{E}},\Pi_{\fr{K}})}) = \hom_{\se{TAb}}(D(H),U_1\ro{R}(H)/U_2\ro{R}(H))
\]
which implies that
\[
 \Phi(H,U_1)/\Phi(H,U_2) \subs \theta_{(H, U_2)}( \hom_{\se{TAb}}(D(H),U_1\ro{R}(H)/U_2\ro{R}(H))).
\]

\item Note that $U_i = U_i \ro{R}(H)$ since $U_i \in \ro{Ext}_{\ro{R}}(\fr{E},H)$. We identify
\[
\hom_{\se{TAb}}(D(H),U_i/U_1 \cap U_2) \subs \hom_{\se{TAb}}(D(H),H/U_1 \cap U_2) = \hom(D^{\fr{E}},\Pi_{\fr{K}})(H,U_1 \cap U_2)
\]
and
\[
\hom_{\se{TAb}}(D(H),U_1U_2/U_1 \cap U_2) \subs \hom_{\se{TAb}}(D(H),H/U_1 \cap U_2) = \hom(D^{\fr{E}},\Pi_{\fr{K}})(H,U_1 \cap U_2)
\]
as in \ref{item:cft_lattice_thm_identif}. Then we have
\[
\hom_{\se{TAb}}(D(H),U_1U_2/U_1 \cap U_2) = \hom_{\se{TAb}}(D(H),U_1/U_1 \cap U_2) \cdot \hom_{\se{TAb}}(D(H),U_2/U_1 \cap U_2).
\]
Hence, using \ref{item:cft_lattice_thm_identif}, we get
\begin{align*}
& \Phi(H,U_1U_2)/\Phi(H,U_1 \cap U_2) \\
& = \theta_{(H, U_1 \cap U_2)}( \hom_{\se{TAb}}(D(H),U_1U_2/U_1 \cap U_2)) \\
&= \theta_{(H, U_1 \cap U_2)}( \hom_{\se{TAb}}(D(H),U_1/U_1 \cap U_2) \cdot \hom_{\se{TAb}}(D(H),U_2/U_1 \cap U_2)) \\
&= \theta_{(H, U_1 \cap U_2)}( \hom_{\se{TAb}}(D(H),U_1/U_1 \cap U_2)) \cdot \theta_{(H, U_1 \cap U_2)}(\hom_{\se{TAb}}(D(H),U_2/U_1 \cap U_2)) \\
&= (\Phi(H,U_1)/\Phi(H,U_1 \cap U_2)) \cdot (\Phi(H,U_2)/\Phi(H,U_1 \cap U_2))\\
&= (\Phi(H,U_1) \cdot \Phi(H,U_2))/\Phi(H,U_1 \cap U_2).
\end{align*}
Since $\Phi(H,-)$ is monotone, we have $\Phi(H,U_i) \geq \Phi(H,U_1 \cap U_2)$ and $\Phi(H,U_1U_2) \geq \Phi(H,U_1 \cap U_2)$. Hence, the equation above implies that
\[
\Phi(H,U_1U_2) = \Phi(H,U_1) \cdot \Phi(H,U_2).
\]

\item We identify
\[
\hom_{\se{TAb}}(D(H),U_i/U_1 \cap U_2) \subs \hom_{\se{TAb}}(D(H),H/U_1 \cap U_2) = \hom(D^{\fr{E}},\Pi_{\fr{K}})(H,U_1 \cap U_2)
\]
as in \ref{item:cft_lattice_thm_identif}. Using \ref{item:cft_lattice_thm_identif} we get
\begin{align*}
1 &= \Phi(H,U_1 \cap U_2)/\Phi(H,U_1 \cap U_2) \\
&= \theta_{(H, U_1 \cap U_2)}( \hom_{\se{TAb}}(D(H),U_1 \cap U_2/U_1 \cap U_2)) \\
&= \theta_{(H, U_1 \cap U_2)}( \hom_{\se{TAb}}(D(H), (U_1/U_1 \cap U_2) \cap (U_2/U_1\cap U_2))) \\
&= \theta_{(H, U_1 \cap U_2)}( \hom_{\se{TAb}}(D(H), U_1/U_1 \cap U_2) \cap \hom_{\se{TAb}}(D(H), U_2/U_1 \cap U_2) ) \\
&= \theta_{(H, U_1 \cap U_2)}( \hom_{\se{TAb}}(D(H), U_1/U_1 \cap U_2)) \cap \theta_{(H, U_1 \cap U_2)}( \hom_{\se{TAb}}(D(H), U_2/U_1 \cap U_2) ) \\
&= (\Phi(H,U_1)/\Phi(H,U_1 \cap U_2)) \cap (\Phi(H,U_2)/\Phi(H,U_1 \cap U_2)) \\
&= (\Phi(H,U_1) \cap \Phi(H,U_2))/\Phi(H,U_1 \cap U_2).
\end{align*}
Hence,
\[
\Phi(H,U_1) \cap \Phi(H,U_2) = \Phi(H,U_1 \cap U_2).
\]

\item We prove the injectivity of the map $\Phi(H,-)|_{\ca{L}(H)}$ by proving that it is strongly monotone. We have already proven that $\Phi(H,-)$ is monotone, so suppose that $U_1,U_2 \in \ca{L}(H)$ with $\Phi(H,U_2) \leq \Phi(H,U_1)$. Then we can apply \ref{item:cft_lattice_thm_intersect} to get
\[
\Phi(H,U_2) = \Phi(H,U_2) \cap \Phi(H,U_1) = \Phi(H,U_1 \cap U_2).
\]
According to \ref{item:cft_lattice_thm_identif} we have
\[
1 = \Phi(H,U_2)/\Phi(H,U_1 \cap U_2) = \theta_{(H, U_2)}( \hom_{\se{TAb}}(D(H),U_2/U_1 \cap U_2))
\]
and as $\theta_{(H, U_2)}$ is an isomorphism, we get
\[
 \hom_{\se{TAb}}(D(H),U_2/U_1 \cap U_2) = 1.
\]
Now, it follows from \ref{para:ccd_coherence}\ref{item:ccd_coherence_lattice_hom_faith} that $U_2 = U_1 \cap U_2$, that is, $U_2 \leq U_1$. \vspace{-\baselineskip}
\end{asparaenum}
\end{proof}

\begin{cor} \label{cor:l_coh_iso_cft}
Suppose that $\fr{K}$ is L-coherent. If $\theta:\hom(D^{\fr{E}},\Pi_{\fr{K}}) \ra C^{\fr{E}}/\Phi$ is an isomorphism, then it is already a $\fr{K}$-class field theory.
\end{cor}

\begin{proof}
This is now obvious.
\end{proof}

\begin{para}
A central aspect of an established class field theory $\theta: \Pi_{\fr{K}} \ra C^{\fr{E}}/\Phi$ is to provide an explicit description of the image of $\Phi(H,-)|_{\ca{L}(H)}$ for certain R-lattices $\ca{L}$ in $\fr{E}$. Theorems about the structure of these images are usually called \word{existence theorems}. One such theorem is the following which states that (in the non-dualized case) the property of $\ca{L}(H)$ being a filter on $(\ca{E}^{\ro{t}}(H),\subs)$ is preserved by $\Phi$.
\end{para}

\begin{prop}
Suppose that $\fr{K}$ is L-coherent. Let $\theta: \Pi_{\fr{K}} \ra C^{\fr{E}}/\Phi$ be a $\fr{K}$-class field theory and let $\ca{L}$ be an R-lattice in $\fr{E}$. Let $H \in \esys{b}^\flat$. If $\ca{L}(H)$ is a filter on $(\ca{E}^{\ro{t}}(H),\subs)$, then $\lbrace \Phi(H,U) \mid U \in \ca{L}(H) \rbrace$ is a filter on $(\ca{E}^{\ro{t}}(C(H)),\subs)$.
\end{prop}

\begin{proof}\hspace{-5pt}\footnote{The proof given here is based on the proof of \cite[chapter IV, theorem 6.7]{Neu99_Algebraic-Number_0}.}
Since $\ca{L}(H)$ is a filter, we have $\ca{L}(H) \neq \emptyset$ and therefore also $\lbrace \Phi(H,U) \mid U \in \ca{L}(H) \rbrace \neq \emptyset$. Moreover, as $\Phi(H,U)$ is a subgroup of $C(H)$ for each $U \in \ca{L}(H)$, we also have $\emptyset \notin \lbrace \Phi(H,U) \mid U \in \ca{L}(H) \rbrace$. This shows that \ref{para:set_filter}\ref{item:set_filter_empty} is satisfied. According to \ref{para:cft_lattice_thm}\ref{item:cft_lattice_thm_intersect} we have $\Phi(H,U_1) \cap \Phi(H,U_2) = \Phi(H,U_1 \cap U_2)$ for all $U_1,U_2 \in \ca{L}(H)$ and this shows that \ref{para:set_filter}\ref{item:set_filter_int_closed} is satisfied. Finally, let $\Omega \in \ca{E}^{\ro{t}}(C(H))$ such that $\Phi(H,V) \subs \Omega$ for some $V \in \ca{L}(H)$. We have to show that $\Omega = \Phi(H,U)$ for some $U \in \ca{L}(H)$. Let $q_{H \mid V}^C:C(H) \ra C(H)/\Phi(H,V)$ and $q_{H \mid V}:H \ra H/V$ be the quotient morphisms. Let
\[
U \dopgleich q_{H \mid V}^{-1}( \theta_{(H, V)}^{-1} \circ q_{H \mid V}^C(\Omega)).
\]
Since $\Omega$ is closed in $C(H)$ and $\Phi(H,V) \leq \Omega$, it follows from \ref{prop:prop_of_quots}\ref{item:top_grp_quot_image_closed_open} that $q_{H \mid V}^C(\Omega)$ is closed in $C(H)/\Phi(H,V)$ and therefore $U$ is a closed normal subgroup of $H$. As $V \leq U$ and $V \in \ca{L}(H)$, we have $U \in \ca{L}(H)$ by assumption. In particular, $(H,V) \in \esys{i}(H,U)$ by \ref{para:ccd_coherence}\ref{item:ccd_coherence_2dim_ind} and as $\theta$ is a morphism, we have the following commutative diagram
\[
\xymatrix{
H \ar[rr]^-{q_{H \mid U}} \ar@{=}[d] & & H/U \ar[rr]^-{\theta_{(H, U)}} & & C(H)/\Phi(H,U) & & C(H) \ar[ll]_-{q_{H \mid U}^C} \ar@{=}[d] \\
H \ar[rr]_-{q_{H \mid V}} & & H/V \ar[rr]_-{\theta_{(H, V)}} \ar[u]^{\ind_{(H,U),(H,V)}^{\Pi_{\fr{K}}}} & & C(H)/\Phi(H,V) \ar[u]_{\ind_{(H,U),(H,V)}^{C^{\fr{E}}/\Phi}} & & C(H) \ar[ll]^-{q_{H \mid V}^C}
}
\]
This diagram yields the implications
\[
x \in \Omega \lRA \theta_{(H, V)}^{-1} \circ q_{H \mid V}^C(x) \in q_{H \mid V}(U) \lRA \ind_{(H,U),(H,V)}^{\Pi_{\fr{K}}} \circ \theta_{(H, V)}^{-1} \circ q_{H \mid V}^C(x) = 1
\]
\[
\lRA \theta_{(H, U)}^{-1} \circ q_{H \mid U}^C(x) = 1 \lRA q_{H \mid U}^C(x) = 1 \lRA x \in \Phi(H,U),
\]
and
\[
x \in \Phi(H,U) \lRA q_{H \mid U}^C(x) = 1 \lRA \theta_{(H, U)}^{-1} \circ q_{H \mid U}^C(x) = 1
\lRA \ind_{(H,U),(H,V)}^{\Pi_{\fr{K}}} \circ \theta_{(H,V)}^{-1} \circ q_{H \mid V}^C(x) = 1
\]
\[
\lRA \theta_{(H,V)}^{-1} \circ q_{H \mid V}^C(x) \in \ker(\ind_{(H,U),(H,V)}^{\Pi_{\fr{K}}}) = q_{H \mid V}(U) = q_{H \mid V}( q_{H \mid V}^{-1}( \theta_{(H,V)}^{-1} \circ q_{H \mid V}^C(\Omega))) = \theta_{(H,V)}^{-1} \circ q_{H \mid V}^C(\Omega)
\]
\[
\lRA q_{H \mid V}^C(x) \in q_{H \mid V}^C(\Omega) \lRA (\exists y \in \Omega)( q_{H \mid V}^C(x) = q_{H \mid V}^C(y)) \lRA xy^{-1} \in \ker(q_{H \mid V}^C) = \Phi(H,V)
\]
\[
\lRA x \in \Phi(H,V) \cdot y \subs \Omega.
\]
Hence, $\Omega = \Phi(H,U)$.
\end{proof}

\begin{defn}
A \word{$\fr{K}$-prime power morphism} (respectively a \word{$\fr{K}$-prime morphism}) is a morphism $\theta: \Phi \ra \Psi$ in $\se{Fct}(\fr{E},\se{TAb})$ such that $\theta_{(H,U)}$ is an isomorphism for all $(H,U) \in \esys{b}^{\ro{R},\ell^m}$ for any prime number $\ell$ and any $m \in \NN$ (respectively if $\theta_{(H,U)}$ is an isomorphism for all $(H,U) \in \esys{b}^{\ro{R},\ell}$ for any prime number $\ell$).
\end{defn}

\begin{lemma} \label{para:snake_lemma_lemma}
If
\[
\xymatrix{
A \ar[r] \ar[d]^\alpha & B \ar[r] \ar[d]^\beta & C \ar[r] \ar[d]^\gamma & 0 \\
A' \ar[r]_\delta & B' \ar[r] & C'
}
\]
is a commutative diagram in $\se{Ab}$ with exact rows and with $\alpha$ and $\gamma$ surjective, then $\beta$ is also surjective.
\end{lemma}

\begin{proof}
Let $A'' \dopgleich A'/\ker(\delta)$, let $\alpha':A \ra A''$ be the morphism obtained by composing $\alpha$ with the quotient morphism $A' \ra A''$ and let $\delta': A'' \ra B'$ be the injective morphism induced by $\delta$. Then the following diagram is commutative with exact rows:
\[
\xymatrix{
& A \ar[r] \ar[d]^{\alpha'} & B \ar[r] \ar[d]^\beta & C \ar[r] \ar[d]^\gamma & 0 \\
0 \ar[r] & A'' \ar[r]_{\delta'} & B' \ar[r] & C'
}
\]
As $\alpha'$ and $\gamma$ are surjective, the snake lemma implies that $\beta$ is also surjective.
\end{proof}

\begin{thm} \label{para:prime_power_all_r_thm}
Suppose that $\fr{K}$ is I-coherent. Let $C^{\fr{E}}/\Phi \in \se{Rep}(\fr{E})$ and let $\theta: \hom(D^{\fr{E}},\Pi_{\fr{K}}) \ra C^{\fr{E}}/\Phi$ be a $\fr{K}$-prime power morphism. Then $\theta_{(H,U)}$ is already an isomorphism for all $(H,U) \in \esys{b}^{\ro{R}}$. Moreover, $\theta_{(H,U)}$ is already injective for all $(H,U) \in \esys{b}$.
\end{thm}

\begin{proof}\hspace{-5pt}\footnote{The proof given here is based on the proof of \cite[theorem 1.1]{Fes92_Multidimensional-local_0}.}  
We first prove by induction on $n \in \NN$ that $\theta_{(H,U)}$ is an isomorphism for all $(H,U) \in \esys{b}$ with $U \in \ro{Ext}_{\ro{R}}(\fr{E},H)$ such that the number of (non-trivial) direct summands in the primary decomposition of the finite abelian group $H/U$ is equal to $n$. For $n = 0$ and $n = 1$ this holds by assumption, so suppose that $n > 1$. Let $\varphi: H/U \overset{\cong}{\ra} \bigoplus_{i=1}^r A_i \gleichdop A$ be the primary decomposition of $H/U$ with $A_i = \ZZ/(\ell_i^{m_i})$ for some distinct prime numbers $\ell_i$ and $m_i \in \NN_{>0}$. Let $q: H \ra H/U$ be the quotient morphism and let $B_j \dopgleich \bigoplus_{i=1, i \neq j}^r A_i \leq A$. Then $U_j \dopgleich q^{-1}(\varphi^{-1}(B_j))$ is an open normal subgroup of $H$ with $U \leq U_j$. An application of \ref{para:ccd_coherence}\ref{item:ccd_coherence_attract} shows that $U_j \in \ro{Ext}_{\ro{R}}(\fr{E},H)$ and $U \in \ro{Ext}(\fr{E},U_j)$. Since $U_j \in \ro{Ext}(\fr{E},H) \subs \esys{i}^\flat(H)$, it follows that $\ro{R}(U_j) \leq \ro{R}(H) \leq U$ and consequently, $U \in \ro{Ext}_{\ro{R}}(\fr{E},U_j)$. As $U_j/U \cong B_j$, the number of direct summands in the primary decomposition of $U_j/U$ is equal to $r-1$ and so $\theta_{(U_j,U)}$ is an isomorphism by induction assumption. As $H/U_j \cong A_j$, the number of direct summands in the primary decomposition of $H/U_j$ is equal to $1$ and consequently $\theta_{(H,U_j)}$ is an isomorphism. By \ref{para:ccd_coherence}\ref{item:ccd_coherence_hom_surj} we now have the following commutative diagram with exact rows in $\se{Ab}$
\[
\xymatrix{
 \hom_{\se{TAb}}(D(H),U_1/U) \ar[r] \ar[d]^{\theta_{(U_1,U)}}_{\cong} & \hom_{\se{TAb}}(D(H),H/U) \ar[r] \ar[d]^{\theta_{(H,U)}} & \hom_{\se{TAb}}(D(H),H/U_1) \ar[r] \ar[d]^{\theta_{(H,U_1)}}_{\cong} & 1 \\
 C(U_1)/\Phi(U_1,U) \ar[r]_{\ind_{(H,U),(U_1,U)}^{C^{\fr{E}}/\Phi}} & C(H)/\Phi(H,U) \ar[r]_{\ind_{(H,U_1),(H,U)}^{C^{\fr{E}}/\Phi}} & C(H)/\Phi(H,U_1)
}
\]
where the upper horizontal morphisms are the corresponding induction morphisms of $\hom(D^{\fr{E}},\Pi_{\fr{K}})$. An application of \ref{para:snake_lemma_lemma} shows that $\theta_{(H,U)}$ is surjective. 

It remains to prove that $\theta_{(H,U)}$ is injective. Let $\chi \in \hom_{\se{TAb}}(D(H),H/U)$ with $\theta_{(H,U)}(\chi) = 1$. Since $\theta$ is a morphism, the following diagram commutes for each $j \in \lbrace 1,\ldots,r \rbrace$:
\[
\xymatrix{
\hom(D(H),H/U_j) \ar[rr]^{\theta_{(H,U_j)}}_{\cong} & & C(H)/\Phi(H,U_j)  \\
\hom(D(H),H/U) \ar[rr]_{\theta_{(H,U)}} \ar[u]^{\ind_{(H,U_j),(H,U)}^{\hom(D^{\fr{E}},\Pi_{\fr{K}})}} & & C(H)/\Phi(H,U) \ar[u]_{\ind_{(H,U_j),(H,U)}^{C^{\fr{E}}/\Phi}} \\
}
\]

Let $q_j:H/U \ra H/U_j$ be the morphism induced by the quotient morphism $H \ra H/U_j$.  Then the commutativity of the diagram implies that
\[
1 = \ind_{(H,U_j),(H,U)}^{\hom(D^{\fr{E}},\Pi_{\fr{K}})}(\chi) = q_j \circ \chi.
\] 
It is easy to see that the diagram
\[
\xymatrix{
& H/U_j \ar[r]^\cong & A_j \\
D(H) \ar[r]_{\chi} & H/U \ar[u]^-{q_j} \ar[r]_{\varphi}^\cong & A \ar[u]_-{p_j}
}
\]
commutes, where $p_j$ is the projection and the upper vertical isomorphism is induced by $\varphi$. This shows that $q_j \circ \chi = 1$ for all $j$ implies that $\chi = 1$ and consequently $\theta_{(H,U)}$ is injective. \\

Now, if $(H,U) \in \esys{b}$ with $U \in \ro{Ext}_{\ro{R}}(\fr{E},H)$, then as $\lbrack H:U \rbrack < \infty$, it follows from the above that $\theta_{(H,U)}$ is an isomorphism. This proves the first assertion. \\

It remains to show that $\theta_{(H,U)}$ is injective for all $(H,U) \in \esys{b}$. Since $U \leq U \ro{R}(H) \lhd_{\ro{o}} H$, it follows from \ref{para:ccd_coherence}\ref{item:ccd_coherence_attract} that $U \ro{R}(H) \in \ro{Ext}_{\ro{R}}(\fr{E},H)$. In particular, $\theta_{(H,U\ro{R}(H))}$ is an isomorphism by the above. As $\theta$ is a morphism, the following diagram commutes:
\[
\xymatrix{
\hom_{\se{TAb}}(D(H),H/U \ro{R}(H)) \ar@{=}[r] & \hom(D^{\fr{E}},\Pi_{\fr{K}})(H,U \ro{R}(H)) \ar[rr]^{\theta_{(H,U \ro{R}(H))}}_{\cong} & & C(H)/\Phi(H,U \ro{R}(H)) \\
\hom_{\se{TAb}}(D(H),H/U \ro{R}(H)) \ar@{=}[r] \ar@{=}[u] & \hom(D^{\fr{E}},\Pi_{\fr{K}})(H,U) \ar[rr]_{\theta_{(H,U)}} \ar[u]^{\ind_{(H,U \ro{R}(H)),(H,U)}^{\hom(D^{\fr{E}},\Pi_{\fr{K}})}}
 & & C(H)/\Phi(H,U) \ar[u]_{\ind_{(H,U \ro{R}(H)),(H,U)}^{C^{\fr{E}},\Phi}}
}
\]
It follows immediately that $\theta_{(H,U)}$ is injective.
\end{proof}

\begin{thm} \label{para:prime_power_to_all}
Suppose that $\fr{K} \leq \fr{K}(G)_{\ro{ab}}^{\tn{t}}$ is I-coherent. Let $C^{\fr{E}}/\Phi \in \se{Rep}(\fr{E})$ with $C \in \se{Fct}^{\ro{c}}(\fr{S},\se{TAb})$ for some arithmetic cover $\fr{S}$ of $\fr{E}$. If $\theta: \hom(D^{\fr{E}},\Pi_{\fr{K}}) \ra C^{\fr{E}}/\Phi$ is a $\fr{K}$-prime power morphism, then $\theta$ is already an isomorphism.
\end{thm}

\begin{proof}\hspace{-5pt}\footnote{The proof given here is based on the proof of \cite[chapter IV, theorem 6.3]{Neu99_Algebraic-Number_0}.}
We prove by induction on $n \in \NN_{>0}$ that $\theta_{(H,U)}$ is an isomorphism for all $(H,U) \in \esys{b}$ with $\lbrack H:U \rbrack = n$. By \ref{para:prime_power_all_r_thm} it is enough to prove that $\theta_{(H,U)}$ is surjective. For $n = 1$ this holds by assumption, so suppose that $n > 1$. If $H/U$ is abelian, then $U \in \ro{Ext}_{\ro{R}}(\fr{E},H)$ and $\theta_{(H,U)}$ is an isomorphism by \ref{para:prime_power_all_r_thm}. Suppose that $H/U$ is solvable but not abelian. Then there exists $L \lhd H/U$ with $H/U > L > 1$ and $(H/U)/L$ abelian. Let $q:H \ra H/U$ be the quotient morphism. Then $U_1 \dopgleich q^{-1}(L)$ is an open normal subgroup of $H$ with $U \leq U_1$. Hence, $U_1 \in \ro{Ext}(\fr{E},H)$ and $U \in \ro{Ext}(\fr{E},U_1)$ by \ref{para:ccd_coherence}\ref{item:ccd_coherence_attract}. We have $H > U_1 > U$ and consequently $\lbrack H:U \rbrack > \lbrack H:U_1 \rbrack$ and $\lbrack H:U \rbrack > \lbrack U_1:U \rbrack$. Hence, both $\theta_{(H,U_1)}$ and $\theta_{(U_1,U)}$ are isomorphisms by the induction assumption. By \ref{para:ccd_coherence}\ref{item:ccd_coherence_hom_surj} we have the following commutative diagram with exact rows in $\se{Ab}$:
\[
\xymatrix{
\hom(D^{\fr{E}},\Pi_{\fr{K}})(U_1,U) \ar[rr] \ar[d]^{\theta_{(U_1,U)}}_{\cong} & & \hom(D^{\fr{E}},\Pi_{\fr{K}})(H,U) \ar[rr] \ar[d]^{\theta_{(H,U)}} & & \hom(D^{\fr{E}},\Pi_{\fr{K}})(H,U_1) \ar[d]^{\theta_{(H,U_1)}}_{\cong} \ar[r] & 1\\
C(U_1)/\Phi(U_1,U) \ar[rr]_{\ind_{(H,U),(U_1,U)}^{C^{\fr{E}}/\Phi}} & & C(H)/\Phi(H,U) \ar[rr]_{\ind_{(H,U_1),(H,U)}^{C^{\fr{E}}/\Phi}} & & C(H)/\Phi(H,U_1)
}
\]
The upper horizontal morphisms are the corresponding induction morphisms of $\hom(D^{\fr{E}},\Pi_{\fr{K}})$. It follows from \ref{para:snake_lemma_lemma} that $\theta_{(H,U)}$ is surjective and thus is an isomorphism. 

Now, let $(H,U) \in \esys{b}$ be arbitrary. Let $\ell$ be a prime number and let $S$ be a Sylow $\ell$-subgroup of $H/U$. Let $q:H \ra H/U$ be the quotient morphism. Then $I \dopgleich q^{-1}(S) \in \esys{i}^\flat(H)$ and $U \in \ro{Ext}(\fr{E},I)$ by \ref{para:ccd_coherence}\ref{item:ccd_coherence_attract2}. Since $I/U = S$ is an $\ell$-group and thus solvable, it follows from the above that $\theta_{(I,U)}$ is an isomorphism. Let $T$ be the Sylow $\ell$-subgroup of $C(H)/\Phi(H,U)$. Using the fact that $C$ is defined on an arithmetic cover of $\fr{E}$ and that $C$ is cohomological, we get the following commutative diagram:
\[
\xymatrix{
\hom(D^{\fr{E}},\Pi_{\fr{K}})(I,U) \ar[rr] \ar[d]_{\theta_{(I,U)}}^{\cong} & & \hom(D^{\fr{E}},\Pi_{\fr{K}})(H,U) \ar[d]^{\theta_{(H,U)}} \\
C(I)/\Phi(I,U) \ar[rr]^{\ind_{(H,U),(I,U)}^{C^{\fr{E}}/\Phi}} & & C(H)/\Phi(H,U) \\
C(H)/\Phi(H,U) \ar[u]^{\res_{(I,U),(H,U)}^{C^{\fr{E}}/\Phi}} \ar[urr]_{(-)^{ \lbrack H:I \rbrack}} & & T \ar@{ >->}[u] \\
T \ar@{ >->}[u] \ar@{->>}[urr]_{(-)^{ \lbrack H:I \rbrack}}
}
\]
Here we used that $\lbrack H:I \rbrack$ is relatively prime to $\ell$ so that exponentiation by $\lbrack H:I \rbrack$ maps $T$ onto $T$. The commutativity of the diagram now shows that $T$ lies in the image of $\ind_{(H,U),(I,U)}^{C^{\fr{E}}/\Phi}$ and thus $T$ lies in the image of $\theta_{(H,U)}$. Hence, any Sylow subgroup of $C(H)/\Phi(H,U)$ lies in the image of $\theta_{(H,U)}$. As $C(H)/\Phi(H,U)$ is abelian and can thus be decomposed into the direct sum of its Sylow subgroups, it follows that $\theta_{(H,U)}$ is surjective.
\end{proof}

\subsection{The case of induction representations} \label{para:ind_cfts}

\begin{para}
In the case of induction representations there exists a further very important reduction theorem. This reduction theorem involves a cohomological assumption on the class functor $C$ and so we start by defining the cohomology groups $\widehat{\ro{H}}^{-1}(C)(H,U)$.
\end{para}

\begin{ass}
Throughout this section $G$ is a group and $\fr{S}$ is a $G$-subgroup system. For $H \in \ssys{b}$ we define $\ssys{\star}^\lhd(H)$ to be the subset of $\ssys{\star}(H)$ consisting of normal subgroups of $H$.
\end{ass}

\begin{defn} 
Let $C \in \se{Stab}(\fr{S},\se{Ab})$. For $H \in \ssys{b}$ and $g \in G$ with $^g \! H = H$ we define the $\se{Ab}$-morphism
\[
\begin{array}{rcl}
\con_{g-1,H}^C: C(H) & \lra & C(H) \\
x & \longmapsto & \frac{\con_{g,H}^C(x)}{x}.
\end{array}
\]
\end{defn}

\begin{para}
Note that the condition $^g \! H = H$ implies that $\con_{g-1,H}^C$ is well-defined. Also note that as $G$ is multiplicatively written, the notion $g-1$ should not produce confusions. If $U \in \ssys{b}$ and $H \leq \ro{N}_G(U)$, then $\con_{\ol{h}-1,U}^C \dopgleich \con_{h-1,U}^C$ is defined for any $\ol{h} \in H/U$ and this definition is independent of the choice of the representative $h$ of $\ol{h}$ due to the stability of $C$. If $H \in \ssys{b}$ and $U \in \ssys{i}^\lhd(H)$, then it is easy to see that $\im( \con_{h-1,U}^C ) \subs \ker(\ind_{H,U}^C C(U))$.
\end{para}

\begin{defn} 
Let $C \in \se{Stab}(\fr{S},\se{Ab})$. For $H \in \ssys{b}$ and $U \in \ssys{i}^\lhd(H)$ we define
\[
\widehat{\ro{H}}^0(C)(H,U) \dopgleich \coker(\ind_{H,U}^C) = C(H)/\ind_{H,U}^C C(U)
\] 
and
\[
\widehat{\ro{H}}^{-1}(C) (H,U) \dopgleich \ker(\ind_{H,U}^C C(U))/I_{(H,U)},
\]
where
\[
I_{(H,U)} \dopgleich \langle \lbrace \con_{\ol{h}-1,U}^C(x) \mid \ol{h} \in H/U \tn{ and } x \in C(U) \rbrace \rangle \leq C(U).
\]
\end{defn}

\begin{para}
Suppose that $\fr{S}$ is I-finite, let $C \in \discgmod$ and let $C_* = \ro{H}^0_{\fr{S}}(C) \in \se{Stab}^{\ro{c}}(\fr{S},\se{Ab})$. If $H \in \ssys{b}$ and $U \in \ssys{i}^\lhd(H)$, then
\[
\widehat{\ro{H}}^0(C_*)(H,U) = C_*(H)/\ind_{H,U}^{C_*} C_*(U) = C^H/\ro{N}_{H/U} C^U = \widehat{\ro{H}}^0(H/U, C^U)
\]
and
\[
\widehat{\ro{H}}^{-1}(C_*)(H,U) = \ker(\ind_{H,U}^C C(U))/I_{(H,U)} = \ker(\ro{N}_{H/U})/ I_{(H,U)} = \widehat{\ro{H}}^{-1}(H/U, C^U),
\]
where $\ro{N}_{H/U}:C^U \ra C^H$ is the norm map in Tate cohomology and $\widehat{\ro{H}}^n(H/U, C^U)$ are the Tate cohomology groups of the $H/U$-module $C^U$.
\end{para}

\begin{prop}
Let $C \in \se{Stab}(\fr{S},\se{Ab})$. If $H \in \ssys{b}$ and $U \in \ssys{i}^\lhd(H)$ with $H/U$ being a finite cyclic group, then $I_{(H,U)} = \im(\con_{\ol{h}-1,U}^C)$, where $\ol{h} \in H/U$ is a generator.
\end{prop}

\begin{proof}
Let $h \in H$ be a representative of the generator $\ol{h} \in H/U$. As $\con_{h,U}^C$ is an endomorphism of the abelian group $C(H)$, we can consider $C(U)$ as a $\ZZ[X]$-module with $X$ acting via $\con_{h,U}^C$. For $n \in \NN_{>0}$ we have
\[
X^n - 1 = (X-1) \cdot (\sum_{i=0}^{n-1} X^i)
\]
in $\ZZ[X]$. Consequently, for $x \in C(U)$ we have
\[
\con_{\ol{h}^n-1,U}^C(x) = \con_{h^n-1,U}^C(x) = (X^n-1)x = (X-1) \cdot (\sum_{i=0}^{n-1} X^i) x = \con_{h-1,U}^C( (\sum_{i=0}^{n-1} X^i) x  ) \in \im(\con_{\ol{h}-1,U}^C).
\]
\end{proof}

\begin{defn}
Let $C \in \se{Stab}(\fr{S},\se{Ab})$, let $H \in \ssys{b}$ and $U \in \ssys{i}^\lhd(H)$. Then $C$ is said to satisfy \word{Hilbert 90} for $(H,U)$ if $\widehat{\ro{H}}^{-1}(C)(H,U) = 1$ and $C$ is said to satisfy the \word{class field axiom} for $(H,U)$ if 
\[
| \widehat{\ro{H}}^0(C)(H,U) | = \lbrack H:U \rbrack \quad \tn{ and } \quad \widehat{\ro{H}}^{-1}(C)(H,U) = 1.
\]
\end{defn}

\begin{lemma} \label{para:poly_ident_lemma}
Let $\ell$ be a prime number and let $n \in \NN_{>0}$. Then the relation
\[
(X-1)^{\ell^n} = X^{\ell^n} - 1 + \ell g(X)
\]
holds in $\ZZ[X]$ for some polynomial $g(X) \in \ZZ[X]$.
\end{lemma}

\begin{proof}
Let $0 < i < \ell^n$. Then ${\ell^n \choose i} = {\ell^n-1 \choose i} \frac{\ell^n}{\ell^n - i}$ and therefore $(\ell^n - i){\ell^n \choose i} = {\ell^n-1 \choose i}\ell^n$. This shows that $\ell^n$ divides $(\ell^n - i){\ell^n \choose i}$ and due to the assumptions on $i$ we conclude that $\ell$ divides ${\ell^n \choose i}$. Hence,
\[
(X-1)^{\ell^n} = \sum_{i=0}^{\ell^n} {\ell^n \choose i} X^{\ell^n-i} (-1)^i = X^{\ell^n} + \ell g(X) + (-1)^{\ell^n},
\]
where $g(X) \dopgleich \frac{1}{\ell} \sum_{i=1}^{\ell^n-1} {\ell^n \choose i} X^{\ell^n-i}(-1)^i \in \ZZ[X]$. If $\ell$ is odd, then the equation above is of the desired form. If $\ell = 2$, then $(X-1)^{\ell^n} = X^{\ell^n} -1 + \ell g'(X)$ with $g'(X) \dopgleich g(X) + 1$.
\end{proof}

\begin{thm} \label{para:prime_to_prime_power}
Suppose that $G$ is a compact group, let $\fr{K} = (\fr{E},\ro{R},\ZZ) \in \ro{CCD}(G)$ be I-coherent and let $C \in \se{Mack}^{\ro{c}}(\fr{M},\se{Ab})$ for an arithmetic Mackey cover $\fr{M}$ of $\fr{E}$. Suppose that $C$ satisfies the class field axiom for all $(H,U) \in \esys{b}^{\ro{R},\ell}$ with $\ell$ being any prime number. If $\theta: \Pi_{\fr{K}} \ra \widehat{\ro{H}}_{\fr{E}}^0(C)$ is a $\fr{K}$-prime morphism, then $\theta$ is already an isomorphism for all $(H,U) \in \esys{b}^{\ro{R}}$. Moreover $\widehat{\ro{H}}^{-1}(C)(H,U) = 1$ for all $(H,U) \in \esys{b}^{\ro{R},\ell^m}$ with $\ell$ being any prime number and $m \in \NN$.
\end{thm}

\begin{proof}\hspace{-5pt}\footnote{The proof given here is based on the proof of \cite[theorem 1.1]{Fes92_Multidimensional-local_0}.}
By \ref{para:prime_power_all_r_thm} it is enough to show that $\theta$ is a $\fr{K}$-prime power morphism. Let $\ell$ be a prime number. We prove simultaneously by induction on $m \in \NN$ that $\widehat{\ro{H}}^{-1}(C)(H,U) = 1$ and that $\theta_{(H,U)}$ is an isomorphism for each $(H,U) \in \esys{b}^{\ro{R},\ell^m}$

The assertions are obvious for $m = 0$ and for $m = 1$ they hold by assumption, so let $m > 1$. Let $Z$ be the subgroup of order $\ell$ of $H/U$ and let $U_1 \dopgleich q^{-1}(Z)$, where $q:H \ra H/U$ is the quotient morphism. Since $\fr{K}$ is I-coherent, it follows that $U_1 \in \ro{Ext}_{\ro{R}}^{\ell^{m-1}}(\fr{E},H)$ and $U \in \ro{Ext}_{\ro{R}}^\ell(\fr{E},U_1)$. Hence, the induction assumption implies that $\theta_{(H,U_1)}$ and $\theta_{(U_1,U)}$ are isomorphisms. We will first show that 
\[
\varphi \dopgleich \ind_{(H,U),(U_1,U)}^{\widehat{\ro{H}}_{\fr{E}}^0(C)}: \widehat{\ro{H}}_{\fr{E}}^0(C)(U_1,U) \ra \widehat{\ro{H}}_{\fr{E}}^0(C)(H,U)
\]
is injective. Suppose that this morphism is not injective. As $|\widehat{\ro{H}}^0(C)(U_1,U)| = \ell$ by assumption, we must have $\ker(\varphi) = \widehat{\ro{H}}_{\fr{E}}^0(C)(U_1,U) = C(U_1)/\ind_{U_1,U}^C C(U)$. Hence, if $x \in C(U_1)$, then $\ind_{H,U_1}^C(x) \in \ind_{H,U}^C C(U)$, that is, 
\[
\ind_{H,U_1}^C(x) = \ind_{H,U}^C( y) = \ind_{H,U_1}^C \circ \ind_{U_1,U}^C(y)
\]
for some $y \in C(U)$ and consequently
\[
x \cdot \ind_{U_1,U}^C(y)^{-1} \in \ker( \ind_{H,U_1}^C(x) ).
\]

By induction assumption we have $\widehat{\ro{H}}^{-1}(C)(H,U_1) = 1$ and so there exists $x_1 \in C(U_1)$ such that $\con_{\ol{h}-1,U_1}^C(x_1) = x \cdot \ind_{U_1,U}^C(y)^{-1}$, that is,
\begin{equation} \label{equ:x_conj_first_step}
x = \con_{\ol{h}-1,U_1}^C(x_1) \cdot \ind_{U_1,U}^C(y),
\end{equation}
where $\ol{h}$ is a generator of $H/U_1$. Let $h \in H$ be a representative of $\ol{h}$. We can apply the above arguments again to $x_1$ and get
\[
x_1 =  \con_{\ol{h}-1,U_1}^C(x_2) \cdot \ind_{U_1,U}^C(y_1)
\]
for some $x_2 \in C(U_1)$ and $y_1 \in C(U)$. Plugging this into \ref{equ:x_conj_first_step} yields
\begin{align*}
x &= \con_{\ol{h}-1,U_1}^C( \con_{\ol{h}-1,U_1}^C(x_2) \cdot \ind_{U_1,U}^C(y_1)) \cdot \ind_{U_1,U}^C(y)\\
&= (\con_{\ol{h}-1,U_1}^C)^2(x_2) \cdot (\con_{\ol{h}-1,U_1}^C \circ \ind_{U_1,U}^C(y_1)) \cdot \ind_{U_1,U}^C(y) \\
& = (\con_{\ol{h}-1,U_1}^C)^2(x_2)  \cdot \ind_{U_1,U}^C( \con_{h-1,U}^C(y_1)) \cdot \ind_{U_1,U}^C(y) \\
&=  (\con_{\ol{h}-1,U_1}^C)^2(x_2)  \cdot \ind_{U_1,U}^C( \con_{h-1,U}^C(y_1) \cdot \ind_{U_1,U}^C(y)).
\end{align*}

If we repeat this process $\ell^{m-1}$-times, then we get for any $x \in C(U_1)$ an equation
\begin{equation}
x =  (\con_{\ol{h}-1,U_1}^C)^{\ell^{m-1}}(x')  \cdot \ind_{U_1,U}^C(y')
\end{equation}
for some $x' \in C(U_1)$ and $y' \in C(U)$. Since $\con_{\ol{h}-1,U_1}^C$ is an endomorphism of the abelian group $C(U_1)$, we can consider $C(U_1)$ as a $\ZZ[X]$-module with $X$ acting via $\con_{h-1,U_1}^C$. Using \ref{para:poly_ident_lemma} we get
\begin{align*}
 (\con_{\ol{h}-1,U_1}^C)^{\ell^{m-1}}(x') &= (X-1)^{\ell^{m-1}} x' = (X^{\ell^{m-1}} - 1 + \ell g(X))x' \\
&= \frac{X^{\ell^{m-1}}x'}{x'} \cdot (g(X)x')^\ell = \frac{ (\con_{h,U_1}^C)^{\ell^{m-1}}(x')}{x'} \cdot (g(X)x')^\ell \\
& = \frac{\con_{h^{\ell^{m-1}},U_1}^C(x')}{x'} \cdot (g(X)x')^\ell
\end{align*}
for some polynomial $g(X) \in \ZZ[X]$. Since $H/U_1$ is of order $\ell^{m-1}$, we have $h^{\ell^{m-1}} \in U_1$ and consequently 
\[
\frac{\con_{h^{\ell^{m-1}},U_1}^C(x')}{x'} = \frac{x'}{x'} = 1.
\]
As $g(X)x' \in C(U_1)$ and $| \widehat{\ro{H}}^0(C)(U_1,U) | = \ell$ by assumption, it follows that $(g(X)x')^\ell \in \im(\ind_{U_1,U}^C)$. Taken together, we have proven that $x \in \im(\ind_{U_1,U}^C)$ and consequently $\ker(\varphi) = 1$ contradicting the assumption that $\varphi$ is not injective. Hence, $\varphi$ has to be injective. 
We now have the following commutative diagram with exact rows in $\se{Ab}$
\[
\xymatrix{
& \Pi_{\fr{K}}(U_1,U) \ar[rr] \ar[d]^{\theta_{(U_1,U)}}_{\cong} & & \Pi_{\fr{K}}(H,U) \ar[rr] \ar[d]^{\theta_{(H,U)}} & & \Pi_{\fr{K}}(H,U_1) \ar[r] \ar[d]^{\theta_{(H,U_1)}}_{\cong} & 1 \\
1 \ar[r] & \widehat{\ro{H}}_{\fr{E}}^0(C)(U_1,U) \ar[rr]_{\ind_{(H,U),(U_1,U)}^{\widehat{\ro{H}}_{\fr{E}}^0(C)}} & & \widehat{\ro{H}}_{\fr{E}}^0(C)(H,U) \ar[rr]_{\ind_{(H,U_1),(H,U)}^{\widehat{\ro{H}}_{\fr{E}}^0(C)}} & & \widehat{\ro{H}}_{\fr{E}}^0(C)(H,U_1)
}
\]
where the upper horizontal morphisms are the corresponding induction morphisms of $\Pi_{\fr{K}}$. The exactness of the upper sequence is obvious and the exactness of the lower sequence follows from the above. An application of the snake lemma now shows that $\theta_{(H,U)}$ is an isomorphism. \\

It remains to show that $\widehat{\ro{H}}^{-1}(C)(H,U) = 1$. We first show that $C(U_1) = \res_{U_1,H}^C C(H) \cdot \ind_{U_1,U}^C C(U)$. Since $\theta_{(H,U)}$ is an isomorphism, the group $\widehat{\ro{H}}^0(C)(H,U)$ is of order $|\Pi_{\fr{K}}(H,U)| = |H/U| = \ell^m$ and consequently $C(H)^{\ell^{m-1}} \nsubseteq \ind_{H,U}^C C(U)$. As $C$ is cohomological, the composition $\ind_{H,U_1}^C \circ \res_{U_1,H}^C$ is equal to exponentiation by $\lbrack H:U_1 \rbrack = \ell^{m-1}$ on $C(H)$ and now it follows from the preceding observation that $\res_{U_1,H}^C C(H) \nsubseteq \ind_{U_1,U}^C C(U)$ because if $\res_{U_1,H}^C C(H) \subs \ind_{U_1,U}^C C(U)$, then
\[
C(H)^{\ell^{m-1}} = \ind_{H,U_1}^C \circ \res_{U_1,H}^C C(H) \subs \ind_{H,U_1}^C \circ \ind_{U_1,U}^C C(U) = \ind_{H,U}^C C(U).
\]

Since $|\widehat{\ro{H}}^0(C)(U_1,U)| = \ell$ by assumption, we conclude that $C(U_1) = \res_{U_1,H}^C C(H) \cdot \ind_{U_1,U}^C C(U)$.

To prove that $\widehat{\ro{H}}^{-1}(C)(H,U) = 1$, it suffices to show that 
\[
\ker(\ind_{H,U}^C) \subs I_{(H,U)} = \im(\con_{(h \modd U)-1,U}^C) = \im(\con_{h-1,U}^C),
\]
where we use the fact that $h \modd U$ generates $H/U$. Let $x \in \ker(\ind_{H,U}^C)$. Then
\[
1 = \ind_{H,U}^C(x) = \ind_{H,U_1}^C  \circ \ind_{U_1,U}^C(x)
\]
and consequently $\ind_{U_1,U}^C(x) \in \ker(\ind_{H,U_1}^C)$. Hence, as $\widehat{\ro{H}}^{-1}(C)(H,U_1) = 1$ by induction assumption, there exists $y \in C(U_1)$ such that $\ind_{U_1,U}^C(x) = \con_{\ol{h}-1,U_1}^C(y)$. By the above we can write $y$ as 
\[
y = \res_{U_1,H}^C(z) \cdot \ind_{U_1,U}^C(w)
\]
for some $z \in C(H)$ and $w \in C(U)$. This yields
\begin{align*}
\ind_{U_1,U}^C(x) &= \con_{\ol{h}-1,U_1}^C(y) = \con_{\ol{h}-1,U_1}^C( \res_{U_1,H}^C(z) \cdot \ind_{U_1,U}^C(w) ) \\
& = \con_{\ol{h}-1,U_1}^C( \res_{U_1,H}^C(z) ) \cdot \con_{\ol{h}-1,U_1}^C( \ind_{U_1,U}^C(w) ) \\
& = \con_{\ol{h}-1,U_1}^C( \ind_{U_1,U}^C(w) ) = \ind_{U_1,U}^C( \con_{h-1,U}^C(w))
\end{align*}
and consequently $x \cdot \con_{h-1,U}^C(w)^{-1} \in \ker(\ind_{U_1,U}^C)$. Since $h^{\ell^{m-1}} \modd U$ is a generator of $U_1/U$ and since $\widehat{\ro{H}}^{-1}(C)(U_1,U) = 1$ by assumption, we have 
\[
\ker(\ind_{U_1,U}^C) = \im(\con_{(h^{\ell^{m-1}} \modd U) -1 ,U}^C) =  \im(\con_{h^{\ell^{m-1}}-1,U}^C)
\]
and so there exists $w_1 \in C(U)$ such that $x \cdot \con_{h-1,U}^C(w)^{-1} = \con_{h^{\ell^{m-1}}-1,U}^C(w_1)$. As $\im(\con_{h^{\ell^{m-1}}-1,U}^C) \subs I_{(H,U)} = \im(\con_{h-1,U}^C)$, there exists $w_2 \in C(U)$ such that $\con_{h^{\ell^{m-1}}-1,U}^C(w_1) = \con_{h-1,U}^C(w_2)$ and consequently
\[
x = \con_{h-1,U}^C(w) \cdot \con_{h-1,U}^C(w_2) \subs \im(\con_{h-1,U}^C) = I_{(H,U)}.
\]
This shows that $\widehat{\ro{H}}^{-1}(C)(H,U) = 1$.
\end{proof}

\begin{cor} \label{para:prime_cft_reduction}
Let $G$ be a compact group, let $\fr{K} = (\fr{E},\ro{R},\ZZ) \leq \fr{K}(G)_{\ro{ab}}^{\tn{t}}$ be coherent and let $C \in \se{Mack}^{\ro{c}}(\fr{M},\se{Ab})$ for an arithmetic Mackey cover $\fr{M}$ of $\fr{E}$. Suppose that $C$ satisfies the class field axiom for all $(H,U) \in \esys{b}^{\ell}$ with $\ell$ being any prime number. Then a $\fr{K}$-prime morphism $\rho: \pi_{\fr{K}} \ra \widehat{\ro{H}}_{\fr{E}}^0(C)$ is already a $\fr{K}$-class field theory.
\end{cor}

\begin{proof}
This follows immediately from \ref{para:prime_to_prime_power} and \ref{para:prime_power_to_all}. 
\end{proof}

\subsection{Class field theories for profinite groups} \label{para:cfts_for_profinite}

\begin{para}
Suppose that $G = \gal(k'|k)$ is the Galois group of a Galois extension and that $\fr{S} \leq \ro{Grp}(G)^{\tn{t}}$ is a $G$-subgroup system. As all $H \in \ssys{b}$ are closed subgroups of $G$, we can also interpret the groups in $\ssys{b}$ as extensions of $k$ under the Galois correspondence. Hence, if $C \in \se{Fct}(\fr{S},\ca{C})$ and $H \in \ssys{b}$, then we will also write $C(K)$ instead of $C(H)$, where $K = (k')^H$ is the corresponding extension of $k$. Similarly, we modify the notations of the restriction, induction and conjugation morphisms. Here, one has to be careful with the directions of the restriction and induction morphisms, but the reader who is familiar with the contravariance of the Galois correspondence will not be confused by this.
\end{para}

\begin{para} \label{para:profinite_classical_cft}
The classical situation of class field theories is about a profinite group $G$, the coabelian classification datum $\fr{K} = \fr{K}(G)_{\ro{ab}}^{\tn{f}}$ and $\fr{K}$-class field theories of the form $\tateco_{\ro{Sp}(G)^{\tn{f}}}^0(C)$ with $C \in \se{Mack}^{\ro{c}}(\ro{Grp}(G)^{\tn{f}},\se{Ab})$. As $\fr{K}$ is coherent and as $\ro{Grp}(G)^{\tn{f}}$ is an arithmetic Mackey cover of $\ro{Sp}(G)^{\tn{f}}$, it follows from the above sections that we ``only'' have to take the following three steps to get a $\fr{K}$-class field theory:
\begin{enumerate}[label=(\roman*),noitemsep,nolistsep]
\item Construct a morphism $\rho: \pi_{\fr{K}} \ra \tateco_{\ro{Sp}(G)^{\tn{f}}}^0(C)$.
\item Verify that $C$ satisfies the class field axiom for all $H \leq_{\ro{o}} G$ and $U \lhd_{\ro{o}} H$ with $H/U$ being a cyclic group of prime order.
\item Verify that $\rho_{(H,U)}$ is an isomorphism for all $H \leq_{\ro{o}} G$ and $U \lhd_{\ro{o}} H$ with $H/U$ being a cyclic group of prime order. \\
\end{enumerate}
\end{para}

\begin{para}
Suppose that $G = \gal(k)$. In the context of a fixed $\fr{K}$-class field theory $\rho: \pi_{\fr{K}} \ra \tateco_{\ro{Sp}(G)^{\tn{f}}}^0(C)$ as in \ref{para:profinite_classical_cft} and a finite separable extension $K|k$ the groups $\ind_{K,L}^C C(L) \leq C(K)$ with $L \in \ca{E}^{\tn{f}}(K)$ are called the \words{norm subgroups}{norm subgroup} of $C(K)$. If $N$ is a norm subgroup of $C(K)$, then the unique abelian extension of $K$ such that $\ind_{K|L}^C C(L) = N$ is called the \word{class field} of $N \leq C(K)$.
\end{para}

\newpage
\section{Fesenko--Neukirch class field theories} \label{chap:fs_cfts}

In this chapter we will discuss in detail the Fesenko--Neukirch class field theories. These are $\fr{K}(G)_{\ro{ab}}^{\tn{f}}$-class field theories of the form $\widehat{\ro{H}}_{\ro{Sp}(G)^{\tn{f}}}^0(C)$ for a profinite group $G$ and a cohomological Mackey functor $C$, where both $G$ and $C$ have to satisfy several conditions to make this theory work. This purely group-theoretical and non-cohomological approach to class field theories was discovered by \name{Jürgen Neukirch} who presented it in \cite{Neu86_Class-Field_0} and later again in \cite{Neu99_Algebraic-Number_0}, where in this context $C$ was only a discrete $G$-module. The fundamental principle of these class field theories is the correspondence between Frobenius automorphisms and prime elements in local class field theory for unramified extensions: if $(k,v)$ is a local field and $L|K$ is a finite unramified Galois extension of a finite separable extension $K|k$, then
\[
\begin{array}{rcl}
\Upsilon_{L \mid K}: \gal(L|K) & \lra & K^\times/\ro{N}_{L|K} L^\times \\
\varphi_{L|K} & \longmapsto & \pi_K \modd \ro{N}_{L|K} L^\times
\end{array}
\]
is an isomorphism, where $\varphi_{L|K}$ is the Frobenius automorphism of the extension $L| K$ and $\pi_K$ is a prime element in $K^\times$ (with respect to the normalization $v_K$ of the unique extension of $v$ to $K$). This isomorphism is canonical since up to equivalence modulo the norm group $\ro{N}_{L|K} L^\times$ there exists only one prime element in $K^\times$. Moreover, the collection of these isomorphisms is compatible with the conjugation, transfer and inclusion morphisms so that in the end we get a $\fr{K}(\fr{E})_{\ro{ab}}$-class field theory $\Upsilon: \pi_{\fr{K}(\fr{E})_{\ro{ab}}} \ra \widehat{\ro{H}}_{\fr{E}}^0(\ro{GL}_1(k^{\tn{s}})_*)$ for the $\gal(k)$-spectrum $\fr{E}$ made up by the finite unramified extensions. The composition of the quotient morphism $K^\times \ra K^\times/\ro{N}_{L|K} L^\times$ with $\Upsilon_{L \mid K}^{-1}$ yields an epimorphism $(-,L|K):K^\times \ra \gal(L|K)$ which is called the \word{norm residue symbol} of the extension $L|K$. By definition we have $(a,L|K) = \varphi_{L|K}^{v_K(a)}$ for each $a \in K^\times$.

The central question is now if and how this beautiful unramified local class field theory can be extended to a $\fr{K}(G)_{\ro{ab}}^{\tn{f}}$-class field theory. One approach to such an extension is the group-cohomological Nakayama--Tate duality which applied to the discrete $\gal(L|K)$-module $K^\times$ for \textit{any} finite Galois extension $L|K$ of a finite separable extension $K|k$ yields a complete duality
\[
\tateco^0(\gal(L|K),K^\times) \times \tateco^{2}(\gal(L|K),\ZZ) \ra \tateco^2(\gal(L|K),K^\times) \oset{\ro{inv}}{\cong} \frac{1}{\lbrack L:K\rbrack} \ZZ / \ZZ \subs \QQ/\ZZ
\]
given by the cup product $\cup(\theta)(-,\gamma)$ with a fundamental class $\gamma \in \tateco^2(\gal(L|K),K^\times)$, where $\theta:K^\times \times \ZZ \ra K^\times$ is the action of $\ZZ$ on the $\ZZ$-module $K^\times$. As this is a complete duality, it induces an isomorphism
\[
K^\times/\ro{N}_{L|K} L^\times = \tateco^0(\gal(L|K),K^\times) \cong \tateco^2(\gal(L|K),K^\times)^{\vee} \cong \ro{H}_1(\gal(L|K),K^\times) \cong \gal(L|K)^{\tn{ab}}.
\]
With these isomorphisms one can indeed show that $\tateco_{\ro{Sp}(G)^{\tn{f}}}^0(\ro{GL}_1(k^{\tn{s}})_*)$ is a $\fr{K}(G)_{\ro{ab}}^{\tn{f}}$-class field theory extending the unramified local class field theory above. However, to make the Nakayama--Tate duality work one has to verify that the $\gal(L|K)$-module $K^\times$ is a class module which requires the vanishing of $\tateco^1(H,K^\times)$ and requires $\tateco^2(H,K^\times)$ to be cyclic of order $|H|$ for all subgroups $H \leq \gal(L|K)$. Such cohomological input, dealing with cohomology groups $\tateco^i$ for $i \notin \lbrace -1,0,1 \rbrace$ and in particular involving the cohomological machinery in general, is the reason why this approach carries the attribute \textit{cohomological}. As Ivan Fesenko and Sergei Vostokov describe in \cite{FesVos02_Local-fields_0}, the real disadvantage of this cohomological approach is its unexplicitness and this raises the question if there is a more straightforward way to extend the unramified local class field theory, avoiding group cohomology. \name{Neukirch} addresses this problem in his first (cohomological treatment) of class field theories \cite[page 143]{Neu69_Klassenkorpertheorie_0}:
\begin{quote}
\foreignlanguage{ngerman}{Der Satz (4.8) [über die explizite Darstellung des Normrestsymbols für unverzweigte Erweiterungen eines lokalen Körpers] wirft die Frage auf, ob man nicht den cohomologischen Kalkül und den Begriff der Klassenformationen vermeiden und auf einem viel natürlicheren Wege zum Reziprozitätzgesetz gelangen kann, indem man nämlich das Normrestsymbol einfach explizit durch die Formel $(a,L|K) = \varphi_{L\mid K}^{v_K(a)}$ definiert und alle wesentlichen Eigenschaften in direkter Weise verifiziert. Dies ist im unverzweigten Fall in der Tat möglich. Bei genauerer Betrachtung haben wir sogar nichts anderes getan, als diesen Gedanken künstlich -- über die für den vorliegenden Fall unangemessen kompliziert erscheinende Invariantenabbildung -- in eine cohomologische Form zu zwingen. Der Grund dafür liegt in dem Problem, auch die verzweigten Erweiterungen der klassenkörpertheoretischen Behandlung zugänglich zu machen. Historisch hat gerade an diesem Punkt die Cohomologie (über die Algebrentheorie) ihren Einzug in die Klassenkörpertheorie gehalten. Für die verzweigten Erweiterungen nämlich lä\ss t sich eine explizite Definition des Normrestsymbols nicht so ohne weiteres angeben, wohl aber eine Invariantenabbildung, die die hier konstruierte in kanonischer Weise auf den Bereich beliebiger normaler Erweiterungen fortsetzt.}\footnote{``The theorem (4.8) [on the explicit presentation of the norm residue symbol for unramified extensions of a local field] raises the question if it is possible to avoid the cohomological calculus and the notion of a class formation and instead get to the reciprocity law on a more natural way by defining the norm residue symbol explicitly as $(a,L|K) = \varphi_{L\mid K}^{v_K(a)}$ and then verifying all essential properties directly. This is indeed possible in the unramified case. On closer examination, we did nothing else than artificially squeezing this idea -- by means of the in the present case inadequately complicated appearing invariant map -- into a cohomological form. The reason lies within the problem to make also the ramified extensions accessible to the class field theoretic treatment. Historically, this was precisely the point where the cohomology theory (via the theory of algebras) found its way into class field theory. Namely, for the ramified extensions an explicit definition of the norm residue symbol cannot be given without further ado, however an invariant map which canonically extends the one constructed here to arbitrary normal extensions can be given.''}
\end{quote}

It looks like Neukirch was already at that time searching for a non-cohomological approach and finally, seventeen years after the appearance of \cite{Neu69_Klassenkorpertheorie_0}, he found an elementary solution. In the preface to the presentation of his discovery in  \cite{Neu86_Class-Field_0} he reflects on this solution:

\begin{quote}
My earlier presentation of the theory \cite{Neu69_Klassenkorpertheorie_0} has strengthened me in the belief that a highly elaborate mechanism, such as, for example, cohomology, might not be adequate for a number-theoretical law admitting a very direct formulation, and that the truth of such a law must be susceptible to a far more immediate insight. I was determined to write the present, new account of class field theory by the discovery that, in fact, both the local and the global reciprocity laws may be subsumed under a purely group-theoretical principle, admitting an entirely elementary description. This description makes possible a new foundation for the entire theory.
\end{quote}

Neukirch observed that for \textit{any} finite Galois extension $L|K$ of a finite separable extension $K|k$ and any element $\sigma \in \gal(L|K) = \gal(K)/\gal(L)$ there exists a representative $\widetilde{\sigma} \in \gal(K)$, called Frobenius lift of $\sigma$, such that the fixed field $\Sigma = (k^{\ro{s}})^{\widetilde{\sigma}} \cap L^{\ro{ur}} = (L^{\ro{ur}})^{\widetilde{\sigma}}$ is a finite extension of $K$ and $\Sigma L| \Sigma$ is a finite unramified Galois extension with $\widetilde{\sigma}|_{\Sigma L} = \varphi_{\Sigma L|\Sigma}$, where $L^{\ro{ur}}$ is the maximal unramified extension of $L$. This means that by moving $\sigma$ to the right position, it can be interpreted as the Frobenius automorphism of some finite unramified Galois extension. If $\Upsilon: \pi_{\fr{K}(G)_{\tn{ab}}^{\tn{f}}} \ra \tateco_{\ro{Sp}(G)^{\tn{f}}}^0(\ro{GL}_1(k^{\tn{s}})_*)$ is now any morphism, then one has the following commutative diagram:
\[
\xymatrix{
\widetilde{\sigma}|_{\Sigma L} = \varphi_{\Sigma L|\Sigma} \ar@{|->}[d] & \gal(\Sigma L|\Sigma) \ar[rr]^{\Upsilon_{\Sigma L|\Sigma}} \ar[d]_{(-)|_L} & & \Sigma^\times / \ro{N}_{\Sigma L|\Sigma} (\Sigma L)^{\times} \ar[d]^{\ro{N}_{\Sigma|K}} \\
\sigma \modd \comm{a}(\gal(L|K)) & \gal(L|K)^{\tn{ab}} \ar[rr]_{\Upsilon_{L|K}} & & K^\times/\ro{N}_{L|K} L^\times
}
\]
If this morphism should extend the unramified local class field theory, then the commutativity of this diagram immediately forces
\[
\Upsilon_{L|K}(\sigma \modd \comm{a}(\gal(L|K))) = \ro{N}_{\Sigma|K}(\pi_\Sigma) \modd \ro{N}_{L|K}L^\times
\]
for a prime element $\pi_\Sigma \in \Sigma^\times$. This observation implies that if there exists an extension of the unramified local class field theory at all, its reciprocity morphism already has to be given by this equation. What is left is to verify that this indeed defines an isomorphism which is independent of all choices. This is the only crux in Neukirch's approach because the verification of the multiplicativity of $\Upsilon_{L|K}$ defined in this way is rather technical and as Fesenko and Vostokov mention in \cite{FesVos02_Local-fields_0}:
\begin{quote}
$\lbrack ... \rbrack$ that proof does not seem to induce a lucid understanding of what is going on.
\end{quote}

Another disadvantage is that this approach does of course not provide information about higher co\-ho\-mo\-lo\-gy groups in the first place and that the calculation of the Frobenius lifts and their fixed fields is still not explicit. However, the striking fact is that Neukirch did not only give the local class field theory in this way a new foundation but that he constructed a whole abstract group-theoretical framework based on the above from which also global class field theory can be obtained. He observed that the main ingredient of the Frobenius lifting mechanism is the maximal unramified extension $k^{\ro{ur}}$ of $k$ and the corresponding epimorphism
\[
d: \gal(k) \twoheadrightarrow \gal(k^{\ro{ur}}|k) \cong \gal(\kappa^{\ro{s}}|\kappa) \cong \widehat{\ZZ},
\] 
where $\kappa$ is the residue field of $k$. With this epimorphism the ramification theory of $k$ can be formulated in a purely group-theoretical context. Neukirch thus considered an arbitrary profinite group $G$ together with an epimorphism $d:G \twoheadrightarrow \widehat{\ZZ}$ and formulated a corresponding abstract ramification theory induced by $d$. In this abstract context the whole Frobenius lifting mechanism can be shown to work. Moreover, to get a notion of prime elements, he considered a discrete $G$-module $C$ equipped with an epimorphism $v:C^{G} \ra \ZZ$ satisfying several properties and based on $v$ he defined a surjective morphism $v_H: C^H \ra \ZZ$ for every open subgroup $H$ of $G$ which generalize the normalized valuations $v_K:K^\times \ra \ZZ$. Finally, he demonstrated that under certain conditions, one gets a class field theory $\Upsilon: \pi_{\fr{K}(G)_{\tn{ab}}^{\tn{f}}} \ra \tateco_{\ro{Sp}(G)^{\tn{f}}}^0(C)$ with $\Upsilon_{(H,U)}$ defined as above using Frobenius lifts. \\

A striking fact, relying on the elementary nature of Neukirch's approach and serving as both a motivation and justification of our RIC-functorial point of view of class field theories, is that \name{Ivan Fesenko} demonstrated in \cite{Fes92_Multidimensional-local_0} that Neukirch's approach still works when $C$ is replaced by a cohomological Mackey functor. This generalization is not only of theoretical value because Fesenko could in this way instantiate a class field theory for higher local fields of positive characteristic generalizing local class field theory. In this class field theory $C$ is given by the Milnor--Par\v{s}in $\ro{K}$-groups which form a cohomological Mackey functor not necessarily having Galois descent. Although around ten years earlier \name{Kazuya Kat\={o}} has proven in \cite{Kat80_A-generalization-of-local_0} the existence of a class field theory for such fields using the non-modified Milnor $\ro{K}$-groups, this construction relies heavily on cohomological considerations and is hard to grasp for people not having Kat\={o}'s insight. The advantage of Fesenko's approach is that it is not only easier but that after the discussion in this and the preceding chapter we know exactly what we have to do to get this class field theory. However, it should be self-evident that this approach still involves a lot of work. \\

The description of the extension from discrete $G$-modules to cohomological Mackey functors was kept rather short in \cite{Fes92_Multidimensional-local_0}. Therefore we will discuss this extension in full detail. Moreover, as proposed by Neukirch in an exercise, we extend the theory to epimorphisms $d:G \twoheadrightarrow \ZZ_P$ with a set of prime numbers $P$ and a pro-$P$ group $G$. It seems that this is not discussed elsewhere in the literature. \\

As we will keep the abstract discussion in this chapter free from examples, the reader should already take a look at chapter \ref{sect:local_cft}, where all abstract notions are instantiated for a local field.

\subsection{Abstract ramification theory for compact groups} \label{sec:abs_ram_theory}

\begin{para}
In this section we will discuss the abstract ramification theory introduced by \name{Neukirch} in \cite{Neu86_Class-Field_0} as an abstract model of the ramification theory of a local field. Instead of just considering an epimorphism $d:G \twoheadrightarrow \widehat{\ZZ}$ from a profinite group as motivated in the introduction, we will define this concept for a general epimorphism $d:G \twoheadrightarrow A$ from a compact group to a separated abelian group. The reason for this generalization is not only that in the abstract discussion it makes no difference but that the Fesenko--Neukirch class field theory still works for an epimorphism $d:G \twoheadrightarrow \ZZ_P$ with a set of prime numbers $P$ and moreover \name{Fesenko}'s totally ramified $p$-class field theory in \cite{Fes95_Abelian-local_0} is based on an epimorphism $d:G \twoheadrightarrow A$ with $A$ being a free pro-$p$ group.
\end{para}

\begin{defn}
A \word{ramification theory} for a compact group $G$ is a surjective morphism $d:G \twoheadrightarrow A$ of topological groups, where $A$ is an additively written separated abelian group.\footnote{Note that this already implies that $A$ is compact by \ref{prop:mor_from_qc_in_sep_strict}.}
\end{defn}
 
\begin{ass}
For the rest of this section we fix a compact group $G$ and a ramification theory $d:G \twoheadrightarrow A$ for $G$. The discussion in this section will be relative to this theory.
\end{ass}

\begin{defn} \wordsym{$f_{H \mid K}$}
The \words{relative inertia degree}{inertia degree!relative} of a pair $K \leq H$ of closed subgroups of $G$ is defined as $f_{H \mid K} \dopgleich [d(H):d(K)]$. The \words{absolute inertia degree}{inertia degree!absolute} of a closed subgroup $H$ of $G$ is defined as $f_H \dopgleich f_{G \mid H} \dopgleich [A:d(H)]$.
\end{defn}

\begin{defn} \wordsym{$e_{H \mid K}$} \wordsym{$I_H$} \wordsym{$\Gamma_H$}
The \word{inertia subgroup} of a closed subgroup $H$ of $G$ is defined as $I_H \dopgleich H \cap \ker(d)$ and the \word{maximal unramified quotient} of $H$ is defined as $\Gamma_H \dopgleich H/I_H$. The \words{relative ramification degree}{ramification degree!relative} of a pair $K \leq H$ of closed subgroups of $G$ is defined as $e_{H \mid K} \dopgleich [I_H:I_K]$. The \words{absolute ramification degree}{ramification degree!absolute} of a closed subgroup $H$ of $G$ is defined as $e_H \dopgleich e_{G \mid H} = [\ker(d):I_H]$.
\end{defn}

\begin{para}
Since $A$ is separated and $d$ is continuous, the kernel of $d$ is closed in $G$ and therefore the inertia subgroup $I_H = H \cap \ker(d)$ of a closed subgroup $H$ of $G$ is also closed in $G$. In particular, both $I_H$ and $\Gamma_H$ are compact.
\end{para}

\begin{defn} 
An \words{inertially finite}{subgroup!inertially finite} subgroup of $G$ is a closed subgroup $H$ of $G$ such that $f_H < \infty$. An \words{unramified}{subgroup!unramified} subgroup of a closed subgroup $H$ of $G$ is a closed subgroup $K$ of $H$ such that $e_{H \mid K} = 1$. A \words{totally ramified}{subgroup!totally ramified} subgroup of a closed subgroup $H$ of $G$ is a closed subgroup $K$ of $H$ such that $f_{H \mid K} = 1$. 
\end{defn}

\begin{prop} \label{para:ram_and_inert_props}
The following holds:
\begin{enumerate}[label=(\roman*),noitemsep,nolistsep]
\item If $L \leq K \leq H$ is a triple of closed subgroups of $G$, then
\[
f_{H \mid L} = f_{H \mid K} \cdot f_{K \mid L} \quad \textnormal{and} \quad e_{H \mid L} = e_{H \mid K} \cdot e_{K \mid L}.
\]

\item If $K \leq H$ is a pair of closed subgroups of $G$, then 
\[
[H:K] = e_{H \mid K} \cdot f_{H \mid K}.
\]

\item If $H$ is an open subgroup of $G$, then $e_H < \infty$ and $f_H < \infty$.

\item If $H$ is a closed subgroup of $G$, then a closed subgroup $K$ of $H$ is an unramified subgroup of $H$ if and only if $K \geq I_H$.\footnote{This justifies that we called the quotient $\Gamma_H = H/I_H$ the maximal unramified quotient of $H$.}

\item If $H$ is a closed subgroup of $G$, then a closed subgroup $K$ of $H$ is a totally ramified subgroup of $H$ if and only if $K \cdot I_H = H$.
\end{enumerate}

\end{prop}

\begin{proof} \hfill

\begin{asparaenum}[(i)]
\item This is obvious.

\item Consider the following commutative diagram of abstract groups
\[
\xymatrix{
I_H \ar@{ >->}[rr] & & H \ar@{->>}[rr] & & H/I_H \ar[rr]^{(d|_H)'}_{\cong} & & d(H) \\
I_K \ar@{ >->}[rr] \ar@{ >->}[u] & & K \ar@{->>}[rr] \ar@{ >->}[u] & & K/I_K \ar[rr]_{(d|_K)'}^{\cong} \ar@{ >->}[u] & & d(K) \ar@{ >->}[u]
}
\]
where the vertical morphisms are the inclusions and the morphism $(d|_H)'$ respectively $(d|_K)'$ is induced by $d|_H$ respectively $d|_K$. 
If $K$ is a normal subgroup of $H$, then this diagram induces the following exact sequence
\[
I_H/I_K \lra H/K \overset{d}{\lra} d(H)/d(K) \lra 1.
\]

Since 
\[
I_H \cap K = H \cap \ker(d) \cap K = \ker(d) \cap K = I_K,
\]
the map $I_H/I_K \lra H/K$ is injective and so we have an exact sequence
\[
1 \ra I_H/I_K \lra H/K \lra d(H)/d(K) \lra 1.
\]
This sequence yields the relation
\[
[H:K] = [I_H : I_K] \cdot [d(H):d(K)] = e_{H \mid K} \cdot f_{H \mid K}.
\]
This proves the assertion in the case where $K$ is a normal subgroup of $H$. Now, let $K$ be an arbitrary closed subgroup of $H$. Let $N = \ro{NC}_H(K)$ be the normal core of $K$ in $H$ (confer \ref{para:normal_core}). Then $N \lhd H$ and $N \leq K$. Moreover, since $N = \bigcap_{h \in H} K^h$ and since $K$ is closed in $H$, it follows that $N$ is closed in $H$. Hence, we can apply the above to the pair $N \leq H$ and get
\[
\lbrack H : N \rbrack = e_{H \mid N} \cdot f_{H \mid N}.
\]
As $H \leq_{\ro{c}} G$, $N \lhd_{\ro{c}} H$ and $N \leq K$, we also have $N \lhd_{\ro{c}} K$. Thus, we can also apply the above to the pair $N \leq K$ and get
\[
\lbrack K:N \rbrack = e_{K \mid N} \cdot f_{K \mid N}.
\]

These two relations yield
\[
\lbrack H:K \rbrack = \frac{ \lbrack H:N \rbrack}{\lbrack K:N \rbrack} = \frac{ e_{H \mid N} \cdot f_{H \mid N}}{e_{K \mid N} \cdot f_{K \mid N}} = \frac{e_{H \mid N}}{e_{K \mid N}} \cdot \frac{f_{H \mid N}}{f_{K \mid N}} = e_{H \mid K} \cdot f_{H \mid K}.
\]

\item We have $\lbrack G:H \rbrack < \infty$ since $G$ is compact and by the above we have
\[
e_H \cdot f_H = e_{G \mid H} \cdot f_{G \mid H} = \lbrack G:H \rbrack < \infty.
\]
Hence, $e_H < \infty$ and $f_H < \infty$.

\item For a closed subgroup $K \leq H$ we have the following equivalences
\[
e_{H \mid K} = 1 \lLRA \lbrack I_H : I_K \rbrack = 1 \lLRA I_H = I_K \lLRA I_H = K \cap \ker(d) \lLRA I_H \leq K.
\]

Hence, $K$ is an unramified subgroup if and only if $I_H \leq K$. 

\item If $K \cdot I_H = H$, then $d(K) = d(K \cdot I_H) = d(H)$ and therefore $f_{H \mid K} = 1$. Conversely, if $f_{H \mid K} = 1$, then
\[
d(K) = d(H) \lRA d|_H(K) = d|_H(H) \lRA (d|_H)'( K \cdot I_H/I_H) = (d|_H)'(H/I_H) 
\]
\[
\lRA K \cdot I_H/I_H = H/I_H \lRA K \cdot I_H = H,
\]
where $(d|_H)':H/I_H \ra d(H)$ is the isomorphism induced by $d|_H$. \vspace{-\baselineskip}
\end{asparaenum}

\end{proof}

\begin{prop}
Let $\fr{S} \leq \ro{Grp}(G)^{\tn{r-f}}$. For $H \in \ssys{b}$ let $\ro{R}(H) \dopgleich I_H$. Then $\ro{R} = \lbrace \ro{R}(H) \mid H \in \ssys{b} \rbrace$ is an abelianization system on $\fr{S}$. We denote $\Gamma \dopgleich \Gamma^{\fr{S}} \dopgleich \pi_{\ro{R}} \in \se{Stab}^{\ro{c}}(\fr{S},\se{TAb})$.
\end{prop}

\begin{proof}
It is evident that $I_{^g \! H} = {^g \! }(I_H)$ and therefore \ref{para:ab_system}\ref{item:ab_system_conj} is satisfied. Let $H \in \ssys{b}$, $K \in \ssys{r}(H)$ and let $T = \lbrace t_1,\ldots,t_n \rbrace$ be a right transversal of $K$ in $H$. Let $h \in \ro{R}(H) = I_H$. Then there exists a permutation $f \in \ro{S}_n$ such that $t_ih = \kappa_T(t_i h) t_{f(i)}$ for all $i = 1,\ldots, n$ and so we have
\[
\ro{V}_{K,H}^{T}(h) = \prod_{i=1}^n \kappa_T(t_i h) = \prod_{i=1}^n t_ih t_{f(i)}^{-1} \in K.
\]
Hence, 
\[
d( \ro{V}_{K,H}^{T}(h) ) = d( \prod_{i=1}^n t_ih t_{f(i)}^{-1}) = \sum_{i=1}^n d(t_i h t_{f(i)}^{-1}) =  \sum_{i=1}^n d(t_i) + \sum_{i=1}^n d(h) - \sum_{i=1}^n d(t_{f(i)}) = \sum_{i=1}^n d(h) = 0
\]
and this shows that $\ro{V}_{K,H}^{T}(h) \in K \cap \ker(d) = I_K = \ro{R}(K)$. Consequently, $\ro{V}_{K,H}^{T}(I_H) \subs I_K$ and therefore \ref{para:ab_system}\ref{item:ab_system_res} is satisfied. Finally, if $H \in \ssys{b}$ and $K \in \ssys{i}(H)$, then
\[
I_K = K \cap \ker(d) \subs H \cap \ker(d) = I_H
\]
and therefore \ref{para:ab_system}\ref{item:ab_system_ind} is satisfied.
\end{proof}

\begin{defn}
An \words{inertially finite}{subgroup system!inertially finite} $G$-subgroup system is a $G$-subgroup system $\fr{S} \leq \ro{Grp}(G)^{\tn{r-f}}$ such that $f_H < \infty$ for each $H \in \ssys{b}$.
\end{defn}

\begin{prop} \label{para:inert_finite_ss} \wordsym{$\ro{Grp}(G)^{\tn{in-f}}$}
The following holds:
\begin{enumerate}[label=(\roman*),noitemsep,nolistsep]
\item Let $\ssys{b}$ be the set of all inertially finite subgroups of $G$. For $H \in \ssys{b}$ let $\ssys{i}(H) = \ssys{b}(H)$ and let $\ssys{r}(H)$ be the set of all open subgroups of $H$. Then $\fr{S}$ is the maximal inertially finite $G$-subgroup system for $G$. It is denoted by $\ro{Grp}(G)^{\tn{in-f}}$.
\item Every $G$-subgroup system $\fr{S} \leq \ro{Grp}(G)^{\tn{f}}$ is inertially finite.
\item \label{item:inert_finite_ss} If $\fr{S} \leq \ro{Grp}(G)^{\tn{in-f}}$, then $e_{H \mid K} < \infty$ for each $H \in \ssys{b}$ and $K \in \ssys{r}(H)$, and $f_{H \mid K} < \infty$ for each $H \in \ssys{b}$ and $K \in \ssys{i}(H)$.
\end{enumerate}
\end{prop}

\begin{proof}
This is easy to verify.
\end{proof}

\begin{defn} \label{para:ramification_functor}
Let $\fr{S} \leq \ro{Grp}(G)^{\tn{in-f}}$ and let $\Omega$ be an additively written abelian topological group. The following data define a RIC-functor $\boldsymbol{\Omega}_d = \boldsymbol{\Omega}_d^{\fr{S}} \in \se{Stab}^{\ro{c}}(\fr{S},\se{TAb})$:
\begin{compactitem}
\item $\boldsymbol{\Omega}_d(H) \dopgleich A$ for each $H \in \ssys{b}$.
\item $\con_{g,H}^{\boldsymbol{\Omega}_d} \dopgleich \id_{A}$ for each $H \in \ssys{b}$ and $g \in G$.
\item $\res_{K,H}^{\boldsymbol{\Omega}_d}$ is multiplication by $e_{H \mid K}$ on $A$ for each $H \in \ssys{b}$ and $K \in \ssys{r}(H)$.
\item $\ind_{H,K}^{\boldsymbol{\Omega}_d}$ is multiplication by $f_{H \mid K}$ on $A$ for each $H \in \ssys{b}$ and $K \in \ssys{i}(H)$.
\end{compactitem}
\end{defn}

\begin{proof}
The assumption that $\fr{S}$ is inertially finite ensures that both the restriction and induction morphisms are well-defined. The rest is easy to verify.
\end{proof}

\subsection{Abstract valuation theory for compact groups}

\begin{para}
In this section we will discuss the abstract valuation theory introduced by \name{Neukirch} in \cite{Neu86_Class-Field_0} as an abstract model of the unique normalized valuation on the finite separable extensions of a local field. The notion of a valuation on a RIC-functor $C:\fr{S} \ra \se{TAb}$ is relative to a ramification theory $d:G \twoheadrightarrow A$ and we will first discuss valuations whose value group $\Omega$ is not necessarily equal to $A$. Later in the Fesenko--Neukirch class field theory these two groups will be equal. To define the notion of prime elements in $C$, we additionally have to choose an element $\omega \in \Omega$. In the Fesenko--Neukirch class field theory this element has to be a topological generator of $\Omega$.
\end{para}

\begin{ass}
Throughout this section $G$ is a compact group, $d:G \twoheadrightarrow A$ is a ramification theory, $\Omega$ is an additively written abelian topological group and $\omega$ is a fixed element of $\Omega$. Moreover, $\fr{S}$ is an inertially finite $G$-subgroup system.
\end{ass}

\begin{defn} \label{para:valuation}
A \words{$d$-compatible $\Omega$-valued valuation with generator $\omega$}{valuation} on a RIC-functor $C \in \se{Fct}(\fr{S},\se{Ab})$ is a morphism $v:C \ra \boldsymbol{\Omega}_d^{\fr{S}}$ in $\se{Fct}(\fr{S},\se{Ab})$ such that $\omega \in \im(v_H)$ for each $H \in \ssys{b}$. The set of all such morphisms is denoted by $\ro{Val}_d^{\Omega,\omega}(C)$.
\end{defn}

\begin{para}
If $\Omega'$ is a subgroup of $\Omega$ containing $\omega$, then composition with the morphism $\boldsymbol{\Omega}'_d \ra \boldsymbol{\Omega}_d$ given by the canonical inclusions induces an injective map $\ro{Val}_d^{\Omega',\omega}(C) \rightarrowtail \ro{Val}_d^{\Omega,\omega}(C)$. We will make use of this map without explicitly mentioning this. We will mainly be interested in valuations $\ro{Val}_d^{\ZZ,1}(C) \subs \ro{Val}_d^{\widehat{\ZZ},1}(C)$.
\end{para}

\begin{defn}
Let $C \in \se{Fct}(\fr{S},\se{Ab})$, let $v \in \ro{Val}_d^{\Omega,\omega}(C)$ and let $H \in \ssys{b}$. An element $\pi \in C(H)$ is called a \word{prime element} (or \word{uniformizer}) with respect to $v$ if $v_H(\pi) = \omega$.
\end{defn}

\begin{para}
Note that since $\omega \in \im(v_H)$ by definition, there exists at least one prime element in $C(H)$ for each $H \in \ssys{b}$.
\end{para}

\begin{prop} \label{para:ker_v_surj_only_one_prime}
Let $C \in \se{Fct}(\fr{S},\se{Ab})$, let $v \in \ro{Val}_d^{\Omega,\omega}(C)$, let $H \in \ssys{b}$ and let $U \in \ssys{i}(H)$. If $\ind_{H,U}^{\ker(v)}:\ker(v_U) \ra \ker(v_H)$ is surjective, then all prime elements in $C(H)$ are equivalent modulo $\ind_{H,U}^C C(U)$. Hence, up to equivalence modulo $\ind_{H,U}^C C(U)$ there exists a unique prime element in $C(H)$.
\end{prop}

\begin{proof}
Let $\pi_H, \pi_H' \in C(H)$ be two prime elements. Then $v_H(\pi_H) = \omega = v_H(\pi_H')$, that is, $v_H(\pi_H' \pi_H^{-1}) = 0$ and therefore $\pi_H' \pi_H^{-1} \in \ker(v_H)$. Hence, there exists $\eps \in \ker(v_H)$ such that $\pi_H' = \eps \pi_H$. Since $\ind_{H,U}^{\ker(v)}$ is surjective by assumption, there exists $\zeta \in \ker(v_U)$ with $\ind_{H,U}^C(\zeta) = \eps$ and consequently
\[
\pi_H' = \eps \pi_H = \ind_{H,U}^C(\zeta) \pi_H \equiv \pi_H \modd \ind_{H,U}^C C(U).
\]
\end{proof}

\begin{para}
The following proposition gives a method for how to construct a valuation from a morphism $v:C(G) \ra \Omega$, assuming that $G \in \ssys{b}$. The idea behind this is the way in which the normalized valuation $v_K:K^\times \ra \ZZ$ of a finite extension $K|k$ of a local field $(k,v)$ is obtained from $v:k^\times \ra \ZZ$.
\end{para}

\begin{prop} \label{para:single_valuation_induces_family}
Suppose that $\Omega$ has trivial $\ZZ$-torsion. Moreover, suppose that $G \in \ssys{b}$ and $\ssys{i}(G) = \ssys{b}$. Let $C \in \se{Stab}^{\ro{c}}(\fr{S},\se{Ab})$. Let $v:C(G) \ra \Omega$ be a morphism such that 
\[
v(\ind_{G,H}^C C(H)) = f_H \Omega
\]
for each $H \in \ssys{b}$.\footnote{Note that $\ssys{b} = \ssys{i}(G)$ and therefore $\ind_{G,H}^C$ is defined.} For $H \in \ssys{b}$ define
\[
v_H \dopgleich \frac{1}{f_H} v \circ \ind_{G,H}^C: C(H) \ra \Omega.
\]
Then $v = \lbrace v_H \mid H \in \ssys{b} \rbrace \in \ro{Val}_d^{\Omega,\omega}(C)$.
\end{prop}

\begin{proof}
First, we have to verify that $v:C \ra \boldsymbol{\Omega}_d^{\fr{S}}$ is a morphism. Since $\Omega$ has trivial $\ZZ$-torsion, the multiplication map $1/f_H:f_H\Omega \ra \Omega$ is defined for each $H \in \ssys{b}$ (confer \ref{para:comp_ab_power_group}). If $H \in \ssys{b}$ and $g \in G$, then using the equivariance of the induction morphisms we get
\[
v_{^g \! H} \circ \con_{g,H}^C = \frac{1}{f_{^g \! H}} v \circ \ind_{G,{^g \! H}} \circ \con_{g,H}^C = \frac{1}{f_H} v \circ \con_{g,G}^C \circ \ind_{G,H}^C = \frac{1}{f_H} v \circ \ind_{G,H}^C = v_H = \con_{g,H}^{\boldsymbol{\Omega}_d} \circ v_H.
\]

If $H \in \ssys{b}$ and $K \in \ssys{r}(H)$, then using the transitivity of the induction morphisms and the fact that $C$ is cohomological we get
\[
v_K \circ \res_{K,H}^C = \frac{1}{f_K} v \circ \ind_{G,K}^C \circ \res_{K,H}^C = \frac{1}{f_K} v \circ \ind_{G,H}^C \circ \ind_{H,K}^C \circ \res_{K,H}^C = \frac{\lbrack H:K \rbrack}{f_K} v \circ \ind_{G,H}^C
\]
\[
= \frac{e_{H \mid K} f_{H \mid K}}{f_{G \mid K}} v \circ \ind_{G,H}^C = \frac{e_{H \mid K}}{f_{G \mid H}} v \circ \ind_{G,H}^C = e_{H \mid K} v_H = \res_{K,H}^{\boldsymbol{\Omega}_d} \circ v_H.
\]
Finally, if $H \in \ssys{b}$ and $K \in \ssys{i}(H)$, then
\[
v_H \circ \ind_{H,K}^C = \frac{1}{f_H} v \circ \ind_{G,H}^C \circ \ind_{H,K}^C = \frac{1}{f_{G \mid H}} v \circ \ind_{G,K}^C = \frac{f_{H \mid K}}{f_{G \mid K}} v \circ \ind_{G,K}^C = f_{H \mid K} v_K = \ind_{H,K}^{\boldsymbol{\Omega}_d} \circ v_K.
\]
\end{proof}

\subsection{Class field theories for unramified extensions}

\begin{para}
In this section we will generalize the class field theory for unramified extensions of a local field described in the introduction to our abstract setting. To get a notion of Frobenius elements, we now have to consider a ramification theory $d:G \twoheadrightarrow \Omega$ with a $\ZZ$-torsion-free procyclic group $\Omega$ and a fixed topological generator $\omega$ of $\Omega$. It follows from \ref{prop:char_torsion_free_procyclic} that already $\Omega \cong \ZZ_P = \prod_{p \in P} \ZZ_p$ for a set of prime numbers $p$. As the fundamental principle of this class field theory is the correspondence between Frobenius elements and prime elements, we then have to consider valuations $v \in \ro{Val}_d^{\Omega,\omega}(C)$. The main example for these choices is $\Omega = \widehat{\ZZ}$, $\omega = 1$ and $v \in \ro{Val}_d^{\ZZ,1}(C) \subs \ro{Val}_d^{\widehat{\ZZ},1}(C)$.
\end{para}

\begin{ass} \label{para:ram_theo_for_unram_cft}
Throughout this section we fix a compact group $G$ and a ramification theory $d:G \twoheadrightarrow \Omega$, where $\Omega$ is an additively written $\ZZ$-torsion-free procyclic group with a fixed topological generator $\omega$.
\end{ass}

\begin{prop} \label{para:morphism_d_H} \wordsym{$d_H$} \wordsym{$d_H'$}
 Let $H$ be an inertially finite subgroup of $G$. The following holds:
\begin{enumerate}[label=(\roman*),noitemsep,nolistsep]
\item $d_H \dopgleich \frac{1}{f_H} d|_H:H \ra \Omega$ is surjective.
\item $d_H$ induces an isomorphism of compact groups $d_H':H/I_H \ra \Omega$.
\item \label{item:morphism_d_H_ind_compat} If $K \leq H$ is a pair of inertially finite subgroups of $G$, then the diagram
\[
\xymatrix{
H \ar[rr]^{d_H} & & \Omega \\
K \ar@{ >->}[u] \ar[rr]_{d_K} & & \Omega \ar[u]_-{f_{H \mid K}}
}
\]
commutes, where the right vertical morphism denotes multiplication by $f_{H \mid K}$
\end{enumerate}
\end{prop}

\begin{proof} \hfill

\begin{asparaenum}[(i)]

\item Since $\lbrack \Omega:d(H) \rbrack = f_H$, it follows from \ref{thm:unique_subgroups_of_procyclic} that $d(H) = f_H \Omega$. The assertion now follows.

\item Let $x \in G$. Since $f_H\Omega \ra A$, $x \mapsto \frac{1}{f_H}x$ is an isomorphism, we have $x \in \ker(d_H)$ if and only if $x \in \ker(d|_H) = H \cap \ker(d) = I_H$ and therefore $\ker(d_H) = I_H$. Hence, $d_H$ induces an isomorphism $d_H':H/I_H \ra \Omega$ by the closed map lemma.

\item By \ref{para:ram_and_inert_props} we have $f_{G \mid K} = f_{G \mid H} \cdot f_{H \mid K}$. Since $K$ is an inertially finite subgroup of $G$, the inertia degree $f_{G | K}$ is finite and so are $f_{G \mid H}$ and $f_{H \mid K}$. Hence, for $x \in K$ we have
\[
f_{H \mid K} d_K(x) = f_{H \mid K} \cdot \frac{1}{f_K} d|_K(x) = \frac{f_{H \mid K}}{f_{G \mid K}} d(x) = \frac{1}{f_{G \mid H}} d(x) = \frac{1}{f_H} d(x) = d_H(x).
\]

\end{asparaenum}
\end{proof}

\begin{prop}
Let $\fr{S} \leq \ro{Grp}(G)^{\tn{in-f}}$. Then the family $\lbrace d_H' \mid H \in \ssys{b} \rbrace$ defines an isomorphism $d':\Gamma^{\fr{S}} \ra \boldsymbol{\Omega}_d^{\fr{S}}$ in $\se{Fct}(\fr{S},\se{TAb})$.
\end{prop}

\begin{proof}
For each $H \in \ssys{b}$ the map $d_H':\Gamma(H) = H/I_H \ra \Omega = \boldsymbol{\Omega}_d(H)$ is an isomorphism of topological groups by \ref{para:morphism_d_H} and therefore it remains to show that $d'$ is a morphism of functors. Let $H \in \ssys{b}$ and $g \in G$. Then for $h \in H$ we have
\[
d_{^g \! H}' \circ \con_{g,H}^{\Gamma}(h \modd I_H) = d_{^g \! H}'(ghg^{-1} \modd I_{^g \! H}) = \frac{1}{f_{^g \! H}} d(ghg^{-1})
\]
\[
 = \frac{1}{f_H} d(h) = d_H(h) = d_H'(h \modd I_H) = \con_{g,H}^{\boldsymbol{\Omega}_d} \circ d_H'(h \modd I_H)
\]
and this shows that $d'$ commutes with the conjugation morphisms. Let $H \in \ssys{b}$ and $K \in \ssys{r}(H)$. Let $T = \lbrace t_1,\ldots,t_n \rbrace$ be a right transversal of $K$ in $H$ and let $h \in H$. Then there exists a permutation $f \in \ro{S}_n$ such that $t_ih = \kappa_T(t_ih) t_{f(i)}$ for all $i \in \lbrace 1,\ldots, n \rbrace$. Hence, 
\[
d_K' \circ \res_{K,H}^{\Gamma}(h \modd I_H) = d_K' \circ \ver_{K,H}^{I_K,I_H}(h \modd I_H) = d_K'( \prod_{i=1}^n \kappa_T(t_ih) \modd I_K)
\]
\[
= \frac{1}{f_K} d( \prod_{i=1}^n \kappa_T(t_ih) ) = \frac{1}{f_K} d( \prod_{i=1}^n t_iht_{f(i)}^{-1}) = \frac{1}{f_K} \sum_{i=1}^n d(h) = \frac{\lbrack H:K \rbrack}{f_{G \mid K}} d(h)
\]
\[
= \frac{e_{H \mid K} f_{H \mid K}}{f_{G \mid H} f_{H \mid K}} d(h) = \frac{e_{H \mid K}}{f_{G \mid H}} d(h) = e_{H \mid K} d_H'(h \modd I_H) = \res_{K,H}^{\boldsymbol{\Omega}_d} \circ d_H'(h \modd I_H)
\]
and this shows that $d'$ commutes with the restriction morphisms. Finally, let $H \in \ssys{b}$ and $K \in \ssys{i}(H)$. If $x \in K$, then it follows from \ref{para:morphism_d_H}\ref{item:morphism_d_H_ind_compat} that
\[
d_H' \circ \ind_{H,K}^\Gamma(x \modd I_K) = d_H'(x \modd I_H) = d_H(k) 
\]
\[
= f_{H \mid K} d_K(x) = f_{H \mid K} d_K'(x \modd I_K) = \ind_{H,K}^{\boldsymbol{\Omega}_d} \circ d_K'(x \modd I_K)
\]
and this shows that $d'$ commutes with the induction morphisms.
\end{proof}

\begin{defn} \wordsym{$\varphi_H$} \wordsym{$\varphi_{H \mid K}$}
Let $H$ be an inertially finite subgroup of $G$ and let $K$ be an unramified subgroup of $H$. The \words{relative Frobenius element}{Frobenius element!relative} of the pair $K \leq H$ is defined as the element 
\[
\varphi_{H \mid K} \dopgleich q_{H \mid K} \circ (d_H')^{-1}(\omega) \in H/K,
\]
where $q_{H \mid K}:H/I_H \ra H/K$ is induced by the quotient morphism $H \ra H/K$.\footnote{Note that $I_H \leq K$ since $K$ is an unramified subgroup of $H$.} The \words{absolute Frobenius element}{Frobenius element!absolute} of $H$ is defined as the element 
\[
\varphi_H \dopgleich \varphi_{H \mid I_H} = (d_H')^{-1}(\omega) \in H/I_H.
\]
\end{defn}

\begin{prop} \label{para:frob_generates_unram_quot}
Let $H$ be an inertially finite subgroup of $G$ and let $K$ be an unramified normal subgroup of $H$. The following holds:
\begin{enumerate}[label=(\roman*),noitemsep,nolistsep]
\item $\varphi_{H \mid K}$ is a topological generator of $H/K$. In particular, $\varphi_H$ is a topological generator of $\Gamma_H$.
\item If $K$ is open in $H$, then $H/K$ is a finite cyclic group which is abstractly generated by $\varphi_{H \mid K}$.
\end{enumerate}
\end{prop}

\begin{proof} \hfill

\begin{asparaenum}[(i)]
\item Since $d_H':H/I_H \ra \Omega$ is an isomorphism and since $\omega$ topologically generates $\Omega$, it follows that $\varphi_H = (d_H')^{-1}(\omega)$ topologically generates $\Gamma_H = H/I_H$. As $q_{H \mid K}$ is a morphism of compact groups, it is closed and therefore $\varphi_{H \mid K} = q_{H \mid K}(\varphi_H)$ topologically generates $H/K$.  

\item This is obvious as $H/K$ is discrete in this case. \vspace{-\baselineskip}
\end{asparaenum}
\end{proof}

\begin{defn} \wordsym{$\ca{E}^{\ro{ur}}(H)$}
For a closed subgroup $H$ of $G$ the set of all unramified open normal subgroups of $H$ is denoted by $\ca{E}^{\tn{ur}}(H)$. 
\end{defn}

\begin{prop} \label{para:unramified_lattice}
For a closed subgroup $H$ of $G$ the set $\ca{E}^{\ro{ur}}(H)$ is a filter on $(\ca{E}^{\tn{t}}(H),\subs)$. In particular, $\ca{E}^{\ro{ur}}(H)$ is a lattice.
\end{prop}

\begin{proof}
This is easy to verify.
\end{proof}

\begin{para}
It is obvious that the family $\ca{E}^{\ro{ur}} = \lbrace \ca{E}^{\tn{ur}}(H) \mid H \in \ro{Grp}(G)^{\tn{in-f}}_{\ro{b}} \rbrace$ is an extension of $\ro{Grp}(G)^{\tn{in-f}}$. We define $\ro{Sp}(G)^{\tn{ur}} \dopgleich \ro{Sp}(\ro{Grp}(G)^{\tn{in-f}},\ca{E}^{\tn{ur}}) \leq \ro{Sp}(G)^{\tn{r-f}}$ and $\fr{K}(G)^{\ro{ur}}_{\ro{ab}} \dopgleich \fr{K}(\ro{Sp}(G)^{\tn{ur}})_{\tn{ab}}$. We define additionally $\ro{Sp}(G)^{\tn{ur,f}} \dopgleich \ro{Sp}(\ro{Grp}(G)^{\tn{f}},\ca{E}^{\tn{ur}}) \leq \ro{Sp}(G)^{\tn{ur}}$ and $\fr{K}(G)_{\ro{ab}}^{\tn{ur,f}} \dopgleich \fr{K}(\ro{Sp}(G)^{\tn{ur,f}})_{\tn{ab}} \leq \fr{K}(G)_{\ro{ab}}^{\tn{ur}}$.
\end{para}

\begin{ass}
For the rest of this section we fix a $G$-spectrum $\fr{E} \leq \ro{Sp}(G)^{\tn{ur}}$ and set $\fr{K} \dopgleich \fr{K}(\fr{E})_{\tn{ab}}$.
\end{ass}

\begin{prop} \label{para:frob_element_compat}
The following holds:
\begin{compactenum}[(i)]
\item If $(H,U) \in \esys{b}$ and $g \in G$, then $\con_{g,(H,U)}^{\pi_{\fr{K}}}( \varphi_{H \mid U}) = \varphi_{^g \! H, ^g \! U}$.
\item If $(H,U) \in \esys{b}$ and $(K,U) \in \esys{r}(H,U)$, then $\res_{(K,U),(H,U)}^{\pi_{\fr{K}}}( \varphi_{H \mid U} ) = \varphi_{K \mid U}^{e_{H \mid K}}$.
\item If $(H,U) \in \esys{b}$ and $(K,V) \in \esys{i}(H,U)$, then $\ind_{(H,U),(K,V)}^{\pi_{\fr{K}}}( \varphi_{K \mid V} ) = \varphi_{H \mid U}^{f_{H \mid K}}$.
\end{compactenum}
\end{prop}

\begin{proof} This is straightforward using the isomorphism $d':\Gamma^{\fr{E}^\flat} \ra \boldsymbol{\Omega}_d^{\fr{E}^\flat}$.
\end{proof}

\begin{defn} \label{para:unram_fs_datum}
An \words{unramified Fesenko--Neukirch datum}{Fesenko--Neukirch datum!unramified} on $\fr{E}$ is a pair $(C,v)$ consisting of a Mackey functor $C \in \se{Mack}(\fr{M},\se{Ab})$ defined on an inertially finite Mackey cover $\fr{M}$ of $\fr{E}$ and a valuation $v \in \ro{Val}_d^{\Omega,\omega}(C)$ satisfying the following conditions for each $(H,U) \in \esys{b}$:
\begin{enumerate}[label=(\roman*),noitemsep,nolistsep]
\item \label{item:unram_fs_datum_gen} The quotient $\im(v_H)/\lbrack H:U \rbrack \im(v_U)$ is of order $\lbrack H:U \rbrack$ and is abstractly generated by the image of $\omega$.\footnote{Note that $\lbrack H:U \rbrack < \infty$ since $U$ is open in $H$. Moreover, note that since $v$ is a morphism and $U \in \ro{Ext}(\fr{E},H) \subs \fr{M}_{\ro{i}}(H)$, we have $v_H \circ \ind_{H,U}^C = \ind_{H,U}^{\boldsymbol{\Omega}_d} \circ v_U = f_{H \mid U} \circ v_U$ and therefore $\im(v_H) \sups f_{H \mid U} \im(v_U) = e_{H \mid U} f_{H \mid U} \im(v_U) = \lbrack H:U \rbrack \im(v_U)$. Also note that $\omega \in \im(v_H)$.}
\item The morphism $\ind_{H,U}^{\ker(v)}:\ker(v_U) \ra \ker(v_H)$ is surjective.
\item \label{item:unram_fs_datum_gen_h0} $|\widehat{\ro{H}}^0(C)(H,U)| \leq \lbrack H:U \rbrack$.
\end{enumerate}

The set of all such pairs $(C,v)$ is denoted by $\ro{urFND}_d^{\omega}(\fr{E})$.
\end{defn}

\begin{conv}
In the following we will define a morphism $\Upsilon$ depending on the choice of $(C,v) \in \ro{urFND}_d^{\omega}(\fr{E})$ (and of course also on $d$ and $\fr{E})$. To simplify notations we will not include a reference to this choice.
\end{conv}

\begin{thm} \label{para:unramified_cft}
Let $(C,v) \in \ro{urFND}_d^{\omega}(\fr{E})$. For each $H \in \esys{b}^\flat$ let $\pi_H \in C(H)$ be a prime element (with respect to $v$). Then for $(H,U) \in \esys{b}$ the assignment 
\[
\begin{array}{rcl}
\Upsilon_{(H, U)}: \pi_{\fr{K}}(H,U) = H/U & \lra & C(H)/\ind_{H,U}^C C(U) = \widehat{\ro{H}}^0(C)(H,U) \\
\varphi_{H \mid U}^k & \longmapsto & \pi_H^k \modd \ind_{H,U}^C C(U)
\end{array}
\]
is an $\se{Ab}$-morphism which is independent of the choice of the prime element. The family $\Upsilon = \lbrace \Upsilon_{(H,U)} \mid (H,U) \in \esys{b} \rbrace$ defines a canonical isomorphism
\[
\Upsilon: \pi_{\fr{K}} \overset{\cong}{\lra} \widehat{\ro{H}}_{\fr{E}}^0(C)
\]
in $\se{Fct}(\fr{E},\se{Ab})$. This isomorphism is called the \word{Fesenko--Neukirch reciprocity morphism} for $(C,v)$ on $\fr{E}$.
\end{thm}

\begin{proof}
Let $(H,U) \in \esys{b}$. Then the relative Frobenius element $\varphi_{H \mid U}$ abstractly generates $H/U$ by \ref{para:frob_generates_unram_quot}. As $C(H)^{\lbrack H:U \rbrack} \subs \ind_{H,U}^C C(U)$ by \ref{para:unram_fs_datum}\ref{item:unram_fs_datum_gen_h0}, it follows that $\Upsilon_{(H,U)}$ is a well-defined $\se{Ab}$-morphism. An application of \ref{para:ker_v_surj_only_one_prime} shows that $\Upsilon_{(H, U)}$ is independent of the choice of the prime element. Now we prove that $\Upsilon_{H \mid U}$ is an isomorphism. Let $n \dopgleich \lbrack H:U \rbrack = e_{H \mid U} f_{H \mid U} = f_{H \mid U}$. Then
\[
n \im(v_U) = \lbrack H:U \rbrack \im(v_U) = f_{H \mid U} \im(v_U).
\]
Since $v_H \circ \ind_{H,U}^C = \ind_{H,U}^{\boldsymbol{\Omega}_d} \circ v_U = f_{H \mid U} \circ v_U = n v_U$, it follows that $v_H( \im(\ind_{H,U}^C)) \subs n \im(v_U)$. Hence, $v_H$ induces a morphism $v_H': C(H)/\ind_{H,U}^C C(U) \ra \im(v_H)/n \im(v_U)$. This morphism is obviously surjective and to prove that it is also injective, let $x \in C(H)$ such that $v_H'(x \modd \ind_{H,U}^C C(U)) = 0$. This implies $v_H(x) \in n \im(v_U)$ and so there exists $y \in C(U)$ such that $v_H(x) = n v_U(y)$. Using the fact that $v$ is a morphism we conclude that
\[
v_H(x) = f_{H \mid U} v_U(y) = \ind_{H,U}^{\boldsymbol{\Omega}_d} \circ v_U(y) = v_H \circ \ind_{H,U}^C(y)
\]
and accordingly there exists $\eps \in \ker(v_H)$ such that $x = \eps \cdot \ind_{H,U}^C(y)$. Since $\ind_{H,U}^{\ker(v)}$ is surjective, there exists $\zeta \in \ker(v_U)$ such that $\eps = \ind_{H,U}^C(\zeta)$. Hence, $x = \ind_{H,U}^C(\zeta) \cdot \ind_{H,U}^C(y) = \ind_{H,U}^C(\zeta y)$ and therefore $x \equiv 1 \modd \ind_{H,U}^C C(U)$.

This shows that $v_H'$ is an isomorphism. According to \ref{para:unram_fs_datum}\ref{item:unram_fs_datum_gen} the quotient $\im(v_H)/n \im(v_U)$ is of order $n$ and is abstractly generated by $\omega \modd n \im(v_U)$. Since $v_H(\pi_H) = \omega$, we have 
\[
v_H'(\pi_H \modd \ind_{H,U}^C) = \omega \modd n \im(v_U)
\]
and so it follows that $C(H)/\ind_{H,U}^C C(U)$ is also of order $n$ and is generated by $\pi_H \modd \ind_{H,U}^C C(U)$. Finally, as $H/U$ is of order $n$ and as this quotient is generated by $\varphi_{H \mid U}$ and moreover
\[
\Upsilon_{(H,U)}(\varphi_{H \mid U}) = \pi_H \modd \ind_{H,U}^C C(U),
\]
we conclude that $\Upsilon_{(H,U)}$ is an isomorphism. The compatibility relations of $\Upsilon$ are straightforward using \ref{para:frob_element_compat}. 
\end{proof}

\begin{cor}
Suppose that $\fr{K}$ is L-coherent and let $(C,v) \in \ro{urFND}_d^{\omega}(\fr{E})$. Then $\Upsilon: \pi_{\fr{K}} \ra \tateco_{\fr{E}}^0(C)$ is a $\fr{K}$-class field theory. This holds in particular for both $\fr{K} = \fr{K}(G)_{\tn{ab}}^{\tn{ur}}$ and $\fr{K} = \fr{K}(G)_{\tn{ab}}^{\tn{ur-f}}$.
\end{cor}

\begin{proof}
This is an application of \ref{cor:l_coh_iso_cft}.
\end{proof}

\subsection{Class field theories for all coabelian extensions}

\begin{para}
As described in the introduction, it is the main task to extend the Fesenko--Neukirch reciprocity morphism $\Upsilon$ on $\ro{Sp}(G)^{\tn{ur,f}}$ to $\ro{Sp}(G)^{\tn{f}}$ so that we get a $\fr{K}(G)_{\tn{ab}}^{\tn{f}}$-class field theory. The Fesenko--Neukirch class field theory, which we will discuss in this section in full detail and in its most general form, gives a (and if the spectrum has enough groups, like $\ro{Sp}(G)^{\tn{f}}$, the unique) solution to this problem using the machinery of Frobenius lifts. To make this work, we have to restrict to ramification theories $d:G \twoheadrightarrow \ZZ_P$ with $P$ a set of prime numbers and $G$ a pro-$P$ group. \name{Neukirch} originally just considered ramification theories $d:G \twoheadrightarrow \widehat{\ZZ}$ but proposed this extension in an exercise in \cite{Neu99_Algebraic-Number_0}. The theory obtained for $P$ being a proper subset of the set of prime numbers is significantly different from the $\widehat{\ZZ}$-theory. Although these theories also are extensions of the class field theory for unramified extensions, the Frobenius lifts can in general not be considered as the absolute Frobenius elements of some extension so that these theories cannot be properly motivated in this way. As far as the author understands this theory, the primary motivation for the $\ZZ_P$-theory is merely the fact that the proofs for the $\widehat{\ZZ}$-theory can be modified to also work in the $\ZZ_P$-situation. All this seems to be not discussed in the existing literature.
\end{para}

\begin{ass}
Throughout this section we fix a set of prime numbers $P$, a pro-$P$ group $G$ and a ramification theory $d:G \twoheadrightarrow \Omega$, where $\Omega$ is isomorphic to $\ZZ_P \dopgleich \prod_{p \in P} \ZZ_p$. Moreover, we fix a topological generator $\omega$ of $\Omega$. We set $\Omega^+ \dopgleich \lbrace n \omega \mid n \in \NN_{>0} \rbrace$ and for $x \in \Omega^+$ we denote the unique $n \in \NN_{>0}$ with $x = n \omega$ by $\ro{mult}(x)$. Moreover for an inertially finite subgroup $H$ of $G$ and $h \in H$ with $d_H(h) \in \Omega^+$, we set $P(d_H(h)) \dopgleich P(\ro{mult}(d_H(h)))$ and $P'(d_H(h)) \dopgleich P'(\ro{mult}(d_H(h)))$.\footnote{$P(n)$ respectively $P'(n)$ denote the $P$-part respectively the $P'$-part of $n$. Confer \ref{para:p_pprime_parts}.}
\end{ass}

\begin{para}
The following definition of Frobenius groups axiomatizes the fixed field of a Frobenius lift as motivated in the introduction to this chapter. We will see that these groups are already uniquely determined.
\end{para}

\begin{defn} \label{para:frob_group_def}
Let $H$ be an inertially finite subgroup of $H$, let $U$ be an open subgroup of $H$ and let $h \in H$. A \word{Frobenius group} for $h$ relative to $U$ is an open subgroup $\Sigma = \Sigma_{h,H|U}$ of $H$ satisfying the following conditions:
\begin{enumerate}[label=(\roman*),noitemsep,nolistsep]
\item $h \in \Sigma$.
\item \label{item:frob_group_def_inert_index} $d_H(h) \in \Omega^+$ and $f_{H \mid \Sigma} = P(d_H(h))$.
\item $I_\Sigma = I_U$.
\end{enumerate}
\end{defn}

\begin{thm} \label{para:frob_group_pres}
Let $H$ be an inertially finite subgroup of $G$, let $U$ be an open subgroup of $H$ and let $h \in H$ with $d_H(h) \in \Omega^+$. Then there exists a unique Frobenius group $\Sigma_{h,H|U}$ for $h$ relative to $U$ which is explicitly given by $\Sigma_{h,H|U} = \langle h \rangle_{\ro{c}} \cdot I_U$.
\end{thm}

\begin{proof}
Let $\Sigma = \Sigma_{h, H|U}$. We first verify that $\Sigma$ is a Frobenius group relative to $U$. The group $\Sigma$ is closed as a product of closed subgroups of a compact group. Let $n = \ro{mult}(d_H(h))$. According to \ref{para:torsion_free_procyclic_index} we have
\begin{align*}
f_{H \mid \Sigma} &= \lbrack d(H) : d(\Sigma) \rbrack = \lbrack d(H) : d( \langle h \rangle_{\ro{c}} \cdot I_U) \rbrack = \lbrack d(H) : d(\langle h \rangle)_{\ro{c}} \rbrack \\
&= \lbrack f_Hd_H(H) : f_Hd_H(\langle h \rangle_{\ro{c}}) \rbrack = \lbrack f_H \Omega : f_H \langle d_H(h) \rangle_{\ro{c}} \rbrack \\
&= \lbrack f_H\Omega: f_H \langle n \omega \rangle_{\ro{c}} \rbrack = \lbrack f_H\Omega: f_H n \Omega \rbrack \\
& = \lbrack \Omega: n \Omega \rbrack = P(n) = P(d_H(h)).
\end{align*}

Since $\Sigma \leq H$, we have $I_{\Sigma} \leq I_H$. Moreover, as $I_U \leq \Sigma$, we have $I_U = I_U \cap \ker(d) \leq \Sigma \cap \ker(d) = I_{\Sigma}$. Hence, $I_U \leq I_{\Sigma} \leq I_H$. As $U$ is an open subgroup of $H$, we have $e_{H \mid U} \cdot f_{H \mid U} = \lbrack H : U \rbrack < \infty$ and so in particular $e_{H \mid U} < \infty$. The combination of the above observations yields
\[
e_{H \mid \Sigma} = \lbrack I_H : I_{\Sigma} \rbrack \leq \lbrack I_H : I_U \rbrack = e_{H \mid U} < \infty.
\]

Now, we get
\[
\lbrack H:\Sigma \rbrack = e_{H \mid \Sigma} \cdot f_{H \mid \Sigma} \leq e_{H \mid U} \cdot \ro{mult}(d_H(h)) < \infty
\]
and so $\Sigma$ is a closed subgroup of $H$ of finite index, that is, $\Sigma$ is open in $H$.

Now, we verify that $I_U = I_\Sigma$. The group $\Sigma$ is an open subgroup of the pro-$P$ group $H$ and is therefore also a pro-$P$ group. Moreover, as $I_U$ is a closed subgroup of $H$ and $I_U \leq \Sigma$, we conclude that $I_U$ is a closed subgroup of $\Sigma$ and so the quotient $\Gamma \dopgleich \Sigma/I_U$ is a pro-$P$ group. Let $q:\Sigma \ra \Sigma/I_U$ be the quotient morphism. By the closed map lemma $q$ is a closed map and so we get
\[
\Gamma = q(\Sigma) = q( \langle h \rangle_{\ro{c}} \cdot I_U) = q(\langle h \rangle_{\ro{c}}) = \langle q(h) \rangle_{\ro{c}}.
\]
In particular, $\Gamma$ is a profinite group which is topologically generated by one element and so it follows from \ref{para:procyclic_char} that $\Gamma$ is procyclic. 
 
As $I_U \leq I_\Sigma$, the quotient morphism $\Sigma \ra \Sigma/I_\Sigma$ induces a morphism $q':\Gamma = \Sigma/I_U \ra \Sigma/I_\Sigma$. Let $\psi \dopgleich d_\Sigma' \circ q':\Gamma \ra \Omega$. Since $q'$ is surjective and $d_\Sigma'$ is an isomorphism, the morphism $\psi$ is  surjective. For $n \in \NN_{>0}$ let $\psi_n$ be the morphism determined by the diagram
\[
\xymatrix{
\Gamma \ar@{->>}[rr]^\psi \ar@{->>}[d]_{q_n} & & \Omega \ar@{->>}[d] \\
\Gamma/\Gamma^n \ar[rr]_{\psi_n} & & \Omega/\psi(\Gamma^n)
}
\]
where the vertical morphisms are the quotient morphisms. Let $\Delta(\Omega)$ be the set of supernatural divisors of $\# \Omega$. Then according to \ref{thm:unique_subgroups_of_procyclic} the open subgroups of $\Omega$ are precisely the groups $n \Omega$ for $n \in \Delta(\Omega) \cap \NN$. Hence, as $\Omega$ is profinite and as $\psi(\Gamma^n) = n \psi(\Gamma) = n \Omega$, the set $\ca{U} \dopgleich \lbrace \psi(\Gamma^n) \mid n \in \Delta(\Omega) \cap \NN \rbrace$ is a neighborhood basis of $0 \in \Omega$. By \ref{thm:approximations} this implies that $\Omega \cong \ro{lim} \ \Omega/\ca{U}$ and that $\psi$ is the unique morphism induced by the family $\lbrace \psi_n \circ q_n \mid n \in \NN_{>0} \rbrace$. As $\psi$ is surjective, each $\psi_n$ is also surjective. Moreover, it follows from  \ref{thm:unique_subgroups_of_procyclic} that
\[
\lbrack \Gamma : \Gamma^n \rbrack \leq n = \lbrack \Omega : n \Omega \rbrack = \lbrack \Omega: \psi(\Gamma^n) \rbrack
\]
for each $n \in \Delta(\Omega) \cap \NN$. Hence, the surjectivity of $\psi_n$ implies the injectivity of $\psi_n$. Now, let $\Delta(\Gamma)$ be the set of supernatural divisors of $\# \Gamma$. As $\Gamma$ is a pro-$P$ group and as $\Omega \cong \ZZ_P$, we have $\Delta(\Gamma) \subs \Delta(\Omega)$ and so, according to \ref{thm:unique_subgroups_of_procyclic}, the groups $\Gamma^n$ for $n \in \Delta(\Omega) \cap \NN$ run through all open subgroups of $\Gamma$. In particular, as $\Gamma$ is separated, we get
\[
\ker(\psi) \subs \bigcap_{n \in \Delta(\Omega) \cap \NN} \ker(\psi_n \circ q_n) = \bigcap_{n \in \Delta(\Omega) \cap \NN} \ker(q_n) = \bigcap_{n \in \Delta(\Omega) \cap \NN} \Gamma^n = 1.
\]
This shows that $\psi = d_\Sigma' \circ q'$ is an isomorphism and as $d_{\Sigma}'$ is an isomorphism, we conclude that $q'$ is an isomorphism implying that $I_U = I_\Sigma$. 

Hence, $\Sigma$ is a Frobenius group for $h$ relative to $U$. \\

It remains to show that $\Sigma$ is unique, so let $\Sigma'$ be another Frobenius group for $h$ relative to $U$. Since $h \in \Sigma'$ and $I_U = I_{\Sigma'} \leq \Sigma'$, it follows that $\Sigma \leq \Sigma'$. As $I_{\Sigma} = I_U = I_{\Sigma'}$, we have $e_{H \mid \Sigma} = e_{H \mid U} = e_{H \mid \Sigma'}$. If we assume that $\Sigma < \Sigma'$, then
\[
e_{H \mid U} \cdot f_{H \mid \Sigma} = e_{H \mid \Sigma} \cdot f_{H \mid \Sigma} = \lbrack H:\Sigma \rbrack > \lbrack H: \Sigma' \rbrack = e_{H \mid \Sigma'} \cdot f_{H \mid \Sigma'} = e_{H \mid U} \cdot f_{H \mid \Sigma'}
\]
and consequently 
\[
P(d_H(h)) = f_{H \mid \Sigma} > f_{H \mid \Sigma'} 
\]
but this is a contradiction to \ref{para:frob_group_def}\ref{item:frob_group_def_inert_index}. Hence, $\Sigma = \Sigma'$.
\end{proof}

\begin{thm} \label{para:frob_lifts_profinite}
Let $H$ be an inertially finite subgroup of $G$. The following holds:
\begin{enumerate}[label=(\roman*),noitemsep,nolistsep]
\item The set $\ro{Frob}_H \dopgleich \lbrace h \in H \mid d_H(h) \in \Omega^+ \rbrace = d_H^{-1}(\Omega^+)$ is a dense sub-semigroup of $H$. 
\item Let $U$ be an open subgroup of $H$ and let $q:H \ra H/U$ be the quotient morphism. Then the semigroup morphism $q|_{\ro{Frob}_H}:\ro{Frob}_H \ra H/U$ is surjective. For $\ol{h} \in H/U$ the elements in $q^{-1}(\ol{h}) \cap \ro{Frob}_H$ are called the \words{Frobenius lifts}{Frobenius lift} of $\ol{h}$.
\end{enumerate}
\end{thm}

\begin{proof} \hfill

\begin{asparaenum}[(i)]
\item\hspace{-5pt}\footnote{The proof given here is partially based on the proof of \cite[proposition 1.2]{Neu94_Micro-primes_0}.} Let $h_1,h_2 \in \ro{Frob}_H$. Then $d(h_i) = n_i \omega$ for some $n_i \in \NN_{>0}$ and
\[
d(h_1h_2) = d(h_1) + d(h_2) = n_1 \omega + n_2 \omega = (n_1 + n_2) \omega \in \Omega^+.
\]
Hence, $h_1h_2 \in \ro{Frob}_H$ and therefore $\ro{Frob}_H$ is a sub-semigroup of $H$.

As $\Omega$ is profinite, the set $\ca{U}$ of all open normal subgroups of $\Omega$ is a neighborhood basis of $0 \in \Omega$. Let $U \in \ca{U}$. Since $d_H$ is a strict surjective morphism, it is an open map. Hence, $d_H(U)$ is an open subgroup of $\Omega$ and as $\langle \omega \rangle$ is dense in $\Omega$, we conclude that $d_H(U) \cap \langle \omega \rangle \neq \emptyset$. If $d_H(U) \cap \langle \omega \rangle = \lbrace 0 \rbrace$, then the image of $\omega$ in $\Omega / d_H(U)$ would have infinite order which is a contradiction since $d_H(U)$ is an open subgroup of $\Omega$ and is therefore of finite index. Thus, $d_H(U) \cap \Omega^+ \neq \emptyset$ and this shows that $U \cap \ro{Frob}_H \neq \emptyset$. 

Now, let $h \in H$. As $hU$ is open in $H$, it follows that $d_H(hU)$ is open in $\Omega$ and so $d_H(hU) \cap \langle \omega \rangle \neq \emptyset$. Let $m \in \ZZ$ such that $m \omega \in d_H(hU)$ and let $u' \in U$ such that $d_H(hu') = m \omega$. Since $d_H(U) \cap \Omega^+ \neq \emptyset$, there exists $n \in \NN_{>0}$ and $u \in U$ such that $d_H(u) = n\omega$. Let $k \in \NN$ such that $m + kn > 0$. Then
\[
\Omega^+ \ni (m + kn) \omega = m \omega + kn \omega = d_H(hu') + d_H(u^k) = d_H(hu'u^k) \in d_H(hU).
\]
Hence, $d_H(hU) \cap \Omega^+ \neq \emptyset$ and so $hU \cap \ro{Frob}_H \neq \emptyset$. As $h\ca{U}$ is a filter basis of the neighborhood filter of $h$, we have proven that the intersection of $\ro{Frob}_H$ with any non-empty open subset of $H$ is non-empty, that is, $\ro{Frob}_H$ is dense in $H$.

\item Since $U$ is open in $H$, the quotient $H/U$ is discrete. Hence, if $\ol{h} \in H/U$, then $\lbrace \ol{h} \rbrace$ is open in $H/U$ and as $q$ is continuous, it follows that $q^{-1}(\ol{h})$ is open in $H$. As $\ro{Frob}_H$ is dense in $H$, there exists $h \in \ro{Frob}_H \cap q^{-1}(\ol{h})$. In particular, $h \in \ro{Frob}_H$ and $q(h) =\ol{h}$. This shows that $q|_{\ro{Frob}_H}$ is surjective. \vspace{-\baselineskip}
\end{asparaenum}

\end{proof}

\begin{ass}
For the rest of this section we fix a $G$-spectrum $\fr{E}$ and set $\fr{K} \dopgleich \fr{K}(\fr{E})_{\tn{ab}}$.
\end{ass}

\begin{defn}
We define the \word{underlying unramified $G$-spectrum} of $\fr{E}$ as $\fr{E}^{\tn{ur}} \dopgleich \fr{E} \cap \ro{Sp}(G)^{\tn{ur}}$.
\end{defn}

\begin{defn} \label{para:fn_datum}
A \word{Fesenko--Neukirch datum} on $\fr{E}$ is a pair $(C,v)$ consisting of a cohomological Mackey functor $C \in \se{Mack}^{\ro{c}}(\fr{M},\se{Ab})$ defined on an arithmetic and inertially finite Mackey cover $\fr{M}$ of $\fr{E}$ and a valuation $v \in \ro{Val}_d^{\Omega,\omega}(C)$ satisfying the following conditions:
\begin{enumerate}[label=(\roman*),noitemsep,nolistsep]
\item $(C,v) \in \ro{urFND}_d^{\omega}(\fr{E}^{\tn{ur}})$.
\item \label{item:fn_datum_sequence} For each open subgroup $U$ of a group $H \in \esys{b}^\flat$ and each $V \in \ca{E}^{\tn{ur}}(U)$ the sequence
\[
\xymatrix{
1 \ar[r] & \ker(v_U) \ar[r]^{\res_{V,U}^C} & \ker(v_V) \ar[rr]^{\con_{\varphi_{U \mid V}-1,V}^C} & & \ker(v_V) \ar[r]^{\ind_{U,V}^C} & \ker(v_U) \ar[r] & 1
}
\]
is exact.\footnote{Note that this sequence is well-defined as $C$ is defined on an arithmetic cover of $\fr{E}$. Moreover, it is easy to see that this sequence is a complex.}
\end{enumerate}

The set of all such pairs $(C,v)$ is denoted by $\ro{FND}_d^{\omega}(\fr{E})$.
\end{defn}

\begin{para}
In practice it will be easier to have conditions only on $C$ which imply that $(C,v) \in \ro{FND}_d^{\omega}(\fr{E})$ because then we do not have to additionally understand $\ker(v)$. The following proposition provides such conditions.
\end{para}

\begin{prop} \label{para:fn_datum_fesenko_style}
Let $C \in \se{Mack}^{\ro{c}}(\fr{M},\se{Ab})$ be defined on an arithmetic and inertially finite Mackey cover $\fr{M}$ of $\fr{E}$ and let $v \in \ro{Val}_d^{\langle \omega \rangle,\omega}(C) \subs \ro{Val}_d^{\Omega,\omega}(C)$. Suppose the following conditions hold:
\begin{enumerate}[label=(\roman*),noitemsep,nolistsep]
\item $|\widehat{\ro{H}}^0(C)(H,U)| = \lbrack H:U \rbrack$ for each $(H,U) \in \esys{b}^{\tn{ur}}$.
\item \label{para:fn_datum_fesenko_style_sequence} For each open subgroup $U$ of a group $H \in \esys{b}^\flat$ and each $V \in \ca{E}^{\tn{ur}}(U)$ the sequence
\[
\xymatrix{
1 \ar[r] & C(U) \ar[r]^{\res_{V,U}^C} & C(V) \ar[rr]^{\con_{\varphi_{U \mid V}-1,V}^C} & & C(V) \ar[r]^{\ind_{U,V}^C} & C(U) 
}
\]
is exact.
\end{enumerate}
Then $(C,v) \in \ro{FND}_d^{\omega}(\fr{E})$.
\end{prop}

\begin{proof}
It is easy to see that $(C,v) \in \ro{urFND}_d^\omega(\fr{E}^{\tn{ur}})$. Let $H \in \esys{b}^\flat$, $U \leq_{\ro{o}} H$ and $V \in \ca{E}^{\tn{ur}}(U)$. Let $y \in \ker(v_V)$ with $\con_{\varphi_{U | V} - 1, V}^C(y) = 1$. Then by assumption there exists $x \in C(V)$ with $\res_{V,U}^C(x) = y$. Since
\[
0 = v_U(y) = v_U \circ \res_{V,U}^C(x) = e_{U \mid V} \cdot v_V(x) = v_V(x),
\]
we have $x \in \ker(v_V)$ and this shows the exactness at the first $\ker(v_V)$. Now, let $y \in \ker(v_V)$ with $\ind_{U,V}^C(y) = 1$. By assumption there exists $x \in C(V)$ with $\con_{\varphi_{U|V}-1,V}^C(x) = y$.

If $\pi \in C(U)$ is a prime element, then $\res_{V,U}^C(\pi) \in C(V)$ is also a prime element as $e_{U|V} = 1$. Hence, we can write $v_V(x) = k \omega = k v_V( \res_{V,U}^C(\pi)) = v_V( \res_{V,U}^C(\pi^k) )$ for some $k \in \ZZ$ and so there exists $\eps \in \ker(v_V)$ such that $x = \eps \cdot \res_{V,U}^C(\pi^k)$. Since
\[
\con_{\varphi_{U|V}-1,V}^C(x) = \con_{\varphi_{U|V}-1,V}^C(\eps) \cdot \con_{\varphi_{U|V}-1,V}^C( \res_{V,U}^C(\pi^k)) = \con_{\varphi_{U|V}-1,V}^C(\eps),
\]
we have proven the exactness at the second $\ker(v_V)$. It remains to show that $\ind_{U,V}^C:\ker(v_V) \ra \ker(v_U)$ is surjective, so let $y \in \ker(v_U)$. Let $n \dopgleich \lbrack U:V \rbrack$ and let $v_U':C(U)/\ind_{U,V}^C C(V) \ra \im(v_U)/n \im(v_V)$ be the morphism induced by $v_U$ as in the proof of \ref{para:unramified_cft}. This morphism is obviously surjective and as $|\tateco^0(C)(U,V)| = n$ by assumption, we conclude that $v_U'$ is an isomorphism. Hence, as $v_U(y) = 0$, there exists $x \in C(V)$ with $y = \ind_{U,V}^C(x)$ and as
\[
0 = v_U(y) = v_U \circ \ind_{U,V}^C(x) = f_{U|V} v_V(x),
\]
we also have $x \in \ker(v_V)$. This shows that \ref{para:fn_datum}\ref{item:fn_datum_sequence} holds.
\end{proof}

\begin{para}
We note that condition \ref{para:fn_datum_fesenko_style}\ref{para:fn_datum_fesenko_style_sequence} is precisely the condition that $C$ has $(U,V)$-Galois descent and satisfies Hilbert 90 for $(U,V)$.
\end{para}

\begin{lemma} \label{para:inv_conj_lemma} 
Let $\fr{S}$ be a $G$-subgroup system and let $C \in \se{Stab}(\fr{S},\se{TAb})$. The following holds:
\begin{enumerate}[label=(\roman*),noitemsep,nolistsep]

\item \label{item:inv_conj_lemma_prod} If $H \in \ssys{b}$, $g \in H$ and $U \in \ssys{r}(H)$ with $^g \! U = U$, then for any $n \in \NN_{>0}$ the relation
\[
\con_{g^n-1,U}^C \circ \res_{U,H}^C  = \con_{g-1,U}^C \left( \prod_{j=0}^{n-1} \con_{g^j,U}^C \circ \res_{U,H}^C \right)
\]
holds in $C(U)$.

\item \label{item:inv_conj_ind_lemma_prod} If $H \in \ssys{b}$, $U \in \ssys{i}(H)$ and $g \in G$ with $^g \! H = H$, then for any $n \in \NN_{>0}$ the relation
\[
\con_{g^n-1,H}^C \circ \ind_{H,U}^C  = \con_{g-1,H}^C \left( \prod_{j=0}^{n-1} \con_{g^j,H}^C \circ \ind_{H,U}^C \right)
\]
holds in $C(H)$.
\end{enumerate}
\end{lemma}

\begin{proof} 
This is easy to verify.
\end{proof}

\begin{lemma} \label{para:neukirch_norm_lemma}
Let $H$ be an inertially finite subgroup of $G$, let $U$ be an open subgroup of $H$, let $h \in \ro{Frob}_H$ and let $\Sigma \dopgleich \Sigma_{h, H|U}$. Let $\varphi \in d_H^{-1}(\omega)$, let $W \dopgleich \ro{NC}_H(\Sigma \cap U)$ and let $W_0 \dopgleich W \cdot I_H$. Then $\lbrace \varphi^j \mid 0 \leq j < P(d_H(h)) \rbrace$ is a complete set of representatives of $W_0 \mybackslash H/\Sigma$.
\end{lemma}

\begin{proof}
Let $n \dopgleich \ro{mult}(d_H(h))$ and let $k \dopgleich P(d_H(h)) = P(n)$. Since $\Sigma I_H \geq I_H$, we have $e_{H \mid \Sigma I_H} = 1$ and therefore 
\[
\lbrack H: \Sigma I_H \rbrack = f_{H \mid \Sigma I_H} = \lbrack d(H) : d(\Sigma I_H) \rbrack = \lbrack d(H):d(\Sigma) \rbrack = f_{H \mid \Sigma} = P(d_H(h)) = k.
\]
Suppose that $\varphi^j \in \Sigma I_H$ for some $j \in \NN_{>0}$. Then 
\[
j \omega = d_H(\varphi^j) \in d_H(\Sigma I_H) = d_H(\Sigma) = d_H( \langle h \rangle_{\ro{c}} \cdot I_U ) = d_H(\langle h \rangle_{\ro{c}} ) = \langle d_H(h) \rangle_{\ro{c}} = \langle n \omega \rangle_{\ro{c}}
\]
and consequently $j\Omega \subs n \Omega$. Hence, using \ref{para:torsion_free_procyclic_index} we get
\[
\lbrack \Omega: j\Omega \rbrack = \lbrack \Omega: n \Omega \rbrack \cdot \lbrack n \Omega: j\Omega \rbrack = k \cdot \lbrack n \Omega: j \Omega \rbrack
\]
and therefore $k \mid \lbrack \Omega: j \Omega \rbrack$. Since $\lbrack \Omega: j \Omega \rbrack \mid j$ by \ref{para:compact_power_group_index}, we thus get $k \mid j$, that is, $j \in k \ZZ$. Hence, the elements $\lbrace \varphi^j \mid 0 \leq j < k \rbrace \subs H$ are pairwise inequivalent modulo $\Sigma I_H$ and since $\lbrack H: \Sigma I_H \rbrack = k$, we conclude that $\lbrace \varphi^j \mid 0 \leq j < k \rbrace$ is a complete set of representatives of $H/\Sigma I_H$. Since $W \lhd \Sigma$, we have $W_0 = W I_H \lhd \Sigma I_H$ and therefore an application of \ref{para:transversal_to_double} shows that $\lbrace \varphi^j \mid 0 \leq j < k \rbrace$ is also a complete set of representatives of $W_0 \mybackslash H / \Sigma I_H$. As $I_H$ is normal in $H$, we can write $\varphi^j \Sigma I_H = I_H \varphi^j \Sigma$ for all $j \in \NN_{>0}$ and as $I_H \leq W_0$, we get
\[
H = \coprod_{j=0}^{k-1} W_0 \varphi^j \Sigma I_H = \coprod_{j=0}^{k-1} W_0 I_H \varphi^j \Sigma = \coprod_{j=0}^{k-1} W_0 \varphi^j \Sigma.
\]
This shows that $\lbrace \varphi^j \mid 0 \leq j < k \rbrace$ is a complete set of representatives of $W_0 \mybackslash H / \Sigma$.
\end{proof}

\begin{conv}
In the following we will construct morphisms $\widetilde{\Upsilon}$, $\Upsilon'$ and $\Upsilon$ depending on the choice of $(C,v) \in \ro{FND}_d^{\omega}(\fr{E})$. To simplify notations we will not include a reference to this choice.
\end{conv}

\begin{thm}
Let $(C,v) \in \ro{FND}_d^{\omega}(\fr{E})$. For each $(H,U) \in \esys{b}$ the assignment
\[
\begin{array}{rcl}
\widetilde{\Upsilon}_{(H,U)}: \ro{Frob}_H & \lra & C(H)/\ind_{H,U}^C C(U) \\
h & \longmapsto & \ind_{H,\Sigma_{h, H \mid U}}^C(\pi_h)^{P'(d_H(h))} \modd \ind_{H,U}^C C(U),
\end{array}
\]
where $\pi_h \in C(\Sigma_{h, H \mid U})$ is a prime element, is a semigroup morphism which is independent of the choice of the prime elements.
\end{thm}

\begin{proof}
Let $h \in \ro{Frob}_H$ and let $\Sigma \dopgleich \Sigma_{h, H \mid U}$. Note that as $\Sigma$ is an open subgroup of $H$ and as $\fr{M}$ is an arithmetic cover of $\fr{E}$, we have $\Sigma \in \fr{M}_{\ro{i}}(H) \cap \fr{M}_{\ro{r}}(H)$. We will first prove that $\widetilde{\Upsilon}_{(H, U)}$ is independent of the choice of the prime elements. Let $\pi, \pi' \in C(\Sigma)$ be two prime elements. Since $\Sigma \cap U \geq (\Sigma \cap U) \cap \ker(d) = I_U \cap I_\Sigma = I_\Sigma$, we have $e_{\Sigma,\Sigma \cap U} = 1$. Hence, the morphism $\ind_{\Sigma,\Sigma \cap U}^C: \ker(v_{\Sigma \cap U}) \ra \ker(v_\Sigma)$ is surjective according to \ref{para:fn_datum}\ref{item:fn_datum_sequence} and now it follows from \ref{para:ker_v_surj_only_one_prime} that $\pi$ and $\pi'$ are equivalent modulo $\ind_{\Sigma,\Sigma \cap U}^C C(\Sigma \cap U)$. Since $\ind_{H,\Sigma \cap U}^C = \ind_{H,U}^C \circ \ind_{U,\Sigma \cap U}^C$, we have $\ind_{H,\Sigma \cap U}^C C(\Sigma \cap U) \subs \ind_{H,U}^C C(U)$ and therefore 
\[
\ind_{H,\Sigma}^C(\pi' \pi^{-1}) \in \ind_{H,\Sigma}^C( \ind_{\Sigma,\Sigma \cap U}^C C(\Sigma \cap U)) = \ind_{H,\Sigma \cap U}^C C(\Sigma \cap U) \subs \ind_{H,U}^C C(U).
\]
This shows that $\ind_{H,\Sigma}^C(\pi)^{P'(d_H(h))}$ and $\ind_{H,\Sigma}^C(\pi')^{P'(d_H(h))}$ are equivalent modulo $\ind_{H,U}^C C(U)$. \\

Now we prove that $\widetilde{\Upsilon}_{(H,U)}$ is multiplicative. Let $h_1,h_2 \in \ro{Frob}_H$ and let $h_3 \dopgleich h_1 h_2$. Let $\Sigma_i \dopgleich \Sigma_{h_i, H \mid U}$ and let $\pi_i \in C(\Sigma_i)$ be a prime element. Let $m_i \dopgleich P'(d_H(h_i))$ and $k_i \dopgleich P(d_H(h_i))$. Then $n_i \dopgleich m_ik_i = \ro{mult}(d_H(h_i))$ and $n_3 = n_1 + n_2$. We have to show that
\[
\ind_{H,\Sigma_1}^C(\pi_1^{m_1}) \cdot \ind_{H,\Sigma_2}^C(\pi_2^{m_2}) \equiv \ind_{H,\Sigma_3}^C(\pi_3^{m_3}) \modd \ind_{H,U}^C C(U)
\]
what is equivalent to
\begin{equation} \label{equ:mult_of_rec_main}
\ind_{H,\Sigma_1}^C(\pi_1^{m_1}) \cdot \ind_{H,\Sigma_2}^C(\pi_2^{m_2}) \cdot \ind_{H,\Sigma_3}^C(\pi_3^{-m_3}) \in \ind_{H,U}^C C(U).
\end{equation}
The proof proceeds in several steps beginning with one simplification step. \\

\begin{asparaenum}[{Step} 1{.}]
\item As $d_H:H \ra \Omega$ is surjective by \ref{para:morphism_d_H}, there exists $\varphi \in d_H^{-1}(\omega) \subs \ro{Frob}_H$. Let $\Sigma \dopgleich \Sigma_{\varphi, H \mid U}$ and let $U_1$ be the normal core of $\Sigma \cap \Sigma_1 \cap \Sigma_2 \cap \Sigma_3 \cap U$ in $H$. An application of \ref{para:normal_core}\ref{item:normal_core_qc_open} shows that $U_1$ is open in $H$.
We have $I_\Sigma = I_U = I_{\Sigma_i}$ for $i \in \lbrace 1,2,3 \rbrace$ and obviously $I_U \leq U$, $I_\Sigma \leq \Sigma$ and $I_{\Sigma_i} \leq \Sigma_i$ for $i \in \lbrace 1,2,3 \rbrace$. Hence, $I_U \leq \Sigma \cap \Sigma_1 \cap \Sigma_2 \cap \Sigma_3 \cap U$. But as $I_U = U \cap \ker(d)$ is normal in $H$ and as $U_1$ is the normal core of $\Sigma \cap \Sigma_1 \cap \Sigma_2 \cap \Sigma_3 \cap U$ in $H$, this implies immediately that $I_U \leq U_1$ and consequently $I_U \leq I_{U_1}$. On the other hand, since $U_1 \leq U$, we have $I_{U_1} \leq I_U$. This shows that $I_U = I_{U_1}$ and so we can write
\[
\Sigma_i = \Sigma_{h_i, H \mid U} = \langle h_i \rangle_{\ro{c}} \cdot I_U = \langle h_i \rangle_{\ro{c}} \cdot I_{U_1} = \Sigma_{h_i, H \mid U_1}.
\]
Hence, if we would have proven that $\widetilde{\Upsilon}_{(H, U_1)}$ is multiplicative, then, since $U_1 \leq U \in \fr{M}_{\ro{i}}(H)$, we would get $\widetilde{\Upsilon}_{(H,U_1)}(h_1) \cdot \widetilde{\Upsilon}_{(H,U_1)}(h_2) = \widetilde{\Upsilon}_{(H,U_1)}(h_3)$, thus
\[
\ind_{H,\Sigma_1}^C(\pi_1^{m_1}) \cdot \ind_{H,\Sigma_2}^C(\pi_2^{m_2}) \equiv \ind_{H,\Sigma_3}^C(\pi_3^{m_3}) \ \ro{mod} \ \ind_{H,U_1}^C C(U_1) \subs \ind_{H,U}^C C(U)
\]
and therefore
\[
\widetilde{\Upsilon}_{(H,U)}(h_1) \cdot \widetilde{\Upsilon}_{(H,U)}(h_2) = \widetilde{\Upsilon}_{(H,U)}(h_3),
\]
so that this would prove the multiplicativity of $\widetilde{\Upsilon}_{(H,U)}$. Hence, it suffices to prove the multiplicativity in the case $U = U_1$, that is, $U \leq \Sigma \cap \Sigma_1 \cap \Sigma_2 \cap \Sigma_3$. \\

\item So, assume that $U = U_1$. The element $h_4 \dopgleich \varphi^{n_2} h_1 \varphi^{-n_2}$ is obviously contained in $\ro{Frob}_H$ and so we can define $\Sigma_4 \dopgleich \Sigma_{h_4, H \mid U}$. Since $U \leq \Sigma_1$, we have 
\[
U = {^{\varphi^{n_2}} U} \leq {^{\varphi^{n_2}} \Sigma_1} = \varphi^{n_2} ( \langle h_1 \rangle_{\ro{c}} \cdot I_U) \varphi^{-n_2} = \langle \varphi^{n_2} h_1 \varphi^{-n_2} \rangle_{\ro{c}} \cdot I_U = \Sigma_4.
\]
Let $\pi_4 \dopgleich \con_{\varphi^{n_2},\Sigma_1}^C(\pi_1) \in C(^{\varphi^{n_2}}\Sigma_1) = C(\Sigma_4)$. Then
\[
v_{\Sigma_4}( \pi_4) = v_{ ^{\varphi^{n_2}}\Sigma_1} \circ \con_{\varphi^{n_2},\Sigma_1}^C(\pi_1 ) = \con_{\varphi^{n_2},\Sigma_1}^{\boldsymbol{\Omega}_d} \circ v_{\Sigma_1}(\pi_1) = v_{\Sigma_1}(\pi_1) = \omega
\]
and therefore $\pi_4$ is a prime element in $C(\Sigma_4)$. Because of the relation
\begin{align*}
\ind_{H,\Sigma_4}^C(\pi_4) &= \ind_{H,\Sigma_4}^C( \con_{\varphi^{n_2},\Sigma_1}^C(\pi_1) ) = \ind_{H,^{\varphi^{n_2}}\Sigma_1}^C \circ \con_{\varphi^{n_2},\Sigma_1}^C(\pi_1) \\
&= \con_{ \varphi^{n_2}, H}^C \circ \ind_{H,\Sigma_1}^C(\pi_1) = \ind_{H,\Sigma_1}^C(\pi_1)
\end{align*}
the logical assertion in \ref{equ:mult_of_rec_main} above now becomes
\begin{equation}
\ind_{H,\Sigma_4}^C(\pi_4^{m_1}) \cdot \ind_{H,\Sigma_2}^C(\pi_2^{m_2}) \cdot \ind_{H,\Sigma_3}^C(\pi_3^{-m_3}) \in \ind_{H,U}^C C(U).
\end{equation}

\item Let $m_4 \dopgleich P'(d_H(h_4)) = m_1$, $k_4 \dopgleich P(d_H(h_4)) = k_1$ and $n_4 \dopgleich m_4k_4 = \ro{mult}(d_H(h_4)) = n_1$. Let $i \in \lbrace 2,3,4 \rbrace$. By the choice of $\varphi$, we have $d_H(h_i) = n_i \omega = n_i d_H(\varphi) = d_H(\varphi^{n_i})$ and so there exists an element $\tau_i \in I_H = \ker(d_H)$ such that $h_i = \tau_i^{-1} \varphi^{n_i}$. These elements satisfy the following relation:
\begin{align} \label{equ:tau_relation}
\tau_4 \tau_2 &= \varphi^{n_4} h_4^{-1} \varphi^{n_2} h_2^{-1} = \varphi^{n_1} (\varphi^{n_2} h_1 \varphi^{-n_2})^{-1} \varphi^{n_2} h_2^{-1} = \varphi^{n_1} ( \varphi^{n_2} h_1^{-1} \varphi^{-n_2}) \varphi^{n_2} h_2^{-1} \notag \\
&= \varphi^{n_1 + n_2} h_1^{-1} h_2^{-1} = \varphi^{n_3} (h_2 h_1)^{-1} = \varphi^{n_3} h_3^{-1} = \tau_3.
\end{align}
Moreover, due to the equivariance of the conjugation morphisms and the fact that $h_i \in \Sigma_i$, we have
\begin{equation} \label{equ:con_tau_relation}
\con_{\tau_i,\Sigma_i}^C(\pi_i) = \con_{\varphi^{n_i} h_i^{-1},\Sigma_i}^C(\pi_i) = \con_{\varphi^{n_i},^{h_i^{-1}} \! \Sigma_i}^C \circ \con_{h_i^{-1},\Sigma_i}^C(\pi_i) = \con_{\varphi^{n_i},\Sigma_i}^C(\pi_i).
\end{equation}

Since $\Sigma_i \in \fr{M}_{\ro{b}}$ and since $U$ is open in $\Sigma_i$, we have $U \in \fr{M}_{\ro{r}}(\Sigma_i)$. This allows us to define
\[
\widehat{\pi}_i \dopgleich \prod_{j=0}^{n_i-1} \con_{\varphi^j, U}^C \circ \res_{U,\Sigma_i}^C(\pi_i) \in C(U).
\]

Using \ref{para:inv_conj_lemma}\ref{item:inv_conj_lemma_prod} we get the following relation for $i \in \lbrace 2,3,4 \rbrace$:
\begin{equation} \label{equ:con_pi_hat_rel}
\con_{\varphi^{n_i}-1,U}^C \circ \res_{U,\Sigma_i}^C(\pi_i) = \con_{\varphi-1,U}^C \left( \prod_{j=0}^{n_i-1} \con_{\varphi^j,U}^C \circ \res_{U,\Sigma_i}^C(\pi_i) \right) = \con_{\varphi-1,U}^C(\widehat{\pi}_i).
\end{equation}
Let
\[
\eps \dopgleich \widehat{\pi}_3 \widehat{\pi}_2^{-1} \widehat{\pi}_4^{-1} \in C(U).
\]
As $U \in \ca{E}^{\ro{ur}}(\Sigma_i)$ and as $v:C \ra \boldsymbol{\Omega}_d$ is a morphism, we get
\[
v_U( \widehat{\pi}_i) = \sum_{j=0}^{n_i-1} v_U( \con_{\varphi^j, U}^C \circ \res_{U,\Sigma_i}^C(\pi_i)) = \sum_{j=0}^{n_i-1} v_U(\res_{U,\Sigma_i}^C(\pi_i)) = \sum_{j=0}^{n_i-1} v_{\Sigma_i}(\pi_i) = \sum_{j=0}^{n_i-1} \omega = n_i\omega.
\]
Hence,
\begin{align*}
 v_U(\eps) &= v_U(\widehat{\pi}_3) - v_U(\widehat{\pi}_2) - v_U(\widehat{\pi}_4) = n_3\omega - n_2\omega - n_1\omega = n_3\omega - (n_1 + n_2)\omega = 0
\end{align*}
and this shows that $\eps \in \ker(v_U)$. Now, define
\[
\eps_2 \dopgleich \res_{U,\Sigma_3}^C(\pi_3) \cdot \res_{U,\Sigma_2}^C(\pi_2^{-1}) \in \ker(v_U)
\]
\[
\eps_4 \dopgleich \res_{U,\Sigma_4}^C(\pi_4^{-1}) \cdot ( \res_{U, {^{\tau_2}\Sigma_3}}^C \circ \con_{\tau_2,\Sigma_3}^C(\pi_3) ) \in \ker(v_U).
\]
We have the following relation in $C(U)$:
\begin{align*}
&\con_{\tau_2-1,U}^C(\eps_2) \cdot \con_{\tau_4-1}^C(\eps_4) = \frac{ \con_{\tau_2,U}^C(\eps_2)}{\eps_2} \cdot \frac{\con_{\tau_4,U}^C(\eps_4)}{\eps_4} \displaybreak[0] \\
&= \frac{\con_{\tau_2,U}^C( \res_{U,\Sigma_3}^C(\pi_3) \cdot \res_{U,\Sigma_2}(\pi_2^{-1}))}{  \res_{U,\Sigma_3}^C(\pi_3) \cdot \res_{U,\Sigma_2}^C(\pi_2^{-1})} \cdot \frac{ \con_{\tau_4,U}^C(\res_{U,\Sigma_4}^C(\pi_4^{-1}) \cdot ( \res_{U, {^{\tau_2}\Sigma_3}}^C \circ \con_{\tau_2,\Sigma_3}^C(\pi_3) ) )}{\res_{U,\Sigma_4}^C(\pi_4^{-1}) \cdot ( \res_{U, {^{\tau_2}\Sigma_3}}^C \circ \con_{\tau_2,\Sigma_3}^C(\pi_3) ) } \displaybreak[0] \\
&= \frac{\con_{\tau_2,U}^C \circ \res_{U,\Sigma_3}^C(\pi_3) \cdot \con_{\tau_2,U}^C \circ \res_{U,\Sigma_2}^C(\pi_2^{-1})}{\res_{U,\Sigma_3}^C(\pi_3) \cdot \res_{U,\Sigma_2}^C(\pi_2^{-1})} \cdot \frac{\con_{\tau_4,U}^C \circ \res_{U,\Sigma_4}^C(\pi_4^{-1}) \cdot \con_{\tau_4,U}^C \circ \res_{U,^{\tau_2}\Sigma_3}^C \circ \con_{\tau_2,\Sigma_3}^C(\pi_3)}{\res_{U,\Sigma_4}^C(\pi_4^{-1}) \cdot ( \res_{U,^{\tau_2}\Sigma_3}^C \circ \con_{\tau_2,\Sigma_3}^C(\pi_3))} \displaybreak[0] \\
&= \scriptstyle{\frac{(\res_{U,^{\tau_2}\Sigma_3}^C \circ \con_{\tau_2,\Sigma_3}^C(\pi_3)) \cdot \res_{U,^{\tau_2}\Sigma_2}^C \circ \con_{\tau_2,\Sigma_2}^C(\pi_2^{-1})}{\res_{U,\Sigma_3}^C(\pi_3) \cdot \res_{U,\Sigma_2}^C(\pi_2^{-1})} \cdot \frac{\res_{U,^{\tau_4}\Sigma_4}^C \circ \con_{\tau_4,\Sigma_4}^C(\pi_4^{-1}) \cdot \res_{U, ^{\tau_4 \tau_2}\Sigma_3}^C \circ \con_{\tau_4,^{\tau_2}\Sigma_3}^C \circ \con_{\tau_2,\Sigma_3}^C(\pi_3)}{\res_{U,\Sigma_4}^C(\pi_4^{-1}) \cdot (\res_{U,^{\tau_2}\Sigma_3}^C \circ \con_{\tau_2,\Sigma_3}^C(\pi_3))} } \displaybreak[0] \\
&= \frac{\res_{U,^{\tau_2}\Sigma_2}^C \circ \con_{\tau_2,\Sigma_2}^C(\pi_2^{-1})}{\res_{U,\Sigma_3}^C(\pi_3) \cdot \res_{U,\Sigma_2}^C(\pi_2^{-1})} \cdot \frac{\res_{U,^{\tau_4}\Sigma_4}^C \circ \con_{\tau_4,\Sigma_4}^C(\pi_4^{-1}) \cdot \res_{U, ^{\tau_4 \tau_2}\Sigma_3}^C \circ \con_{\tau_4 \tau_2,\Sigma_3}^C(\pi_3)}{\res_{U,\Sigma_4}^C(\pi_4^{-1})} \displaybreak[0] \\
&\overset{\ref{equ:tau_relation}}{=} \frac{\res_{U,^{\tau_2}\Sigma_2}^C \circ \con_{\tau_2,\Sigma_2}^C(\pi_2^{-1})}{\res_{U,\Sigma_3}^C(\pi_3) \cdot \res_{U,\Sigma_2}^C(\pi_2^{-1})} \cdot \frac{\res_{U,^{\tau_4}\Sigma_4}^C \circ \con_{\tau_4,\Sigma_4}^C(\pi_4^{-1}) \cdot \res_{U, ^{\tau_3}\Sigma_3}^C \circ \con_{\tau_3,\Sigma_3}^C(\pi_3)}{\res_{U,\Sigma_4}^C(\pi_4^{-1})} \displaybreak[0] \\
&= \frac{\res_{U,^{\tau_2}\Sigma_2}^C \circ \con_{\tau_2,\Sigma_2}^C(\pi_2^{-1})}{\res_{U,\Sigma_2}^C(\pi_2^{-1})} \cdot \frac{\res_{U,^{\tau_3}\Sigma_3}^C \circ \con_{\tau_3,\Sigma_3}^C(\pi_3)}{\res_{U,\Sigma_3}^C(\pi_3)} \cdot \frac{\res_{U,^{\tau_4} \Sigma_4}^C \circ \con_{\tau_4,\Sigma_4}^C(\pi_4^{-1})}{\res_{U,\Sigma_4}^C(\pi_4^{-1})} \displaybreak[0] \\
&= \frac{\res_{U,\Sigma_2}^C(\pi_2)}{\res_{U,^{\tau_2}\Sigma_2}^C \circ \con_{\tau_2,\Sigma_2}^C(\pi_2)} \cdot \frac{\res_{U,^{\tau_3}\Sigma_3}^C \circ \con_{\tau_3,\Sigma_3}^C(\pi_3)}{\res_{U,\Sigma_3}^C(\pi_3)} \cdot \frac{\res_{U,\Sigma_4}^C(\pi_4)}{\res_{U,^{\tau_4} \Sigma_4}^C \circ \con_{\tau_4,\Sigma_4}^C(\pi_4)} \displaybreak[0] \\
&\overset{\ref{equ:con_tau_relation}}{=} \frac{\res_{U,\Sigma_2}^C(\pi_2)}{\res_{U,^{\varphi^{n_2}}\Sigma_2}^C \circ \con_{\varphi^{n_2},\Sigma_2}^C(\pi_2)} \cdot \frac{\res_{U,^{\varphi^{n_3}}\Sigma_3}^C \circ \con_{\varphi^{n_3},\Sigma_3}^C(\pi_3)}{\res_{U,\Sigma_3}^C(\pi_3)} \cdot \frac{\res_{U,\Sigma_4}^C(\pi_4)}{\res_{U,^{\varphi^{n_4}} \Sigma_4}^C \circ \con_{\varphi^{n_4},\Sigma_4}^C(\pi_4)} \displaybreak[0] \\
&= \frac{\res_{U,\Sigma_2}^C(\pi_2)}{\con_{\varphi^{n_2},U}^C \circ \res_{U,\Sigma_2}^C(\pi_2)} \cdot \frac{\con_{\varphi^{n_3},U}^C \circ \res_{U,\Sigma_3}^C(\pi_3)}{\res_{U,\Sigma_3}^C(\pi_3)} \cdot \frac{\res_{U,\Sigma_4}^C(\pi_4)}{\con_{\varphi^{n_4},U}^C \circ \res_{U,\Sigma_4}^C(\pi_4)} \displaybreak[0] \\
&= \frac{1}{\con_{\varphi^{n_2}-1,U}^C \circ \res_{U,\Sigma_2}^C(\pi_2)} \cdot \con_{\varphi^{n_3}-1,U}^C \circ \res_{U,\Sigma_3}^C(\pi_3) \cdot \frac{1}{\con_{\varphi^{n_4}-1,U}^C \circ \res_{U,\Sigma_4}^C(\pi_4)} \displaybreak[0] \\
&\overset{\ref{equ:con_pi_hat_rel}}{=} \frac{1}{\con_{\varphi-1,U}^C(\widehat{\pi}_2)} \cdot \con_{\varphi-1,U}^C(\widehat{\pi}_3) \cdot \frac{1}{\con_{\varphi-1,U}^C(\widehat{\pi}_4)} = \con_{\varphi-1,U}^C(\widehat{\pi}_3 \widehat{\pi}_2^{-1} \widehat{\pi}_4^{-1}) = \con_{\varphi-1,U}^C(\eps)
\end{align*}
and this yields
\begin{equation} \label{equ:phi_eps_con_relation}
\con_{\tau_2-1,U}^C(\eps_2) \cdot \con_{\tau_4-1}^C(\eps_4) = \con_{\varphi-1,U}^C(\eps).
\end{equation}

\item The group $U_0 \dopgleich U \cdot I_H$ is an open normal subgroup of $H$. Let $\sigma \in \ro{Frob}_{U_0}$ with $d_{U_0}(\sigma) = n \dopgleich \lbrack H:U \rbrack$ and let $\Sigma^0 \dopgleich I_{U_0} \cdot \Sigma_{\sigma, H \mid U} \leq_{\ro{o}} U_0$. As $I_{U_0} \leq \Sigma^0$, we have $e_{U_0 \mid \Sigma^0} = 1$ and as $G$ is a pro-$P$ group, an application of \ref{para:torsion_free_procyclic_index} yields
\begin{align*}
\lbrack U_0 : \Sigma^0 \rbrack &= e_{U_0 \mid \Sigma^0} \cdot f_{U_0 \mid \Sigma^0} = f_{U_0 \mid \Sigma^0} = \lbrack d(U_0) : d(\Sigma^0) \rbrack = \lbrack f_{U_0} d_{U_0}(U_0) : f_{U_0}d_{U_0}(\Sigma^0) \rbrack \\
& = \lbrack f_{U_0} \Omega: f_{U_0} d_{U_0}(I_{U_0} \cdot \langle \sigma \rangle_{\ro{c}} \cdot I_U) \rbrack = \lbrack f_{U_0} \Omega: f_{U_0} d_{U_0}( \langle \sigma \rangle_{\ro{c}} ) \rbrack = \lbrack f_{U_0} \Omega: f_{U_0}n \Omega \rbrack \\
&= \lbrack \Omega: n \Omega \rbrack = P(n) = n.
\end{align*}
Since $H \cap \ker(d) = I_H \leq U_0 \leq H$, we have $H \cap \ker(d) \leq U_0 \cap \ker(d) \leq H \cap \ker(d)$. Hence, $I_H \leq I_{U_0} \leq I_H$ and therefore $I_H = I_{U_0}$. As $\Sigma^0 = I_H \cdot \Sigma_{\sigma, H \mid U} \geq I_H$ and as $H/I_H = \Gamma_H$ is abelian, we  conclude that $\Sigma^0$ is normal in $H$. Moreover, as $U_0 \lhd_{\ro{o}} H$ and $\Sigma^0 \leq_{\ro{o}} U_0$, we have $\Sigma^0 \lhd_{\ro{o}} H$. 

Let $W \dopgleich \Sigma^0 \cap U \lhd_{\ro{o}} H$. From $I_U \leq U$ and $I_U \leq I_H \leq I_H \cdot \Sigma_{\sigma, H \mid U} = \Sigma^0$, it follows that $I_U \leq U \cap \Sigma^0 = W \leq U$ and this shows that $e_{U \mid W} = 1$. Hence, according to \ref{para:fn_datum}\ref{item:fn_datum_sequence}, the morphism $\ind_{U,W}^C:\ker(v_W) \ra \ker(v_U)$ is surjective and so there exist $\eta,\eta_2,\eta_4 \in \ker(v_W)$ such that
\[
\eps = \ind_{U,W}^C(\eta), \quad \eps_2 = \ind_{U,W}^C(\eta_2), \quad \eps_4 = \ind_{U,W}^C(\eta_4).
\]
Now, equation \ref{equ:phi_eps_con_relation} yields
\[
 \con_{\tau_2-1,U}^C \circ \ind_{U,W}^C(\eta_2) \cdot \con_{\tau_4-1,U}^C \circ \ind_{U,W}^C(\eta_4) = \con_{\varphi-1,U}^C \circ \ind_{U,W}^C(\eta),
\]
that is,
\[
\frac{\con_{\tau_2,U}^C \circ \ind_{U,W}^C(\eta_2)}{\ind_{U,W}^C(\eta_2)} \cdot \frac{\con_{\tau_4,U}^C \circ \ind_{U,W}^C(\eta_4)}{\ind_{U,W}^C(\eta_4)} = \frac{\con_{\varphi,U}^C \circ  \ind_{U,W}^C(\eta)}{ \ind_{U,W}^C(\eta)}.
\]
Using the equivariance we get
\[
\frac{\ind_{U,W}^C \circ \con_{\tau_2,W}^C(\eta_2)}{\ind_{U,W}^C(\eta_2)} \cdot \frac{\ind_{U,W}^C \circ \con_{\tau_4,W}^C(\eta_4)}{\ind_{U,W}^C(\eta_4)} = \frac{\ind_{U,W}^C \circ \con_{\varphi,W}^C(\eta)}{\ind_{U,W}^C(\eta)},
\]
that is,
\[
\ind_{U,W}^C \left( \frac{\con_{\tau_2,W}^C(\eta_2)}{\eta_2} \right) \cdot  \ind_{U,W}^C \left( \frac{\con_{\tau_4,W}^C(\eta_4)}{\eta_4} \right) =  \ind_{U,W}^C \left( \frac{\con_{\varphi,W}^C(\eta)}{\eta} \right) 
\]
and therefore
\[
\ind_{U,W}^C( \con_{\varphi-1,W}^C(\eta) \cdot \con_{\tau_2-1,W}^C(\eta_2) \cdot \con_{\tau_4-1,W}^C(\eta_4)) = 1.
\]
It is obvious that $\con_{\varphi-1,W}^C(\eta) \cdot \con_{\tau_2-1,W}^C(\eta_2) \cdot \con_{\tau_4-1,W}^C(\eta_4) \in \ker(v_W)$ and therefore \ref{para:fn_datum}\ref{item:fn_datum_sequence} implies that there exists $\beta \in \ker(v_W)$ such that
\begin{equation} \label{equ:con_eta_beta_relation}
\con_{\varphi_{U \mid W}-1,W}^C(\beta) = \con_{\varphi-1,W}^C(\eta) \cdot \con_{\tau_2-1,W}^C(\eta_2) \cdot \con_{\tau_4-1,W}^C(\eta_4) .
\end{equation}
Since $\tau_2 \in I_H$ and $I_H \leq \Sigma^0$, we have
\[
\ind_{\Sigma^0,W}^C \circ \con_{\tau_2,W}^C(\eta_2) = \con_{\tau_2,\Sigma^0}^C \circ \ind_{\Sigma^0,W}^C(\eta_2) = \ind_{\Sigma^0,W}^C(\eta_2)
\]
and similarly 
\[
\ind_{\Sigma^0,W}^C \circ \con_{\tau_4,W}^C(\eta_4) = \ind_{\Sigma^0,W}^C(\eta_4).
\]
Hence, equation \ref{equ:con_eta_beta_relation} yields
\begin{align*}
\ind_{\Sigma^0,W}^C \circ \con_{\varphi_{U \mid W}-1,W}^C(\beta) &= \ind_{\Sigma^0,W}^C \circ \con_{\varphi-1,W}^C(\eta) \cdot \ind_{\Sigma^0,W}^C \circ \con_{\tau_2-1,W}^C(\eta_2)  \cdot \ind_{\Sigma^0,W}^C \circ \con_{\tau_4-1,W}^C(\eta_4) \\
&= \ind_{\Sigma^0,W}^C \circ \con_{\varphi-1,W}^C(\eta) \cdot \frac{\ind_{\Sigma^0,W}^C \circ \con_{\tau_2,W}^C(\eta_2) }{\ind_{\Sigma^0,W}^C(\eta_2)} \cdot \frac{\ind_{\Sigma^0,W}^C \circ \con_{\tau_4,W}^C(\eta_4) }{\ind_{\Sigma^0,W}^C(\eta_4)} \\
&= \ind_{\Sigma^0,W}^C \circ \con_{\varphi-1,W}^C(\eta)
\end{align*}
and this implies
\begin{equation} \label{equ:con_ind_beta_gamma_relation}
\frac{\ind_{\Sigma^0,W}^C \circ \con_{\varphi_{U \mid W},W}^C(\beta)}{\ind_{\Sigma^0,W}^C(\beta)} = \frac{\ind_{\Sigma^0,W}^C \circ \con_{\varphi,W}^C(\eta)}{\ind_{\Sigma^0,W}^C(\eta)}.
\end{equation}

Since $U_0 \geq I_H$, we have $e_{H \mid U_0} = 1$ and therefore $f_{H \mid U_0} = \lbrack H:U_0 \rbrack$. Obviously, $\varphi^{\lbrack H:U_0 \rbrack} \in U_0$ and so we can apply \ref{para:morphism_d_H}\ref{item:morphism_d_H_ind_compat} to get
 \[
 d_{U_0}(\varphi^{\lbrack H:U_0 \rbrack}) =  \frac{1}{f_{H \mid U_0}} d_H(\varphi^{\lbrack H:U_0 \rbrack}) =  \frac{1}{\lbrack H:U_0 \rbrack} d_H(\varphi^{\lbrack H:U_0 \rbrack}) = \omega.
 \]
Let $\psi \in d_U^{-1}(\omega)$. Since $U \cdot I_{U_0} = U \cdot I_H = U_0$, we have $f_{U_0 \mid U} = 1$ and so a further application of \ref{para:morphism_d_H}\ref{item:morphism_d_H_ind_compat} shows that
\[
\omega = d_U(\psi) = \frac{1}{f_{U_0 \mid U}} d_{U_0}(\psi) = d_{U_0}(\psi).
\]
Hence, $\varphi^{\lbrack H:U_0 \rbrack} \psi^{-1} \in \ker(d_{U_0}) = I_{U_0} = I_H \leq \Sigma^0$ and therefore $\varphi^{\lbrack H:U_0 \rbrack} \equiv \psi \modd \Sigma^0$. Using the fact that $\psi \modd W = \varphi_{U \mid W}$ this yields
\[
\ind_{\Sigma^0,W}^C \circ \con_{\varphi_{U \mid W},W}^C = \ind_{\Sigma^0,W}^C \circ \con_{\psi,W}^C = \con_{\psi,\Sigma^0}^C \circ \ind_{\Sigma^0,W} = \con_{\varphi^{\lbrack H:U_0 \rbrack},\Sigma^0}^C \circ \ind_{\Sigma^0,W}^C.
\]
In particular, equation \ref{equ:con_ind_beta_gamma_relation} is now equivalent to
\[
\frac{\con_{\varphi^{\lbrack H:U_0 \rbrack},\Sigma^0}^C \circ \ind_{\Sigma^0,W}^C(\beta)}{\ind_{\Sigma^0,W}^C(\beta)} = \frac{\ind_{\Sigma^0,W}^C \circ \con_{\varphi,W}^C(\eta)}{\ind_{\Sigma^0,W}^C(\eta)} = \frac{\con_{\varphi,\Sigma^0}^C \circ \ind_{\Sigma^0,W}^C(\eta)}{\ind_{\Sigma^0,W}^C(\eta)}
\]
and this is equivalent to
\begin{equation} \label{equ:scheissgleichung}
\con_{\varphi^{\lbrack H:U_0 \rbrack}-1,\Sigma^0}^C \circ \ind_{\Sigma^0,W}^C(\beta) = \con_{\varphi-1,\Sigma^0}^C \circ \ind_{\Sigma^0,W}^C(\eta).
\end{equation}

\item Let
\[
\gamma \dopgleich \prod_{j=0}^{\lbrack H:U_0 \rbrack -1} \con_{\varphi^j, W}^C(\beta) \in \ker(v_W) \subs C(W).
\]
Then it follows from \ref{para:inv_conj_lemma}\ref{item:inv_conj_ind_lemma_prod} and equation  \ref{equ:scheissgleichung} that
\[
\con_{\varphi-1,\Sigma^0}^C \circ \ind_{\Sigma^0,W}^C(\gamma) = \con_{\varphi^{\lbrack H:U_0 \rbrack}-1,\Sigma^0}^C \circ \ind_{\Sigma^0,W}^C(\beta) = \con_{\varphi-1,\Sigma^0}^C \circ \ind_{\Sigma^0,W}^C(\eta)
\]
As $\Sigma^0 \geq I_H$, we have $e_{H \mid \Sigma^0} = 1$ and therefore the sequence
\[
\xymatrix{
\ker(v_{H}) \ar[rr]^{\res_{\Sigma^0,H}^C} & & \ker(v_{\Sigma^0}) \ar[rr]^{\con_{\varphi_{H \mid \Sigma^0}-1,\Sigma^0}^C} & & \ker(v_{\Sigma^0})
} 
\]
is exact by \ref{para:fn_datum}\ref{item:fn_datum_sequence}. Hence, there exists $\delta \in \ker(v_H)$ such that
\begin{equation} \label{equ:eta_gamma_delta_relation}
\ind_{\Sigma^0,W}^C(\eta) = \ind_{\Sigma^0,W}^C(\gamma) \cdot \res_{\Sigma^0,H}^C(\delta).
\end{equation}

Since $e_{H \mid U_0} = 1$, we have $H/U_0 = \langle \varphi_{H \mid U_0} \rangle$ and so $\lbrace 1,\varphi,\ldots,\varphi^{\lbrack H:U_0 \rbrack - 1} \rbrace$ is a complete set of representatives of the right cosets $U_0 \mybackslash H$. As $W \lhd U_0$, an application of \ref{para:transversal_to_double} shows that this is also a complete set of representatives of $U_0 \mybackslash H / W$. Now we can use the Mackey formula for $C$ to get
\begin{align*}
\res_{U_0,H}^C \circ \ind_{H,W}^C &= \prod_{j=0}^{\lbrack H:U_0 \rbrack -1} \ind_{U_0,H \cap ^{\varphi^j} \! W}^C \circ \res_{U_0 \cap ^{\varphi^j} \! W,^{\varphi^j} \! W}^C \circ \con_{\varphi^j, W}^C \\
& = \prod_{j=0}^{\lbrack H:U_0 \rbrack -1} \ind_{U_0,H \cap W}^C \circ \res_{U_0 \cap W,W}^C \circ \con_{\varphi^j, W}^C \\
&= \prod_{j=0}^{\lbrack H:U_0 \rbrack -1} \ind_{U_0,W}^C \circ \res_{W,W}^C \circ \con_{\varphi^j, W}^C \\
&= \prod_{j=0}^{\lbrack H:U_0 \rbrack -1} \ind_{U_0,W}^C \circ \con_{\varphi^j, W}^C
\end{align*}
and this immediately implies that
\[
\ind_{U_0,W}^C(\gamma) = \res_{U_0,H}^C \circ \ind_{H,W}^C(\beta).
\]
Using this relation and the cohomologicality of $C$ we derive from equation \ref{equ:eta_gamma_delta_relation} that
\begin{align}
\ind_{U_0,U}^C(\eps) &= \ind_{U_0,U}^C \circ \ind_{U,W}^C(\eta) = \ind_{U_0,W}^C(\eta) = \ind_{U_0,\Sigma^0}^C \circ \ind_{\Sigma^0,W}^C(\eta) \notag \\
&= \ind_{U_0,\Sigma^0}^C (\ind_{\Sigma^0,W}^C(\gamma) \cdot \res_{\Sigma^0,H}^C(\delta)) = \ind_{U_0,W}^C(\gamma) \cdot \ind_{U_0,\Sigma^0}^C( \res_{\Sigma^0,H}^C(\delta)) \notag \\
&= \ind_{U_0,W}^C(\gamma) \cdot \ind_{U_0,\Sigma^0}^C \circ \res_{\Sigma^0,U_0}^C \circ \res_{U_0,H}^C(\delta) =  \ind_{U_0,W}^C(\gamma) \cdot \res_{U_0,H}^C(\delta^{\lbrack U_0:\Sigma^0 \rbrack}) \displaybreak[0] \notag \\
&=  \ind_{U_0,W}^C(\gamma) \cdot\res_{U_0,H}^C( \delta^{\lbrack H:U \rbrack}) =  \ind_{U_0,W}^C(\gamma) \cdot \res_{U_0,H}^C( \ind_{H,U}^C \circ \res_{U,H}^C(\delta)) \displaybreak[0] \notag \\
&= \res_{U_0,H}^C \circ \ind_{H,W}^C(\beta) \cdot \res_{U_0,H}^C( \ind_{U,H}^C \circ \res_{H,U}^C(\delta)) \displaybreak[0] \notag \\
&= \res_{U_0,H}^C( \ind_{H,W}^C(\beta) \cdot \ind_{H,U}^C \circ \res_{U,H}^C(\delta)) \displaybreak[0] \notag \\
&= \res_{U_0,H}^C( \ind_{H,U}^C \circ \ind_{U,W}^C(\beta) \cdot \ind_{U,H}^C \circ \res_{H,U}^C(\delta))  \displaybreak[0] \notag \\
\label{equ:hauptgleichung}
&= \res_{U_0,H}^C( \ind_{H,U}^C( \ind_{U,W}^C(\beta) \cdot \res_{U,H}^C(\delta))).
\end{align}

\item 
As $U = \ro{NC}_H(\Sigma \cap U)$, it follows from \ref{para:neukirch_norm_lemma} that $\lbrace \varphi^j \mid 0 \leq j < k_i \rbrace$ is a complete set of representatives of $U_0 \mybackslash H / \Sigma_i$. An application of the Mackey formula for $C$ now shows that
\[
\res_{U_0,H}^C \circ \ind_{H,\Sigma_i}^C(\pi_i) = \prod_{j=0}^{k_i-1} \ind_{U_0,U_0 \cap {^{\varphi^j} \! \Sigma_i}}^C \circ \res_{U_0 \cap {^{\varphi^j} \! \Sigma_i},{^{\varphi^j} \! \Sigma_i}}^C \circ \con_{\varphi^j,\Sigma_i}^C(\pi_i).
\]

It is obvious that $U \leq U_0 \cap \Sigma_i$. But the converse inclusion also holds. To see this, let $h \in U_0 \cap \Sigma_i$. Then $h = uh' = s$ for some $u \in U$, $h' \in I_H$ and $s \in \Sigma_i$. Hence, $h' = u^{-1}s \in U \Sigma_i = \Sigma_i$ and therefore $h' \in I_H \cap \Sigma_i = I_{\Sigma_i} = I_U \leq U$ which implies that $h = uh' \in U$. Now, we also have $U_0 \cap {^{\varphi^j} \! \Sigma_i} = (U_0 \cap \Sigma_i)^{\varphi^j} = U^{\varphi^j} = U$ and so the above equation becomes
\begin{align}
\res_{U_0,H}^C \circ \ind_{H,\Sigma_i}^C(\pi_i) &= \prod_{j=0}^{k_i-1} \ind_{U_0,U}^C \circ \res_{U,{^{\varphi^j} \! \Sigma_i}}^C \circ \con_{\varphi^j,\Sigma_i}^C(\pi_i) \notag \\
&= \ind_{U_0,U}^C \left( \prod_{j=0}^{k_i-1} \res_{U,{^{\varphi^j} \! \Sigma_i}}^C \circ \con_{\varphi^j,\Sigma_i}^C(\pi_i) \right) \notag \\
\label{equ:mainstep_mult} &= \ind_{U_0,U}^C \left( \prod_{j=0}^{k_i-1} \con_{\varphi^j,U}^C \circ \res_{U,\Sigma_i}^C(\pi_i) \right).
\end{align}

For arbitrary $l \in \NN$ we have the following relation:
\begin{align*}
\res_{U_0,H}^C \circ \ind_{H,\Sigma_i}^C(\pi_i) &= \res_{U_0,H}^C \circ \id_{C(H)} \circ \ind_{H,\Sigma_i}^C(\pi_i) = \res_{U_0,H}^C \circ \con_{\varphi^{lk_i},H}^C \circ \ind_{H,\Sigma_i}(\pi_i) \displaybreak[0] \\
&=\con_{\varphi^{lk_i},U_0}^C \circ \res_{U_0,H}^C \circ \ind_{H,\Sigma_i}^C(\pi_i) \displaybreak[0] \\
&=\con_{\varphi^{lk_i},U_0}^C \circ \ind_{U_0,U}^C \left( \prod_{j=0}^{k_i-1} \con_{\varphi^j,U}^C \circ \res_{U,\Sigma_i}^C(\pi_i) \right) \displaybreak[0] \\
&=\ind_{U_0,U}^C \circ \con_{\varphi^{lk_i},U}^C \left( \prod_{j=0}^{k_i-1} \con_{\varphi^j,U}^C \circ \res_{U,\Sigma_i}^C(\pi_i) \right) \displaybreak[0] \\
&=\ind_{U_0,U}^C \left( \prod_{j=0}^{k_i-1}  \con_{\varphi^{lk_i},U}^C  \circ \con_{\varphi^j,U}^C \circ \res_{U,\Sigma_i}^C(\pi_i) \right) \displaybreak[0] \\
&=\ind_{U_0,U}^C \left( \prod_{j=0}^{k_i-1}  \con_{\varphi^{lk_i +j},U}^C \circ \res_{U,\Sigma_i}^C(\pi_i) \right) \displaybreak[0] \\
&=\ind_{U_0,U}^C \left( \prod_{j=lk_i}^{(l+1)k_i-1}  \con_{\varphi^{j},U}^C \circ \res_{U,\Sigma_i}^C(\pi_i) \right) \displaybreak[0] \\
\end{align*}

Hence, it follows from equation \ref{equ:mainstep_mult} that
\begin{align*}
\res_{U_0,H}^C \circ \ind_{H,\Sigma_i}^C(\pi_i^{m_i}) &= \ind_{U_0,U}^C \left( \prod_{j=0}^{m_i k_i-1}  \con_{\varphi^{j},U}^C \circ \res_{U,\Sigma_i}^C(\pi_i) \right) \\
&= \ind_{U_0,U}^C \left( \prod_{j=0}^{n_i-1}  \con_{\varphi^{j},U}^C \circ \res_{U,\Sigma_i}^C(\pi_i) \right) \\
&= \ind_{U_0,U}^C(\widehat{\pi}_i).
\end{align*}

Now, we have
\begin{align}
& \ \res_{U_0,H}^C( \ind_{H,\Sigma_3}^C(\pi_3^{m_3}) \cdot \ind_{H,\Sigma_2}^C(\pi_2^{-m_2}) \cdot \ind_{H,\Sigma_4}^C(\pi_4^{-m_4})) \notag \displaybreak[0] \\
= & \ \res_{U_0,H}^C \circ  \ind_{H,\Sigma_3}^C(\pi_3^{m_3}) \cdot \res_{U_0,H}^C \circ \ind_{H,\Sigma_2}^C(\pi_2^{-m_2}) \cdot \res_{U_0,H}^C \circ \ind_{H,\Sigma_4}^C(\pi_4^{-m_4}) \notag \displaybreak[0] \\
= & \ \ind_{U_0,U}^C(\widehat{\pi}_3) \cdot \ind_{U_0,U}^C(\widehat{\pi}_2^{-1}) \cdot \ind_{U_0,U}^C(\widehat{\pi}_4^{-1}) \notag \displaybreak[0] \\
= & \ \ind_{U_0,U}^C(\widehat{\pi}_3 \widehat{\pi}_2^{-1} \widehat{\pi}_4^{-1}) \notag \displaybreak[0] \\
\label{equ:eps_pi_relation}
= & \ \ind_{U_0,U}^C(\eps) \overset{\ref{equ:hauptgleichung}}{=} \res_{U_0,H}^C( \ind_{H,U}^C( \ind_{U,W}^C(\beta) \cdot \res_{U,H}^C(\delta))).
\end{align}
We have
\[
v_H( \res_{U_0,H}^C( \ind_{H,U}^C( \ind_{U,W}^C(\beta) \cdot \res_{U,H}^C(\delta))) ) = e_{H|U_0} \cdot f_{H|U} \cdot ( f_{U|W} v_W(\beta) + e_{H|U}v_H(\delta) ) = 0
\]
and
\begin{align*}
v_H( \ind_{H,\Sigma_3}^C(\pi_3^{m_3}) \cdot \ind_{H,\Sigma_2}^C(\pi_2^{-m_2}) \cdot \ind_{H,\Sigma_4}^C(\pi_4^{-m_4})) &= m_3f_{H \mid \Sigma_3}\omega - m_2f_{H \mid \Sigma_2}\omega - m_1f_{H \mid \Sigma_4}\omega \\
& = m_3k_3\omega - m_2k_2 \omega - m_1k_1 \omega \\
& = n_3\omega - n_2\omega - n_1\omega = 0.
\end{align*}
Hence 
\[
\res_{U_0,H}^C( \ind_{H,U}^C( \ind_{U,W}^C(\beta) \cdot \res_{U,H}^C(\delta))) \in \ker(v_H)
\]
and
\[
\ind_{H,\Sigma_3}^C(\pi_3^{m_3}) \cdot \ind_{H,\Sigma_2}^C(\pi_2^{-m_2}) \cdot \ind_{H,\Sigma_4}^C(\pi_4^{-m_1}) \in \ker(v_H).
\]
As $e_{H \mid U_0} = 1$, the morphism $\res_{U_0,H}^{\ker(v)}$ is injective by \ref{para:fn_datum}\ref{item:fn_datum_sequence} and therefore equation \ref{equ:eps_pi_relation} implies that
\[
\ind_{H,\Sigma_3}^C(\pi_3^{m_3}) \cdot \ind_{H,\Sigma_2}^C(\pi_2^{-m_2}) \cdot \ind_{H,\Sigma_4}^C(\pi_4^{-m_1}) = \ind_{H,U}^C( \ind_{U,W}^C(\beta) \cdot \res_{U,H}^C(\delta)) \in \ind_{H,U}^C C(U).
\]
This proves the multiplicativity.
\end{asparaenum} \vspace{-\baselineskip}
\end{proof}

\begin{prop}
Let $(C,v) \in \ro{FND}_d^{\omega}(\fr{E})$. Then for each $(H,U) \in \esys{b}$ the assignment
\[
\begin{array}{rcl}
\Upsilon_{(H, U)}': H/U & \lra & C(H)/\ind_{H,U}^C C(U) \\
\ol{h} & \longmapsto & \widetilde{\Upsilon}_{(H,U)}(h),
\end{array}
\]
where $h \in \ro{Frob}_H$ is a Frobenius lift of $\ol{h}$, is a group morphism which is independent of the choice of the Frobenius lifts. Moreover, the diagram
\[
\xymatrix{
\ro{Frob}_H \ar[rr]^-{\widetilde{\Upsilon}_{(H,U)}} \ar[d] & & C(H)/\ind_{H,U}^C C(U) \\
H/U \ar[rru]_-{\Upsilon_{(H,U)}'}
}
\]
commutes, where the vertical morphism is the restriction of the quotient morphism $H \ra H/U$.
\end{prop}

\begin{proof}
We will first show that $\Upsilon_{(H,U)}'$ is independent of the choice of the Frobenius lifts. Let $h_1,h_2 \in \ro{Frob}_H$ be two Frobenius lifts of $\ol{h}$. Let $\Sigma_i \dopgleich \Sigma_{h_i, H \mid U}$ and let $\pi_i \in C(\Sigma_i)$ be a prime element. Then 
\[
\widetilde{\Upsilon}_{(H,U)}(h_i) = \ind_{H,\Sigma_i}^C(\pi_i^{P'(d_H(h_i))}) \modd \ind_{H,U}^C C(U).
\]
by definition. First, suppose that $d_H(h_1) = d_H(h_2)$. Then $h_1 \equiv h_2 \modd I_H$ and as both $h_i$ are Frobenius lifts of $\ol{h}$, we also have $h_1 \equiv h_2 \modd U$. Hence, $h_1 \equiv h_2 \modd I_H \cap U = I_U$ what implies that $\Sigma_1 = \Sigma_2$ and consequently $\Upsilon'_{(H,U)}(\ol{h})$ does not depend on the choice of the Frobenius lift in this case. Now, suppose that $d_H(h_1) < d_H(h_2)$. Then $\tau \dopgleich h_1^{-1}h_2 \in \ro{Frob}_H$ and since both $h_1$ and $h_2$ are Frobenius lifts of $\ol{h}$, we also have $\tau \in U$. In particular, $\Sigma \dopgleich \Sigma_{\tau, H \mid U} \subs U$ and if $\pi \in C(\Sigma)$ is a prime element, then 
\[
\widetilde{\Upsilon}_{(H,U)}(\tau) = \ind_{H,\Sigma}^C(\pi^{P'(d_H(\tau))}) \modd \ind_{H,U}^C C(U) = \ind_{H,U}^C \circ \ind_{U,\Sigma}^C(\pi^{P'(d_H(\tau))}) \modd \ind_{H,U}^C C(U) = 1.
\]
Hence, using the multiplicativity of $\widetilde{\Upsilon}_{(H,U)}$ we get
\[
\widetilde{\Upsilon}_{(H,U)}(h_2) = \widetilde{\Upsilon}_{(H,U)}(h_1 \tau) = \widetilde{\Upsilon}_{(H,U)}(h_1) \cdot \widetilde{\Upsilon}_{(H,U)}(\tau) = \widetilde{\Upsilon}_{(H,U)}(h_1)
\]
and this shows that $\Upsilon_{(H,U)}'$ is independent of the choice of the Frobenius lifts. It remains to show that $\Upsilon_{(H,U)}'$ is multiplicative. Let $\ol{h}_1,\ol{h}_2 \in H/U$ and let $h_1,h_2 \in \ro{Frob}_H$ be Frobenius lifts. It is obvious that $h_1h_2$ is a Frobenius lift of $\ol{h}_1 \ol{h}_2$ and therefore
\[
\Upsilon_{(H,U)}'(\ol{h}_1 \ol{h}_2) = \widetilde{\Upsilon}_{(H,U)}(h_1h_2) = \widetilde{\Upsilon}_{(H,U)}(h_1) \cdot \widetilde{\Upsilon}_{(H,U)}(h_2) = \Upsilon_{(H,U)}'(\ol{h}_1) \cdot \Upsilon_{(H,U)}'(\ol{h}_2).
\]
The commutativity of the diagram is evident.
\end{proof}

\begin{thm}
Let $(C,v) \in \ro{FND}_d^{\omega}(\fr{E})$. For $(H,U) \in \esys{b}$ let $\Upsilon_{(H,U)}:(H/U)^{\ro{ab}} \ra C(H)/\ind_{H,U}^C C(U)$ be the unique morphism induced by $\Upsilon_{(H,U)}'$, that is,
\[
\begin{array}{rcl}
\Upsilon_{(H,U)}: (H/U)^{\ro{ab}} & \lra & C(H)/\ind_{H,U}^C C(U) \\
\ol{h} \modd \comm{a}(H/U) & \longmapsto & \ind_{H,\Sigma_{h, H \mid U}}^C(\pi_h^{P'(d_H(h))}) \modd \ind_{H,U}^C C(U),
\end{array}
\]
where $h \in \ro{Frob}_H$ is a Frobenius lift of $\ol{h} \in H/U$ and $\pi_h \in C(\Sigma_{h, H \mid U})$ is a prime element. Then the family $\Upsilon = \lbrace \Upsilon_{(H, U)} \mid (H,U) \in \esys{b} \rbrace$ is a canonical morphism
\[
\Upsilon: \pi_{\fr{K}} \ra \tateco_{\fr{E}}^0(C)
\]
in $\se{Fct}(\fr{E},\se{TAb})$, called the \word{Fesenko--Neukirch morphism} for $(C,v)$ on $\fr{E}$. Moreover, this morphism is an extension of the Fesenko--Neukirch reciprocity morphism for $(C,v)$ on $\fr{E}^{\tn{ur}}$.
\end{thm}

\begin{proof}
We have to show that $\Upsilon$ is compatible with the conjugation, restriction and induction morphisms. Let $(H,U) \in \esys{b}$ and let $g \in G$. Let $\ol{h} \in H/U$, let $h \in \ro{Frob}_H$ be a Frobenius lift of $\ol{h}$, let $\Sigma \dopgleich \Sigma_{h, H \mid U}$ and let $\pi \in C(\Sigma)$ be a prime element. Let $\con_{g,(H,U)}:H/U \ra {^g \! H}/{^g \! U}$ be the morphism induced by the conjugation $H \ra {^g \! H}$, $h \mapsto ghg^{-1}$. It is obvious that $ghg^{-1} \in \ro{Frob}_{^g \! H}$ is a Frobenius lift of $\con_{g,(H, U)}(\ol{h})$ and $^g \! \Sigma = \Sigma_{ghg^{-1}, {^g \! H} \mid {^g \! U}}$. Since $\con_{g,\Sigma}^C(\pi) \in C(^g \! \Sigma)$ is a prime element, it follows that
\begin{align*}
\Upsilon_{(^g \! H, ^g \! U)}' \circ \con_{g,(H,U)}(\ol{h}) &= \widetilde{\Upsilon}_{(^g \! H,^g \! U)}(ghg^{-1}) \\
&=\ind_{^g \! H, {^g \! \Sigma}}^C( \con_{g,\Sigma}^C(\pi)^{P'(d_{^g \! H}(ghg^{-1}))}) \modd \ind_{^g \! H, ^g \! U}^C C(^g \! U) \\
&=\ind_{^g \! H, {^g \! \Sigma}}^C( \con_{g,\Sigma}^C(\pi)^{P'(d_{H}(h))}) \modd \ind_{^g \! H, ^g \! U}^C C(^g \! U) \\
&=\con_{g,H}^C \circ \ind_{H,\Sigma}^C(\pi^{P'(d_{H}(h))}) \modd \ind_{^g \! H, ^g \! U}^C C(^g \! U) \\
&=\con_{g,(H,U)}^{\tateco_{\fr{E}}^0(C)}(  \ind_{H,\Sigma}^C(\pi^{P'(d_{H}(h))}) \modd \ind_{H,U}^C C(U)) \\
&=\con_{g,(H,U)}^{\tateco_{\fr{E}}^0(C)}( \widetilde{\Upsilon}_{(H,U)}(h) ) \\
&=\con_{g,(H,U)}^{\tateco_{\fr{E}}^0(C)}( \Upsilon'_{(H,U)}(\ol{h}) )
\end{align*}
and consequently the diagram
\[
\xymatrix{
H/U \ar[rr]^-{\Upsilon_{(H, U)}'} \ar[d]_{\con_{g,(H,U)}} & & C(H)/\ind_{H,U}^C C(U) \ar[d]^{\con_{g,(H,U)}^{\tateco_{\fr{E}}^0(C)}} \\
^g \! H/ ^g \! U \ar[rr]_-{\Upsilon_{(^g \! H, ^g \! U)}'} & & C(^g \! H)/\ind_{^g \! H, ^g \! U}^C C(^g \! U) 
}
\]
commutes. An application of the functor $(-)^{\ro{ab}}$ now shows that the diagram
\[
\xymatrix{
(H/U)^{\ro{ab}} \ar[rr]^-{\Upsilon_{(H,U)}} \ar[d]_{\con_{g,(H,U)}^{\pi_{\fr{K}}}} & & C(H)/\ind_{H,U}^C C(U) \ar[d]^{\con_{g,(H,U)}^{\tateco_{\fr{E}}^0(C)}} \\
(^g \! H/ ^g \! U)^{\ro{ab}} \ar[rr]_-{\Upsilon_{(^g \! H, ^g \! U)}} & & C(^g \! H)/\ind_{^g \! H, ^g \! U}^C C(^g \! U) 
}
\]
commutes, that is, $\Upsilon$ is compatible with the conjugation morphisms. 

Now, we prove the compatibility with the restriction morphisms. Let $(H,U) \in \esys{b}$, $(I,U) \in \esys{r}(H,U)$ and let $\ol{h} \in H/U$. Moreover, let $h \in \ro{Frob}_H$ be a Frobenius lift of $\ol{h}$, let $\Sigma \dopgleich \Sigma_{h, H \mid U}$ and let $\pi \in C(\Sigma)$ be a prime element. Let $R = \lbrace \rho_1,\ldots,\rho_n \rbrace$ be a complete set of representatives of $I \mybackslash H / \Sigma$ and let $q:H \ra H/U$ be the quotient morphism. Then 
\[
q(\Sigma) = q(\langle h \rangle_{\ro{c}} \cdot I_U) = q(\langle h \rangle_{\ro{c}}) = \langle q(h) \rangle_{\ro{c}} = \langle \ol{h} \rangle,
\]
where the last equality follows from the fact that $H/U$ is discrete. Hence, as $H = \coprod_{i=1}^n I \rho_i \Sigma$, we have 
\[
H/U = q(H) = \bigcup_{i=1}^n q(I) q(\rho_i) q(\Sigma) = \bigcup_{i=1}^n (I/U) \cdot q(\rho_i) \cdot \langle \ol{h} \rangle.
\]
This union is also disjoint since $\emptyset \neq q(I \rho_i \Sigma) \cap q(I \rho_j \Sigma)$ implies that 
\begin{align*}
\emptyset \neq q^{-1}(q(I \rho_i \Sigma) \cap q(I \rho_j \Sigma)) &= q^{-1}(q(I \rho_i \Sigma)) \cap q^{-1}(q(I \rho_j \Sigma)) \\
&= I \rho_i \Sigma U \cap I \rho_j \Sigma U \\
&= U I \rho_i \Sigma \cap U I \rho_j \Sigma \\
&= I \rho_i \Sigma \cap I \rho_j \Sigma
\end{align*}
and consequently $i=j$.
In particular, if $\ol{\rho}_i \dopgleich q(\rho_i)$, then $\ol{R} \dopgleich \lbrace \ol{\rho}_1,\ldots,\ol{\rho}_n \rbrace$ is a complete set of representatives of $(I/U) \mybackslash (H/U) / \langle \ol{h} \rangle$ and so it follows from \ref{para:transfer_alternative_rep} that
\[
\ro{Ver}_{I/U,H/U}^{\comm{t}(I/U),\comm{t}(H/U)}(\ol{h} \modd \comm{t}(H/U)) = \prod_{i=1}^n \ol{\rho}_i \ol{h}^{f(i)} \ol{\rho}_i^{-1} \modd \comm{t}(I/U),
\]
where
\[
f(i) \dopgleich \lambda_{\ol{h}}(\ol{\rho}_i) = \ro{min} \ \lbrace j \mid j \in \NN_{>0} \wedge \ol{\rho}_i \ol{h}^j \ol{\rho}_i^{-1} \in I/U \leq H/U \rbrace.
\]

First note that as $G$ is a pro-$P$ group and as $f(i)$ divides $\lbrack H/U:I/U \rbrack$, the prime factors of $f(i)$ are contained in $P$ so that $P(f(i)) = f(i)$ and $P'(f(i)) = 1$. \\

Let $i \in \lbrace 1,\ldots, n \rbrace$. We define $\Sigma_i \dopgleich {^{\rho_i} \Sigma} = \langle \rho_i h \rho_i^{-1} \rangle_{\ro{c}} \cdot I_U$. As $\ol{\rho}_i \ol{h}^{f(i)} \ol{\rho}_i^{-1} \in I/U$, we can conclude that $\rho_i h^{f(i)} \rho_i^{-1} \in I$. Since $f_{H|I} d_I = (d_H)|_I$, we have $f_{H|I} \Omega = f_{H|I} d_I(I) = d_H(I)$ and with $m \dopgleich \ro{mult}(d_H(h))$ we thus get
\[
f(i) m \omega = d_H( \rho_i h^{f(i)} \rho_i^{-1} ) \in d_H(I) = f_{H|I} \Omega.
\]
This implies that $f(i) m \Omega \subs f_{H|I} \Omega$ and therefore
\[
\lbrack \Omega: f(i) m \Omega \rbrack = \lbrack \Omega: f_{H|I} \Omega \rbrack \cdot \lbrack f_{H|I} \Omega: f(i) m \Omega \rbrack = P(f_{H|I}) \cdot \lbrack f_{H|I} \Omega: f(i) m \Omega \rbrack = f_{H|I} \cdot \lbrack f_{H|I} \Omega: f(i) m \Omega \rbrack.
\]

Hence, $f_{H|I} \mid \lbrack \Omega: f(i) m \Omega \rbrack$ and as $\lbrack \Omega: f(i) m \Omega \rbrack \mid f(i) m$ we conclude that $f_{H|I} \mid f(i) m$. Consequently,
\[
d_I( \rho_i h^{f(i)} \rho_i^{-1} ) = \frac{1}{f_{H|I}} d_H(\rho_i h^{f(i)} \rho_i^{-1} ) = \frac{f(i) m}{f_{H|I}} \omega \in \Omega^+
\]
and accordingly $\rho_i h^{f(i)} \rho_i^{-1} \in \ro{Frob}_I$. This allows us to define $\Sigma_i^* \dopgleich \Sigma_{\rho_i h^{f(i)} \rho_i^{-1}, I \mid U} = \langle \rho_i h^{f(i)} \rho_i^{-1} \rangle_{\ro{c}} \cdot I_U$. We already note that the above yields the relation
\[
P'(d_I( \rho_i h^{f(i)} \rho_i^{-1} ) ) = P'\left( \frac{f(i) m}{f_{H|I}} \right) = P'(m) = P'(d_H(h)).
\]

As $\Sigma_i^* \leq \Sigma_i$ and as $I_{\Sigma} = I_U = I_{\Sigma_i^*} \leq \Sigma_i^*$, we have $e_{\Sigma_i \mid \Sigma_i^*} = 1$. We will now prove that $I \cap \Sigma_i = \Sigma_i^*$. The inclusion $\Sigma_i^* \subs I \cap \Sigma_i$ follows from the fact that $\rho_i h^{f(i)} \rho_i^{-1} \in I$. To prove the converse inclusion, we will first show that $(\langle \rho_i h \rho_i^{-1} \rangle \cdot I_U ) \cap I \subs \Sigma_i^*$. If $x \in (\langle \rho_i h \rho_i^{-1} \rangle \cdot I_U ) \cap I$, then $x = \rho_i h^j \rho_i^{-1} u$ for some $j \in \ZZ$ and $u \in I_U$. Now, $q(x) = \ol{\rho}_i \ol{h}^j \ol{\rho}_i^{-1} \in I/U$ and so it follows from \ref{para:transfer_alternative_rep} that $j \in f(i) \ZZ$. Hence, $x \in \langle \rho_i h^{f(i)} \rho_i^{-1} \rangle \cdot I_U \subs \Sigma_i^*$. It is easy to see that $\ro{cl}(\langle \rho_i h
 \rho_i^{-1} \rangle) \cdot I_U \subs \ro{cl}(\langle \rho_i h \rho_i^{-1} \rangle \cdot I_U)$.\footnote{If $A,B$ are subgroups of $H$, then $A,B \subs AB$, hence $\ro{cl}(A) \subs \ro{cl}(AB)$ and $\ro{cl}(B) \subs \ro{cl}(AB)$. Using the fact that $\ro{cl}(AB)$ is a subgroup, we get $\ro{cl}(A) \cdot \ro{cl}(B) \subs \ro{cl}(AB)$. In the above, we also use that $I_U = U \cap \ker(d)$ is closed in $H$.} Moreover, noting that $I$ is of finite index in $H$ and is thus open in $H$, we can use \cite[chapitre I, \S1.6, proposition 5]{Bou71_Topologie-Generale_0} to get $\ro{cl}( \langle \rho_i h \rho_i^{-1} \rangle \cdot I_U) \cap I \subs \ro{cl}( (\langle \rho_i h \rho_i^{-1} \rangle \cdot I_U) \cap I)$. Combining the above results, we finally get
\begin{align*}
\Sigma_i \cap I &= (\langle \rho_i h \rho_i^{-1} \rangle_{\ro{c}} \cdot I_U) \cap I = (\ro{cl}(\langle \rho_i h \rho_i^{-1} \rangle) \cdot I_U) \cap I\\
&\subs \ro{cl}(\langle \rho_i h \rho_i^{-1}  \rangle \cdot I_U) \cap I  \subs \ro{cl}( (\langle \rho_i h \rho_i^{-1} \rangle \cdot I_U) \cap I) \\
&\subs \ro{cl}(\Sigma_i^*) = \Sigma_i^*.
\end{align*}

Now, the Mackey formula for $C$ yields
\begin{align*}
\res_{I,H}^C \circ \ind_{H,\Sigma}^C &= \prod_{i=1}^n \ind_{I, I \cap {^{\rho_i}\Sigma}}^C \circ \res_{I \cap {^{\rho_i}\Sigma},{^{\rho_i}\Sigma}}^C \circ \con_{\rho_i,\Sigma}^C \\
&= \prod_{i=1}^n \ind_{I, I \cap \Sigma_i}^C \circ \res_{I \cap \Sigma_i, \Sigma_i}^C \circ \con_{\rho_i,\Sigma}^C \\
&= \prod_{i=1}^n \ind_{I, \Sigma_i^*}^C \circ \res_{\Sigma_i^*,\Sigma_i}^C \circ \con_{\rho_i,\Sigma}^C.
\end{align*}
Since $e_{\Sigma_i \mid \Sigma_i^*} = 1$, we conclude that $\pi_i^* \dopgleich \res_{\Sigma_i^*,\Sigma_i}^C \circ \con_{\rho_i,\Sigma}^C(\pi) \in C(\Sigma_i^*)$ is a prime element. Recalling that $\rho_i h^{f(i)} \rho_i^{-1} \in \ro{Frob}_I$ is a Frobenius lift of $\ol{\rho}_i \ol{h}^{f(i)} \ol{\rho}_i^{-1} \in I/U$ and that $\Sigma_i^* = \Sigma_{\rho_i h^{f(i)} \rho_i^{-1}, I \mid U}$, we get
\begin{align*}
& \ \Upsilon_{(I,U)} \circ \res_{(I,U),(H,U)}^{\pi_{\fr{K}}}( \ol{h} \modd \comm{t}(H/U)) \displaybreak[0] \\
= & \ \Upsilon_{(I,U)} \circ \ver_{I/U,H/U}^{\comm{t}(I/U),\comm{t}(H/U)}( \ol{h} \modd \comm{t}(H/U)) \displaybreak[0] \\
= & \ \Upsilon_{(I,U)} \left( \prod_{i=1}^n \ol{\rho}_i \ol{h}^{f(i)} \ol{\rho}_i^{-1} \modd \comm{t}(I/U) \right) \displaybreak[0] \\
= & \ \prod_{i=1}^n \Upsilon_{(I,U)}( \ol{\rho}_i \ol{h}^{f(i)} \ol{\rho}_i^{-1} \modd \comm{t}(I/U) ) \displaybreak[0] \\
= & \ \prod_{i=1}^n \Upsilon_{(I,U)}'( \ol{\rho}_i \ol{h}^{f(i)} \ol{\rho}_i^{-1}) \displaybreak[0] \\
= & \ \prod_{i=1}^n \widetilde{\Upsilon}_{(I,U)}( \rho_i h^{f(i)} \rho_i^{-1}) \displaybreak[0] \\
= & \ \prod_{i=1}^n \ind_{I,\Sigma_i^*}^C( \pi_i^* )^{P'(d_I( \rho_i h^{f(i)} \rho_i^{-1}) )} \modd \ind_{I,U}^C C(I) \displaybreak[0] \\
= & \ \prod_{i=1}^n \ind_{I,\Sigma_i^*}^C( \pi_i^* )^{P'(d_H(h))} \modd \ind_{I,U}^C C(I) \displaybreak[0] \\
= & \ \prod_{i=1}^n \ind_{I,\Sigma_i^*}^C \circ \res_{\Sigma_i^*,\Sigma_i}^C \circ \con_{\rho_i,\Sigma}^C(\pi^{P'(d_H(h))}) \modd \ind_{I,U}^C C(I) \displaybreak[0] \\
= & \ \res_{I,H}^C \circ \ind_{H,\Sigma}^C(\pi^{P'(d_H(h))}) \modd \ind_{I,U}^C C(U) \displaybreak[0] \\
= & \ \res_{(I,U),(H,U)}^{\tateco_{\fr{E}}^0(C)}( \ind_{H,\Sigma}^C(\pi^{P'(d_H(h))}) \modd \ind_{H,U}^C C(U)) \displaybreak[0] \\
= & \ \res_{(I,U),(H,U)}^{\tateco_{\fr{E}}^0(C)}( \widetilde{\Upsilon}_{(H,U)}( h) ) \displaybreak[0] \\
= & \ \res_{(I,U),(H,U)}^{\tateco_{\fr{E}}^0(C)} \circ \Upsilon_{(H,U)}(\ol{h} \modd \comm{t}(H/U)).
\end{align*}

This proves the compatibility with the restriction morphisms. It remains to verify the compatibility with the induction morphisms. Let $(H,U) \in \esys{b}$, $(I,V) \in \esys{i}(H,U)$ and let $\ol{x} \in I/V$. Let $x \in \ro{Frob}_I$ be a Frobenius lift of $\ol{x}$, let $\Sigma \dopgleich \Sigma_{x, I \mid V}$ and let $\pi \in C(\Sigma)$ be a prime element. Let $\ind_{(H,U),(I,V)}:I/V \ra H/U$ be the morphism induced by the inclusion $I \rightarrowtail H$. The relation $f_{H|I} d_I = (d_H)|_I$ implies that $\ro{Frob}_I \subs \ro{Frob}_H$ and therefore it is evident that $x \in \ro{Frob}_H$ is a Frobenius lift of $\ind_{(H,U),(I,V)}(\ol{x})$. Let $\Sigma' \dopgleich \Sigma_{x, H \mid U}$. Then
\[
\Sigma' = \langle x \rangle_{\ro{c}} \cdot I_U = (\langle x \rangle_{\ro{c}} \cdot I_V) \cdot I_U = \Sigma \cdot I_U = \Sigma \cdot I_{\Sigma'}
\]
and this shows that $\Sigma$ is a totally ramified subgroup of $\Sigma'$, that is, $f_{\Sigma' \mid \Sigma} = 1$. Hence, $\pi' \dopgleich \ind_{\Sigma',\Sigma}^C(\pi) \in C(\Sigma')$ is a prime element. Using the relation
\[
P'(d_I(x)) = 1 \cdot P'(d_I(x)) = P'( f_{H \mid I}) \cdot P'(d_I(x)) = P'( f_{H \mid I} d_I(x)) = P'(d_H(x))
\]
we conclude that
\begin{align*}
& \ \ind_{(H,U),(I,V)}^{\tateco_{\fr{E}}^0(C)} \circ \Upsilon'_{(I,V)}( \ol{x} ) \\
= & \ \ind_{(H,U),(I,V)}^{\tateco_{\fr{E}}^0(C)}( \ind_{I,\Sigma}^C(\pi^{P'(d_I(x))}) \modd \ind_{I,V}^C C(V)) \\
= & \ \ind_{H,I}^C( \ind_{I,\Sigma}^C(\pi^{P'(d_H(x))}) ) \modd \ind_{H,U}^C C(U) \\
= & \ \ind_{H,\Sigma}^C(\pi^{P'(d_H(x))}) \modd \ind_{H,U}^C C(U) \\
= & \ \ind_{H,\Sigma'}^C \circ \ind_{\Sigma',\Sigma}^C(\pi^{P'(d_H(x))}) \modd \ind_{H,U}^C C(U) \\
= & \ \ind_{H,\Sigma'}^C(\pi')^{P'(d_H(x))} \modd \ind_{H,U}^C C(U) \\
= & \ \Upsilon'_{(H,U)}( \ind_{(H,U),(I,V)}(\ol{x}) ).
\end{align*}

This shows that the diagram
\[
\xymatrix{
H/U \ar[rr]^-{\Upsilon_{(H,U)}'} & & C(H)/\ind_{H,U}^C C(U) \\
I/V \ar[rr]_-{\Upsilon_{(I,V)}'} \ar[u]^{\ind_{(H,U),(I,V)}} & & C(I)/\ind_{I,V}^C C(V) \ar[u]_{\ind_{(H,U),(I,V)}^{\tateco_{\fr{E}}^0(C)}}
}
\]
commutes and an application of the functor $(-)^{\ro{ab}}$ finally shows that the diagram
\[
\xymatrix{
(H/U)^{\ro{ab}} \ar[rr]^-{\Upsilon_{(H,U)}} & & C(H)/\ind_{H,U}^C C(U) \\
(I/V)^{\ro{ab}} \ar[rr]_-{\Upsilon_{(I, V)}} \ar[u]^{\ind_{(H,U),(I,V)}^{\pi_{\fr{K}}}} & & C(I)/\ind_{I,V}^C C(V) \ar[u]_{\ind_{(H,U),(I,V)}^{\tateco_{\fr{E}}^0(C)}}
}
\]
commutes. Hence, $\Upsilon$ is compatible with the induction morphisms and we conclude that $\Upsilon:\pi_{\fr{K}} \ra \tateco_{\fr{E}}^0(C)$ is a morphism in $\se{Fct}(\fr{E},\se{TAb})$.

It remains to show that $\Upsilon$ is an extension of the Fesenko--Neukirch reciprocity morphism for $(C,v)$ on $\fr{E}^{\tn{ur}}$. For this, let $(H,U) \in \esys{b}$ such that $e_{H \mid U} = 1$. Let $\varphi \in d_H^{-1}(\omega)$. Then $\varphi \in \ro{Frob}_H$ and $\varphi \modd U = \varphi_{H | U}$. Let $\Sigma \dopgleich \Sigma_{\varphi, H|U}$. Then
\[
\lbrack H:\Sigma \rbrack = e_{H \mid \Sigma} \cdot f_{H \mid \Sigma} = \lbrack I_H:I_\Sigma \rbrack \cdot P(d_H(\varphi)) = \lbrack I_H:I_U \rbrack \cdot 1 = 1
\]
and therefore $H = \Sigma$. If $\pi_H \in C(H)$ is a prime element, then
\[
\Upsilon_{(H,U)}( \varphi_{H \mid U}) = \widetilde{\Upsilon}_{(H,U)}( \varphi) = \ind_{H,H}^C(\pi_H^{P'(d_H(\varphi))}) \modd \ind_{H,U}^C C(U) = \pi_H \modd \ind_{H,U}^C C(U). 
\]
Hence, $\Upsilon_{(H,U)}$ coincides with the morphism in \ref{para:unramified_cft}.
\end{proof}

\begin{prop} \label{para:fs_cft_reduction}
Let $(C,v) \in \ro{FND}_d^{\omega}(\fr{E})$. Assume that $\fr{K}$ is coherent and assume that $C$ satisfies the class field axiom for all $(H,U) \in \esys{b}^\ell$ with $\ell$ being any prime number. If $\Upsilon_{(H,U)}$ is injective for all $(H,U) \in \esys{b}^\ell$ with $f_{H|U} = 1$ and $\ell$ being any prime number, then $\Upsilon:\pi_{\fr{K}} \ra \tateco^0_{\fr{E}}(C)$ is already a $\fr{K}$-class field theory.
\end{prop}

\begin{proof}
According to \ref{para:prime_cft_reduction} it is enough to show that $\Upsilon$ is a $\fr{K}$-prime morphism, so let $\ell$ be a prime number and let $(H,U) \in \esys{b}^\ell$. Since $\ell = \lbrack H:U \rbrack = e_{H \mid U} \cdot f_{H \mid U}$, we have either $e_{H \mid U} = 1$ or $f_{H \mid U} = 1$. If $e_{H \mid U} = 1$, then $\Upsilon$ is an isomorphism according to \ref{para:unramified_cft} because $(C,v) \in \ro{urFND}_d^{\Omega,\omega}(\fr{E}^{\tn{ur}})$. If $f_{H \mid U} = 1$, then $\Upsilon_{(H,U)}: H/U \ra \tateco^0(C)(H,U)$ is injective by assumption and as $|\tateco^0(C)(H,U)| = \lbrack H:U \rbrack$, this morphism is already surjective and is thus an isomorphism. Hence, $\Upsilon$ is a $\fr{K}$-prime morphism.
\end{proof}

\begin{cor} \label{para:fs_cft_main_case}
Assume that $\fr{E} = \ro{Sp}(G)^{\tn{f}}$. Let $\fr{M} = \ro{Grp}(G)^{\tn{f}}$, let $C \in \se{Mack}^{\ro{c}}(\fr{M},\se{Ab})$ and let $v \in \ro{Val}_d^{\langle \omega \rangle,\omega}(C) \subs \ro{Val}_d^{\Omega,\omega}$. Suppose that the following conditions are satisfied:
\begin{enumerate}[label=(\roman*),noitemsep,nolistsep]
\item \label{item:fs_cft_main_case_cft_ax} $C$ satisfies the class field axiom for all $(H,U) \in \esys{b}^{\tn{ur}}$ and for all $(H,U) \in \esys{b}^\ell$ with $f_{H|U} = 1$ and $\ell$ being a prime number.
\item \label{item:fs_cft_main_case_gal_desc} $C$ has $(H,U)$-Galois descent for all $(H,U) \in \esys{b}^{\tn{ur}}$.
\item \label{item:fs_cft_main_case_gal_inj} $\Upsilon_{(H,U)}$ is injective for all $(H,U) \in \esys{b}^\ell$ with $f_{H|U} = 1$ and $\ell$ being a prime number.
\end{enumerate}
Then $(C,v) \in \ro{FND}_d^{\omega}(\fr{E})$ and $\Upsilon:\pi_{\fr{K}(G)_{\tn{ab}}^{\tn{f}}} \ra \tateco^0_{\fr{E}}(C)$ is a $\fr{K}(G)_{\tn{ab}}^{\tn{f}}$-class field theory.
\end{cor}

\begin{proof}
First note that $\fr{M}$ is an arithmetic and inertially finite Mackey cover of $\fr{E}$. The class field axiom and the Galois descent for all $(H,U) \in \esys{b}^{\tn{ur}}$ are precisely the conditions in \ref{para:fn_datum_fesenko_style} and thus imply that $(C,v) \in \ro{FND}_d^\omega(\fr{E})$. Now, it follows immediately from \ref{para:fs_cft_reduction} that $\Upsilon$ is an isomorphism.
\end{proof}

\begin{prop} \label{para:fs_cft_tot_ram_iso}
Let $(C,v) \in \ro{FND}_d^{\omega}(\fr{E})$. Assume that the following conditions are satisfied:
\begin{enumerate}[label=(\roman*),noitemsep,nolistsep]
\item $\tateco^{-1}(C)(H,U) = 1$ for all $(H,U) \in \esys{b}^\ell$ with $f_{H|U} = 1$ and $\ell$ being any prime number.
\item \label{item:para:fs_cft_tot_ram_iso_gal_desc} For each open subgroup $U$ of a group $H \in \esys{b}^\flat$ and each $V \lhd_{\ro{o}} U$ with $f_{U|V} = 1$ and $U/V$ being cyclic of prime degree the sequence
\[
\xymatrix{
C(U) \ar[r]^{\res_{V,U}^C} & C(V) \ar[rr]^{\con_{\ol{u}-1,V}^C} & & C(V) 
}
\]
is exact, where $\ol{u}$ is a generator of $U/V$.
\end{enumerate}

Then $\Upsilon_{(H,U)}$ is injective for all $(H,U) \in \esys{b}^\ell$ with $f_{H|U} = 1$ and $\ell$ being a prime number.
\end{prop}

\begin{proof} \hspace{-5pt}\footnote{The proof given here is partially based on the proof of \cite[chapter IV, theorem 6.3]{Neu99_Algebraic-Number_0}.}
Let $(H,U) \in \esys{b}^\ell$ with $f_{H|U} = 1$ and $\ell$ being a prime number. Since $f_{H|U} = 1$, we have $U I_H = H$ and so there exists $h \in I_H$ such that $h \modd U$ is a generator of $H/U$. Let $\varphi \in d_U^{-1}(\omega)$. Since $d_U = f_{H|U} d_U = (d_H)|_U$, we have $\omega = d_U(\varphi) = d_H(\varphi)$ and therefore $\omega = d_H(h \varphi)$. This implies that $h \varphi \in \ro{Frob}_H$ is a Frobenius lift of the generator $h \modd U \in H/U$. Let $\Sigma \dopgleich \Sigma_{h \varphi, H|U}$. Then $f_{H|\Sigma} = P(d_H(h \varphi)) = 1$. Let $W \dopgleich \ro{NC}_H(U \cap \Sigma)$ and let $W_0 \dopgleich WI_H$. Since $I_U$ is normal in $H$ and $\Sigma \cap U \geq I_\Sigma \cap I_U = I_U$, we have $W \geq I_U$ and thus $I_W \geq I_U$. But as $U \geq W$, we also have $I_U \geq I_W$ and so we can conclude that $I_\Sigma = I_U = I_W$. Moreover, we have
\[
I_H \geq I_{W_0} = W_0 \cap \ker(d) = (WI_H) \cap \ker(d) \geq I_H \cap \ker(d) = I_H
\]
and consequently $I_H = I_{W_0}$. Since $f_{H|U} = 1$, we have $\lbrack H:U \rbrack = e_{H|U} = \lbrack I_H:I_U \rbrack$ and this implies that the canonical inclusion
\[
I_H/I_U = (H \cap \ker(d)) / (U \cap \ker(d)) \rightarrowtail H/U 
\]
is an isomorphism. As $W I_{W_0} = W I_H = W_0$, we also have $f_{W_0|W} = 1$ and the same argument as above implies that
\[
I_{W_0}/I_W \rightarrowtail W_0/W
\]
is an isomorphism. Hence, we have an isomorphism
\[
\xymatrix{
H/U \ar[r] & I_H/I_U \ar@{=}[r] & I_{W_0}/I_W \ar[r] & W_0/W \\
h \varphi \modd U \ar@{|->}[r] & h \modd I_U \ar@{|->}[r] & h \modd I_W \ar@{|->}[r] & h \modd W.
}
\]
In particular, $W_0/W$ is a cyclic group of order $\ell$ with generator $h \modd W$. Since $W_0 \cap U \geq W$, we have
\[
\ell = \lbrack W_0:W \rbrack = \lbrack W_0:W_0 \cap U \rbrack \cdot \lbrack W_0 \cap U:W \rbrack
\]
and therefore $\lbrack W_0:W_0 \cap U \rbrack = 1$ or $\lbrack W_0 \cap U:W \rbrack = 1$. In the first case, we would have $W_0 = W_0 \cap U$ and this would imply that $I_H \leq W_0 \leq U$ and so $H = U$ what is a contradiction to $\lbrack H:U \rbrack = \ell$. Hence, we must have $W_0 \cap U = W$. Similarly, since $W_0 \cap \Sigma \geq W$, we have
\[
\ell = \lbrack W_0:W \rbrack = \lbrack W_0:W_0 \cap \Sigma \rbrack \cdot \lbrack W_0 \cap \Sigma:W \rbrack
\]
and if $\lbrack W_0 \cap \Sigma:W_0 \rbrack = 1$, then $I_H \leq \Sigma$ and therefore $\ell = \lbrack H:U \rbrack = e_{H|U} = e_{H|\Sigma} = 1$ what is a contradiction. Hence, $W_0 \cap \Sigma = W$.

Now, let $\pi_\Sigma \in C(\Sigma)$ be a prime element. Then
\[
\Upsilon_{(H,U)}(h \modd U) = \widetilde{\Upsilon}_{(H,U)}(h \varphi) = \ind_{H,\Sigma}^C( \pi_\Sigma^{P'(d_H(h \varphi))}) \modd \ind_{H,U}^C C(U) = \ind_{H,\Sigma}^C(\pi_\Sigma) \modd \ind_{H,U}^C C(U)
\]
and therefore
\[
\Upsilon_{(H,U)}(h^k \modd U) = \ind_{H,\Sigma}^C(\pi_\Sigma^k) \modd \ind_{H,U}^C C(U)
\]
for all $k \in \ZZ$. Hence, to show that $\Upsilon_{(H,U)}$ is injective, we have to show that $\ind_{H,\Sigma}^C(\pi_\Sigma^k) \in \ind_{H,U}^C C(U)$ for $0 \leq k < \lbrack H:U \rbrack$ implies $k = 0$. So, suppose that $\ind_{H,\Sigma}^C(\pi_\Sigma^k) = \ind_{H,U}^C(z)$ for some $z \in C(U)$.

Let $\pi_U \in C(U)$ be a prime element. As $e_{U|W} = 1$ and $e_{\Sigma|W} = 1$, both $\res_{W,U}^C(\pi_U) \in C(W)$ and $\res_{W,\Sigma}^C(\pi_\Sigma) \in C(W)$ are prime elements. Consequently, there exists $\eps \in \ker(v_W)$ such that $\eps \cdot \res_{W,U}^C(\pi_U^k) = \res_{W,\Sigma}^C(\pi_\Sigma^k)$. As $W_0 \Sigma = W I_H \Sigma = WH = H$, it follows that $\lbrace 1 \rbrace$ is a complete set of representatives of $W_0 \mybackslash H / \Sigma$. By the same argument, the set $\lbrace 1 \rbrace$ is a complete set of representatives of $W_0 \mybackslash H/U$. Hence, an application of the Mackey formula yields
\[
\res_{W_0,H}^C \circ \ind_{H,\Sigma}^C = \ind_{W_0,W_0 \cap \Sigma}^C \circ \res_{W_0 \cap \Sigma,\Sigma}^C = \ind_{W_0,W}^C \circ \res_{W,\Sigma}^C
\]
and
\[
\res_{W_0,H}^C \circ \ind_{H,U}^C = \ind_{W_0,W_0 \cap U}^C \circ \res_{W_0 \cap U, U}^C = \ind_{W_0,W}^C \circ \res_{W,U}^C.
\]
These relations yield
\begin{align*}
\ind_{W_0,W}^C(\eps \cdot \res_{W,U}^C(\pi_U^k)) &= \ind_{W_0,W}^C \circ \res_{W,\Sigma}^C(\pi_\Sigma^k) = \res_{W_0,H}^C \circ \ind_{H,\Sigma}^C(\pi_\Sigma^k) \\
&= \res_{W_0,H}^C \circ \ind_{H,U}^C(z) = \ind_{W_0,W}^C \circ \res_{W,U}^C(z)
\end{align*}
and we get
\[
\ind_{W_0,W}^C(\eps) = \ind_{W_0,W}^C( \res_{W,U}^C(z) \cdot \res_{W,U}^C(\pi_U^{-k})) = \ind_{W_0,W}^C( \res_{W,U}^C(\delta))
\]
with $\delta \dopgleich z\pi_U^{-k} \in C(U)$. Since
\[
v_U(z) = f_{H|U} v_U(z) = v_H \circ \ind_{H,U}^C(z) = v_H \circ \ind_{H,\Sigma}^C(\pi_\Sigma^k) = f_{H|\Sigma} v_\Sigma(\pi_\Sigma^k) =  k\omega,
\]
we conclude that $\delta \in \ker(v_U)$. Now, $\ind_{W_0,W}^C(\eps^{-1} \cdot \res_{W,U}^C(\delta)) = 1$ and as $\tateco^{-1}(C)(W_0,W) = 1$ by assumption, there exists an element $a \in C(W)$ such that 
\[
\eps^{-1} \cdot \res_{W,U}^C(\delta) = \con_{h \modd W -1, W}^C(a) = \con_{h-1,W}^C(a).
\]

As the canonical morphism
\[
H/I_U = H/(I_H \cap U) \rightarrow H/I_H \times H/U
\]
is injective, it follows that $H/I_U$ is abelian and as $W \geq I_U$, we conclude that $H/W$ is also abelian. In particular, $h \varphi \equiv \varphi h \modd W$ and this implies that
\[
\con_{h\varphi-1,W}^C \circ \con_{h-1,W}^C = \con_{h-1,W}^C \circ \con_{h\varphi-1,W}^C.
\]
Now, using the fact that $\varphi \in U$ and $h \varphi \in \Sigma$, we get
\begin{align*}
\con_{h-1,W}^C(\res_{W,U}^C(\pi_U^k \cdot \delta)) &= \frac{\con_{h,W}^C \circ \res_{W,U}^C(\pi_U^k \cdot \delta)}{\res_{W,U}^C(\pi_U^k \cdot \delta)} = \frac{\res_{W,U}^C \circ \con_{h,U}^C(\pi_U^k \cdot \delta)}{\res_{W,U}^C(\pi_U^k \cdot \delta)} \\
&= \frac{\res_{W,U}^C \circ \con_{h \varphi,U}^C(\pi_U^k \cdot \delta)}{\res_{W,U}^C(\pi_U^k \cdot \delta)} =  \frac{\con_{h \varphi,W} \circ \res_{W,U}^C(\pi_U^k \cdot \delta)}{\res_{W,U}^C(\pi_U^k \cdot \delta)} \\
&= \con_{h \varphi-1,W}^C( \res_{W,U}^C(\pi_U^k \cdot \delta) ) = \con_{h \varphi-1,W}^C( \eps^{-1} \cdot \res_{W,\Sigma}^C(\pi_\Sigma^k) \cdot \res_{W,U}^C(\delta)) \\
&= \con_{h \varphi-1,W}^C( \con_{h-1,W}^C(a) ) \cdot \con_{h \varphi-1,W}^C( \res_{W,\Sigma}^C(\pi_\Sigma^k) ) \\
&= \con_{h-1,W}^C( \con_{h \varphi-1,W}^C(a)) \cdot \res_{W,\Sigma}^C \circ \con_{h \varphi-1,\Sigma}^C( \pi_\Sigma^k) \\
&= \con_{h-1,W}^C( \con_{h \varphi-1,W}^C(a))
\end{align*}
and therefore
\[
x \dopgleich \res_{W,U}^C(\pi_U^k \cdot \delta) \cdot \con_{1-h \varphi,W}^C(a) \in \ker( \con_{h-1,W}^C ) = \im(\res_{W,W_0}^C)
\]
by assumption. Let $y \in C(W_0)$ such that $x = \res_{W,W_0}^C(y)$. Then
\[
k \omega = v_W(x) = v_W \circ \res_{W_0,W}^C(y) = e_{W_0|W} \cdot v_{W_0}(y) = e_{H|U} \cdot v_{W_0}(y) = \lbrack H:U \rbrack \cdot v_{W_0}(y)
\]
and since $0 \leq k < \lbrack H:U \rbrack$, this implies $k = 0$. Hence, $\Upsilon_{(H,U)}$ is injective.
\end{proof}

\begin{cor} \label{para:fs_cft_main_classical_case}
Assume that $\fr{E} = \ro{Sp}(G)^{\tn{f}}$. Let $\fr{M} = \ro{Grp}(G)^{\tn{f}}$, let $C \in \discgmod$, let $C_* \dopgleich \ro{H}_{\fr{M}}^0(C)$ and let $v \in \ro{Val}_d^{\langle \omega \rangle,\omega}(C_*) \subs \ro{Val}_d^{\Omega,\omega}(C_*)$. Suppose that $C_*$ satisfies the class field axiom for all $(H,U) \in \esys{b}^{\tn{ur}}$ and for all $(H,U) \in \esys{b}^\ell$ with $f_{H|U} = 1$ and $\ell$ being a prime number. Then $(C_*,v) \in \ro{FND}_d^{\omega}(\fr{E})$ and $\Upsilon: \pi_{\fr{K}(G)_{\tn{ab}}^{\tn{f}}} \ra \tateco_{\fr{E}}^0(C_*)$ is a $\fr{K}(G)_{\tn{ab}}^{\tn{f}}$-class field theory.
\end{cor}

\begin{proof}
This follows immediately from \ref{para:fs_cft_main_case} and \ref{para:fs_cft_tot_ram_iso}, recalling that $C_*$ has $(H,U)$-Galois descent for all $(H,U) \in \esys{b}$ so that in particular both \ref{para:fs_cft_main_case}\ref{item:fs_cft_main_case_gal_desc} and \ref{para:fs_cft_tot_ram_iso}\ref{item:para:fs_cft_tot_ram_iso_gal_desc} are also satisfied.
\end{proof}

\begin{para}
We rewrite the definition of the Fesenko--Neukirch morphism once in the language of field extensions under the Galois correspondence to make it more explicit. Let $G = \gal(k)$ be a pro-$P$ group and let $(C,v) \in \ro{FND}_d^\omega(\fr{E})$. Let $K \in \esys{b}^\flat$, let $L \in \ro{Ext}(\fr{E},K)$ and let $\sigma \in \gal(L|K)$. Then
\[
\Upsilon_{L|K}: \gal(L|K)^{\ro{ab}} \lra C(K)/\ind_{K,L}^C C(L)
\]
is a morphism which maps the element $\sigma \modd \comm{a}(\gal(L|K))$ to $\ind_{K,\Sigma}^C( \pi_\Sigma^{P'(d_\Sigma(\sigma))} ) \modd \ind_{K,L}^C C(L)$, where $\Sigma$ is the fixed field of the restriction $\widetilde{\sigma}|_{L^{\ro{ur}}}$ of a Frobenius lift $\widetilde{\sigma} \in \gal(K)$ of $\sigma$ to the maximal unramified extension $L^{\ro{ur}}$ of $L$ and $\pi_\Sigma \in \Sigma$ is a prime element. The automorphism $\widetilde{\sigma}$ is characterized by $\widetilde{\sigma}|_L = \sigma$ and $d_K(\widetilde{\sigma}) = \omega$. We see that it is enough to lift $\sigma$ to $\gal(L^{\tn{ur}}|K)$.
\end{para}



\newpage
\section{Discrete valuation fields of higher rank} \label{chap:disc_vals}

In this chapter we will review the basic notions necessary to describe \name{Fesenko}'s approach to higher local class field theory. After \ref{sec:ram_val} the reader may already go to \ref{sect:local_cft} to see how classical local class field theory is obtained as a Fesenko--Neukirch class field theory. \\

In \ref{sec:vals} and \ref{sec:ram_val} some basic material about valuations on fields is recalled which may for the reader familiar with this material be thought of as fixing notations. In \ref{sec:higher_disc_vals} discrete valuations of higher (finite) rank are introduced and basic properties are discussed. In \ref{sec:milnor_k} the basic notions of Milnor $\ro{K}$-theory and in \ref{sec:seq_tops} the Par\v{s}in topology on a higher local field is sketched.

\subsection{Basic facts about valuations on fields} \label{sec:vals}

\begin{para}
In this section we will recall some basic facts about valuations on fields. The references for all the unexplained material presented here are \cite[chapter 2]{Neu99_Algebraic-Number_0}, \cite[chapter VI]{Bou72_Commutative-Algebra_0}, \cite[chapter II]{Ser79_Local-Fields_0} and \cite[chapter I and II]{FesVos02_Local-fields_0}.
\end{para}

\begin{ass}
Throughout this section $k$ denotes a field.
\end{ass}

\begin{defn} \label{para:valuation_def}
Let $\Gamma$ be a totally ordered abelian group. A (non-trivial) \words{$\Gamma$-valuation}{valuation} on $k$ is a map $v:k^\times \ra \Gamma$ satisfying the following conditions:
\begin{enumerate}[label=(\roman*),noitemsep,nolistsep]
\item \label{item:valuation_def_1} $v$ is a morphism, that is, $v(xy) = v(x) + v(y)$ for all $x,y \in k^\times$.
\item \label{item:valuation_def_2} $v$ satisfies the \word{ultrametric inequality}, that is, $v(x+y) \geq \ro{min} \lbrace v(x),v(y) \rbrace$ for all $x,y \in k^\times$ with $x+y \neq 0$.
\item \label{item:valuation_def_3} $v$ is non-trivial, that is, there exists $x \in k^\times$ with $v(x) \neq 0$.
\end{enumerate}
\end{defn}

\begin{para} \label{para:valuations_props}
The map $v$ is usually extended to all of $k$ as follows: the total order and the addition of $\Gamma$ are extended to $\Gamma_\infty \dopgleich \Gamma \amalg \lbrace \infty \rbrace$ by defining $\infty + \infty = \infty$, $x + \infty = \infty$ and $x < \infty$ for all $x \in \Gamma$. This makes $\Gamma_\infty$ into a totally ordered commutative monoid. Now, by defining $v(0) \dopgleich \infty$, one gets a map $v:k \ra \Gamma_\infty$ satisfying \ref{para:valuation_def}\ref{item:valuation_def_1} and \ref{para:valuation_def}\ref{item:valuation_def_2} for all $x,y \in k$. We will use this canonical extension of a valuation to the whole field from now on without explicitly mentioning this.
\end{para}

\begin{para}
The set of all $\Gamma$-valuations on $k$ is denoted by $\ro{Val}^\Gamma \! (k)$. Two valuations $v:k^\times \ra \Gamma_v$ and $w:k^\times \ra \Gamma_w$ are called \words{equivalent}{valuations!equivalent} if there exists an isomorphism of ordered groups $f:v(k^\times) \ra w(k^\times)$ such that $w = fv$ on $k^\times$. The set of all equivalence classes of valuations on $k$ is denoted by $\ro{Val}(k)$.
\end{para}

\begin{para} \label{para:val_ring_props}
If $k$ is a field and $v \in \ro{Val}^\Gamma \! (k)$, then the following properties are well-known:
\begin{enumerate}[label=(\roman*),noitemsep,nolistsep]
\item $v(1) = 0$ and $v(x) = \infty$ if and only if $x = 0$.
\item $v(x^{-1}) = -v(x)$ and $v(-x) = v(x)$ for all $x \in k^\times$.
\item The set $\calo_v \dopgleich \lbrace x \in k \mid v(x) \geq 0 \rbrace$ is a subring of $k$ which is called the \word{valuation ring} of $v$.
\item $\calo_v$ is local with maximal ideal $\fr{m}_v \dopgleich \lbrace x \in k \mid v(x) > 0 \rbrace$ and unit group $\calo_v^\times = \lbrace x \in k \mid v(x) = 0 \rbrace$. The field $\kappa_v \dopgleich \calo_v/\fr{m}_v$ is called the \word{residue field} of $v$.
\item \label{item:val_ring_props_is_val_ring} If $x \in k^\times$, then $x \in \calo_v$ or $x^{-1} \in \calo_v$.
\item If $w : k^\times \ra \Gamma_w$ is another valuation, then $\calo_v = \calo_{w}$ if and only if $v$ and $w$ are equivalent.
\item $\calo_v$, $\calo_v^\times$ and $\fr{m}_v$ just depend on the equivalence class of $v$.
\item The \word{value group} $v(k^\times)$ depends up to isomorphism of ordered groups only on the equivalence class of $v$.
\end{enumerate}
\end{para}

\begin{para} \label{para:valuation_rings}
A domain $\calo$ such that there exists a field $k$ with $\calo < k$ and such that $x \in k^\times$ implies $x \in \calo$ or $x^{-1} \in \calo$ is called a (non-trivial) \word{valuation ring}. In this case the field $k$ is already equal to the quotient field $\ro{Q}(\calo)$ of $\calo$. Any valuation ring $\calo$ is a normal B\'ezout domain and the ideals in $\calo$ are totally ordered by inclusion. In particular, the Krull dimension of $\calo$ is equal to the number of non-zero prime ideals in $\calo$. Moreover any ring $\calo'$ with $\calo \leq \calo' < k$ is again a valuation ring and the map $\fr{p} \mapsto \calo_{\fr{p}}$ is a decreasing bijection from the set of non-zero prime ideals of $\calo$ to the set of such intermediate rings. In particular, the set of such rings is totally ordered by inclusion and the number of such rings is equal to the number of non-zero prime ideals of $\calo$.

According to \ref{para:val_ring_props}\ref{item:val_ring_props_is_val_ring} the valuation ring $\calo_v$ of a valuation $v$ on a field $k$ is a valuation ring. But the converse also holds: if $\calo$ is a valuation ring and $k \dopgleich \ro{Q}(\calo)$, then $\Gamma_\calo \dopgleich k^\times/\calo^\times$ is totally ordered by $\ol{x} \leq \ol{y} \LRA yx^{-1} \in \calo$ and the quotient morphism $v: k^\times \ra \Gamma_\calo$ is a valuation with valuation ring $\calo$. Hence, there exists a canonical bijection between the valuation rings in $k$ and $\ro{Val}(k)$. If $\calo$ is a valuation ring in $k$, then a valuation $v$ on $k$ with $\calo_v = \calo$ is called an \words{admissible valuation}{valuation!admissible} for $\calo$. All equivalence-invariant notions for valuations will interchangeably be used for valuation rings.\footnote{The value group $v(k^\times)$ itself is for example not an equivalence-invariant notion. But its isomorphism class as an ordered group is an equivalence-invariant notion and so are $\calo_v$, $\calo_v^\times$ and $\fr{m}_v$.}
\end{para}

\begin{para}
A valuation ring $\calo$ in $k$ with residue field $\kappa$ is said to be of \words{equal characteristic}{valuation!equal characteristic} if $\ro{char}(k) = \ro{char}(\kappa)$. Otherwise it is said to be of \words{mixed characteristic}{valuation!mixed characteristic}. In this case necessarily $\ro{char}(k) = 0$ and $\ro{char}(\kappa) = p > 0$.
\end{para}

\begin{para}
If $v:k^\times \ra \Gamma$ is a valuation, then the sets $k_{v > \gamma} \dopgleich \lbrace x \in k \mid v(x) > \gamma \rbrace$, $\gamma \in v(k^\times)$, form a compatible filter basis on the additive group $k^+$ of $k$ and give $k^+$ the structure of a separated topological group. Moreover, with this topology, which we denote by $\ca{T}_v$, the field $k$ is a topological field and the map $v:k^\times \ra \Gamma$ is continuous with respect to the discrete topology on $\Gamma$. The topology $\ca{T}_v$ is equivalence-invariant. But two valuations $v,w$ on $k$ with $\ca{T}_v = \ca{T}_w$ are not necessarily equivalent. Such valuations are called \words{dependent}{valuation!dependent}.
\end{para}

\begin{conv}
Subgroups of $\RR$ are always considered with the natural order.
\end{conv}

\begin{para} \label{para:real_vals}
A \words{real valuation}{valuation!real} (also \words{valuation of height 1}{valuation!of height 1} or \words{valuation of rank 1}{valuation!of rank 1}) on $k$ is a valuation $v$ on $k$ whose value group $v(k^\times)$ is as an ordered group isomorphic to a subgroup of $\RR$. This notion is obviously equivalence-invariant and the subset of $\ro{Val}(k)$ consisting of the equivalence classes such that one (and then any) representative is a real valuation is denoted by $\ro{Val}^{\tn{r}}(k)$.
There exists a canonical bijection between $\ro{Val}^{\tn{r}}(k)$ and the set of non-archimedian places of $k$, that is, the equivalence classes of non-archimedian norms on $k$. The topologies defined by a real valuation and the corresponding non-archimedian norm coincide and this fact implies that two real valuations $v,w$ on $k$ are equivalent if and only if $\ca{T}_v = \ca{T}_{w}$ and this is equivalent to the existence of $s \in \RR_{>0}$ such that $w = sv$ on $k^\times$. In particular, real valuations are equivalent if and only if they are dependent.
\end{para}

\begin{para} \label{para:disc_val_def}
A \words{discrete valuation}{valuation!discrete} on $k$ is a valuation $v$ on $k$ whose value group $v(k^\times)$ is as an ordered group isomorphic to $\ZZ$. This is obviously an equivalence-invariant notion and any discrete valuation is real. If $v:k^\times \ra \Gamma$ is a discrete valuation, then as the ordered group $\ZZ$ has no non-trivial automorphisms, there exists an already unique isomorphism $\varphi$ of ordered groups from the value group $v(k^\times)$ to $\ZZ$. The valuation $v$ is called \words{normalized}{valuation!normalized} if $v(k^\times) = \ZZ$. The composition $\varphi \circ v$ is obviously a normalized discrete valuation which is equivalent to $v$ and this valuation is called the \words{normalization}{valuation!normalization} of $v$. If $\gamma \dopgleich \varphi^{-1}(1) \in \Gamma$, then any element $\pi \in k$ with $v(\pi) = \gamma$ is called a \word{uniformizer} \invword{valuation!uniformizer} of $v$. Any uniformizer is a prime element of $\calo_v$ and generates the maximal ideal $\fr{m}_v$. Any non-zero ideal of $\calo_v$ is of the form $\fr{m}_v^n$ for a unique $n \in \NN$ and the set of ideals is a neighborhood basis of $0 \in k$. The \word{higher unit groups} $U_v^{(n)} \dopgleich 1 + \fr{m}_v^n \leq \calo_v^\times$, $n \in \NN$, form a neighborhood basis of $1 \in k^\times$. The group $U_v^{(1)}$ is called the group of \word{principal units}.
\words{Discrete valuation rings}{valuation ring!discrete}, that is, valuation rings of discrete valuations, can be characterized as valuation rings of Krull dimension $1$ or as local Dedekind domains.

\end{para}

\begin{para}
If $A$ is a Dedekind domain and $k \dopgleich \ro{Q}(A)$ is its quotient field, then for any $\fr{p} \in \mspec(A)$ and $a \in A$ one defines $\ro{ord}_{\fr{p}}(a) \dopgleich \ro{sup} \lbrace n \in \NN \mid \fr{p}^n \sups (a) \rbrace$. For $a \in A \setminus \lbrace 0 \rbrace$ the natural number $\ro{ord}_{\fr{p}}(a)$ is equal to the power of $\fr{p}$ in the prime ideal decomposition of $(a)$ and using the properties of Dedekind domains one can show that the map $\ro{ord}_{\fr{p}}: k^\times \ra \ZZ$, $\frac{a}{b} \mapsto \ro{ord}_{\fr{p}}(a) - \ro{ord}_{\fr{p}}(b)$, is a normalized discrete valuation on $k$ which is called the \words{$\fr{p}$-adic valuation}{valuation!$\fr{p}$-adic}. Its valuation ring is equal to the localization $A_{\fr{p}} \subs k$ and its residue field is canonically isomorphic to $A/\fr{p}$. Moreover, the map $\mspec(A) \ra \ro{Val}^{\tn{d}}(k)$, $\fr{p} \mapsto \ro{ord}_{\fr{p}}$, is injective.
\end{para}

\begin{para}
The field of  \word{formal Laurent series} $k ((X))$ over a field $k$ is defined as the quotient field of $k \lbrack \lbrack X \rbrack \rbrack$. Any element of $k((X))$ can be written as $f = \sum_{n \geq n_0} a_nX^n$ with $n_0 \in \ZZ$ and $a_n \in k$. The \word{order valuation} on $k((X))$ is defined as $\ro{ord}( \sum_{n \in \ZZ} a_n X^n ) \dopgleich \inf \lbrace n \in \ZZ \mid a_n \neq 0 \rbrace$. It is easy to see that this is an equal characteristic normalized discrete valuation on $k((X))$ with valuation ring $k\lbrack\lbrack X \rbrack\rbrack$ and with residue field canonically isomorphic to $k$. If nothing else is mentioned, then we always consider $k((X))$ with this valuation.
\end{para}

\begin{para} \label{para:val_ext}
Let $K|k$ be a field extension and let $v \in \ro{Val}^\Gamma \! (k)$. If $w \in \ro{Val}^{\Gamma'} \! (K)$ such that $\Gamma \leq \Gamma'$ as ordered groups and $w|_k = v$, then $w$ is called an \words{extension}{valuation!extension} of $v$ and we denote this by $w|v$ or by $(K,w) | (k,v)$. In this case the diagram
\[
\xymatrix{
K^\times \ar[rr]^w & & \Gamma' \\
k^\times \ar[rr]_v \ar@{ >->}[u] & & \Gamma \ar@{ >->}[u],
}
\]
commutes, where the vertical morphisms are the canonical inclusions. Moreover, $\calo_v = \calo_w \cap k$, $\fr{m}_v = \fr{m}_w \cap k$ and there exists a canonical embedding $\kappa_v \rightarrowtail \kappa_w$. 
We denote by $\ro{Ext}_v(K)$ the set of all valuations on $K$ which extend $v$. If $(K,w) | (k,v)$ is an algebraic extension, then the residue field extension $\kappa_w|\kappa_v$ is algebraic and $w$ is a real valuation if and only if $v$ is a real valuation. In this case, the alternative characterization of the equivalence of real valuations in \ref{para:real_vals} implies that if $v$ is a real valuation and $w,w' \in \ro{Ext}_v(K)$ are equivalent, then already $w=w'$. If $(K,w)|(k,v)$ is a finite extension, then $w$ is discrete if and only if $v$ is discrete. 
\end{para}

\begin{para} \label{para:complete_val}
A valuation $v \in \ro{Val}^\Gamma \! (k)$ is called \words{complete}{valuation!complete} if the topological group $k^+$ is complete. The Cauchy-completion $k_v \dopgleich \widehat{k} \sups k$ of $k$ with respect to $\ca{T}_v$ is again a topological field and the valuation $v$ extends uniquely to a valuation $\widehat{v} \in \ro{Val}^\Gamma(\widehat{k})$ such that $\widehat{v}:k^\times \ra \Gamma$ is continuous with respect to the discrete topology on $\Gamma$. The value group of $\widehat{v}$ is that of $v$ and the topology on the Cauchy-completion $\widehat{k}$ is that of $\widehat{v}$. For $\gamma \in v(k^\times) = \widehat{v}(\widehat{k}^\times)$ the set $\widehat{k}_{\widehat{v} > \gamma}$ is equal to the closure of $k_{v>\gamma}$ in $\widehat{k}$. The valuation ring $\calo_{\widehat{v}}$ respectively the maximal ideal $\fr{m}_{\widehat{v}}$ is canonically isomorphic to the Cauchy-completion of $\calo_v$ respectively of $\fr{m}_{v}$. Finally, the residue fields of $\widehat{v}$ and $v$ are canonically isomorphic.
\end{para}

\begin{para} \label{para:complete_real_val_ext}
A valuation $v \in \ro{Val}^\Gamma \! (k)$ is called \words{henselian}{valuation!henselian} if $v$ admits a unique extension to every algebraic extension $K$ of $k$. This is equivalent to $\calo_v$ being a henselian ring. If $v$ is henselian and nothing else is mentioned, then we will always equip an algebraic extension of $k$ with the unique extension of $v$. If $v$ is henselian and $(K,w) | (k,v)$ is an algebraic extension, then by definition $(K,w)$ is also henselian. 
If $v$ is a real henselian valuation and if $(K,w) | (k,v)$ is an algebraic extension, then $\calo_w$ is equal to the integral closure of $\calo_v$ in $K$. If additionally $\lbrack K:k \rbrack = n < \infty$, then $w$ is explicitly given by $w = \frac{1}{n}v(\ro{N}_{K|k}( \cdot )): K^\times \ra \RR$. Any complete real valuation $v$ is henselian and its unique extension $w$ to an algebraic extension $K|k$ is again real by \ref{para:val_ext}. If additionally $\lbrack K:k \rbrack < \infty$, then $w$ is also complete.
\end{para}

\begin{para} \label{para:complete_disc_field}
A \words{complete discrete valuation field}{discrete valuation field!complete} is a field $k$ equipped with a complete discrete valuation. If $p$ is a prime number, then the completion $\QQ_p$ of $\QQ$ with respect to the $(p)$-adic valuation is according to \ref{para:complete_val} a mixed characteristic complete discrete valuation field with valuation ring $\ZZ_p$, maximal ideal $p \ZZ_p$ and with residue field canonically isomorphic to $\FF_p$. It is not hard to verify that the field of formal Laurent series $k((X))$ over any field $k$ is a complete discrete valuation field. As a finite extension $K$ of a complete field $k$ is again complete with respect to the unique extension of the valuation on $k$ to $K$, it follows from \ref{para:complete_real_val_ext} that the class of complete discrete valuation fields is closed under forming finite extensions.
\end{para}

\begin{para} \label{para:complete_disc_field_props}
A consequence of Hensel's lemma and the approximation theorem for discrete valuations is that on a complete discrete valuation field there exists up to equivalence no other discrete valuation. Therefore the omission of the valuation in the notation of complete discrete valuation fields will not produce ambiguities. Moreover, this implies that if $\varphi:k \ra k'$ is a field isomorphism, $v$ is a complete discrete valuation on $k$ and $v'$ is a discrete valuation on $k'$, then $v'$ is already complete and equivalent to the valuation $v \circ \varphi^{-1}$. Moreover, $\varphi$ is already a homeomorphism.

An important property of a complete discrete valuation field $(k,v)$ is the fact that after choosing a set $R \subs \calo_v$ of representatives of $\calo_v/\fr{m}_v = \kappa_v$ and a uniformizer $\pi \in \calo_v$ any element $x \in k$ admits a unique representation as a convergent series $x = \sum_{n \geq n_0} r_n\pi^n$ with $r_n \in R$ and $n_0 \in \ZZ$. Using this, one can show that for a complete discrete valuation field $(k,v)$ the set of field isomorphisms $\varphi: \kappa_v((X)) \ra k$ is in canonical one-to-one correspondence with pairs $(\lambda,\pi)$ consisting of a prime element $\pi \in \calo_v$ and a \word{coefficient field} $\lambda$ of $k$, that is, a subfield $\lambda \subs \calo_v$ such that $q|_{\lambda}:\lambda \ra \kappa_v$ is an isomorphism, where $q:\calo_v \ra \kappa_v$ is the quotient morphism. If such a pair $(\lambda,\pi)$ is given, then $\varphi(\sum_{n \in \ZZ} a_n X^n) \mapsto \sum_{n \in \ZZ} s(a_n) \pi^n$ is the corresponding isomorphism, where $s$ is the inverse of $q|_\lambda:\lambda \ra \kappa_v$. If an isomorphism $\varphi: \kappa_v((X)) \ra k$ is given, then $(\varphi(\kappa_v), \varphi(X))$ is the corresponding pair. If $(k,v)$ is an equal characteristic complete discrete valuation field, then one can prove the existence of a coefficient field and thus concludes that $k$ is as a topological field already isomorphic to $\kappa_v((X))$. If $\ro{Char}(k) = p > 0$ and $\kappa_v$ is perfect, then there exists indeed only one coefficient field formed by the \word{Teichmüller lifts}. If $\ro{Char}(k) = 0$ or $\kappa_v$ is imperfect, then there exist many coefficient fields in general.
\end{para}

\begin{para} \label{para:local_fields}
If $v \in \ro{Val}^\Gamma \! (k)$, then the topology $\ca{T}_v$ is locally compact if and only if $v$ is complete and discrete and the residue field $\kappa_v$ is finite. In this case the valuation ring $\calo_v$ is compact and the pair $(k,v)$ is called a \word{local field}. Obviously $\QQ_p$ and $\FF_{p^r}((X))$ are local fields for any prime number $p$ and $r \in \NN_{>0}$. Moreover, the class of local fields is closed under finite extensions. Any local field $k$ is as a topological field already isomorphic to either a finite extension of $\QQ_p$ or to $\FF_q((X))$ for a prime number $p$ and $q = p^f$ with $f \in \NN_{>0}$.
\end{para}

\subsection{Ramification theory} \label{sec:ram_val}

\begin{para}
In this section we will recall the basic notions concerning ramification of valuations. The references for all the unexplained material mentioned here are the same as in \ref{sec:vals}.
\end{para}

\begin{para} \label{para:ramification_index}
Let $K|k$ be a field extension, let $v$ be a valuation on $k$ and let $w \in \ro{Ext}_v(K)$. The \word{ramification index} of $w$ over $v$ is defined as $e_{w|v} \dopgleich \lbrack w(K^\times):v(k^\times) \rbrack$ and the \word{inertia index} is defined as $f_{w|v} \dopgleich \lbrack \kappa_w : \kappa_v \rbrack$. All these notions obviously equivalence-independent and both the ramification and the inertia index are transitive on extensions. Suppose from now on that $\lbrack K:k \rbrack = n < \infty$. Then $e_{w|v} \cdot f_{w|v} \leq n$ and therefore both the ramification and inertia index are finite in these cases. Equality holds if $v$ is complete or if $v$ is discrete and the extension $K|k$ is separable.
The extension $w|v$ is called \words{unramified}{valuation!unramified extension} if $\lbrack K:k \rbrack = \lbrack \kappa_w:\kappa_v \rbrack$ and the residue field extension $\kappa_w \sups \kappa_v$ is separable, otherwise it is called \words{ramified}{valuation!ramified extension}. It is called \words{totally ramified}{valuation!totally ramified} if it is ramified and $f_{w|v} = 1$. 
If $w$ is discrete, then $\fr{m}_v \calo_w = \fr{m}_w^{e(w|v)}$ and therefore $e_{w|v}$ is equal to the algebraic ramification index $e_{\fr{m}_w|\fr{m}_v}$.
\end{para}

\begin{para}
If $v$ is a henselian valuation on a field $K$, then an arbitrary algebraic extension $K|k$ is called \words{unramified}{extension!unramified}, if $K'|k$ is unramified for all intermediate extensions $K \sups K' \sups k$ with $K'|k$ finite.
\end{para}

\begin{para} \label{para:discrete_normalization}
Let $K|k$ be an extension with $\lbrack K:k \rbrack = n < \infty$. Let $v$ be a normalized discrete valuation on $k$ and let $w = \frac{1}{n} v(\ro{N}_{K|k}( \cdot )):K^\times \ra \frac{1}{n} \ZZ \subs \RR$ be the unique extension of $v$ to $K$. Then $w(K^\times) = \frac{1}{e_{K|k}} \ZZ$ and $e_{K|k} w$ is the normalization of $w$.
\end{para}

\begin{para} \label{para:decomp_inert_groups}
Let $K|k$ be a Galois extension, let $v$ be a real valuation on $k$ and let $G = \gal(K|k)$. If $w \in \ro{Ext}_v(K)$, then $w \circ \sigma \in \ro{Ext}_v(K)$ for each $\sigma \in G$ and consequently $G$ acts on the set $\ro{Ext}_v(K)$. This action is indeed transitive. From now on we fix $w \in \ro{Ext}_v(K)$. The \word{decomposition group} of $w|v$ is defined as $D_{w|v} \dopgleich \lbrace \sigma \in G \mid w \circ \sigma = w \rbrace$ and the \word{inertia group} of $w|v$ is defined as $I_{w|v} \dopgleich \lbrace \sigma \in D_{w|v} \mid \sigma x \equiv x \modd \fr{m}_w \tn{ for all } x \in \calo_w \rbrace$. Both the decomposition and the inertia group are closed subgroups of $G$. The fixed field $Z_{w|v} \dopgleich K^{D_{w|v}}$ is called the \word{decomposition field} of $w|v$ and the fixed field $T_{w|v} \dopgleich K^{I_{w|v}}$ is called the \word{inertia field} of $w|v$. The map $D_{w|v} \mybackslash G \ra \ro{Ext}_v(K)$, $D_{w|v} \sigma \mapsto w \circ \sigma$, is a bijection and therefore we have in particular $\lbrack Z_{w|v}:k \rbrack = \# \ro{Ext}_v(K)$. Of great importance is the fact that the residue field extension $\kappa_w | \kappa_v$ is normal and that the sequence
\[
1 \ra I_{w|v} \ra D_{w|v} \ra \gal(\kappa_w | \kappa_v) \ra 1
\]
is exact, where the morphism $I_{w|v} \ra D_{w|v}$ is the canonical inclusion and the morphism $D_{w|v} \ra \gal(\kappa_w | \kappa_v)$ is given by $\sigma \mapsto \ol{\sigma}$ with $\ol{\sigma}(\ol{x}) \dopgleich \sigma(x)$ for $\ol{x} \in \kappa_w$ and a representative $x \in \calo_w$. In particular, we have a canonical isomorphism $\gal(T_{w|v}|Z_{w|v}) \cong G_{w|v} / I_{w|v} \cong \gal(\kappa_w|\kappa_v)$. Moreover, if $\kappa_v$ is perfect, then $\kappa_w | \kappa_v$ is a Galois extension and consequently
\[
\lbrack D_{w|v} : I_{w|v} \rbrack = \#( D_{w|v}/I_{w|v}) = \# \gal(\kappa_w|\kappa_v) = \lbrack \kappa_w:\kappa_v \rbrack = f_{w|v}.
\]
\end{para}

\begin{para} \label{para:ram_theory_of_real_henselian}
Let $v$ be a real henselian valuation on a field $k$. For a Galois extension $K | k$ we simply write $e_{K|k} \dopgleich e_{w|v}$, $f_{K|k} \dopgleich f_{w|v}$, $D_{K|k} \dopgleich D_{w|v}$, $I_{K|k} \dopgleich I_{w|v}$, $Z_{K|k} \dopgleich Z_{w|v}$ and $T_{K|k} \dopgleich T_{w|v}$, where $w$ is the unique extension of $v$ to $K$. Moreover, we write $e_k \dopgleich e_{k^{\tn{s}}|k}$, $f_k \dopgleich f_{k^{\tn{s}}|k}$, $D_k \dopgleich D_{k^{\tn{s}}|k}$, $I_k \dopgleich I_{k^{\tn{s}}|k}$, $Z_k \dopgleich Z_{k^{\tn{s}}|k}$ and $T_k \dopgleich T_{k^{\tn{s}}|k}$. For any Galois extension $K | k$ we have $D_{K|k} = \gal(K|k)$ and thus $Z_{K|k} = k$ because $\# \ro{Ext}_v(K) = 1$. If $K$ is any intermediate extension of $k^{\tn{s}} | k$, then obviously 
\[
I_K = I_{K^{\tn{s}}|K} = I_{k^{\tn{s}}|K} = I_{k^{\tn{s}}|k} \cap \gal(k^{\tn{s}}|K) = I_k \cap \gal(K).
\]

We say that $K|k$ is \noword{unramified} if $w|v$ is unramified, where $w$ is the unique extension of $w$ to $K$. The compositum of all intermediate extensions $K$ of $k^{\ro{s}} | k$ such that $K | k$ is a finite unramified extension is called the \word{maximal unramified extension} of $(k,v)$ and is denoted by $k^{\ro{ur}}$. The extensions of $k$ contained in $k^{\tn{ur}}$ are precisely the separable unramified extensions of $k$. Moreover, $k^{\tn{ur}}|k$ is a Galois extension and the residue field of $k^{\tn{ur}}$ is equal to $\kappa^{\ro{s}}_v$. 

If $K$ is a finite unramified Galois extension of $k$ and $w$ is the unique extension of $v$ to $K$, then $\kappa_w | \kappa_v$ is a Galois extension and as $\gal(K|k) = \lbrack K:k \rbrack = \lbrack \kappa_w:\kappa_v \rbrack = \gal(\kappa_w|\kappa_v)$, the epimorphism $\gal(K|k) = D_{K|k} \ra \gal(\kappa_w|\kappa_v)$ from the exact sequence in \ref{para:decomp_inert_groups} is an isomorphism. In particular, $I_{K|k} = 1$. Now, if $\sigma \in I_{k^{\tn{ur}}|k} \subs D_{k^{\tn{ur}}|k} = \gal(k^{\tn{ur}}|k)$, then obviously $\sigma|_K \in I_{K|k} = 1$ and as $k^{\tn{ur}}$ is composed of such extensions $K|k$, it follows that $I_{k^{\tn{ur}}|k} = 1$. In particular, the epimorphism $D_{k^{\tn{ur}}|k} = \gal(k^{\tn{ur}}|k) \ra \gal(\kappa_v^{\tn{s}}|\kappa_v)$ from the sequence in \ref{para:decomp_inert_groups} is an isomorphism. We denote this isomorphism by $d_k'$ and we denote by $d_k$ the composition of the quotient morphism $\gal(k) = \gal(k^{\tn{s}}|k) \ra \gal(k^{\tn{s}}|k)/\gal(k^{\tn{s}}|k^{\tn{ur}}) \cong \gal(k^{\tn{ur}}|k)$ with $d_k'$. If $v^{\tn{s}}$ denotes the unique extension of $v$ to $k^{\tn{s}}$, then $\kappa_{v^{\tn{s}}} | \kappa_v$ is an algebraic extension containing $\kappa_v^{\tn{s}}$ and as the diagram
\[
\xymatrix{
& &  D_{k^{\tn{ur}}|k} \ar[r]^-\cong & \gal(\kappa_v^{\tn{s}}|\kappa_v) \\
1 \ar[r] & I_{k^{\tn{s}}|k} \ar[r] & D_{k^{\tn{s}}|k} \ar[r] \ar@{->>}[u]^{(-)|_{k^{\tn{ur}}}} & \gal(\kappa_{v^s}|\kappa_v) \ar[u]_{(-)|_{\kappa_v^{\tn{s}}}}^\cong
}
\]
commutes, where we used the common fact that the restriction on the right is an isomorphism, it follows that $I_k$ is the kernel of the restriction $D_k \ra D_{k^{\tn{ur}}|k}$ and consequently $\gal(k^{\tn{s}}|k^{\tn{ur}}) = I_k$. In particular, $k^{\tn{ur}} = T_k$ and the kernel of $d_k$ is equal to $I_k$. Moreover, a separable extension $K | k$ is unramified if and only if $\gal(K) \geq I_k$.
\end{para}

\begin{para} \label{para:ram_theory_of_local}
Let $(k,v)$ be a local field. Let $v^{\tn{s}}$ be the unique extension of $v$ to $k^{\tn{s}}$. As $k^{\tn{s}} \sups k^{\tn{ur}}$, we have $\kappa_{v^{\tn{s}}} \sups \kappa_v^{\tn{s}} \sups \kappa_v$. As $\kappa_{v^{\tn{s}}} | \kappa_v$ is algebraic and as $\kappa_v$ is finite, we thus have $\kappa_{v^{\tn{s}}} = \kappa_v^{\tn{s}}$.

Let $q \dopgleich \# \kappa_v$. Then $\gal(\kappa_v) \cong \widehat{\ZZ}$ is procyclic with canonical topological generator being the Frobenius automorphism $\ro{Fr}_q: \kappa_v^{\ro{s}} \ra \kappa_v^{\ro{s}}$, $x \mapsto x^q$. It follows that $\varphi_k \dopgleich (d_k')^{-1}(\ro{Fr}_q)$ is a topological generator of $\gal(k^{\tn{ur}}|k)$ which is called the \word{absolute Frobenius automorphism} of $k$. By definition of $d_k'$ this automorphism is uniquely determined by $\varphi_k(x) \equiv \ro{Fr}_q(\ol{x}) \modd \fr{m}_{v^{\tn{ur}}}$ for all $x \in \calo_{v^{\tn{ur}}}$, where $v^{\tn{ur}}$ is the unique extension of $v$ to $k^{\tn{ur}}$. As $\widehat{\ZZ} \cong \gal(\kappa_v) \cong \gal(k^{\tn{ur}}|k)$ has for each supernatural number $n$ a unique closed subgroup of index $n$, it follows that $k$ has a unique separable unramified extension of degree $n$.

Let $K | k$ be a separable extension and let $w$ be the unique extension of $v$ to $K$. Since $w$ is henselian, an application of the theory in \ref{para:ram_theory_of_real_henselian} to $(K,w)$ yields an epimorphism $d_K: \gal(K) \ra \gal(\kappa_w)$. As 
\[
\gal(k^{\tn{s}}|k^{\tn{ur}}) = I_k \geq I_k \cap \gal(k^{\tn{s}}|K) = I_K = \gal(K^{\tn{s}}|K^{\tn{ur}}) = \gal(k^{\tn{s}}|K^{\tn{ur}}),
\]
we conclude that $k^{\tn{ur}} \subs K^{\tn{ur}}$. Now it is not hard to see that the diagram
\[
\xymatrix{
\gal(K) \ar@{->>}[r] \ar@{=}[d] \ar@/^2pc/[rrrr]^{d_K} & \gal(k^{\tn{s}}|K)/\gal(k^{\tn{s}}|K^{\tn{ur}}) \ar@{ >->}[d] \ar[rr]^-{(-)|_{K^{\tn{ur}}}} & & \gal(K^{\tn{ur}}|K) \ar[r] \ar[d]^{(-)|_{k^{\tn{ur}}}} & \gal(\kappa_w) \ar@{ >->}[d] \\
\gal(K) \ar@{->>}[r]  \ar@/_2pc/[rrrr]_{d_k} & \gal(k^{\tn{s}}|k)/\gal(k^{\tn{s}}|k^{\tn{ur}}) \ar[rr]_-{(-)|_{k^{\tn{ur}}}} & & \gal(k^{\tn{ur}}|k) \ar[r] & \gal(\kappa_v) 
}
\]
commutes, that is, $d_k|_{\gal(K)} = d_K$. The restriction of $d_k$ to $\gal(K) \subs \gal(k)$ has kernel $I_k \cap \gal(K) = I_K$ and consequently
\[
\lbrack \kappa_v^{\tn{s}}:\kappa_w \rbrack = f_{k^{\tn{s}}|K} = \lbrack D_{k^{\tn{s}}|K}:I_{k^{\tn{s}}|K} \rbrack = \lbrack \gal(K): I_K \rbrack = \# d_k(\gal(K)) = \lbrack d_k(\gal(K)):1 \rbrack
\]
as supernatural numbers. Since $\gal(\kappa_v) \cong \widehat{\ZZ}$ has a unique closed subgroup of order $\lbrack \kappa_v^{\tn{s}}:\kappa_w \rbrack$ and as $\lbrack \kappa_v^{\tn{s}}:\kappa_w \rbrack = \lbrack \gal(\kappa_w):1 \rbrack$, it follows that $d_k(\gal(K)) = \gal(\kappa_w)$. Hence, if $L | k$ is a further separable extension with $L \sups K$ and if $u$ is the unique extension of $v$ to $L$, then
\[
\lbrack d_k(\gal(K)):d_k(\gal(L)) \rbrack = \lbrack \gal(\kappa_w):\gal(\kappa_{u}) \rbrack = \lbrack \kappa_{u}:\kappa_w \rbrack = f_{L|K}.
\]

If we assume that $f_{K|k} < \infty$ then as $\gal(\kappa_v) \cong \widehat{\ZZ}$ and as $\lbrack \gal(\kappa_v):\gal(\kappa_w) \rbrack = \lbrack \kappa_w:\kappa_v \rbrack = f_{K|k}$, it follows that $f_{K|k}\gal(\kappa_v) = \gal(\kappa_w)$ and consequently $\frac{1}{f_{K|k}} \gal(\kappa_w) = \gal(\kappa_v)$. Hence, the morphism
\[
d_{K|k} \dopgleich \frac{1}{f_{K|k}} (d_k)|_{\gal(K)}: \gal(K) \ra \gal(\kappa_v)
\]
is surjective with kernel $I_K$. We denote the induced isomorphism $\gal(K^{\tn{ur}}|K) = \gal(K)/I_K \ra \gal(\kappa_v)$ by $d_{K|k}'$. Obviously, $d_{K|k}' = \frac{1}{f_{K|k}} d_K'$ and therefore
\[
(d_{K|k}')^{-1}( \ro{Fr}_q ) = ( \frac{1}{f_{K|k}} d_K')^{-1}(\ro{Fr}_q) = (d_K')^{-1} \circ (f_{K|k} \ro{Fr}_q ) = (d_K')^{-1}( \ro{Fr}_{q^{f_{K|k}}}) = \varphi_K.
\]
\end{para}

\subsection{Discrete valuation fields of higher rank} \label{sec:higher_disc_vals}

\begin{para}
In this section we will discuss some basics about discrete valuations of higher rank and discuss the notion of higher local fields. Although any reference on valuations discussing the height of valuations might already contain all results of this chapter, the author did not find a reference discussing the material in the way the author thought about it. In particular, the author could (except for some parts in \cite[\S1]{MR1363290} which motivated this section) not find a reference explaining the passage between a discrete valuation of higher rank and a sequence of discrete valuations of rank 1 and vice versa. Therefore all this is discussed in detail with the remark that this section may be obvious for a reader having a broader knowledge than the author. 
\end{para}

\begin{prop} \label{para:rlo}
Let $(\Gamma_i,\leq_i)$, $i=1,\ldots,n$, be totally ordered abelian groups. The \word{reverse lexicographical order} (RLO for short) on the direct sum $\Gamma \dopgleich \Gamma_1 \oplus \ldots \oplus \Gamma_n$ is for $a = (a_1,\ldots,a_n), b = (b_1,\ldots,b_n) \in \Gamma$ defined as follows:
\[
a \leq b \lLRA a = b \tn{ or there exists } 1 \leq m \leq n \tn{ with } a_m <_m b_m \tn{ and } a_{m+1} = b_{m+1}, \ldots, a_n = b_n.
\]
The following holds:
\begin{enumerate}[label=(\roman*),noitemsep,nolistsep]
\item $\Gamma$ is a totally ordered abelian group with respect to the RLO.
\item \label{item:rlo_inductive} Let $\Gamma' \dopgleich \Gamma_1 \oplus \ldots \oplus \Gamma_{n-1}$ be equipped with the RLO. Then the RLO on $\Gamma$ is the RLO on the direct sum $\Gamma' \oplus \Gamma_n$.
\item \label{item:rlo_projection} For $1 \leq r \leq n$ let $\Gamma^{(r)} \dopgleich \Gamma_{n-r+1} \oplus \ldots \oplus \Gamma_n$ be equipped with the RLO. Then the map
\[
\begin{array}{rcl}
\Gamma & \lra & \Gamma^{(r)} \\
(a_1,\ldots,a_n) & \longmapsto & (a_{n-r+1},\ldots,a_n)
\end{array}
\]
is a morphism of ordered groups.

\item \label{item:rlo_iso} Let $\Gamma^0 \dopgleich \lbrace (a_1,\ldots,a_{n-1},0) \in \Gamma \rbrace < \Gamma$. Then the map 
\[
\begin{array}{rcl}
\Gamma^0 & \lra & \Gamma' \\
(a_1,\ldots,a_{n-1},0) & \longmapsto & (a_1,\ldots,a_{n-1})
\end{array}
\]
is an isomorphism of ordered groups.
\end{enumerate}
\end{prop}

\begin{proof}
All assertions are straightforward.
\end{proof}

\begin{conv}
For $n \in \NN_{>0}$ the abelian group $\ZZ^n = \ZZ \oplus \ldots \oplus \ZZ$ is always considered with the reverse lexicographical order.
\end{conv}

\begin{defn}
Let $n \in \NN_{>0}$. A \words{discrete valuation of rank $n$}{discrete valuation!of rank $n$} on a field $k$ is a surjective valuation $v:k^\times \ra \ZZ^n$. The pair $(k,v)$ is then called a \noword{discrete valuation field of rank $n$}. A \words{discrete valuation ring of rank $n$}{discrete valuation ring!of rank $n$} is a valuation ring $\calo$ whose value group is as an ordered group isomorphic to $\ZZ^n$, or equivalently, $\calo$ is the valuation ring of a discrete valuation of rank $n$. 
\end{defn}

\begin{para}
The reason why we did not (compared to the definition of a discrete valuation in \ref{para:disc_val_def}) define a discrete valuation of rank $n$ to be a valuation with value group only being isomorphic to $\ZZ^n$ but already being equal to $\ZZ^n$ is that the ordered group $\ZZ^n$ has non-trivial automorphisms for $n > 1$ and therefore there does not exist a unique normalization. Our definition already forces a choice of an isomorphism from the value group to $\ZZ^n$.
\end{para}

\begin{para} \label{para:disc_val_of_rank_n_observation}
As the height (confer \cite[chapter VI, \S4.4]{Bou72_Commutative-Algebra_0}) of the totally ordered group $\ZZ^n$ is equal to $n$, the number of non-zero prime ideals in a discrete valuation ring $\calo$ of rank $n$ is equal to $n$ and therefore also the Krull dimension of $\calo$ is equal to $n$. Hence, according to \ref{para:valuation_rings} the intermediate rings between $\calo$ and $k$ are given by the chain
\[
\calo \gleichdop \calo^{(n)} < \calo^{(n-1)} < \ldots < \calo^{(2)} < \calo^{(1)} < k, 
\]
where $\calo^{(r)}$ is the localization of $\calo$ at its prime ideal $\fr{m}^{(r)}$ of height $r$ for all $1 \leq r \leq n$. The next proposition provides explicit admissible valuations for the rings $\calo^{(r)}$ in terms of $v$ if $\calo = \calo_v$ for a discrete valuation $v$ of rank $n$. 
\end{para}

\begin{prop} \label{para:disc_val_rank_n_lower}
Let $v = (v^1,\ldots,v^n):k^\times \ra \ZZ^n$ be a discrete valuation of rank $n$ on a field $k$. The following holds:
\begin{enumerate}[label=(\roman*),noitemsep,nolistsep]
\item For any $1 \leq r \leq n$ the map $v^{(r)} \dopgleich (v^{n-r+1},\ldots,v^n) \ra \ZZ^r$ is a discrete valuation of rank $r$ on $k$.
\item $\calo_{v^{(r)}}$ is equal to the localization of $\calo_v = \calo_{v^{(n)}}$ at the prime ideal of height $r$ for all $1 \leq r \leq n$, that is, $\calo_{v^{(r)}} = \calo_v^{(r)}$ in the notation of \ref{para:disc_val_of_rank_n_observation}.
\end{enumerate}
\end{prop}

\begin{proof} \hfill

\begin{asparaenum}[(i)]
\item The map $v^{(r)}$ is obviously multiplicative. Let $x,y \in k^\times$ with $x+y \neq 0$.  Then $v(x+y) \geq \ro{min} \lbrace v(x),v(y) \rbrace$. Without loss of generality we can assume that $v(x) \leq v(y)$. An application of \ref{para:rlo}\ref{item:rlo_projection} shows that $v^{(r)}(x+y) \geq v^{(r)}(x)$ and $v^{(r)}(x) \leq v^{(r)}(y)$. Hence, $v^{(r)}(x+y) \geq \ro{min} \lbrace v^{(r)}(x), v^{(r)}(y) \rbrace$. Since $v$ is surjective, the map $v^{(r)}$ is also surjective and we conclude that $v^{(r)}$ is a discrete valuation of rank $r$.

\item If $v^{(n)}(x) \geq 0$, then also $v^{(r)}(x) \geq 0$ for all $1 \leq r \leq n$ by \ref{para:rlo}\ref{item:rlo_projection}, that is, $\calo_{v^{(n)}} \leq \calo_{v^{(r)}}$. Hence, $\calo_{v^{(r)}}$ is an intermediate ring of $\calo_{v^{(n)}} = \calo_v$ and $k$ of Krull dimension $r$ and therefore it already has to be equal to $\calo_v^{(r)}$ by \ref{para:disc_val_of_rank_n_observation}.
\end{asparaenum}  \vspace{-\baselineskip}
\end{proof}

\begin{cor}
If $\calo$ is a discrete valuation ring of rank $n$ in a field $k$, then the intermediate ring $\calo^{(r)}$ as in \ref{para:disc_val_of_rank_n_observation} is a discrete valuation ring of rank $r$.
\end{cor}

\begin{proof}
This follows immediately from \ref{para:disc_val_rank_n_lower}.
\end{proof}

\begin{para}
The above shows that a discrete valuation of rank $n$ automatically induces discrete valuations of ranks $r \leq n$. In the next paragraphs we will introduce the concepts of pushforward and pullback of valuation rings which allow to inductively transform a discrete valuation of rank $n$ into a family of discrete valuations of rank $1$ and vice versa. We note that pushforward and pullback are not standard terminology.
\end{para}

\begin{prop} \label{para:finer_valuation}
For two valuation rings $\calo_v,\calo_w$ in a field $k$ the following are equivalent:
\begin{enumerate}[label=(\roman*),noitemsep,nolistsep]
\item \label{item:finer_valuation_1} $\calo_{v} < \calo_w$.
\item \label{item:finer_valuation_2} $\calo_{v}^\times < \calo_w^\times$.
\item \label{item:finer_valuation_3} $\fr{m}_v > \fr{m}_w$.
\item \label{item:finer_valuation_4} $q(\calo_v) < \kappa_w$ is a valuation ring with maximal ideal $q(\fr{m}_v)$, where $q:\calo_w \ra \calo_w/\fr{m}_w = \kappa_w$ is the quotient morphism
\end{enumerate}

If these conditions are satisfied, then $\calo_v$ is said to be \words{finer}{valuation!finer} than $\calo_w$. The valuation ring $q(\calo_v) = \calo_v/\fr{m}_w$ is called the \word{pushforward} of the pair $(\calo_v, \calo_w)$ and is denoted by $(\calo_v,\calo_w)_*$. The above equivalences also hold with $<$ replaced by $\leq$. 
\end{prop}

\begin{proof} \hfill

\ref{item:finer_valuation_1} $\RA$ \ref{item:finer_valuation_2}: Obviously, $\calo_v^\times \leq \calo_w^\times$. Let $x \in \calo_w \setminus \calo_v$. If $x^{-1} \notin \calo_w$, then as $\calo_v \leq \calo_w$, we would also have $x^{-1} \notin \calo_v$ but this is a contradiction since $\calo_v$ is a valuation ring. Hence, $x \in \calo_w^\times$. Since $x \notin \calo_v$ and since $\calo_v$ is a valuation ring, we have $x^{-1} \in \calo_v$ and therefore $x^{-1} \in \calo_w^\times \setminus \calo_v^\times$.
 
\ref{item:finer_valuation_2} $\RA$ \ref{item:finer_valuation_3}: We have $\calo_v \setminus \fr{m}_v = \calo_v^\times < \calo_w^\times = \calo_w \setminus \fr{m}_w$. Hence, if $x \in \fr{m}_w$, then $x \notin \calo_w \setminus \fr{m}_w$ and therefore $x \notin \calo_v \setminus \fr{m}_v$, that is, $x \in \fr{m}_v$. Let $x \in \calo_w^\times \setminus \calo_v^\times$. Then $x \notin \fr{m}_w$ and $x^{-1} \notin \fr{m}_w$. If $x \in \calo_v$, then as $x \notin \calo_v^\times$, we have $x \in \fr{m}_v$ and consequently $\fr{m}_v > \fr{m}_w$. Otherwise, we have $x^{-1} \in \calo_v$ because $\calo_v$ is a valuation ring. Hence, $x^{-1} \in \calo_v \setminus \calo_v^\times = \fr{m}_v$ and therefore $x^{-1} \in \fr{m}_v \setminus \fr{m}_w$.

\ref{item:finer_valuation_3} $\RA$ \ref{item:finer_valuation_1}: If $x \notin \calo_w$, then $x^{-1} \in \calo_w \setminus \calo_w^\times = \fr{m}_w < \fr{m}_w$ since $\calo_w$ is a valuation ring and this implies $x \notin \calo_v$. Hence, $\calo_v \leq \calo_w$. Let $x \in \fr{m}_v \setminus \fr{m}_w$. Then $x^{-1} \notin \calo_v$ and as $\calo_v \leq \calo_w$, we have $x \in \calo_w \setminus \fr{m}_w = \calo_w^\times$. Hence, $x^{-1} \in \calo_w \setminus \calo_v$.

\ref{item:finer_valuation_1} $\RA$ \ref{item:finer_valuation_4}: Since $\fr{m}_v > \fr{m}_w$ by \ref{item:finer_valuation_1} $\LRA$ \ref{item:finer_valuation_3}, we have $0 < \calo_v/\fr{m}_w < \calo_w/\fr{m}_w$ and therefore $q(\calo_v) = \calo_v/\fr{m}_w$ is an integral domain and a proper subset of $\kappa_w$. If $x \in \calo_w$ such that $q(x) \notin q(\calo_v)$, then $x \notin \calo_v$ and therefore $x^{-1} \in \calo_v < \calo_w$ so that $q(x)^{-1} = q(x^{-1}) \in q(\calo_v)$. Hence, $q(\calo_v)$ is a valuation ring. It is easy to see that $q(\fr{m}_v)$ is the maximal ideal of $q(\calo_v)$.

\ref{item:finer_valuation_4} $\RA$ \ref{item:finer_valuation_1}: Since $q(\calo_v) < \kappa_w = \calo_w/\fr{m}_w$, it follows immediately that $\calo_v < \calo_w$.
\end{proof}

\begin{prop}
Let $\calo_w$ be a valuation ring in a field $k$, let $\calo_v$ be a valuation ring in $\kappa_w$ and let $q:\calo_w \ra \calo_w/\fr{m}_w = \kappa_w$ be the quotient morphism. Then $q^{-1}(\calo_v)$ is a valuation ring in $k$ with $q^{-1}(\calo_v) < \calo_w$ and maximal ideal $q^{-1}(\fr{m}_v)$. This valuation ring is called the \word{pullback} of the pair $(\calo_v,\calo_w)$ and is denoted by $(\calo_v,\calo_w)^*$.
\end{prop}

\begin{proof}
The ring $q^{-1}(\calo_v)$ is a domain since $\fr{m}_w \leq q^{-1}(\calo_v)$.  Obviously $q^{-1}(\calo_v) \leq \calo_w$ and if $q^{-1}(\calo_v)$ would be equal to $\calo_w$, then $\calo_v = q(q^{-1}(\calo_v)) = q(\calo_w) = \kappa_w$ and this is a contradiction since $\calo_v$ is a valuation ring. Hence, $q^{-1}(\calo_v) < \calo_w$. Let $x \in k \setminus q^{-1}(\calo_v)$. If $x \in \calo_w$, then $x \notin \fr{m}_w$ because $\fr{m}_w \leq q^{-1}(\calo_v)$ and therefore $x \in \calo_w^\times$ with $q(x) \in \kappa_w^\times \setminus \calo_v$. As $\calo_v$ is a valuation ring, we get $q(x^{-1}) = q(x)^{-1} \in \calo_v$ and this implies that $x^{-1} \in q^{-1}(\calo_v)$. On the other hand, if $x \notin \calo_w$, then $x^{-1} \in \fr{m}_w \leq q^{-1}(\calo_v)$. This shows that $q^{-1}(\calo_v)$ is a valuation ring. It is easy to see that $q^{-1}(\fr{m}_v)$ is the maximal ideal of $q^{-1}(\calo_v)$.
\end{proof}

\begin{para}
One can think of pushforward respectively pullback as an operation $(-,\calo_w)_*$ respectively $(-,\calo_w)^*$ with a fixed valuation ring $\calo_w$ in a field $k$ that transforms in the case of pushforwards a valuation ring $\calo_v < \calo_w$ into a certain valuation ring in the residue field $\kappa_w$ and transforms in the case of pullbacks a valuation ring in $\kappa_w$ into a certain valuation ring in $k$. The next proposition shows that these two constructions are inverse to each other.
\end{para}

\begin{prop} \label{para:pushforward_pullback_inverse}
Let $\calo_w$ be a valuation ring in a field $k$. The following holds:
\begin{enumerate}[label=(\roman*),noitemsep,nolistsep]
\item If $\calo_v$ is a valuation ring in $k$ with $\calo_v < \calo_w$, then $\left( \left(\calo_v,\calo_w\right)_*, \calo_w \right)^* = \calo_v$.
\item If $\calo_v$ is a valuation ring in $\kappa_w$, then $\left( \left(\calo_v,\calo_w\right)^*, \calo_w \right)_* = \calo_v$.
\end{enumerate}
\end{prop}

\begin{proof}
Let $q:\calo_w \ra \calo_w/\fr{m}_w = \kappa_w$ be the quotient morphism. In the first case we have  
\[
\left( \left(\calo_v,\calo_w\right)_*, \calo_w \right)^* = \left( q(\calo_v),\calo_w \right)^* = q^{-1}( q(\calo_v)) = \calo_v
\]
because $\fr{m}_w < \fr{m}_v < \calo_v$. In the second case we have
\[
\left( \left(\calo_v,\calo_w\right)^*, \calo_w \right)_* = \left( q^{-1}(\calo_v),\calo_w \right)_* = q(q^{-1}(\calo_v)) = \calo_v.
\]
Note that the pushforward is defined as $\left(\calo_v,\calo_w\right)^* < \calo_w$. 
\end{proof}

\begin{prop} \label{para:explicit_quotient_valuation}
Let $\calo_v < \calo_w$ be valuation rings in a field $k$ and let $v:k^\times \ra \Gamma$ be an admissible valuation for $\calo_v$. For an element $\ol{x} \in \kappa_w^\times$ let $(v,\calo_w)_*(\ol{x}) \dopgleich v(x)$, where $x \in \calo_w \setminus \fr{m}_w = \calo_w^\times$ is a representative of $\ol{x}$. Then $(v,\calo_w)_* : \kappa_w^\times \ra \Gamma$ is an admissible valuation for $(\calo_v,\calo_w)_*$, that is, its valuation ring is equal to $\calo_v/\fr{m}_w$.
\end{prop}

\begin{proof}
To simplify notations, let $u \dopgleich (v,\calo_w)_*$. First, we have to verify that $u$ is well-defined. Let $x,y \in \calo_w^\times$ such that $x \equiv y \modd \fr{m}_w$. Then there exists $a \in \fr{m}_w$ such that $x = a+y$ and using the fact that $\fr{m}_w < \fr{m}_v$ we get
\[
xy^{-1} = ay^{-1} + 1 \in 1 + \fr{m}_w \leq 1 + \fr{m}_v \leq \calo_v^\times.
\]
Hence, $v(xy^{-1}) = 0$ and therefore $v(x) = v(y)$. This shows that $u$ is well-defined. To see that $u$ is non-trivial, note that as $\calo_v < \calo_w$, there exists $x \in \fr{m}_v \setminus \fr{m}_w$. Then $x \in \calo_w^\times$ and $u(\ol{x}) = v(x) > 0$. It is obvious that $u$ satisfies the remaining properties and thus is a valuation on $\kappa_w$. 

Let $0 \neq \ol{x} \in \calo_{u}$ and let $x \in \calo_w^\times$ be a representative. Then $0 \leq u(\ol{x}) = v(x)$ and therefore $x \in \calo_v$, that is, $\ol{x} \in \calo_v/\fr{m}_w$. Conversely, if $x \in \calo_v$ with $\ol{x} \neq 0$, then $x \in \calo_w^\times$ and consequently $u(\ol{x}) = v(x) \geq 0$, that is, $\ol{x} \in \calo_{u}$.
\end{proof}

\begin{prop} \label{para:disc_val_rank_n_push}
Let $v = (v^1,\ldots,v^n):k^\times \ra \ZZ^n$ be a discrete valuation of rank $n$ on a field $k$. Then the map 
\[
\begin{array}{rcl}
\widetilde{v}: \kappa_{v^n}^\times = \kappa_{v^{(1)}}^\times & \lra & \ZZ^{n-1} \\
\ol{x} & \longmapsto & (v^1(x),\ldots,v^{n-1}(x)),
\end{array}
\]
where $x \in \calo_{v^{(1)}}^\times = \calo_{v^n}^\times$ is a representative of $\ol{x}$, is a discrete valuation of rank $n-1$ with 
\[
\calo_{\widetilde{v}^{(r)}} = (\calo_{v^{(r+1)}}, \calo_{v^{(1)}})_*
\]
for all $1 \leq r \leq n-1$.
\end{prop}

\begin{proof} 
As $\calo_v = \calo_{v^{(n)}} < \calo_{v^{(1)}} = \calo_{v^n}$, the map 
\[
\begin{array}{rcl}
u \dopgleich (v,\calo_{v^{(1)}})_*: \kappa_{v^{(1)}}^\times & \lra & \ZZ^n \\
\ol{x} & \longmapsto & (v^1(x),\ldots,v^n(x)),
\end{array}
\]
where $x \in \calo_{v^{(1)}}^\times = \calo_{v^n}^\times$ is a representative, is a well-defined valuation by \ref{para:explicit_quotient_valuation}. But as $v^n(x) = 0$ and as $v^i(\calo_{v^n}^\times) = \ZZ$ for $i < n$, the value group of $u$ is equal to $(\ZZ^n)^0$ in the notation of \ref{para:rlo}\ref{item:rlo_iso} and so the composition of $u$ with the canonical isomorphism $(\ZZ^n)^0 \cong \ZZ^{n-1}$ of \ref{para:rlo}\ref{item:rlo_iso} shows that $\widetilde{v}$ is a discrete valuation of rank $n-1$ on $\kappa_{v^{(1)}} = \kappa_{v^n}$.

Let $q:\calo_{v^{(1)}} \ra \calo_{v^{(1)}}/\fr{m}_{v^{(1)}} = \kappa_{v^{(1)}}$ be the quotient morphism. Then $(\calo_{v^{(r+1)}},\calo_{v^{(1)}})_* = q(\calo_{v^{(r+1)}})$ and so we have to verify that $q(\calo_{v^{(r+1)}}) = \calo_{\widetilde{v}^{(r)}}$ for all $1 \leq r \leq n-1$. Let $0 \neq \ol{x} \in \calo_{\widetilde{v}^{(r)}}$ and let $x \in \calo_{v^n}^\times$ be a representative. Then $0 \leq \widetilde{v}^{(r)}(\ol{x}) = (v^{n-1-r+1}(x),\ldots,v^{n-1}(x))$ and now it follows from \ref{para:rlo}\ref{item:rlo_iso} that also 
\[
0 \leq (v^{n-r}(x),\ldots,v^{n-1}(x), 0) = (v^{n-r}(x),\ldots,v^{n-1}(x), v^n(x)) = v^{(r+1)}(x),
\]
that is, $x \in \calo_{v^{(r+1)}}$ and therefore $\ol{x} \in q(\calo_{v^{(r+1)}})$. The converse inclusion follows similarly.
\end{proof}

\begin{cor} \label{para:rank_n_pushforward_chain}
Let $\calo$ be a discrete valuation ring of rank $n$ in a field $k$ and let 
\[
\calo \gleichdop \calo^{(n)} < \calo^{(n-1)} < \ldots < \calo^{(2)} < \calo^{(1)} < k
\]
be the chain of intermediate rings between $\calo$ and $k$. Then $\calo_{n-1}^{(r)} \dopgleich (\calo^{(r+1)},\calo^{(1)})_*$ is a discrete valuation ring of rank $r$ in the residue field $k_{n-1}$ of $\calo^{(1)}$ for all $1 \leq r \leq n-1$ and 
\[
\calo_{n-1}^{(n-1)} < \calo_{n-1}^{(n-2)} < \ldots < \calo_{n-1}^{(2)} < \calo_{n-1}^{(1)} < k_{n-1}
\]
is the chain of intermediate rings between $\calo_{n-1}^{(n-1)}$ and $k_{n-1}$.
\end{cor}

\begin{proof}
This follows immediately from \ref{para:disc_val_rank_n_push}.
\end{proof}

\begin{para} \label{para:rank_n_pushforward_diagram}
Using the concept of pushforwards inductively, we can extract the whole valuation theoretic structure of lower rank out of a discrete valuation ring $\calo$ of rank $n$ in a field $k$. For this, we inductively define the following data for all $1 \leq i \leq n$:
\begin{compactitem}
\item $\calo \gleichdop \calo_n^{(n)} < \calo_n^{(n-1)} < \ldots < \calo_n^{(2)} < \calo_n^{(1)} < k_n \dopgleich k$ is the chain of intermediate rings between $\calo$ and $k$, and $k_{n-1}$ is defined as the residue field of $\calo_n^{(1)}$.
\item $\calo_i^{(r)} \dopgleich (\calo_{{i+1}}^{(r+1)}, \calo_{i+1}^{(1)})_*$ and $q_{i+1}^{(r+1)}: \calo_{i+1}^{(r+1)} \ra \calo_i^{(r)}$ is induced by the quotient morphism $q_{i+1}^{(1)}: \calo_{i+1}^{(1)} \ra k_i$ for all $1 \leq r \leq i$, and $k_{i-1}$ is the residue field of $\calo_i^{(1)}$.
\end{compactitem}

According to \ref{para:rank_n_pushforward_chain} these data yield the following commutative diagram:

\[
\xymatrix   @W=20pt @H=20pt @M=0pt @L=0pt @C=20pt @R=20pt  {
k_0  & k_1 & k_2 & k_3 & \dots & \dots & k_{n-2} & k_{n-1} & k_n \\
&\calo_1^{{(1)}} \ar@{ >->}[u] \ar@{->>}[ul]_{q_1^{(1)}} & \calo_2^{{(1)}} \ar@{ >->}[u] \ar@{->>}[ul]_{q_2^{(1)}} & \calo_3^{{(1)}} \ar@{ >->}[u] \ar@{->>}[ul]_{q_3^{(1)}} & \dots \ar@{->>}[ul] \ar@{ >->}[u] & \dots \ar@{->>}[ul] \ar@{ >->}[u] & \calo_{n-2}^{{(1)}} \ar@{ >->}[u] \ar@{->>}[ul] & \calo_{n-1}^{{(1)}} \ar@{ >->}[u] \ar@{->>}[ul]_{q_{n-1}^{(1)}} & \calo_n^{{(1)}} \ar@{ >->}[u] \ar@{->>}[ul]_{q_n^{(1)}} \\
&& \calo_2^{{(2)}} \ar@{ >->}[u] \ar@{->>}[ul]_{q_2^{(2)}} & \calo_3^{{(2)}} \ar@{ >->}[u] \ar@{->>}[ul]_{q_3^{(2)}}  & \dots \ar@{->>}[ul] \ar@{ >->}[u] & \dots \ar@{->>}[ul] \ar@{ >->}[u] & \calo_{n-2}^{{(2)}} \ar@{ >->}[u] \ar@{->>}[ul] & \calo_{n-1}^{{(2)}} \ar@{ >->}[u] \ar@{->>}[ul]_{q_{n-1}^{(2)}} & \calo_n^{{(2)}} \ar@{ >->}[u] \ar@{->>}[ul]_{q_n^{(2)}} \\
&& & \calo_3^{{(3)}} \ar@{->>}[ul]_{q_3^{(3)}} \ar@{ >->}[u] & \dots \ar@{->>}[ul] \ar@{ >->}[u] & \dots \ar@{->>}[ul] \ar@{ >->}[u] & \calo_{n-2}^{{(3)}} \ar@{ >->}[u] \ar@{->>}[ul] & \calo_{n-1}^{{(3)}} \ar@{ >->}[u] \ar@{->>}[ul]_{q_{n-1}^{(3)}} & \calo_n^{{(3)}} \ar@{ >->}[u] \ar@{->>}[ul]_{q_n^{(3)}}  \\
& && & \dots \ar@{->>}[ul] \ar@{ >->}[u] & \dots \ar@{->>}[ul] \ar@{ >->}[u] & \dots \ar@{->>}[ul] \ar@{ >->}[u] & \dots \ar@{->>}[ul] \ar@{ >->}[u] & \dots \ar@{->>}[ul] \ar@{ >->}[u] \\
& && & & \dots \ar@{->>}[ul] \ar@{ >->}[u] &  \calo_{n-2}^{{(n-3)}} \ar@{ >->}[u] \ar@{->>}[ul] & \calo_{n-1}^{{(n-3)}} \ar@{ >->}[u] \ar@{->>}[ul]  & \calo_n^{{(n-3)}} \ar@{ >->}[u] \ar@{->>}[ul] \\
& & && & &\calo_{n-2}^{{(n-2)}} \ar@{ >->}[u] \ar@{->>}[ul] & \calo_{n-1}^{{(n-2)}} \ar@{ >->}[u] \ar@{->>}[ul]_{q_{n-1}^{(n-2)}} & \calo_n^{{(n-2)}} \ar@{ >->}[u] \ar@{->>}[ul]_{q_{n}^{(n-2)}} \\
& & && & & & \calo_{n-1}^{{(n-1)}} \ar@{ >->}[u] \ar@{->>}[ul]_{q_{n-1}^{(n-1)}} & \calo_n^{{(n-1)}} \ar@{ >->}[u] \ar@{->>}[ul]_{q_{n}^{(n-1)}} \\
& & & && & & & \calo_n^{{(n)}} \ar@{ >->}[u] \ar@{->>}[ul]_{q_{n}^{(n)}} 
}
\]

If $\calo = \calo_v$ for a discrete valuation $v$ of rank $n$ on $k$, then we also have explicit admissible valuations for each of the rings in the diagram in terms of $v$. For this we define by induction $v_n \dopgleich v$ and $v_i \dopgleich \widetilde{v}_{i+1}$ for $1 \leq i \leq n-1$, where $\widetilde{v}_{i+1}$ denotes the discrete valuation of rank $i$ induced by $v_{i+1}$ on $k_i$ as in \ref{para:disc_val_rank_n_push}. The ring $\calo_{i}^{(r)}$ is then equal to the valuation ring of $v_i^{(r)}$. By definition of the $v_i$ we have $v_i^{(1)}(\ol{x}) = v^i(x)$ for $\ol{x} \in k^\times$, where $x \in k_n^\times = k^\times$ is a lift of $\ol{x}$ along the zig-zag line between the first two rows in the diagram. 
\end{para}

\begin{para} \label{para:pullback_of_rank_n_mot}
By \ref{para:pushforward_pullback_inverse} we have 
\[
\calo_{i+1}^{(r+1)} = ( (\calo_{i+1}^{(r+1)}, \calo_{i+1}^{(1)})_*, \calo_{i+1}^{(1)})^* = ( \calo_{i}^{(r)}, \calo_{i+1}^{(r)})^*,
\]
for all $1 \leq i \leq n -1$ and $1 \leq r \leq i$. This shows that we can reconstruct the whole diagram (and thus the discrete valuation $v$ of rank $n$ on $k = k_n$ up to equivalence) inductively via pullbacks from only the first two rows of the diagram. Hence, the diagram provides an invertible transformation from a discrete valuation of rank $n$ on $k$ (the $n$-th column) to a family of discrete valuations of rank $1$ (the first two rows). This family is an important tool in studying discrete valuations of higher rank and therefore we formalize this situation in the next paragraph.
\end{para}

\begin{defn}
Let $n \in \NN_{>0}$. A \words{discrete valuation ring stack of rank $n$}{discrete valuation ring stack} is a family $\ca{O} = \lbrace \calo_i \mid i=1,\ldots,n \rbrace$, where $\calo_i$ is a discrete valuation ring of rank $1$ whose residue field is equal to the quotient field of $\calo_{i-1}$ for all $i=2,\ldots,n$.
A \words{discrete valuation stack of rank $n$}{discrete valuation stack} is a family $\ca{K} = \lbrace (k_i,w_i) \mid i=1,\ldots,n \rbrace$, where $(k_i,w_i)$ is a discrete valuation field of rank $1$ for all $i$ and $\kappa_{w_i} = k_{i-1}$ for all $i=2,\ldots,n$.
\end{defn}

\begin{para}
If  $\ca{K} = \lbrace (k_i,w_i) \mid i=1,\ldots,n \rbrace$ is a discrete valuation stack of rank $n$, then $\ca{O}(\ca{K}) \dopgleich \lbrace \calo_{w_i} \mid i=1,\ldots, n \rbrace$ is obviously a discrete valuation ring stack of rank $n$. If $\calo$ is a discrete valuation ring of rank $n$, then using the notation of \ref{para:rank_n_pushforward_diagram} we see that $\calo_* \dopgleich \lbrace \calo_i^{(1)} \mid i=1,\ldots,n \rbrace$ is a discrete valuation ring stack of rank $n$. If $v$ is a discrete valuation of rank $n$ on a field $k$, then using the notation of \ref{para:rank_n_pushforward_diagram} with $\calo = \calo_v$ we see that $(k,v)_* \dopgleich \lbrace (k_i,v_i^{(1)}), i=1,\ldots,n \rbrace$ is a discrete valuation stack of rank $n$ with $\ca{O}( (k,v)_*) = (\calo_v)_*$.
\end{para}

\begin{para}
To see that a general discrete valuation ring stack $\ca{O} = \lbrace \calo_i \mid i=1,\ldots,n \rbrace$ induces via inductive pullbacks a discrete valuation ring of rank $n$ in the quotient field of $\calo_n$, we first have to determine an explicit admissible valuation for the pullbacks.
\end{para}

\begin{prop} \label{para:disc_val_rank_n_pullback}
Let $n \in \NN_{>0}$. Let $w:k^\times \ra \ZZ$ be a discrete valuation of rank 1 on a field $k$ and let $v:\kappa_w^\times \ra \ZZ^{n}$ be a discrete valuation of rank $n$. Then the map
\[
\begin{array}{rcl}
(v \circ w)_t: k^\times & \lra & \ZZ^n \oplus \ZZ = \ZZ^{n+1} \\
x & \longmapsto & (v \circ q(xt^{-w(x)}), w(x)),
\end{array}
\]
where $t \in \calo_w$ is a uniformizer and $q:\calo_w \ra \kappa_w$ is the quotient morphism, is a discrete valuation of rank $n+1$ on $k$ whose valuation ring is equal to $(\calo_v,\calo_w)^*$.
\end{prop}

\begin{proof}
To simplify notations, let $u \dopgleich (v \circ w)_t$ and let $u^1$ be the first component of $u$. Since $w(xt^{-w(x)}) = w(x) - w(x) = 0$, we have $xt^{-w(x)} \in \calo_w^\times$ and therefore $q(xt^{-w(x)}) \in \kappa_w^\times$ so that $u^1(x)$ is well-defined. If $x,y \in k^\times$, then
\[
v \circ q(xyt^{-w(xy)}) = v \circ q(xyt^{-w(x) - w(y)}) = v( q(xt^{-w(x)}) \cdot q(yt^{-w(y)})) = v \circ q(xt^{-w(x)}) + v \circ q(yt^{-w(y)})
\]
and therefore $u$ is multiplicative. To see that $u$ is surjective, let $a_{n+1} \in \ZZ$. Since $w$ is surjective, there exists $x \in k^\times$ with $w(x) = a_{n+1}$. Then $w(x \calo_w^\times) = \lbrace a_{n+1} \rbrace$. If $\eps \in \calo_w^\times$, then $u^1(\eps) = v \circ q(\eps)$. Hence, since $q(\calo_w^\times) = \kappa_w^\times$ and since $v$ is surjective, it follows that
\[
u^1(x \calo_w^\times) = u^1(x) + u^1(\calo_w^\times) = u^1(x) + v \circ q(\calo_w^\times) = u^1(x) + \ZZ^n = \ZZ^n.
\]
This shows that $u(x \calo_w^\times) = \ZZ^n \times \lbrace a_{n+1} \rbrace$ and as $a_{n+1}$ was arbitrary, we conclude that $u(k^\times) = \ZZ^{n+1}$.

It remains to verify the ultrametric inequality. Let $x,y \in k^\times$ with $x+y \neq 0$. First suppose that $w(x) < w(y)$. Then $u(x) < u(y)$ and $w(x+y) \geq w(x)$. If $w(x+y) > w(x)$, then 
\[
u(x+y) = (u^1(x+y),w(x+y)) > (u^1(x),w(x)) = u(x) = \ro{min} \lbrace u(x),u(y) \rbrace.
\]
Otherwise, we have $w(x+y) = w(x)$. Since $w(yt^{-w(x)}) = w(y) - w(x) > 0$, we conclude that $yt^{-w(x)} \in \fr{m}_w$ and therefore 
\[
u^1(x+y) = v \circ q( (x+y)t^{-w(x+y)}) = v \circ q( xt^{-w(x)} + yt^{-w(x)}) = v( q( xt^{-w(x)}) + 0) = u^1(x).
\]
Hence,
\[
u(x+y) = (u^1(x+y),w(x+y)) = (u^1(x),w(x)) = u(x) = \ro{min} \lbrace u(x),u(y) \rbrace.
\]

The case $w(x) > w(y)$ is similar, so it remains to consider the case $w(x) = w(y)$. First suppose that $u(x) \leq u(y)$. Then $u^1(x) \leq u^1(y)$. If $w(x+y) > \ro{min} \lbrace w(x),w(y) \rbrace = w(x)$, then 
\[
u(x+y) = (u^1(x+y),w(x+y)) > (u^1(x),w(x)) = u(x) = \ro{min} \lbrace u(x),u(y) \rbrace.
\]
If $w(x+y) = \ro{min} \lbrace w(x),w(y) \rbrace = w(x)$, then
\begin{align*}
u^1(x+y) &= v \circ q( (x+y)t^{-w(x+y)}) = v( q (x t^{-w(x)} + yt^{-w(y)})) \\
&\geq \ro{min} \lbrace v \circ q(xt^{-w(x)}), v \circ q( yt^{-w(y)}) \rbrace = u^1(x)
\end{align*}
and therefore
\[
u(x+y) = (u^1(x+y),w(x+y)) = (u^1(x+y),w(x)) \geq (u^1(x),w(x)) = u(x) = \ro{min} \lbrace u(x),u(y) \rbrace.
\]

The case $u(x) \geq u(y)$ is similar. This shows that $u$ satisfies the ultrametric inequality and is thus a discrete valuation of rank $n+1$.

It remains to verify that $\calo_u = (\calo_v,\calo_w)^*$. We already know that $(\calo_v,\calo_w)^*$ is a valuation ring with $(\calo_v,\calo_w)^* < \calo_w$ and maximal ideal $\fr{m} \dopgleich q^{-1}(\fr{m}_v)$. Hence, $\fr{m}_w < \fr{m}$. Now, if $x \in \calo_u$, then $(u^1(x),w(x)) = u(x) \geq 0$ and this inequality leaves two cases. The first case is that $u^1(x) \geq 0$ and $w(x) = 0$. Then $0 \leq u^1(x) = v \circ q(xt^{-w(x)}) = v \circ q(x)$ and therefore $q(x) \in \calo_v$, that is, $x \in q^{-1}(\calo_v) = (\calo_v,\calo_w)^*$. The second case is that $w(x) > 0$. Then $x \in \fr{m}_w < \fr{m} < (\calo_v,\calo_w)^*$. Hence, $\calo_u \leq (\calo_v,\calo_w)^*$. 

According to \ref{para:finer_valuation} we can prove the converse inclusion $\calo_u \geq (\calo_v,\calo_w)^*$ by proving that $\fr{m}_u \leq \fr{m}$. If $x \in \fr{m}_u$, then $u(x) > 0$ and this leaves two cases. The first case is that $w(x) > 0$ and this implies $x \in \fr{m}_w < \fr{m}$. The second case is that $w(x) = 0$ and $v \circ q(xt^{-w(x)}) > 0$. Then $v \circ q(x) > 0$, that is, $q(x) \in \fr{m}_v$ and therefore $x \in q^{-1}(\fr{m}_v) = \fr{m}$.
\end{proof}

\begin{cor} \label{para:disc_val_rank_n_pullback_chain}
Let $\calo_n^{(1)}$ be a discrete valuation ring of rank $1$ in a field $k_n$. Let $\calo_{n-1}$ be a discrete valuation ring of rank $n-1$ in the residue field $k_{n-1}$ of $\calo^{(1)}$ and let
\[
\calo_{n-1} \gleichdop \calo_{n-1}^{(n-1)} < \calo_{n-1}^{(n-2)} < \ldots < \calo_{n-1}^{(2)} < \calo_{n-1}^{(1)} < k_{n-1}
\]
be the chain of intermediate rings between $\calo_{n-1}$ and $k_{n-1}$. Then $\calo_{n}^{(r)} \dopgleich (\calo_{n-1}^{(r-1)},\calo_n^{(1)})_*$ is a discrete valuation ring of rank $r$ in $k_n$ for all $2 \leq r \leq n$ and 
\[
\calo_{n}^{(n)} < \calo_{n}^{(n-1)} < \ldots < \calo_{n}^{(2)} < \calo_{n}^{(1)} < k_{n}
\]
is the chain of intermediate rings between $\calo_{n}^{(n)}$ and $k_{n}$.
\end{cor}

\begin{proof}
This follows immediately from \ref{para:disc_val_rank_n_pullback}.
\end{proof}

\begin{para} \label{para:disc_stack_pullback}
Let $\ca{O} = \lbrace \calo_i \mid i=1,\ldots,n \rbrace$ be a discrete valuation ring stack of rank $n$. We inductively define the following data for all $r=1,\ldots,n$:
\begin{compactitem}
\item If $r = 1$, then $\calo_i^{(1)} \dopgleich \calo_i$ and $k_i$ is the residue field of $\calo_i$ for all $1 \leq i \leq n$.
\item If $r > 1$, then $\calo_{i}^{(r)} \dopgleich (\calo_{i-1}^{(r-1)},\calo_{i}^{(1)})^*$ and $q_{i}^{(r)}: \calo_{i}^{(r)} \ra \calo_{i-1}^{(r-1)}$ is the quotient morphism for all $r \leq i \leq n$.
\end{compactitem}

According to \ref{para:disc_val_rank_n_pullback_chain} the above data yield a commutative diagram as in \ref{para:rank_n_pushforward_diagram} and the ring $\ca{O}^* \dopgleich \calo_n^{(n)}$ is a discrete valuation ring of rank $n$ in $k_n$.
If $\ca{K} = \lbrace (k_i,w_i) \mid i=1,\ldots,n \rbrace$ is a discrete valuation stack of rank $n$ and $\ca{O} = \ca{O}(\ca{K})$, then using \ref{para:disc_val_rank_n_pullback} an admissible valuation for $\ca{O}^*$ is given as follows: for each $i=1,\ldots,n$ let $t_i' \in \calo_i^{(1)}$ be a uniformizer and let $t_i \in \calo_n^{(n-i+1)}$ be a lift of $t_i'$ along the diagonal between $\calo_i^{(1)}$ and $\calo_n^{(n-i+1)}$ in the diagram. The family $(t_1,\ldots,t_n) \in (k_n^\times)^n$ is then called a \word{local system of parameters} for $\ca{K}$. An inductive application of \ref{para:disc_val_rank_n_pullback} now shows that
\[
\left( \left( \left(w_1 \circ w_2 \right)_{t_2'} \circ w_3 \right)_{t_3'} \circ \ldots \circ w_n \right)_{t_n'} \gleichdop v = (v^1,\ldots,v^n): k_n^\times \ra \ZZ^n
\]
is a discrete valuation of rank $n$ on $k_n$ which is admissible for $\ca{O}^*$. The $v^i$ are determined inductively as follows:
\begin{compactitem}
\item $v^n = w_n$.
\item $v^i = w_i \circ q_{i+1}^{(1)} \left( q_{i+2}^{(1)} \left( \dots \left( q_n^{(1)} \left( xt_n^{-w_n(x)} t_{n-1}^{-w_{n-1}(x)} \dots t_{i+1}^{-w_{i+1}(x)} \right) \right) \right) \right)$ for $i < n$.
\end{compactitem}

We define $\ca{K}^*_{(t_1,\ldots,t_n)} \dopgleich (k_n,v)$. To get an intuition for $v^i$ note that any $x \in k_n$ can be projected along the zig-zag line in the diagram between the first two rows from the right to the left into $k_i$ as follows: multiplying $x$ with $t_n^{-w_n(x)}$ pushes $x$ into the units of the valuation ring $\calo_n^{(1)}$ and then $q_n^{(1)}$ can be applied which gives an element in $k_{n-1}$. This element is now again pushed by multiplication with a certain power of $t_{n-1}'$ into the units of $\calo_{n-1}^{(1)}$ and we can apply $q_{n-1}^{(1)}$ to get an element in $k_{n-2}$ and so on.
\end{para}

\begin{para} \label{para:pushforward_pullback_inductive_inverse}
Using \ref{para:pushforward_pullback_inverse} we see that the constructions in \ref{para:rank_n_pushforward_diagram} and \ref{para:disc_stack_pullback} are inverse to each other: if $\calo$ is a discrete valuation ring of rank $n$, then $(\calo_*)^* = \calo$ and if $\ca{O}$ is a discrete valuation ring stack of rank $n$, then $(\ca{O}^*)_* = \ca{O}$.
\end{para}

\begin{defn}
A discrete valuation stack $\ca{K} = \lbrace (k_i,w_i) \mid i=1,\ldots,n \rbrace$ of rank $n$ is said to be \words{complete}{discrete valuation stack!complete} if $w_i$ is complete for all $i$. A \words{local field of rank $n$}{local field!of rank $n$} is a discrete valuation field $(k,v)$ of rank $n$ such that the corresponding stack $(k,v)_* = \lbrace (k_i,w_i) \mid i=1,\ldots,n \rbrace$ is complete and $(k_1,w_1)$ is a local field, that is, $k_0 = \kappa_{w_1} = \kappa_v$ is finite.
\end{defn}

\begin{para} \label{para:standard_higher_local}
Let $(k_1,w_1)$ be a complete discrete valuation field. We define inductively for $i=2,\ldots,n$ the field $k_i \dopgleich k_{i-1}((T_{i}))$ and denote the order valuation on this field by $w_i$. The residue field of $k_i$ can canonically be identified with $k_{i-1}$ and therefore $\ca{K} = \lbrace (k_i,w_i) \mid i=1,\ldots,n \rbrace$ is a complete discrete valuation stack of rank $n$. A local system of parameters for this stack is $(\pi,T_2,\ldots,T_n) \in k_n = k_1((T_2))\dots((T_n))$, where $\pi \in k_1$ is a uniformizer. Hence, if $(k_1,w_1)$ is a local field, then $\ca{K}_{(\pi,T_2,\ldots,T_n)}^* = (  k_1((T_2))\dots((T_n)), v)$ is a local field of rank $n$. We will always equip $k_1((T_2))\dots((T_n))$ with this induced discrete valuation of rank $n$ if nothing else is mentioned.
\end{para}

\begin{prop}
If $(k,v)$ is a local field of rank $n$ and of characteristic $p>0$, then $k$ is as a topological field already isomorphic to $\FF_q((T_1))\dots((T_n))$ with $q = p^f$ and $f \in \NN_{>0}$.
\end{prop}

\begin{proof}
This is proven in \cite[\S1]{0579.12012} and follows by induction from the classification of equal characteristic complete discrete valuation fields mentioned in \ref{para:complete_disc_field_props}.
\end{proof}

\begin{prop} \label{para:higher_local_ext_closed}
Let $(k,v)$ be a local field of rank $n$. Then any finite extension $K$ of $k$ is canonically a local field of rank $n$. Hence, the class of local fields of fixed rank is closed under finite extensions.
\end{prop}

\begin{proof}
Let $(k,v) = \lbrace (k_i,v_i) \mid i=1,\ldots, n \rbrace$ be the corresponding stack. As $K_n \dopgleich K$ is a finite extension of $k_n = k$ and as $v_n$ is complete, there exists according to \ref{para:complete_real_val_ext} a unique extension $w_n$ of $v_n$ to $K_n$ which is again discrete and complete. Moreover, according to \ref{para:ramification_index} the residue field extension $K_{n-1} \dopgleich \kappa_{w_n} | \kappa_{v_n} = k_{n-1}$ is finite. Hence, repeating this process yields a complete discrete valuation stack $\lbrace (K_i,w_i) \mid i = 1,\ldots,n \rbrace$, where $(K_i,w_i)|(k_i,v_i)$ is a finite extension. As $(K_1,w_1)$ is a local field, it follows that $K$ is an $n$-dimensional local field. 
\end{proof}

\subsection{Milnor $\ro{K}$-theory} \label{sec:milnor_k}

\begin{para}
In this section we will provide the definition of Milnor $\ro{K}$-groups and mention some of their properties. A certain quotient of Milnor $\ro{K}$-theory will be the class functor in \name{Fesenko}'s approach to higher local class field theory.
\end{para}

\begin{para}
Milnor $\ro{K}$-theory is an algebraic invariant, more precisely a sequence of abelian groups, attached to a field which was introduced by \name{John Milnor} in \cite{Mil70_Algebraic-K-theory_0} (although not under this name of course). Milnor studied quadratic forms over a field $k$ of characteristic not equal to $2$ and observed that the first three algebraic $\ro{K}$-groups of $k$ reduced modulo $2$ are closely related to the Witt ring $W(k)$ of anisotropic quadratic modules over $k$. More precisely, there exist canonical isomorphisms $I^n/I^{n+1} \cong \ro{K}_n(k)/2\ro{K}_n(k)$ for $n \in \lbrace 0,1,2 \rbrace$, where $I$ is the maximal ideal of $W(k)$. Milnor tried to generalize this to higher degree by introducing the Milnor $\ro{K}$-groups $\ro{K}_n^{\ro{M}}(k)$ for all $n \in \NN$ and his motivation for their definition was \name{Hideya Matsumoto}'s presentation of the second algebraic $\ro{K}$-group of a field which was for a general ring $R$ defined earlier by Milnor as the center of the Steinberg group $\ro{St}(R)$. Matsumoto showed that $\ro{K}_2(k)$ is for a field $k$ equal to the abelian group $k^\times \otimes_\ZZ k^\times/ \langle a_1 \otimes a_2 \mid a_1 + a_2 = 1 \rangle$. Milnor simply took this presentation, generalized it to an $n$-fold tensor product modulo a similar subgroup and defined the Milnor $\ro{K}$-groups $\ro{K}_n^{\ro{M}}(k)$ in this way. He then showed in \cite{Mil70_Algebraic-K-theory_0} the existence of a canonical epimorphism $I^n/I^{n+1} \ra \ro{K}_n^{\ro{M}}(k)/2\ro{K}_n^{\ro{M}}(k)$ and conjectured that this is an isomorphism for all $n \in \NN$. This became the famous Milnor conjecture which was proven by \name{Vladimir Voevodskij} in 1997 using very advanced techniques. 

Milnor himself states in \cite{Mil70_Algebraic-K-theory_0} that his definition of $\ro{K}_n^{\ro{M}}(k)$ is purely ad hoc.\footnote{The author believes that Milnor's deep insight into the theory is mainly responsible for this definition because it leads to such deep connections. As the author unfortunately does not share this insight, further motivations for this definition cannot be given and also seem not to be given in the literature. Another motivation may be provided in terms of symbols as given below, but this is essentially just a reformulation of the definition.} In spite of its simple definition, Milnor $\ro{K}$-theory is incredibly hard to compute and it turned out that deep information about the arithmetic of the field is encoded in it. One hint in this direction is the Bloch-Kat\={o} conjecture which generalizes the above isomorphisms to reductions of the Milnor $\ro{K}$-groups modulo any natural number prime to the characteristic of the field and replaces the ideal quotient by a Galois cohomology group. This conjecture has (probably) also just recently been proven by \name{Vladimir Voevodskij}, \name{Markus Rost}, \name{Charles Weibel} and perhaps further people.
\end{para}

\begin{para}
As indicated in the footnote above, the Milnor $\ro{K}$-groups may also be characterized as being the universal codomains of symbols. We motivate this by the following example. Let $p$ be a prime number and let $k \dopgleich \QQ_p$ (the following indeed works for any local field of characteristic not equal to 2). The \word{Hilbert symbol} of $\QQ_p$ is for $a,b \in k^\times$ defined as follows:
\[
(a,b)_p \dopgleich \left\lbrace \begin{array}{rl} 1 & ax^2 + by^2 = z^2 \tn{ has a solution } (x,y,z) \in k^3 \setminus \lbrace (0,0,0) \rbrace \\ -1 & \tn{otherwise.} \end{array} \right.
\]

One can show (confer \cite[chapter III, proposition 2]{Ser73_A-Course-in-Arithmetic_0}) that $(\cdot,\cdot)_p:k^\times \times k^\times \ra \ZZ/(2)$ is a $\ZZ$-bilinear map with the additional property that $(a,b)_p = 1$ if $a+b = 1$. 

We take this as a motivation for defining for $n \in \NN_{>0}$ an \words{$n$-symbol}{symbol} (or \words{Steinberg $n$-cocycle}{Steinberg cocycle}) on a field $k$ with values in a multiplicatively written abelian group $A$ to be a $\ZZ$-multilinear map $\varphi: (k^\times)^n \ra A$ satisfying the \word{Steinberg property}:
\[
\tn{if } (a_1,\ldots,a_n) \in (k^\times)^n \tn{ such that } a_i + a_j = 1 \tn{ for some } i \neq j, \tn{ then } \varphi(a_1,\ldots,a_n) = 1.
\]

A \words{universal $n$-symbol}{symbol!universal} is now defined as an $n$-symbol $\lbrace \cdot \rbrace:(k^\times)^n \ra \ro{K}_n^{\ro{M}}(k)$ which is universal among $n$-symbols on $k$, that is, for any $n$-symbol $\varphi$ on $k$ with values in $A$ there exists a unique morphism $\widetilde{\varphi}: \ro{K}_n^{\ro{M}}(k) \ra A$ making the diagram
\[
\xymatrix{
(k^\times)^n \ar[r]^-\varphi \ar[dr]_{ \lbrace \cdot \rbrace } & A \\
& \ro{K}_n^{\ro{M}}(k) \ar[u]_{\widetilde{\varphi}}
}
\]
commutative. The group $\ro{K}_n^{\ro{M}}(k)$ is then unique up to unique isomorphism. Using the universal property of tensor products it is easy to see that the map
\[
\begin{array}{rcl}
\lbrace \cdot \rbrace: (k^\times)^n & \lra & (k^\times)^{\otimes_\ZZ^n}/ \ca{S}_n \\
(a_1,\ldots,a_n) & \longmapsto & \lbrace a_1,\ldots,a_n \rbrace \dopgleich a_1 \otimes \dots \otimes a_n \modd \ca{S}_n
\end{array}
\]
is a universal $n$-symbol, where $(k^\times)^{\otimes_\ZZ^n}$ denotes the $n$-fold tensor product of $k^\times$ over $\ZZ$ and $\ca{S}_n \dopgleich \langle a_1 \otimes \dots \otimes a_n \mid a_i + a_j = 1 \tn{ for some } i\neq j \rangle$. We will now fix $\ro{K}_n^{\ro{M}}(k) \dopgleich  (k^\times)^{\otimes_\ZZ^n}/ \ca{S}_n$ for all $n \in \NN$ including $n=0$ where $\ro{K}_0^{\ro{M}}(k) = \ZZ$, and call this group the $n$-th \word{Milnor $\ro{K}$-group}. The groups $\ro{K}_n^{\ro{M}}(k)$ are usually written additively. 
\end{para}

\begin{para}
It is not hard to see that the \word{cup-product} 
\[
\begin{array}{rcl}
\cup: \ro{K}_n^{\ro{M}}(k) \times \ro{K}_m^{\ro{M}}(k) & \lra & \ro{K}_{n+m}^{\ro{M}}(k) \\
(\lbrace a_1,\ldots,a_n \rbrace, \lbrace b_1,\ldots, b_m \rbrace) & \longmapsto & \lbrace a_1,\ldots,a_n,b_1,\ldots,b_m \rbrace
\end{array}
\]
for $n,m \in \NN_{>0}$ and $\cup: \ro{K}_0^{\ro{M}}(k) \times \ro{K}_n^{\ro{M}}(k) \ra \ro{K}_n^{\ro{M}}(k)$, $\cup: \ro{K}_n^{\ro{M}}(k) \times \ro{K}_0^{\ro{M}}(k) \ra \ro{K}_n^{\ro{M}}(k)$ given by the $\ZZ = \ro{K}_0^{\ro{M}}(k)$-action on the abelian group $\ro{K}_n^{\ro{M}}(k)$ for all $n \in \NN$ are well-defined morphisms which make $\ro{K}_*^{\ro{M}}(k) \dopgleich \bigoplus_{n \in \NN} \ro{K}_n^{\ro{M}}(k)$ into an $\NN$-graded  ring, the \word{Milnor ring} of $k$. 
\end{para}

\begin{para}
As mentioned at the beginning, we have $\ro{K}_2^{\ro{M}}(k) = \ro{K}_2(k)$ and we obviously have $\ro{K}_1^{\ro{M}}(k) = k^\times = \ro{K}_1(k)$. Hence, Milnor $\ro{K}$-theory coincides with (Quillen) algebraic $\ro{K}$-theory in degrees $0 \leq n \leq 2$ and as mentioned in \cite[5.1.5(ii)]{Kuk07_Representation-theory_0} and \cite[section 4.3.1]{MR2181827} there exist canonical morphisms $\ro{K}_n^{\ro{M}}(k) \ra \ro{K}_n(k)$ which are however in general only isomorphisms for $n \leq 2$, that is, Milnor $\ro{K}$-theory and algebraic $\ro{K}$-theory split into different directions in higher degree.
\end{para}

\begin{para}
Let $n \in \NN_{>0}$. If $\sigma:k \ra K$ is a morphism of fields, then composition of the induced map $(k^\times)^n \ra (K^\times)^n$ with the universal $n$-symbol $\lbrace \cdot \rbrace: (K^\times)^n \ra \ro{K}_n^{\ro{M}}(K)$ is a symbol and thus induces due to the universal property a unique morphism $\sigma_*:\ro{K}_n^{\ro{M}}(k) \ra \ro{K}_n^{\ro{M}}(K)$ such that the diagram
\[
\xymatrix{
\ro{K}_n^{\ro{M}}(k) \ar[r]^-{\sigma_*} & \ro{K}_n^{\ro{M}}(K) \\
(k^\times)^n \ar[u]^{\lbrace \cdot \rbrace} \ar[r]_{\sigma} & (K^\times)^n \ar[u]_{\lbrace \cdot \rbrace}
}
\]
commutes. If $n = 0$, then we define $\sigma_*:\ro{K}_0^{\ro{M}}(k) = \ZZ \ra \ZZ = \ro{K}_0^{\ro{M}}(K)$ to be the identity and this yields a degree zero morphism $\sigma_*: \ro{K}_*^{\ro{M}}(k) \ra \ro{K}_*^{\ro{M}}(K)$. This morphism obviously satisfies the following properties:
\begin{enumerate}[label=(\roman*),noitemsep,nolistsep]
\item $(\id_k)_* = \id_{\ro{K}_*^{\ro{M}}(k)}$.
\item $(\tau \circ \sigma)_* = \tau_* \circ \sigma_*$ if $\tau:K \ra L$ is a further morphism.
\end{enumerate}

Hence, $\ro{K}_*^{\ro{M}}(-)$ is a (covariant) functor from the category of fields to the category of graded rings. If $K|k$ is a field extension, then we will denote the morphism $ \ro{K}_*^{\ro{M}}(k) \ra \ro{K}_*^{\ro{M}}(K)$ induced by the canonical inclusion $k \ra K$ by $j_{K|k}$.
\end{para}

\begin{para}
In the above we have lifted a natural relation between field extensions, the inclusion, to the Milnor $\ro{K}$-theory. Another natural relation for a finite extension of fields is the norm map and so the question is if this map can also be lifted to Milnor $\ro{K}$-theory. This question was raised by \name{Hyman Bass} and \name{John Tate} in \cite[chapter I, \S 5]{BasTat73_The-Milnor-ring_0} and an idea of a construction was given. The idea is to choose a tower of simple intermediate extensions for which a lift of the norm map to the Milnor $\ro{K}$-theory could be defined and then take the composition of all these norm maps in the tower. The problem was that the independence of the choice of the tower could not be proven. It took several years before this was proven by \name{Kazuya Kat\={o}} in \cite[\S1.7]{Kat80_A-generalization-of-local_0}. Even the definition of the norm map is intricate and therefore we will just review its definition and basic properties and refer the reader to the nice expositions in \cite[chapter 7]{MR2266528} and \cite[chapter IX]{FesVos02_Local-fields_0} for the proofs.
\end{para}

\begin{para} 
The basis for the definition of the norm map is the so-called tame symbol: if $(k,v)$ is a discrete valuation field, then for each $n \in \NN_{>0}$ one can prove that there exists a unique morphism $\partial^{\ro{M}} = \partial_{v}^{\ro{M}}: \ro{K}_n^{\ro{M}}(k) \ra \ro{K}_{n-1}^{\ro{M}}(\kappa_v)$, called the \word{tame symbol}, such that $\partial^{\ro{M}}( \lbrace \pi, u_2, \ldots, u_n \rbrace) = \lbrace \ol{u}_2,\ldots,\ol{u}_n \rbrace$ for all uniformizers $\pi$ and all $(u_2,\ldots,u_n) \in (\calo_v^\times)^{n-1}$. For $n=1$ the tame symbol $\partial^{\ro{M}}: \ro{K}_1^{\ro{M}}(k) = k^\times \ra \ZZ = \ro{K}_0^{\ro{M}}(\kappa_v)$ is just the valuation $v$. If $(K,w)|(k,v)$ is an extension of discrete valuation fields, then $\partial_{w}^{\ro{M}} \circ j_{K|k} = e(w|v) \cdot j_{\kappa_w|\kappa_v} \circ \partial_{v}^{\ro{M}}$. 
\end{para}

\begin{para} \label{para:milnor_k_norm}
The discrete valuations on a function field $k(X)$ which are trivial on $k$ are up to equivalence either the $p$-adic valuations for a prime element $p \in k \lbrack X \rbrack$ or the degree valuation which can also be identified as the $(1/X)$-adic valuation of the polynomial $1/X \in k \lbrack 1/X \rbrack \subs k(X)$. Hence, the equivalence classes of discrete valuations on the function field $k(X)$ which are trivial on $k$ are in one-to-one correspondence with the closed points of the scheme $\PP^1_k$. We define $\infty \dopgleich (1/X)$ and for $P \in \PP_k^1$ we set $\partial_{P}^{\ro{M}} \dopgleich \partial_{v_P}^{\ro{M}}$, where $v_P$ is the corresponding discrete valuation on $k(X)$. A theorem by \name{Milnor} and \name{Tate} (\cite[theorem 7.2.1]{MR2266528}) now states that for $n \in \NN_{>0}$ the sequence
\[
\xymatrix{
0 \ar[r] & \ro{K}_n^{\ro{M}}(k) \ar[r]^-{j_{k(X)|k}} & \ro{K}_n^{\ro{M}}(k(X)) \ar[r]^-{\bigoplus \partial_P^{\ro{M}}} & \bigoplus_{P \in \PP_k^1 \setminus \lbrace \infty \rbrace} \ro{K}_{n-1}^{\ro{M}}( \kappa(P) ) \ar[r] & 0
}
\]
is split exact. A consequence of this theorem is that there exist unique morphisms $\ro{N}_P: \ro{K}_n^{\ro{M}}( \kappa(P)) \ra \ro{K}_n^{\ro{M}}(k)$ for all $P \in \PP_k^1 \setminus \lbrace \infty \rbrace$ such that $- \partial_\infty^{\ro{M}}(x) = \sum_{P \in \PP_k^1 \setminus \lbrace \infty \rbrace} \ro{N}_P \circ \partial_P^{\ro{M}}(x)$ for all $x \in \ro{K}_{n}^{\ro{M}}(k(X))$. If we additionally define $\ro{N}_\infty = \id_{\ro{K}_n^{\ro{M}}(k)}$, then we get the \word{Weil reciprocity formula} $\sum_{P \in \PP_k^1} \ro{N}_P \partial_P^{\ro{M}}(x) = 0$ for all $x \in \ro{K}_{n}^{\ro{M}}(k(X))$.

Now, suppose that $K|k$ is a simple extension, that is, $K = k(\theta)$ for some $\theta$ which is algebraic over $k$. If $\mu$ is the minimal polynomial of $\theta$, then $(\mu)$ is a maximal ideal in $k \lbrack X \rbrack$ and thus a closed point in $\PP_k^1$ with residue field isomorphic to $k(\theta) = K$. Hence, we have a norm map $\ro{N}_{\theta|k} \dopgleich \ro{N}_{(\mu)}: \ro{K}_n^{\ro{M}}(K) \ra \ro{K}_n^{\ro{M}}(k)$ for all $n \in \NN$. 
If $K|k$ is any finite extension, then we choose algebraic elements $\theta_1,\ldots,\theta_l$ such that $K = k(\theta_1,\ldots,\theta_l)$ and so we get a tower
\[
k < k(\theta_1) < k(\theta_1,\theta_2) < \dots < k(\theta_1,\ldots,\theta_l) = K
\]
of simple extensions. Now we define the norm $\ro{N}_{K|k}: \ro{K}_{n}^{\ro{M}}(K) \ra \ro{K}_n^{\ro{M}}(k)$ as the composition of the norm maps of the simple extensions. As mentioned above, it is a non-trivial result by \name{Kat\={o}} that this definition is independent of the tower of simple extensions. We denote the map $\ro{K}_*^{\ro{M}}(K) \ra \ro{K}_*^{\ro{M}}(k)$ given by the norm map in each degree again by $\ro{N}_{K|k}$. This map has the following properties:
\begin{enumerate}[label=(\roman*),noitemsep,nolistsep]
\item In degree zero $\ro{N}_{K|k}$ is multiplication by $\lbrack K:k \rbrack$.
\item In degree one $\ro{N}_{K|k}$ is the usual field norm.
\item $\ro{N}_{k|k} = \id_{\ro{K}_\star^{\ro{M}}(k)}$.
\item If $L|K$ is a further finite extension, then $\ro{N}_{K|k} \circ \ro{N}_{L|K} = \ro{N}_{L|k}$.
\item $\ro{N}_{K|k}( j_{K|k}(\alpha) \cup \beta) = \alpha \cup \ro{N}_{K|k}(\beta)$ for all $\alpha \in \ro{K}_\star^{\ro{M}}(k)$ and $\beta \in \ro{K}_\star^{\ro{M}}(K)$.
\item \label{item:milnor_k_norm_cohom} The compositum $\ro{N}_{K|k} \circ j_{K|k}: \ro{K}_n^{\ro{M}}(k) \ra \ro{K}_n^{\ro{M}}(k)$ is multiplication by $\lbrack K:k \rbrack$.
\end{enumerate}
\end{para}

\begin{prop}
Let $n \in \NN$, let $k$ be a field, let $G \dopgleich \gal(k)$ and let $\fr{M} \leq \ro{Grp}(G)^{\tn{i-f}}$ be a Mackey system. We identify the groups in $\fr{M}$ as fields under the Galois correspondence $H \leftrightarrow (k^{\ro{s}})^H$. Then the map $K \mapsto \ro{K}_n^{\ro{M}}(K)$ together with the following data define a cohomological Mackey functor $\ro{K}_{n}^{\ro{M}}: \fr{M} \ra \se{Ab}$:
\begin{compactitem}
\item $\con_{\sigma,K}^{\ro{K}_n^{\ro{M}}}: \ro{K}_{n}^{\ro{M}}(K) \ra \ro{K}_n^{\ro{M}}(\sigma K)$ is for each $K \in \msys{b}$ and $\sigma \in G$ the morphism induced by the corestriction $\sigma|_K:K \ra \sigma K$.
\item $\res_{L,K}^{\ro{K}_n^{\ro{M}}} \dopgleich j_{L|K}: \ro{K}_{n}^{\ro{M}}(K) \ra \ro{K}_n^{\ro{M}}(L)$ for all $K \in \msys{b}$ and $L \in \msys{r}(K)$.
\item $\ind_{K,L}^{\ro{K}_n^{\ro{M}}} \dopgleich \ro{N}_{L|K}: \ro{K}_{n}^{\ro{M}}(L) \ra \ro{K}_n^{\ro{M}}(K)$ for all $K \in \msys{b}$ and $L \in \msys{i}(K)$.
\end{compactitem}
\end{prop}

\begin{proof}
All relations not concerning the norm maps are evident. The cohomologicality relation is \ref{para:milnor_k_norm}\ref{item:milnor_k_norm_cohom}, so it remains to verify the equivariance of the norm maps and the Mackey formula. Both these relations are precisely \cite[chapter IX, exercises 3.3 and 3.4]{FesVos02_Local-fields_0}.
\end{proof}

\subsection{Sequential topologies} \label{sec:seq_tops}

\begin{para}
In this section we will review the Par\v{s}in topology on a higher local field introduced by \name{Aleksej Par\v{s}in} in his approach to higher local class field theory presented in \cite{0579.12012}.
\end{para}

\begin{para} \label{para:seq_topologies}
In \ref{para:complete_disc_field_props} we have seen that in a complete discrete valuation field $(k,v)$ any element $x$ has a unique representation as the limit of a series $\sum_{\mu \in \ZZ} a_\mu t^\mu$, where $t$ is a uniformizer and the $a_\mu$ are elements of a complete set of representatives of $\kappa_v = \calo_v/\fr{m}_v$ with $a_\mu = 0$ for all sufficiently small $\mu$. Moreover, each such series is convergent. This is a very important property which is extensively used and therefore it would be important if this would generalize to higher local fields. A proper generalization of these rank 1 phenomena to rank $n$ would be to require that in a local field $(k,v)$ of rank $n$ the following two conditions hold:
\begin{compactitem}[$\bullet$]
\item If $(t_1,\ldots,t_n)$ is a local system of parameters for $k$, then any family $(a_\mu t_1^{\mu_1} \dots t_n^{\mu_n})_{\mu=(\mu_1,\ldots,\mu_n) \in \Omega}$ with elements $a_\mu \in s(\kappa_v)$ is summable\footnote{Confer \cite[chapitre III, \S5]{Bou71_Topologie-Generale_0} for the concept of summability.} in the topological group $k^+$ with respect to the topology induced by the rank 1 valuation on $k$, where $s$ is a section of the quotient morphism $\calo_v \ra \kappa_v = k_0$ with $s(0) = 0$ and $\Omega \subs \ZZ^n$ is an \textit{admissible} subset meaning that for every $1 \leq l \leq n$ and every $j_{l+1},\ldots,j_n \in \ZZ$ there is an $i(j_{l+1},\ldots,j_n) \in \ZZ$ such that
\[
(i_1,\ldots,i_n) \in \Omega \tn{ and } i_{l+1} = j_{l+1}, \ldots, i_n = j_n \tn{ implies } i_l \geq i(j_{l+1},\ldots,j_n).
\]
The admissibility of $\Omega$ essentially characterizes that the family $a_\mu$ is obtained by inductive lifts of coefficients of series expansions in the residue fields as above because the coefficients of such series expansions vanish for small enough indices and this is the property the admissibility captures. Note that in particular for $n=1$ the admissibility is precisely the property that $\Omega$ is bounded from below.

\item Any $x \in k$ has a unique representation as the sum of family $(a_\mu t_1^{\mu_1} \dots t_n^{\mu_n})_{\mu=(\mu_1,\ldots,\mu_n) \in \Omega}$ as above. \\
\end{compactitem}

Unfortunately, this natural generalization does not work. To see this, consider the field $k \dopgleich \FF_q((T_1))((T_2))$. The set $\Omega = \lbrace (\nu,0) \mid \nu > 0 \rbrace \subs \ZZ^2$ is admissible and so the first property would imply that the family $(T_1^\nu)_{\nu > 0}$ is summable in $k$ with respect to the topology induced by the rank 1 valuation $w_2$ on $k$. A necessary condition for summability is that the sequence $(T_1^\nu)_{\nu > 0}$ converges to zero. But $w_2(T_1^\nu) = 0 < 1$ so that $T_1^\nu \notin k_{w_2>1}$ for all $\nu$ and as $k_{w_2>1}$ is an open neighborhood of $0 \in k$, this sequence does not converge to $0$ with respect to $w_2$. As the rank $2$ valuation on $k$ is finer than $w_2$, it also does not converge to $0$ with respect to this valuation which would have been even more natural than convergence with respect to the rank 1 valuation. \\

This shows that the ``valuative'' topologies that come with a local field of rank $n$, namely the topologies induced by the discrete valuations of ranks $1$ and $n$, do not induce a structure that resembles the structure of a classical local field, that is, a local field of rank 1. The intuitive explanation for this failure is that the valuative topologies are not aware of the topologies on the residue fields. But as the representation of elements as  unique limits of sums is of fundamental importance we should not abandon the desire to make this work. Instead, we should simply abandon the valuative topologies and equip a higher local field with a suitable topology that takes into account the topologies on the fields in the corresponding stack and makes all of the above possible. Precisely this was done by \name{Aleksej Par\v{s}in}. It is remarkable that his topology is still compatible with the group structure on the additive group but unfortunately the multiplication on the field is continuous only for $n = 1$ and is for $n > 1$ just sequentially continuous. This observation induced a complete change of perspective because it seems that sequential continuity is the right notion and that we have just been cosseted by the stronger condition of continuity in rank 1.
\end{para}

\begin{para}
The Par\v{s}in topology is defined for any higher local field but it is defined differently depending on the characteristics of the fields in the stack. We will only consider the positive characteristic case and although this case is significantly easier than the mixed characteristic cases, we can just state the definition and basic properties of the Par\v{s}in topology here. The reader is referred to \cite[\S1]{0579.12012} and \cite[\S1 and \S2]{MR1363290} for a detailed exposition and to \cite{MR1804916} for an overview. The first two references seem to be the only detailed accounts of this theory. Matthew Morrow, a student of Fesenko, started to write an introduction \cite{An-introduction-to-h00} to higher local fields which, although still in its beginnings, already contains some additional explanations which are necessary from the author's point of view.
\end{para}

\begin{para}
Let $\kappa$ be a field equipped with a topology such that the additive group $\kappa^+$ is a topological group. Let $(k,v)$ be a complete discrete valuation field with residue field $\kappa$ of the same characteristic as $k$. Let $\pi \in \calo_v$ be a prime element and let $s: \kappa \ra \lambda \subs \calo_v$ be a coefficient field with $s(0) = 0$ as mentioned in \ref{para:complete_disc_field_props}. Then $k = \lambda((\pi))$. For a family $\lbrace U_i \rbrace_{i \in \ZZ}$ of open neighborhoods of $0 \in \kappa$ with the property $U_i = \kappa$ for sufficiently large $i$ we define
\[
\ca{U}_{\lbrace U_i \rbrace} \dopgleich \left\lbrace \sum_{i \in \ZZ} s(a_i) \pi^i \mid a_i \in U_i \tn{ for all } i \right\rbrace.
\]

The family of all such sets is a compatible filter basis on the group $k^+$. The resulting topology is called the \word{Par\v{s}in topology} on $(k,v)$ with respect to $(\lambda,\pi)$.
\end{para}

\begin{para}
Let $(k,v)$ be a local field of rank $n$ and of characteristic $p > 0$. Let $(k,v)_* = \lbrace (k_i,w_i) \mid i = 1,\ldots,n \rbrace$ be the corresponding stack. The  \word{Par\v{s}in topology} on $(k,v)$ is defined as the topology obtained by equipping $k_0 = \kappa_{w_1}$ with the discrete topology, then choosing inductively a coefficient field $\lambda_i$ of $k_i$ and a prime element $\pi_i \in k_i$ and then equipping $k_i$ with the Par\v{s}in topology on $k_{i}$ with respect to $(\lambda_i,\pi_i)$ for all $i=1,\ldots,n$. This topology has the following properties:
\begin{enumerate}[label=(\roman*),noitemsep,nolistsep]
\item It is independent of the choice of prime elements and coefficient fields.\footnote{To the author it is not clear why in the literature the Par\v{s}in topology is in this case defined without first choosing a coefficient field because if $n > 1$ there is not a unique coefficient field due to the imperfectness of the residue field. The independence of the coefficient field is hidden in the proof concerning the independence of the prime elements given in \cite[\S1]{0579.12012}. As this is mentioned nowhere except at this point, it may be obvious. Our way of first choosing a coefficient field is then at least transferable to an equal characteristic 0 step and in this situation the Par\v{s}in topology indeed depends on the choice of coefficient fields as mentioned in \cite[\S1]{MR1363290}.}
\item $k^+$ is a separated and complete topological group with respect to this topology.
\item The summability properties discussed in \ref{para:seq_topologies} are satisfied.
\item The multiplication with a fixed element on $k$ is continuous.
\item If $n = 1$, then the topology is the valuation topology.
\item If $n>1$, then the multiplication $k \times k \ra k$ is not continuous but it is sequentially continuous.
\item If $n>1$, then the topology is not locally compact.
\end{enumerate}
\end{para}

\begin{para} \label{para:parshin_top_mult_group}
The desire to present principal units of a local field $(k,v)$ of rank $n$, that is, elements of $U_v^{(1)} = 1 + \fr{m}_v$, as a limit of a convergent product similar to the situation above led \name{Par\v{s}in} to also modify the topology on the multiplicative group $k^\times$. We again just sketch its definition in characteristic $p > 0$ and refer to the above references for details. Let $(t_1,\ldots,t_n)$ be a system of local parameters for $(k,v)$ and let $(k,v)_* = \lbrace (k_i,w_i) \rbrace$ be the corresponding stack. Then the sequence $1 \ra \calo_v^\times \ra k^\times \overset{v}{\ra} \ZZ^n \ra 0$ splits with a section of $v$ given by the map which sends the $i$-th unit vector to $t_i$. Hence, we have $k^\times = \calo_v^\times \times \langle t_1 \rangle \times \ldots \times \langle t_n \rangle$. Let $q:\calo_v \twoheadrightarrow \calo_v/\fr{m}_v = k_0 \cong \FF_q$ be the quotient morphism. It is easy to see that the sequence $1 \ra U_v^{(1)} \ra \calo_v^\times \overset{q}{\ra} k_0^\times \ra 1$ is exact. As $\calo_{w_n}$ is complete, the ring $\calo_v$ is as a closed subring of $\calo_{w_n}$ also complete. Now, it follows from \cite[chapter II, \S4, proposition 8]{Ser79_Local-Fields_0} that the map $q$ in this sequence admits a section (the Teichmüller map) and therefore the sequence splits, that is, $\calo_v^\times = U_v^{(1)} \times \ca{R}^\times$, where $\ca{R}^\times$ is the image of this section (the non-zero Teichmüller representatives). Hence, we have
\[
k^\times = U_v^{(1)} \times \ca{R}^\times \times \langle t_1 \rangle \times \ldots \times \langle t_n \rangle.
\]

The \word{Par\v{s}in topology} on $k^\times$ is now defined as the product topology of the topology on $U_v^{(1)}$ induced by the Par\v{s}in topology on $k$ and the discrete topology on the remaining factors. This topology has the following properties:
\begin{enumerate}[label=(\roman*),noitemsep,nolistsep]
\item It is independent of the choice of local parameters.
\item Multiplication by a fixed element on $k^\times$ is continuous.
\item \label{item:parshin_top_mult_group_mult_seq} The multiplication map $k^\times \times k^\times \ra k^\times$ and the inversion map $k^\times \ra k^\times$ are sequentially continuous.
\item Any element $\eps \in U_v^{(1)}$ has a unique presentation as the limit of a product $\prod_{\mu \in \Omega} (1 + \theta_\mu t_1^{\mu_1} \dots t_n^{\mu_n})$ with an admissible subset $\Omega \subs \ZZ^n$ and elements $\theta_\mu \in \ca{R}$.
\item If $n=1$, it coincides with the valuation topology induced on $k^\times$.\footnote{Perhaps due to its obviousness, the author could not find this statement in the literature. To see this, let $\alpha$ be a retraction of $U_v^{(1)} \ra \calo_v^\times$, let $\beta$ be a retraction of $\calo_v^\times \ra k^\times$ and let $\gamma:k_0^\times \ra \calo_v^\times$ be the Teichmüller map. Then the map $U_v^{(1)} \times \ca{R}^\times \times \langle t_1 \rangle \ra k^\times$, $(\eps,\theta,t_1^i) \mapsto \eps \theta t_1^i$ is an isomorphism with inverse $a \mapsto ( \alpha(\beta(a)), \gamma(q(\beta(a))), t_1^{v(a)})$. If we equip $k^\times$ with the valuation topology and $U_v^{(1)} \times \ca{R}^\times \times \langle t_1 \rangle$ with the product topology, where $U_v^{(1)}$ is equipped with the topology induced by $k^\times$ and the remaining factors are equipped with the discrete topology, then all maps are continuous due to the closed map lemma because $\calo_v$ is compact and so $k^\times$ is homeomorphic to $U_v^{(1)} \times \ca{R}^\times \times \langle t_1 \rangle$. The assertion  follows since the Par\v{s}in topology on $k$ coincides with the valuation topology.}
\item If $n \leq 2$, then multiplication and inversion are continuous so that $k^\times$ is a topological group.
\end{enumerate}
\end{para}

\begin{para} \label{para:top_milnor_k}
The modification of the topology on the multiplicative group motivates that this modification should be reflected in the Milnor $\ro{K}$-groups and due to \ref{para:parshin_top_mult_group}\ref{item:parshin_top_mult_group_mult_seq} the sequential continuity plays a central role in this modification. Again we can just give an overview and refer the reader to \cite[section 4]{MR1850194}. To present the abstract idea behind this, we will first define the topological Milnor $\ro{K}$-groups for any pair $(k,\tau)$ consisting of a field $k$ and a topology $\tau$ on $k^\times$. Let $n \in \NN_{>0}$ and consider the set $\ca{T}_n(k,\tau)$ of all topologies on $\ro{K}_n^{\ro{M}}(k)$ satisfying the following conditions:
\begin{enumerate}[label=(\roman*),noitemsep,nolistsep]
\item \label{item:top_milnor_k_symbol_cont} The universal symbol $\lbrace \cdot \rbrace: (k^\times)^n \ra \ro{K}_n^{\ro{M}}(k)$ is sequentially continuous with respect to the product topology on $(k^\times)^n$.
\item \label{item:top_milnor_k_add_cont} Addition and negation in $\ro{K}_n^{\ro{M}}(k)$ are sequentially continuous.
\end{enumerate}

The set $\ca{T}_n(k,\tau)$ is non-empty as it contains the indiscrete (weakest) topology. Hence, there exists the supremum topology $\ro{sup} \ca{T}_n(k,\tau)$ on $\ro{K}_n^{\ro{M}}(k)$ which has $\bigcup \ca{T}_n(k,\tau)$ as a subbase. It is not hard to see that this topology itself satisfies the above properties and that the intersection $\Lambda_n(k,\tau)$ of all neighborhoods of $0 \in \ro{K}_n^{\ro{M}}(k)$ with respect to this topology is a closed subgroup of $\ro{K}_n^{\ro{M}}(k)$. The \words{$n$-th topological Milnor $\ro{K}$-group}{Milnor $\ro{K}$-group!topological} of $(k,\tau)$ is now defined as the quotient $\ro{K}_n^{\ro{M}}(k,\tau) \dopgleich \ro{K}_n^{\ro{M}}(k)/\Lambda_n(k,\tau)$ equipped with the quotient topology and for $n = 0$ we define $\ro{K}_0^{\ro{M}}(k,\tau) \dopgleich \ZZ$. The sequentially continuous universal symbol $\lbrace \cdot \rbrace: (k^\times)^n \ra \ro{K}_n^{\ro{M}}(k,\tau)$ is universal among sequentially continuous symbols into separated groups. 
\end{para}

\begin{para}
If $\tau$ is $\ro{T}_1$ and if multiplication and inversion on $k^\times$ are sequentially continuous, then $\ro{K}_1^{\ro{M}}(k,\tau) = \ro{K}_1^{\ro{M}}(k)$. This is not explicitly mentioned in the references. To see this, note that the assumption implies that $\tau \in \ca{T}_1(k,\tau)$ as the universal symbol $\lbrace \cdot \rbrace: k^\times \ra k^\times = \ro{K}_1^{\ro{M}}(k)$ is just the identity and is  thus sequentially continuous. Hence, $\sup \ca{T}_1(k,\tau)$ contains $\tau$. But as $\tau$ is $\ro{T}_1$, the intersection of all neighborhoods of $1 \in k^\times$ is equal to $\lbrace 1 \rbrace$ and therefore also $\Lambda_1(k,\tau) = 1$.
\end{para}

\begin{para}
Let $(k,v)$ be a local field of rank $n$ and let $\tau$ be the Par\v{s}in topology on $k^\times$ from \ref{para:parshin_top_mult_group}. Then we define $\Lambda_m^{\ro{P}}(k) \dopgleich \Lambda_m(k,\tau)$ and call $\ro{K}_m^{\ro{MP}}(k) \dopgleich \ro{K}_m^{\ro{M}}(k,\tau)$ the $m$-th \word{Milnor--Par\v{s}in} $\ro{K}$-group of $k$.
\end{para}

\begin{para}
We have to warn the reader about the following: in the definition of topological Milnor $\ro{K}$-groups in \cite{0579.12012} it is required that the universal symbol is continuous, not only sequentially continuous. The author does not know about the implications of this difference. As Par\v{s}in only discusses the positive characteristic case, the universal symbol may already be continuous but Fesenko discusses in \cite{Fes92_Multidimensional-local_0} also just this case and there only sequential continuity is assumed. 
We stick to the sequential continuity because in this way the topological Milnor $\ro{K}$-groups are also defined in  \cite{MR1850194} and this paper is, according to its introduction, a replacement of the earlier work (including \cite{Fes92_Multidimensional-local_0}, \cite{Fes95_Abelian-local_0}, and \cite{0579.12012}) which ``$\lbrack \dots \rbrack$ corrects and clarifies some statements or proofs $\lbrack \dots \rbrack$''. Unfortunately, it is not explained which statements were wrong or were replaced, only that Par\v{s}in's description of subgroups of the topological Milnor $\ro{K}$-groups was wrong. In \cite{MR1850194} the topologies on $k^+$ and $k^\times$ are defined differently, but the topological Milnor $\ro{K}$-groups defined there coincide with the Milnor--Par\v{s}in $\ro{K}$-groups defined here according to the remark in \cite[section 4]{MR1850194}. 
\end{para}

\begin{para} \label{para:milnor_k_higher_local_max_div}
If $(k,v)$ is a local field of rank $n$, then it is proven in \cite[section 4]{MR1850194} that
\[
\Lambda_m^{\ro{P}}(k) = \bigcap_{l \in \NN_{>0} } l \ro{K}_m^{\ro{M}}(k)
\]
and 
\[
\ro{K}_m^{\ro{M}}(k)/l\ro{K}_m^{\ro{M}}(k) \cong \ro{K}_m^{\ro{MP}}(k)/l \ro{K}_m^{\ro{MP}}(k)
\]
for all $m,l \in \NN_{>0}$. This is an important result because it allows to consider the Milnor--Par\v{s}in $\ro{K}$-groups purely algebraically. It seems that the first relation was not known in Fesenko's first article \cite{Fes92_Multidimensional-local_0} on higher local class field theory because there only one inclusion is mentioned. 
\end{para}

\begin{prop} \label{para:top_milnor_k_functor}
Let $(k,v)$ be a local field of rank $n$ and let $m \in \NN_{>0}$. The following holds:
\begin{enumerate}[label=(\roman*),noitemsep,nolistsep]
\item If $K|k$ is a finite extension, then $j_{K|k}( \Lambda_m^{\ro{P}}(k)) \subs \Lambda_m^{\ro{P}}(K)$ and $\ro{N}_{K|k}( \Lambda_m^{\ro{P}}(K)) \subs \Lambda_m^{\ro{P}}(k)$.
\item Let $G = \gal(k)$ and let $\fr{M} \leq \ro{Grp}(G)^{\tn{i-f}}$ be a Mackey system. Then the map $\Lambda_m^{\ro{P}}: \msys{b} \ra \se{Ab}$, $K \mapsto \Lambda_m^{\ro{P}}(K)$, is a subfunctor of $\ro{K}_m^{\ro{M}}: \fr{M} \ra \se{Ab}$. Hence, $\ro{K}_m^{\ro{MP}} \dopgleich \ro{K}_m^{\ro{M}} / \Lambda_m^{\ro{P}}: \fr{M} \ra \se{Ab}$ is a cohomological Mackey functor.
\end{enumerate}
\end{prop}

\begin{proof} \hfill

\begin{asparaenum}[(i)]
\item First note that $K$ is canonically a local field of rank $n$ by \ref{para:higher_local_ext_closed} so that $\ro{K}_m^{\ro{PM}}(K)$ is defined. Since $j_{K|k}(l \ro{K}_m^{\ro{M}}(k)) \subs l \ro{K}_m^{\ro{M}}(K)$ and $\ro{N}_{K|k}( l \ro{K}_m^{\ro{M}}(K)) \subs  l \ro{K}_m^{\ro{M}}(k)$, the relations are evident by \ref{para:milnor_k_higher_local_max_div}.

\item Due to the above relations, it just remains to show that $\con_{\sigma,K}^{\ro{K}_m^{\ro{M}}}( \Lambda_m^{\ro{P}}(K)) \subs \Lambda_m^{\ro{P}}(\sigma K)$. But this is again evident because of \ref{para:milnor_k_higher_local_max_div}.
\end{asparaenum} \vspace{-\baselineskip}
\end{proof}

\begin{para}
As indicated above, the algebraic presentation of $\Lambda_m^{\ro{P}}(k)$ in \ref{para:milnor_k_higher_local_max_div} was probably not available in \cite{Fes92_Multidimensional-local_0}. The author does not know how \ref{para:top_milnor_k_functor} can be proven without this presentation. This is not mentioned in \cite{Fes92_Multidimensional-local_0}.
\end{para}

\begin{para}
An important justification of our functorial approach to class field theories is that the functors $\ro{K}_m^{\ro{M}}$ and $\ro{K}_m^{\ro{MP}}$ do not necessarily have Galois descent for $m > 1$. A counterexample was given by \name{Ivan Fesenko} in \cite[\S8]{MR1628793} for a local field of rank $2$ and of characteristic 0. In private communication between Fesenko and the author, Fesenko explained that also a counterexample in characteristic $p > 0$ can  be obtained by using an extension $k(\sqrt[\ell]{\pi})$ of a local field $k$ of rank $2$ and of characteristic $p > 0$ containing a non-trivial $\ell$-th root of unity with a prime number $\ell \neq p$ and a uniformizer $\pi$.
\end{para}

\newpage
\section{Fesenko--Neukirch class field theory for certain types of discrete valuation fields} \label{sect:cfts_for_disc_val}

In this chapter the material of chapter \ref{chap:disc_vals} is used to explain how classical local class field theory and higher local class field theory are obtained as Fesenko--Neukirch class field theories.


\subsection{Local fields} \label{sect:local_cft}

\begin{para}
Local class field theory is a Fesenko--Neukirch class field theory for a local field $k$ of arbitrary characteristic. The modeling of the finite abelian extensions takes place inside the multiplicative group $(k^{\ro{s}})^\times = \ro{GL}_1(k^{\ro{s}})$ of the separable closure of $k$.
\end{para}

\begin{ass}
Throughout this section $(k,v)$ is a local field and $G \dopgleich \gal(k)$. We assume without loss of generality that $v$ is normalized and for an algebraic extension $K|k$ we denote by $v_K$ the normalization of the unique extension $\widetilde{v}_K$ of $v$ to $K$. Extensions of $k$ are always assumed to be separable if nothing else is mentioned. We set $q \dopgleich \# \kappa_v$ and identify $\gal(\kappa_v)$ with $\widehat{\ZZ}$ via the canonical isomorphism mapping the Frobenius automorphism $\ro{Fr}_q$ to $1 \in \widehat{\ZZ}$. We set $\fr{E} \dopgleich \ro{Sp}(G)^{\tn{f}}$, $\fr{K} \dopgleich \fr{K}(\fr{E}) = \fr{K}(G)_{\tn{ab}}^{\tn{f}}$, $\fr{M} \dopgleich \ro{Grp}(G)^{\tn{f}}$ and $\ro{GL}_1(k^{\ro{s}})_* \dopgleich \ro{H}_{\fr{M}}^0(\ro{GL}_1(k^{\ro{s}})) \in \se{Mack}^{\ro{c}}(\fr{M},\se{Ab})$.
\end{ass}

\begin{para}
We note that $\ro{GL}_1(k^{\ro{s}})$ is indeed a \textit{discrete} $\gal(k)$-module with the obvious action of automorphisms so that the above is well-defined. Moreover, due to \cite[chapter I, proposition 2.6(iii)]{Neu99_Algebraic-Number_0} and the definition of $\ro{GL}_1(k^{\ro{s}})_*$ we have $\ind_{K,L}^{\ro{GL}_1(k^{\ro{s}})_*} = \ro{N}_{L|K}$ for all $K \in \msys{b}$ and $L \in \msys{i}(K)$, where $\ro{N}_{L|K}:L^\times \ra K^\times$ is the usual field norm (here we identify the groups in $\msys{b}$ as finite separable extensions of $k$ using the Galois correspondence as mentioned in \ref{para:cfts_for_profinite}).
\end{para}

\begin{para} \label{para:ram_theory_local_standard}
Let $\delta \dopgleich d_k:\gal(k) \twoheadrightarrow \widehat{\ZZ}$ be the epimorphism defined in \ref{para:ram_theory_of_real_henselian}. We always consider $1 \in \widehat{\ZZ}$ as the canonical topological generator and thus $\delta$ induces an abstract ramification theory on $G$ including a notion of Frobenius elements. To be careful, we first attach a $\delta$ in the notation of all notions induced by this ramifications theory. For example we write $e_{L|K}^{(\delta)}$ for the ramification index $e_{\gal(K)|\gal(L)}$ defined by $\delta$ and similarly for the inertia index.

Now let $K|k$ and $L|k$ be two extensions with $L \sups K$. Then according to \ref{para:ram_theory_of_local} we have $f_{L|K}^{(\delta)} = f_{L|K}$ and so the $\delta$-inertia indices correspond to the valuation theoretic inertia indices. Moreover, we have $I_K^{(\delta)} = I_K$ and therefore the $\delta$-inertia subgroups correspond to the valuation theoretic inertia subgroups. Also, $L|K$ is $\delta$-unramified if and only if $L|K$ is valuation theoretically unramified. As $\Gamma_K = \gal(K)/I_K  \cong \gal(K^{\tn{ur}}|K)$, the maximal $\delta$-unramified quotient of $K$ corresponds to the maximal valuation theoretically unramified extension $K^{\tn{ur}}$ of the henselian field $K$. We have $f_K^{(\delta)} = f_{K|k}$ and if $f_{K|k} < \infty$, then we have
\[
\delta_K = \frac{1}{f_K^{(\delta)}} \delta|_{\gal(K)} = \frac{1}{f_{K|k}} (d_k)|_{\gal(K)} =  d_{K|k}.
\]
We remark here, that this equality was the reason for writing $\delta$ instead of $d$ because $d_K$ is also defined but is different from $d_{K|k}$ as mentioned in \ref{para:ram_theory_of_local}.

If $K|k$ is finite, then $\varphi_K^{(\delta)} = (\delta_K')^{-1}(1) = (d_{K|k}')^{-1}(1) = \varphi_K$ and this shows that for finite extensions the absolute Frobenius elements of $\delta$ correspond to the valuation theoretic absolute Frobenius automorphisms. If $L|K$ is finite, then according to \ref{para:ram_and_inert_props} and \ref{para:ramification_index} we have
\[
e_{L|K}^{(\delta)} \cdot f_{L|K}^{(\delta)} = \lbrack \gal(K):\gal(L) \rbrack = \lbrack L:K \rbrack = e_{L|K} \cdot f_{L|K} = e_{L|K} \cdot f_{L|K}^{(\delta)}
\]
and therefore $e_{L|K}^{(\delta)} = e_{L|K}$. We note that both the abstract and the valuation theoretic ramification theory provide the notion of ramification indices for infinite extensions but in this situation it is not clear if they coincide. We summarize, that at least for finite extensions of a finite extension of $k$, all notions provided by the abstract ramification theory induced by $\delta$ coincide with the valuation theoretic notions.
\end{para}

\begin{prop} \label{para:disc_val_is_val_on_mult_group}
$\boldsymbol{v} \dopgleich \lbrace v_K \mid K \in \msys{b} \rbrace \in \ro{Val}_d^{\ZZ,1}(\ro{GL}_1(k^{\tn{s}})_*) \subs \ro{Val}_d^{\widehat{\ZZ},1}(\ro{GL}_1(k^{\tn{s}})_*)$.
\end{prop}

\begin{proof}
It follows from \ref{para:discrete_normalization} that
\[
v_K = e_{K|k}\widetilde{v}_K = e_{K|k} \frac{1}{\lbrack K:k \rbrack} v \circ \ro{N}_{K|k} = \frac{1}{f_{K|k}} v \circ \ro{N}_{K|k} = \frac{1}{f_K^{(\delta)}} v \circ \ind_{K,L}^{\ro{GL}_1(k^{\tn{s}})_*},
\]
and therefore
\[
f_K^{(\delta)} \ZZ = f_{K|k} \ZZ = f_{K|k} v(k^\times) = f_{K|k} (\lbrack \widetilde{v}_K(K^\times):v(k^\times) \rbrack \widetilde{v}_K(K^\times)) = f_{K|k} e_{K|k} \widetilde{v}_K(K^\times) = v( \ind_{K,L}^{\ro{GL}_1(k)_*}(K^\times)).
\]
Hence, the assertion follows from an application of \ref{para:single_valuation_induces_family} to $v:\ro{GL}_1(k^{\tn{s}})_*(G) = k^\times \ra \ZZ$.
\end{proof}

\begin{thm} \label{para:local_cft_main_theorems}
The following holds:
\begin{enumerate}[label=(\roman*),noitemsep,nolistsep]
\item $\ro{GL}_1(k^{\tn{s}})_*$ satisfies the class field axiom for all $(H,U) \in \esys{b}^{\tn{cyc}}$.
\item If $K \in \msys{b}$, then the groups $\im( \ind_{K,L}^{\ro{GL}_1(k^{\tn{s}})_*}) \subs K^\times$ with $L \in \ro{Ext}(\fr{E},K)$ are precisely the open subgroups of finite index of $K^\times$ with respect to the topology induced by $v_K$.
\end{enumerate}

\end{thm}

\begin{proof}
The first assertion is \cite[chapter V, theorem 1.1]{Neu99_Algebraic-Number_0} and the second assertion follows from \cite[chapter IV, theorem 6.2]{FesVos02_Local-fields_0}.
\end{proof}

\begin{cor} \label{para:local_cft}
$(\ro{GL}_1(k^{\tn{s}})_*,\boldsymbol{v}) \in \ro{FND}_d^{1}(\fr{E})$ and the corresponding Fesenko--Neukirch morphism 
\[
\Upsilon: \pi_{\fr{K}} \ra \tateco_{\fr{E}}^0(\ro{GL}_1(k^{\tn{s}})_*)
\]
is a $\fr{K}$-class field theory, called \word{local class field theory}.
\end{cor}

\begin{proof}
This follows from \ref{para:local_cft_main_theorems} in combination with \ref{para:fs_cft_main_classical_case}.
\end{proof}

\begin{para}
To make explicit what the content of \ref{para:local_cft} is, let $K|k$ be a finite  extension and let $L|K$ be a finite Galois extension, that is, $K \in \msys{b}$ and $L \in \ro{Ext}(\fr{E},K)$. The the Fesenko--Neukirch morphism
\[
\Upsilon_{L|K}: \gal(L|K)^{\tn{ab}} \lra K^\times/\ro{N}_{L|K} L^\times
\]
mapping $\sigma \modd \comm{a}(\gal(L|K))$ to $\ro{N}_{\Sigma|K}( \pi_\Sigma )$ is a canonical isomorphism, where $\Sigma$ is the fixed field of the restriction $\widetilde{\sigma}|_{L^{\tn{ur}}}$ of a Frobenius lift $\widetilde{\sigma} \in \gal(K)$ of $\sigma$ and $\pi_\Sigma \in \Sigma$ is a uniformizer. If $L|K$ is unramified, then $\varphi_{L|K} \dopgleich (\varphi_K)|_L \in \gal(L|K)$ is a generator and $\Upsilon_{L|K}(\varphi_{L|K}^n) = \pi_K^n$ for all $n \in \ZZ$, where $\pi_K \in K$ is a uniformizer. 
The assignment $L \mapsto \ro{N}_{L|K} L^\times \subs K^\times$ yields a lattice isomorphism between the lattice of abelian extensions of $K$ and the lattice of open subgroups of finite index of $K^\times$. Finally, the isomorphisms $\Upsilon_{L|K}$ are compatible with the conjugation, inclusion and norm maps. 
\end{para}

\begin{para}
In the next few paragraphs we want to sketch how class field theory is used in practice and as an example we prove the local Kronecker-Weber theorem to present at least one explicit result obtained by class field theory (although this particular result can also be derived without class field theory).
\end{para}

\begin{para}
One general technique of class field theory is to calculate the class field corresponding to an open subgroup $\ca{N}$ of finite index of $k^\times$, that is, to determine a (due to local class field theory already unique) finite abelian extension $K|k$ with $\ro{N}_{K|k} K^\times = \ca{N}$, and to transfer ``lattice-theoretic'' properties of $\ca{N}$ to statements about extensions of $k$. Consider for example the subgroups $\langle \pi^f \rangle \cdot U_v^{(n)} \leq k^\times$, where $\pi \in k$ is a prime element, $f \in \NN_{>0}$ and $n \in \NN$. As $U_v^{(n)}$ is open in $k^\times$, these groups are open in $k^\times$. These groups are also of finite index in $k^\times$. To see this, write an element $x \in k^\times$ as $x = \pi^l \eps$ with $\eps \in \calo_v^\times$ and $l \in \NN$, and use the fact that $\calo_v^\times/U_v^{(1)} \cong \kappa_v^\times = \FF_q^\times$ and $U_v^{(i)}/U_v^{(i+1)} \cong \kappa_v = \FF_q$ for all $i \in \NN_{>0}$ (confer \cite[chapter II, proposition 3.10]{Neu99_Algebraic-Number_0}) implying that a certain power of $\eps$ is contained in $U_v^{(n)}$. Now, if $\ca{N}$ is any open subgroup of finite index of $k^\times$, then it contains some group of the form $\langle \pi^f \rangle \cdot U_v^{(n)}$. To see this, note that the $U_v^{(n)}$ form a neighborhood basis of $1 \in k^\times$ according to \ref{para:disc_val_def} so that $U_v^{(n)} \subs \ca{N}$ for some $n \in \NN$, and as $\pi^{\lbrack k^\times: \ca{N} \rbrack} \in \ca{N}$, the assertion follows. Hence, class field theory implies that \textit{any} finite abelian extension of $k$ is contained in the class field of one of the subgroups $\langle \pi^f \rangle \cdot U_v^{(n)} \leq k^\times$. Therefore the determination of these class fields will provide important information about abelian extensions. These class fields are determined explicitly by Lubin--Tate theory (confer \cite[chapter V, \S5]{Neu99_Algebraic-Number_0}) and in the next paragraph we will give some first ideas about these class fields.
\end{para}

\begin{prop} \label{para:some_lubin_tate_class_fields}
Let $\pi \in k$ be a prime element and let $f \in \NN_{>0}$ and let $n \in \NN$. The following holds:
\begin{enumerate}[label=(\roman*),noitemsep,nolistsep]
\item The class field of $\langle \pi^f \rangle \cdot U_v^{(n)} \subs k^\times$ is equal to the composite of the class field of $\langle \pi^f \rangle \cdot U_v^{(0)} \subs k^\times$ and the class field of $\langle \pi \rangle \cdot U_v^{(n)} \subs k^\times$.
\item The class field of the group $\langle \pi^f \rangle \cdot U_v^{(0)} \subs k^\times$ is the unramified extension of degree $f$ of $k$.  
\item Let $p$ be a prime number and suppose that $(k,v) = \QQ_p$. Then the class field of the group $\langle p \rangle \cdot U_v^{(n)} \subs k^\times$ is $\QQ_p(\zeta_{p^n})$, where $\zeta_{p^n}$ is a primitive $p^n$-th root of unity. 
\end{enumerate}
\end{prop}

\begin{proof} \hfill

\begin{asparaenum}[(i)]
\item This follows immediately from the equality
\[
\langle \pi^f \rangle \cdot U_v^{(n)} = (\langle \pi^f \rangle \cdot U_v^{(0)}) \cap (\langle \pi \rangle \cdot U_v^{(n)}).
\]

\item Let $K$ be the unramified extension of degree $f$ of $k$. We have to show that $\ro{N}_{K|k} K^\times = \langle \pi^f \rangle \cdot U_v^{(0)}$ and this follows immediately from
\[
v( \ro{N}_{K|k} K^\times) = f_{K|k} v_K(\ZZ) = f \ZZ = v( \langle \pi^f \rangle).
\]

\item This is \cite[chapter V, proposition 1.8]{Neu99_Algebraic-Number_0}. \vspace{-\baselineskip}
\end{asparaenum}
\end{proof}

\begin{cor}
Let $p$ be a prime number. Then every finite abelian extension of $\QQ_p$ is contained in a field $\QQ_p(\zeta)$ for some primitive root of unity $\zeta$. In particular, the maximal abelian extension $\QQ_p^{\tn{ab}}$ of $\QQ_p$ is equal to $\QQ_p( \lbrace \zeta_n \mid n \in \NN_{>2} \rbrace)$, where $\zeta_n$ is a primitive $n$-th root of unity.
\end{cor}

\begin{proof}
Let $(k,v) = \QQ_p$ and let $K|k$ be a finite abelian extension. Let $f \in \NN_{>0}$ and $n \in \NN$ such that $\langle p^f \rangle \cdot U_v^{(n)} \subs \ro{N}_{K|k}K^\times \subs k^\times$. Then according to \ref{para:some_lubin_tate_class_fields} the field $K$ is contained in the compositum of $\QQ_p(\zeta_{p^n})$ and the unramified extension of degree $f$ of $k$ which is equal to $\QQ(\zeta_{p^{f}-1})$ by \cite[chapter 2, proposition 7.13]{Neu99_Algebraic-Number_0}. Hence, $K \subs \QQ_p(\zeta_{p^f-1}, \zeta_{p^n}) = \QQ_p(\zeta_{(p^f-1)p^n})$. The second assertion follows since any extension $\QQ_p(\zeta_n)|\QQ_p$ is abelian.
\end{proof}

\begin{para}
The \word{conductor} of a finite abelian extension $K|k$ is defined as $\fr{f}_{K|k} = \fr{m}_v^n$, where $n \in \NN$ is minimal with the property that $U_v^{(n)} \subs \ro{N}_{K|k} K^\times$. Note that this is well-defined as $\ro{N}_{K|k} K^\times \subs k^\times$ is open and the higher unit groups $U_v^{(n)}$ form a neighborhood basis of $1 \in k^\times$. In the next paragraph we will see that the conductor detects if an extension is unramified and as the conductor can be calculated within $k^\times$, this is a hint towards the deep arithmetic information contained in a class field theory mentioned at the very beginning.
\end{para}

\begin{cor}
A finite abelian extension $K|k$ is unramified if and only if $\fr{f}_{K|k} = 1$.
\end{cor}

\begin{proof}
This is \cite[chapter V, proposition 1.7]{Neu99_Algebraic-Number_0} but we repeat the proof here. If $K|k$ is unramified, then as $(\ro{GL}_1(k^{\ro{s}})_*,\boldsymbol{v}) \in \ro{FND}_d^1(\fr{E})$, it follows from \ref{para:fn_datum}\ref{item:fn_datum_sequence} that 
\[
U_v^{(0)} = \calo_v^\times = \ker(v) = \ker(v_k) = \ro{N}_{K|k} \ker(v_K) \subs \ro{N}_{K|k} K^\times 
\]
and consequently $\fr{f}_{K|k} = \fr{m}_v^0 = 1$. Conversely, suppose that $\fr{f}_{K|k} = 1$. Then $U_v^{(0)} \subs \ro{N}_{K|k} K^\times$ by definition. Let $n = \lbrack k^\times: \ro{N}_{K|k} K^\times \rbrack$ and let $\pi \in k$ be a uniformizer. Then $\pi^n \in \ro{N}_{K|k} K^\times$ and therefore $\langle \pi^n \rangle \cdot U_v^{(0)} \subs \ro{N}_{K|k} K^\times$. If $M|k$ is the unramified extension of degree $n$, then according to \ref{para:some_lubin_tate_class_fields}  we have
\[
\ro{N}_{M|k} M^\times = \langle \pi^n \rangle \cdot U_v^{(0)} \subs \ro{N}_{K|k} U_{v_K}^{(0)} \subs \ro{N}_{K|k} K^\times
\]
and therefore $M \sups K$, that is, $K|k$ is unramified.
\end{proof}

\subsection{Local fields of higher rank and of positive characteristic}

\begin{para}
In this section we will explain what the input to \name{Fesenko}'s higher local class field theory in positive characteristic is. To get an idea of this theory, let $(k,v)$ be a local field of rank $1$. The local class field theory for $(k,v)$ was obtained as a Fesenko--Neukirch theory with respect to the ramification theory
\[
\delta = d_k: \gal(k) \twoheadrightarrow \gal(k^{\tn{ur}}|k) \overset{\cong}{\ra} \gal(\kappa_v) \cong \widehat{\ZZ}
\]
and with respect to the Fesenko--Neukirch datum $(\ro{GL}_1(k^{\ro{s}})_*,\boldsymbol{v})$ as demonstrated in the previous section. Now, we can also write $\ro{GL}_1(k^{\ro{s}})(K) = K^\times = \ro{K}_1^{\ro{M}}(K)$ and we can consider the valuation $v:k^\times \ra \ZZ$, which induces the valuation $\boldsymbol{v}$, as the tame symbol $\partial_{v}^{\ro{M}}: \ro{K}_1^{\ro{M}}(k) \ra \ro{K}_0^{\ro{M}}(\kappa_v) = \ZZ$ by \ref{para:milnor_k_norm}. If we would completely empty our mind, then we could come up with the idea to consider for a local field $(k,v)$ of rank $n$ with corresponding stack $(k,v)_* = \lbrace (k_i,w_i) \rbrace$ the ramification theory induced by the composition
\[
\delta: \gal(k) = \gal(k_n) \twoheadrightarrow \gal(k_n^{\tn{ur}}|k_n) \overset{\cong}{\ra} \gal(k_{n-1}) \twoheadrightarrow \dots \overset{\cong}{\ra} \gal(k_1) \twoheadrightarrow \gal(k_1^{\tn{ur}}|k_1) \overset{\cong}{\ra} \gal(k_0) = \widehat{\ZZ},
\]
where $k_i^{\tn{ur}}$ denotes the maximal unramified extension with respect to $w_i$, and moreover we could consider $\ro{K}_n^{\ro{M}}(k)$ as being the right generalization of $k^\times$ to rank $n$ and we could consider as valuation the composition
\[
\widetilde{v}: \ro{K}_n^{\ro{M}}(k) = \ro{K}_n^{\ro{M}}(k_n) \overset{\partial_{w_n}^{\ro{M}}}{\lra} \ro{K}_{n-1}^{\ro{M}}(k_{n-1}) \overset{\partial_{w_{n-1}}^{\ro{M}}}{\lra} \dots \overset{\partial_{w_{2}}^{\ro{M}}}{\lra} \ro{K}_1^{\ro{M}}(k_1) \overset{\partial_{w_{1}}^{\ro{M}}}{\lra} \ro{K}_0^{\ro{M}}(k_0) = \ZZ.
\]

Surprisingly, this idea indeed nearly works but two modifications are necessary. First, the Milnor $\ro{K}$-groups have to be replaced by the Milnor--Par\v{s}in $\ro{K}$-groups and second, this approach only works in positive characteristic. In the following paragraphs we will explain a few additional details so that all what is left to do to get this class field theory is to verify \ref{para:fs_cft_main_case}, although this requires serious work.
\end{para}

\begin{ass}
Throughout this section $(k,v)$ is a local field of rank $n$ and of characteristic $p>0$. The corresponding stack is $(k,v)_* = \lbrace (k_i,w_i) \rbrace$ and the absolute Galois group is $G \dopgleich \gal(k)$. Extensions of $k$ are always assumed to be separable. We set $q \dopgleich \# \kappa_v$ and identify $\gal(\kappa_0)$ with $\widehat{\ZZ}$ via the canonical isomorphism mapping the Frobenius automorphism $\ro{Fr}_q$ to $1 \in \widehat{\ZZ}$. We set $\fr{E} \dopgleich \ro{Sp}(G)^{\tn{f}}$, $\fr{K} \dopgleich \fr{K}(\fr{E}) = \fr{K}(G)_{\tn{ab}}^{\tn{f}}$ and $\fr{M} \dopgleich \ro{Grp}(G)^{\tn{f}}$.
\end{ass}

\begin{para}
We set up a ramification theory for $(k,v)$ by inductively applying the results of \ref{para:ram_theory_of_real_henselian} to the complete discrete valuation fields $(k,w_i)$. For each $1 \leq i \leq n$ we have the epimorphism $d_{k_i}: \gal(k_i) \twoheadrightarrow \gal(k_{i-1})$ given by the composition of the quotient morphism $\gal(k_i) \ra \gal(k_i^{\tn{ur}}|k_i)$ and the isomorphism $\gal(k_i^{\tn{ur}}|k_i) \ra \gal(k_{i-1})$ as in \ref{para:ram_theory_of_real_henselian}. Then $\delta_i \dopgleich d_{k_1} \circ \dots \circ d_{k_i} : \gal(k_i) \twoheadrightarrow \gal(k_0) = \widehat{\ZZ}$ induces a ramification theory on $\gal(k_i)$ for all $1 \leq i \leq n$. We set $\delta \dopgleich \delta_n: \gal(k) \twoheadrightarrow \widehat{\ZZ}$.
\end{para}

\begin{conv}
All notions concerning ramification which do not carry one of the epimorphisms $\delta_i$ in their notation, will always refer to the rank $1$ valuation.
\end{conv}

\begin{prop} \label{para:higher_ram_theory}
Let $L|K$ be two finite separable extensions of $k$. Let $\lbrace (K_i,u_i) \rbrace$ respectively $\lbrace (L_i,s_i) \rbrace$ be the stack corresponding to $K$ respectively $L$ as in \ref{para:higher_local_ext_closed}. The following holds:
\begin{enumerate}[label=(\roman*),noitemsep,nolistsep]
\item \label{item:higher_ram_theory_inert} $f_{L|K}^{(\delta)} = f_{L_1|K_1} = \lbrack L_0:K_0 \rbrack$.
\item $e_{L|K}^{(\delta)} = \prod_{i=1}^n e_{L_i|K_i}$.
\end{enumerate}
\end{prop}

\begin{proof} \hfill

\begin{asparaenum}[(i)]
\item Since $\delta$ is surjective, we have $\delta(K) = \delta_{K_1}(K_1) = I_{K_1}^{(\delta_1)}$ and similarly $\delta(L) = I_{L_1}^{(\delta_1)}$. Now it follows from \ref{para:ram_theory_local_standard} that 
\[
f_{L|K}^{(\delta)} = \lbrack \delta(K):\delta(L) \rbrack = \lbrack I_{K_1}^{(\delta_1)}: I_{L_1}^{(\delta_1)} \rbrack = f_{L_1|K_1} = \lbrack L_0:K_0 \rbrack.
\]

\item We prove this by induction on $n$. According to \ref{para:ram_theory_local_standard} the relation holds for $n = 1$. Now let $n > 1$. Using \ref{para:ram_and_inert_props} and \ref{para:ramification_index} we get
\begin{align*}
f_{L_1|K_1} e_{L|K}^{(\delta)} &= f_{L|K}^{(\delta)} e_{L|K}^{(\delta)} = \lbrack L:K \rbrack = f_{L|K} e_{L|K} = \lbrack L_{n-1}: K_{n-1} \rbrack \cdot e_{L|K} \\
& = f_{L_{n-1}|K_{n-1}}^{(\delta_{n-1})} \cdot e_{L_{n-1}|K_{n-1}}^{(\delta_{n-1})} e_{L|K} = f_{L_1|K_1} \cdot e_{L|K} \cdot \prod_{i=1}^{n-1} e_{L_i|K_i}.
\end{align*}
\end{asparaenum}
\end{proof}

\begin{para}
A $\delta$-unramified extension $L|K$ is usually called \word{purely unramified} in the literature. By the above this is equivalent to $\prod_{i=1}^n e_{L_i|K_i} = 1$. We note that at least in the higher rank theory it is important to make precise to which ramification theory the notions belong to because there exist further ramification theories on a higher local field.
\end{para}

\begin{para}
From the presentation of the Milnor--Par\v{s}in $\ro{K}$-groups in \ref{para:milnor_k_higher_local_max_div} it follows that the tame symbol $\partial^{\ro{M}}$ induces a morphism $\partial^{\ro{MP}}$ between Milnor--Par\v{s}in $\ro{K}$-groups. Hence, we can define the composition
\[
\widetilde{v}: \ro{K}_n^{\ro{MP}}(k) = \ro{K}_n^{\ro{MP}}(k_n) \overset{\partial_{w_n}^{\ro{MP}}}{\lra} \ro{K}_{n-1}^{\ro{MP}}(k_{n-1}) \overset{\partial_{w_{n-1}}^{\ro{MP}}}{\lra} \dots \overset{\partial_{w_{2}}^{\ro{MP}}}{\lra} \ro{K}_1^{\ro{MP}}(k_1) \overset{\partial_{w_{1}}^{\ro{MP}}}{\lra} \ro{K}_0^{\ro{MP}}(k_0) = \ZZ.
\]
\end{para}

\begin{prop}
The map $\widetilde{v}: \ro{K}_n^{\ro{MP}}(k) \ra \ZZ$ satisfies $\widetilde{v}( \ro{N}_{K|k}( \ro{K}_n^{\ro{MP}}(K))) = f_K^{(\delta)} \ZZ$ for any finite extension $K|k$ and thus induces according to \ref{para:single_valuation_induces_family} a valuation $\widetilde{\boldsymbol{v}} \in \ro{Val}_d^{\ZZ,1}(\ro{K}_n^{\ro{MP}})$.
\end{prop}

\begin{proof}
We prove this by induction on the rank $n$. For $n = 1$ we have $\widetilde{v} = v: \ro{K}_m^{\ro{MP}}(k) = k^\times \ra \ZZ$ according to \ref{para:milnor_k_norm} and so it follows from the proof of \ref{para:disc_val_is_val_on_mult_group} that the relation holds. Let $n > 1$ and let $\lbrace (K_i,u_i) \rbrace$ be the stack corresponding to $K$. The assertion then follows from \ref{para:higher_ram_theory}\ref{item:higher_ram_theory_inert} and the commutativity of the diagram
\[
\xymatrix{
\ro{K}_n^{\ro{MP}}(K) \ar[r]^-{\partial^{\ro{MP}}_{u_n} } \ar[d]_{\ro{N}_{K|k}} & \ro{K}_{n-1}^{\ro{MP}}(K_{n-1}) \ar[d]^{\ro{N}_{K_{n-1}|k_{n-1}}} \\
\ro{K}_n^{\ro{MP}}(k) \ar[r]_-{\partial^{\ro{MP}}_{v_n} } & \ro{K}_{n-1}^{\ro{MP}}(k_{n-1})
}
\]
which was proven in \cite[chapter IX, theorem 3.7]{FesVos02_Local-fields_0} in the case of the non-topological Milnor-$\ro{K}$ groups and so obviously also holds for the topological ones.
\end{proof}

\begin{thm} \label{para:higher_local_cft}
The pair $(\ro{K}_n^{\ro{MP}}, \widetilde{\boldsymbol{v}})$ satisfies all the properties in \ref{para:fs_cft_main_case} on $\fr{E}$. Hence, $(\ro{K}_n^{\ro{MP}}, \widetilde{\boldsymbol{v}}) \in \ro{FND}_d^1(\fr{E})$ and the induced Fesenko--Neukirch morphism $\Upsilon: \pi_{\fr{K}} \ra \tateco_{\fr{E}}^0(\ro{K}_n^{\ro{MP}})$ is a $\fr{K}$-class field theory.
\end{thm}

\begin{proof}
Property \ref{para:fs_cft_main_case}\ref{item:fs_cft_main_case_cft_ax} (the class field axiom) for unramified extensions follows from \cite[theorem 3.3]{Fes92_Multidimensional-local_0} and for the totally ramified extensions of prime degree it follows from \cite[theorems 3.1 and 3.2]{Fes92_Multidimensional-local_0}. Property \ref{para:fs_cft_main_case}\ref{item:fs_cft_main_case_gal_desc} (the Galois descent for unramified extensions) follows from \cite[theorem 3.3]{Fes92_Multidimensional-local_0}. Finally the injectivity of $\Upsilon$ for totally ramified extensions of prime degree follows from \cite[theorem 4.1]{Fes92_Multidimensional-local_0}.
\end{proof}

\begin{para}
To make explicit what the content of \ref{para:higher_local_cft} is, let $K|k$ be a finite  extension and let $L|K$ be a finite Galois extension, that is, $K \in \msys{b}$ and $L \in \ro{Ext}(\fr{E},K)$. Then the Fesenko--Neukirch morphism
\[
\Upsilon_{L|K}: \gal(L|K)^{\tn{ab}} \lra \ro{K}_n^{\ro{MP}}(K)/\ro{N}_{L|K} \ro{K}_n^{\ro{MP}}(L)
\]
mapping $\sigma \modd \comm{a}(\gal(L|K))$ to $\ro{N}_{\Sigma|K}( \pi_\Sigma )$ is a canonical isomorphism, where $\Sigma$ is the fixed field of the restriction $\widetilde{\sigma}|_{L^{\tn{pur}}}$ of a Frobenius lift $\widetilde{\sigma} \in \gal(K)$ of $\sigma$ to the maximal purely unramified extension $L^{\tn{pur}}$ and $\pi_\Sigma \in \Sigma$ is an element with $\widetilde{v}(\pi_\Sigma) = f_K^{(\delta)}$.
The assignment $L \mapsto \ro{N}_{L|K} L^\times \subs K^\times$ yields an injective lattice morphism from the lattice of abelian extensions of $K$ to the lattice of subgroups of $\ro{K}_n^{\ro{MP}}(K)$. Moreover, the isomorphisms $\Upsilon_{L|K}$ are compatible with the conjugation, inclusion and norm maps. \name{Fesenko} presents in \cite{Fes92_Multidimensional-local_0} also an existence theorem, that is, a description of the image of the map $L \mapsto \ro{N}_{L|K} L^\times \subs K^\times$. We note that this theory is for $n=1$ precisely the local class field theory (in positive characteristic) discussed in \ref{sect:local_cft}.
\end{para}

\begin{para}
\name{Fesenko} mentions in \cite{Fes92_Multidimensional-local_0} that this approach does not work in characteristic 0 because the Galois descent for unramified extensions does not hold.
\end{para}

\newpage
\appendix
\section{Topological groups} \label{chap:top_groups}

In this chapter basic facts about topological groups are collected. In \ref{sect:filters} filters are discussed which are used in the theory of profinite groups. In \ref{sect:basic_facts_about_top_grp} basic facts about general topological groups and in \ref{sec:comp_groups} basic facts about compact groups are collected. In \ref{sec:transversals} some facts about transversals, as the definition of the map $\kappa_T$ which was used in the definition of the transfer map, are collected. Finally, in \ref{sec:top_ab_grps} and \ref{sec:top_commutator} some properties of topological abelian groups as the maximal separated abelian quotient are discussed.

\subsection{Filters} \label{sect:filters}

\begin{defn}
A \word{filter} on a partially ordered set $(A,\leq)$ is a subset $\fr{F} \subs A$ satisfying the following conditions:
\begin{enumerate}[label=(\roman*),noitemsep,nolistsep]
\item \label{item:general_filter_attract} If $y \in \fr{F}$ and $a \in A$ such that $y \leq a$, then $a \in \fr{F}$. 
\item \label{item:general_filter_minimum} If $y,z \in \fr{F}$, then there exists $x \in \fr{F}$ with $x \leq y$ and $x \leq z$. 
\item \label{item:general_filter_empty} $\fr{F} \neq A$.
\end{enumerate}
\end{defn}

\begin{defn}
If $X$ is a set, then a filter on $(\fr{P}(X),\subs)$ is simply called a \word{filter} on $X$.
\end{defn}

\begin{prop} \label{para:set_filter}
If $X$ is a set, then a subset $\fr{F} \subs \fr{P}(X)$ is a filter on $X$ if and only if it satisfies the following properties:
\begin{enumerate}[label=(\roman*),noitemsep,nolistsep]
\item \label{item:set_filter_generate_itself} If $Y \in \fr{F}$ and $Z \in \fr{P}(X)$ such that $Y \subs Z$, then already $Z \in \fr{F}$.
\item \label{item:set_filter_int_closed} If $Y_1,Y_2 \in \fr{F}$, then $Y_1 \cap Y_2 \in \fr{F}$.
\item \label{item:set_filter_empty} $\fr{F} \neq \emptyset$ and $\emptyset \notin \fr{F}$.
\end{enumerate}
\end{prop}

\begin{proof}
This is easy to verify. 
\end{proof}

\begin{para}
It follows immediately from the definition that $X \in \fr{F}$ for any filter $\fr{F}$ on $X$. In particular, there is no filter on the empty set.
\end{para}

\begin{prop} \label{para:filter_generating_system} 
Let $X$ be a set and let $\ca{S} \subs \fr{P}(X)$. Let 
\[
\langle \ca{S} \rangle = \lbrace X \rbrace \cup \lbrace Y \mid Y \in \fr{P}(X) \tn{ and there exists }  S \in \widetilde{\ca{S}} \tn{ such that } S \subs Y \rbrace,
\]
where
\[
\widetilde{\ca{S}} = \lbrace \bigcap_{i \in I} S_i \mid I \tn{ is a finite non-empty set and }  S_i \in \ca{S} \rbrace.
\]

Then $\langle \ca{S} \rangle$ is a filter on $X$ if and only if $X \neq \emptyset$ and $S_1 \cap S_2 \neq \emptyset$ for all $S_1,S_2 \in \ca{S}$. In this case, $\langle \ca{S} \rangle$ is the smallest filter on $X$ containing $\ca{S}$. It is called the filter \words{generated}{filter!generated by} by $\ca{S}$.
\end{prop}

\begin{proof}
It is obvious that the condition is necessary as there is no filter on the empty set and $\emptyset$ cannot be an element of a filter. So, assume that the condition holds. We have to verify that $\langle \ca{S} \rangle$ is a filter on $X$ which contains $\ca{S}$. Note that $\ca{S} \subs \widetilde{\ca{S}} \subs \langle \ca{S} \rangle$, so it remains to check that $\langle \ca{S} \rangle$ is a filter. By definition, $X \in \langle \ca{S} \rangle$ and therefore $\langle \ca{S} \rangle \neq \emptyset$. Moreover, the condition implies that $\emptyset \notin \langle \ca{S} \rangle$. Hence, $\langle \ca{S} \rangle$ satisfies \ref{para:set_filter}\ref{item:set_filter_empty}. Let $Y_1,Y_2 \in \langle \ca{S} \rangle$. Then there exist $S_1,S_2 \in \widetilde{\ca{S}}$ such that $S_i \subs Y_i$. By definition of $\widetilde{\ca{S}}$, we have $S_1 \cap S_2 \in \widetilde{\ca{S}}$ and as $S_1 \cap S_2 \subs Y_1 \cap Y_2$, it follows that $Y_1 \cap Y_2 \in \langle \ca{S} \rangle$. Hence, \ref{para:set_filter}\ref{item:set_filter_int_closed} is satisfied. Condition \ref{para:set_filter}\ref{item:set_filter_generate_itself} is obviously satisfied, so $\langle \ca{S} \rangle$ is a filter on $X$ which contains $\ca{S}$.

Finally, let $\fr{F}$ be a filter on $X$ which contains $\ca{S}$. Then it follows from \ref{para:set_filter}\ref{item:set_filter_int_closed} that $\widetilde{\ca{S}} \subs \fr{F}$ and then it follows from \ref{para:set_filter}\ref{item:set_filter_generate_itself} that $\langle \ca{S} \rangle \subs \fr{F}$. Hence, $\langle \ca{S} \rangle$ is the smallest filter on $X$ containing $\ca{S}$.
\end{proof}

\begin{defn}
A subset $\ca{S}$ of a filter $\fr{F}$ on a set $X$ is called a \words{generating system}{filter!generating system} of $\fr{F}$ if $\langle \ca{S} \rangle = \fr{F}$.
\end{defn}

\begin{para}
If $\fr{F}$ is a filter on a set $X$, then $\langle \fr{F} \rangle = \fr{F}$, so any filter has a generating system.
\end{para}

\begin{defn}
A partially ordered set $(A,\leq)$ is called \words{directed}{partially ordered set!directed} if every two elements of $A$ have an upper bound.
\end{defn}

\begin{defn}
Let $(A,\leq)$ be a partially ordered set. A set $B$ is called \words{cofinal}{preordered set!cofinal} in $(A,\leq)$ if $B \subs A$ and if for each $a \in A$ there exists $b \in B$ such that $a \leq b$.
\end{defn}

\begin{prop}
If $X$ is a set and $\ca{B} \subs \fr{P}(X)$, then $(\ca{B},\sups)$ is directed if and only if for all $B_1,B_2 \in \ca{B}$ there exists $B \in \ca{B}$ such that $B_1 \cap B_2 \sups B$. Moreover, $\ca{B}'$ is cofinal in $(\ca{B},\sups)$ if and only if $\ca{B}' \subs \ca{B}$ and for each $B \in \ca{B}$ there exists $B' \in \ca{B}'$ such that $B \sups B'$.
\end{prop}

\begin{proof}
This is obvious.
\end{proof}

\begin{prop} \label{prop:filter_basis}
Let $X$ be a set and let $\ca{B} \subs \fr{P}(X)$. The following are equivalent:
\begin{enumerate}[label=(\roman*),noitemsep,nolistsep]
\item \label{item:filter_basis_equ_1} $\ca{B} \neq \emptyset$, $\emptyset \notin \ca{B}$ and $(\ca{B},\sups)$ is directed.
\item \label{item:filter_basis_equ_2} $\lbrace Y \mid Y \in \fr{P}(X) \wedge (\exists B \in \ca{B})(B \subs Y) \rbrace$ is a filter on $X$.
\end{enumerate}
If $\ca{B}$ satisfies the above properties, then 
\[
\langle \ca{B} \rangle = \lbrace Y \mid Y \in \fr{P}(X) \wedge (\exists B \in \ca{B})(B \subs Y) \rbrace
\]
and $\ca{B}$ is called a \word{filter basis} on $X$. If $\fr{F}$ is a filter on $X$ and $\ca{B}$ is a filter basis on $X$ such that $\fr{F} = \langle \ca{B} \rangle$, then we call $\ca{B}$ a filter basis of $\fr{F}$.
\end{prop}

\begin{proof} 
Both assertions are easy to verify.
\end{proof}

\begin{prop} \label{prop:filter_basis_of_filter}
Let $X$ be a set. The following holds:
\begin{compactenum}[(i)]
\item Any filter $\fr{F}$ on $X$ is also a filter basis on $X$ and $\langle \fr{F} \rangle = \fr{F}$.
\item Let $\ca{B},\ca{B}'$ be two filter bases on $X$. Then $\langle \ca{B} \rangle \subs \langle \ca{B}' \rangle$ if and only if for each $B \in \ca{B}$ there exists $B' \in \ca{B}'$ such that $B' \subs B$.
\item If $\fr{F}$ is a filter on $X$, then a subset $\ca{B} \subs \fr{P}(X)$ is a filter basis of $\fr{F}$ if and only if $\ca{B}$ is cofinal in $(\fr{F},\sups)$.
\end{compactenum}
\end{prop}

\begin{proof} 
All assertions are straightforward.
\end{proof}

\begin{para} \wordsym{$\lbrack \ca{B} \rbrack$}
Let $X$ be a set. Two filter bases $\ca{B},\ca{B}'$ on a set $X$ are called \words{equivalent}{filter basis!equivalent} if $\langle \ca{B} \rangle = \langle \ca{B}' \rangle$. This defines an equivalence relation on the set of all filter bases on $X$ and we denote the equivalence class of a filter basis $\ca{B}$ by $\lbrack \ca{B} \rbrack$. The assignment $\lbrack \ca{B} \rbrack \mapsto \langle \ca{B} \rangle$ defines a bijection between the set of equivalence classes of filter bases on $X$ and the set of filters on $X$.
\end{para}

\begin{defn}
Let $X$ be a set. A \word{neighborhood system} on $X$ is a family $\lbrace \fr{V}(x) \mid x \in X \rbrace$ of filters on $X$ satisfying the following properties:
\begin{enumerate}[label=(\roman*),noitemsep,nolistsep]
\item If $x \in X$ and $U \in \fr{V}(x)$, then $x \in U$.
\item If $x \in X$ and $U \in \fr{V}(x)$, then there exists $V \in \fr{V}(x)$ such that $U \in \fr{V}(y)$ for each $y \in V$.
\end{enumerate}
\end{defn}

\begin{prop} \label{para:topology_induced_by_ns} 
Let $X$ be a set. Then there exists a canonical bijection between the set of topologies on $X$ and the set of neighborhood systems on $X$. Moreover, if $\fr{V} = \lbrace \fr{V}(x) \mid x \in X \rbrace$ is a neighborhood system on $X$ and if $\tau$ is the corresponding topology on $X$, then the following holds:
\begin{compactenum}[(i)]
\item For each $x \in X$ the filter $\fr{V}(x)$ is the neighborhood filter of the $\tau$-neighborhoods of $x$.
\item A subset $A \subs X$ is $\tau$-open if and only if $A \in \fr{V}(x)$ for all $x \in A$.
\end{compactenum}
\end{prop}

\begin{proof}
Confer \cite[chapitre I, \S1.2]{Bou71_Topologie-Generale_0}.
\end{proof}

\begin{defn} \wordsym{$\fr{V}(x)$}
Let $X$ be a topological space. The neighborhood system $\fr{V}$ on $X$ corresponding to the topology on $X$ is simply called the \word{neighborhood system} of $X$ and for $x \in X$ the set $\fr{V}(x)$ is called the \word{neighborhood filter} of $x$. A filter basis of $\fr{V}(x)$ is called a \word{neighborhood basis} of $x$.  
\end{defn}

\begin{prop}
Let $f:X \ra Y$ be a map between topological spaces. Then $f$ is continuous in $x \in X$ if and only if there exists a neighborhood basis $\ca{B}(x)$ of $x$ and a neighborhood basis $\ca{B}(f(x))$ of $f(x)$ such that for each $V \in \ca{B}(f(x))$ there exists $U \in \ca{B}(x)$ such that $f(U) \subs V$.
\end{prop}

\begin{proof}
This is obvious.
\end{proof}

\begin{prop} \wordsym{$\ro{lim} \ \ca{B}$}
Let $X$ be a topological space and let $\ca{B}$ be a filter basis on $X$. The following are equivalent for a point $x \in X$:
\begin{enumerate}[label=(\roman*),noitemsep,nolistsep]
\item $\fr{V}(x) \subs \langle \ca{B} \rangle$.
\item For every neighborhood basis $\ca{B}(x)$ of $x$ and every $U \in \ca{B}(x)$ there exists $B \in \ca{B}$ such that $B \subs U$.
\item There exists a neighborhood basis $\ca{B}(x)$ of $x$ such that for each $U \in \ca{B}(x)$ there exists $B \in \ca{B}$ such that $B \subs U$.
\end{enumerate}
If these conditions are satisfied, then $x \in X$ is called a \words{limit point}{filter basis!limit point} of $\ca{B}$. The set of all limit points of $\ca{B}$ is denoted by $\ro{lim} \ \ca{B}$. If $\ro{lim} \ \ca{B} \neq \emptyset$, then  $\ca{B}$ is called \words{convergent}{filter basis!convergent}.
\end{prop}

\begin{proof}
This follows immediately from \ref{prop:filter_basis_of_filter}.
\end{proof}

\begin{prop} \wordsym{$\ro{acc} \ \ca{B}$}
Let $\ca{B}$ be a filter basis on a topological space $X$. The following are equivalent for a point $x \in X$:
\begin{enumerate}[label=(\roman*),noitemsep,nolistsep]
\item \label{item:acc_points_1} $x \in \ro{cl}(B)$ for each $B \in \ca{B}$.
\item \label{item:acc_points_2} If $\ca{B}(x)$ is a neighborhood basis of $x$, then $B \cap U \neq \emptyset$ for all $B \in \ca{B}$ and $U \in \ca{B}(x)$. 
\item \label{item:acc_points_3} There exists a neighborhood basis $\ca{B}(x)$ of $x$ such that $B \cap U \neq \emptyset$ for all $B \in \ca{B}$ and $U \in \ca{B}(x)$.
\item \label{item:acc_points_4} There exists a filter basis $\ca{B}'$ on $X$ such that $\ca{B}' \sups \ca{B}$ and $x \in \ro{lim} \ \ca{B}'$.
\end{enumerate}
If these conditions are satisfied, then $x$ is called an \words{accumulation point}{filter basis!accumulation point} (or \words{cluster point}{filter basis!cluster point}) of $\ca{B}$. The set of all accumulation points of $\ca{B}$ is denoted by $\ro{acc} \ \ca{B}$.
\end{prop}

\begin{proof} 
All assertions are straightforward.
\end{proof}

\begin{prop} \label{para:lim_acc_props}
Let $X$ be a topological space and let $\ca{B}, \ca{B}'$ be two filter bases on $X$ with $\ca{B} \subs \ca{B}'$. The following holds:
\begin{compactenum}[(i)]
\item $\ro{lim} \ \ca{B} \subs \ro{acc} \ \ca{B}$.
\item $\ro{acc} \ \ca{B}' \subs \ro{acc} \ \ca{B}$ and $\ro{lim} \ \ca{B} \subs \ro{lim} \ \ca{B}'$.
\item $\ro{acc} \ \ca{B} = \ro{acc} \ \langle \ca{B} \rangle$ and $\ro{lim} \ \ca{B} = \ro{lim} \ \langle \ca{B} \rangle$.
\end{compactenum}
\end{prop}

\begin{proof} 
All relations are easy to verify.
\end{proof}

\begin{prop} \label{prop:limit_compact_member_intersect}
Let $\ca{B}$ be a filter basis on a separated topological space $X$ such that $\ca{B}$ contains a compact subset of $X$. Then the following are equivalent for $x \in X$:
\begin{compactenum}[(i)]
\item $\ro{lim} \ \ca{B} = x$.
\item $\bigcap_{B \in \ca{B}} B = \lbrace x \rbrace$.
\end{compactenum}
\end{prop}

\begin{proof}
This is  \cite[lemma A on page 101]{HofMor09_Contributions-to-the-structure_0}.
\end{proof}

\subsection{Topological groups} \label{sect:basic_facts_about_top_grp}

\begin{prop} \label{thm:group_topology_filter_bases}
Let $G$ be a topological group and let $\fr{V}$ be the neighborhood filter of $1 \in G$. Then the neighborhood filter of $g \in G$ is equal to both $g \fr{V} \dopgleich \lbrace gU \mid U \in \fr{V} \rbrace$ and $\fr{V} g \dopgleich \lbrace Ug \mid U \in \fr{V} \rbrace$. Moreover, if $\ca{B}$ is a filter basis of $\fr{V}$, then both $g\ca{B}$ and $\ca{B}g$ are filter bases of the neighborhood filter of $g \in G$. In particular, the topology on $G$ is already uniquely determined by $\ca{B}$.
\end{prop}

\begin{proof}
The first assertion is obvious as multiplication by $g$ on $G$ is a homeomorphism. For the second assertion first note that $g \ca{B}$ is a filter basis on $G$ as multiplication by $G$ is an isomorphism. Let $gU \in g \fr{V}$. Since $\fr{V} = \langle \ca{B} \rangle$, there exists $B \in \ca{B}$ such that $B \subs U$. Hence, $gB \subs gU$ and therefore $gU \in \langle g \ca{B} \rangle$. Conversely, if $Y \in \langle g \ca{B} \rangle$, then there exists $B \in \ca{B}$ such that $gB \subs Y$. In particular, $B \subs g^{-1}Y$. Thus, $g^{-1}Y \in \langle \ca{B} \rangle = \fr{V}$ and therefore $Y \in g\fr{V}$.
\end{proof}

\begin{defn}
Let $G$ be a group. A \words{compatible topology}{compatible!topology} on $G$ is a topology on $G$ such that $G$ is a topological group with respect to this topology.
\end{defn}

\begin{defn} \label{para:compat_filter_basis}
A \word{compatible filter basis} on a group $G$ is a filter basis $\ca{B}$ on $G$ satisfying the following properties:
\begin{enumerate}[label=(\roman*),noitemsep,nolistsep]
\item \label{item:compat_filter_basis_mult_cont} For each $U \in \ca{B}$ there exists $V \in \ca{B}$ such that $V^2 \dopgleich \lbrace g g' \mid g,g' \in V \rbrace \subs U$.
\item \label{item:compat_filter_basis_inv_cont} For each $U \in \ca{B}$ there exists $V \in \ca{B}$ such that $V^{-1} \dopgleich \lbrace g^{-1} \mid g \in V \rbrace \subs U$.
\item \label{item:compat_filter_basis_conj_cont} For each $g \in G$ and each $U \in \ca{B}$ there exists $V \in \ca{B}$ such that $V \subs gUg^{-1}$.
\end{enumerate}
\end{defn}

\begin{prop} \label{prop:group_filter_topology}
Let $G$ be a group. Then there exists a bijection between the set of equivalence classes of compatible filter bases on $G$ and the set of compatible topologies on $G$. Moreover, if $\ca{B}$ is a compatible filter basis on $G$ and if $\tau$ is the corresponding topology on $G$, then the following holds:
\begin{enumerate}[label=(\roman*),noitemsep,nolistsep]
\item \label{item:group_filter_top_nf} For each $g \in G$ the set $g \langle \ca{B} \rangle$ is the filter of the $\tau$-neighborhoods of $g$.
\item \label{item:compat_filter_basis_open} A subset $A \subs G$ is $\tau$-open if and only if for each $g \in A$ there exists $V \in \ca{B}$ such that $gV \subs A$.
\item If $V \in \ca{B}$ is a subgroup of $G$, then $V$ is $\tau$-open.
\end{enumerate}

\end{prop}

\begin{proof}
Let $\ca{B}$ be a compatible filter basis on $G$. It is easy to verify that $\langle \ca{B} \rangle$ is also a compatible filter basis. For $g \in G$ let $\fr{V}(g) = g \langle \ca{B} \rangle$ and let $\fr{V} = \lbrace \fr{V}(g) \mid g \in G \rbrace$. It follows from the proof of \cite[chapitre III, $\S 1.2$, proposition 1]{Bou71_Topologie-Generale_0} that $\fr{V}$ is a neighborhood system on $G$ and that the corresponding topology on $G$ is compatible. This defines a map from the set of compatible filter bases on $G$ to the set of compatible topologies on $G$ and it is obvious that this map induces a map from the set of equivalence classes of compatible filter bases on $G$ to the set of compatible topologies on $G$. We prove that this induced map is a bijection.

Let $\ca{B}_1$, $\ca{B}_2$ be two compatible filter bases on $G$ and let $\tau_1$, $\tau_2$ be the corresponding compatible topologies on $G$. If $\tau_1 = \tau_2$, then in particular the corresponding neighborhood filters $\fr{V}_1(1)$ and $\fr{V}_2(1)$ of $1 \in G$ coincide. As $\fr{V}_i(1) = \langle \ca{B}_i \rangle$ by \ref{para:topology_induced_by_ns}, we have $\langle \ca{B}_1 \rangle = \langle \ca{B}_2 \rangle$ and therefore both filter bases are equivalent. This shows that the assignment is injective.

Let $\tau$ be a compatible topology on $G$. Then it is easy to see that the neighborhood filter $\fr{V}(1)$ of $1 \in G$ is a compatible filter basis on $G$. Let $\tau'$ be the compatible topology defined by $\fr{V}(1)$. By \ref{para:topology_induced_by_ns} the neighborhood filter of $1 \in G$ with respect to this topology is equal to $\fr{V}(1)$ and so it follows from \ref{thm:group_topology_filter_bases} that $\tau = \tau'$. Hence, the assignment is surjective.

It remains to prove the second part of the proposition. Statement \ref{item:group_filter_top_nf} follows directly from \ref{para:topology_induced_by_ns}. To prove \ref{item:compat_filter_basis_open}, let $A \subs X$. It follows from \ref{para:topology_induced_by_ns} that $A$ is open if and only if $A \in \fr{V}(g) = g \langle \ca{B} \rangle$ for all $g \in A$ and this is easily seen to be equivalent to the statement. The last statement is now obvious.
\end{proof}

\begin{prop} \label{para:power_map_cont}
Let $G$ be a topological group. Then for each $n \in \ZZ$ the $n$-th power map $\mu_n:G \ra G$, $g \mapsto g^n$ is continuous.
\end{prop}

\begin{proof} 
This is easy to see.
\end{proof}

\begin{prop} \label{prop:prop_of_top_grps}
Let $G$ be a topological group. The following holds:
\begin{enumerate}[label=(\roman*),noitemsep,nolistsep]
\item \label{item:closed_is_open} Every open subgroup of $G$ is closed.
\item \label{item:closed_finite_open} Every closed subgroup of finite index of $G$ is open.
\item $G$ is separated if and only if $\lbrace 1 \rbrace$ is closed in $G$.
\item If $G$ is quasi-compact, then a subgroup $H$ of $G$ is open if and only if $H$ is closed and $\lbrack G : H \rbrack < \infty$.
\end{enumerate}
\end{prop}

\begin{proof} \hfill

\begin{asparaenum}[(i)]
\item This is \cite[chapitre III, \S2.1, corollaire de proposition 4]{Bou71_Topologie-Generale_0}.
\item Let $H$ be a closed subgroup of finite index of $G$. Let $T$ be a right transversal of $H$ in $G$ such that $1 \in T$. Since $G = \coprod_{t \in T} Ht$, we have $G \setminus \bigcup_{t \in T, t \neq 1} Ht = H$ and since $\bigcup_{t \in T, t \neq 1} Ht$ is a finite union of closed subsets of $G$, it is closed in $G$ and so its complement $H$ is open.
\item This is \cite[chapitre III, \S1.3, proposition 2]{Bou71_Topologie-Generale_0}.
\item A closed subgroup of finite index of $G$ is open by \ref{item:closed_finite_open}. Let $H$ be an open subgroup of $G$. Then $H$ is closed by \ref{item:closed_is_open}. Let $T$ be a right transversal of $H$ in $G$. Then $G = \coprod_{t \in T} Ht$ is an open cover of $G$ and as $G$ is quasi-compact, there exists a finite subset $T' \subs T$ such that $G = \coprod_{t \in T'} Ht$. Hence, $T'$ is also a right transversal of $H$ in $G$ and therefore $\lbrack G:H \rbrack < \infty$.
\end{asparaenum} \vspace{-\baselineskip}
\end{proof}

\begin{prop} \label{prop:prop_of_quots}
Let $G$ be a topological group and let $H$ be a subgroup of $G$. The following holds:
\begin{enumerate}[label=(\roman*),noitemsep,nolistsep]
\item The quotient map $q:G \ra G/H$ is open.
\item The quotient space $G/H$ is separated if and only if $H$ is closed.
\item The quotient space $G/H$ is discrete if and only if $H$ is open.
\item If $H$ is a normal subgroup, then $G/H$ is a topological group with respect to the quotient topology.
\item The closure $\ro{cl}(H)$ of $H$ is also a subgroup of $G$. If $H$ is a normal subgroup, then $\ro{cl}(H)$ is also a normal subgroup of $G$.
\item Let $N$ be a normal subgroup of $G$ and let $q:G \ra G/N$ be the quotient morphism. Then the canonical isomorphism of abstract groups $HN/N \cong q(H)$ is an isomorphism of topological groups.
\item \label{item:top_grp_2_iso_theorem} Let $H$ and $K$ be normal subgroups of $G$ such that $H \leq K$. Then the canonical isomorphism of abstract groups $G/K \cong (G/H)/(K/H)$ is an isomorphism of topological groups.

\item \label{item:top_grp_quot_image_closed_open} Let $H$ be normal in $G$ and let $q:G \ra G/H$ be the quotient morphism. If $N$ is a closed (open) normal subgroup of $G$ such that $N \geq H$, then $q(N)$ is a closed (open) normal subgroup of $G/H$.
\end{enumerate}
\end{prop}

\begin{proof} \hfill

\begin{asparaenum}[(i)]
\item Let $U \subs G$ be an open subset. Then
\[
q^{-1}( q(U)) = \bigcup_{h \in H} Uh.
\]
As right multiplication on $G$ is a homeomorphism, each $Uh$ is open in $G$ and therefore $q^{-1}(q(U))$ is open. This implies that $q(U)$ is open because $q$ is a quotient map.
\item This is \cite[chapitre III, \S2.5, proposition 12]{Bou71_Topologie-Generale_0}.
\item This is \cite[chapitre III, \S2.5, proposition 14]{Bou71_Topologie-Generale_0}.
\item This is \cite[chapitre III, \S2.6, proposition 16]{Bou71_Topologie-Generale_0}.
\item This is \cite[chapitre III, \S2.1, proposition 1]{Bou71_Topologie-Generale_0}.
\item This is \cite[chapitre III, \S2.7, proposition 20]{Bou71_Topologie-Generale_0}.
\item This is \cite[chapitre III, \S2.7, corollaire de proposition 22]{Bou71_Topologie-Generale_0}.
\item By \ref{item:top_grp_2_iso_theorem} we have a canonical isomorphism
\[
(G/H)/q(N) = (G/H)/(N/H) \cong G/N.
\]
Hence, if $N$ is closed (open) in $G$, then $G/N$ is separated (discrete) and therefore $(G/H)/q(N)$ is also separated (discrete) which implies that $q(N)$ is closed (open) in $G/H$.
\end{asparaenum} \vspace{-\baselineskip}
\end{proof}

\begin{prop} \label{para:morphism_quotient_induced}
Let $\varphi:G \ra G'$ be a morphism of topological groups and let $N$ be a normal subgroup of $G$ such that $N \subs \ker(\varphi)$. Then the induced morphism $\varphi':G/N \ra G'$ of abstract groups is continuous. 
\end{prop}

\begin{proof}
Let $q:G \ra G/N$ be the quotient morphism. The following is a commutative diagram of abstract groups
\[
\xymatrix{
G \ar[r]^\varphi \ar[d]_q & G' \\
G/N \ar[ur]_{\varphi'}
}
\]
Since $q$ is a (topological) quotient map and $\varphi$ is continuous, $\varphi'$ is also continuous. 
\end{proof}

\begin{defn}
A morphism $\varphi:G \ra G'$ of topological groups is called \words{strict}{morphism!strict} if the induced isomorphism $\varphi':G/\ker(\varphi) \ra \im(\varphi)$ of abstract groups is a homeomorphism.
\end{defn}

\begin{prop} The following are equivalent for a morphism $\varphi:G \ra G'$ of topological groups:
\begin{compactenum}[(i)]
\item $\varphi$ is strict.
\item The corestriction $\varphi:G \ra \im(\varphi)$ is an open map, that is for any open subset $U$ of $G$ the set $\varphi(U)$ is open in $\im(\varphi)$. 
\item For any neighborhood $\Omega$ of $1$ in $G$ the set $\varphi(\Omega)$ is a neighborhood of $1$ in $\im(\varphi)$.
\end{compactenum}
\end{prop}

\begin{proof}
This is \cite[chapitre III, \S2.8, proposition 24]{Bou71_Topologie-Generale_0}.
\end{proof}

\subsection{Quasi-compact and compact groups} \label{sec:comp_groups}

\begin{prop}[Closed map lemma] \invword{closed map lemma} \label{prop:closed_map_lemma}
A continuous map $f:X \ra Y$ from a quasi-compact space $X$ into a separated space $Y$ is closed. If $f$ is moreover bijective, then $f$ is already a homeomorphism.
\end{prop}

\begin{proof}
This is \cite[chapitre I, \S 9.4, corollaire 2]{Bou71_Topologie-Generale_0}.
\end{proof}

\begin{prop} \label{prop:mor_from_qc_in_sep_strict}
Any morphism from a quasi-compact group into a separated group is strict and has compact image.
\end{prop}

\begin{proof}
Let $\varphi:G \ra G'$ be a morphism, where $G$ is a quasi-compact group and $G'$ is a separated group. Since $G'$ is separated, $\lbrace 1 \rbrace$ is closed in $G'$ and as $\varphi$ is continuous, $\ker(\varphi) = \varphi^{-1}(\lbrace 1 \rbrace)$ is a closed subgroup of $G$. Hence, the induced isomorphism of abstract groups $\varphi':G/\ker(\varphi) \ra \im(\varphi)$ is a continuous map from a compact space into a separated space and therefore already a homeomorphism by the closed map lemma. Hence, $\varphi'$ is strict and $\im(\varphi)$ is compact.
\end{proof}

\begin{prop} \label{prop:qc_intersection_props}
Let $G$ be a quasi-compact group. The following holds:
\begin{compactenum}[(i)]
\item Let $(\ca{X},\sups)$ be a directed set of closed subsets of $G$ and let $Y$ be a closed subset of $G$. Then
\[
\bigcap_{X \in \ca{X}} XY = (\bigcap_{X \in \ca{X}} X) \cdot Y.
\]

\item Let $(\ca{H},\sups)$ be a directed set of closed subgroups of $G$, let $G'$ be a separated group and let $\varphi:G \ra G'$ be a surjective morphism of topological groups. Then
\[
\varphi( \bigcap_{H \in \ca{H}} H) = \bigcap_{H \in \ca{H}} \varphi(H).
\]

\item Let $U$ be an open subset of $G$. If $\lbrace H_i \mid i \in I \rbrace$ is a family of closed subgroups of $G$ such that $\bigcap_{i \in I} H_i \subs U$, then there is a finite subset $J$ of $I$ such that $\bigcap_{j \in J} H_j \subs U$.
\end{compactenum}
\end{prop}

\begin{proof} \hfill

\begin{asparaenum}[(i)]

\item This is \cite[lemma 0.3.1(h)]{Wil97_Profinite-Groups_0}.

\item\hspace{-5pt}\footnote{This is essentially the proof of \cite[proposition 2.1.4(b)]{RibZal00_Profinite-Groups_0}.}  Let $K \dopgleich \ker(\varphi)$ and let $q:G \ra G/K$ be the quotient morphism. We will first prove that
\[
q( \bigcap_{H \in \ca{H}} (H \cdot K) ) = \bigcap_{H \in \ca{H}} q(H \cdot K).
\]
The inclusion $\subs$ is evident. To see the converse inclusion, let $\ol{x} \in \bigcap_{H \in \ca{H}} q(H \cdot K)$ and let $x \in G$ be a representative. Then for each $H \in \ca{H}$ we have $\ol{x} \in q(H \cdot K) = (H \cdot K)/K$ and so there exists $y_H \in H \cdot K$ such that $x \equiv y_H \ \ro{mod} \ K$ and consequently 
\[
 x \in y_H \cdot K \subs (H \cdot K) \cdot K = H \cdot K.
\]
This shows that $x \in \bigcap_{H \in \ca{H}} (H \cdot K)$ and therefore $\ol{x} = q(x) \in q(\bigcap_{H \in \ca{H}} (H \cdot K))$. 

Since $G'$ is separated, $K$ is a closed subset of $G$ and so we can use the first statement of the proposition to get
\[
\bigcap_{H \in \ca{H}} (H \cdot K) = (\bigcap_{H \in \ca{H}} H) \cdot K.
\]

Let $\varphi':G/K \ra G'$ be the isomorphism of abstract groups induced by $\varphi$. Then using the fact that $\varphi'$ is bijective, we get
\[
\bigcap_{H \in \ca{H}} \varphi(H) = \bigcap_{H \in \ca{H}} \varphi' \circ q(H) = \varphi' ( \bigcap_{H \in \ca{H}} q(H) ) = \varphi'( \bigcap_{H \in \ca{H}} q(H \cdot K) ) = \varphi'( q (\bigcap_{H \in \ca{H}} H \cdot K) ) 
\]
\[
= \varphi' q(\bigcap_{H \in \ca{H}} H) \cdot K ) ) = \varphi' ( q (\bigcap_{H \in \ca{H}} H ) ) = \varphi( \bigcap_{H \in \ca{H}} H ).
\]

\item Since $\bigcap_{i \in I} H_i \subs U$, we get 
\[
G \setminus U \subs G \setminus \bigcap_{i \in I} H_i = \bigcup_{i \in I} G \setminus H_i.
\]

Hence, $\lbrace G \setminus H_i \mid i \in I \rbrace$ is an open cover of $G \setminus U$. Because $G \setminus U$ is a closed subset of the quasi-compact space $G$, it is also quasi-compact and so there exists a finite subset $J$ of $I$ such that $G \setminus U \subs \bigcup_{j \in J} G \setminus H_j$. This implies $\bigcap_{j \in J} H_j \subs U$. \vspace{-\baselineskip}.

\end{asparaenum}

\end{proof}

\begin{prop} \label{prop:compact_product_closed}
Let $G$ be a compact group and let $X,Y$ be closed subsets of $G$. Then the set $X \cdot Y$ is closed in $G$.
\end{prop}

\begin{proof}
This is \cite[lemma 0.3.1(g)]{Wil97_Profinite-Groups_0}.
\end{proof}

\subsection{Transversals} \label{sec:transversals}

\begin{defn}
Let $G$ be a group and let $H$ be a subgroup of $G$. A complete set of representatives of the right cosets of $H$ in $G$ is also called a \words{right transversal}{transversal} of $H$ in $G$. A right transversal $T$ is called \words{unitary}{transversal!unitary} if $1 \in T$. 
If $G$ is a topological group and $H$ is a closed subgroup of $G$, then a \words{closed right transversal}{transversal!closed} of $H$ in $G$ is a right transversal $T$ of $H$ in $G$ such that $T$ is a closed subset of $G$. A (unitary, closed) left transversal is defined similarly.
\end{defn}

\begin{para} \label{para:transversals_left_right}
All statements in this section about right transversals hold analogously also for left transversals. 
\end{para}

\begin{prop}  \label{para:t_remover} \wordsym{$\kappa_T$} \wordsym{$\sigma_{T,g}$}
Let $G$ be a group, let $H$ be a subgroup of $G$ and let $T$ be a right transversal of $H$ in $G$. The following holds:
\begin{compactenum}[(i)]
\item For each $g \in G$ there exist unique elements $\kappa_T(g) \in H$ and $t \in T$ such that $g = \kappa_T(g)t$. The map $\kappa_T:G \ra H, g \mapsto \kappa_T(g)$, is called the \word{$T$-remover}. 
\item For each $g \in G$ there exists a unique permutation $\sigma_{T,g}$ on $T$ such that $tg = \kappa_T(tg) \sigma_{T,g}(t)$ for all $t \in T$. The permutation $\sigma_{T,g}$ is called the \word{$T$-permutation} for $g$. If $T = \lbrace t_1,\ldots,t_n \rbrace$ then we usually identify $\sigma_{T,g}$ with a permutation on $\lbrace 1,\ldots,n \rbrace$.

\item If $g,g' \in G$, then $\kappa_T(tgg') = \kappa_T(tg) \kappa_T(\sigma_{T,g}(t)g')$ and $\sigma_{T,gg'}(t) = \sigma_{T,g'}( \sigma_{T,g}(t))$ for all $t \in T$.

\item Let $q:G \ra H \mybackslash G$ be the quotient morphism. Then $q|_T:T \ra H \mybackslash G$ is bijective and 
\[
s \dopgleich \iota \circ (q|_T)^{-1}:H \mybackslash G \ra G
\]
is a set-theoretical section of $q$, where $\iota:H \ra G$ is the inclusion. Moreover, $\kappa_T(g) = g(sq(g))^{-1}$ for all $g \in G$.
\item Let $T$ be a unitary right transversal of $H$ in $G$. Then $\kappa_T(1) = 1$ and $\kappa_T(hg) = h \kappa_T(g)$ for all $h \in H$, $g \in G$.

\end{compactenum}
\end{prop}

\begin{proof} \hfill

\begin{asparaenum}[(i)]
\item Since $T$ is a right transversal, we have $G = \coprod_{t \in T} Ht$ and consequently there exists a unique $t \in T$ such that $g \in Ht$. This shows uniqueness of the element $t$ in the decomposition. Moreover, we can now write $g = \kappa_T(g)t$ for some $\kappa_T(g) \in H$ and if $g = ht$ for another element $h \in H$, then $g = \kappa_T(g)t = ht$ which implies $\kappa_T(g)=h$.

\item If $t \in T$, then by the above there exists a unique element $t_g$ such that $tg = \kappa_T(tg) t_g$. The uniqueness implies that we must have $\sigma_{T,g}(t) = t_g$ and we also conclude that the map $\sigma_{T,g}:T \ra T$, $t \mapsto t_g$, is injective. To see that it is surjective, note that $G = \bigcup_{t \in T} Ht$ and consequently $G = Gg = \bigcup_{t \in T} Htg$. Hence, if $t' \in T$, then there exists $t \in T$ such that $t' \in Htg$, that is, $t' = htg$ for some $h \in H$ and therefore $h^{-1}t' = tg$. This implies that $\kappa_T(tg) = h^{-1}$ and $t' = t_g = \sigma_{T,g}(t)$ proving the surjectivity of $\sigma_{T,g}$.

\item We have $tg = \kappa_T(tg) \sigma_{T,g}(t)$ and $\sigma_{T,g}(t)g' = \kappa_T(\sigma_{T,g}(t)g')\sigma_{T,g'}(\sigma_{T,g}(t))$. Hence,
\[
t(gg') = (tg)g' = (\kappa_T(tg) \sigma_{T,g}(t))g' = \kappa_T(tg)( \sigma_{T,g}(t)g') 
\]
\[
= \kappa_T(tg)(\kappa_T(\sigma_{T,g}(t)g') \sigma_{T,g'}(\sigma_{T,g}(t))) = (\kappa_T(tg)\kappa_T(\sigma_{T,g}(t)g')) \sigma_{T,g}'(\sigma_{T,g}(t))
\]
and consequently
\[
\kappa_T(tgg') = \kappa_T(tg)\kappa_T(\sigma_{T,g}(t)g'), \quad \sigma_{T,gg'}(t) =  \sigma_{T,g'}(\sigma_{T,g}(t)).
\]
\item By definition, $q|_T$ is bijective and it is obvious that $s$ is a set-theoretical section of $q$. Now, let $g \in G$. Then $g$ has a unique decomposition $g = ht$ with $h \in H$ and $t \in T$. Hence, $q(g) = q(ht) = Ht$ and therefore $sq(g) = s(Ht) = t$ by construction of $s$. It follows that $g = ht = h (sq(g))$ and we get $h = g(sq(g))^{-1}$. Consequently, $\kappa_T(g) = h = g(sq(g))^{-1}$.
\item Since $1 \in T$, the equation $1 = 1 \cdot 1$ is a decomposition of the form $1 = ht, h \in H, t \in T$ and therefore $\kappa_T(1) = 1$. If $g \in G$, then $g = \kappa_T(g)t$ for some $t \in T$. Hence, $hg = h\kappa_T(g)t$ and as $h \kappa_T(g) \in H$, it follows that $\kappa_T(hg) = h \kappa_T(g)$. \vspace{-\baselineskip}
\end{asparaenum}
\end{proof}

\begin{para} \label{para:left_transversal_t_kappa_and_permut}
Note that if $T$ is a left transversal of $H$ in $G$, then we have the relations 
\[
\kappa_T(g'gt) = \kappa_T(g'\sigma_{T,g}(t)) \kappa_T(gt) \quad \tn{and} \quad \sigma_{T,g'g}(t) = \sigma_{T,g'}(\sigma_{T,g}(t)).
\]
\end{para}

\begin{prop} \label{para:transversal_reduction}
Let $G$ be a topological group, let $H$ be a subgroup of $G$. Let $K \lhd G$ such that $K \leq H$ and let $q:G \ra G/K$ be the quotient morphism. If $T$ is a right transversal of $H$ in $G$, then $\ol{T} \dopgleich q(T)$ is a right transversal of $H/K$ in $G/K$ and $\kappa_{ \ol{T}} \circ q = q \circ \kappa_T:G \ra H/K$.
\end{prop}
\begin{proof}
Since $G = \bigcup_{t \in T} Ht$, we have
\[
q(G) = q(\bigcup_{t \in T} Ht) = \bigcup_{t \in T} q(H) q(t) = \bigcup_{\ol{t} \in \ol{T}} (H/K) \cdot \ol{t}.
\]
To show that this union is disjoint, let $g \in G$ and suppose that $q(g) = q(h)q(t) = q(h')q(t')$ for some $h,h' \in H$ and $t,t' \in T$. Then $q(g) = q(ht) = q(h't')$ and therefore $ht = kh't'$ for some $k \in K$. Since $K \subs H$, we have $kh' \in H$ and thus it follows from the uniqueness in \ref{para:t_remover} that $h = kh'$ and $t=t'$. Consequently, $q(h) = q(kh') = q(h')$ and $q(t) = q(t')$ and this shows that the decomposition is unique.

If $g \in G$, then $g = \kappa_T(g) t$ for some $t \in T$ and therefore $q(g) = q(\kappa_T(g)) \cdot q(t)$ and since $\ol{T}$ is a right transversal of $H/K$ in $G/K$, the uniqueness in \ref{para:t_remover} implies that $q(\kappa_T(g)) = \kappa_{\ol{T}}(q(g))$. This shows that $q \circ \kappa_T = \kappa_{\ol{T}} \circ q$.
\end{proof}

\begin{prop} \label{para:t_remover_continuous_for_finite}
Let $G$ be a topological group, let $H$ be a closed subgroup of finite index of $G$ and let $T$ be a right transversal of $H$ in $G$. Then $\kappa_T:G \ra H$ is continuous. 
\end{prop}
\begin{proof}
Let $q:G \ra H \mybackslash G$ be the quotient morphism. Since $H$ is closed, it follows from \ref{prop:prop_of_quots} that $H \mybackslash G$ is separated (note that \ref{prop:prop_of_quots} holds similarly for right coset spaces instead of left coset spaces). Hence, the restriction $q|_T:T \ra H \mybackslash G$ is a bijective continuous map from a finite (and hence quasi-compact) space to a separated space and therefore it is a homeomorphism according the closed map lemma. This implies that the set-theoretical section $s \dopgleich \iota \circ (q|_T)^{-1}:H \mybackslash G \ra G$ of $q$, where $\iota:T \ra G$ is the inclusion, is continuous. As multiplication and inversion on $G$ is continuous, it follows that the map $g \mapsto g(sq(g))^{-1} = \kappa_T(g)$ is continuous.
\end{proof}

\subsection{Topological abelian groups} \label{sec:top_ab_grps}

\begin{prop} \label{prop:top_ab_is_top_z_mod}
Let $A$ be a topological abelian group. If $\ZZ$ is equipped with the discrete topology, then the map $\ZZ \times A \ra A$, $(n,a) \mapsto na$, makes $A$ into a topological $\ZZ$-module.
\end{prop}

\begin{proof} 
It is obvious that the given map makes $A$ into an abstract $\ZZ$-module. As $\ZZ$ is equipped with the discrete topology, it thus remains to check that for each $n \in \ZZ$ the group morphism $\mu_n:A \ra A$, $a \mapsto na$ is continuous. But this follows immediately from \ref{para:power_map_cont}.
\end{proof}

\begin{prop} \label{para:comp_ab_power_group}
Let $A$ be a compact abelian group and let $n \in \NN_{>0}$. The following holds:
\begin{enumerate}[label=(\roman*),noitemsep,nolistsep]
\item The map $\mu_n:A \ra A$, $a \mapsto na$ is continuous and $nA \dopgleich \mu_n(A)$ is a closed subgroup of $A$.
\item If $A$ has trivial torsion, then $\mu_n$ induces an isomorphism of compact groups $\mu_n':A \ra nA$. For $a \in nA$ we define $\frac{1}{n} a \dopgleich (\mu_n')^{-1}(a)$.
\end{enumerate}
\end{prop}

\begin{proof}
By \ref{prop:top_ab_is_top_z_mod} the group morphism $\mu_n$ is continuous. Due to the closed map lemma $\mu_n$ is a closed map and therefore $nA = \mu_n(A)$ is a closed subgroup of $A$. Hence, $nA$ is a compact group. If $A$ has trivial $\ZZ$-torsion, then $\mu_n$ is injective and therefore $\mu_n':A \ra nA$, $a \mapsto na$, is an isomorphism of compact groups by the closed map lemma.
\end{proof}

\subsection{Abelianization} \label{sec:top_commutator}

\begin{defn} \label{para:commutator_subgroup}
\wordsym{$\comm{a}(G)$} \wordsym{$\comm{t}(G)$} \wordsym{$\lbrack x,y \rbrack$}
Let $G$ be a topological group. The \word{commutator} of two elements $x,y \in G$ is the element $\lbrack x,y \rbrack \dopgleich xy(yx)^{-1} \in G$. The \words{abstract commutator subgroup}{commutator subgroup!abstract} $\comm{a}(G)$ of $G$ is defined as the subgroup generated by all commutators, that is,
\[
\comm{a}(G) = \langle \lbrack x,y \rbrack \mid x,y \in G \rangle.
\]
The \words{topological commutator subgroup}{commutator subgroup!topological} $\comm{t}(G)$ of $G$ is defined as the subgroup topologically generated by all commutators, that is,
\[
\comm{t}(G) = \ro{cl}(\comm{a}(G)).
\]
\end{defn}

\begin{prop} \label{para:props_of_commuator_subgroup}
\wordsym{$G^{\ro{ab}}$} \wordsym{$\varphi^{\ro{ab}}$}
Let $G$ be a topological group. The following holds:
\begin{enumerate}[label=(\roman*),noitemsep,nolistsep]
\item \label{item:top_commutator_is_normal} $\comm{t}(G)$ is a closed normal subgroup of $G$.
\item $G$ is separated and abelian if and only if $\comm{t}(G) = 1$.
\item $G^{\ro{ab}} \dopgleich G/\comm{t}(G)$ is a separated abelian group.
\item Let $N \lhd_{\ro{c}} G$. Then $G/N$ is a separated abelian group if and only if $N \geq \comm{t}(G)$.
\item \label{item:top_commutator_mor} If $\varphi:G \ra G'$ is a morphism of topological groups, then $\varphi(\comm{t}(G)) \subs \comm{t}(G')$.
\item \label{item:top_commutator_surj_mor} If $\varphi:G \ra G'$ is a closed surjective morphism of topological groups, then $\varphi(\comm{t}(G)) = \comm{t}(G')$.
\item The quotient morphism $q:G \ra G^{\ro{ab}}$ is universal among morphisms from $G$ into separated abelian groups, that is, if $\varphi:G \ra A$ is a morphism into a separated abelian group $A$, then there exists a unique morphism of topological groups $\varphi':G^{\ro{ab}} \ra A$ making the diagram
\[
\xymatrix{
G \ar[r]^\varphi \ar[d]_q & A \\
G^{\ro{ab}} \ar[ur]_{\varphi'}
}
\]
commutative.

\item If $\varphi:G \ra G'$ is a morphism of topological groups, then there exists a unique morphism of topological groups $\varphi^{\ro{ab}}:G^{\ro{ab}} \ra (G')^{\ro{ab}}$ making the diagram
\[
\xymatrix{
G \ar[r]^\varphi \ar[d]_{q_G} & G' \ar[d]^{q_{G'}} \\
G^{\ro{ab}} \ar[r]_{\varphi^{\ro{ab}}} & (G')^{\ro{ab}}
}
\]
commutative, where the vertical morphisms are the quotient morphisms.

\item The maps
\[
\begin{array}{rcl}
(-)^{\ro{ab}}: \se{TGrp} & \lra & \se{TAb}^{\ro{s}} \\
G & \longmapsto & G^{\ro{ab}} \\
\varphi:G \ra G' & \longmapsto & \varphi^{\ro{ab}}:G^{\ro{ab}} \ra (G')^{\ro{ab}} 
\end{array}
\]
define a functor from the category of topological groups to the category of separated abelian groups.

\end{enumerate}
\end{prop}

\begin{proof} \hfill

\begin{asparaenum}[(i)]
\item Let $x,y,g \in G$. Then
\[
\lbrack g^{-1}xg, g^{-1}yg \rbrack = ( (g^{-1}xg)(g^{-1}yg) )( (g^{-1}yg)(g^{-1}xg))^{-1} = (g^{-1}xgg^{-1}yg) ( g^{-1}xg)^{-1} (g^{-1}yg)^{-1} 
\]
\[
= (g^{-1}xyg)( g^{-1}x^{-1}g g^{-1}y^{-1}g) = g^{-1}xyx^{-1}y^{-1}g = g^{-1} \lbrack x,y \rbrack g.
\]
This shows that the set of commutators is invariant under conjugation and therefore $\comm{a}(G)$ is a normal subgroup of $G$. As $\comm{t}(G)$ is the closure of $\comm{a}(G)$, it follows from proposition \ref{prop:prop_of_quots} that $\comm{t}(G)$ is also a normal subgroup of $G$.

\item It is evident that $G$ is abelian if and only if $\comm{a}(G) = 1$. Let $G$ be separated and abelian. Then $\comm{a}(G) = 1$ and therefore $\comm{t}(G) = \ro{cl}( \lbrace 1 \rbrace ) = 1$ by proposition \ref{prop:prop_of_top_grps}. Conversely, if $\comm{t}(G) = 1$, then $G$ is abelian because $\comm{a}(G) \subs \comm{t}(G)$. Moreover, since 
\[
1 = \comm{t}(G) = \ro{cl}(\comm{a}(G)) = \ro{cl}(\lbrace 1 \rbrace),
\]
if follows from proposition \ref{prop:prop_of_top_grps} that $G$ is separated.

\item Let $x,y \in G$. Then $xy(yx)^{-1} = \lbrack x,y \rbrack \in \comm{a}(G)$ so that $xy \equiv yx \ \ro{mod} \ \comm{a}(G)$. Hence, $G/\comm{a}(G)$ is abelian. Since $\comm{a}(G) \subs \comm{t}(G)$, the quotient group $G/\comm{t}(G)$ is also abelian. Moreover, as $\comm{t}(G)$ is closed in $G$, it follows from proposition \ref{prop:prop_of_quots} that $G/\comm{t}(G)$ is separated.

\item If $G/N$ is abelian, then $xy \equiv yx \ \ro{mod} \ N$ and so $xy(yx)^{-1} = \lbrack x, y \rbrack \in N$ for all $x,y \in G$. Hence, $N \geq \comm{a}(G)$ and since $N$ is closed, we get $N \geq \comm{t}(G)$. Conversely, if $N \geq \comm{t}(G)$ then in particular $N \geq \comm{a}(G)$. Thus, if $x,y \in G$, then $xy(yx)^{-1} = \lbrack x,y \rbrack \in \comm{a}(G) \subs N$ and therefore $xy \equiv yx \ \ro{mod} \ N$ so that $G/N$ is abelian. 

\item Let $x,y \in G$. Then 
\[
\varphi( \lbrack x,y \rbrack ) = \varphi( xy(yx)^{-1} ) = \varphi(x) \varphi(y) ( \varphi(y) \varphi(x))^{-1} = \lbrack \varphi(x), \varphi(y) \rbrack \in \comm{a}(G')
\]
and therefore $\varphi(\comm{a}(G)) \subs \comm{a}(G')$. Using the continuity of $\varphi$, we get
\[
\varphi(\comm{t}(G)) = \varphi( \ro{cl}(\comm{a}(G))) \subs \ro{cl}( \varphi(\comm{a}(G))) \subs \ro{cl}( \comm{a}(G') ) =  \comm{t}(G').
\]

\item By the above, we have $\varphi(\comm{t}(G)) \subs \comm{t}(G')$, so it remains to show the converse inclusion. Let $x',y' \in G'$ and let $x,y \in G$ such that $\varphi(x) = x'$, $\varphi(y) = y'$. Then
\[
\varphi( \lbrack x,y \rbrack ) = \varphi( xy(yx)^{-1}) = x'y'(y'x')^{-1} = \lbrack x',y' \rbrack
\]
and this implies $\comm{a}(G') \subs \varphi(\comm{a}(G))$. As $\varphi$ is closed, we get
\[
\comm{t}(G') = \ro{cl}\comm{a}(G')) \subs \ro{cl}(\varphi(\comm{a}(G))) \subs \varphi( \ro{cl}(\comm{a}(G))) = \varphi(\comm{t}(G)).
\]

\item If such a morphism $\varphi'$ exists, then due to the commutativity of the diagram, it is necessarily given by $\varphi'(g \ \ro{mod} \ \comm{t}(G)) = \varphi'(G)$. It remains to verify that this indeed defines a morphism of topological groups. Since $\varphi(\comm{t}(G)) \subs \comm{t}(A) = 1$, we have $\comm{t}(G) \subs \ker(\varphi)$ and therefore $\varphi'$ is a well-defined morphism of abstract groups satisfying $\varphi' \circ q = \varphi$. It follows from \ref{para:morphism_quotient_induced} that $\varphi'$ is continuous so that $\varphi'$ is a morphism of topological groups.

\item The morphism $q_{G'} \circ \varphi:G \ra (G')^{\ro{ab}}$ is a morphism into a separated abelian group and so there exists a unique morphism $\varphi^{\ro{ab}}: G^{\ro{ab}} \ra (G')^{\ro{ab}}$ such that $\varphi^{\ro{ab}} \circ q_G = q_{G'} \circ \varphi$.

\item This follows immediately from the uniqueness of $\varphi^{\ro{ab}}$.
\end{asparaenum} \vspace{-\baselineskip}
\end{proof}

\begin{prop} \label{para:abelianization_of_quotient}
Let $G$ be a compact group and let $H$ be a closed subgroup of $G$. Then 
\[
\comm{t}(G/H) = H\comm{t}(G)/H \lhd G/H
\]
and there exists a canonical isomorphism
\[
G/H\comm{t}(G) \cong (G/H)^{\ro{ab}}.
\]
\end{prop}

\begin{proof}
Let $q:G \ra G/H$ be the quotient morphism. Since $H$ is closed, $G/H$ is separated and therefore $q$ is closed by the closed map lemma. Hence, we can apply \ref{para:props_of_commuator_subgroup}\ref{item:top_commutator_surj_mor} and get $\comm{t}(G/H) = q(\comm{t}(G)) = H\comm{t}(G)/H$. An application of \ref{prop:prop_of_quots}\ref{item:top_grp_2_iso_theorem} yields the canonical isomorphism of topological groups 
\[
(G/H)^{\ro{ab}} = (G/H)/\comm{t}(G/H) = (G/H)/(H\comm{t}(G)/H) \cong G/H\comm{t}(G).
\]
\end{proof}

\newpage
\section{Projective limits of topological groups} \label{sec:proj_limits}

In this chapter some basic facts about profinite groups are collected. This chapter mainly emerged from the author's considerations about how much from the profinite theory can be generalized to projective limits of compact groups, leading to the notion of complete proto-$\ca{C}$ groups for a formation $\ca{C}$ of compact groups. In particular, the theory of maximal pro-$\ca{C}$ quotients as discussed in \cite[section 3.4]{RibZal00_Profinite-Groups_0} was generalized to this setting. The main reference and motivation for the first part of this chapter was  \cite{HofMor09_Contributions-to-the-structure_0}.

\subsection{Basic facts}

\begin{prop} \label{prop:tgrp_complete} \wordsym{$\se{TGrp}$} \wordsym{$\se{TGrp}^{\ro{s}}$} \wordsym{$\se{TGrp}^{\ro{s}, \ro{cpl}}$} \wordsym{$\se{TGrp}^{\ro{com}}$}
The following holds:
\begin{enumerate}[label=(\roman*),noitemsep,nolistsep]
\item \label{item:tgrp_complete} The category $\se{TGrp}$ of topological groups\footnote{Not necessarily separated.} is complete. The product of a family $\lbrace G_i \mid i \in I \rbrace$ of topological groups is given by the direct product $\prod_{i \in I} G_i$ equipped with the product topology. The pullback of the diagram formed by two morphisms $\varphi:G \ra G''$ and $\varphi':G' \ra G''$ is given by the subgroup $P = \lbrace (g,g') \in G \times G' \mid \varphi(g) = \varphi'(g') \rbrace$ of the product $G \times G'$.

\item The pullback of the diagram formed by two coherent coretractions\footnote{Confer \cite[definition 1.6]{HofMor09_Contributions-to-the-structure_0} for this notion.} $\varphi,\varphi':G \ra G''$ in $\se{TGrp}$ is also given by $\ker(\gamma) \leq G$, where $\gamma:G \ra G''$, $g \mapsto \varphi(g) \varphi'(g)^{-1}$.

\item \label{item:closed_full_subcat_of_tgrp} If $\ca{C}$ is a full subcategory of $\se{TGrp}$ consisting of separated groups and which is closed under the formation of products and passing to closed subgroups, then $\ca{C}$ is closed under the formation of limits. In particular, $\ca{C}$ is complete. 

\item The full subcategory $\se{TGrp}^{\ro{s}}$ ($\se{TGrp}^{\ro{s,cpl}}$, $\se{TGrp}^{\ro{com}}$) of $\se{TGrp}$ consisting of separated (separated and complete, compact) groups is closed under the formation of limits. In particular, it is a complete category.
\end{enumerate}

\end{prop}

\begin{proof} \hfill

\begin{asparaenum}[(i)]
\item This is easy to verify.

\item We have $K \dopgleich \ker(\gamma) = \lbrace g \in G \mid \varphi(g) = \varphi'(g) \rbrace$. By \ref{item:tgrp_complete}, the pullback of the diagram formed by $\varphi$ and $\varphi'$ is given by $P \dopgleich \lbrace (g,g') \in G \times G \mid \varphi(g) = \varphi'(g') \rbrace$. Consider the map $\alpha: K \ra P$, $g \mapsto (g,g)$ and let $p:G \times G \ra G$ be the projection onto the first factor. Both maps are morphisms of topological groups and we have $p \circ \alpha = \id_{K}$. Now, let $(g,g') \in P$. As $\varphi$ and $\varphi'$ are coherent coretractions, there exists a morphism $r:G'' \ra G$ such that $r \varphi = r \varphi' = \id_G$. Since $(g,g') \in P$, we have $\varphi(g) = \varphi(g')$ and therefore $r \varphi(g) = r \varphi'(g')$ which implies $g = g'$. Hence, 
\[
(g,g') = (g,g) = \alpha(g) = \alpha \circ p( (g,g'))
\]
and so we conclude that $\alpha \circ p |_P = \id_{P}$. This shows that $\alpha:K \ra P$ is an isomorphism of topological groups.

\item If $\varphi,\varphi':G \ra G''$ are two coherent coretractions in $\ca{C}$, then the pullback of the diagram formed by these morphisms is given by $\ker(\gamma)$, where $\gamma:G \ra G''$ is defined as above. As $G''$ is separated, $\ker(\gamma)$ is a closed subgroup of $G$ and as $G \in \ca{C}$, it follows that $\ker(\gamma) \in \ca{C}$ by assumption. Hence, $\ca{C}$ is closed under the formation of products and the passing to intersections of coherent retracts. As $\se{TGrp}$ is complete, an application of \cite[theorem 1.10(iv)]{HofMor09_Contributions-to-the-structure_0} shows that $\ca{C}$ is closed under the formation of limits.

\item By \ref{item:closed_full_subcat_of_tgrp} it is enough to show that the closed subspaces and products of separated (separated and complete, compact) spaces are again separated (separated and complete, compact). All this follows from \cite[chapitre I, \S8.2]{Bou71_Topologie-Generale_0}, \cite[chapitre II, \S3.5, proposition 10]{Bou71_Topologie-Generale_0}, \cite[chapitre II, \S3.4, proposition 8]{Bou71_Topologie-Generale_0}, \cite[chapitre I, \S9.5, th\'eor\`eme 2]{Bou71_Topologie-Generale_0} and \cite[chapitre I, \S9.3, proposition 2]{Bou71_Topologie-Generale_0}.
\end{asparaenum} \vspace{-\baselineskip}
\end{proof}

\begin{defn}
Any partially ordered set $I = (I,\leq)$ can be considered as a category whose objects are the elements of $I$ and in which for every $i,j \in I$ with $i \leq j$ there exists a unique morphism $j \ra i$. We say that $I$ is a \word{directed set} if any two elements of $I$ have an upper bound. A \word{projective system} in a category $\ca{C}$ is a functor $\ca{P}:I \ra \ca{C}$, where $I$ is a directed set. We often write $\ca{P}$ as a family of morphisms $\ca{P} = \lbrace \varphi_{ij}: X_j \ra X_i \mid (i,j) \in I \times I, i \leq j \rbrace$, where $X_i \dopgleich \ca{P}(i)$ and $\varphi_{ij} \dopgleich \ca{P}(j \ra i)$.
\end{defn}

\begin{prop} \label{prop:tgrp_proj_limit_explicit}
The limit of a projective system $\ca{P} = \lbrace \varphi_{ij}: G_j \ra G_i \mid (i,j) \in I \times I, i \leq j \rbrace$ in $\se{TGrp}$ is given by the subgroup 
\[
\lbrace (g_i)_{i \in I} \in \prod_{i \in I} G_i \mid \varphi_{ij}(g_j) = g_i \ \textnormal{for all} \ i \leq j \rbrace \leq \prod_{i \in I} G_i.
\]
\end{prop}

\begin{proof}
This is easy to verify.
\end{proof}

\subsection{Approximations}

\begin{defn} \wordsym{$\se{FTGrp}$}
The category $\se{FTGrp}$ of \words{filtered topological groups}{filtered topological group} is defined as follows: 
\begin{compactitem}
\item The objects are pairs $(G,\ca{N})$, where $G$ is a \textit{separated} topological group and $\ca{N}$ is a filter basis on $G$ consisting of closed normal subgroups of $G$. 
\item The morphisms $(G,\ca{N}_G) \ra (H,\ca{N}_H)$ are morphisms $\varphi:G \ra H$ in $\se{TGrp}$ satisfying the following properties:
\begin{compactenum}
\item $\varphi(N) \in \ca{N}_H$ for each $N \in \ca{N}_G$.
\item For each $M \in \ca{N}_H$ there exists $N \in \ca{N}_G$ such that $\varphi(N) \subs M$.
\end{compactenum}
\item The composition is the composition in $\se{TGrp}$. Note that this is well-defined: Let $\varphi_1: (G_1,\ca{N}_1) \ra (G_2,\ca{N}_2)$ and $\varphi_2: (G_2,\ca{N}_2) \ra (G_3,\ca{N}_3)$ be two morphisms in $\se{FTGrp}$. If $N_1 \in \ca{N}_1$, then $\varphi_1(N_1) \in \ca{N}_2$ and so $\varphi_2 \circ \varphi_1(N_1) \in \ca{N}_3$. If $N_3 \in \ca{N}_3$, then there exists $N_2 \in \ca{N}_2$ such that $\varphi_2(N_2) \subs N_3$. Now, there exists $N_1 \in \ca{N}_1$ such that $\varphi_1(N_1) \subs N_2$. Hence, $\varphi_2 \circ \varphi_1(N_1) \subs N_3$ and so $\varphi_2 \circ \varphi_1$ is a morphism in $\se{FTGrp}$.
\end{compactitem}
\end{defn}

\begin{defn} \wordsym{$G/\ca{N}$}
For $(G,\ca{N}) \in \se{FTGrp}$ the projective system $G/\ca{N}$ in $\se{TGrp}^{\ro{s}}$ is defined as follows: The index set is $(\ca{N},\sups)$, the objects are $\lbrace G/N \mid N \in \ca{N} \rbrace$, and the morphisms $\lbrace G/M \ra G/N \mid M \leq N \in \ca{N} \rbrace$ are induced by the quotient morphisms. The limit $\ro{lim} \ G/\ca{N}$ is denoted by $G_{\ca{N}}$ and is called the \word{approximation} of $G$ by $\ca{N}$.
\end{defn}

\begin{thm} \label{thm:approximations}
Let $(G,\ca{N}) \in \se{FTGrp}$. The following holds:
\begin{compactenum}[(i)]
\item The map $\gamma_{\ca{N}}:G \ra G_{\ca{N}}$, $g \mapsto (gN)_{N \in \ca{N}}$ is a morphism of topological groups.
\item For all $N \in \ca{N}$ the diagram
\[
\xymatrix{
G \ar[r]^{\gamma_{\ca{N}}} \ar[dr]_{q_N} & G_{\ca{N}} \ar[d]^{p_N} \\
& G/N
}
\]
commutes, where $p_N$ is the canonical morphism and $q_N$ is the quotient morphism.

\item $\gamma_{\ca{N}}(G)$ is dense in $G_{\ca{N}}$ and $\ker(\gamma_{\ca{N}}) = \bigcap_{N \in \ca{N}} N$.

\item The canonical morphisms $p_N:G_{\ca{N}} \ra G/N$, $N \in \ca{N}$, are strict surjective. Hence, $p_N$ induces an isomorphism of topological groups $G_{\ca{N}}/\ker(p_N) \cong G/N$. Moreover, $\ker(p_N) = \ro{cl}(\gamma_{\ca{N}}(N))$.

\item $G_{\ca{N}}$ is separated.

\item $\gamma_{\ca{N}}$ is a strict injective morphism (and thus topologically an embedding) if and only if $\ro{lim} \ \ca{N} = 1$.

\item If $\ro{lim} \ \ca{N} = 1$ and if $\ca{N}$ contains a complete group, then $\gamma_{\ca{N}}$ is surjective and hence it is an isomorphism of topological groups.

\item If $\varphi:(G,\ca{N}_G) \ra (H,\ca{N}_H)$ is a morphism in $\se{FTGrp}$, then there exists a unique morphism $\varphi^{\ro{a}}:G_{\ca{N}_G} \ra H_{\ca{N}_H}$ such that the diagram
\[
\xymatrix{
G_{\ca{N}_G} \ar[r]^{\varphi^{\ro{a}}} & H_{\ca{N}_H} \\
G \ar[r]_{\varphi} \ar[u]^{\gamma_{\ca{N}_G}} & H \ar[u]_{\gamma_{\ca{N}_H}}
}
\]
commutes.

\item The maps
\[
\begin{array}{rcl}
\se{FTGrp} & \lra & \se{TGrp}^{\ro{s}} \\
(G,\ca{N}) & \longmapsto & G_{\ca{N}} \\
(G,\ca{N}_G) \xrightarrow{\varphi} (H,\ca{N}_H) & \longmapsto & G_{\ca{N}_G} \xrightarrow{\varphi^{\ro{a}}} H_{\ca{N}_H}
\end{array}
\]
define a functor. This functor is called the \word{approximation functor} on $\se{FTGrp}$.

\end{compactenum}

\end{thm}

\begin{proof} \hfill

\begin{asparaenum}[(i)]
\item This is the first part of \cite[theorem 1.29(i)]{HofMor09_Contributions-to-the-structure_0}.

\item This is \cite[theorem 1.29(ii)]{HofMor09_Contributions-to-the-structure_0}.

\item This is the second part of \cite[theorem 1.29(i)]{HofMor09_Contributions-to-the-structure_0}.

\item This is \cite[theorem 1.29(iii)]{HofMor09_Contributions-to-the-structure_0}.

\item This follows from \ref{prop:tgrp_complete} since all $N \in \ca{N}$ are closed in $G$ so that $G/N$ is separated. 

\item Let $\gamma_{\ca{N}}$ be a strict injective morphism. Then $1 = \ker(\gamma_{\ca{N}}) = \bigcap_{N \in \ca{N}} N$ and as $\gamma_{\ca{N}}$ is topologically an embedding, it follows from \cite[theorem 1.30]{HofMor09_Contributions-to-the-structure_0} that $\ro{lim} \ \ca{N} = 1$. Conversely, if $\ro{lim} \ \ca{N} = 1$, then it follows from \cite[chapitre III, \S7.2, proposition 2]{Bou71_Topologie-Generale_0} that $\gamma_{\ca{N}}$ is a strict morphism and 
\[
\bigcap_{N \in \ca{N}} N = \ker(\gamma_{\ca{N}}) = \ro{cl}(\lbrace 1 \rbrace).
\]
Since $G$ is separated, we have $\ro{cl}(\lbrace 1 \rbrace) = 1$ and so $\gamma_{\ca{N}}$ is a strict injective morphism.

\item By the above, $\gamma_{\ca{N}}$ is a strict injective morphism. It follows from \cite[chapitre III, \S7.2, proposition 2]{Bou71_Topologie-Generale_0} that $\gamma_{\ca{N}}$ is surjective and so it is an isomorphism of topological groups.

\item This is \cite[theorem 1.29(v)]{HofMor09_Contributions-to-the-structure_0}.

\item This follows from the uniqueness of $\varphi^{\ro{a}}$. \vspace{-\baselineskip}
\end{asparaenum}
\end{proof}

\begin{defn} \wordsym{$\se{FTGrp}^1$}
The category $\se{FTGrp}^1$ is defined as the full subcategory of $\se{FTGrp}$ consisting of all $(G,\ca{N})$ such that $\ro{lim} \ \ca{N} = 1$.
\end{defn}

\begin{prop}
 Let $(G,\ca{N}) \in \se{FTGrp}^1$. If the quotient $G/N$ is complete for each $N \in \ca{N}$, then $G$ admits a Hausdorff completion which is given by $\gamma_{\ca{N}}$.
\end{prop}

\begin{proof}
This is  \cite[chapitre III, \S7.2, corollaire 1 de proposition 2]{Bou71_Topologie-Generale_0}.
\end{proof}

\begin{prop} \label{prop:closed_subgrp_proj_lim}
Let $(G,\ca{N}) \in \se{FTGrp}^1$ such that $G$ is complete and the quotient $G/N$ is complete for each $N \in \ca{N}$. The following holds for a closed subgroup $H$ of $G$:
\begin{compactenum}[(i)]
\item Let $H\ca{N}/\ca{N}$ be the projective system in $\se{TGrp}^{\ro{s}}$ with index set $(\ca{N},\sups)$, objects $\lbrace HN/N \leq G/N \mid N \in \ca{N} \rbrace$, and morphisms induced by the quotient morphisms. Similarly, let $\ro{cl}(H \ca{N})/\ca{N}$ be the projective system in $\se{TGrp}^{\ro{s}}$ with index set $(\ca{N},\sups)$, objects $\lbrace \ro{cl}(HN)/N \leq G/N \mid N \in \ca{N} \rbrace$, and morphisms induced by the quotient morphisms. Then canonically
\[
\ro{lim} \ H \ca{N}/\ca{N} \subs \ro{lim} \ \ro{cl}(H \ca{N})/\ca{N} \subs \ro{lim} \ G/\ca{N} = G_{\ca{N}}.
\]
Moreover,
\[
\ro{lim} \ H \ca{N}/\ca{N} = \ro{lim} \ \ro{cl}(H \ca{N})/\ca{N}
\]
and the isomorphism $\gamma_{\ca{N}}:G \ra G_{\ca{N}}$ maps $H$ isomorphically onto $\ro{lim} \ H \ca{N}/\ca{N}$.

\item The set $H \cap \ca{N} \dopgleich \lbrace H \cap N \mid N \in \ca{N} \rbrace$ is a filter basis on $H$ consisting of closed subgroups and having $1$ as limit point. Moreover, 
\[
\gamma_{H \cap \ca{N}}:H \ra \ro{lim} \ H/H \cap \ca{N} = H_{H \cap \ca{N}}
\]
 is an isomorphism.

\end{compactenum}
\end{prop}

\begin{proof}
This is a slight reformulation of \cite[theorem 1.34(i)-(ii)]{HofMor09_Contributions-to-the-structure_0} using some details of the proof given there.\footnote{Note that \cite[theorem 1.34(iii)-(iv)]{HofMor09_Contributions-to-the-structure_0} is false as pointed out by Helge Gl\"ockner. Confer also \cite{HofMor09_Enhancements-for-the-Lie-theory_0}.}
\end{proof}

\begin{prop} \label{prop:quot_grp_proj_lim}
Let $(G,\ca{N}) \in \se{FTGrp}^1$ such that all $N \in \ca{N}$ are compact. Let $H$ be a closed normal subgroup of $G$. Let $q:G \ra G/H$ be the quotient morphism. The set $q(\ca{N}) \dopgleich \lbrace q(N) \mid N \in \ca{N} \rbrace$ is a filter basis on $G/H$ consisting of closed normal subgroups of $G/H$ and having $1$ as limit point. The morphism
\[
\gamma_{q(\ca{N})}: G/H \ra \ro{lim} \ (G/H)/q(\ca{N}) = (G/H)_{q(\ca{N})}
\]
is an isomorphism.
\end{prop}

\begin{proof} 
 The set $q(\ca{N})$ is obviously non-empty and does not contain the empty set. If $q(N),q(N') \in q(\ca{N})$, then there exists $N'' \in \ca{N}$ such that $N'' \subs N \cap N'$ because $\ca{N}$ is a filter basis on $G$. Hence, $q(N'') \subs q(N) \cap q(N')$ and this shows that $q(\ca{N})$ is a filter basis on $G/H$. Since each $N \in \ca{N}$ is a compact normal subgroup of $G$ and $q$ is continuous, each $q(N)$ is a compact (and thus closed) normal subgroup of $G/H$. Let $U$ be an open neighborhood of $1 \in G/H$. Then $q^{-1}(U)$ is an open neighborhood of $1 \in G$ and since $\ro{lim} \ \ca{N} = 1$, there exists $N \in \ca{N}$ such that $N \subs q^{-1}(U)$. Hence, $q(N) \subs q(q^{-1}(U)) = U$ and this shows that $\ro{lim} \ q(\ca{N}) = 1$. Hence, $(G/N,q(\ca{N})) \in \se{FTGrp}^1$ and as $q(\ca{N})$ consists of compact (and thus complete) normal subgroups of $G/H$, it follows from \ref{thm:approximations} that $\gamma_{q(\ca{N})}:G/H \ra \ro{lim} \ (G/H)/q(\ca{N})$ is an isomorphism.
\end{proof}

\begin{thm} \label{thm:proj_limits_internal}
Let $\ca{P} = \lbrace \varphi_{ij}: G_j \ra G_i \mid (i,j) \in I \times I, i \leq j \rbrace$ be a projective system in $\se{TGrp}^{\ro{s}}$, let $G \dopgleich \ro{lim} \ \ca{P}$ and let $p_i:G \ra G_i$ be the canonical morphism. The following holds:
\begin{enumerate}[label=(\roman*),noitemsep,nolistsep]
\item For each $i \in I$ let $\fr{V}_i$ be the neighborhood filter of $1 \in G_i$. Then $\lbrace p_i^{-1}(U) \mid i \in I \wedge U \in \fr{V}_i \rbrace$ is a neighborhood basis of $1 \in G$.

\item $\ca{N} \dopgleich \lbrace \ker(p_i) \mid i \in I \rbrace$ is a filter basis on $G$ consisting of closed normal subgroups of $G$ and having $1$ as limit point.

\item There exists a unique isomorphism $\eta_{\ca{N}}: G_{\ca{N}} \ra G$ such that the following diagram commutes for all $i \leq j$ in $I$
\[
\xymatrix{
G_{\ca{N}} \ar[rr]^{\eta_{\ca{N}}}_{\cong} \ar[d]_{\widetilde{p_j}} & & G \ar[d]^{p_j} \\
G/\ker(p_j) \ar[rr]^{p_j'} \ar[d]_{q_{ij}} & & G_j \ar[d]^{\varphi_{ij}} \\
G/\ker(p_i) \ar[rr]^{p_i'}  & & G_i 
}
\]
where $\widetilde{p_j}$ is the canonical morphism, $q_{ij}$ is induced by the quotient morphism and the two lower horizontal morphisms are induced by the canonical morphisms.\footnote{Note that $q_{ij}$ is well-defined since $p_i = \varphi_{ij}p_j$ and therefore $\ker(p_j) \subs \ker(p_i)$.} Moreover, $\eta_{\ca{N}} \circ \gamma_{\ca{N}} = \id_G$.

\item If $\ca{N}$ contains a complete group, then $\eta_{\ca{N}}$ is inverse to $\gamma_{\ca{N}}$.

\item \label{item:proj_system_surjective_mod} Let $G_i' \dopgleich \ro{cl}(p_i(G))$ for each $i \in I$ and let $\varphi_{ij}':G_j' \ra G_i'$ be the morphism induced by $\varphi_{ij}$ for each pair $i \leq j \in I$.\footnote{This is well-defined: Since $\varphi_{ij}:G_j \ra G_i$ is continuous, we have $\varphi_{ij}(G_j') = \varphi_{ij}( \ro{cl}(p_j(G))) \subs \ro{cl}( \varphi_{ij} \circ p_j(G)) = \ro{cl}( p_i(G)) = G_i'$.} Let $\ca{P}' = \lbrace \varphi_{ij}': G_j' \ra G_i' \mid (i,j) \in I \times I, i \leq j \rbrace$. Then $\ro{lim} \ \ca{P} \cong \ro{lim} \ \ca{P}'$. Moreover, for each $i \in I$ the canonical morphism $\ro{lim} \ \ca{P}' \ra G_i'$ is the corestriction of $p_i$ and has dense image.

\end{enumerate}
\end{thm}

\begin{proof}
The first three statements are \cite[theorem 1.27(i)-(ii)]{HofMor09_Contributions-to-the-structure_0}. The relation $\eta_{\ca{N}} \circ \gamma_{\ca{N}} = \id_G$ is shown in the proof of the theorem. The fourth statement follows from this relation since $\gamma_{\ca{N}}$ is an isomorphism in this case. The last statement is \cite[theorem 1.27(iv)]{HofMor09_Contributions-to-the-structure_0}.
\end{proof}

\begin{defn}
A projective system $\ca{P}:I \ra \se{TGrp}$ is called \words{surjective}{projective system!surjective} (\words{strict surjective}{projective system!strict surjective}) if $\ca{P}(j \ra i)$ is surjective (strict surjective) for all $i \leq j \in I$. 
It is called \words{universally surjective}{projective system!universally surjective} (\words{universally strict surjective}{projective system!universally strict surjective}) if the canonical morphism $\ro{lim} \ \ca{P} \ra \ca{P}(i)$ is surjective (strict surjective) for all $i \in I$.
\end{defn}

\begin{prop} \label{prop:proj_systems_surjectivity}
Let $\ca{P}:I \ra \se{TGrp}^{\ro{s}}$ be a projective system. The following holds:
\begin{compactenum}[(i)]
\item If $\ca{P}$ is universally surjective (universally strict surjective), then $\ca{P}$ is surjective (strict surjective).

\item If $\ca{P}$ is strict surjective and universally surjective, then $\ca{P}$ is universally strict surjective.

\item If $\ca{P}$ is surjective and all morphisms in $\ca{P}$ have compact kernel, then $\ca{P}$ is universally surjective. In this case, all canonical morphisms $\ro{lim} \ \ca{P} \ra \ca{P}(i)$ also have compact kernel.

\end{compactenum}
\end{prop}

\begin{proof} \hfill

\begin{asparaenum}[(i)]
\item Let $G_i \dopgleich \ca{P}(i)$, $\varphi_{ij} \dopgleich \ca{P}(j \ra i)$ and let $p_i: \ro{lim} \ \ca{P} \ra G_i$ be the canonical morphism. The diagram
\[
\xymatrix{
& \ro{lim} \ \ca{P} \ar[dl]_{p_j} \ar[dr]^{p_i} & \\
G_j \ar[rr]^{\varphi_{ij}} & & G_i
}
\]
commutes by definition. If $\ca{P}$ is universally surjective, then both $p_i$ and $p_j$ are surjective and therefore $\varphi_{ij}$ is also surjective. Now, assume that $\ca{P}$ is universally strict surjective. Then $\ca{P}$ is in particular universally surjective and therefore surjective by the above. It remains to show that $\varphi_{ij}$ is strict and since $\varphi_{ij}$ is surjective, this amounts to proving that $\varphi_{ij}$ is open. So, let $U$ be an open subset of $G_j$. Since both $p_i$ and $p_j$ are surjective, it follows that $\varphi_{ij}(U) = p_i( p_j^{-1}(U))$. In fact, if $x \in U$, then we can choose $y \in p_j^{-1}(x) \subs p_j^{-1}(U)$ and get $\varphi_{ij}(x) = \varphi_{ij}(p_j(y)) = p_i(x) \in p_i( p_j^{-1}(U))$. Conversely, we have $p_i( p_j^{-1}(U)) = \varphi_{ij} \circ p_j( p_j^{-1}(U)) \subs \varphi_{ij}(U)$. Since $p_j$ is continuous, the preimage $p_j^{-1}(U)$ is open in $\ro{lim} \ \ca{P}$ and as $p_i$ is a strict surjective morphism, it is open and so $\varphi_{ij}(U) = p_i(p_i^{-1}(U))$ is open. Hence, $\varphi_{ij}$ is strict surjective.

\item This is \cite[proposition 1.27(iii)]{HofMor09_Contributions-to-the-structure_0}.

\item This is the second part of \cite[proposition 1.22]{HofMor09_Contributions-to-the-structure_0}. 
\end{asparaenum} \vspace{-\baselineskip}
\end{proof}


\begin{defn}
For a subclass $\ca{C} \subs \se{TGrp}^{\ro{s}}$ the following full subcategories of $\se{TGrp}^{\ro{s}}$ are defined:
\begin{compactitem}
\item $\se{pro}\tn{-}\ca{C}$ ($\se{spro}\tn{-}\ca{C}$, $\se{uspro}\tn{-}\ca{C}$, $\se{sspro}\tn{-}\ca{C}$, $\se{usspro}\tn{-}\ca{C}$) consists of all topological groups which are isomorphic to the limit of an arbitrary (surjective, universally surjective, strict surjective, universally strict surjective) projective system $\ca{P}:I \ra \se{TGrp}$ such that $\ca{P}(i) \in \ca{C}$ for all $i \in I$.
\item $\se{proto}\tn{-}\ca{C}$ ($\se{cplproto}\tn{-}\ca{C}$) consists of all separated (separated and complete) topological groups $G$ for which there exists a filter basis $\ca{N}$ on $G$ consisting of closed normal subgroups of $G$ such that $\ro{lim} \ \ca{N} = 1$ and $G/N \in \ca{C}$ for each $N \in \ca{N}$.
\end{compactitem}
\end{defn}

\begin{defn} \label{def:group_class_props}
We introduce the following list of properties a subclass $\ca{C} \subs \se{TGrp}^{\ro{s}}$ might satisfy:
\begin{compactenum}[(C1)]
\item[(C0)] $\ca{C} \neq \emptyset$ and if $G \in \se{TGrp}$ is isomorphic to a group in $\ca{C}$, then $G$ is already contained in $\ca{C}$.\footnote{This property is well-defined due to our assumptions on the set theory.}
\item If $G \in \ca{C}$ and $H \leq_{\ro{c}} G$, then $H \in \ca{C}$.
\item If $G \in \ca{C}$ and $N \lhd_{\ro{c}} G$, then $G/N \in \ca{C}$.
\item If $n \in \NN$ and $G_1,\ldots,G_n \in \ca{C}$, then $\prod_{i=1}^n G_i \in \ca{C}$.
\end{compactenum}
\end{defn}

\begin{prop} \label{prop:pro_cats_inclusions}
Let $\ca{C} \subs \se{TGrp}^{\ro{s}}$. The following holds:
\begin{compactenum}[(i)]
\item If $\ca{C}' \subs \ca{C}$, then $\star\se{pro}\tn{-}\ca{C}' \subs \star\se{pro}\tn{-}\ca{C}$ for $\star \in \lbrace \emptyset, \se{s,us,ss,uss} \rbrace$ and $\star\se{proto}\tn{-}\ca{C}' \subs \star\se{proto}\tn{-}\ca{C}$ for $\star \in \lbrace \emptyset, \se{cpl} \rbrace$.

\item There exist the following inclusions of full subcategories of $\se{TGrp}^{\ro{s}}$, where the dashed inclusion with label C0 only holds if $\ca{C}$ satisfies C0 and where the dashed arrow with label cpl$\wedge$C0 symbolizes that the converse inclusion holds if $\ca{C}$ satisfies C0 and consists of complete and separated groups: 
\[
\xymatrix 
@C=10pt
@R=15pt
{
& & & & & & \se{pro}\tn{-}\ca{C} \\
& & & & & & \se{spro}\tn{-}\ca{C} \ar@{-}[u] \\
& & & & \se{uspro}\tn{-}\ca{C} \ar@{-}[rru] \\
\se{proto}\tn{-}\ca{C} & & & & & \se{sspro}\tn{-}\ca{C} \ar@{-}[ruu] \\
& & & \se{usspro}\tn{-}\ca{C} \ar@{--}[lllu]_{\ro{C0}} \ar@{-}[ruu] \ar@{-}[rru] \ar@/^1pc/@{-->}[dlll]^{\ro{cpl}\wedge\ro{C0}} \\
\se{cplproto}\tn{-}\ca{C} \ar@{-}[urrr] \ar@{-}[uu] & & & \ca{C} \ar@{-}[u]
}
\]

\end{compactenum}
\end{prop}

\begin{proof} \hfill

\begin{asparaenum}[(i)]

\item This is obvious.

\item The inclusions $\se{cplproto}\tn{-}\ca{C} \subs \se{proto}\tn{-}\ca{C}$, $\ca{C} \subs \se{usspro}\tn{-}\ca{C}$, $\se{usspro}\tn{-}\ca{C} \subs \se{uspro}\tn{-}\ca{C}$, $\se{sspro}\tn{-}\ca{C} \subs \se{spro}\tn{-}\ca{C}$ and $\se{spro}\tn{-}\ca{C} \subs \se{pro}\tn{-}\ca{C}$ are obvious. 
The inclusions $\se{usspro}\tn{-}\ca{C} \subs \se{sspro}\tn{-}\ca{C}$, $\se{uspro}\tn{-}\ca{C} \subs \se{spro}\tn{-}\ca{C}$ follow from \ref{prop:proj_systems_surjectivity}.

Let $G \in \se{cplproto}\tn{-}\ca{C}$. By definition, $G$ is a complete and separated group and there exists a filter basis $\ca{N}$ on $G$ consisting of closed normal subgroups such that $G/N \in \ca{C}$ for each $N \in \ca{N}$ and $\ro{lim} \ \ca{N} = 1$. Since each $N \in \ca{N}$ is closed in $G$ and since $G$ is complete, each $N \in \ca{N}$ is also complete by \cite[chapitre II, \S3.4, proposition 8]{Bou71_Topologie-Generale_0}. Hence, $\gamma_{\ca{N}}:G \ra G_{\ca{N}}$ is an isomorphism by \ref{thm:approximations} and $G_{\ca{N}} \in \se{usspro}\tn{-}\ca{C}$. This shows that $G \in \se{usspro}\tn{-}\ca{C}$ and therefore $\se{cplproto}\tn{-}\ca{C} \subs \se{usspro}\tn{-}\ca{C}$.

Let $G \in \se{usspro}\tn{-}\ca{C}$ and suppose that $\ca{C}$ satisfies C0. Then there exists a universally strict surjective projective system $\ca{P}:I \ra \se{TGrp}$ such that $\ca{P}(i) \in \ca{C}$ for all $i \in I$ and $G \cong \ro{lim} \ \ca{P}$. We can assume without loss of generality that $G = \ro{lim} \ \ca{P}$. Since $\ca{C}$ consists of separated groups, $G$ is also separated by \ref{prop:tgrp_complete}. Let $G_i \dopgleich \ca{P}(i)$, $\varphi_{ij} \dopgleich \ca{P}(j \ra i)$ and let $p_i:G \ra G_i$ be the canonical morphism. According to \ref{thm:proj_limits_internal} the set $\ca{N} \dopgleich \lbrace \ker(p_i) \mid i \in I \rbrace$ is a filter basis on $G$ consisting of closed normal subgroups and having $1$ as limit point. Since $\ca{P}$ is universally strict surjective, $p_i$ is strict surjective and thus induces an isomorphism of topological groups $G/\ker(p_i) \cong G_i$. As $G_i \in \ca{C}$ and $\ca{C}$ satisfies C0, this shows that $G/\ker(p_i) \in \ca{C}$ and so $G \in \se{proto}\tn{-}\ca{C}$. Hence, $\se{usspro}\tn{-}\ca{C} \subs \se{proto}\tn{-}\ca{C}$ if $\ca{C}$ satisfies C0. Moreover, if $\ca{C}$ consists of complete and separated groups, then $G = \ro{lim} \ \ca{P}$ is also complete and separated by \ref{prop:tgrp_complete}. Hence, $\se{usspro}\tn{-}\ca{C} \subs \se{cplproto}\tn{-}\ca{C}$ if $\ca{C}$ satisfies C0 and consists of complete and separated groups.
\end{asparaenum} \vspace{-\baselineskip}
\end{proof}

\subsection{Complete proto-$\ca{C}$ groups for classes of compact groups} \label{sec:proto_for_compact}

\begin{prop} \label{prop:pro_cats_for_compact}
Let $\ca{C}$ be a C0-class of compact groups. There exist the following inclusions and equalities of full subcategories of $\se{TGrp}^{\ro{s}}$, where the dashed arrow with label C1 symbolizes that the converse inclusion holds if $\ca{C}$ satisfies C1:
\[
\xymatrix 
@C=20pt
@R=25pt
{
\se{pro}\tn{-}\ca{C} \ar@/^1pc/@{-->}[d]^{\tn{C}1} \\
\se{cplproto}\tn{-}\ca{C} = \se{usspro}\tn{-}\ca{C} = \se{sspro}\tn{-}\ca{C} = \se{uspro}\tn{-}\ca{C} = \se{spro}\tn{-}\ca{C} \ar@{-}[u] \\
\ca{C} \ar@{-}[u]
}
\]

\end{prop}

\begin{proof}
Since compact groups are complete, the equality $\se{cplproto}\tn{-}\ca{C} = \se{usspro}\tn{-}\ca{C}$ follows from \ref{prop:pro_cats_inclusions}. The inclusion $\se{usspro}\tn{-}\ca{C} \subs \se{sspro}\tn{-}\ca{C}$ is also part of \ref{prop:pro_cats_inclusions}. To prove the converse inclusion, let $G \in \se{sspro}\tn{-}\ca{C}$. We can assume without loss of generality that $G = \ro{lim} \ \ca{P}$, where $\ca{P}$ is a strict surjective projective system in $\se{TGrp}$ whose objects are contained in $\ca{C}$. The projective system $\ca{P}$ is in particular surjective and as $\ca{C}$ consists of compact groups, the morphisms in $\ca{P}$ have compact kernel. Hence, it follows from \ref{prop:proj_systems_surjectivity} that $\ca{P}$ is universally surjective. Now, $\ca{P}$ is a strict surjective and universally surjective projective system and so an application of \ref{prop:proj_systems_surjectivity} shows that $\ca{P}$ is universally strict surjective. Hence, $G \in \se{usspro}\tn{-}\ca{C}$.

The inclusion $\se{uspro}\tn{-}\ca{C} \subs \se{spro}\tn{-}\ca{C}$ is part of \ref{prop:pro_cats_inclusions}. To prove the converse inclusion, let $G \in \se{spro}\tn{-}\ca{C}$. We can assume without loss of generality that $G = \ro{lim} \ \ca{P}$, where $\ca{P}$ is a surjective projective system in $\se{TGrp}$ whose objects are contained in $\ca{C}$. Since $\ca{C}$ consists of compact groups, all morphisms in $\ca{P}$ have compact kernel and so it follows from \ref{prop:proj_systems_surjectivity} that $\ca{P}$ is universally surjective. Hence, $G \in \se{uspro}\tn{-}\ca{C}$.

So far we have proven that $\se{cplproto}\tn{-}\ca{C} = \se{usspro}\tn{-}\ca{C} = \se{sspro}\tn{-}\ca{C}$ and $\se{uspro}\tn{-}\ca{C} = \se{spro}\tn{-}\ca{C}$. The inclusion $\se{sspro}\tn{-}\ca{C} \subs \se{spro}\tn{-}\ca{C}$ is obvious and the converse inclusion follows from the fact that $\ca{C}$ consists of compact groups and the fact that morphisms of compact groups are strict by \ref{prop:mor_from_qc_in_sep_strict}.

Finally, let $G \in \se{pro}\tn{-}\ca{C}$ and assume that $\ca{C}$ also satisfies C1. We can assume without loss of generality that $G = \ro{lim} \ \ca{P}$, where $\ca{P} = \lbrace \varphi_{ij}: G_j \ra G_i \mid (i,j) \in I \times I, i \leq j \rbrace$ is a projective system in $\se{TGrp}$ such that $G_i \in \ca{C}$ for all $i \in I$. Let $p_i:G \ra G_i$ be the canonical morphism. Let $\ca{P}' = \lbrace \varphi_{ij}': G_j' \ra G_i' \mid (i,j) \in I \times I, i \leq j \rbrace$ be the modification of $\ca{P}$ as described in \ref{thm:proj_limits_internal}\ref{item:proj_system_surjective_mod}, that is, $G_i' \dopgleich \ro{cl}(p_i(G)) \leq_{\ro{c}} G_i$ and $\varphi_{ij}':G_j' \ra G_i'$ is induced by $\varphi_{ij}$. Then $G = \ro{lim} \ \ca{P} \cong \ro{lim} \ \ca{P}'$ and the canonical morphism $p_i': \ro{lim} \ \ca{P} \ra G_i'$ has dense image for each $i \in I$ by \ref{thm:proj_limits_internal}. As $G$ is the limit of a projective system of compact groups, $G$ is also compact by \ref{prop:tgrp_complete}. Hence, $\ro{lim} \ \ca{P}'$ is compact and so the canonical morphism $p_i':\ro{lim} \ \ca{P}' \ra G_i' \leq_{\ro{c}} G_i$ is closed by the closed map lemma. This implies that $G_i' = \ro{cl}( \im(p_i')) = \im(p_i')$ and therefore $\ca{P}'$ is a universally surjective projective system. Moreover, as $\ca{C}$ satisfies C1 and $G_i' \leq_{\ro{c}} G_i$, the objects of $\ca{P}'$ are contained in $\ca{C}$. This shows that $G$ is isomorphic to the limit of a universally projective system whose objects are contained in $\ca{C}$, that is, $G \in \se{uspro}\tn{-}\ca{C}$.
\end{proof}

\begin{defn}
For a class $\ca{C}$ of compact groups we introduce the following additional property:
\begin{compactenum}
\item[(C4)] $1 \in \ca{C}$ and if $G$ is a compact group and $N_1,N_2 \lhd_{\ro{c}} G$ such that $G/N_1, G/N_2 \in \ca{C}$, then $G/N_1 \cap N_2 \in \ca{C}$.
\end{compactenum}
\end{defn}

\begin{prop} \label{prop:class_prop_implications}
For a class $\ca{C}$ of compact groups there exist the following implications:
\begin{compactenum}[(i)]
\item $\ro{C}0 \wedge \ro{C}1 \wedge \ro{C}3 \lRA \ro{C}4$.
\item $\ro{C}0 \wedge \ro{C}4 \lRA \ro{C}3$.
\item $\ro{C}0 \wedge \ro{C}4 \lRA $ for any compact group $G$ the set $\ca{E}_{\ca{C}}^{\ro{t}}(G) \dopgleich \lbrace N \mid N \lhd_{\ro{c}} G \wedge G/N \in \ca{C} \rbrace$ is a filter basis on $G$.
\end{compactenum} 
\end{prop}

\begin{proof} \hfill

\begin{asparaenum}[(i)]
\item Since $\ca{C}$ satisfies C0, there exists $G \in \ca{C}$. As $G$ is separated, $1$ is a closed subgroup of $G$ and so it follows from C1 that $1 \in \ca{C}$. Now, let $G$ be a compact group and let $N_1,N_2 \lhd_{\ro{c}} G$ such that $G/N_1$, $G/N_2 \in \ca{C}$. The quotient morphism $q_i:G \ra G/N_i$ induces a continuous morphism $q_i':G/N_1 \cap N_2 \ra G/N_i$ and so we get a continuous morphism $\iota:G/N_1 \cap N_2 \ra G/N_1 \times G/N_2$. As $G$ is compact, the quotient $G/N_1 \cap N_2$ is also compact and as $G/N_i$ is separated, the product $G/N_1 \times G/N_2$ is also separated. Hence, $\iota$ is a continuous map from a compact space to a separated space and therefore it is a closed map. Moreover, $\iota$ is also injective and so $\iota$ induces an isomorphism of topological groups $G/N_1 \cap N_2 \cong \im(\iota)$. Now, $G/N_1 \times G/N_2 \in \ca{C}$ by C3 and since $\iota$ is closed, $\im(\iota)$ is a closed subgroup of $G/N_1 \times G/N_2$ and so we get $\im(\iota) \in \ca{C}$ by C1. Finally, we get $G/N_1 \cap N_2 \in \ca{C}$ by C0.

\item Let $G_1, G_2 \in \ca{C}$. Let $G \dopgleich G_1 \times G_2$ and let $p_i:G \ra G_i$ be the projection for $i \in \lbrace 1, 2 \rbrace$. As $p_i$ is continuous and $G_i$ is separated, $N_i \dopgleich \ker(p_i)$ is a closed normal subgroup of $G$. Moreover, as $p_i$ is surjective and as $G$ and $G_i$ are compact, $p_i$ is strict and so the induced morphism $G/N_i \ra G_i$ is an isomorphism of topological groups. This implies that $G/N_1, G/N_2 \in \ca{C}$ because $\ca{C}$ satisfies C0 and as $\ca{C}$ satisfies C4, we conclude that $G/N_1 \cap N_2 \in \ca{C}$. Since $N_1 \cap N_2 = \lbrace 1 \rbrace \times \lbrace 1 \rbrace$, we have $G/N_1 \cap N_2 \cong G$ and therefore $G = G_1 \times G_2 \in \ca{C}$. It now follows by induction that $\ca{C}$ is closed under finite direct products, that is, $\ca{C}$ satisfies C3.

\item By C4, we have $1 \in \ca{C}$ so that $G \in \ca{E}_{\ca{C}}^{\ro{t}}(G)$ and therefore $\ca{E}_{\ca{C}}^{\ro{t}}(G) \neq \emptyset$. Moreover, we obviously have $\emptyset \notin \ca{E}_{\ca{C}}^{\ro{t}}(G)$. The fact that $\ca{C}$ satisfies C4 also implies that $\ca{E}_{\ca{C}}^{\ro{t}}(G)$ is closed under finite intersections. Hence, $\ca{E}_{\ca{C}}^{\ro{t}}(G)$ is a filter basis on $G$.
\end{asparaenum} \vspace{-\baselineskip}
\end{proof}

\begin{defn} \label{para:def_of_formation_variety}
A class $\ca{C}$ of compact groups is called a \word{formation} if it satisfies $\ro{C}0, \ro{C}2$ and $\ro{C}4$. It is called a \word{variety} if it is a formation and satisfies $\ro{C}1$.
\end{defn}

\begin{prop} \label{prop:char_of_pro_compact}
Let $\ca{C}$ be a C0-class of compact groups. The following are equivalent for a topological group $G$:
\begin{enumerate}[label=(\roman*),noitemsep,nolistsep]
\item \label{item:char_of_pro_compact_1} $G \in \se{cplproto}\tn{-}\ca{C} = \star\se{pro}\tn{-}\ca{C}$ for $\star \in \lbrace \se{uss},\se{ss},\se{us},\se{s} \rbrace$.
\item \label{item:char_of_pro_compact_2} $G$ is compact and there exists a filter basis $\ca{N}$ on $G$ consisting of closed normal subgroups of $G$ such that $\bigcap_{N \in \ca{N}} N = 1$ and $G/N \in \ca{C}$ for all $N \in \ca{N}$.
\end{enumerate}
If $\ca{C}$ satisfies C1, then \ref{item:char_of_pro_compact_1} is also equivalent to:
\begin{enumerate}[resume,label=(\roman*),noitemsep,nolistsep]
\item \label{item:char_of_pro_compact_3} $G \in \se{pro}\tn{-}\ca{C}$.
\end{enumerate}
If $\ca{C}$ satisfies C4, then \ref{item:char_of_pro_compact_1} is also equivalent to:
\begin{enumerate}[resume,label=(\roman*),noitemsep,nolistsep]
\item \label{item:char_of_pro_compact_4} $G$ is compact and $\bigcap_{ N \in \ca{E}_{\ca{C}}^{\ro{t}}(G)} N = 1$.
\end{enumerate}
If $\ca{C}$ satisfies C1 and C4, then \ref{item:char_of_pro_compact_1} is also equivalent to:
\begin{enumerate}[resume,label=(\roman*),noitemsep,nolistsep]
\item \label{item:char_of_pro_compact_5} $G$ is isomorphic to a closed subgroup of a direct product of groups contained in $\ca{C}$.
\end{enumerate}
\end{prop}

\begin{proof} \hfill

\ref{item:char_of_pro_compact_1} $\RA$ \ref{item:char_of_pro_compact_2}: Let $G \in \se{cplproto}\tn{-}\ca{C}$. Since $\se{cplproto}\tn{-}\ca{C} \subs \se{pro}\tn{-}\ca{C}$ by \ref{prop:pro_cats_inclusions}, there exists an isomorphism $G \cong \ro{lim} \ \ca{P}$, where $\ca{P}$ is a projective system whose objects are contained in $\ca{C}$. As $\ca{C}$ consists of compact groups, \ref{prop:tgrp_complete} implies that $G$ is compact. As $G$ is a proto-$\ca{C}$ group, there exists a filter basis $\ca{N}$ on $G$ consisting of closed normal subgroups of $G$ such that $\ro{lim} \ \ca{N} = 1$ and $G/N \in \ca{C}$ for each $N \in \ca{N}$. As $G$ is compact and all $N \in \ca{N}$ are closed in $G$, the filter basis $\ca{N}$ consists of compact sets and so it follows from \ref{prop:limit_compact_member_intersect} that $\bigcap_{N \in \ca{N}} N = 1$. 

\ref{item:char_of_pro_compact_2} $\RA$ \ref{item:char_of_pro_compact_1}: It follows from \ref{prop:limit_compact_member_intersect} that $\ro{lim} \ \ca{N} = 1$. Hence, $G$ is a compact proto-$\ca{C}$ group so that in particular $G \in \se{cplproto}\tn{-}\ca{C}$.

\ref{item:char_of_pro_compact_1} $\LRA$ \ref{item:char_of_pro_compact_3}: This follows from the equality $\se{cplproto}\tn{-}\ca{C} = \se{pro}\tn{-}\ca{C}$ which is part of \ref{prop:pro_cats_for_compact}.

\ref{item:char_of_pro_compact_2} $\RA$ \ref{item:char_of_pro_compact_4}: This always holds due to $1 = \bigcap_{N \in \ca{N}} N \sups \bigcap_{ \substack{N \lhd_{\ro{c}} G \\ G/N \in \ca{C}} } N$.

\ref{item:char_of_pro_compact_4} $\RA$ \ref{item:char_of_pro_compact_2}: This holds because $\ca{E}_{\ca{C}}^{\ro{t}}(G)$ is a filter basis on $G$ by \ref{prop:class_prop_implications}.

\ref{item:char_of_pro_compact_1} $\RA$ \ref{item:char_of_pro_compact_5}: This always holds due to \ref{prop:tgrp_proj_limit_explicit}.

\ref{item:char_of_pro_compact_5} $\RA$ \ref{item:char_of_pro_compact_1}: By assumption there exists an isomorphism $G \ra G'$ onto a closed subgroup $G' \leq_{\ro{c}} X \dopgleich \prod_{i \in I} G_i$, where $G_i \in \ca{C}$ for each $i \in I$. Let $p_i:X \ra G_i$ be the projection. As all $G_i$ are compact, $X$ is also compact. In particular, $K_i \dopgleich \ker(p_i)$ is a closed normal subgroup of $X$ and so $N_i \dopgleich G' \cap K_i$ is a closed normal subgroup of $G'$. As $\bigcap_{i \in I} K_i = 1$, we also have $\bigcap_{i \in I} N_i = 1$. The projection $p_i:X \ra G_i$ is a strict surjective morphism by \ref{prop:mor_from_qc_in_sep_strict} because $X$ is compact and therefore $p_i$ induces an isomorphism $X/K_i \cong G_i \in \ca{C}$. Moreover, the composition of the inclusion $G' \rightarrowtail X$ composed with the quotient morphism $X \ra X/K_i$ gives a morphism $G' \ra X/K_i$ with kernel equal to $G' \cap K_i = N_i$. As $G'/N_i$ is compact, the induced morphism $G'/N_i \rightarrowtail X/K_i$ is strict injective by \ref{prop:mor_from_qc_in_sep_strict} and therefore $G'/N_i$ is isomorphic to a closed subgroup of $X/K_i$. The assumption that $\ca{C}$ satisfies C0 and C1 thus implies that $G'/N_i \in \ca{C}$. Hence, $1 = \bigcap_{i \in I} N_i \sups \bigcap_{ \substack{N \lhd_{\ro{c}} G' \\ G'/N \in \ca{C}} } N$ and therefore $G'$ satisfies (iii). Now, as $\ca{C}$ also satisfies C4, the direction (iii) $\RA$ (i) holds so that $G' \in \se{cplproto}\tn{-}\ca{C}$ and as $G' \cong G$ we can finally conclude that $G \in \se{cplproto}\tn{-}\ca{C}$.
\end{proof}

\begin{prop} \label{prop:stability_of_cplproto_c_for_compact}
Let $\ca{C}$ be a C0-class of compact groups. The following holds:
\begin{compactenum}[(i)]
\item If $\ca{C}$ satisfies C1 (C2), then $\se{cplproto}\tn{-}\ca{C}$ also satisfies C1 (C2).
\item If $\ca{C}$ satisfies C1 and C4, then the full subcategory $\se{cplproto}\tn{-}\ca{C}$ of $\se{TGrp}$ is closed under the formation of limits. In particular, it is a complete category.
\end{compactenum}
\end{prop}

\begin{proof} \hfill

\begin{asparaenum}[(i)]
\item Let $G \in \se{cplproto}\tn{-}\ca{C}$ and let $\ca{N}$ be a filter basis on $G$ consisting of closed normal subgroups such that $\ro{lim} \ \ca{N} = 1$ and $G/N \in \ca{C}$ for all $N \in \ca{N}$. Suppose that $\ca{C}$ satisfies C1 and let $H$ be a closed subgroup of $G$. It follows from \ref{prop:closed_subgrp_proj_lim} that $\gamma_{H \cap \ca{N}}: H \ra \ro{lim} \ H/H \cap \ca{N} = H_{H \cap \ca{N}}$ is an isomorphism. For each $N \in \ca{N}$, the composition of the inclusion $H \rightarrowtail G$ with the quotient morphism $G \ra G/N$ gives a morphism $H \ra G/N$ with kernel equal to $H \cap N$. The induced morphism $H/H \cap N \rightarrowtail G/N$ is strict injective by \ref{prop:mor_from_qc_in_sep_strict} and therefore $H/H \cap N$ is isomorphic to a closed subgroup of $G/N$. Since $\ca{C}$ satisfies C0 and C1, it follows that $H/H \cap N \in \ca{C}$ and this proves that $H \cong H_{H \cap \ca{N}} \in \se{usspro}\tn{-}\ca{C} = \se{cplproto}\tn{-}\ca{C}$.

Now, assume that $\ca{C}$ satisfies C2 and let $H$ be a closed normal subgroup of $G$. Let $q:G \ra G/H$ be the quotient morphism. It follows from \ref{prop:quot_grp_proj_lim} that $\gamma_{q(\ca{N})}: G/H \ra \ro{lim} \ (G/H)/q(\ca{N}) = (G/H)_{q(\ca{N})}$ is an isomorphism. Let $N \in \ca{N}$. The quotient morphism $q_N:G \ra G/N$ is closed by the closed map lemma and therefore $q_N(H)$ is a closed normal subgroup of $G/N$. As $G/N \in \ca{C}$ and as $\ca{C}$ satisfies C2, it follows that $(G/N)/q_N(H) \in \ca{C}$. By \ref{prop:prop_of_quots} we have the following isomorphisms 
\[
(G/H)/q(N) \cong G/HN \cong (G/N)/q_N(H)
\]
and as $\ca{C}$ satisfies C0, we conclude that $(G/H)/q(N) \in \ca{C}$. Hence, $G/H \cong (G/H)_{q(\ca{N})} \in \se{usspro}\tn{-}\ca{C} = \se{cplproto}\tn{-}\ca{C}$.

\item According to \ref{prop:tgrp_complete} it is enough to show that $\se{cplproto}\tn{-}\ca{C}$ is closed under the formation of products and passing to closed subgroups. As $\se{cplproto}\tn{-}\ca{C}$ satisfies C1 by the above, it remains to show that $\se{cplproto}\tn{-}\ca{C}$ is closed under the formation of products. So, let $\lbrace G_i \mid i \in I \rbrace$ be a family in $\se{cplproto}\tn{-}\ca{C}$. Since $\ca{C}$ satisfies C1 and C4, it follows from \ref{prop:char_of_pro_compact} that for each $i \in I$ there exists a family $\lbrace G_{ij} \mid j \in J_i \rbrace$ in $\ca{C}$ and an isomorphism $\varphi_i:G_i \ra G_i'$ onto a closed subgroup $G_i' \leq_{\ro{c}} \prod_{j \in J_i} G_{ij}$. These isomorphisms induce an isomorphism 
\[
\xymatrix{
\prod_{i \in I} G_i \ar[rr]^{\varphi \dopgleich \prod_{i \in i} \varphi_i} & & \prod_{i \in I} G_i'.
}
\]
Since $G_i'$ is closed in $\prod_{j \in J_i} G_{ij}$, it follows that $\prod_{i \in I} G_i'$ is closed in $\prod_{i \in I} \prod_{j \in J_i} G_{ij}$.\footnote{In general, if $\lbrace X_i \mid i \in I \rbrace$ is a family of topological spaces and $Y_i$ is a closed subspace of $X_i$ for each $i \in I$, then $\prod_{i \in I} Y_i$ is a closed subspace of $X \dopgleich \prod_{i \in I} X_i$ because $\prod_{i \in I} Y_i = \bigcap_{i \in I} p_i^{-1}(Y_i)$, where $p_i:X \ra X_i$ is the projection.} Hence, $\varphi$ is an isomorphism from $\prod_{i \in I} G_i$ onto a closed subgroup of a direct product of groups in $\ca{C}$ and as $\ca{C}$ satisfies C1 and C4, this implies that $\prod_{i \in I} G_i \in \se{cplproto}\tn{-}\ca{C}$.
\end{asparaenum} \vspace{-\baselineskip}
\end{proof}

\begin{prop} \label{prop:maximal_cplproto_quotient}
 \wordsym{$\ro{R}_{\ca{C}}(G)$} \wordsym{$\pi_{\ca{C}}(G)$}
Let $G$ be a compact group and let $\ca{C}$ be a formation of compact groups. The following holds:
\begin{enumerate}[label=(\roman*),noitemsep,nolistsep]
\item Let $\ca{E}^{\ro{t}}(G) \dopgleich \lbrace N \mid N \lhd_{\ro{c}} G \rbrace$. Then $\ca{E}_{\ca{C}}^{\ro{t}}(G) = \lbrace N \mid N \lhd_{\ro{c}} G \wedge G/N \in \ca{C} \rbrace$ is a filter on $(\ca{E}^{\ro{t}}(G), \subs)$. Moreover, $\ca{E}_{\ca{C}}^{\ro{t}}(G)$ it is a lattice with respect to intersections and products.

\item $\ro{R}_{\ca{C}}(G) \dopgleich \bigcap_{N \in \ca{E}_{\ca{C}}^{\ro{t}}(G)} N$ is a closed normal subgroup of $G$. It is called the \word{$\ca{C}$-radical} \invword{radical} of $G$.

\item \label{item:maximal_cplproto_surj} If $\varphi:G \ra G'$ is a surjective morphism of compact groups, then $\varphi(\ca{E}_{\ca{C}}^{\ro{t}}(G)) = \ca{E}_{\ca{C}}^{\ro{t}}(G')$ and $\varphi(\ro{R}_{\ca{C}}(G)) = \ro{R}_{\ca{C}}(G')$.

\item $G$ is a complete proto-$\ca{C}$ group if and only if $\ro{R}_{\ca{C}}(G) = 1$.

\item The group $\pi_{\ca{C}}(G) \dopgleich G/\ro{R}_{\ca{C}}(G)$ is the \word{maximal complete proto-$\ca{C}$ quotient} of $G$, that is, $\pi_{\ca{C}}(G)$ is a complete proto-$\ca{C}$ group and if $H$ is a closed normal subgroup of $G$, then $G/H$ is a complete proto-$\ca{C}$ group if and only if $H \geq \ro{R}_{\ca{C}}(G)$.

\end{enumerate}
\end{prop}

\begin{proof} \hfill

\begin{asparaenum}[(i)]

\item By \ref{prop:class_prop_implications}, the set $\ca{E}_{\ca{C}}^{\ro{t}}(G)$ is a filter basis on $G$. So, to prove that it is a filter on $(\ca{E}^{\ro{t}}(G),\subs)$, it remains to show that for $N \in \ca{E}_{\ca{C}}^{\ro{t}}(G)$ and $M \in \ca{E}^{\ro{t}}(G)$ such that $N \subs M$ we also have $M \in \ca{E}_{\ca{C}}^{\ro{t}}(G)$. The quotient morphism $q:G \ra G/N$ is closed by the closed map lemma and consequently $q(M) \lhd_{\ro{c}} G/N$. As $\ca{C}$ satisfies C2 and as $G/N \in \ca{C}$, we conclude that $(G/N)/q(M) \in \ca{C}$. By \ref{prop:prop_of_quots} we have an isomorphism
\[
G/M \cong (G/N)/(M/N) = (G/N)/q(M) \in \ca{C}
\]
and therefore $G/M \in \ca{C}$, that is, $M \in \ca{E}_{\ca{C}}^{\ro{t}}(G)$. Hence, $\ca{E}_{\ca{C}}^{\ro{t}}(G)$ is a filter on $(\ca{E}^{\ro{t}}(G),\subs)$. 

To prove the second assertion, it is enough to show that $\ca{E}_{\ca{C}}^{\ro{t}}(G)$ is closed under finite intersections and products because the set of all normal subgroups of $G$ is already a lattice with respect to these operations. By the above, $\ca{E}_{\ca{C}}^{\ro{t}}(G)$ is closed under finite intersections. If $N_1,N_2 \in \ca{E}_{\ca{C}}^{\ro{t}}(G)$, then $N_1 \cdot N_2$ is a closed normal subgroup of $G$ by \ref{prop:compact_product_closed}. Since $N_1 \cdot N_2 \geq N_1$ and since $\ca{E}_{\ca{C}}^{\ro{t}}(G)$ is a filter on $(\ca{E}^{\ro{t}}(G),\subs)$, it follows that $N_1 \cdot N_2 \in \ca{E}_{\ca{C}}^{\ro{t}}(G)$.

\item This is obvious.

\item Let $N \in \ca{E}_{\ca{C}}^{\ro{t}}(G)$. Then $\varphi(N)$ is a closed normal subgroup of $G'$ as $\varphi$ is closed by the closed map lemma. The composition of $\varphi$ with the quotient morphism $G' \ra G'/\varphi(N)$ is a strict surjective morphism with kernel $N \cdot \ker(\varphi)$. Hence, $G/ (N \cdot \ker(\varphi)) \cong G'/\varphi(N)$. As $N \cdot \ker(\varphi)$ is closed, moreover $N \in \ca{E}_{\ca{C}}^{\ro{t}}(G)$ and $N \cdot \ker(\varphi) \geq N$, it follows that $N \cdot \ker(\varphi) \in \ca{E}_{\ca{C}}^{\ro{t}}(G)$. Consequently, $G/(N \cdot \ker(\varphi)) \in \ca{C}$ which implies $G'/\varphi(N) \in \ca{C}$ and therefore $\varphi(N) \in \ca{E}_{\ca{C}}^{\ro{t}}(G')$. This shows that $\varphi(\ca{E}_{\ca{C}}^{\ro{t}}(G)) \subs \ca{E}_{\ca{C}}^{\ro{t}}(G')$. Conversely, let $N \in \ca{E}_{\ca{C}}^{\ro{t}}(G')$. Then $\varphi^{-1}(N)$ is a closed normal subgroup of $G$ and the composition of $\varphi$ with the quotient morphism $G' \ra G'/N$ is a strict surjective morphism having kernel $\varphi^{-1}(N) \cdot \ker(\varphi)$. Hence, $G/(\varphi^{-1}(N) \cdot \ker(\varphi)) \cong G'/N \in \ca{C}$ and so we conclude that $\varphi^{-1}(N) \cdot \ker(\varphi) \in \ca{E}_{\ca{C}}^{\ro{t}}(G)$. Now, $\varphi( \varphi^{-1}(N) \cdot \ker(\varphi)) = N$ and this shows that $\ca{E}_{\ca{C}}^{\ro{t}}(G') \subs \varphi(\ca{E}_{\ca{C}}^{\ro{t}}(G))$.

As $\ca{E}_{\ca{C}}^{\ro{t}}(G)$ is a filter basis on $G$, we can apply \ref{prop:qc_intersection_props} to get
\[
\varphi( \ro{R}_{\ca{C}}(G) ) = \varphi( \bigcap_{N \in \ca{E}_{\ca{C}}^{\ro{t}}(G)} N ) = \bigcap_{N \in \ca{E}_{\ca{C}}^{\ro{t}}(G)} \varphi(N) = \bigcap_{N \in \ca{E}_{\ca{C}}^{\ro{t}}(G')} N = \ro{R}_{\ca{C}}(G').
\]

\item This follows immediately from \ref{prop:char_of_pro_compact}.

\item First, we prove that $G/\ro{R}_{\ca{C}}(G)$ is a complete proto-$\ca{C}$ group. This group is obviously compact as $G$ is compact and $\ro{R}_{\ca{C}}(G) \lhd_{\ro{c}} G$. Let $q:G \ra G/\ro{R}_{\ca{C}}(G)$ be the quotient morphism. By the above, we have
\[
\bigcap_{N \in \ca{E}_{\ca{C}}^{\ro{t}}(G/\ro{R}_{\ca{C}}(G))} N = \ro{R}_{\ca{C}}(G/ \ro{R}_{\ca{C}}(G))= q( \ro{R}_{\ca{C}}(G)) = 1.
\]
As $\ca{C}$ satisfies C4, it follows from \ref{prop:char_of_pro_compact} that $G/\ro{R}_{\ca{C}}(G)$ is a complete proto-$\ca{C}$ group.

Now, let $H$ be a closed normal subgroup of $G$ such that $G/H$ is a complete proto-$\ca{C}$ group. By \ref{prop:char_of_pro_compact} we have 
\[
\ro{R}_{\ca{C}}(G/H) = \bigcap_{N \in \ca{E}_{\ca{C}}^{\ro{t}}(G/H)} N = 1.
\]
Let $q:G \ra G/H$ be the quotient morphism. By the above, we have $1 = \ro{R}_{\ca{C}}(G/H) = q(\ro{R}_{\ca{C}}(G))$ and therefore $\ro{R}_{\ca{C}}(G) \leq H$. On the other hand, let $H$ be a closed normal subgroup of $G$ such that $\ro{R}_{\ca{C}}(G) \leq H$. Let $q:G \ra G/H$ be the quotient morphism. The group $G/H$ is compact and we have $\ro{R}_{\ca{C}}(G/H) = q(\ro{R}_{\ca{C}}(G)) = 1$ and so it follows from \ref{prop:char_of_pro_compact} that $G/H$ is a complete proto-$\ca{C}$ group.
\end{asparaenum} \vspace{-\baselineskip}
\end{proof}

\begin{prop} \label{para:maximal_proto_c_for_variety}
Let $\ca{C}$ be a variety of compact groups. The following holds:
\begin{enumerate}[label=(\roman*),noitemsep,nolistsep]
\item \label{item:maximal_proto_c_mor} If $\varphi:G \ra G'$ is a morphism of compact groups, then $\varphi(\ro{R}_{\ca{C}}(G)) \leq \ro{R}_{\ca{C}}(G')$.

\item If $G$ is a compact group, then the quotient morphism $q:G \ra G/\ro{R}_{\ca{C}}(G)$ is universal among morphisms into complete proto-$\ca{C}$ groups, that is, if $\varphi:G \ra G'$ is a morphism into a complete proto-$\ca{C}$ group $G'$, then there exists a unique morphism $\varphi':G/\ro{R}_{\ca{C}}(G) \ra G'$ making the diagram
\[
\xymatrix{
G \ar[r]^\varphi \ar[d]_q & G' \\
G/\ro{R}_{\ca{C}}(G) \ar[ur]_{\varphi'}
}
\]
commutative.

\item If $\varphi:G \ra G'$ is a morphism of compact groups, then there exists a unique morphism $\pi_{\ca{C}}(\varphi): \pi_{\ca{C}}(G) \ra \pi_{\ca{C}}(G')$ making the following diagram commutative
\[
\xymatrix{
G \ar[rr]^\varphi \ar[d] & & G' \ar[d] \\
\pi_{\ca{C}}(G) \ar[rr]_{\pi_{\ca{C}}(\varphi)} & & \pi_{\ca{C}}(G')
}
\]

where the vertical morphisms are the quotient morphisms.

\item The maps
\[
\begin{array}{rcl}
\se{TGrp}^{\ro{com}} & \lra & \se{cplproto}\tn{-}\ca{C} \\
G & \longmapsto & \pi_{\ca{C}}(G) \\
G \overset{\varphi}{\lra} G' & \longmapsto & \pi_{\ca{C}}(G) \overset{\pi_{\ca{C}}(\varphi)}{\lra} \pi_{\ca{C}}(G')
\end{array}
\]
define a functor. 
\end{enumerate}
\end{prop}

\begin{proof} \hfill

\begin{asparaenum}[(i)]
\item The map $\varphi$ is closed according to the closed map lemma and therefore $B \dopgleich \varphi(G)$ is a closed subgroup of $G'$. As $G'$ is compact, it follows that $B$ is also compact. The morphism $\psi$ obtained by composing the canonical inclusion $B \rightarrowtail G'$ with the quotient morphism $G' \ra G'/\ro{R}_{\ca{C}}(G')$ has kernel $B \cap \ro{R}_{\ca{C}}(G') \lhd_{\ro{c}} B$. The induced morphism $\psi':B/(B \cap \ro{R}_{\ca{C}}(G')) \ra G'/\ro{R}_{\ca{C}}(G')$ is closed and strict injective so that $B/(B \cap \ro{R}_{\ca{C}}(G'))$ is isomorphic to a closed subgroup of $G'/\ro{R}_{\ca{C}}(G')$. Since $\ca{C}$ satisfies C1, it follows from \ref{prop:stability_of_cplproto_c_for_compact} that $\se{cplproto}\tn{-}\ca{C}$ also satisfies C1 and so we conclude that $B/(B \cap \ro{R}_{\ca{C}}(G'))$ is a complete proto-$\ca{C}$ group. Now, $B$ is a compact group and $B \cap \ro{R}_{\ca{C}}(G')$ is a closed normal subgroup of $B$ such that $B/(B \cap \ro{R}_{\ca{C}}(G'))$ is a complete proto-$\ca{C}$ group. Thus, \ref{prop:maximal_cplproto_quotient} implies that $B \cap \ro{R}_{\ca{C}}(G') \geq \ro{R}_{\ca{C}}(B)$. The corestriction $\varphi':G \ra B = \im(\varphi)$ is a surjective morphism of compact groups and so it follows from \ref{prop:maximal_cplproto_quotient} that $\varphi'(\ro{R}_{\ca{C}}(G)) = \ro{R}_{\ca{C}}(B)$. Altogether, we have
\[
\ro{R}_{\ca{C}}(G') \geq B \cap \ro{R}_{\ca{C}}(G') \geq \ro{R}_{\ca{C}}(B) = \varphi'(\ro{R}_{\ca{C}}(G)) = \varphi(\ro{R}_{\ca{C}}(G)).
\]

\item Since $G'$ is a complete proto-$\ca{C}$ group, we have $\ro{R}_{\ca{C}}(G') = 1$ and by the above we have $\varphi(\ro{R}_{\ca{C}}(G)) \subs \ro{R}_{\ca{C}}(G') = 1$, that is, $\ro{R}_{\ca{C}}(G) \subs \ker(\varphi)$. Hence, it follows from \ref{para:morphism_quotient_induced} that $\varphi$ induces a morphism of topological groups $\varphi':G/\ro{R}_{\ca{C}}(G) \ra G'$ such that $\varphi' \circ q = \varphi$. This shows existence of such a morphism. The uniqueness is evident. 

\item Let $q:G \ra \pi_{\ca{C}}(G)$ and $q':G' \ra \pi_{\ca{C}}(G')$ be the quotient morphisms. Then $q' \circ \varphi: G \ra \pi_{\ca{C}}(G')$ is a morphism of compact groups and so we have $q' \circ \varphi(\ro{R}_{\ca{C}}(G)) \leq \ro{R}_{\ca{C}}(\pi_{\ca{C}}(G')) = 1$ by the above. Hence, $\ro{R}_{\ca{C}}(G) \leq \ker(q' \circ \varphi)$ and therefore $q' \circ \varphi$ induces a morphism $\pi_{\ca{C}}(\varphi): \pi_{\ca{C}}(G) \ra \pi_{\ca{C}}(G')$ making the diagram 
\[
\xymatrix{
G \ar[r]^\varphi \ar[d]_q & G' \ar[r]^{q'} & \pi_{\ca{C}}(G') \\
\pi_{\ca{C}}(G) \ar[urr]_{\pi_{\ca{C}}(\varphi)}
}
\]
commutative. This shows existence of such a morphism. The uniqueness is evident.

\item This follows immediately from the uniqueness of $\pi_{\ca{C}}(\varphi)$.
\end{asparaenum} \vspace{-\baselineskip}
\end{proof}

\begin{prop} \label{para:pi_of_quotient}
Let $G$ be a compact group, let $H$ be a closed normal subgroup of $G$ and let $\ca{C}$ be a variety of compact groups. Then
\[
\ro{R}_{\ca{C}}(G/H) = H\ro{R}_{\ca{C}}(G)/H \lhd G/H
\]
and there exists a canonical isomorphism
\[
G/H\ro{R}_{\ca{C}}(G) \cong \pi_{\ca{C}}(G/H).
\]
\end{prop}

\begin{proof}
Let $q:G \ra G/H$ be the quotient morphism. An application of \ref{prop:maximal_cplproto_quotient}\ref{item:maximal_cplproto_surj} yields $H\ro{R}_{\ca{C}}(G)/H = q(\ro{R}_{\ca{C}}(G)) = \ro{R}_{\ca{C}}(G/H)$. Hence, using \ref{prop:prop_of_quots}\ref{item:top_grp_2_iso_theorem}, we have a canonical isomorphism
\[
\pi_{\ca{C}}(G/H) = (G/H)/\ro{R}_{\ca{C}}(G/H) = (G/H)/(H\ro{R}_{\ca{C}}(G)/H) \cong G/H\ro{R}_{\ca{C}}(G).
\]
\end{proof}

\begin{prop}
The following holds:
\begin{compactenum}[(i)]
\item The class $\se{Ab}^{\ro{com}}$ of compact abelian groups is a variety.
\item If $G$ is a compact group, then $\comm{t}(G) = \ro{R}_{\se{Ab}^{\ro{com}}}(G)$. 
\item The functors $(-)^{\ro{ab}}$ and $\pi_{\se{Ab}^{\ro{com}}}$ coincide on $\se{TGrp}^{\ro{com}}$.
\end{compactenum}
\end{prop}

\begin{proof} \hfill

\begin{asparaenum}[(i)]
\item It is evident that $\se{Ab}^{\ro{com}}$ satisfies $\ro{C}0 - \ro{C}3$ and therefore it also satisfies C4, that is, $\se{Ab}^{\ro{com}}$ is a variety.
\item The quotient $G/\comm{t}(G)$ is a compact abelian group and therefore it is in particular a complete proto-$\se{Ab}^{\ro{com}}$ group. Hence, it follows from \ref{prop:maximal_cplproto_quotient} that $\comm{t}(G) \geq \ro{R}_{\se{Ab}^{\ro{com}}}(G)$. Conversely, let $N$ be a closed normal subgroup of $G$ such that $G/N \in \se{Ab}^{\ro{com}}$. Then $G/N$ is in particular a separated abelian group and therefore $N \geq \comm{t}(G)$. Hence, 
\[
\ro{R}_{\se{Ab}^{\ro{com}}}(G) = \bigcap_{ \substack{ N \lhd_{\ro{c}} G \\ G/N \in \se{Ab}^{\ro{com}}}} N \geq \comm{t}(G).
\]
\item This is now obvious.
\end{asparaenum} \vspace{-\baselineskip}
\end{proof}

\subsection{Profinite groups}

\begin{conv}
A class $\ca{C}$ of finite groups is always considered as a subclass of $\se{TGrp}^{\ro{s}}$ with respect to the discrete topology.
\end{conv}

\begin{defn}
For a class $\ca{C}$ of finite groups we introduce the following additional properties that $\ca{C}$ might satisfy:
\begin{compactenum}[({C}1)]
\item[({C}4')] If $G$ is a finite group and $N_1,N_2 \lhd G$ with $G/N_1, G/N_2 \in \ca{C}$, then also $G/N_1 \cap N_2 \in \ca{C}$.
\item[({C}4'')] If $G$ is a group and $N_1,N_2 \lhd G$ with $G/N_1, G/N_2 \in \ca{C}$, then also $G/N_1 \cap N_2 \in \ca{C}$.
\end{compactenum}
\end{defn}

\begin{prop}
For a class $\ca{C}$ of finite groups there exist the following equivalences:
\begin{compactenum}[(i)]
\item $\ro{C}0 \wedge \ro{C}4' \lLRA \ro{C}0 \wedge \ro{C}4''$.
\item $\ro{C}0 \wedge \ro{C}2 \wedge \ro{C}4 \lLRA \ro{C}0 \wedge \ro{C}2 \wedge \ro{C}4'$.
\end{compactenum}
In particular, the definition of a formation (variety) of finite groups given in \cite{RibZal00_Profinite-Groups_0} coincides with the definition given in \ref{para:def_of_formation_variety}.
\end{prop}

\begin{proof} \hfill

\begin{asparaenum}[(i)]
\item Suppose that $\ro{C}0 \wedge \ro{C}4'$ holds and let $G,N_1,N_2$ as in $\ro{C}4''$. Let $G' \dopgleich G/(N_1 \cap N_2)$ and $N_i \dopgleich N_i/(N_1 \cap N_2) \lhd G'$. Since $G/N_i \in \ca{C}$ and since $\ca{C}$ is a class of finite groups, we have
\[
|G'| = \lbrack G:N_1 \cap N_2 \rbrack \leq \lbrack G:N_1 \rbrack \cdot \lbrack G:N_2 \rbrack < \infty.
\]
As $G'/N_i' \cong G/N_i \in \ca{C}$, it follows from $\ro{C}0$ that $G'/N_i' \in \ca{C}$ and now it follows from $\ro{C}4'$ and $\ro{C}0$ that
\[
\ca{C} \ni G'/N_1' \cap N_2' = G'/\lbrace 1 \rbrace \cong G' = G/N_1 \cap N_2 \lRA G/N_1 \cap N_2 \in \ca{C}.
\]
Hence, $\ro{C}4''$ is satisfied. The converse implication is obvious.

\item The implication $\lRA$ is obvious, so suppose that $\ro{C}0 \wedge \ro{C}2 \wedge \ro{C}4'$ holds. It follows from the above that $\ro{C}0 \wedge \ro{C}2 \wedge \ro{C}4''$ holds and then it remains to show that $1 \in \ca{C}$. Since $\ca{C}$ satisfies $\ro{C}0$, there exists $G \in \ca{C}$ and according to $\ro{C}2$ we have $1 \cong G/G \in \ca{C}$. 
\end{asparaenum} \vspace{-\baselineskip}
\end{proof}

\begin{prop}
The following holds:
\begin{compactenum}[(i)]
\item The class $\se{Grp}^{\ro{f}}$ of finite groups is a variety. A pro-$\se{Grp}^{\ro{f}}$ group is called a \word{profinite group}.
\item For a set of prime numbers $P$ the class $\ca{C}$ of finite $P$-groups is a variety. A pro-$\ca{C}$ group is called a \word{pro-$P$ group}.
\end{compactenum}
\end{prop}

\begin{proof}
This is obvious.
\end{proof}

\begin{prop} \label{para:tot_disc_fb}
For a locally compact group $G$ the following are equivalent:
\begin{enumerate}[label=(\roman*),noitemsep,nolistsep]
\item There exists a filter basis of $1 \in G$ consisting of open subgroups.
\item \label{item:tot_disc_fb_2} $G$ is totally disconnected.
\end{enumerate} 
If moreover $G$ is compact, then the above is also equivalent to: 
\begin{enumerate}[resume,label=(\roman*),noitemsep,nolistsep] 
\item \label{item:tot_disc_fb_3} There exists a filter basis of $1 \in G$ consisting of open normal subgroups.
\item \label{item:tot_disc_fb_4} $G$ is a profinite group.
\end{enumerate} 
\end{prop}

\begin{proof}
This is \cite[theorem 1.34]{HofMor06_The-structure-of-compact_0} (note that instead of \ref{item:tot_disc_fb_4} it is stated that $G \in \se{spro}\tn{-}\se{Grp}^{\ro{f}}$ but as $\se{Grp}^{\ro{f}}$ satisfies $\ro{C}1$ we know that $\se{spro}\tn{-}\se{Grp}^{\ro{f}} = \se{pro}\tn{-}\se{Grp}^{\ro{f}}$).
\end{proof}

\begin{prop} \label{para:compact_filter_basis}
The following holds:
\begin{enumerate}[label=(\roman*),noitemsep,nolistsep]
\item Let $G$ be a quasi-compact group and let $U$ be an open subset of $G$. If $\lbrace H_i \mid i \in I \rbrace$ is a family of closed subgroups of $G$ such that $\bigcap_{i \in I} H_i \subs U$, then there is a finite subset $J$ of $I$ such that $\bigcap_{j \in J} H_j \subs U$.
\item If $G$ is compact and if $\lbrace U_i \mid i \in I \rbrace$ is a filter basis on $G$ which consists of open subgroups of $G$ such that $\bigcap_{i \in I} U_i = 1$, then
\[
\ca{U} = \lbrace \bigcap_{j \in J} U_j \mid J \subs I \wedge \ro{Card}(J) < \infty \rbrace
\]
 is a filter basis of $1 \in G$.
\end{enumerate}
\end{prop}

\begin{proof} \hfill

\begin{asparaenum}[(i)]
\item Since $\bigcap_{i \in I} H_i \subs U$, we get 
\[
G \setminus U \subs G \setminus \bigcap_{i \in I} H_i = \bigcup_{i \in I} G \setminus H_i.
\]

Hence, $\lbrace G \setminus H_i \mid i \in I \rbrace$ is an open cover of $G \setminus U$. Because $G \setminus U$ is a closed subset of the quasi-compact space $G$, it is also quasi-compact and so there exists a finite subset $J$ of $I$ such that $G \setminus U \subs \bigcup_{j \in J} G \setminus H_j$. This implies $\bigcap_{j \in J} H_j \subs U$.

\item Let $\Omega$ be a neighborhood of $1 \in G$. Then $\Omega$ contains an open subset $U \subs G$ with $1 \in U$. Since $\bigcap_{i \in I} U_i = 1 \subs U$, the statement above implies the existence of a finite subset $J$ of $I$ such that $\bigcap_{j \in J} U_i \subs U$. As $\bigcap_{j \in J} U_i \in \ca{U}$, this shows that for each neighborhood $\Omega$ of $1 \in G$ there exists some $U \in \ca{U}$ with $U \subs \Omega$. Hence, $\ca{U}$ is a filter basis of $1 \in G$.
\end{asparaenum} \vspace{-\baselineskip}
\end{proof}

\begin{prop} \label{para:char_of_pro_finite}
Let $\ca{C}$ be a $\ro{C}0$-class of finite groups. The following are equivalent for a topological group $G$:
\begin{enumerate}[label=(\roman*),noitemsep,nolistsep]
\item \label{item:char_of_pro_finite_1} $G \in \se{cplproto}\tn{-}\ca{C} = \star \se{pro}\tn{-}\ca{C}$ for $\star \in \lbrace \se{uss},\se{ss},\se{us},\se{s} \rbrace$.
\item \label{item:char_of_pro_finite_2} $G$ is compact and there exists a filter basis $\ca{U}$ of $1 \in G$ consisting of open normal subgroups of $G$ such that $G/U \in \ca{U}$ for each $U \in \ca{U}$.
\end{enumerate}
If $\ca{C}$ satisfies $\ro{C}1$, then \ref{item:char_of_pro_finite_1} is also equivalent to
\begin{enumerate}[resume,label=(\roman*),noitemsep,nolistsep]
\item \label{item:char_of_pro_finite_3} $G \in \se{pro}\tn{-}\ca{C}$.
\end{enumerate}
If $\ca{C}$ satisfies $\ro{C}2$, then \ref{item:char_of_pro_finite_1} is also equivalent to
\begin{enumerate}[resume,label=(\roman*),noitemsep,nolistsep]
\item \label{item:char_of_pro_finite_4} $G$ is a compact totally disconnected group and $G/U \in \ca{C}$ for each open normal subgroup $U$ of $G$.
\end{enumerate}
If $\ca{C}$ satisfies $\ro{C}4$, then \ref{item:char_of_pro_finite_1} is also equivalent to
\begin{enumerate}[resume,label=(\roman*),noitemsep,nolistsep]
\item \label{item:char_of_pro_finite_5} $G$ is compact and $\bigcap_{U \in \ca{E}_{\ca{C}}^{\ro{f}}(G)} U = 1$.
\end{enumerate}
If $\ca{C}$ satisfies $\ro{C}1$ and $\ro{C}4$, then \ref{item:char_of_pro_finite_1} is also equivalent to
\begin{enumerate}[resume,label=(\roman*),noitemsep,nolistsep]
\item \label{item:char_of_pro_finite_6} $G$ is isomorphic to a closed subgroup of a direct product of groups contained in $\ca{C}$.
\end{enumerate}

\end{prop}

\begin{proof} \hfill

\ref{item:char_of_pro_finite_1} $\RA$ \ref{item:char_of_pro_finite_2}: It follows from \ref{prop:char_of_pro_compact}\ref{item:char_of_pro_compact_1} $\RA$ \ref{prop:char_of_pro_compact}\ref{item:char_of_pro_compact_2} that $G$ is compact and that there exists a filter basis $\ca{U}$ on $G$ consisting of closed normal subgroups of $G$ such that $\bigcap_{U \in \ca{U}} U = 1$ and $G/U \in \ca{C}$ for each $U \in \ca{U}$. Since $\ca{C}$ consists of finite groups, we conclude that each $U \in \ca{U}$ is open in $G$. As $\ca{U}$ is a filter basis, it is closed under finite intersections and so it follows from \ref{para:compact_filter_basis} that $\ca{U}$ is a filter basis of $1 \in G$.

\ref{item:char_of_pro_finite_2} $\RA$ \ref{item:char_of_pro_finite_1}: Since $\ca{U}$ is a filter basis of $1 \in G$, we have $\langle \ca{U} \rangle = \fr{V}(1)$ and therefore $\ro{lim} \ \ca{U} = \lim \ \fr{V}(1) = \lbrace 1 \rbrace$. According to \ref{prop:limit_compact_member_intersect} this implies $\bigcap_{U \in \ca{U}} U = 1$. Now the statement follows immediately from \ref{prop:char_of_pro_compact}\ref{item:char_of_pro_compact_2} $\RA$ \ref{prop:char_of_pro_compact}\ref{item:char_of_pro_compact_1}.

\ref{item:char_of_pro_finite_1} $\LRA$ \ref{item:char_of_pro_finite_3}: This is \ref{prop:char_of_pro_compact}\ref{item:char_of_pro_compact_1} $\LRA$ \ref{prop:char_of_pro_compact}\ref{item:char_of_pro_compact_3}.

\ref{item:char_of_pro_finite_2} $\RA$ \ref{item:char_of_pro_finite_4}: It follows from \ref{para:tot_disc_fb}\ref{item:tot_disc_fb_3} $\RA$ \ref{para:tot_disc_fb}\ref{item:tot_disc_fb_2} that $G$ is totally disconnected. Let $U$ be an open normal subgroup of $G$. Since $\ca{U}$ is a filter basis of $1 \in G$, there exists $V \in \ca{U}$ with $V \leq U$. By assumption $G/V \in \ca{C}$ and as $G/U \cong (G/V)/(U/V)$, it follows from $\ro{C}0$ and $\ro{C}2$ that $G/U \in \ca{C}$.

\ref{item:char_of_pro_finite_4} $\RA$ \ref{item:char_of_pro_finite_2}:  It follows from \ref{para:tot_disc_fb}\ref{item:tot_disc_fb_2} $\RA$ \ref{para:tot_disc_fb}\ref{item:tot_disc_fb_3} that there exists a filter basis $\ca{U}$ of $1 \in G$ consisting of open normal subgroups of $G$. By assumption we have $G/U \in \ca{C}$ for each $U \in \ca{U}$.

\ref{item:char_of_pro_finite_1} $\LRA$ \ref{item:char_of_pro_finite_5}: Since $G$ is compact and $\ca{C}$ consists of finite groups, we have $\ca{E}_{\ca{C}}^{\ro{f}}(G) = \ca{E}_{\ca{C}}^{\ro{f}}(G)$ and therefore the statement follows from \ref{prop:char_of_pro_compact}\ref{item:char_of_pro_compact_1} $\LRA$ \ref{prop:char_of_pro_compact}\ref{item:char_of_pro_compact_4}.

\ref{item:char_of_pro_finite_1} $\LRA$ \ref{item:char_of_pro_finite_6}: This is \ref{prop:char_of_pro_compact}\ref{item:char_of_pro_compact_1} $\LRA$ \ref{prop:char_of_pro_compact}\ref{item:char_of_pro_compact_5}.
\end{proof}

\begin{cor} \label{para:char_of_profinite}
The following are equivalent for a topological group $G$:
\begin{enumerate}[label=(\roman*),noitemsep,nolistsep]
\item $G \in \se{cplproto}\tn{-}\se{Grp}^{\ro{f}} = \star \se{pro}\tn{-}\se{Grp}^{\ro{f}}$ for $\star \in \lbrace \emptyset, \se{uss},\se{ss},\se{us},\se{s} \rbrace$.
\item $G$ is compact and there exists a filter basis of $1 \in G$ consisting of open normal subgroups of $G$.
\item $G$ is compact and totally disconnected.
\item $G$ is compact and $\bigcap_{U \lhd_{\ro{o}} G} U = 1$.
\item $G$ is isomorphic to a closed subgroup of a direct product of finite groups.
\end{enumerate}
\end{cor}

\begin{proof}
This is obvious.
\end{proof}

\begin{prop}
The category $\se{pro}\tn{-}\se{Grp}^{\ro{f}}$ is complete.
\end{prop}

\begin{proof}
This is an application of \ref{prop:stability_of_cplproto_c_for_compact}.
\end{proof}

\subsection{Procyclic groups}

\begin{prop}
The class $\se{Grp}^{\ro{f},\ro{cyc}}$ of all finite cyclic groups satisfies $\ro{C}0$, $\ro{C}1$ and $\ro{C}2$ but not $\ro{C}4$. In particular, $\se{Grp}^{\ro{f},\ro{cyc}}$ is not a formation. A pro-$\se{Grp}^{\ro{f},\ro{cyc}}$ group is called a \word{procyclic group}.\footnote{More precisely it should be called a pro-(finite cyclic) group.}
\end{prop}

\begin{proof}
It is obvious that $\ca{C}$ satisfies $\ro{C}0$, $\ro{C}1$ and $\ro{C}2$. To see that $\ca{C}$ does not satisfy C4, consider the following counter-example: Let $G = \ro{C}_2 \times \ro{C}_2$ and $N_1 = \ro{C}_2 \times \lbrace 1 \rbrace, N_2 = \lbrace 1 \rbrace \times \ro{C}_2$. Then $G/N_1$ and $G/N_2$ are finite cyclic groups but $G/N_1 \cap N_2 = G/\lbrace 1 \rbrace \cong G$ is non-cyclic.
\end{proof}

\begin{prop} \label{para:stability_of_procyclic}
The following holds:
\begin{enumerate}[label=(\roman*),noitemsep,nolistsep]
\item The quotient of a procyclic group by a closed subgroup is again procyclic.
\item A closed subgroup of a procyclic group is again procyclic.
\end{enumerate}
\end{prop}

\begin{proof}
This is an application of \ref{prop:stability_of_cplproto_c_for_compact}.
\end{proof}

\begin{defn} 
Let $G$ be a topological group. A subset $X \subs G$ is said to \words{topologically generate}{generator!topological} $G$, if the abstract group $\langle X \rangle$ generated by $X$ is dense in $G$, that is, $\langle X \rangle_{\ro{c}} \dopgleich \ro{cl}(\langle X \rangle) = G$.
\end{defn}

\begin{prop} \label{para:generator_limit}
Let $\ca{P}$ be a surjective projective system of compact groups with index set $I$, objects $G_i$ and morphisms $\varphi_{ij}: G_j \ra G_i$, $i \leq j$. Let $G = \ro{lim} \ \ca{P}$ and let $p_i:G \ra G_i$ be the canonical morphism. Then a subset $X \subs G$ topologically generates $G$ if and only if $p_i(X)$ topologically generates $G_i$ for each $i \in I$.
\end{prop}

\begin{proof}
Let $X$ topologically generate $G$. Since $\ca{P}$ is a surjective projective system of compact groups, it follows from \ref{prop:proj_systems_surjectivity} that $\ca{P}$ is already universally surjective, that is, $p_i$ is surjective for each $i \in I$. As $p_i$ is a morphism of compact groups, it is closed by the closed map lemma and therefore we get 
\[
G_i = p_i(G) = p_i( \ro{cl}( \langle X \rangle)) = \ro{cl}( p_i( \langle X \rangle ) ) = \ro{cl}( \langle p_i(X) \rangle ).
\] 
Hence, $p_i(X)$ topologically generates $G_i$ for each $i \in I$. Conversely, let $p_i(X)$ topologically generate $G_i$ for each $i \in I$. Let $Y \dopgleich \ro{cl}( \langle X \rangle) \rangle$. Since $p_i$ is closed, we have
\[
p_i(Y) = p_i(\ro{cl}( \langle X \rangle) ) = \ro{cl}( p_i(\langle X \rangle ) ) = \ro{cl}( \langle p_i(X) \rangle ) = G_i.
\]

Hence, an application of \cite[proposition 1.1.8]{RibZal00_Profinite-Groups_0} yields
\[
Y = \ro{cl}(Y) = \ro{lim} \ \ca{P} = G
\]
and therefore $X$ topologically generates $G$.
\end{proof}

\begin{prop} \label{para:epimorphism_generators}
Let $\varphi:G \ra H$ be a surjective morphism of compact groups. If $X \subs G$ topologically generates $G$, then $\varphi(X)$ topologically generates $H$.
\end{prop}

\begin{proof}
Since $\varphi$ is continuous and (due to the closed map lemma) closed, we have
\[
H = \varphi(G) = \varphi( \ro{cl}(\langle X \rangle) ) = \ro{cl}( \varphi(\langle X \rangle)) = \ro{cl}( \langle \varphi(X) \rangle ). 
\]
\end{proof}

\begin{prop} \label{para:procyclic_char}
For a topological group $G$ the following are equivalent:
\begin{enumerate}[label=(\roman*),noitemsep,nolistsep]
\item \label{item:procyclic_char_1} $G$ is procyclic.
\item \label{item:procyclic_char_2} $G$ is compact, totally disconnected and topologically generated by one element.
\item \label{item:procyclic_char_3} There exists an isomorphism
\[
G \cong \prod_{p \in P} H(p),
\] 
where $P$ is a set of prime numbers and each $H(p)$ is a procyclic pro-$p$ group.
\end{enumerate}
\end{prop}

\begin{proof} \hfill

\ref{item:procyclic_char_1} $\RA$ \ref{item:procyclic_char_2}: An application of \ref{para:char_of_pro_finite} shows that $G$ is compact and totally disconnected. As $\se{Grp}^{\ro{f},\ro{cyc}}$ satisfies $\ro{C}1$, there exists an isomorphism $G \cong \ro{lim} \ \ca{P}$ where $\ca{P}$ is a surjective projective system in $\se{TGrp}^{\ro{s}}$ whose objects are finite cyclic groups. We can assume without loss of generality that $G$ is equal to such a projective limit. Let $\ca{P}$ have index set $(I,\leq)$, objects $H_i$ and morphisms $\varphi_{ij}:H_j \ra H_i$ for $i \leq j$. For each $i \in I$ let $X_i$ be the set of all elements of $H_i$ which generate $H_i$. As $H_i$ is cyclic, $X_i$ is non-empty. Since $\varphi_{ij}$ is surjective, $\varphi_{ij}$ maps generators of $H_j$ to generators of $H_i$, that is, $\varphi_{ij}(X_j) \subs X_i$. Hence, we can define a projective system $\ca{P}'$ with index set $(I,\leq)$, objects $X_i$ and morphisms $\varphi_{ij}|_{X_j}:X_j \ra X_i$. As each $H_i$ is finite, $X_i$ is also finite and so it follows from proposition \cite[proposition 1.1.4]{RibZal00_Profinite-Groups_0} that $X \dopgleich \ro{lim} \ \ca{P}'$ is non-empty. By definition, the diagram
\[
\xymatrix{
& X \ar[dl]_{p_j'} \ar[dr]^{p_i'} & \\
X_j \ar[rr]_{\varphi_{ij}|_{X_j}} \ar@{>->}[d]_{\iota_j} & & X_i \ar@{>->}[d]^{\iota_i} \\
H_j \ar[rr]_{\varphi_{ij}} & & H_i
}
\]

commutes for all $i \leq j \in I$, where $p_i'$ is the canonical morphism and $\iota_i$ is the inclusion. Hence, the family $\lbrace \iota_i \circ p_i' \mid i \in I \rbrace$ induces a morphism $\eta:X \ra G$ making the diagram
\[
\xymatrix{
X \ar[rr]^{\eta} \ar[drr]_{\iota_i \circ p_i'} & & G \ar[d]^{p_i} \\
& & H_i
}
\]
commutative for each $i \in I$. Now, since $X \neq \emptyset$, we can choose an element $x \in X$ and by the commutativity of the diagram above this element satisfies 
\[
p_i \circ \eta(x)  = \iota_i \circ p_i'(\eta(x)) \in X_i
\]
for each $i \in I$. Thus, $p_i \circ \eta(x)$ is a generator of $H_i$ and so it follows from \ref{para:generator_limit} that $\eta(x)$ is a topological generator of $G$. 

\ref{item:procyclic_char_2} $\RA$ \ref{item:procyclic_char_1}: It follows from \ref{para:char_of_profinite} that there exists a filter basis $\ca{U}$ of $1 \in G$ consisting of open normal subgroups of $G$. Then $G \cong \ro{lim} \ G/\ca{U}$ and it follows from \ref{para:generator_limit} that $q_U(x)$ topologically generates $G/U$ for each $U \in \ca{U}$. As each $U \in \ca{U}$ is open in $G$, the quotient $G/U$ is discrete and therefore already $G/U = \langle q_U(x) \rangle$. Thus, $G/U$ is a finite cyclic group for each $U \in \ca{U}$ and this shows that $G$ is procyclic.

\ref{item:procyclic_char_1} $\RA$ \ref{item:procyclic_char_3}:  If $G$ is procyclic, then $G$ is in particular pronilpotent and so by \cite[proposition 2.4.3]{Wil97_Profinite-Groups_0} there exists an isomorphism $G \cong \prod_{p} G_p$, where $G_p$ is the $p$-Sylow subgroup of $G$. The group $G_p$ is a pro-$p$ group and as $G_p$ is a closed subgroup of $G$, it is procyclic by \ref{para:stability_of_procyclic}. Hence, $G$ is isomorphic to a product of procyclic pro-$p$ groups for different prime numbers $p$. 

\ref{item:procyclic_char_3} $\RA$ \ref{item:procyclic_char_1}: As each $H(p)$ is compact, the group $G$ is also compact. Let
\[
\ca{U} \dopgleich \lbrace \prod_{p \in P} U(p) \mid U(p) \lhd_{\ro{o}} H(p) \wedge \ro{Card}( \lbrace p \mid U(p) \neq H(p) \rbrace) < \infty \rbrace.
\]
Since $H(p)$ is profinite, it follows from \ref{para:char_of_profinite} that $\bigcap_{U \lhd_{\ro{o}} H(p)} U = 1$ for each $p \in P$. Hence, $\bigcap_{U \in \ca{U}} U = 1$ and as $\ca{U}$ is closed under finite intersections, an application of \ref{para:compact_filter_basis} shows that $\ca{U}$ is a filter basis of $1 \in G$. If $U = \prod_{p \in P} U(p) \in \ca{U}$, then 
\[
G/U \cong \prod_{p \in P} H(p)/U(p).
\]

This product is finite since $U(p) = H(p)$ for almost all $p$. Moreover, as $H(p)$ is a procyclic pro-$p$ group, it follows that $H(p)/U(p)$ is a finite cyclic $p$-group. Hence, $G/U$ is a finite product of finite cyclic $p$-groups for pairwise different prime numbers $p$ and therefore $G/U$ is a finite cyclic group by the Chinese remainder theorem. Thus, $G$ is procyclic by proposition \ref{para:char_of_pro_finite}.
\end{proof}

\begin{prop} \label{prop:uniqueness_of_p_power_procyclic}
Let $p$ be a prime number and let $n \in \NN \cup \lbrace \infty \rbrace$.\footnote{Confer \cite[chapter 2.3]{RibZal00_Profinite-Groups_0} for a discussion of supernatural numbers.} The following holds:
\begin{enumerate}[label=(\roman*),noitemsep,nolistsep]
\item Up to isomorphism there exists a unique procyclic group $\ro{C}_{p^n}$ of order $p^n$, namely $\ro{C}_{p^n} \cong \ZZ/p^n \ZZ$ if $n< \infty$ and $\ro{C}_{p^n} \cong \ZZ_p$ if $n = \infty$.
\item The group $\ZZ_p$ has a unique closed subgroup $H$ of index $p^n$, namely $H = p^n \ZZ$ if $n < \infty$ and $H = 1$ if $n = \infty$.
\end{enumerate}
\end{prop}

\begin{proof}
This is part of \cite[theorem 2.7.1]{RibZal00_Profinite-Groups_0}.
\end{proof}

\begin{thm} \label{thm:classification_of_procyclic}
For each supernatural number $n=\prod_p p^{n(p)}$, there exists up to isomorphism a unique procyclic group $\ro{C}_n$ of order $n$, namely $\ro{C}_n \dopgleich \prod_p \ro{C}_{p^{n(p)}}$.
\end{thm}

\begin{proof}
Let $n = \prod_{p} p^{n(p)}$. By \ref{para:procyclic_char} the group $\prod_{p} \ro{C}_{p^{n(p)}}$ is procyclic and by \cite[proposition 2.3.2]{RibZal00_Profinite-Groups_0} the order of this group is equal to 
\[
\prod_{p} \# \ro{C}_{p^{n(p)}} = \prod_{p} p^{n(p)} = n.
\]

This proves existence of a procyclic group of order $n$. To prove uniqueness, let $\Omega$ be a procyclic group of order $n$. By proposition \ref{para:procyclic_char} there exists an isomorphism $\Omega \cong \prod_{p} H(p)$, where $H(p)$ is a procyclic pro-$p$ group. In particular
\[
\prod_{p} p^{n(p)} = n = \# \Omega = \prod_{p} \# H(p).
\]
As $H(p)$ is a pro-$p$ group, its order is a (supernatural) power of $p$ by \cite[proposition 2.3.2]{RibZal00_Profinite-Groups_0}. So, the equation above implies that $\#H(p) = p^{n(p)}$ and therefore $H(p) \cong \ro{C}_{p^{n(p)}}$ by \ref{prop:uniqueness_of_p_power_procyclic}. Hence, 
\[
\Omega \cong \prod_p \ro{C}_{p^{n(p)}}
\]
and this proves uniqueness.
\end{proof}

\begin{prop} \label{para:compact_power_group_index}
Let $\Omega$ be a compact abelian group with topological generator $\omega$. If $n \in \NN_{>0}$, then $\Omega^n$ is an open subgroup of $\Omega$ with $\lbrack \Omega: \Omega^n \rbrack \mid n$ and $\Omega^n = \langle \omega^n \rangle_{\ro{c}}$.
\end{prop}

\begin{proof} 
It follows from \ref{para:comp_ab_power_group} that $\Omega^n$ is a closed subgroup of $\Omega$. Let $q:\Omega \ra \Omega/\Omega^n$ be the quotient morphism. Then
\[
\Omega/\Omega^n = q(\Omega) = q(\langle \omega \rangle_{\ro{c}}) = \langle q(\omega) \rangle_{\ro{c}}.
\]
As $\omega^n \in \Omega^n$, it follows that $\langle q(\omega) \rangle$ is finite and since $\Omega/\Omega^n$ is separated, we can conclude that
\[
\langle q(\omega) \rangle = \langle q(\omega) \rangle_{\ro{c}} = \Omega/\Omega^n
\]
is finite. Hence, $\Omega^n$ is open and $\lbrack \Omega: \Omega^n \rbrack = \ro{Ord}(q(\omega)) \mid n$.

Since the map $\mu_n:\Omega \ra \Omega$, $x \mapsto x^n$, is continuous and closed, it follows that
\[
\Omega^n = \mu_n(\Omega) = \mu_n(\ro{cl}(\langle \omega \rangle)) = \ro{cl}( \mu_n(\langle \omega \rangle)) = \ro{cl}( \langle \mu_n(\omega) \rangle) = \ro{cl}(\langle \omega^n \rangle).
\] 
\end{proof}

\begin{thm} \label{thm:unique_subgroups_of_procyclic}
Let $\Omega$ be a procyclic group and let $\Delta(\Omega)$ be the set of supernatural divisors of $\# \Omega$. The following holds:
\begin{enumerate}[label=(\roman*),noitemsep,nolistsep]
\item There exists a bijection between the set of closed subgroups of $\Omega$ and $\Delta(\Omega)$, given by $H \mapsto \lbrack \Omega : H \rbrack$. 
\item The above bijection induces a bijection between the set of open subgroups of $\Omega$ and $\Delta(\Omega) \cap \NN$.
\item If $n \in \Delta(\Omega) \cap \NN$, then the open subgroup of index $n$ of $\Omega$ is equal to $\Omega^n$.
\end{enumerate}
\end{thm}

\begin{proof} \hfill

\begin{asparaenum}[(i)]
\item Let $H$ be a closed subgroup of $\Omega$. By \cite[proposition 2.3.2]{RibZal00_Profinite-Groups_0} we have
\[
\# \Omega = \lbrack \Omega: 1 \rbrack = \lbrack \Omega: H \rbrack \cdot \lbrack H : 1 \rbrack = \lbrack \Omega : H \rbrack \cdot \# H
\]
and consequently $\lbrack \Omega:H \rbrack \in \Delta(\Omega)$ so that the map is well-defined. 

To show that the map is injective, we prove that $H$ is the only closed subgroup of $\Omega$ of index $n = \lbrack \Omega : H \rbrack$. We can assume that $\Omega = \prod_{p} \ro{C}_{p^{m(p)}}$. Then $H$ is procyclic according to \ref{para:stability_of_procyclic} and so it follows from \ref{para:procyclic_char} that $H = \prod_p H(p)$, where $H(p)$ is the $p$-Sylow subgroup subgroup of $H$. The group $H(p)$ is then necessarily a closed subgroup of $\ro{C}_{p^{m(p)}}$. Since
\[
\prod_{p} p^{n(p)} = n = \lbrack \Omega : H \rbrack = \# (\Omega / H) = \# ( \prod_p \ro{C}_{p^{m(p)}}/H(p) ) = \prod_p \# ( \ro{C}_{p^{m(p)}}/H(p) )
\]
it follows that $\lbrack \ro{C}_{p^{m(p)}}:H(p) \rbrack = p^{n(p)}$. Hence, $H(p)$ is the unique subgroup of index $p^{n(p)}$ of $\ro{C}_{p^{m(p)}}$ and therefore $H$ is unique.

To show that the map is surjective, let $n = \prod_{p} p^{n(p)} \in \Delta(\Omega)$. As $\# \Omega = \prod_p p^{m(p)}$, it follows that $p^{n(p)}$ divides $p^{m(p)}$ for each prime number $p$. Let $H(p)$ be the closed subgroup of index $p^{n(p)}$ of $\ro{C}_{p^{m(p)}}$. Then $H \dopgleich \prod_{p} H(p)$ is a closed subgroup of $\Omega$ and similar to the above one can show that $\lbrack \Omega:H \rbrack = n$.

\item This is evident since $\Omega$ is compact.

\item Let $H$ be the open subgroup of index $n$ of $\Omega$. Let $\omega$ be a topological generator of $\Omega$. As $\Omega/H$ is discrete, it follows that $\Omega/H = \langle \omega \ \ro{mod} \ H \rangle$. As $\#(\Omega/H) = n$, we conclude that $\omega^n \in H$ and therefore $\Omega^n = \langle \omega^n \rangle_{\ro{c}} \leq H$. Hence,
\[
n = \lbrack \Omega : H \rbrack \leq \lbrack \Omega: \Omega^n \rbrack \leq n
\]
and this implies that $H = \Omega^n$. 
\end{asparaenum} \vspace{-\baselineskip}
\end{proof}

\begin{prop} \wordsym{$\ZZ_P$} \label{prop:char_torsion_free_procyclic}
For a topological group $G$ the following are equivalent:
\begin{enumerate}[label=(\roman*),noitemsep,nolistsep]
\item $G$ is a $\ZZ$-torsion-free procyclic group.
\item $G$ is isomorphic to a group of the form $\ZZ_P \dopgleich \prod_{p \in P} \ZZ_p$ where $P$ is a set of prime numbers.
\end{enumerate}
\end{prop}

\begin{proof}
Let $G$ be a $\ZZ$-torsion-free procyclic group. Then $G$ is isomorphic to $\prod_{p \in P} \ro{C}_{p^{n(p)}}$, where $P$ is a set of prime numbers and $n(p) \in \NN_{>0} \cup \lbrace \infty \rbrace$. As $G$ is $\ZZ$-torsion-free, each $\ro{C}_{p^{n(p)}}$ is $\ZZ$-torsion-free and therefore $n(p) = \infty$ for all $p \in P$. In particular, $\ro{C}_{p^{n(p)}} \cong \ZZ_p$ for each $p \in P$ and therefore $G \cong \ZZ_P$. Conversely, it is evident that $\ZZ_P$ is a $\ZZ$-torsion-free procyclic group.
\end{proof}

\begin{lemma} \label{para:p_pprime_parts}
Let $P$ be a set of prime numbers. For a supernatural number $n = \prod_p p^{n(p)}$ let
\[
P(n) \dopgleich \prod_{p \in P} p^{n(p)}
\]
and
\[
P'(n) \dopgleich \prod_{p \notin P} p^{n(p)}.
\]
Then $n = P(n) \cdot P'(n)$, $P(nm) = P(n) \cdot P(m)$ and $P'(nm) = P'(n) \cdot P'(m)$ for all supernatural numbers $n$, $m$.
\end{lemma}

\begin{proof}
This is evident.
\end{proof}

\begin{prop} \label{para:torsion_free_procyclic_index}
Let $P$ be a set of prime numbers and let $\Omega \cong \ZZ_P$. Then $\lbrack \Omega: \Omega^n \rbrack = P(n)$ for all $n \in \NN_{>0}$.
\end{prop}

\begin{proof}
We can assume without loss of generality that $\Omega = \ZZ_P$. First suppose that $P = \lbrace p \rbrace$ for a prime number $p$. The ring $\ZZ_p$ is local with maximal ideal $(p)$. Hence, if $q$ is a prime number different from $p$, then $q \in \ZZ_p^\times$ and consequently $q\ZZ_p = \ZZ_p$. As $n = P(n) \cdot P'(n) = p^{n(p)} \cdot P'(n)$, it follows that $n \ZZ_p = P(n) \ZZ_p$ and an application of \ref{prop:uniqueness_of_p_power_procyclic} now shows that $\lbrack \ZZ_p : P(n) \ZZ_p \rbrack = P(n)$. This proves the claim in the case $P = \lbrace p \rbrace$.
Now let $P$ be arbitrary and let $n = \prod_p p^{n(p)}$. Since $n \ZZ_P = \prod_{p \in P} n \ZZ_p$, it follows from the above that
\[
\lbrack \ZZ_P: n \ZZ_P \rbrack = \prod_{p \in P} \lbrack \ZZ_p:n\ZZ_p \rbrack = \prod_{p \in P} p(n) = \prod_{p \in P} p^{n(p)} = n.
\]

\end{proof}

\newpage
\section{Miscellaneous}

In this chapter miscellaneous results are collected.

\subsection{Category theory}

\begin{prop} \label{para:mono_kernel_relation}
 Let $\ca{C}$ be a category and let $f:X \ra Y$ be a morphism in $\ca{C}$ having a kernel $k:K \ra X$.\footnote{This implies of course that $\ca{C}$ has a zero object.} The following holds:
\begin{enumerate}[label=(\roman*),noitemsep,nolistsep]
\item If $f$ is a monomorphism, then $k = 0$.
\item If $\ca{C}$ is preadditive and $k = 0$, then $f$ is a monomorphism.
\end{enumerate}
\end{prop}

\begin{proof} \hfill

\begin{asparaenum}[(i)]
\item By definition of the kernel, the diagram 
\[
\xymatrix{ K \ar[r]^k & X \ar@<1ex>[r]^f \ar@<-1ex>[r]_0 & Y }
\]
 commutes. Hence, we have $f \circ k = 0 \circ k = 0 = f \circ 0$ and as $f$ is a monomorphism, this implies that $k = 0$.

\item Let
\[
\xymatrix{
X' \ar@<1ex>[r]^{g_1} \ar@<-1ex>[r]_{g_2} & X \ar[r]^f & Y
}
\]
be a commutative diagram in $\ca{C}$. Since $\ca{C}$ is preadditive, we get
\[
f \circ g_1 = f \circ g_2 \lRA f \circ g_1 - f \circ g_2 = 0 \lRA f \circ (g_1 - g_2) = 0.
\]
Hence, the diagram
\[
\xymatrix{
X' \ar[r]^{g_1-g_2} \ar@{-->}[dr] & X \ar@<1ex>[r]^f \ar@<-1ex>[r]_0 & Y \\
& K \ar[u]_k
} 
\]
commutes, where the dashed arrow is induced by the universal property of $k$. As $k=0$, we get $g_1-g_2 = 0$, that is, $g_1 = g_2$ and so $f$ is a monomorphism.
\end{asparaenum} \vspace{-\baselineskip}
\end{proof}

\subsection{Double cosets} \label{sec:double_cosets} \wordsym{$U \textbackslash G / V$}

\begin{defn}
Let $G$ be a group and let $U,V$ be subgroups of $G$. A $(U,V)$-\word{double coset} in $G$ is an equivalence class of the equivalence relation $\sim$ on $G$ defined by
\[
g \sim g' \lLRA (\exists u \in U, v \in V)(g' = ugv).
\]

The set of all equivalence classes is denoted by $U \mybackslash G / V$.

\end{defn}

\begin{prop} \label{thm:basic_prop_of_double_cos}
Let $G$ be a group and let $U,V$ be subgroups of $G$. The following holds:
\begin{compactenum}[(i)]
\item The equivalence class of $g \in G$ with respect to $\sim$ is equal to $UgV$.
\item A subset $R \subs G$ is a complete set of representatives of the $(U,V)$-double cosets if and only if
\[
G = \coprod_{g \in R} UgV.
\]
\item The set $U \mybackslash G / V$ is the orbit space of the right action of $V$ on $U \mybackslash G$ defined by $U \mybackslash G \times V \ra U \mybackslash G$, $(Ug, v) \mapsto Ugv$. Similarly it is the orbit space of the left action of $U$ on $G/V$ defined by $U \times G/V \ra G/V$, $(u,gV) \mapsto ugV$.\footnote{This also explains the notation $U \mybackslash G/V$.}
\item $|U \mybackslash G / V| \leq \ro{min} \lbrace |U \mybackslash G|, |G/V| \rbrace$.
\end{compactenum}
\end{prop}

\begin{proof} 
 Let $E$ be the equivalence class of $g$. If $u \in U$ and $v \in V$, then by definition $ugv \sim g$ and therefore $UgV \subs E$. On the other hand, if $g' \in E$, then $g' \sim g$ and so there exist $u \in U, v \in V$ such that $g' = ugv \in UgV$. Hence, $E \subs UgV$. The remaining assertions are now obvious.
\end{proof}

\begin{prop} \label{para:transversal_to_double}
Let $G$ be a group and let $U,V$ be subgroups of $G$. The following holds:
\begin{enumerate}[label=(\roman*),noitemsep,nolistsep]
\item Suppose that $U^g \leq V$ for each $g \in G$. If $R$ is a complete set of representatives of $G/V$, then $R$ is also a complete set of representatives of $U \mybackslash G / V$.
\item Suppose that $^g \! V \leq U$ for each $g \in G$. If $R$ is a complete set of representatives of $U \mybackslash G$, then $R$ is also a complete set of representatives of $U \mybackslash G / V$.
\end{enumerate}

\end{prop}

\begin{proof}
We only prove the first statement, the second is proven similarly. By assumption we have a decomposition $G = \coprod_{g \in R} gV$ and therefore $G = \bigcup_{g \in G} UgV$. It remains to show that this union is disjoint. So, let $g \in Ug_1V \cap Ug_2V$ for some $g_1,g_2 \in R$. Then $g=u_1g_1v_1 = u_2g_2v_2$ for some $u_1,u_2 \in U, v_1,v_2 \in V$. Since $U^g \leq V$ for all $g \in G$, we can write $u_ig_i = g_iv_i'$ for some $v_i' \in V$. Now, $u_1g_1v_1 = u_2g_2v_2$, therefore $g_1v_1'v_1 = g_2v_2'v_2$ and consequently $g_2^{-1}g_1 = (v_2'v_2)(v_1'v_1)^{-1} \in V$. Hence, $g_1 = g_2$.
\end{proof}

\begin{prop} \label{para:double_coset_rep_lift}
Let $G$ be a group and let $U,V$ be subgroups of $G$. Let $R$ be a complete set of representatives of the $(U,V)$ double cosets in $G$. The following holds:
\begin{compactenum}[(i)]
\item For each $\rho \in R$ let $T_\rho$ be a right transversal of $U^\rho \cap V$ in $V$. Then $T = \lbrace \rho t \mid \rho \in R \tn{ and } t \in T_\rho \rbrace$ is a right transversal of $U$ in $G$.

\item For each $\rho \in R$ let $_\rho T$ be a left transversal of $U \cap {^\rho  V}$ in $U$. Then $T = \lbrace t \rho \mid \rho \in R \tn{ and } t \in {_\rho} T \rbrace$ is a left transversal of $V$ in $G$.
\end{compactenum}

\end{prop}

\begin{proof}
Let $g \in G$. Since we have a decomposition $G = \coprod_{\rho \in R} U \rho V$, we can write $g = u \rho v$ for some $\rho \in R, u \in U, v \in V$. As $T_\rho$ is a right transversal of $U^\rho \cap V$ in $V$, we have a decomposition $V = \coprod_{t \in T_\rho} (U^\rho \cap V)t$ and so we can write $v = w t$ for some $w \in U^\rho \cap V$, $t \in T_\rho$. Now, we have
\[
g = u \rho v = u \rho w t = u \rho w (\rho^{-1} \rho) t = (u \rho w \rho^{-1}) \rho t.
\]

As $u \rho w \rho^{-1} \in U \cdot {^\rho \! (U^\rho \cap V)} = U \cdot (U \cap {^\rho V}) \subs U$ and $\rho t \in T$, this implies that $g \in U T$. Hence, $G = \bigcup_{\tau \in T} U \tau$. It remains to show that this union is disjoint. For this, let $\tau_1,\tau_2 \in T$ and $g \in U \tau_1 \cap U \tau_2$. Then $g = u_1 \rho_1 t_1 = u_2 \rho_2 t_2$ for some $u_i \in U$ and $\rho_i \in R, t_i \in T_{\rho_i}$ with $\tau_i = \rho_i t_i$. In particular, $g \in U \rho_1 V \cap U \rho_2 V$ and therefore $\rho_1 = \rho_2 =: \rho$. Moreover,
\[
u_1 \rho t_1 = u_2 \rho t_2 \lRA \rho^{-1} u_2^{-1} u_1 \rho = t_2 t_1^{-1} \in V
\] 

and because $u_2^{-1} u_1 \in U$, we get $t_2 t_1^{-1} \in U^\rho$, so $t_2 t_1^{-1} \in U^\rho \cap V$. Since $t_1,t_2 \in T_\rho$, this implies $t_1 = t_2$ and consequently $\tau_1 = \tau_2$.
The second assertion is proven in the same way.
\end{proof}

\subsection{$G$-modules}

\begin{prop} \label{para:char_disc_modules}
Let $G$ be a topological group and let $A$ be an abstract $G$-module. The following are equi\-va\-lent:
\begin{enumerate}[label=(\roman*),noitemsep,nolistsep]
\item \label{para:char_disc_modules:1} $A$ is a topological $G$-module with respect to the discrete topology on $A$.
\item \label{para:char_disc_modules:2} For each $a \in A$ the stabilizer subgroup $G_a = \lbrace g \in G \mid ga = a \rbrace$ is open in $G$.
\item \label{para:char_disc_modules:3} $A = \bigcup_{U \in \ca{U}} A^U$, where $\ca{U}$ is the set of open subgroups of $G$. 
\end{enumerate}
\end{prop}

\begin{proof} Let $\mu: G \times A \ra A$ be the map defining the abstract $G$-module structure on $A$. 

\ref{para:char_disc_modules:1} $\RA$ \ref{para:char_disc_modules:2}: By assumption, $\mu$ is continuous with respect to the discrete topology on $A$ and so is its restriction to the open subset $G \times \lbrace a \rbrace$ of $G \times A$. Hence, $\mu|_{G \times \lbrace a \rbrace}^{-1}(\lbrace a \rbrace) = G_a \times \lbrace a \rbrace$ is open in $G \times \lbrace a \rbrace$ and therefore $G_a$ is open in $G$.

\ref{para:char_disc_modules:2} $\RA$ \ref{para:char_disc_modules:3}: As $a \in A^{G_a}$ for each $a \in A$, this is obvious. 

\ref{para:char_disc_modules:3} $\RA$ \ref{para:char_disc_modules:1} Let $g \in G$, $a \in A$ and $a' = \mu(g,a)$. By assumption, there is an open subgroup $U \subs G$ such that $a \in A^U$. Hence, $gU \times \lbrace a \rbrace$ is an open neighborhood of $(g,a) \in G \times A$ mapped to $a'$ by $\mu$. This shows that $\mu$ is continuous.
\end{proof}

\begin{prop} \label{para:gmodule_struct_on_ind_limit} \wordsym{$[U,a]$}
Let $G$ be a topological group and let $(\ca{N},\sups)$ be a filter basis on $G$ consisting of open normal subgroups of $G$. Let $\ca{I}$ be an inductive system in $\se{Ab}$ with index set $(\ca{N}, \sups)$, objects $A_U \in \se{Ab}$ and morphisms $\varphi_{V,U}: A_U \ra A_V$ for $U \sups V \in \ca{N}$. The following holds:

\begin{enumerate}[label=(\roman*),noitemsep,nolistsep]
\item \label{item:disc_gmod_inductive_lim} If each $A_U$, $U \in \ca{N}$, has the structure of an abstract $G/U$-module and if each morphism $\varphi_{V,U}$, $U \sups V \in \ca{N}$ is $G/V$-equivariant, then the abelian group $\ro{colim} \ \ca{I}$ admits a unique structure of a discrete $G$-module such that for each $U \in \ca{N}$ and $g \in G$ the diagram
\[
 \xymatrix{
 \ro{colim} \ \ca{I} \ar[r]^{\mu_g} & \ro{colim} \ \ca{I} \\
 A_U \ar[u]^{\iota_U} \ar[r]_{\mu_{U,g}} & A_U \ar[u]_{\iota_U}
 }
\] 
commutes, where $\mu_{U,g}$ denotes the action of $(g \ \ro{mod} \ U)$ on $A_U$, $\mu_g$ denotes the action of $g$ on $\ro{colim} \ \ca{I}$ and $\iota_U$ denotes the canonical map.
\item Let $\ca{I}$, $\ca{I}'$ be two inductive systems as in \ref{item:disc_gmod_inductive_lim} having the same index set $(\ca{N},\sups)$ and let $\psi_U: A_U \ra A_U', U \in \ca{N}$ be $G/U$-module morphisms such that for each pair $U \sups V$ in $\ca{N}$ the diagram
\[
\xymatrix{
A_U \ar[r]^{\psi_U} \ar[d]_{\varphi_{V,U}} & A_U' \ar[d]^{\varphi_{V,U}'} \\
A_V \ar[r]_{\psi_V} & A_V'
}
\]

commutes. Then there exists a unique $G$-module morphism $\psi: \ro{colim} \ \ca{I} \ra \ro{colim} \ \ca{I}'$ with respect to the unique $G$-module structure on the inductive limits such that for each $U \sups V \in \ca{N}$ the diagram
\[
\xymatrix{
\ro{colim} \ \ca{I}  \ar[r]^{\psi} & \ro{colim} \ \ca{I}' \\ 
A_U \ar[u]^{\iota_U} \ar[r]_{\psi_U} & A_U' \ar[u]_{\iota_U}
}
\]

commutes.

\end{enumerate}
\end{prop}

\begin{proof} \hfill

\begin{asparaenum}[(i)]
\item The colimit $A \dopgleich \ro{colim} \ \ca{I}$ obviously exists in $\se{Ab}$. We use the standard presentation of the inductive limit having as underlying set a quotient of $\coprod_{U \in \ca{N}} A_U$ and denote the equivalence class of $a \in A_U$ in $A$ by $[U,a]$. Any element of $A$ is of the form $[U,a]$ for some $U \in \ca{N}$, $a \in A_U$. If $A$ admits a structure of an abstract $G$-module such that the diagram above commutes, then the $G$-action is necessarily given by
\begin{equation} \label{equ:module_struct_on_ind_lim}
g [U,a] = [U, (g \ \ro{mod} \ U)  a].
\end{equation}

This proves uniqueness of such a $G$-module structure. It remains to prove that the above indeed defines a discrete $G$-module structure on $A$. First, we have to check that it is independent of the choice of a representative of an element $[U,a] \in A$. If $[U,a] = [V,b] \in A$, then there exists $W \in \ca{N}$ such that $W \subs U \cap V$ and $\varphi_{W,U}(a) = \varphi_{W,V}(b)$. Hence,
\[
(g \ \ro{mod} \ W) \varphi_{W,U}(a) = (g \ \ro{mod} \ W) \varphi_{W,V}(b) \lRA \varphi_{W,U}( (g \ \ro{mod} \ U)a) = \varphi_{W,V}( (g \ \ro{mod} \ V) b)
\]
\[
\lRA [U, (g \ \ro{mod} \ U)a ] = [V, (g \ \ro{mod} \ V) b].
\]

Obviously, $1[U,a] = [U,a]$ and $(g'g)[U,a] = g'(g[U,a])$ for all $[U,a] \in A$ and $g,g' \in G$. If $[U,a], [V,b] \in A$, then $[U,a] + [V,b] = [W, \varphi_{W,U}(a) + \varphi_{W,V}(b)]$ for any $W \in \ca{N}$ with $W \subs U \cap V$. Hence,
\[
g([U,a] + [V,b]) = g[W, \varphi_{W,U}(a) + \varphi_{W,V}(b)] = [W, (g \ \ro{mod} \ W) (\varphi_{W,U}(a) + \varphi_{W,V}(b))] 
\]
\[
= [W, (g \ \ro{mod} \ W) \varphi_{W,U}(a) + (g \ \ro{mod} \ W)\varphi_{W,V}(b)] = [W, \varphi_{W,U}( (g \ \ro{mod} \ U)a) + \varphi_{W,V}( (g \ \ro{mod} \ V) b)]
\]
\[
= [U, (g \ \ro{mod} \ U)a] + [V, (g \ \ro{mod} \ V) b] = g[U,a] + g[V,b].
\]

This shows that the $G$-action given by equation \ref{equ:module_struct_on_ind_lim} defines an abstract $G$-module structure on $A$. If $[U,a] \in A$ is any element, then $g[U,a] = [U,a]$ for all $g \in U$ and therefore $[U,a] \in A^U$. Hence, $A$ is indeed a discrete $G$-module by \ref{para:char_disc_modules}.

\item Let $A = \ro{colim} \ \ca{I}$ and $A' = \ro{colim} \ \ca{I}'$. By the commutativity assumption, $\psi$ is necessarily given by
\[
 \psi([U,a]) = [U,\psi_U(a)]
\]
for each $[U,a] \in A$. It remains to verify that this indeed defines a $G$-module morphism. If $[U,a], [V,b] \in A$, then $[U,a]+[V,b] = [W,\varphi_{W,U}(a) + \varphi_{W,V}(b)]$ for any $W \in \ca{N}$ with $W \leq U \cap V$ and therefore
\[
 \psi( [U,a] + [V,b]) = \psi(  [W,\varphi_{W,U}(a) + \varphi_{W,V}(b)] ) = [W,\psi_W( \varphi_{W,U}(a) + \varphi_{W,V}(b))]
\]
\[
[W, \psi_W(\varphi_{W,U}(a)) + \psi_W(\varphi_{W,V}(b))] = [W, \varphi_{W,U}'(\psi_U(a)) + \varphi_{W,V}'(\psi_V(b))]
\]
\[
= [U, \psi_U(a)] + [V,\psi_V(b)] = \psi([U,a]) + \psi([V,b]). 
\] 

Finally, for each $g \in G$,
\[
g \psi([U,a]) = g[U,\psi_U(a)] = [U, (g \ \ro{mod} \ U) \psi_U(a)] = [U, \psi_U( (g \ \ro{mod} \ U)a)]
\]
\[
 \psi( [ U, (g \ \ro{mod} \ U) a]) = \psi(g [U,a]),
\]

so $\psi$ is $G$-equivariant.
\end{asparaenum} \vspace{-\baselineskip}
\end{proof}

\subsection{Normal cores}

\begin{prop} \label{para:normal_core} \wordsym{$\ro{NC}_G(H)$}
Let $G$ be a group and let $H$ be a subgroup of $G$. The \word{normal core} of $H$ in $G$ is defined as
\[
\ro{NC}_G(H) \dopgleich \bigcap_{g \in G} H^g.
\]
The following holds:
\begin{enumerate}[label=(\roman*),noitemsep,nolistsep]
\item $\ro{NC}_G(H)$ is the largest normal subgroup of $G$ that is contained in $H$.
\item If $T$ is a right transversal of $H$ in $G$, then $\ro{NC}_G(H) = \bigcap_{t \in T} H^t.$
\item If $\lbrack G:H \rbrack < \infty$, then also $\lbrack G:\ro{NC}_G(H) \rbrack < \infty$.
\item \label{item:normal_core_qc_open} If $G$ is quasi-compact and $H$ is an open subgroup of $G$, then $\ro{NC}_G(H)$ is also an open subgroup of $G$.
\end{enumerate}
\end{prop}

\begin{proof} \hfill

\begin{asparaenum}[(i)]
\item It is obvious that $\ro{NC}_G(H) \leq H$ and since
\[
\ro{NC}_G(H)^{g'} = ( \bigcap_{g \in G} H^g )^{g'} = \bigcap_{g \in G} H^{gg'} = \bigcap_{g \in G} H^g = \ro{NC}_G(H)
\]
for any $g' \in G$ it follows that $\ro{NC}_G(H) \lhd G$. If $N \lhd G$ such that $N \leq H$, then $N = N^g \leq H^g$ for all $g \in G$ and therefore $N \leq \ro{NC}_G(H)$.

\item If $g \in G$, then $g = \kappa_T(g) t$ for some $t \in T$ and therefore $H^g = H^{\kappa_T(g) t} = (H^{\kappa_T(g)})^t = H^t$. Hence, $\ro{NC}_G(H) = \bigcap_{g \in G} H^g = \bigcap_{t \in T} H^t$. 

\item Let $T$ be a right transversal of $H$ in $G$. Since $T$ is finite, we have
\[
\lbrack G:\ro{NC}_G(H) \rbrack = \lbrack G: \bigcap_{t \in T} H^t \rbrack \leq \prod_{t \in T} \lbrack G:H^t \rbrack < \infty.
\]

\item Let $T$ be a right transversal of $H$ in $G$. Since $H$ is open and $G$ is quasi-compact, $T$ is finite and so $\ro{NC}_G(H)$ is a finite intersection of open subgroups of $G$ which is again open in $G$.
\end{asparaenum} \vspace{-\baselineskip}
\end{proof}

\subsection{The compact-open topology} \label{sect:compact_open}

\begin{prop}
Let $X$ be a set and let $\ca{B} \subs \fr{P}(X)$. The following holds:
\begin{compactenum}[(i)]
\item Let $\ca{T}(\ca{B})$ be the set of all topologies on $X$ which contain $\ca{B}$. Then
\[
\langle \ca{B} \rangle \dopgleich \bigcap_{\tau \in \ca{T}(\ca{B})} \tau
\]
is the smallest topology on $X$ which contains $\ca{B}$. It is called the topology on $X$ \word{generated} by $\ca{B}$.

\item The topology $\langle \ca{B} \rangle$ has the following explicit description:
\[
\langle \ca{B} \rangle = \lbrace X, \emptyset \rbrace \cup \lbrace \bigcup_{i \in I} \bigcap_{j \in J_i} U_{ij} \mid I \in \se{Set} \wedge (\forall i \in I)( J_i \subs I \wedge J_i \neq \emptyset \wedge |J_i| < \infty \wedge \lbrace U_{ij} \rbrace_{i \in I, j \in J_i} \subs \ca{B}) \rbrace.
\]

\item If $\tau$ is a topology on $X$, then $\langle \tau \rangle = \tau$.

\end{compactenum}
\end{prop}

\begin{proof} \hfill

\begin{asparaenum}[(i)]
\item The set $\ca{T}(\ca{B})$ is non-empty because it contains the discrete topology and then it follows immediately that $\langle \ca{B} \rangle$ is a topology on $X$. It is also evident that it is the smallest topology on $X$ containing $\ca{B}$.

\item Let $\ca{S}$ be the set on the right-hand side of the equation. Since $\langle \ca{B} \rangle$ is a topology containing $\ca{B}$, it is obvious that $\langle \ca{B} \rangle \sups \ca{S}$. Conversely, it is easy to verify that $\ca{S}$ is a topology on $X$ which contains $\ca{B}$ and this implies $\langle \ca{B} \rangle \subs \ca{S}$.

\item This is evident.
\end{asparaenum} \vspace{-\baselineskip}
\end{proof}

\begin{defn}
A \word{subbase} for a topological space $(X,\tau)$ is a subset $\ca{B} \subs \fr{P}(X)$ such that $\langle \ca{B} \rangle = \tau$.
\end{defn}

\begin{defn}
  Let $X$ and $Y$ be topological spaces and let $y \in Y$. The function $c_{Y,X}^y \in \hom_{\se{Top}}(X,Y)$ defined by $c_{Y,X}^y(x) = y$ for all $x \in X$ is called the \word{constant function} with value $y$.\footnote{Note that $c_{Y,X}^Y$ is also defined for $X = \emptyset$. In this case $c_{Y,X}^y = \emptyset$.}
\end{defn}

\begin{prop} \label{para:v_sets} \wordsym{$\ro{T}(A,B)$}
Let $X$ and $Y$ be topological spaces. For subsets $A \subs X$ and $B \subs Y$ we define 
\[
\ro{T}(A,B) = \lbrace f \in \hom_{\se{Top}}(X,Y) \mid f(A) \subs B \rbrace.
\]
The following holds:
\begin{compactenum}[(i)]
\item If $B \subs Y$, then $\ro{T}(\emptyset, B) = \hom_{\se{Top}}(X,Y)$.
\item If $A \subs X$ and $A \neq \emptyset$, then $\ro{T}(A,\emptyset) = \emptyset$.
\item \label{item:v_set_not_empty} If $A \subs X$ and $B \subs Y$, $B \neq \emptyset$, then $c_{Y,X}^b \in \ro{T}(A,B)$ for all $b \in B$. In particular, $\ro{T}(A,B) \neq \emptyset$. 
\item If $A' \subs A \subs X$ and $B' \subs B \subs Y$, then $\ro{T}(A,B') \subs \ro{T}(A',B)$.
\item If $A_1,A_2 \subs X$ and $B \subs Y$, then $\ro{T}(A_1 \cup A_2, B) = \ro{T}(A_1, B) \cap \ro{T}(A_2,B)$.
\item If $A \subs X$ and $B_1,B_2 \subs Y$, then $\ro{T}(A,B_1 \cap B_2) = \ro{T}(A,B_1) \cap \ro{T}(A,B_2)$.

\end{compactenum}

\end{prop}

\begin{proof} 
All statements are easy to verify.
\end{proof}

\begin{defn} \label{para:compact_open_top}
Let $X$ and $Y$ be topological spaces. The \word{compact-open topology} on $\hom_{\se{Top}}(X,Y)$ is the topology generated by
\[
\lbrace \ro{T}(K,U) \mid K \subs X \tn{ compact} \wedge U \subs Y \tn{ open} \rbrace.
\]
\end{defn}

\begin{conv}
If nothing else is mentioned, then $\hom_{\se{Top}}(X,Y)$ is always considered with the compact-open topology and any subset $H \subs \hom_{\se{Top}}(X,Y)$ is considered with the subspace topology which is then called the compact-open topology on $H$. 
\end{conv}

\begin{prop} \label{para:hom_co_separated}
Let $X$ and $Y$ be topological spaces. If $Y$ is separated, then $\hom_{\se{Top}}(X,Y)$ is also separated.
\end{prop}

\begin{proof}
This is \cite[chapitre X, \S3.4, remarque 1]{Bou07_Topologie-generale5-10}.
\end{proof}

\begin{prop} \label{para:hom_co_discrete}
Let $X$ be a compact space and let $Y$ be a discrete space. Then $\hom_{\se{Top}}(X,Y)$ is discrete.
\end{prop}

\begin{proof}
If $X = \emptyset$ and $Y \in \se{Set}$, then $\hom_{\se{Top}}(X,Y) = \lbrace \emptyset \rbrace$ is discrete. If $X \neq \emptyset$ and $Y = \emptyset$, then $\hom_{\se{Top}}(X,Y) = \emptyset$ is also discrete. Now, assume that both $X$ and $Y$ are non-empty. Let $f \in \hom_{\se{Top}}(X,Y)$. For each $y \in Y$ the set $\lbrace y \rbrace$ is closed in $Y$ because $Y$ is discrete and since $f$ is continuous, $K_y = f^{-1}(y)$ is a closed subset of $X$ which implies that $K_y$ is compact. As $\lbrace y \rbrace$ is also open in $Y$ and $f(K_y) \subs \lbrace y \rbrace$, we conclude that $f \in \ro{T}(K_y, \lbrace y \rbrace)$. Now, 
\[
X = f^{-1}(Y) = \bigcup_{y \in Y} f^{-1}(y) = \bigcup_{y \in Y} K_y.
\]
Since $X$ is compact and non-empty, we can find $y_1,\ldots,y_n \in Y$ such that $X = \bigcup_{i=1}^n K_{y_i}$ and $K_{y_i} \neq \emptyset$ for all $i \in \lbrace 1, \ldots, n \rbrace$. Let $V = \bigcap_{i=1}^n \ro{T}(K_{y_i}, \lbrace y_i \rbrace)$. Since $f \in \ro{T}(K_y,\lbrace y \rbrace)$ for all $y \in Y$, we have in particular $f \in V$. If $g \in V$ is another element, then $g \in \ro{T}(K_{y_i}, \lbrace y_i \rbrace)$ and as $K_{y_i} \neq \emptyset$, we thus have $g(K_{y_i}) = \lbrace y_i \rbrace$ for all $i \in \lbrace 1,\ldots,n \rbrace$. But as $K_{y_i} = f^{-1}(y_i)$, this implies $f = g$ on $K_{y_i}$, and as $X = \bigcup_{i=1}^n K_{y_i}$ we can conclude that $f = g$. Hence, $V = \lbrace f \rbrace$ and since $V$ is a finite intersection of elements of the subbase given in \ref{para:compact_open_top}, it is open. This shows that $\lbrace f \rbrace$ is open for each $f \in \hom_{\se{Top}}(X,Y)$ and therefore $\hom_{\se{Top}}(X,Y)$ is discrete.  
\end{proof}

\begin{prop} \label{para:hom_co_comp_cont}
Let $X,Y$ and $Z$ be topological spaces and let $Y$ be locally compact. Then the map
\[
\begin{array}{rcl}
\hom_{\se{Top}}(X,Y) \times \hom_{\se{Top}}(Y,Z) & \lra & \hom_{\se{Top}}(X,Z) \\
(u,v) & \longmapsto & u \circ v
\end{array}
\]
is continuous.
\end{prop}

\begin{proof}
This is \cite[chapitre X, \S3.4, proposition 9]{Bou07_Topologie-generale5-10}.
\end{proof}

\begin{prop} \label{para:hom_co_top_grp}
The following holds:
\begin{compactenum}[(i)]
\item \label{item:hom_co_top_grp} Let $X$ be a topological space and let $G$ be a topological group whose left and right uniform structures coincide (confer \cite[chapitre III, \S1]{Bou71_Topologie-Generale_0}). Then $\hom_{\se{Top}}(X,G)$ is a topological group with respect to pointwise multiplication.
\item Let $G$ be a topological group and let $A$ be an abelian topological group. Then $\hom_{\se{TGrp}}(G,A)$ is a topological group.
\end{compactenum}
\end{prop}

\begin{proof} \hfill

\begin{asparaenum}[(i)]
\item It follows from \cite[chapitre X, \S1.4, corollaire 2]{Bou07_Topologie-generale5-10} that $\hom_{\se{Set}}(X,G)$ is a topological group with respect to the topology of compact convergence and so the subgroup $\hom_{\se{Top}}(X,G) \subs \hom_{\se{Set}}(X,G)$ is a topological group with respect to the topology of compact convergence. Since $G$ is a uniform space, it follows from \cite[chapitre X, \S3.4, th\'eor\`eme 2]{Bou07_Topologie-generale5-10} that the topology of compact convergence on $\hom_{\se{Top}}(X,G)$ coincides with the compact-open topology. Consequently, $\hom_{\se{Top}}(X,G)$ is a topological group with respect to the compact-open topology.

\item As $A$ is abelian, the left and right uniform structures on $A$ coincide and therefore $\hom_{\se{Top}}(G,A)$ is a topological group by (\ref{item:hom_co_top_grp}). Hence, the subgroup $\hom_{\se{TGrp}}(G,A)$ of $\hom_{\se{Top}}(G,A)$ is also a topological group.
\end{asparaenum} \vspace{-\baselineskip}
\end{proof}

\begin{prop} \label{para:hom_co_zz}
Let $A$ be an abelian topological group and consider $\ZZ$ as a topological group with respect to the discrete topology. Then the map
\[
\begin{array}{rcl}
\Phi: \hom_{\se{TGrp}}(\ZZ,A) & \lra & A \\
f & \longmapsto & f(1)
\end{array}
\]
is an isomorphism of topological groups.
\end{prop}

\begin{proof}
It is easy to verify that $\Phi$ is an isomorphism of abstract groups and so it remains to verify that $\Phi$ is a homeomorphism. For $k \in \ZZ$ the $k$-th power map $\mu_k:A \ra A$ is continuous by \ref{para:power_map_cont}. Let $K$ be a compact subset of $\ZZ$ and let $U$ be an open subset of $A$. Since $\ZZ$ is discrete, we have $K = \lbrace x_1,\ldots,x_n \rbrace$. Then
\[
\ro{T}(K,U) \cap \hom_{\se{TGrp}}(\ZZ,A) = \lbrace f \mid f \in \hom_{\se{TGrp}}(\ZZ,A) \wedge (\forall i \in \lbrace 1,\ldots,n \rbrace)(f(x_i) \in U ) \rbrace
\] 
\[
= \bigcap_{i=1}^n \lbrace f \mid f \in \hom_{\se{TGrp}}(\ZZ,A) \wedge f(x_i) \in U \rbrace = \bigcap_{i=1}^n \lbrace f \mid f \in \hom_{\se{TGrp}}(\ZZ,A) \wedge \mu_{x_i}(f(1)) \in U \rbrace
\]
\[
\bigcap_{i=1}^n \lbrace f \mid f \in \hom_{\se{TGrp}}(\ZZ,A) \wedge f(1) \in \mu_{x_i}^{-1}(U) \rbrace
\]
and therefore
\[
\Phi( \ro{T}(K,U) \cap \hom_{\se{TGrp}}(\ZZ,A) ) = \bigcap_{i=1}^n \mu_{x_i}^{-1}(U). 
\]
As $\mu_k$ is continuous for all $k \in \ZZ$, the sets $\bigcap_{i=1}^n \mu_{x_i}^{-1}(U)$ are open in $A$ and as $\ro{T}(K,U) \cap \hom_{\se{TGrp}}(\ZZ,A)$ is a general element of the subbase of the compact-open topology on $\hom_{\se{TGrp}}(\ZZ,A)$, we conclude that $\Phi$ is an open map. But choosing $K = \lbrace 1 \rbrace$, we also get
\[
\Phi( \ro{T}(\lbrace 1 \rbrace,U) \cap \hom_{\se{TGrp}}(\ZZ,A) ) = \mu_{1}^{-1}(U) = U 
\]
and therefore
\[
\Phi^{-1}(U) = \ro{T}(\lbrace 1 \rbrace,U) \cap \hom_{\se{TGrp}}(\ZZ,A) 
\]
showing that $\Phi$ is continuous. Hence, $\Phi$ is a homeomorphism.
\end{proof}

\newpage
\clearpage
\ifthenelse{\boolean{version_pdf_hyperref}}
{
\phantomsection
}{}
\addcontentsline{toc}{section}{References}
\bibliographystyle{amsalpha}
\bibliography{library}

\newpage
\clearpage
\ifthenelse{\boolean{version_pdf_hyperref}}
{
\phantomsection
}{}

\begin{multicols}{2}[{\begin{center} \Large\bf Index \vskip2ex \end{center}}]
\input{general.ind}
\end{multicols}



\end{document}